\documentclass[12pt]{article}
\usepackage{amsfonts}
\usepackage{amsmath}
\usepackage{amssymb}
\usepackage{tipa}
\usepackage{authblk}

\usepackage{graphicx}
\usepackage{enumerate}

\DeclareMathOperator*{\sign}{sign}

\usepackage{enumitem}
\usepackage{bbm}
\usepackage[table]{xcolor}
\usepackage{url}
\usepackage{natbib}
\setcitestyle{round}
\renewcommand{\hat}{\widehat}

\newtheorem{theorem}{Theorem}
\newtheorem{lemma}{Lemma}

\newtheorem{remark}{Remark}
\newcommand{\E}{\mathbb{E}}
\newcommand{\Var}{\mathbb{V}\mathrm{ar}}

\newenvironment{proof}[1][Proof]{\noindent\textit{#1.} }{\ \rule{0.5em}{0.5em}}

\begin{document}

\title{A Weighted Likelihood Approach Based on Statistical Data Depths}
\author[1]{Claudio Agostinelli}
\author[2]{Ayanendranath Basu}
\author[3]{Giulia Bertagnolli}
\author[4]{Arun Kumar Kuchibhotla}

\affil[1]{Department of Mathematics, University of Trento, Trento, Italy. claudio.agostinelli@unitn.it}
\affil[2]{Interdisciplinary Statistical Research Unit, Indian Statistical Institute, Kolkata, India. ayanbasu@isical.ac.in}
\affil[3]{Faculty of Economics and Management, Free University of Bozen, Bozen, Italy. giulia.bertagnolli@unibz.it}
\affil[4]{Department of Statistics and Data Science, Carnegie Mellon University, Pittsburgh, PA, USA. arunku@stat.cmu.edu}

\maketitle

\begin{abstract}
We propose a general approach to construct weighted likelihood estimating equations with the aim of obtaining robust parameter estimates. 
We modify the standard likelihood equations by incorporating a weight that reflects the statistical depth of each data point relative to the model, as opposed to the sample.
An observation is considered regular when the corresponding difference of these two depths is close to zero. When this difference is large the observation score contribution is downweighted. 
We study the asymptotic properties of the proposed estimator, including consistency and asymptotic normality, for a broad class of weight functions. In particular, we establish asymptotic normality under the standard regularity conditions typically assumed for the maximum likelihood estimator (MLE). Our weighted likelihood estimator achieves the same asymptotic efficiency as the MLE in the absence of contamination, while maintaining a high degree of robustness in contaminated settings.
In stark contrast to the traditional minimum divergence/disparity estimators, our results hold even if the dimension of the data diverges with the sample size, without requiring additional assumptions on the existence or smoothness of the underlying densities.
We also derive the finite sample breakdown point of our estimator for both location and scatter matrix in the elliptically symmetric model.
Detailed results and examples are presented for robust parameter estimation in the multivariate normal model. Robustness is further illustrated using two real data sets and a Monte Carlo simulation study.
\end{abstract}

\section{Introduction}
Weighted Likelihood Estimating Equations (WLEEs) are often used with the aim of obtaining robust estimators. \citet{green1984} is perhaps one of the earliest example, \citet{field1994} proposes a WLEE with weights that depend on the tail behavior of the distribution function, \citet{markatou1997,markatou1998} define a WLEE with weights derived from the estimating equations of a disparity minimization problem.
By providing a strict connection between WLEE and disparity minimization problems, \citet{kuchibhotla2017} further improve \citeauthor{markatou1998}'s approach.
\citet{biswas2015} get ideas from both \citet{field1994} and \citet{markatou1998} and provide a similar procedure based on distribution functions. This is very natural and easy to implement, with the drawback, however, that the resulting estimators are not affine equivariant.
Moreover, the simplicity and appeal of these estimators in the univariate setting is quickly lost in multivariate setups, where $2^p$ comparisons are needed to construct the Pearson residual when the data dimension is $p$.

Statistical data depths offer a multivariate generalization of quantiles and, consequently, a measure of outlyingness. See \citet{liu2006} and the references therein for a comprehensive review. The primary objective of a depth function is to induce a center-outward ordering of multivariate observations, which serves several purposes, including: (i) finding centers of the distribution---depth maximizers; (ii) identifying regions containing the bulk of the data---depth regions and quantile depth based regions; (iii) comparing distributions using tools such as Depth-Depth plots and tests based on depth, among others.
A key property of depth functions is their invariance under affine transformations. Furthermore, for certain statistical depth functions and distributions, it has been shown that both the empirical distribution and the distribution function can be fully characterized by the depth contours \citep{struyf1999,kong2010,nagy2023}.

Our main goal is to propose a simple WLEE whose weights are based on statistical data depths, instead of densities or distribution functions.

Our methodology is partly motivated by the lack of generic robust fitting of parametric models for multivariate or infinite-dimensional data. Given that data depths can be defined for any data set with observations from a metric space \citep{nagy2020, nietoreyes2021, geenens2023, dai2023}, our proposal provides a simple technique to obtain robust and efficient estimators in multivariate or even infinite-dimensional data. To illustrate this fact, we use the wind speed data obtained from the Col De la Roa (Italy) meteorological station and fit a Gaussian process model. The data set constitutes daily data on wind speed recorded at regular intervals of $15$ min in years $2001-2004$, see Figure \ref{fig:wind-data}. Our analysis, as summarized in Figures \ref{fig:wind-l2} and \ref{fig:wind-e1p} implies that our proposed estimator is computationally fast and is significantly more robust compared to the MLE.

Following a review of weighted likelihood approach in Section \ref{sec:wlee}, we illustrate our proposal in Section \ref{sec:pearsonresiduals}. The asymptotic properties of the proposed estimators are discussed in Section \ref{sec:asymptotics}, while Section \ref{sec:breakdown} studies its finite sample breakdown point for location and scatter parameters within elliptically symmetric families. Section \ref{sec:examples} provides two real-data examples to illustrate the methodology. Section \ref{sec:montecarlo} reports the results of an extensive Monte Carlo experiment. Comments and conclusions can be found in Section \ref{sec:conclusions}. Appendix contains in Section \ref{sup:sec:proofs} detailed proofs of the theoretical results together with some auxiliary results. Section \ref{sup:sec:wind} provides some extra figures for the example presented in Section \ref{secsub:wind}, while Section \ref{sup:sec:monte} reports complete results of the Monte Carlo experiments.

\section{Weighted Likelihood Estimating Equations}
\label{sec:wlee}

Let $\mathbf{X} = (X_1, \ldots, X_n)$ be an i.i.d. random sample from a $p$-random vector $X$ with unknown distribution function $G$ and corresponding density function $g$, with respect to some sigma-finite measure. We assume a model for $X$ by $\mathcal{F} = \{F_\theta(x); \theta \in \Theta \subset \mathbb{R}^q, q \geq 1 \}$, and we denote by $f_\theta(x)$ the corresponding probability density function. 
Performing maximum likelihood estimation (MLE) on the data $\mathbf{X}$ with the parametric family $\mathcal{F}$ asymptotically yields the ``closest'' member of $\mathcal{F}$ to the distribution $G$, where the ``closeness'' is measured in terms of the Kullback--Leibler (KL) divergence. 
Formally, the MLE targets $\theta^*$ defined by
\[
  \theta^* := \arg\min_{\theta\in\Theta}\,\mathbb{E}_G[-\log f_{\theta}(X)].
\]  
Assuming classical regularity conditions that permit the interchange of derivative and integral, $\theta^*$ can also be written as a solution to the equation
\begin{equation}\label{eq:MLE-equation}
  \int u(x; \theta)dG(x) = 0\ ,
\end{equation}
where $u(x; \theta) = \partial \log f_{\theta}(x)/\partial \theta$.
To impart robustness, one may consider alternate ``targets'' that solve the weighted equation $\int w(x; \theta)u(x; \theta)dG(x) = 0$, where the weight function $w(\cdot; \cdot)$ downweights observations that deviate from the parametric family $\mathcal{F}$~\citep{basu2011statistical}. In the literature on the minimum distance approach, several such weight functions have been derived based on the so-called Pearson residuals introduced by~\cite{lindsay1994}. 
Note that, when the true distribution belongs to the parametric model, i.e., $g(x) = f_{\theta^*}(x)$ (a.s.) for some $\theta^* \in \Theta$, full asymptotic efficiency will require $w(x; \theta^*) = 1$ for all $x$.

Let $\hat{G}_n$ be the empirical distribution function. The Pearson residual at $x$, denoted by $\delta$, is defined by comparing the true density to the model density at $x$, as
\begin{equation*}
\delta(x; \theta) = \delta(x; F_\theta, G) = \frac{g(x)}{f_{\theta}(x)} - 1 \ ,
\end{equation*}
so that, when $g = f_\theta$  (a.s.) for a given $\theta \in \Theta$, the Pearson residuals are identically equal to zero for all $x$, whereas if $g \neq f_\theta$, in regions where the density $g$ exceeds $f_\theta$, the Pearson residuals are large positive values indicating a greater concentration of observations in those regions than is expected number under the model.
The finite sample version of the Pearson residuals is given by $\delta_n(x; \theta) = \delta(x; F_\theta, \hat{G}_n) = \frac{\hat{g}_n(x)}{f_{\theta}(x)} - 1$, which compares $\hat{g}_n(x)$, a non-parametric estimate of $g(x)$, to the model density $f_{\theta}(x)$.
\cite{lindsay1994} studied a class of estimators based on the Pearson residuals for discrete models, while \citet{basu1994} and \citet{markatou1998} discussed proposals for continuous ones. In particular, \citet{markatou1997,markatou1998} introduced weights $w(x; F_\theta, \hat{G}_n)=w(\delta_n(x; \theta))$ defined by
\begin{equation}
\label{equ:weights:markatou}
w(\delta_n(x; \theta)) = \frac{A(\delta_n(x; \theta)) + 1}{\delta_n(x; \theta) + 1} \ ,
\end{equation}
where $A(\cdot)$ is a Residual Adjustment Function \citep[RAF,][]{lindsay1994,park2002} of an appropriate disparity, thus obtaining a Weighted Likelihood Estimating Equation (WLEE)
\begin{equation}
\label{equ:wleeold}
\frac{1}{n} \sum_{i=1}^n w(X_i; F_\theta, \hat{G}_n) \ u(X_i; \theta) = 0\ ,
\end{equation}
where $u(X_i; \theta)$ denotes the contribution of the $i$-th observation to the score function.
Although these estimating equations are motivated by a disparity measure, there is no exact correspondence between the two approaches. 
\citet{kuchibhotla2017} proposed a WLEE in the same spirit, exactly corresponding to a disparity measure, and formally established its asymptotic and robustness properties. 
In the attempt to avoid the use of non-parametric density estimators, \citet{biswas2015} propose weights defined using the cumulative distribution function. Their methodology can be understood in terms of the following residual:
\begin{equation}
\label{eq:Pearson-residual-CDFs}
\delta_n(x; \theta) = \frac{\sign(2 F_\theta(x) - 1) \left[ \hat{G}_n(x) - F_\theta(x) \right]}{\min\{F_\theta(x), 1 - F_\theta(x)\}} \ ,
\end{equation}
where $\sign(x) = 2 \mathbf{1}(x \ge 0) - 1$ is a sign function. Note that $\min\{F_{\theta}(x), 1 - F_{\theta}(x)\}$ is a notion of depth \citep{tukey1975}, maximum at the median and minimum at the extremes. Instead of defining weights in the form of \eqref{equ:weights:markatou} --- which remains a valid option --- \citet{biswas2015} propose the following weight function
\begin{equation*}
w(x; F_\theta, G) = 1 + (H(\delta(x;\theta)) - 1)\mathbf{1}\{F_\theta(x) \notin[p,1-p]\} \ ,
\end{equation*}
for some $p \in [0, 0.5)$. Here, $H(\delta)$ is a smooth function defined on $[-1,\infty)$, which attains its maximum value of $1$ at $\delta=0$ and decreases smoothly in both tails as $\delta$ moves away from $0$. Specifically, $H(0) = 1$, $H'(0) = 0$ and the next higher non-zero derivative of $H(\delta)$ at $\delta=0$ is negative. One example of such a function is $H(\delta) = \exp(-a\delta^2)$ for some non-negative constant $a$.

Their approach is general and can be applied in the multivariate setting, as they demonstrate in their bivariate example. However, the definition in Equation~\eqref{eq:Pearson-residual-CDFs} effectively restricts their methodology to the univariate or, at most low-dimensional settings, and, more importantly, results in Pearson residuals which are not affine invariant. As a consequence, the resulting estimators are also not affine equivariant.

The goal of the present paper is to propose a general framework for constructing weights in the same spirit, but the residuals here are based on statistical data depths. This approach is broadly applicable and naturally leads to affine equivariant estimators. 

Another important contribution of our paper is that it works for multivariate distributions supported on all of $\mathbb{R}^p$. This is a significant achievement in the literature of the minimum distance approach for two reasons. 
\begin{enumerate}
  \item The Person residuals, as defined above, involve a ratio of densities which can be highly sensitive, especially as $x$ diverges.
  For example, consider the normal location model $F_{\theta} = N(\theta, 1)$ on the real line. If the true distribution is $G = N(0, 1)$, then $\delta(x; F_0, G) = 0$ for all $x$; however the empirical residual $\delta_n(x; \theta) = \widehat{g}_n(x)/f_{\theta}(x) - 1$ may not be close to zero for large values of $|x|$. For any $x$ outside the support of the data, any reasonable density estimator $\widehat{g}_n(x)$ would be close to zero, while $f_{\theta}(x) > 0$. This implies that Pearson residuals cannot be uniformly consistently estimated for distributions supported on the entire real line. Similar issues arise with ``modified'' residuals defined by \citet{biswas2015}. To address this problem, we consider residuals of the form $\delta(x; \theta) = (g(x) - f_{\theta}(x))/f_\theta^{\alpha}(x)$ for some constant $\alpha < 1$; see Equation~\eqref{equ:dpr} below.
  \item Traditional minimum disparity estimation has not been very successful with multivariate data, in the context of general parametric families. This is primarily because estimating Pearson residuals by directly plugging-in a density estimator introduces significant bias in the resulting estimator; see, for example,~\citet{krishnamurthy2014nonparametric} and references therein. 
  Moreover, multivariate density estimation is inherently challenging due to the curse of dimensionality, often requiring strong smoothness assumptions to achieve reasonable convergence rates. Even under such assumptions, the plug-in estimator tends to exhibit significant bias, preventing the estimator from attaining a mean-zero normal limiting distribution. See \citet[Thm 4.1]{tamura1986minimum} and~\citet[Sec. 3.3]{basu2011statistical} for more details. Smoothness conditions can be avoided defining Pearson residuals through distribution functions as in~\eqref{eq:Pearson-residual-CDFs}. This, as already noted, breaks the affine invariance in the multivariate setting. In contrast, our approach uses statistical depths which are a natural generalization of quantiles to the multivariate case and even applies to infinite-dimensional data. More importantly, we can estimate depths at a parametric rate for $p > 1$ and obtain an asymptotically normal estimator in the multivariate case without asymptotic bias. 
\end{enumerate}
\section{Depth based Pearson Residuals}
\label{sec:pearsonresiduals}
Let $d(x; G)$ be a statistical data depth \citep{zuo2000a,liu2006} for the point $x \in \mathbb{R}^p$ according to the distribution $G$ of the random variable $X \in \mathbb{R}^p$. Let $d_n(x; \hat{G}_n)$ be the finite sample version based on the empirical distribution function $\hat{G}_n$ of the sample $\mathbf{X}$. Denote by $\mathcal{G}$ the class of distributions in $\mathbb{R}^p$. 
Traditionally, depth functions $d(x; G)$ are assumed to satisfy the following desirable properties~\citep{liu1990,zuo2000a}. None of these are required for our theoretical results; see Remark~\ref{rem:choice-of-depth}.
\begin{enumerate}[label=\textbf{(P\arabic*)}]
\item \textit{Affine Invariance}. $d(A x + b; G_{A X + b}) = d(x; G)$ for any distribution function $G \in \mathcal{G}$, any $p \times p$ nonsingular matrix $A$ and any $p$-vector $b$. \label{eq:depth-invariance}
\item \textit{Maximality at Center}. For any distribution $G$ having ``center'' $\mu$ (e.g. the point of symmetry relative to some notion of symmetry), $d(\mu;G) = \sup\limits_{x \in \mathbb{R}^p} d(x; G)$. \label{eq:depth-maximality}
\item \textit{Monotonicity Relative to Deepest Point}. For any $G$ having deepest point $\mu$ (i.e., point of maximal depth), and any $x$, the function $a \mapsto d(\mu + a(x - \mu); G)$ is non-decreasing on $[0, 1]$. \label{eq:depth-monotonicity}
\item \textit{Vanishing at Infinity}. $d(x; G) \rightarrow 0$ as $\Vert x \Vert \rightarrow \infty$, for each $G \in \mathcal{G}$, where $\Vert \cdot \Vert$ is the $L_2$-norm. \label{eq:depth-vanishing}
\end{enumerate}
For any given $x$ in the support of $G$, we define the analogue of the Pearson residual function in the spirit of \citet{lindsay1994} as
\begin{equation} \label{equ:dpr}
\tau_n(x; \theta) = \tau(x; F_\theta, \hat{G}_n) = \frac{d_n(x;\hat{G}_n) - d(x;F_\theta)}{d(x;F_\theta)^\alpha} \ ,
\end{equation}
for some $0 < \alpha < 3/4$ to be specified later. The intuition for raising the denominator to an exponent of $\alpha$ can be understood by considering the Pearson residual with distribution functions as in~\eqref{eq:Pearson-residual-CDFs}. Note that
\begin{equation*}
\Var\left(\frac{\hat{G}_n(x) - G(x)}{\min\{G(x), 1 - G(x)\}} \right) = \frac{\max\{G(x), 1 - G(x)\}}{n \min\{G(x), 1 - G(x)\}},
\end{equation*}
while
\begin{equation*}
\Var\left(\frac{\hat{G}_n(x) - G(x)}{\sqrt{\min\{G(x), 1 - G(x)\}}} \right) = \frac{\sqrt{\max\{G(x), 1 - G(x)\}}}{n} \le \frac{1}{n},
\end{equation*}
which is independent of $G(x)$. Note that $\min\{G(x), 1 - G(x)\}$ serves as a notion of depth; it attains its minimum at the endpoints of the support and its maximum at the center (median) of the distribution. The two equalities above show that a direct analogue of Pearson residuals is not a ``stable'' function of the empirical distribution function near the endpoints of the support of the distribution. Raising the denominator to an appropriate exponent stabilizes the residuals across all values of $x$. A formal result in this direction for~\eqref{equ:dpr} is provided as Lemma \ref{lemma:ratescaleddepth} in the Appendix \ref{sup:sec:proofs}.
This change in the residual still retains the downweighting property of the usual Pearson residuals. If, at the observation $x$, the depths $d_n(x;\hat{G}_n)$ and $d(x;F_\theta)$ are substantially different, then $\tau_n(x; \theta)$ will be significantly different from $0$, suggesting the need for downweighting.
We refer to the quantity $\tau_n$, as defined through Equation \eqref{equ:dpr}, as the Depth Pearson Residuals (DPRs).

We apply a weight function $w(\cdot)$ to the DPRs $\tau_n(x; \theta)$, which is designed to attain its maximum at $0$ and descend smoothly on both sides of $0$. This ensures that the good observations are given proper importance while incompatible ones are downweighted. The weighted likelihood estimator is the solution of
\begin{equation} \label{equ:wlee}
\frac{1}{n} \sum_{i=1}^n w(\tau_n(X_i;\theta)) u(X_i;\theta) = 0 \ .
\end{equation}
The true parameter $\theta_0 \in \Theta$ is defined as the solution of the following equation
\begin{equation} \label{equ:wleetheoretical}
\int w(\tau(x;\theta)) u(x;\theta) \ dG(x) = 0 \ ,
\end{equation}
which is the theoretical version of the WLEE~\eqref{equ:wlee} with $\tau(x;\theta) = (d(x;G) - d(x;F_\theta))/d^\alpha(x;F_{\theta})$. 
Using weights based on \eqref{equ:weights:markatou} or on the proposal by \citet{biswas2015}, leads to a WLEE that can be solved via an iterative reweighting algorithm.

The DPRs have the desired behavior of being equal to $0$ whenever $G = F_{\theta^*}$ identically for some $\theta^* \in \Theta$, and of attaining large values in regions where the two distributions differ. 
Furthermore, due to the invariance property~\ref{eq:depth-invariance} of the depth function $d$, the DPR is also invariant to affine transformations. 
Hereafter, we employ the half-space depth, although other depth functions could also be considered. 
Our asymptotics depend only on the rate of convergence of the $\sup_x|d_n(x;\hat{G}_n) - d(x;G)|$, which can be derived for many depths. For instance, the simplicial depth of \citet{liu1990} enjoys the same rate of convergence as the half-space depth \citep[see][]{arcones1993} and the asymptotics work similarly. 

\begin{remark}[Choice of Depth.]\label{rem:choice-of-depth} Suppose the true data generating distribution $G$ belongs to the parametric family $\mathcal{F}$, i.e., $G = F_{\theta^*}$ for all $x$ for some $\theta^*\in\Theta$. Then the population version $\theta_0$ of our weighted likelihood estimator equals $\theta^*$ (i.e., it is Fisher consistent) so that $\tau(x; \theta^*) = 0$ for all $x$ and $w(0) = 1$. 
  This condition on $\tau$ is equivalent to $d(x; G) = d(x; F_{\theta^*})$ whenever $G = F_{\theta^*}$. The weighted likelihood estimator is consistent for $\theta^*$ if $d_n(x; \hat{G}_n)$ is uniformly (in $x$) close to $d(x; F_{\theta^*})$ as the sample size increases. 
  It is important to emphasize that this condition is not related to the ability of the data depth function to characterize probability distributions. Recall that a data depth $d(\cdot; \cdot)$ is said to characterize distributions if $d(x; P) = d(x; Q)$ for all $x$ for two distributions $P$ and $Q$ implies $P = Q$. While this is a natural condition to impose on a data depth, there is a rich literature on this subject showing that some natural data depths do not satisfy this characterization condition. In particular, \cite{nagy2021halfspace} has shown that the half-space depth does not characterize distributions, in general. \cite{koshevoy1998lift} and \cite{koshevoy2002tukey} have shown that the zonoid depth and simplicial volume depth both characterize distributions under mild conditions. See also~\cite{laketa2023simplicial} for more details and references. Another natural condition on data depths is that the contours produced by the data depth should match the high probability regions of the distributions, i.e., for every $c > 0$, there exists some $l = l(c) > 0$ such that $\{x:\, d(x; G) \le c\} = \{x:\, g(x) \ge l\}$. This, in particular, implies that if $G$ is a unimodal distribution, then the mode is the highest depth point (cf. property~\ref{eq:depth-maximality}) and the level sets of the density correspond to level sets of the data depth. We can call this condition ``contour characterization'', which is slightly weaker than distribution characterization. Contour characterization may also be violated by many data depths. In particular,~\cite{dutta2011intriguing} have shown that the half-space depth does not satisfy this condition, in general. Neither the Fisher consistency nor the asymptotic normality/efficiency results depend on the two cited characterization properties of the depth function used in the construction of the weights.
\end{remark}

\section{Asymptotics}
\label{sec:asymptotics}
In this section, we present the consistency and the asymptotic normality of the proposed weighted likelihood estimator.

For any $K_0, K_1 > 0$, define the class of weight functions
\begin{align*}
\mathcal{W}(K_0, K_1) &:= \bigg\{w(\cdot):\,0 \le w(t) \le 1\mbox{ for all }-1 \le t < \infty,\, w(0) = 1,\,w'(0) = 0,\,\\
&\qquad\qquad\sup_t|w'(t)(t + 1)| \le K_0,\,\sup_t|w''(t)(t + 2)^2| \le K_1\bigg\},
\end{align*}
where we assume that $w$ admits at least two derivative for all $\tau \in [-1,\infty)$ and denote these by $w^\prime$ and $w^{\prime\prime}$. For any function $\gamma: \mathbb{R}^q \rightarrow \mathbb{R}$ and $j,k,h \in \{1, \ldots, q\}$, we write
\begin{equation*}
\nabla_j \gamma(\theta) = \frac{\partial}{\partial \theta_j} \gamma(\theta), \quad \nabla_{j,k} \gamma(\theta) = \frac{\partial^2}{\partial \theta_j \partial \theta_k} \gamma(\theta), \quad \nabla_{j,k,h} \gamma(\theta) = \frac{\partial^3}{\partial \theta_j \partial \theta_k \partial \theta_h} \gamma(\theta).
\end{equation*}  
We also set $\nu_k(x;\theta) = \nabla_k d(x;F_\theta)/d(x;F_\theta)$, $\nu_{j,k}(x;\theta) = \nabla_{j,k} d(x;F_\theta)/d(x;F_\theta)$. We let $u(x;\theta)$ be the usual likelihood score function and $u_j(x;\theta) = \nabla_j \log f_\theta(x)$, $u_{j,k}(x;\theta) = \nabla_{j,k} \log f_\theta(x)$ and $u_{j,k,h}(x;\theta) = \nabla_{j,k,h} \log f_\theta(x)$. Mathematically, we state the required conditions as follows.
\begin{enumerate}[label=\textbf{(A\arabic*)}]
\item \label{ass:4} The quantity $d(X;G)$ as a random variable on $[0,1/2)$ has a Lebesgue density in a right neighborhood of $0$.
\item \label{ass:5} The parameter space $\Theta$ is a convex subset of $\mathbb{R}^q$ with $\theta_0$ in its interior. Moreover, there exists an open neighborhood $N(\theta_0)$ of the true parameter $\theta_0 \in \Theta$ and integrable functions $M_{ijkh}(\cdot), 1\le i\le 4, 1\le j,k,h\le q$, such that for all $\theta \in N(\theta_0)$, we have
\begin{enumerate}
\item $|\nu_j(X;\theta) u_{k,h}(X;\theta)| \le M_{1jkh}(X)$,
\item $|\nu_j(X;\theta) \nu_k(X;\theta) u_h(X;\theta)|\le M_{2jkh}(X)$,
\item $|\nu_{j,k}(X;\theta) u_h(X;\theta)|\le M_{3jkh}(X)$,
\item $|u_{j,k,h}(X;\theta)| \le M_{4jkh}(X)$.
\end{enumerate}
\item \label{ass:6} $\E\left( u_{j,k}(X;\theta_0) \right)^2 < \infty$ and $\E\left( \nu_k(X;\theta_0) u_j(X;\theta_0) \right)^2 < \infty$ for all $j,k=1,\ldots,p$.
\item \label{ass:7} The Fisher information $I(\theta) = \E \left( u(X;\theta)u^\top(X;\theta) \right)$ is a finite positive definite matrix for all $\theta \in N(\theta_0)$.
\end{enumerate}

\begin{remark}[Randomness of weight function.]
The weight function $w$ might depend on the random variables $X_1, \ldots, X_n$ and hence it might be a random function, however constants $K_0$ and $K_1$ in the definition of $\mathcal{W}(K_0, K_1)$ must be non-stochastic. In this context, the weight function $w$ in \eqref{equ:wleetheoretical} should be interpreted as the limiting version as $n \rightarrow \infty$. Observe that $\theta_0 = \theta^*$ for any weight function in $\mathcal{W}(K_0, K_1)$ if $G = F_{\theta^*}$ identically.
\end{remark}

\begin{remark}[Uniform Convergence of Data Depth and Assumption~\ref{ass:4}]
\citet[Remark A.3 of][]{zuo2000b} provide some references for the uniform convergence of $d_n(x; \hat{G}_n)$ to $d(x; G)$ for various depth functions. Also see \citet{nagy2020uniform}.
\end{remark}

\begin{remark}[Elliptically symmetric distributions]
Given an elliptically symmetric distribution $F_{\mu,\Sigma}$ with location parameter $\mu$ and scatter matrix $\Sigma$, let $\Delta(x; \mu, \Sigma) = (x - \mu)^\top \Sigma^{-1}(x - \mu)$ be the squared Mahalanobis distance and $F_\Delta$ be the distribution function of $\Delta(X; \mu, \Sigma)$. Assumption \ref{ass:4} holds for the multivariate normal distribution since by \citet[][Corollary 4.3]{zuo2000b} we have, for the half-space depth
\begin{equation}
d(x;F_{\mu,\Sigma}) = \frac{1 - F_{\chi^2_q}(\Delta(x; \mu,\Sigma))}{2}
\end{equation}
where $F_{\chi^2_q}$ is the distribution function of a $\chi^2_q$ chi-squared with $q$ degrees of freedom random variable. Hence, $d(X; F_{\mu,\Sigma})$ has a density when $\Delta(X; \mu,\Sigma)$ has a density.

More generally, \citet[][Example 4.1]{masse2004} shows that for any class of elliptically symmetric distributions $F_{\mu, \Sigma}$ with a Lebesgue density, the level sets of the half-space depth are given by
\begin{equation*}
\mathcal{Q}_{t} = \left\{x:\, d(x; F_{\mu, \Sigma}) \ge t \right\} = \left\{x:\, \Delta(x; \mu,\Sigma) \le k(t) \right\} \ ,
\end{equation*}
for some strictly decreasing continuous function $k:\, (0, 1/2] \to [0, \infty)$ such that $k(1/2) = 0$. 
Then,
\begin{align*}
\Pr\left(d(X; F_{\mu_0, \Sigma_0}) \le t \right) = \Pr(X \in \mathcal{Q}_{t}^{c}) & = \Pr\left( \Delta(X, \mu_0;\Sigma_0) \ge k(t) \right) \\
& = \lim_{s \to 0}\, F_\Delta(k(s)) - F_\Delta(k(t)) \ .
\end{align*}
The quantity $F_\Delta(k(s)) - F_\Delta(k(t))$ behaves linearly for $t\approx 0$ if $F_\Delta$ and $k$ are differentiable. The function $F_\Delta(\cdot)$ is differentiable if $\Delta(X; \mu_0,\Sigma_0)$ is absolutely continuous with respect to the Lebesgue measure. This condition holds as long as at least one coordinate of $X$ is absolutely continuous with respect to the Lebesgue measure.

These calculations also imply that the DPRs can be efficiently computed for the class of elliptically symmetric parametric families. We note here that the finite sample half-space depth in dimension $p$ can be computed efficiently using the algorithms proposed in \citet{liu2017} and \citet{dyckerhoff2016}, which are available in the \texttt{R} package \texttt{ddalpha} by \citet{pokotylo2016}.
\end{remark}

\begin{theorem} \label{the:normalityresults}
Let the true distribution belong to the model (i.e., $G = F_{\theta_0}$ identically). Fix $0 < \alpha < 3/4$ and assume that the depth Pearson residuals are computed using the half-space depth. Under assumptions \ref{ass:4}-\ref{ass:7}, the following results hold.
\begin{enumerate}
\item Set $A_{n,j}^{(w)} =  n^{-1} \sum_{i=1}^n w(\tau_n(X_i, \theta_0)) u_j(X_i;\theta_0)$. Then as $n\to\infty$,
\begin{equation} \label{equ:normalityresultsA}
\sup_{w \in \mathcal{W}(K_0, K_1)} \sqrt{n}\left| A_{n,j}^{(w)} - \frac{1}{n} \sum_{i=1}^n u_j(X_i;\theta_0)\right| = o_p(1) \ , \qquad \forall j=1,\ldots,p \ .
\end{equation}
\item Set $B_{n,j,k}^{(w)} = n^{-1}\sum_{i=1}^n \nabla_k \left[w(\tau_n(X_i, \theta)) u_j(X_i;\theta) \right]_{\theta=\theta_0}$. Then as $n\to\infty$,
\begin{equation} \label{equ:normalityresultsB}
\sup_{w \in \mathcal{W}(K_0, K_1)} \left|B_{n,j,k}^{(w)} - \frac{1}{n} \sum_{i=1}^n u_{j,k}(X_i;\theta_0)\right| = o_p(1) \ , \qquad \forall j,k = 1,\ldots,p \ .
\end{equation}
\item Set $C_{n,j,k,h}^{(w)} = n^{-1} \sum_{i=1}^n \nabla_{k,h} \left[w(\tau_n(X_i, \theta)) u_j(X_i;\theta) \right]_{\theta=\bar{\theta}},$ for any (data-dependent) $\bar{\theta}\in N(\theta_0)$. Then as $n\to\infty$,
\begin{equation} \label{equ:normalityresultsC}
\sup_{w \in \mathcal{W}(K_0, K_1)} |C^{(w)}_{n,j,k,h}| = O_p(1) \ , \qquad \forall j,k,h=1, \ldots, p \ .
\end{equation}
\end{enumerate}
\end{theorem}

\begin{remark} \label{remark:2}
Theorem~\ref{the:normalityresults} verifies assumptions 1-3 of~\cite{yuan1998asymptotics}. Hence, by Theorems 1 and 4 of~\cite{yuan1998asymptotics}, we get that there exists a root to the weighted likelihood estimating equation (WLEE) in the neighborhood of $\theta_0$ that is consistent and asymptotically normal for $\theta_0$. 
Convergence in Equation \eqref{equ:normalityresultsB} establishes that the matrix $B_n = ((B_{n,j,k} ))_{p \times p}$ converges in probability to the negative of the Fisher information matrix $I(\theta_0)$, which implies that the root of the WLEE in the neighborhood of $\theta_0$ is an asymptotically efficient estimator. It should be stressed here that the WLEE need not have a unique root and in general, under contamination, one would expect at least two roots: a stable robust root (the desirable one); and a root that behaves like the MLE (somewhere between
the truth and the outlier data). Depending on the distribution of the outliers and the level of contamination the WLEE might exhibits other roots that might provide further insight into the data, e.g. one root that fits only to the outliers. This is illustrated in the Vowel Recognition Data analysed in Section \ref{secsub:vowel}.
\end{remark}

\begin{remark}
\label{remark:alpha}  
The assumption $0 < \alpha < 3/4$ comes from second moment assumption on $u_j(X; \theta_0)$. This can be relaxed if we can assume more moments for the score function. For example, if $u_j(X; \theta_0)$ has $q$-moments, we get that the largest value of $\alpha$ allowed is the one that satisfies
\begin{equation*}
  \min_{\alpha \in [0,1]} \{3/(4 \alpha), 0.5/(2 \alpha - 1)\} = 0.5/(1 - q^{-1}),
\end{equation*}  
which leads to $3/4$ for $q=2$, $7/8$ for $q=4$, $11/12$ for $q=6$, $15/16$ for $q=8$ and $19/20$ for $q=10$. 
Also, note that for the normal distribution, all moments are finite, so that the largest allowed value for $\alpha$ is $1$.
\end{remark}
The above results might also be used to obtain the asymptotic normality convergence.
\begin{theorem} \label{the:normality}
Under assumptions of Theorem \ref{the:normalityresults} we have
\begin{equation}
\sqrt{n}( \hat{\theta}_w - \theta_0) \stackrel{d}{\rightarrow} N_p(0, I^{-1}(\theta_0)) \ .
\end{equation}
\end{theorem}

\begin{proof}[Proof of Theorem \ref{the:normality}]
The proof follows exactly as in the proof of Theorem 4.1 (iii) of \citet[][pp. 429-435]{lehmann1983} by considering a Taylor expansion of the estimating equations in \eqref{equ:wlee} and by using the convergences in Theorem \ref{the:normalityresults}.
\end{proof}

\section{Finite Sample Breakdown Point}
\label{sec:breakdown}
For a given fixed $\theta$ let us consider $\tau_i = \tau_n(X_i; \theta)$ $i=1,\ldots,n$ and let $\tau^\ast_\theta = \operatorname{median}(\tau_1, \ldots, \tau_n)$ be the median DPR. We then define
\begin{equation} \label{equ:wstar}
w^\ast(\tau) =
\begin{cases}
w(\tau) & \text{if } \tau \le \tau^\ast_\theta + \xi \ ,\\
0 & \text{otherwise} \ ,
\end{cases}
\end{equation}
as our working weight function, where $\xi > 1$ is a suitable constant.

We consider the finite breakdown point for location $\mu$ and scale $\Sigma$ parameters for the elliptically symmetric parametric family with density of the form
\begin{equation*}
f_X(x) = |\Sigma|^{-1/2} h((x-\mu)^\top \Sigma^{-1} (x-\mu))
\end{equation*}
where $h \ge 0$ such that $f_X$ is a density, and we also assume that the function $h^\ast(y) = h'(y)/h(y)$ is such that $\sup_{y \ge 0} |h^\ast(y)| < \infty$. This implies that the solutions of the weighted likelihood estimating equations can be written in the following form
\begin{align} \label{equ:formoffunctional}
\sum_{i=1}^n W(X_i; \mu, \Sigma) (X_i - \mu) & = 0 \ ,\\
-\frac{2}{n} \sum_{i=1}^n W(X_i; \mu, \Sigma) (X_i - \mu)(X_i - \mu)^\top & = \Sigma \ ,
\end{align}
where
\begin{equation}
W(x; \mu, \Sigma) = w^\ast(\tau_n(x;\mu,\Sigma)) h^\ast((x - \mu)^\top \Sigma^{-1}(x - \mu)) \ ,
\end{equation}
is a weight function obtained as the product of our working weight function $w^\ast$ with the function $h^\ast$, which is specific to the assumed elliptically symmetric parametric family.

\begin{remark}
For multivariate normal distribution, $h^\ast(y) = -1/2$ while for the $t$ distribution on $\mathbb{R}^q$ with a known degrees of freedom $v$, we have
\begin{equation*}
h^\ast(y) = - \frac{q+v}{2 (y + v)} \ .
\end{equation*}
Hence, both of them satisfy the boundness of $h^\ast$ function.
\end{remark}

\subsection{Finite Sample Location Breakdown Point}
\label{sec:locationbreakdown}

Given two samples $x^{(n)} = (x_1, \ldots, x_n)$ and $y^{(m)} = (y_1, \ldots, y_m)$ we consider the functionals $T(x^{(n)})$ and $T(x^{(n)} \cup y^{(m)})$, and we assume that the sample $x^{(n)}$ is such that:
\begin{enumerate}[label=\textbf{(B\arabic*)}]
\item The observations $x_1, \ldots, x_n \in \mathbb{R}^q$ are in general position, i.e., no hyperplane contains more than $q$ points. \label{ass:b:1}
\item $\max\limits_{i=1,\ldots,n} \tau_n(x_i;\theta) \le K < \infty$ for all fixed $\theta$ with $\Vert \theta \Vert < \infty$. \label{ass:b:2}
\item $n > 2 \times q$. \label{ass:b:3}
\end{enumerate}
We also assume that $w(\tau) \ge c$ for some fixed constant $0 < c < 1$ and for all $\tau \in [-1,\infty)$.
Notice that, under these assumptions $\Vert x_i \Vert < \infty$ for all $i=1,\dots,n$. Because of the form of the estimating equation assumed in \eqref{equ:formoffunctional}, we also have that $\Vert T(x^{(n)})\Vert < \infty$ whenever $\Vert x_i \Vert < \infty$ for all $i=1,\dots,n$.
The additive finite sample breakdown point for the location parameter $T$ is given by
\begin{equation}
\varepsilon^\ast = \inf_{m}\left\{ \frac{m}{n+m}: \Vert T(x^{(n)} \cup y^{(m)}) - T(x^{(n)}) \Vert = \infty \right\}
\end{equation}
hence, breakdown occurs when  $\Vert T(x^{(n)} \cup y^{(m)}) \Vert = \infty$. For this to be the case we need to choose $y_j \in y^{(m)}$ ($j=1,\ldots,m$) in a sequence of observations $\{y_{j,k}\}_{k=1}^\infty$ such that $\Vert y_{j,k} \Vert \stackrel{k \rightarrow \infty} \longrightarrow \infty$. Regarding our location functional, for each of these sequences, we observe that
\begin{equation}
\lim_{k \rightarrow \infty} \tau_{n+m}(y_{j,k}; \theta) = \lim_{k \rightarrow \infty} \frac{d_{n+m}(y_{j,k};\hat{G}_{n+m})}{d(y_{j,k};F_\theta)} - 1 = +\infty
\end{equation}
since $d_{n+m}(y_{j,k};\hat{G}_{n+m}) \ge 1/{n+m}$ and $\lim_{k \rightarrow \infty} d(y_{j,k};F_\theta) = 0$. Consider the case $m < n$, then by the definition
\begin{equation}
\tau^\ast_\theta = \operatorname{median}(\tau_{n+m}(x_1;\theta), \ldots, \tau_{n+m}(x_n;\theta), \tau_{n+m}(y_{1,k};\theta), \ldots, \tau_{n+m}(y_{m,k};\theta)),
\end{equation}
would be finite for all $k$. 
Hence, there exists a $k_0$ such that, for all $k > k_0$, $\min_{j=1,\ldots,m} \tau_{n+m}(y_{j,k};\theta) > \tau^\ast_\theta + \xi$, where $\xi$ is as in~\eqref{equ:wstar}, implying that, for all $k > k_0$, $w^\ast(\tau(y_{j,k};\theta)) = 0$ for all $j=1, \ldots, m$. 
The weighted likelihood estimating equations reduce to
\begin{equation}
\frac{1}{n+m} \sum_{i=1}^n w^\ast(\tau_{n+m}(x_i;\theta)) u(x_i;\theta) = 0,
\end{equation}
but at least half of these weights are ensured to be strictly positive. 
Since all the observations involved are finite, the solution of the estimating equations is finite as well, and no breakdown occurs. 
On the other hand, let $j^\ast = \arg\min_{j=1,\ldots,m}\tau_{n+m}(y_{j,k};\theta)$, if $m > n$ then there exists a $k_0$ such that for all $k > k_0$, $\tau^\ast_\theta \ge \tau_{n+m}(y_{j^\ast,k};\theta)$, which implies that $\lim_{k \rightarrow \infty} \tau^\ast_\theta = \infty$, and that at least the observation $y_{j^\ast,k}$ will have a strictly positive weight. Hence, breakdown occurs.
For the case $m = n$ the behavior of the functional $T$ depends on the definition of the median for an even number of observations, and  thus we will not this case in detail. In any case, we have proved the following result.

\begin{theorem}
Under assumptions \ref{ass:b:1}--\ref{ass:b:3}, the additive finite sample location breakdown point of our estimator is no smaller than $0.5$.
\end{theorem}

\subsection{Finite Sample Scatter Breakdown Point}
\label{scatterbreakdown}

Let now $V$ be a scatter functional and, as before, we consider the two samples $x^{(n)}$ and $y^{(m)}$. Let $\lambda^+(m)$ and $\lambda^-(m)$ be the largest and smallest eigenvalues of $V(x^{(n)} \cup y^{(m)})$ respectively, then breakdown cannot occur if there exists two finite positive constants $a$ and $b$, such that $0 < a \le \lambda^-(m) \le \lambda^+(m) \le b < \infty$. 
The additive finite sample breakdown point for a scatter matrix $V$ is given by
\begin{equation}
\varepsilon^\ast = \inf_{m}\left\{ \frac{m}{n+m}: \frac{1}{\lambda^-(m)} + \lambda^+(m) = \infty \right\} \ .
\end{equation}
If the largest eigenvalue is unbounded, the explosion finite sample breakdown (EFSB) occurs, whereas if the smallest eigenvalue is zero, the implosion finite sample breakdown (IFSB) occurs. We will now examine the two cases in detail.

Denote by $v_{jk}$ be the entries of the matrix $V(x^{(n)} \cup y^{(m)})$, from the total variance equality we have $\sum_{j=1}^q v_{jj} = \sum_{j=1}^q \lambda_j$, so that EFSB occurs if and only if at least one of the variances is unbounded. However, the form of the estimating equations assumed in \eqref{equ:formoffunctional} implies $\lambda^+(m) < \infty$ whenever the involved observations are bounded. Following the same reasoning as for the location finite sample breakdown point, we can conclude that the EFSB point is no smaller than $0.5$.

For the IFSB point, consider the subspace generated by $q$ observations from $x^{(n)}$, without loss of generality, take the first $q$ observations. Place all the $y^{(m)}$ in this subspace.
Let $u$ be the unit vector orthogonal to this subspace, i.e., $u^\top y_j = 0$ for all $j=1,\ldots, m$, $u^\top x_i = 0$ for all $i=1,\ldots, q$, and $u^\top x_i \neq 0$ for all $i=(q+1),\ldots,n$. The IFSB occurs if more than half of the points are projected to $0$, which happens only if $m > n + 2 q$. This implies that the IFSB point of our scatter functional is no smaller than $0.5 - q/n$.
Notice also that the maximum possible finite sample breakdown point of any equivariant scatter estimate cannot be larger than this quantity \citep[][Theorem 6]{davies1987}.

\section{Examples}
\label{sec:examples}

We apply our method to two data sets, a multivariate and a functional one, and compare it with usual estimation techniques.

\subsection{Vowel Recognition Data}
\label{secsub:vowel}

We consider the data set \texttt{pb52} available in the \texttt{R} package \texttt{phonTools} \citet{barreda2015} which contains the Vowel Recognition Data considered in \citet{peterson1952}, see also \citet{boersma2012}. A multivariate normal model is assumed. In the first example a bivariate data set is illustrated, where we consider the vowels ``u'' (close back rounded vowel) and ``\ae'' (near-open front unrounded vowel, ``\{'' as x-sampa symbol) with a sample of size $304$ equally divided for each vowel and the log transformed F1 and F2 frequencies measured in Hz. Our procedure use the following settings $\alpha=0.5$, $a=0.05$, $c=200$, and uses $500$ subsamples of size $6$ as starting values for finding the roots. We also consider the Maximum Likelihood (MLE), the Minimum Covariance Determinant (MCD), the Minimum Volume Ellipsoid (MVE) and the S-Estimates (S) as implemented in the \texttt{R} \citep{cran} package \texttt{rrcov} by \citet{todorov2009}, the last three procedures were used with the exhaustive subsampling explorations.  Figure \ref{fig:x1} in the left panel reports the ellipsoids generated by the three estimates. While the first root coincides with the MLE, the other two nicely identify the two subgroups. This is not the case for all the other investigated methods, see Figure \ref{fig:x1} right panel, where the estimates are approximately all coincident with the MLE.

In the second example a tri-variate data set is used considering vowels ``u'' (close back rounded vowel) and ``\textrhookrevepsilon'' (open-mid central unrounded vowel, ``3'' or ``$3^\backprime$'' as x-sampa symbol) with a sample of size $304$ equally divided for each vowel and the log transformed F1, F2 and F3 frequencies measured in Hz. Figure (\ref{fig:x3}) shows the results, with the following setting for the roots search: $\alpha=0.25$, $a=0.1$, $c=30$, and $1000$ subsamples of size $6$ as starting values.
For the classical robust procedures only MCD is able to somehow recover the structure of the observations for the vowels ``\textrhookrevepsilon'', while the other behave like the MLE. Our procedure correctly finds the two substructure.
Observe that using the same setting of the first example we find only the first two roots.

\begin{figure}
\centering
\includegraphics[width=0.31\textwidth]{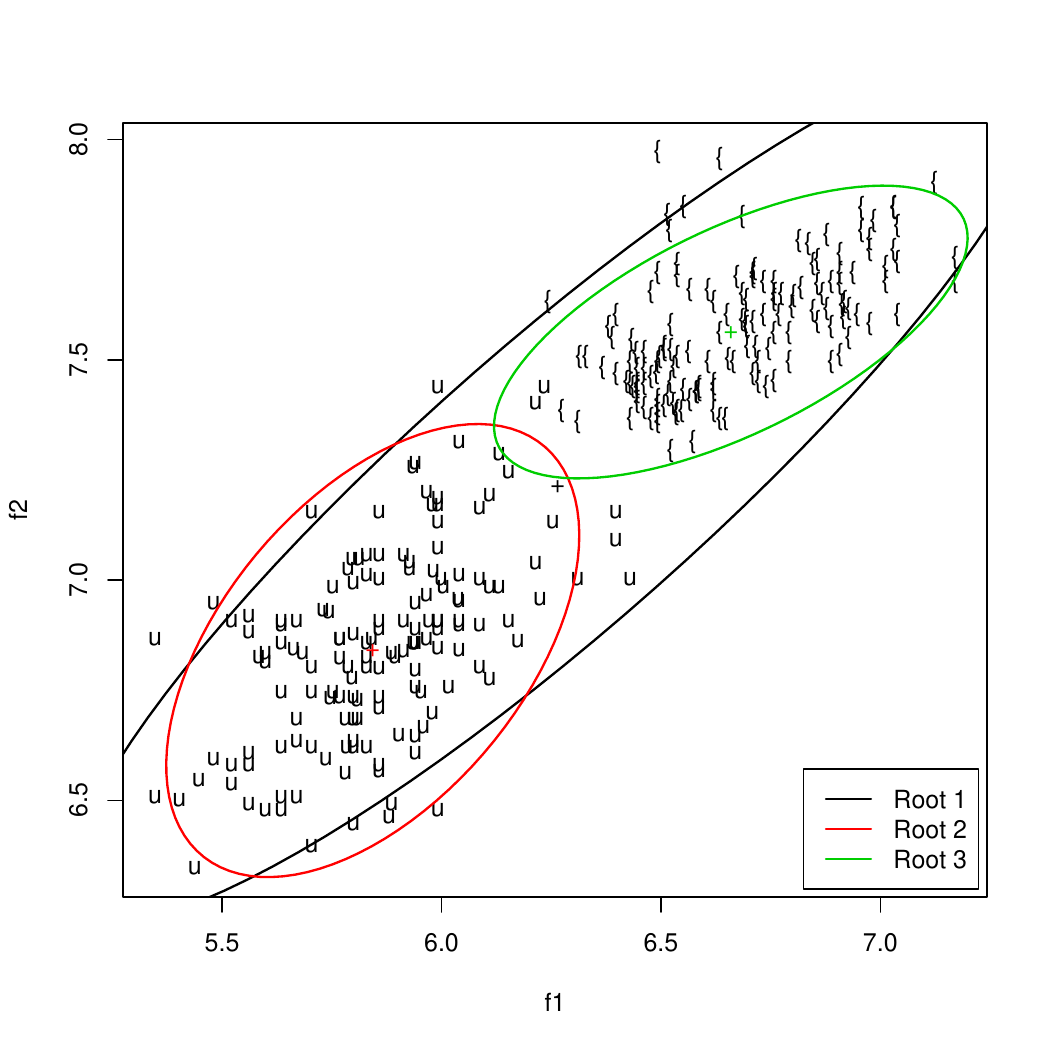}
\includegraphics[width=0.31\textwidth]{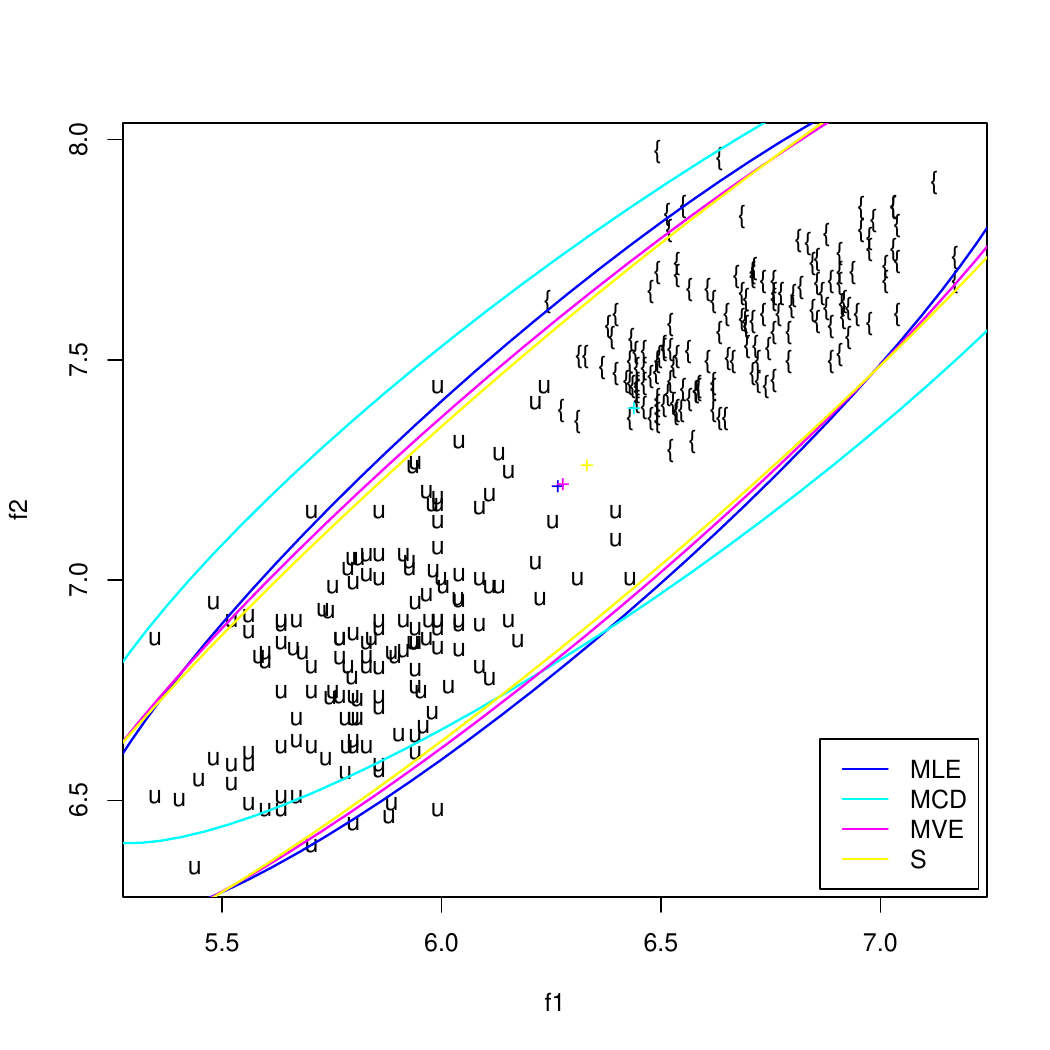}
\caption{Vowel Recognition Data. Bivariate example, vowels ``u'' and ``\ae''. Left: the three roots of the proposed method. Right: estimates provided by MLE and robust procedures. Ellipses are $95\%$ confidence regions.}
\label{fig:x1}
\end{figure}

\begin{figure}
\centering
\includegraphics[width=0.31\textwidth]{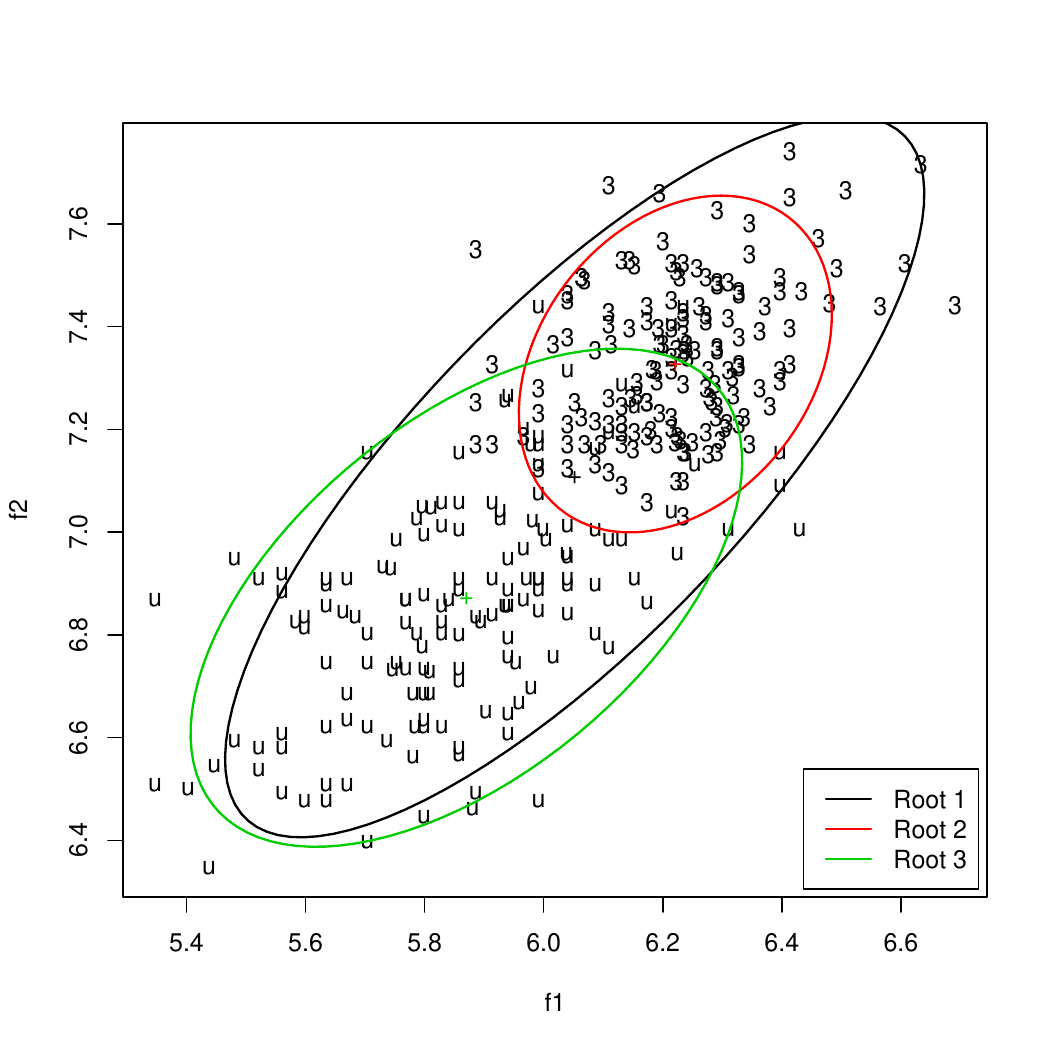}
\includegraphics[width=0.31\textwidth]{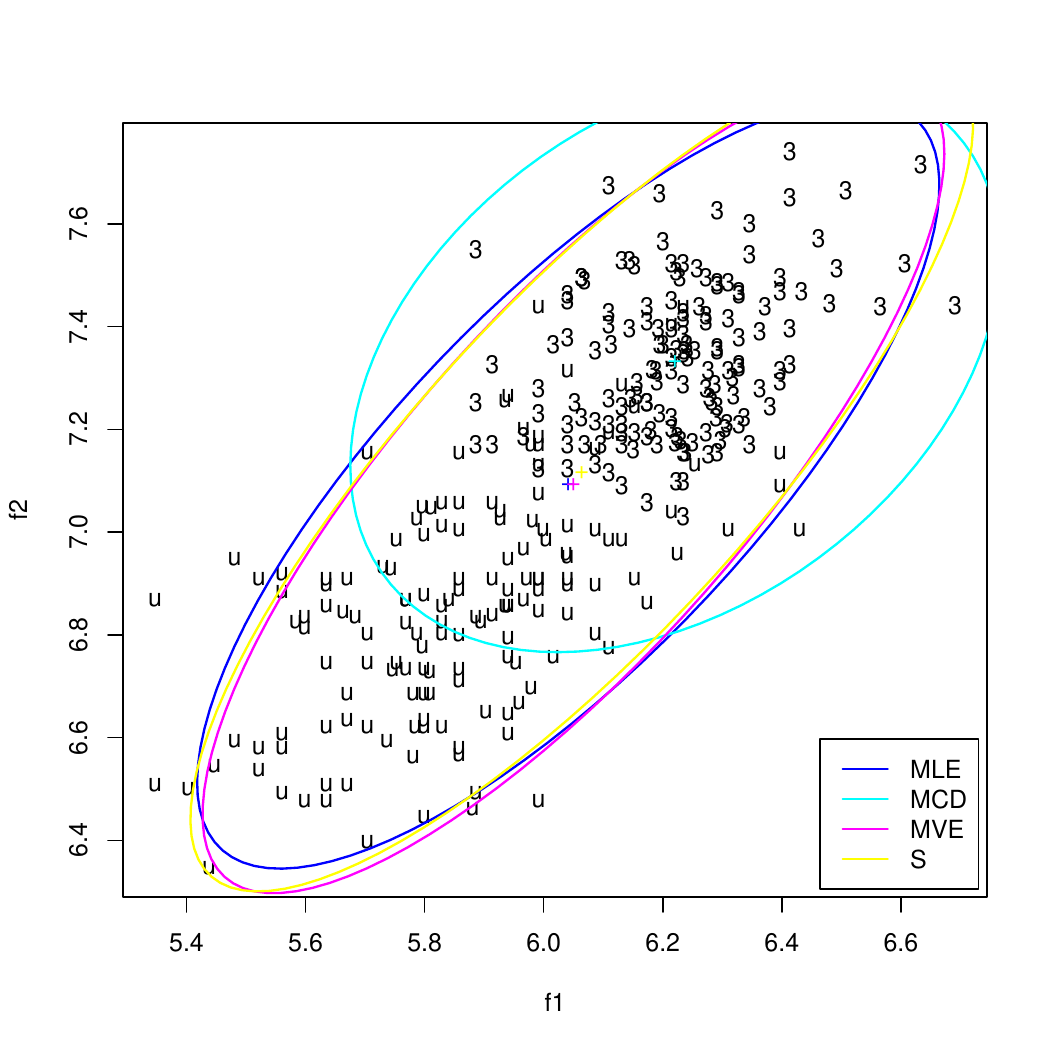} \\
\includegraphics[width=0.31\textwidth]{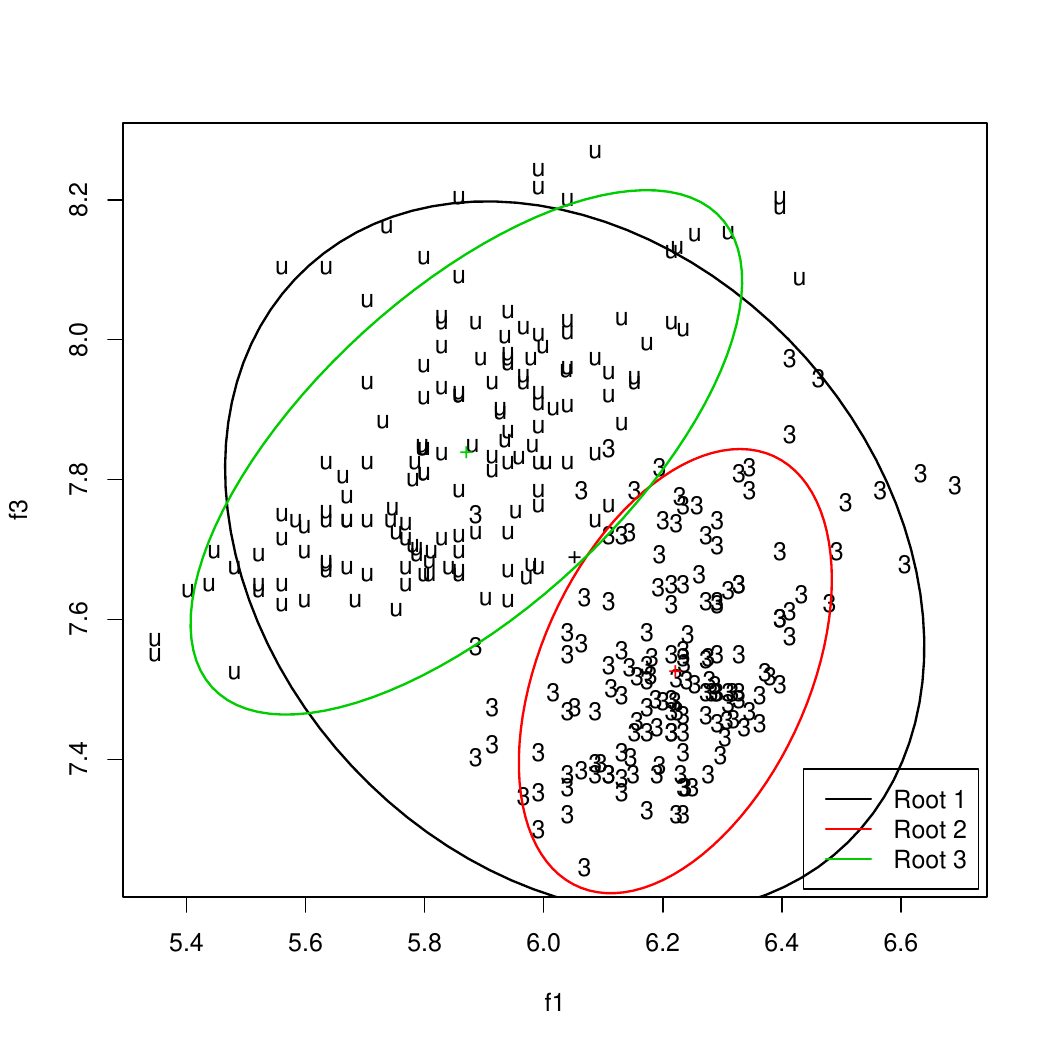}
\includegraphics[width=0.31\textwidth]{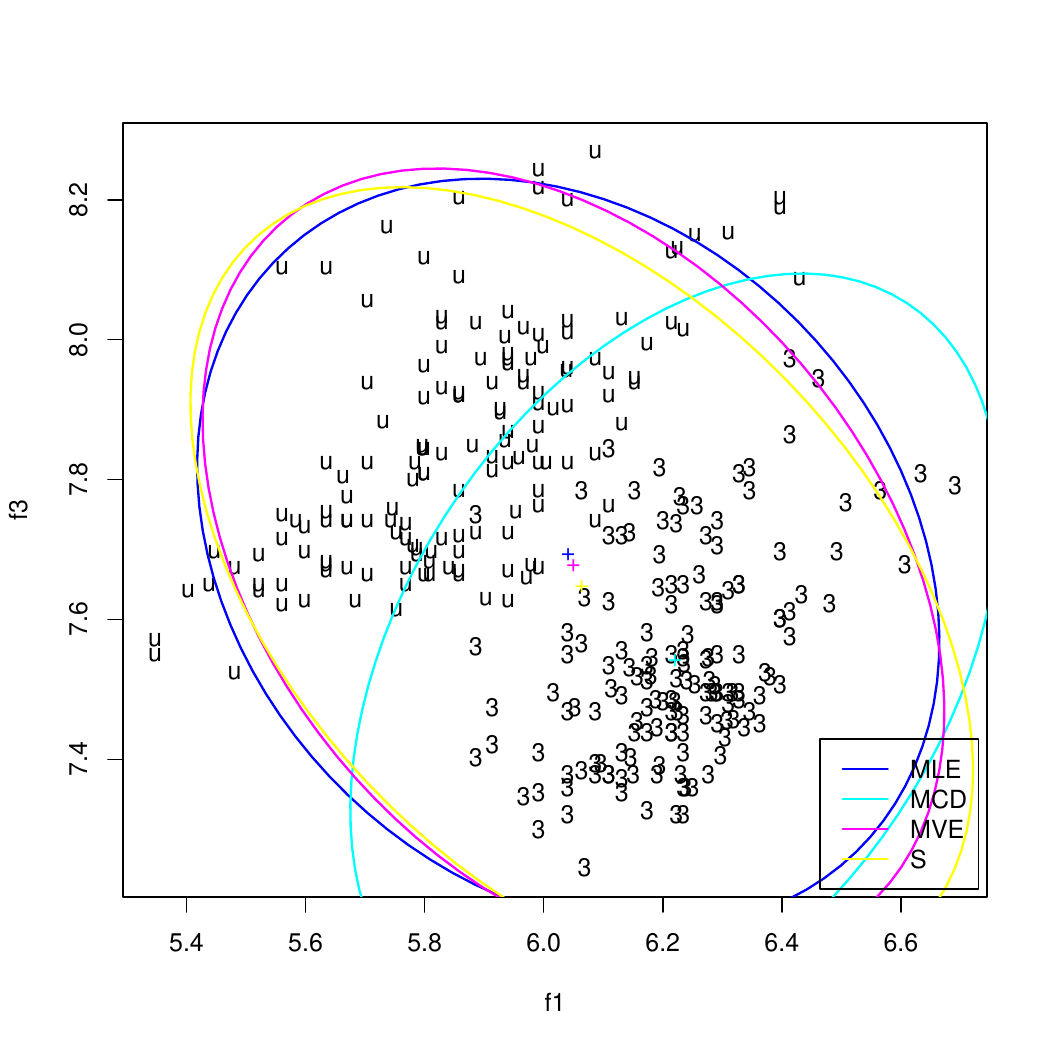} \\
\includegraphics[width=0.31\textwidth]{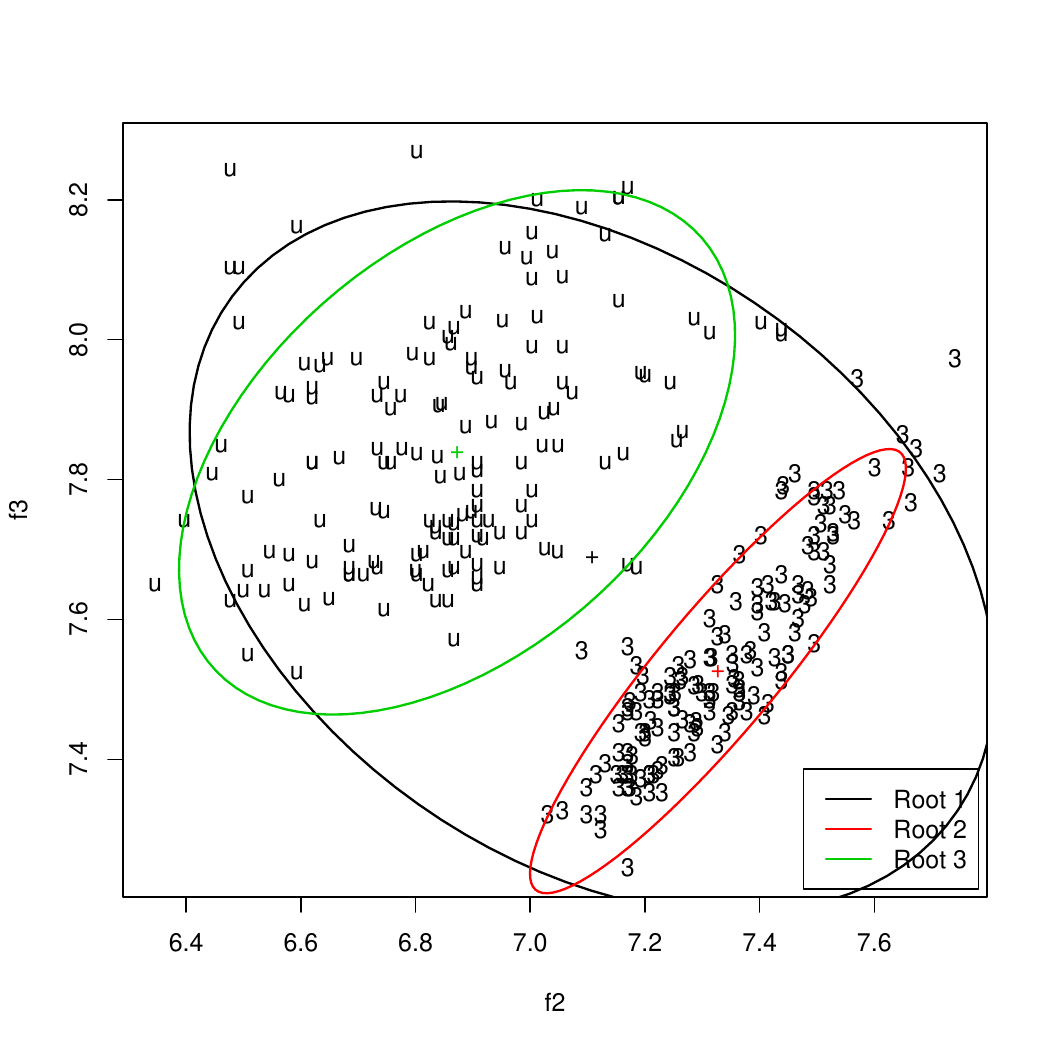}
\includegraphics[width=0.31\textwidth]{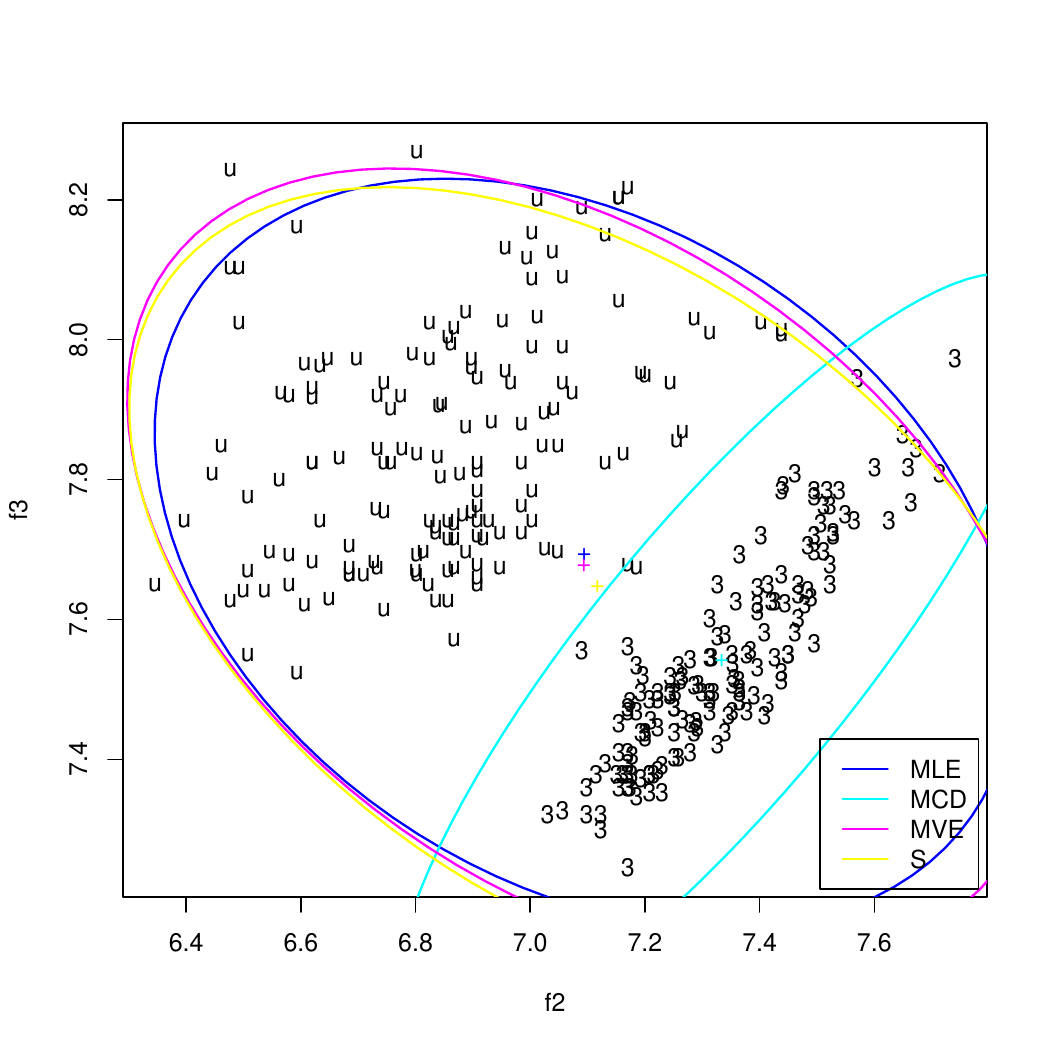} \\
\caption{Vowel Recognition Data. Trivariate example, vowels ``u'' and ``\textrhookrevepsilon''. Left: the three roots of the proposed method. Right: estimates provided by MLE and robust procedures. First row: $f_1$ vs $f_2$, second row: $f_1$ vs $f_3$ and third row: $f_2$ vs $f_3$. Ellipses are $95\%$ confidence regions.}
\label{fig:x3}
\end{figure}

\subsection{Wind Speed Data}
\label{secsub:wind}

We consider a data set obtained from the Col De la Roa (Italy) meteorological station available at \texttt{intra.tesaf.unipd.it/Sanvito}. We concentrate on the wind speed recorded regularly every $15$ minutes in the years $2001$-$2004$. After removing $42$ incomplete days we obtain $1420$ daily time series of length $4 \times 24 = 96$. The statistical analysis based on the modified local half-region depth performed in \citet{agostinelli2018} identified $3$ groups: first group ($536$ days) corresponds to the cold season (circular mean around December), the second group ($415$ days) corresponds to the hot season (circular mean around June), while the third group ($469$ days), given its nature, is spread over the whole year, with some predominance in the hot season. For the present analysis we consider the second group and a selection of the third group for a total of $548$ days that we split into two groups with size $413$ and $135$ respectively. Then, we consider $10$ data sets obtained by adding to the first group recursively $15$ observations from the other so that we should have approximately a level of contamination ranging form $0$ to approximately $25\%$ ($0.00$, $0.035$, $0.068$, $0.098$, $0.127$, $0.154$, $0.179$, $0.203$, $0.225$, $0.246$). For each data set we run our procedure using the modified half-region depth \citep{lopez-pintado2011} and the maximum likelihood estimator on a Gaussian process model. Figure \ref{fig:wind-data} shows the data set with the observations form the first group in green-blue and those from the second group in purple. The superimposed thicker lines represent the estimate mean curves from our procedure (green-blue) and maximum likelihood (purple) using observations from both groups. 
In fact, as it is shown in Figure \ref{fig:wind-l2} and also in Figure \ref{sup:fig:wind-mean} of the Appendix \ref{sup:sec:wind}, as we add observations from the second group to the first one, the maximum likelihood estimate of the mean curve is affected and lead to an increased mean at every hour of the day, while our procedure demonstrates strong stability. Similarly, Figure \ref{fig:wind-e1p} shows the behavior of the largest (left panel) and smallest (right panel) eigenvalues of the estimated covariance matrix by both methods, and we notice again how much the maximum likelihood estimate is affected. The estimated correlation structure (see Figures \ref{sup:fig:wind-cor-1} and \ref{sup:fig:wind-cor-2} of the Appendix \ref{sup:sec:wind}) is also enhanced for maximum likelihood, while remaining stable for our proposal.

\begin{figure}
\centering
\includegraphics[width=0.9\textwidth]{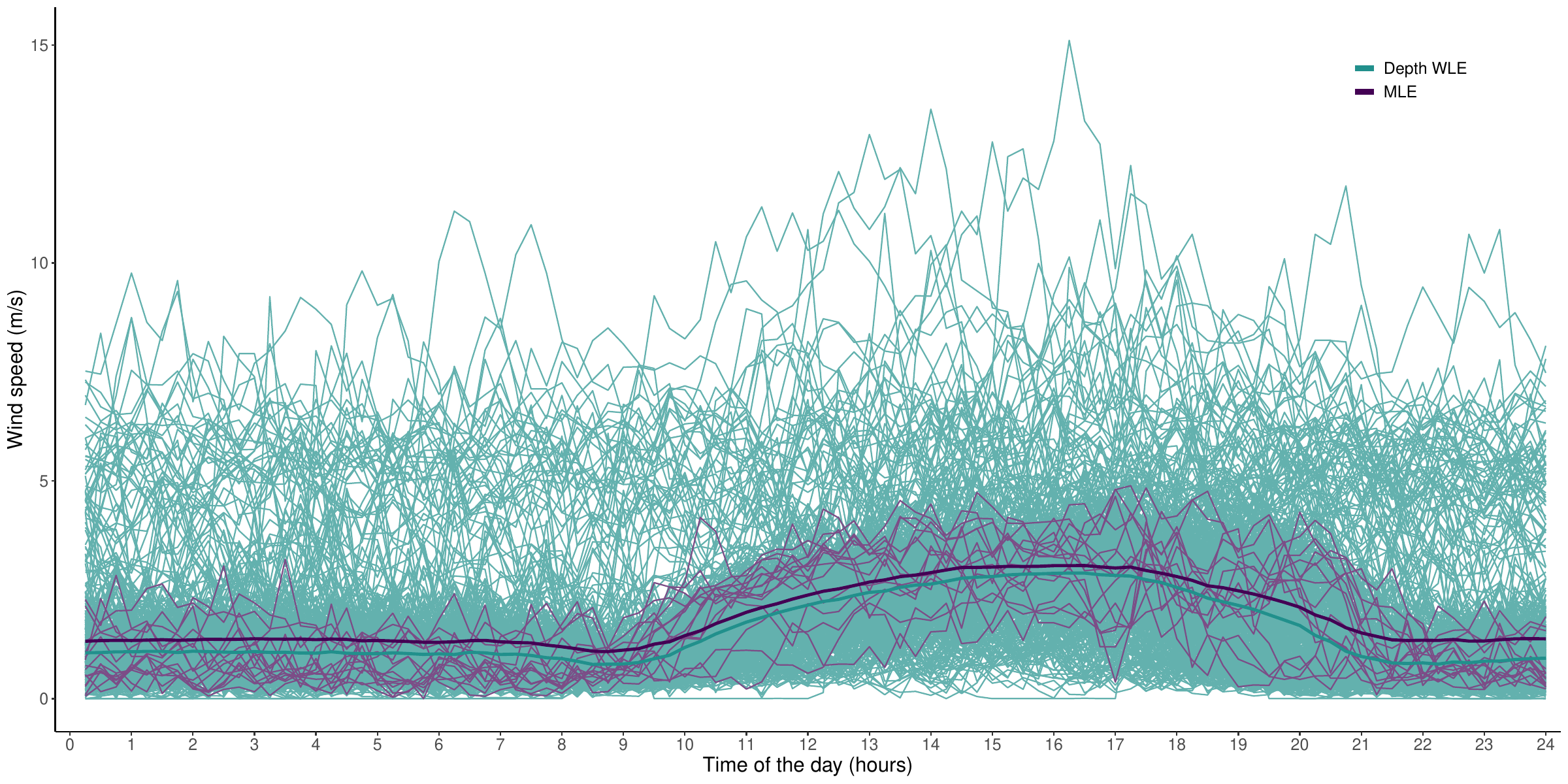}
\caption{Wind Data. First group of data are in green-blue, while the second set of observations are in purple. Estimate mean curves (thicker lines) provided by the MLE and the weighted likelihood estimates based on depth procedure are superimposed.}
\label{fig:wind-data}
\end{figure}

\begin{figure}
\centering
\includegraphics[width=0.9\textwidth]{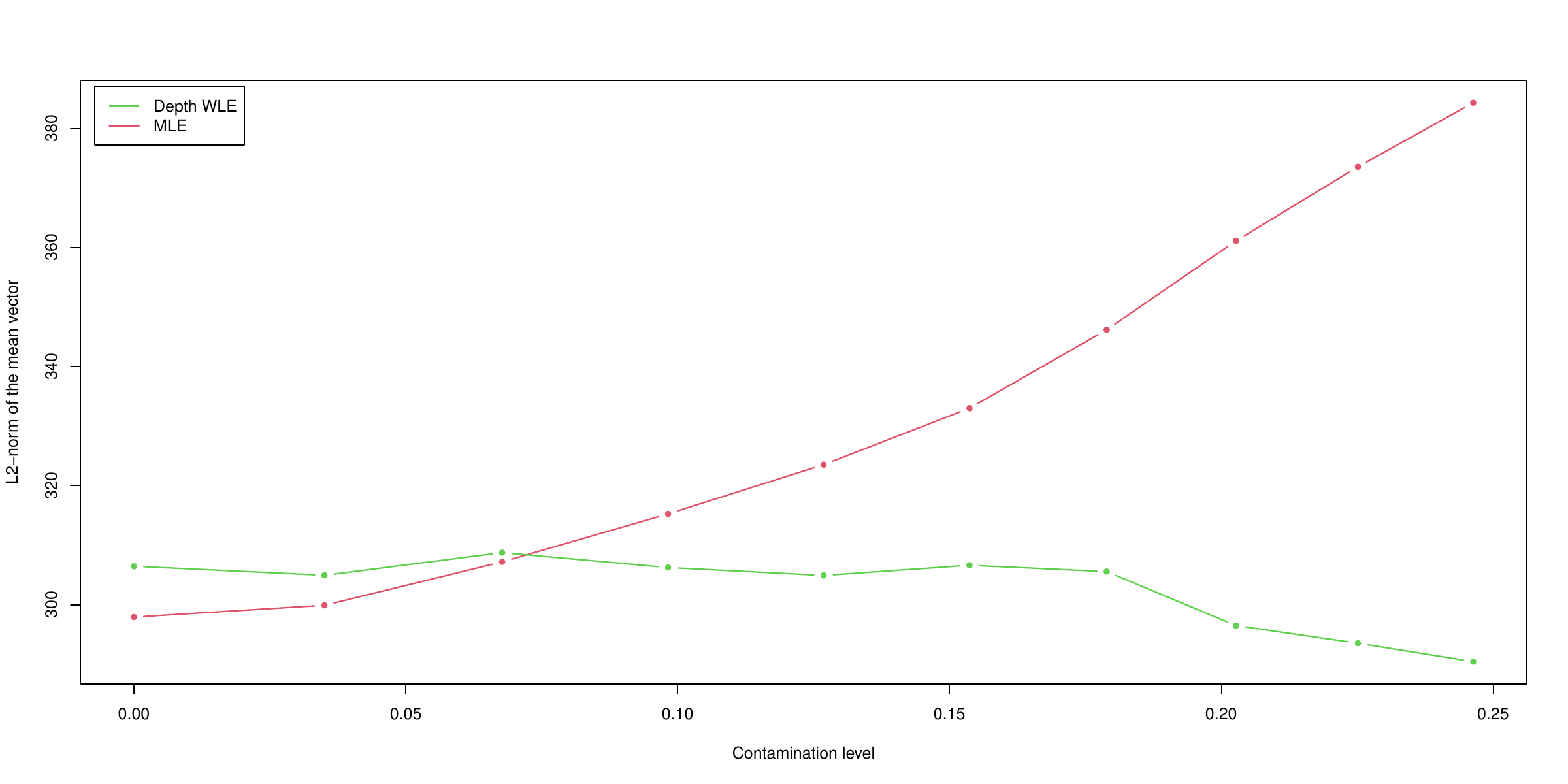}
\caption{Wind Data. $L_2$-norm of the estimate mean curves provided by the MLE (red) and the weighted likelihood estimates based on depth procedure (green). Contamination level ranges from $0$ to approximately $25\%$.}
\label{fig:wind-l2}
\end{figure}

\begin{figure}
\centering
\includegraphics[width=0.45\textwidth]{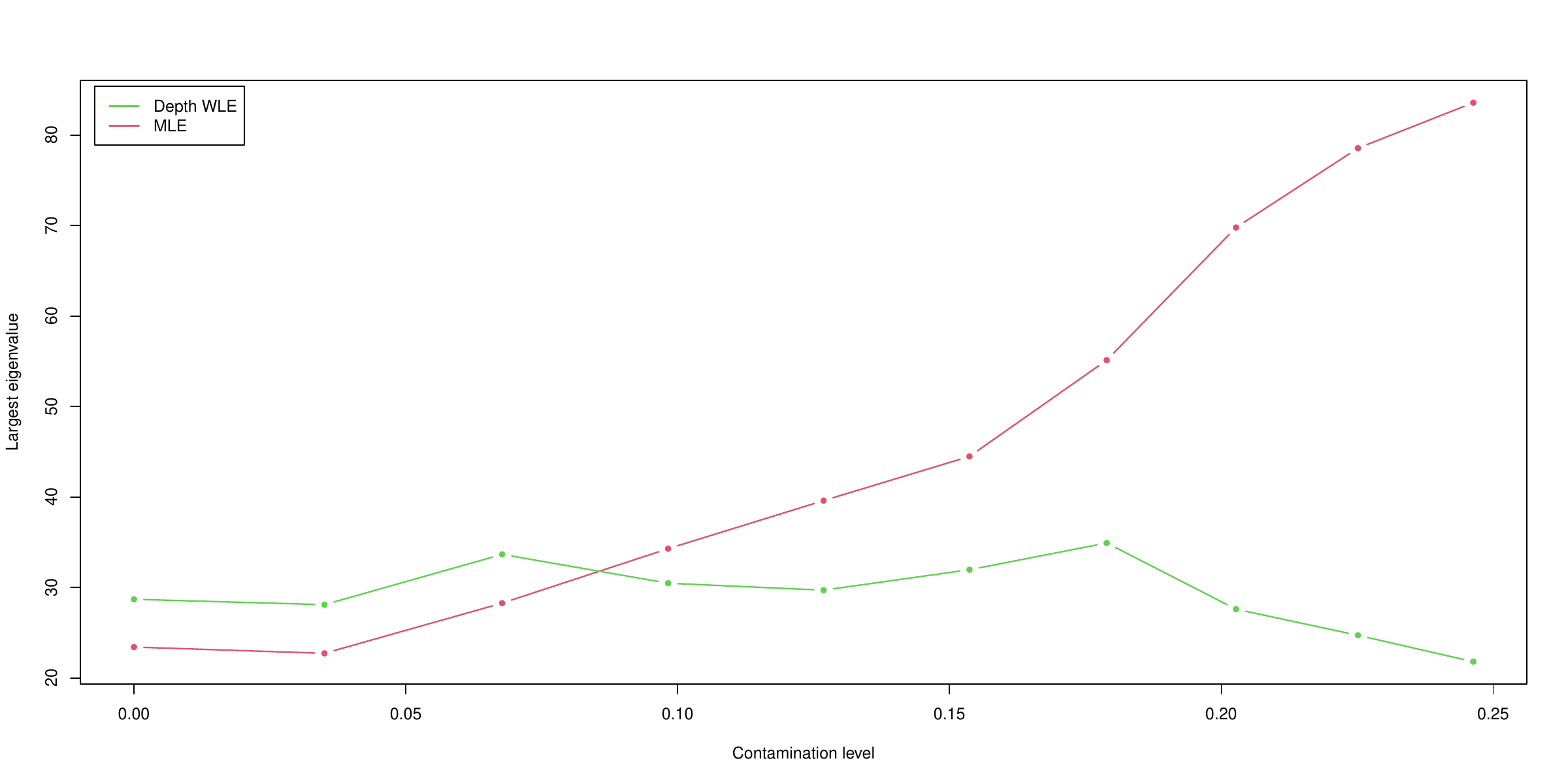}
\includegraphics[width=0.45\textwidth]{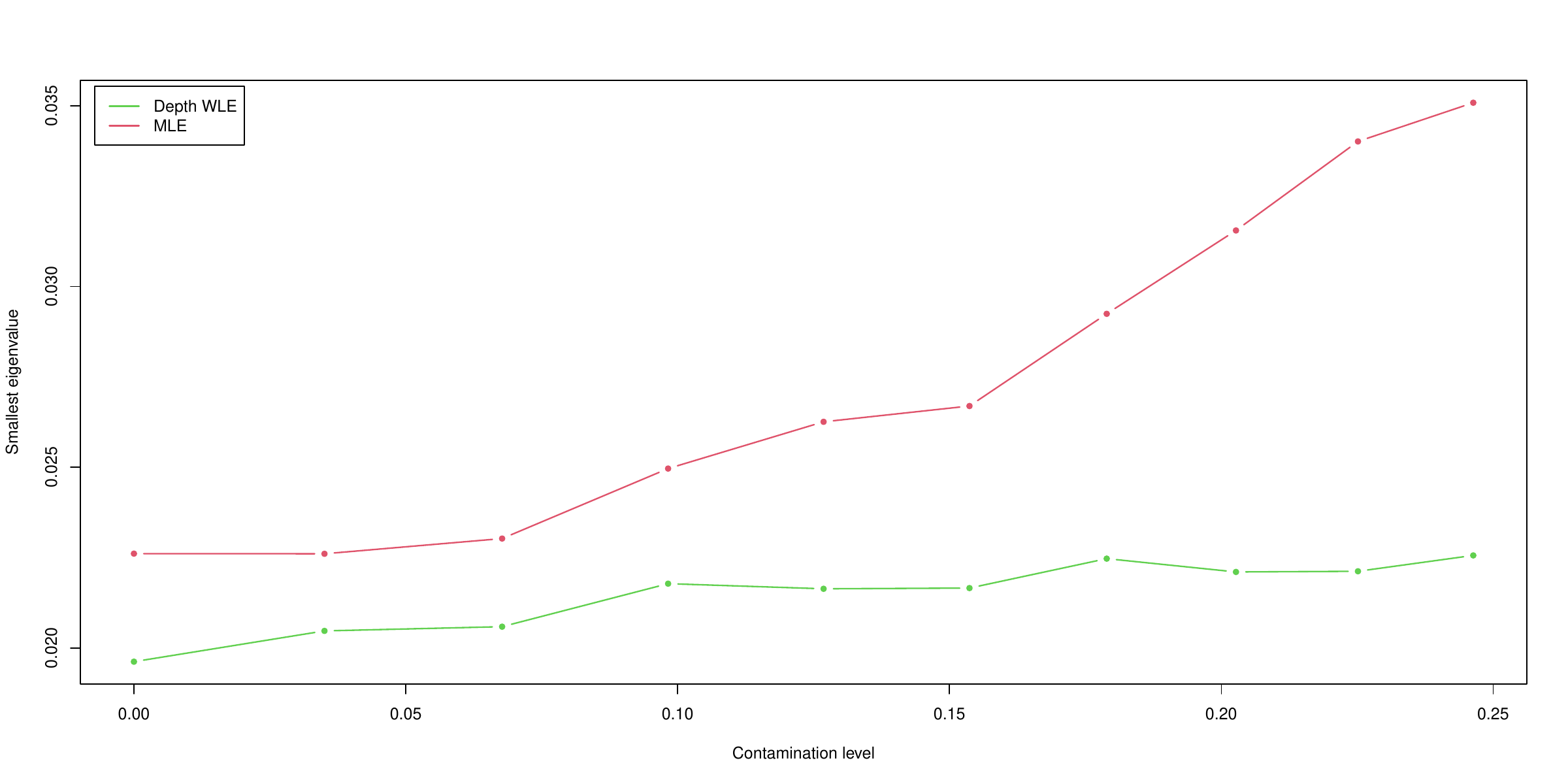}
\caption{Wind Data. Largest eigenvalue (left) and smallest eigenvalue (right) of the estimate covariance matrix provided by the MLE (red) and the weighted likelihood estimates based on depth procedure (green). Contamination level ranges from $0$ to approximately $25\%$.}
\label{fig:wind-e1p}
\end{figure}

\section{Monte Carlo Experiment}
\label{sec:montecarlo}

We run a Monte Carlo experiment with the following factors: dimensions $p=1, 2, 5, 10, 20$, sample sizes is $n=s (p (p+1)/2+p)$ with $s=2, 5, 10$. The data are simulated from a multivariate standard normal. We consider the levels of contamination $\varepsilon = 0, 0.05, 0.1, 0.2, 0.3, 0.4, 0.45$ where the contaminated data are simulated from a multivariate normal model $N_p(\boldsymbol{\mu}, \sigma^2 I)$, where $\boldsymbol{\mu} = (\mu, \ldots, \mu)^\top$, $\mu = 0, 1, 2, 3, 4, 5, 6, 7, 8, 9, 10$ and $\sigma = 0.001, 1, 2, 3, 4, 5$. For each combination of these factors we run $100$ replications. The constrained M-Estimate (covMest) of multivariate location and scatter based on the translated biweight function using a high breakdown point initial estimate \citep{woodruff1994,rocke1996a,rocke1996b} as implemented in function \texttt{covMest} of the \texttt{R} package \texttt{rrcov} \citep{todorov2009} is computed for comparison reasons. For the new procedure we consider the following settings: $\alpha= 1/4, 1/2, 3/4, 1$ and the following weight function $w(\tau, \delta_1, \delta_2, \gamma) = (h(\tau, \delta_1, \delta_2) + \gamma)/((1 + \gamma))$ where
\begin{equation*}
h(\tau, \delta_1, \delta_2) =
\begin{cases}
0 & \text{ if } \tau > \delta_2 \\
\frac{\delta_2 - \tau}{\delta_2 - \delta_1} & \text{ if } \delta_1 < \tau \le \delta_2 \\
1 & \text{ otherwise }
\end{cases}
\end{equation*}
with $\delta_1 = 0.5,1,2$, $\delta_2 = 3,6,9$ and $\gamma = 0.1,0.2,0.3$. Finally, the weight $w^\ast(\tau)$ as defined in \eqref{equ:wstar} is used with $\delta = 1,5,10,20$. To understand the impact on the performance of the new procedure of all the parameters used to define the weights, we initially run a simulation where the Iterative Re-Weighted Least Square (IRWLS) algorithm of new procedure is started from the true values. From these simulations (see Section \ref{sup:sec:monte} in the Appendix) it is evident that all these parameters have a very little influence on the performance of the proposed methods. 
Based on these results we decide to fix the weights parameters by minimizing the $95\%$ quantile of MSE/Kullback--Leibler divergence. These ``optimal'' values are reported in Table \ref{tab:monte:optimal}.

\begin{table}
\centering
\begin{tabular}{lrrrr}
\hline
$\alpha$ & 0.25 & 0.5 & 0.75 & 1 \\
\hline
$\gamma$ & 0.1 & 0.3 & 0.3 & 0.3 \\
$\delta_1$ & 2 & 2 & 2 & 2 \\
$\delta_2$ & 3 & 9 & 9 & 9 \\
$\delta$ & 1 & 1 & 5 & 5 \\
\hline
\end{tabular}
\caption{Monte Carlo simulation. ``Optimal values'' for the weight function parameters for different values of $\alpha$.}
\label{tab:monte:optimal}
\end{table}

Tables \ref{tab:MSE:max} and \ref{tab:DIV:max} report the estimated maximum Mean Square Error and estimated maximum Kullback-Leibler divergence for the proposed method with $\alpha=0.5$ starting at the true values and \texttt{CovMest} for contamination levels $0.05$, $0.1$ and $0.2$. Complete results are available in the Appendix \ref{sup:sec:monte}. These results show that the weighted likelihood estimating equations always have a ``robust'' root if the algorithm is appropriately initialized. The performance is extremely good even under high contamination with some deterioration for small value of $\alpha$ ($0.25$) and large value ($1$); in any case the method outperform the \texttt{CovMest} in every setting.

Since the true values are not available in a realistic setting, we then implement two strategies for obtaining data-driven initial values.
Firstly, we use a classic subsampling procedure with $B=500$ samples of size $p+p(p + 1)/2+1$, where $p$ is the sample size factor mentioned above. 
Alternatively, we initiate the proposed algorithm deterministically from the deepest observation for the location, and from the sample covariance matrix of the 50\% deepest observations, for the shape. 
We just mention some key insights, while detailed results are presented in the Appendix \ref{sup:sec:monte}. When the contamination level is moderate to high $\varepsilon \geq 0.2$, it might be difficult to obtain subsamples with only ``clean'' observations. 
In this setting, instead of increasing the number of subsamples, we found that starting from the deepest points achieves similar or better results (except for $p=1$, where the subsampling outdoes the depth initialization). 
As expected, the distance of the outliers from the data affects the weighted likelihood estimators. We observed that when the contamination location $\mu$ is near zero and the contamination scale $\sigma$ is small, our procedure usually identifies the robust root independently of the starting strategy, with subsampling performing better for small values of $p$ and $s$. As the contamination is moved away from zero it becomes more difficult to retrieve the robust root, especially when the variability of the contamination is small ($\sigma = 0.001$).
In conclusion, for every setting we observe that the weighted likelihood estimating equations have always a ``robust'' root; however initialization of the algorithm might be still an issue in particular settings: moderate/high level of contamination, high dimension and small sample sizes.

\begin{table}[ht]
\centering
\begin{tabular}{lrrrrrrrrr}
  \hline
  $\varepsilon$ & \multicolumn{3}{c}{0.05} & \multicolumn{3}{c}{0.1} & \multicolumn{3}{c}{0.2} \\
  \hline  
  $s$ & 2 & 5 & 10 & 2 & 5 & 10 & 2 & 5 & 10 \\
  \hline
  & \multicolumn{9}{c}{$\alpha=0.5$} \\
  \hline
$p=1$ & 0.10 & 0.06 & 0.05 & 0.10 & 0.10 & 0.06 & 0.15 & 0.15 & 0.13 \\ 
  2 & 0.09 & 0.04 & 0.02 & 0.09 & 0.05 & 0.05 & 0.11 & 0.17 & 0.17 \\ 
  5 & 0.03 & 0.01 & 0.01 & 0.03 & 0.03 & 0.04 & 0.06 & 0.16 & 0.17 \\ 
 10 & 0.01 & 0.01 & 0.00 & 0.02 & 0.01 & 0.04 & 0.05 & 0.16 & 0.16 \\ 
 20 & 0.01 & 0.00 & 0.00 & 0.01 & 0.01 & 0.01 & 0.04 & 0.04 & 0.04 \\
  \hline
  & \multicolumn{9}{c}{\texttt{CovMest}} \\
  \hline
$p=1$ & 0.21 & 0.06 & 0.04 & 0.21 & 0.11 & 0.04 & 0.21 & 0.14 & 0.13 \\ 
  2 & 0.12 & 0.06 & 0.03 & 0.16 & 0.06 & 0.04 & 0.29 & 0.34 & 0.37 \\ 
  5 & 0.04 & 0.01 & 0.01 & 0.05 & 0.02 & 0.02 & 0.78 & 0.46 & 0.47 \\ 
 10 & 0.01 & 0.01 & 0.00 & 0.02 & 0.01 & 0.01 & 0.50 & 0.44 & 0.43 \\ 
 20 & 0.01 & 0.00 & 0.00 & 0.01 & 0.01 & 0.04 & 0.45 & 0.42 & 0.19 \\
  \hline  
\end{tabular}
\caption{Monte Carlo simulation. Estimated maximum Mean Square Error for the proposed method with $\alpha=0.5$, starting from the true values and \texttt{CovMest}, under contamination ($\varepsilon = 0.05, 0.1, 0.2$) for different number of variables $p$ and sample size factor $s$.}
\label{tab:MSE:max}
\end{table}

\begin{table}[ht]
\centering
\begin{tabular}{rrrrrrrrrr}
  \hline
  $\varepsilon$ & \multicolumn{3}{c}{0.05} & \multicolumn{3}{c}{0.1} & \multicolumn{3}{c}{0.2} \\
  \hline  
  $s$ & 2 & 5 & 10 & 2 & 5 & 10 & 2 & 5 & 10 \\
  \hline
  & \multicolumn{9}{c}{$\alpha=0.5$} \\
  \hline
$p=1$ & 0.70 & 0.11 & 0.04 & 0.70 & 0.19 & 0.10 & 1.17 & 0.43 & 0.29 \\ 
  2 & 0.41 & 0.15 & 0.06 & 0.48 & 0.15 & 0.17 & 0.83 & 0.37 & 0.42 \\ 
  5 & 0.43 & 0.19 & 0.33 & 0.48 & 0.39 & 0.79 & 0.75 & 1.82 & 1.78 \\ 
 10 & 0.51 & 0.25 & 0.17 & 0.72 & 0.43 & 2.10 & 1.28 & 4.50 & 4.46 \\ 
 20 & 0.74 & 0.47 & 0.38 & 1.29 & 1.00 & 0.91 & 2.60 & 2.28 & 2.17 \\ 
  \hline
  & \multicolumn{9}{c}{\texttt{CovMest}} \\
  \hline
$p=1$ & 0.06 & 0.02 & 0.01 & 0.06 & 0.03 & 0.02 & 0.15 & 0.09 & 0.09 \\ 
  2 & 1.81 & 0.45 & 0.18 & 1.86 & 0.45 & 0.25 & 2.35 & 1.86 & 1.91 \\ 
  5 & 1.27 & 0.27 & 0.15 & 1.80 & 0.39 & 0.25 & 14.57 & 6.63 & 6.49 \\ 
 10 & 0.74 & 0.34 & 0.22 & 1.09 & 0.60 & 0.48 & 17.67 & 15.53 & 15.00 \\ 
 20 & 0.94 & 0.58 & 0.47 & 1.66 & 1.25 & 6.17 & 35.28 & 33.40 & 13.87 \\  
\hline
\end{tabular} 
\caption{Monte Carlo simulation. Estimated maximum Kullback--Leibler divergence for the proposed method, starting from the true values, under contamination ($\varepsilon = 0.05, 0.1, 0.2$) for different values of $\alpha$, number of variables $p$ and sample size factor $s$.}
\label{tab:DIV:max}
\end{table}

\section{Conclusions}
\label{sec:conclusions}
We have outlined a new form of weighted likelihood estimating equations where weights are based on comparing statistical data depth of the sample with that of the model. This approach avoids the use of non-parametric density estimates which can lead to problems for multivariate data, while retaining the nice characteristics of the classical WLEE approach, that possess high efficiency at the model, affine equivariance, and robustness. We establish under regular conditions the asymptotic normality of the parameters for a broad class of models, and we establish that the finite breakdown point for location and scatter parameters in a symmetric elliptical model is $50\%$. In this last context an iterative algorithm is developed and applied to a multivariate data set, and to show the broad applicability of the method we perform a robust analysis for a functional data set. By Monte Carlo experiments we confirm the good performance of the procedure, however the algorithm is sensitive to the initial values and in particular cases (moderate/high contamination, high dimension and small sample size) is not easy to obtain the ``robust'' root. In this respect we explore the performance of two approaches, a classic one based on subsampling method (also known as bootstrap approach) and a deterministic one based on depth values. Both methods show criticallity which are discussed in the Appendix \ref{sup:sec:monte}. A future research should address this point.


\appendix

\section{Proofs}
\label{sup:sec:proofs}
To prove Theorem \ref{the:normalityresults} we first need a result about the rate of convergence of the scaled depth for $0 < \alpha \le 1/2$.
\begin{lemma}[Rate of convergence of scaled depth] \label{lemma:ratescaleddepth}
For $0 < \alpha \le 1/2$ and $0 < \delta \le 1/2$ we have 
\begin{equation*}
\sup_{x:\,\delta \le d(x; F_{\theta_0}) < 1/2}\,\left|\frac{d_n(x; \hat{G}_n) - d(x; F_{\theta_0})}{{d^{\alpha}(x; F_{\theta_0})}}\right| = O_p\left(\sqrt{\frac{\log(\delta^{-1})}{n}} + \frac{\log(\delta^{-1})}{\sqrt{\delta}n}\right) \ .
\end{equation*}
\end{lemma}

\begin{proof}[Proof of Lemma \ref{lemma:ratescaleddepth}]
For any measure $\mu$, the half-space depth $d(x; \mu)$ is defined as
\begin{equation*}
d(x; \mu) = \min_{H\in\mathcal{H}_x}\,\mu(H),
\end{equation*}
where $\mathcal{H}_x$ is the class of all half-space containing $x$. Define for $0 \le \delta_1 < \delta_2 \le 1/2$,
\begin{equation*}
T_n([\delta_1, \delta_2)) = \sup_{x:\,\delta_1 \le d(x; F_{\theta_0}) < \delta_2}\left|\frac{d_n(x;\hat{G}_n) - d(x; F_{\theta_0})}{d^{\alpha}(x; F_{\theta_0})}\right|.
\end{equation*}
Let $0 < \delta \le 1/2$ and consider $T_n([\delta, 1/2))$, where quantities $\alpha$ and $\theta_0$ are suppressed in the notation. To control $T_n([\delta,1/2))$, note that for all $\varepsilon > 0$,
\begin{equation*}
\left\{T_n([\delta, 1/2)) \ge \varepsilon\right\} = \bigcup_{j = 0}^{\infty}\left\{T_n([q_j\delta, q_{j+1}\delta)) \ge \varepsilon\right\},
\end{equation*} 
for any sequence $\{q_j\}$ satisfying $1 = q_0 < q_1 < \cdots$ and let $J = \min\{j \ge 0:\, q_{j+1}\delta > 1/2 \}$. Since,
\begin{equation*}
T_n([q_j\delta, q_{j+1}\delta)) \le \sup_{x:\,q_j\delta \le d(x; F_{\theta_0}) < q_{j+1}\delta}\left|d_n(x; \hat{G}_n) - d(x; F_{\theta_0})\right|\frac{1}{q_{j}^{\alpha}\delta^{\alpha}},
\end{equation*}
it follows that
\begin{equation*}
\left\{T_n([q_j\delta, q_{j+1}\delta)) \ge \varepsilon\right\}~\subseteq~ \left\{\sup_{x:\,q_j\delta \le d(x; F_{\theta_0}) < q_{j+1}\delta}\left|d_n(x; \hat{G}_n) - d(x; F_{\theta_0})\right| \ge q_j^{\alpha}\delta^{\alpha}\varepsilon\right\}.
\end{equation*}
Also, note that
\begin{align*}
\left|d_n(x; \hat{G}_n) - d(x; F_{\theta_0})\right| &= \left|\inf_{H\in\mathcal{H}_x}\hat{G}_n(H) - \inf_{H\in\mathcal{H}_x}F_{\theta_0}(H)\right|\\
&\le \sup_{H\in\mathcal{H}_x}\left|\hat{G}_n(H) - F_{\theta_0}(H)\right|,
\end{align*}
and so,
\begin{equation*}
\sup_{x:\,q_j\delta \le d(x; F_{\theta_0}) < q_{j+1}\delta}\left|d_n(x; \hat{G}_n) - d(x; F_{\theta_0})\right| \le \sup_{\substack{H \in \mathcal{H} \\ q_j\delta \le F_{\theta_0}(H) < q_{j+1}\delta}} \left|\hat{G}_n(H) - F_{\theta_0}(H)\right|.
\end{equation*}
where $\mathcal{H}$ is the set of all half-spaces in $\mathbb{R}^q$. Combining the relations above, we get
\begin{equation*}
\left\{T_n([\delta, 1)) \ge \varepsilon\right\} {{\subseteq}} \bigcup_{j = 0}^{J}\left\{\sup_{\substack{H \in \mathcal{H} \\ q_j\delta \le F_{\theta_0}(H) < q_{j+1}\delta}}\left|\hat{G}_n(H) - F_{\theta_0}(H)\right| \ge q_j^{\alpha}\delta^{\alpha}\varepsilon\right\}.
\end{equation*}
For notational convenience, let for $j \ge 0$,
\begin{equation*}
\mathcal{A}_j := \left\{H \in \mathcal{H}: \, q_j \delta \le F_{\theta_0}(H) \le q_{j+1} \delta \right\} \ .
\end{equation*}
By the union bound, we get
\begin{equation*}
\Pr\left(T_n([\delta, 1/2)) \ge \varepsilon\right) {{\le}} \sum_{j = 0}^{J}\Pr\left(\sup_{H\in\mathcal{A}_j}\left|\hat{G}_n(H) - F_{\theta_0}(H)\right| \ge q_j^{\alpha}\delta^{\alpha}\varepsilon\right) \ .
\end{equation*}
To bound the summands on the right-hand side, we apply Talagrand's inequality. 
By Theorem 3.3.16 of \citet{gine2016}, we have for all $t \ge 0$,
\begin{equation*}
\Pr\left(\sup_{H\in\mathcal{A}_j}\sqrt{n}\left|\hat{G}_n(H) - F_{\theta_0}(H)\right| \ge 3E_j + \sqrt{2\sigma_j^2t} + \frac{5t}{2\sqrt{n}}\right) \le \exp(-t),
\end{equation*}
where
\begin{equation*}
E_j = \mathbb{E}\left[\sup_{H\in\mathcal{A}_j}\sqrt{n}\left|\hat{G}_n(H) - F_{\theta_0}(H)\right|\right],
\end{equation*}
and
\begin{equation*}
\sigma_j^2 = \sup_{H\in\mathcal{A}_j}\,\mbox{Var}\left(\sqrt{n} (\hat{G}_n(H) - F_{\theta_0}(H)) \right) \le \sup_{H\in\mathcal{A}_j}F_{\theta_0(H)} \le {q_{j+1}\delta}.
\end{equation*}
For bounding $E_j$, we note that the class $\mathcal{H}$ is a VC-class with VC dimension $p+2$ and by Corollary 3.5.8 of \citet{gine2016},
\begin{equation}\label{equ:ExpectationBound}
E_j \le \mathfrak{C} C_A\sigma_j\sqrt{\nu\log(A/\sigma_j)} + \mathfrak{C} C_A\frac{\nu\log(A/\sigma_j)}{\sqrt{n}} \ ,
\end{equation}
where $C_A = 2\log A/(2\log A - 1)$, for some $A \ge 2$, $\nu \ge 1$ and $\sup_QN(\mathcal{F}, L^2(Q), \epsilon) \le (A/\epsilon)^{\nu}$. Here $N(\mathcal{F}, L^2(Q), \epsilon)$ represents the minimal number of $L^2(Q)$ balls of radius $\epsilon$ that cover the class $\mathcal{F}$ and for our purpose $\mathcal{F} = \{\mathbbm{1}\{x \in H \}:\, H \in \mathcal{H} \}$. Because this is a VC class with VC index $p+2$ \citep[Example 3, page 833 of][]{wellner1986}, it follows from Theorem 2.6.7 of \citet{vardervaart1996} that
\begin{equation*}
\sup_{Q}N(\mathcal{F}, L^2(Q), \epsilon) \le K(p+2)(16e)^{p+2}\left(\frac{1}{\epsilon}\right)^{2(p+1)} \le \left(\frac{Ke^{1/e}(16e)}{\epsilon}\right)^{2(p+1)},
\end{equation*}
for $\epsilon > 0$ and a universal constant $K \ge 1$ \citep[also see inequality (3.235) of][]{gine2016}. Note that for our function class $\mathcal{F}$, the envelope function can be taken to be the constant function taking value $1$. The second inequality above follows from $(p+2)^{1/(d + 2)} \le e^{1/e}$ {{for $q \ge 1$}}. Substituting the inequality above in~\eqref{equ:ExpectationBound}, we get
\begin{equation*}
E_j \le \mathfrak{C}\left[\sqrt{{q_{j+1}\delta(p+1)\log(\mathfrak{C}/(q_{j+1}\delta))}} + \frac{2(p+1)\log(\mathfrak{C}/(q_{j+1}\delta))}{\sqrt{n}}\right].
\end{equation*}
Therefore, with probability at least $1 - \exp(-t)$,
\begin{align*}
\sup_{H\in\mathcal{A}_j}\sqrt{n}\left|\hat{G}_n(H) - F_{\theta_0}(H)\right| & \le \mathfrak{C}\left[\sqrt{{q_{j+1}\delta(p+1)\log(\mathfrak{C}/(q_{j+1}\delta))}} \right. \\
& \quad + \left. \frac{2(p+1)\log(\mathfrak{C}/(q_{j+1}\delta))}{\sqrt{n}}\right] \\ 
& \quad + \sqrt{2q_{j+1}\delta t} + \frac{5t}{2\sqrt{n}} \ .
\end{align*}
Thus with probability at least $1 - \exp(-t)$,
\begin{align*}
\frac{\sup_{H\in\mathcal{A}_j}\sqrt{n}\left|\hat{G}_n(H) - F_{\theta_0}(H)\right|}{(q_j\delta)^{\alpha}} & \le \mathfrak{C}\left[\frac{\sqrt{{q_{j+1}\delta(p+1)\log(\mathfrak{C}/(q_{j+1}\delta))}}}{(q_j\delta)^{\alpha}} \right. \\
& \quad + \left. \frac{2(p+1)\log(\mathfrak{C}/(q_{j+1}\delta))}{(q_j\delta)^{\alpha}\sqrt{n}}\right] \\ 
& \quad + \frac{\sqrt{2q_{j+1}\delta t}}{(q_j\delta)^{\alpha}} + \frac{5t}{2(q_j\delta)^{\alpha}\sqrt{n}} \ .
\end{align*}
Let this event be denoted by $\mathcal{E}_j(t)$. Then we proved, for each $0 \le j\le J$,
\begin{equation*}
\Pr\left(\mathcal{E}_j^c(t)\right) \le \exp(-t) \ .
\end{equation*} 
Therefore, by union bound
\begin{align*}
1 - \Pr\left(\bigcap_{j = 0}^{J}\mathcal{E}_j(t + 2\log(j + 1))\right) & = \Pr\left(\bigcup_{j = 0}^{J}\mathcal{E}_j^c(t + 2\log(j + 1))\right) \\ 
& \le \sum_{j = 0}^J\exp(-t - 2\log(j+1)) \\
& \le \sum_{j = 0}^{J}\frac{\exp(-t)}{(j+1)^2} \le \frac{\pi^2}{6}\exp(-t) \ .
\end{align*}
Therefore, with probability at least $1 - (\pi^2/6)\exp(-t)$, for all $0 \le j \le J$,
\begin{align} 
\sup_{H\in\mathcal{A}_j}\frac{\sqrt{n}\left|\hat{G}_n(H) - F_{\theta_0}(H)\right|}{(q_j\delta)^{\alpha}} & \le \mathfrak{C}\left[\frac{\sqrt{{q_{j+1}\delta(p+1)\log(\mathfrak{C}/(q_{j+1}\delta))}}}{(q_j\delta)^{\alpha}} \right. \nonumber \\
& \quad \left. + \frac{2(p+1)\log(\mathfrak{C}/(q_{j+1}\delta))}{(q_j\delta)^{\alpha}\sqrt{n}}\right] \nonumber \\ 
& \quad + \frac{\sqrt{2q_{j+1}\delta (t + 2\log(j+1))}}{(q_j\delta)^{\alpha}} \nonumber \\
& \quad + \frac{5(t + 2\log(j+1))}{2(q_j\delta)^{\alpha}\sqrt{n}} \ . \label{equ:FirstUnion}
\end{align}
Take $\alpha = 1/2$ and for $j\ge 0$, $q_j = 4^j.$
For this sequence, it is clear that $\sqrt{q_{j+1}} = 2q_j^{\alpha}$ and $J + 1 \le 0.5\log_2(\delta^{-1})$. So, the bound~\eqref{equ:FirstUnion} implies that with probability at least $1 - (\pi^2/6)\exp(-t)$, for all $j \ge 0$,
\begin{align*}
\frac{\sup_{H\in\mathcal{A}_j}\sqrt{n}\left|\hat{G}_n(H) - F_{\theta_0}(H)\right|}{(q_j\delta)^{\alpha}} & \le \mathfrak{C}\left[2\delta^{1/2 - \alpha}\sqrt{{(p+1)\log(\mathfrak{C}\delta^{-1})}} \right. \\
& \quad \left. + \frac{2(p+1)\log(\mathfrak{C}\delta^{-1})}{(q_j\delta)^{\alpha}\sqrt{n}}\right] \\ 
& \quad + {2\sqrt{2}\delta^{1/2 - \alpha} t} + \frac{5t}{2(q_j\delta)^{\alpha}\sqrt{n}} \ .
\end{align*}
Here we used two inequalities: $\log(\mathfrak{C}/(q_{j+1}\delta)) \le \log(\mathfrak{C}\delta^{-1})$ and $\log(j+1) \le \log(J+1) \le J \le 0.5\log_2(\delta^{-1})$. So, the inequality holds with a possibly increased $\mathfrak{C}$.

Taking the maximum over $0 \le j \le J$ on the right-hand side, we get with probability at least $1 - (\pi^2/6)\exp(-t)$,
\begin{align*}
\max_{0 \le j\le J}\sup_{H\in\mathcal{A}_j}\,\frac{\sqrt{n}|\hat{G}_n(H) - F_{\theta_0}(H)|}{(q_j\delta)^{\alpha}} & \le \mathfrak{C}\left[2\delta^{1/2 - \alpha}\sqrt{(p+1)\log(\mathfrak{C}\delta^{-1})} \right. \\
& \quad \left. + \frac{2(p+1)\log(\mathfrak{C}\delta^{-1})}{\delta^{\alpha}\sqrt{n}}\right] \\ 
& \quad + 2\sqrt{2}\delta^{1/2 - \alpha}t + \frac{5t}{2\delta^{\alpha}\sqrt{n}} \ .
\end{align*}
Hence,
\begin{equation*}
\max_{0 \le j\le J}\sup_{H\in\mathcal{A}_j}\,\frac{\sqrt{n}|\hat{G}_n(H) - F_{\theta_0}(H)|}{(q_j\delta)^{\alpha}} = O_p\left(\delta^{1/2 - \alpha}\sqrt{\log(\delta^{-1})} + \frac{\log(\delta^{-1})}{\delta^{\alpha}\sqrt{n}}\right) \ .
\end{equation*}
Note that since $\alpha = 1/2$, $\delta^{1/2 - \alpha} = 1$. So, for any $0 < \delta \le 1/2$,
\begin{equation*}
\sup_{x:\,\delta \le d(x; F_{\theta_0}) \le 1/2}\,\left|\frac{d_n(x; \hat{G}_n) - d(x; F_{\theta_0})}{\sqrt{d(x; F_{\theta_0})}}\right| = O_p\left(\sqrt{\frac{\log(\delta^{-1})}{n}} + \frac{\log(\delta^{-1})}{\sqrt{\delta}n}\right) \ .
\end{equation*}
This implies that for any $\alpha < 1/2$,
\begin{equation*}
\sup_{x:\,\delta \le d(x; F_{\theta_0}) \le 1/2}\,\left|\frac{d_n(x; \hat{G}_n) - d(x; F_{\theta_0})}{{d^{\alpha}(x; F_{\theta_0})}}\right| = O_p\left(\sqrt{\frac{\log(\delta^{-1})}{n}} + \frac{\log(\delta^{-1})}{\sqrt{\delta}n}\right) \ ,
\end{equation*}
since $(d(x; F_{\theta_0}))^{1/2 - \alpha} \le 1$.
\end{proof}

Now, we prove Theorem \ref{the:normalityresults}.

\begin{proof}[Proof of Theorem \ref{the:normalityresults}]
Before we start the proof, let us remark on the properties of weight functions in the class $\mathcal{W}(K_0, K_1)$.
\begin{remark}[Implications of derivative assumptions on the weight function.] \label{remark:1}
Note that every function $w\in\mathcal{W}(K_0, K_1)$ satisfies
\begin{align*}
\sup_{\tau} |w'(\tau)(\zeta \tau + 1)| & \le \sup_{\tau \in [-1,0)} |w'(\tau)(\zeta \tau + 1)| \vee \sup_{\tau \in [0,\infty)} |w'(\tau)(\zeta \tau + 1)| \\
& \le \sup_{\tau \in [-1,0)} |w'(\tau)| \vee K_0 \ ,
\end{align*}
but
\begin{align*}
\sup_{\tau \in [-1,0)} |w'(\tau)| & = \sup_{\tau \in [-1,0)} |w'(\tau) - w'(0)| \\
& \le \sup_{\tau \in [-1,0)} |w''(\tau)| \ |\tau| \\
& \le \sup_{\tau \in [-1,0)} |w''(\tau)| \le K_1 \ ,
\end{align*}
by the assumption on the second derivative of $w(\cdot)$. Moreover, by the definition of $\mathcal W(K_0, K_1)$, we also have
\begin{equation*}
|w''(t)| = |w''(t)(1 + 0)^2| \le |w''(t)(1 + t + 1)^2| \le K_1 \ ,
\end{equation*}
the first inequality above follows since $1 + t \ge 0$. It is also clear that $w''(\tau)(\zeta \tau + 1)^2$ is bounded by $2 (\zeta^2 + (1-\zeta)^2) K_1$ as well.
\end{remark}
\noindent \textbf{Convergence of  $A_{n,j}$}. We show a basic inequality about
\begin{equation*}
A_{n,j}^{(w)} - \frac{1}{n}\sum_{i=1}^n u_j(X_i;\theta_0) \ ,
\end{equation*}
where
\begin{equation*}
A_{n,j}^{(w)} = \frac{1}{n}\sum_{i=1}^n w(\tau_n(X_i;\theta_0))u_j(X_i; \theta_0) \ ,\end{equation*}
and
\begin{equation*}
\tau_n(x;\theta) := \left(\frac{d_n(x; \hat{G}_n) - d(x; F_{\theta})}{d^{\alpha}(x; F_{\theta})}\right) \ ,
\end{equation*}
the quantity $\alpha$ is suppressed in the notation of $\tau_n(x; \theta)$. Observe that for any $a_n > 0$, we have
\begin{align*}
A_{n,j}^{(w)} - \frac{1}{n}\sum_{i=1}^n u_j(X_i; \theta_0) & = \frac{1}{n}\sum_{i=1}^n \left\{w(\tau_n(X_i; \theta_0)) - 1\right\}u_j(X_i; \theta_0) \\
& = \frac{1}{n}\sum_{i=1}^n \left\{w(\tau_n(X_i; \theta_0)) - 1\right\}u_j(X_i; \theta_0)\mathbbm{1}\{d(X_i; F_{\theta_0}) \le a_n\} \\
& \quad + \frac{1}{n}\sum_{i=1}^n \left\{w(\tau_n(X_i; \theta_0)) - 1\right\}u_j(X_i; \theta_0)\mathbbm{1}\{d(X_i; F_{\theta_0}) \ge a_n\} \ .
\end{align*}
For the first term, note that, because $w \in [0,1]$ for all $w \in \mathcal{W}(K_0, K_1)$,
\begin{equation} \label{equ:FirstTermBound}
\begin{split}
& \left|\frac{1}{n}\sum_{i=1}^n \left\{w(\tau_n(X_i; \theta_0)) - 1\right\}u_j(X_i; \theta_0)\mathbbm{1}\{d(X_i; F_{\theta_0}) \le a_n\}\right| \\
& \qquad \le \frac{1}{n}\sum_{i=1}^n \left|u_j(X_i; \theta_0)\right|\mathbbm{1}\{d(X_i; F_{\theta_0}) \le a_n\}.
\end{split}
\end{equation}
For the second term, note that
\begin{align*}
w(\tau_n(X_i; \theta_0)) - 1 & = w(\tau_n(X_i; \theta_0)) - w(0) \\
& = w'(0)\tau_n(X_i; \theta_0) + \frac{w''(\bar{\tau})}{2} \tau_n^2(X_i; \theta_0) = \frac{w''(\bar{\tau}_i)}{2}\tau_n^2(X_i; \theta_0) \ ,
\end{align*}
for some $\bar{\tau}_i$ that lies on the line segment joining $0$ and $\tau_n(X_i; \theta_0)$. Therefore, using $\vert w'' \vert_{\infty} \le K_1$, we get
\begin{equation}\label{equ:SecondTermBound}
\begin{split}
& \left|\frac{1}{n}\sum_{i=1}^n \left\{w(\tau_n(X_i; \theta_0)) - 1\right\}u_j(X_i; \theta_0)\mathbbm{1}\{d(X_i; F_{\theta_0}) \ge a_n\}\right| \\
& \qquad \le \frac{K_1}{2n}\sum_{i=1}^n \tau_n^2(X_i; \theta_0) \mathbbm{1}\{d(X_i; F_{\theta_0}) \ge a_n\}|u_j(X_i; \theta_0)| \ .
\end{split}
\end{equation}
Combining the bounds~\eqref{equ:FirstTermBound} and~\eqref{equ:SecondTermBound}, we obtain
\begin{align*}
& \sup_{w\in\mathcal{W}(K_0, K_1)}\left|A_{n,j}^{(w)} - \frac{1}{n}\sum_{i=1}^n u_j(X_i; \theta_0)\right| \\ 
& \le \frac{1}{n}\sum_{i=1}^n |u_j(X_i; \theta_0)|\mathbbm{1}\{d(X_i; F_{\theta_0}) \le a_n\} + \frac{K_1}{2n}\sum_{i=1}^n \tau_n^2(X_i; \theta_0)\mathbbm{1}\{d(X_i; F_{\theta_0}) \ge a_n\}|u_j(X_i; \theta_0)| \\
& \le \frac{1}{n}\sum_{i=1}^n |u_j(X_i; \theta_0)|\mathbbm{1}\{d(X_i; F_{\theta_0}) \le a_n\} \\ 
& \quad + \frac{K_1}{2}\sup_{x:\,d(x; F_{\theta_0}) \ge a_n}\,\left|\frac{d_n(x; \hat{G}_n) - d(x; F_{\theta_0})}{d^{\alpha}(x; F_{\theta_0})}\right|^2\frac{1}{n}\sum_{i=1}^n |u_j(X_i; \theta_0)| \ .
\end{align*}
It is clear that the condition on $w'(t)(t + 1)$ is not relevant for the bound above. For the first term above, note that
\begin{align*}
\frac{1}{n}\sum_{i=1}^n |u_j(X_i; \theta_0)|\mathbbm{1}\{d(X_i; F_{\theta_0}) \le a_n\} & \le \left(\frac{1}{n}\sum_{i=1}^n |u_j(X_i; \theta_0)|^2\right)^{1/2}\left(\frac{1}{n}\sum_{i=1}^n \mathbbm{1}\{d(X_i; F_{\theta_0}) \le a_n\}\right)^{1/2} \\
& = \left(\mathbb{E}\left[u_j^2(X_1; \theta_0)\right]\right)^{1/2}O_p\left(\mathbb{P}(d(X_1; F_{\theta_0}) \le a_n)\right)^{1/2} \ .
\end{align*}
The result in Lemma~\ref{lemma:ratescaleddepth} shows that the second term satisfies (for $0 \le \alpha \le 1/2$),
\begin{equation*}
\sup_{x:d(x; F_{\theta_0}) \ge a_n}\left|\frac{d_n(x; \hat{G}_n) - d(x; F_{\theta_0})}{d^{\alpha}(x; F_{\theta_0})}\right|^2\frac{1}{n}\sum_{i=1}^n |u_j(X_i; \theta_0)| = \mathbb{E}\left[|u_j(X_1; \theta_0)|\right]O_p\left(\frac{\log(a_n^{-1})}{n} + \frac{\log^2(a_n^{-1})}{a_n n^2}\right) \ .
\end{equation*}
Therefore,
\begin{equation*}
\sup_{w\in\mathcal{W}(K_0, K_1)}\sqrt{n}\left|A_{n,j}^{(w)} - \frac{1}{n}\sum_{i=1}^n u_j(X_i; \theta_0)\right| = O_p\left(\sqrt{n\mathbb{P}(d(X_1; F_{\theta_0}) \le a_n)} + \frac{\log(a_n^{-1})}{\sqrt{n}} + \frac{\log^2(a_n^{-1})}{a_n n^{3/2}}\right) \ .
\end{equation*}
The left hand side is independent of $a_n$, and the right hand side can be made to converge to zero by choosing $a_n = 1/n$, for example, because by (A1), $\Pr(d(X; F_{\theta_0}) \le a_n) = O(a_n)$. This completes the proof for $0 \le \alpha \le 1/2$.

\noindent For $\alpha > 1/2$, note that
\begin{equation*}
\sup_{x:\,d(x; F_{\theta_0}) \ge a_n}\,\left|\frac{d_n(x; \hat{G}_n) - d(x; F_{\theta_0})}{d^{\alpha}(x; F_{\theta_0})}\right| = O_p\left(\frac{\sqrt{\log(a_n^{-1})}}{a_n^{\alpha - 1/2}\sqrt{n}} + \frac{\log(a_n^{-1})}{a_n^{\alpha}n}\right) \ .
\end{equation*}
Thus,
\begin{equation*}
\sup_{w\in\mathcal{W}(K_0, K_1)}\sqrt{n}\left|A_{n,j}^{(w)} - \frac{1}{n}\sum_{i=1}^n u_j(X_i; \theta_0)\right| = O_p\left(\sqrt{n\mathbb{P}(d(X_1; F_{\theta_0}) \le a_n)} + \frac{\log(a_n^{-1})}{a_n^{2\alpha - 1}\sqrt{n}} + \frac{\log^2(a_n^{-1})}{a_n^{2\alpha}n^{3/2}}\right)
\end{equation*}
Even here the right hand side can be made to converge to zero by choosing $a_n = n^{-\beta}$ for some $1 < \beta < \min\{3(4\alpha)^{-1}, 0.5(2\alpha - 1)^{-1}\}$. This completes the proof for $1/2 \le \alpha < 3/4$. So, the final result is that for $0 \le \alpha < 3/4$, as $n \to \infty$,
\begin{equation*}
\sup_{w\in\mathcal{W}(K_0, K_1)}\sqrt{n}\left|A_{n,j}^{(w)} - \frac{1}{n}\sum_{i=1}^n u_j(X_i; \theta_0)\right| = o_p(1).
\end{equation*}
\medskip
\noindent \textbf{Convergence for $B_{n,j,k}$}. We give a basic inequality for
\begin{equation*}
B_{n,j,k}^{(w)} - \frac{1}{n}\sum_{i=1}^n u_{j,k}(X_i; \theta_0) \ ,
\end{equation*}
where
\begin{equation*}
B_{n,j,k}^{(w)} := \frac{1}{n}\sum_{i=1}^n \nabla_k\left[w(\tau_n(X_i; \theta))u_j(X_i; \theta)\right]_{\theta = \theta_0} \ .
\end{equation*}
First note that
\begin{align*}
\left| B_{n,j,k}^{(w)} - \frac{1}{n} \sum_{i=1}^n u_{j,k}(X_i;\theta_0)  \right| & \le \frac{1}{n} \sum_{i=1}^n |w^{\prime}(\tau_n(X_i;\theta_0)) \nabla_k \tau_n(X_i;\theta_0) u_j(X_i;\theta_0)| \\
&\quad + \left|\frac{1}{n} \sum_{i=1}^n (w(\tau_n(X_i;\theta_0)) - 1) u_{j,k}(X_i;\theta_0) \right|\\
&\le \frac{1}{n}\sum_{i=1}^n |w^{\prime}(\tau_n(X_i;\theta_0)) \nabla_k \tau_n(X_i;\theta_0) u_j(X_i;\theta_0)|\\
&\quad+ \frac{1}{n}\sum_{i=1}^n |u_{j,k}(X_i; \theta_0)|\mathbbm{1}\{d(X_i; F_{\theta_0}) \le a_n\}\\ 
&\quad+ \frac{K_1}{2n}\sum_{i=1}^n \tau_n^2(X_i; \theta_0)\mathbbm{1}\{d(X_i; F_{\theta_0}) \ge a_n\}|u_{j,k}(X_i; \theta_0)| \ .
\end{align*}
For the first term, observe that
\begin{equation*}
\nabla_k\tau_n(x; \theta) = - \left(\alpha\tau_n(x; \theta) + \{d(x; F_{\theta})\}^{1 - \alpha}\right)\nu_k(x; \theta) \ .
\end{equation*}
Also, observe that
\begin{align*}
\left|w'(\tau_n(X_i; \theta_0))\right| &= \left|w'(\tau_n(X_i; \theta_0)) - w'(0)\right| = \left|w''(\bar{\tau_i})\tau_n(X_i; \theta_0)\right| \le K_1\left|\tau_n(X_i; \theta_0)\right| \ .
\end{align*}
Here $\bar{\tau}_i$ represents a real number that lies on the line segment joining $0$ and $\tau_n(X_i; \theta_0)$. Thus,
\begin{align*}
& \frac{1}{n} \sum_{i=1}^n |w^{\prime}(\tau_n(X_i;\theta_0)) \nabla_k \tau_n(X_i;\theta_0) u_j(X_i;\theta_0)| \\
& \qquad \le \frac{1}{n}\sum_{i=1}^n |w'(\tau_n(X_i;\theta_0))(\alpha\tau_n(X_i;\theta_0) + \{d(X_i;F_{\theta_0})\}^{1 - \alpha})u_j(X_i;\theta_0)|\\
& \qquad \le \frac{1}{n}\sum_{i=1}^n |w'(\tau_n(X_i;\theta_0))(\alpha\tau_n(X_i;\theta_0) + 1)u_j(X_i;\theta_0)| \\
& \qquad\le \frac{K_0\vee K_1}{n}\sum_{i=1}^n |u_j(X_i; \theta_0)|\mathbbm{1}\{d(X_i; F_{\theta_0}) \le a_n\} \\
& \qquad\quad + \frac{K_1}{n}\sum_{i=1}^n |\tau_n(X_i; \theta_0)(\alpha\tau_n(X_i;\theta_0) + 1)|\times|u_j(X_i; \theta_0)\nu_k(X_i; \theta_0)|\mathbbm{1}\{d(X_i;F_{\theta_0}) \ge a_n\} \\
& \qquad \le \frac{K_0\vee K_1}{n}\sum_{i=1}^n |u_j(X_i; \theta_0)|\mathbbm{1}\{d(X_i; F_{\theta_0}) \le a_n\} \\
& \qquad\quad +\sup_{x:\,d(x; F_{\theta_0}) \ge a_n}\,\left(|\tau_n(x; \theta_0)| + \tau_n^2(x; \theta_0)\right)\times\frac{K_1}{n}\sum_{i=1}^n |u_j(X_i; \theta_0)\nu_k(X_i; \theta_0)| \ . 
\end{align*}
Therefore, 
\begin{align*}
& \sup_{w\in\mathcal{W}(K_0, K_1)}\left| B_{n,j,k}^{(w)} - \frac{1}{n} \sum_{i=1}^n u_{j,k}(X_i;\theta_0)  \right| \\
& \qquad \le \frac{K_0\vee K_1}{n}\sum_{i=1}^n |u_j(X_i; \theta_0)|\mathbbm{1}\{d(X_i; F_{\theta_0}) \le a_n\} \\
& \qquad\quad + \sup_{x:\,d(x; F_{\theta_0}) \ge a_n}\,\left(|\tau_n(x; \theta_0)| + \tau_n^2(x; \theta_0)\right)\times\frac{K_1}{n}\sum_{i=1}^n |u_j(X_i; \theta_0)\nu_k(X_i; \theta_0)| \\
& \quad\qquad + \frac{1}{n}\sum_{i=1}^n |u_{j,k}(X_i; \theta_0)|\mathbbm{1}\{d(X_i; F_{\theta_0}) \le a_n\} \\ 
& \quad\qquad + \frac{K_1}{2n}\sum_{i=1}^n \tau_n^2(X_i; \theta_0)\mathbbm{1}\{d(X_i; F_{\theta_0}) \ge a_n\}|u_{j,k}(X_i; \theta_0)|.
\end{align*}
The last two terms converge to zero as $n\to\infty$ and are actually of order $o_p(n^{-1/2})$ using the calculations of previous part of the proof. The first term is of the order
\begin{equation*}
O_p\left(\mathbb{E}[|u_j(X_1;\theta_0)\nu_k(X_1;\theta_0)|]\sup_{x:d(x;F_{\theta_0}) \ge a_n}|\tau_n(x;\theta_0)|\right) = o_p(1) \ ,
\end{equation*}
where the $o_p(1)$ follows from the result in Lemma~\ref{lemma:ratescaleddepth}.
\medskip

\noindent \textbf{Convergence for $C_{n,j,k,h}$}. From the calculations in the above part of the proof it is clear that
\begin{align*}
\sup_{w\in\mathcal{W}(K_0, K_1)}\left|C_{n,j,k,h}^{(w)}\right| & \le \frac{1}{n}\sum_{i=1}^n |u_{j,k,h}(X_i; \bar{\theta})| + \frac{K_0\vee K_1}{n}\sum_{i=1}^n |\nu_k(X_i;\bar{\theta})u_{j,h}(X_i;\bar{\theta})| \\
& \quad + \frac{K_0\vee K_1}{n}\sum_{i=1}^n |\nu_{h}(X_i; \bar{\theta})u_{j,k}(X_i;\bar{\theta})| + \frac{K_0\vee K_1}{n}\sum_{i=1}^n |\nu_{k,h}(X_i;\bar{\theta})u_j(X_i; \bar{\theta})| \\
& \quad + \frac{K_0\vee K_1}{n}\sum_{i=1}^n |\nu_k(X_i;\bar{\theta})\nu_h(X_i;\bar{\theta})u_j(X_i;\bar{\theta})| \\
& \quad + \frac{K_1}{n}\sum_{i=1}^n |\nu_k(X_i;\bar{\theta})\nu_h(X_i;\bar{\theta})u_j(X_i;\bar{\theta})| \ .
\end{align*}
Each of the terms above can be bounded in terms of $M_{ijkh}$ for $i = 1,2,3,4$ for all $\bar{\theta} \in N(\theta_0)$, uniformly. This proves
\begin{equation*}
\sup_{w\in\mathcal{W}(K_0, K_1)}|C_{n,j,k,h}^{(w)}| = O_p(1) \ .
\end{equation*}
\end{proof}

\begin{remark}
  All the theoretical results have been stated and proved for $\alpha \in (0, 3/4)$. Nevertheless, in the simulations we consider also $\alpha = 3/4, 1$ since in the multivariate normal setting asymptotic results still hold (See Remark \ref{remark:alpha} in the main document).
\end{remark}

\section{Wind Speed}
\label{sup:sec:wind}

In Figure \ref{sup:fig:wind-mean} we report the behavior of our procedure (green-blue) and maximum likelihood (purple) in the estimation of the mean curve as the contamination level increases.

\begin{figure}
\centering
\includegraphics[width=0.9\textwidth]{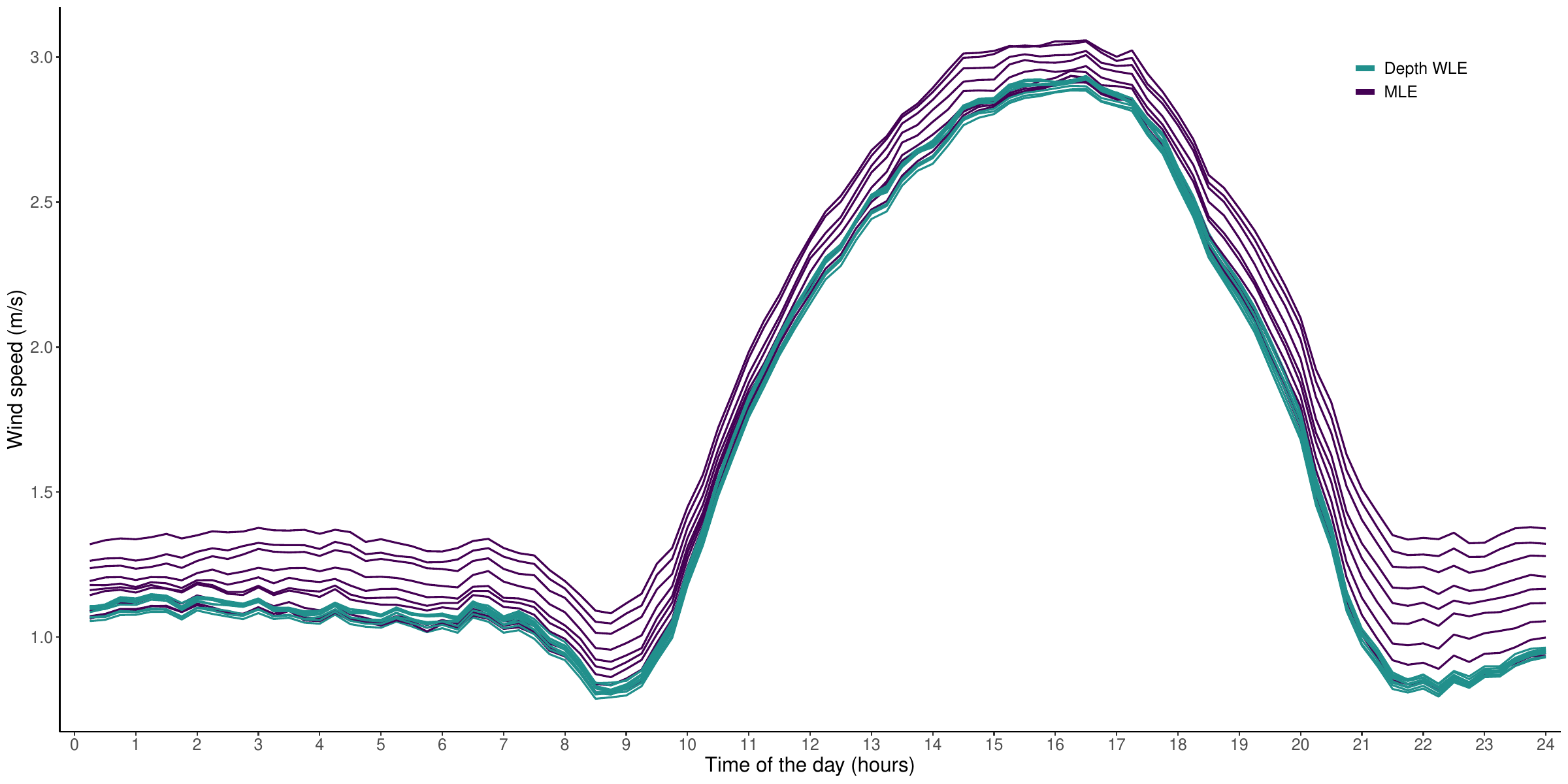}
\caption{Wind Data. Estimate mean vector provided by MLE (yellow) and weighted likelihood based on depth procedure (purple). Contamination level range from $0$ to $22.4\%$.}
\label{sup:fig:wind-mean}
\end{figure}  

The estimated correlation structure is reported in Figures \ref{sup:fig:wind-cor-1} and \ref{sup:fig:wind-cor-2} for MLE and our method respectively. While the estimated correlation is very stable for our procedure, this is not the case for MLE.

\begin{figure}
\centering
\includegraphics[width=0.4\textwidth]{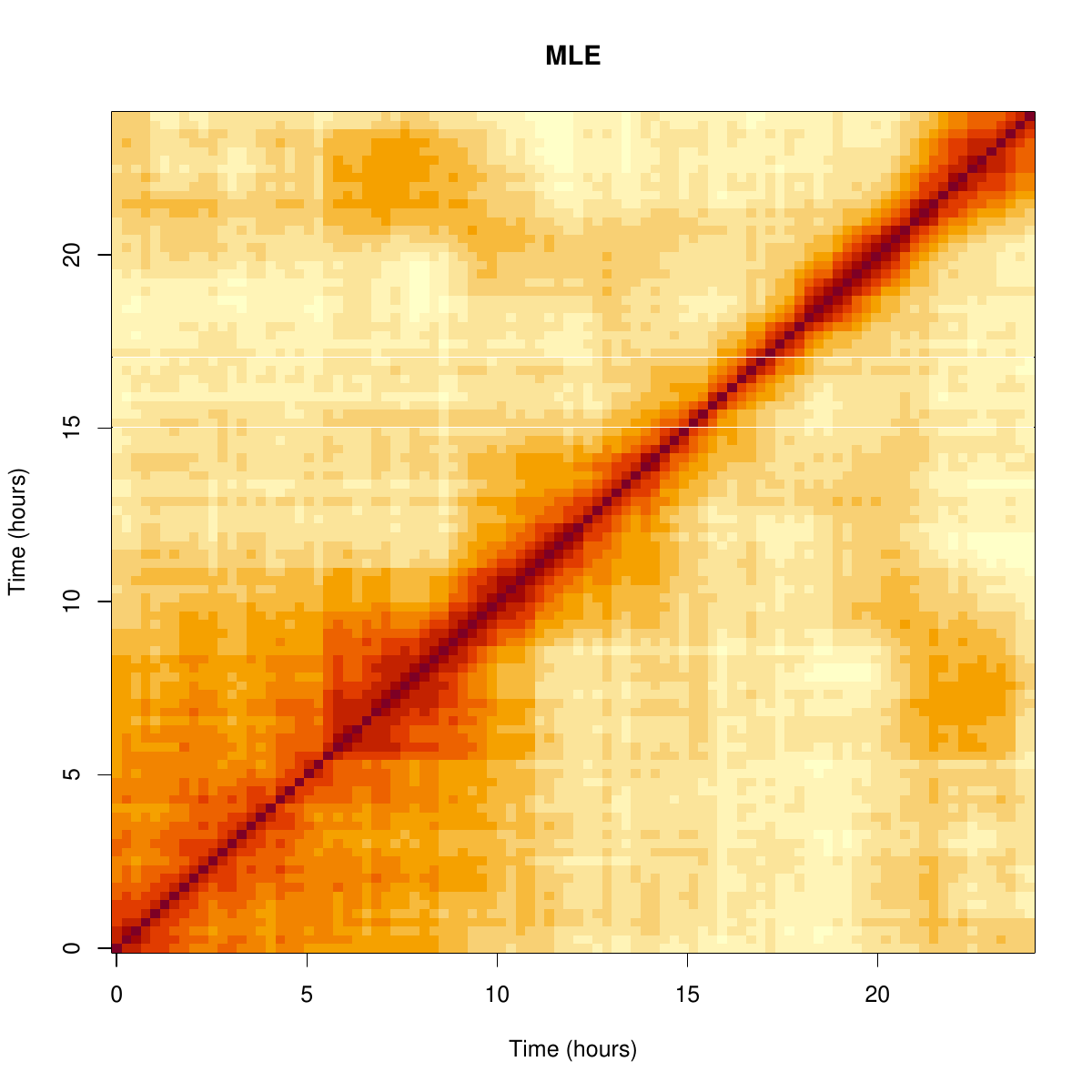}
\includegraphics[width=0.4\textwidth]{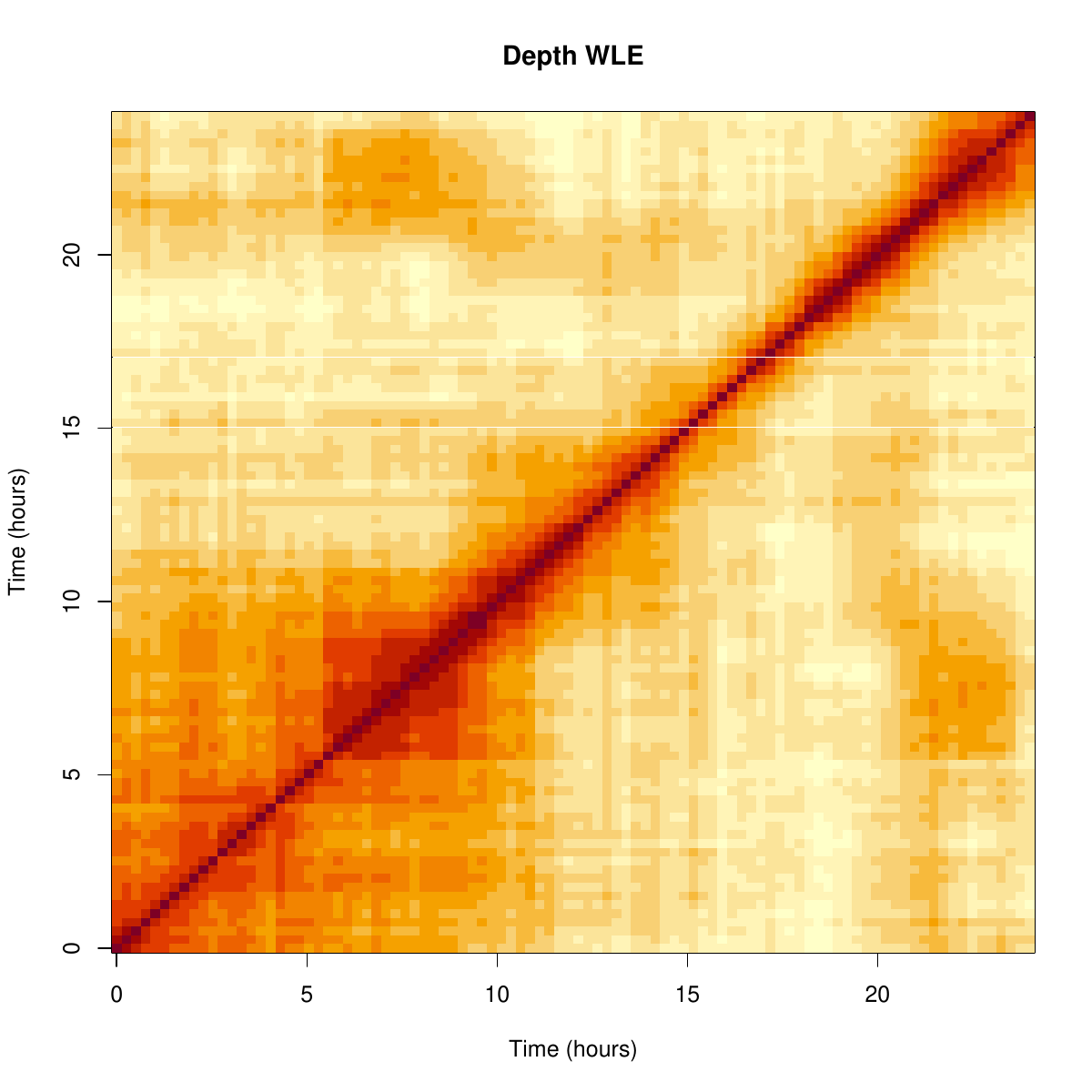} \\
\includegraphics[width=0.4\textwidth]{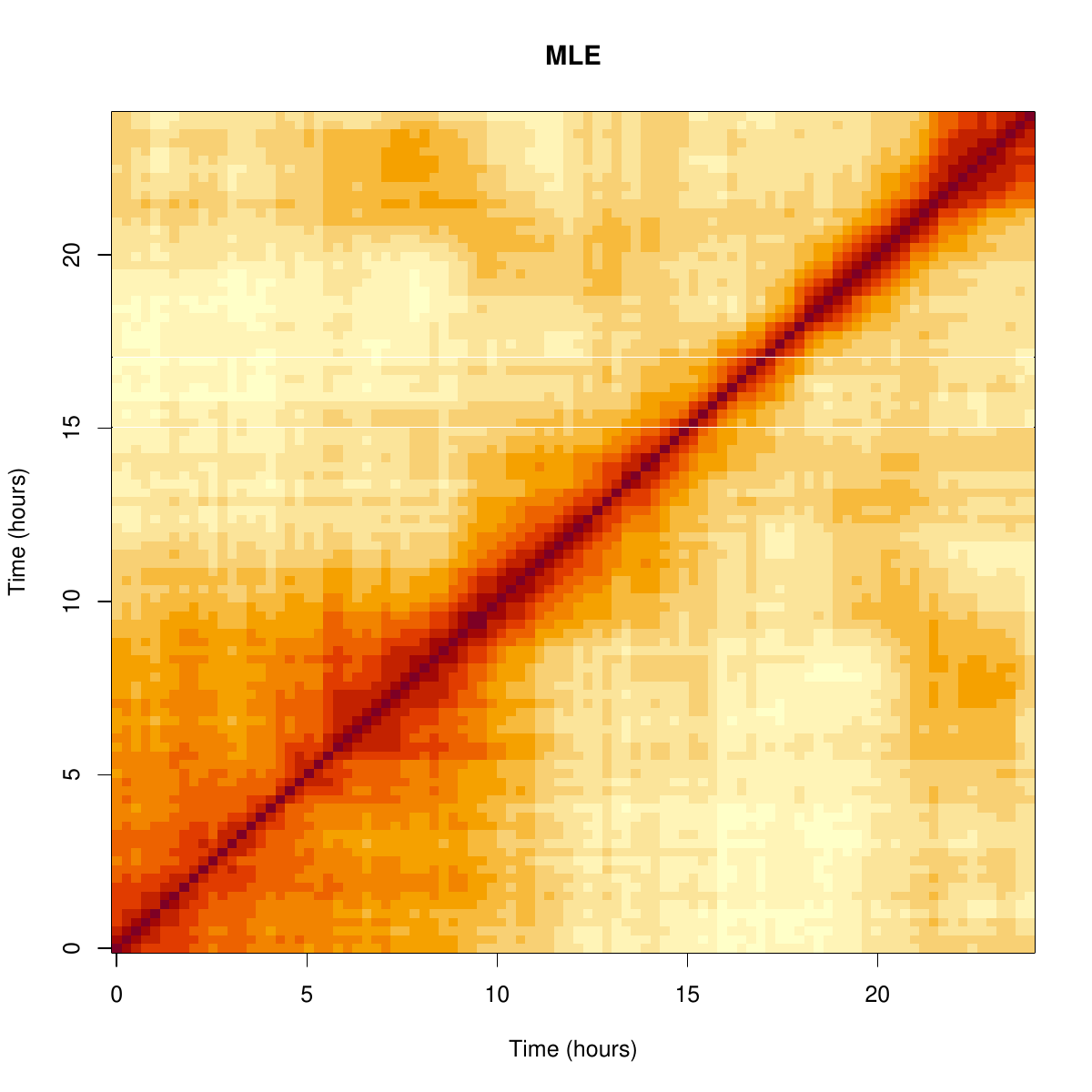}
\includegraphics[width=0.4\textwidth]{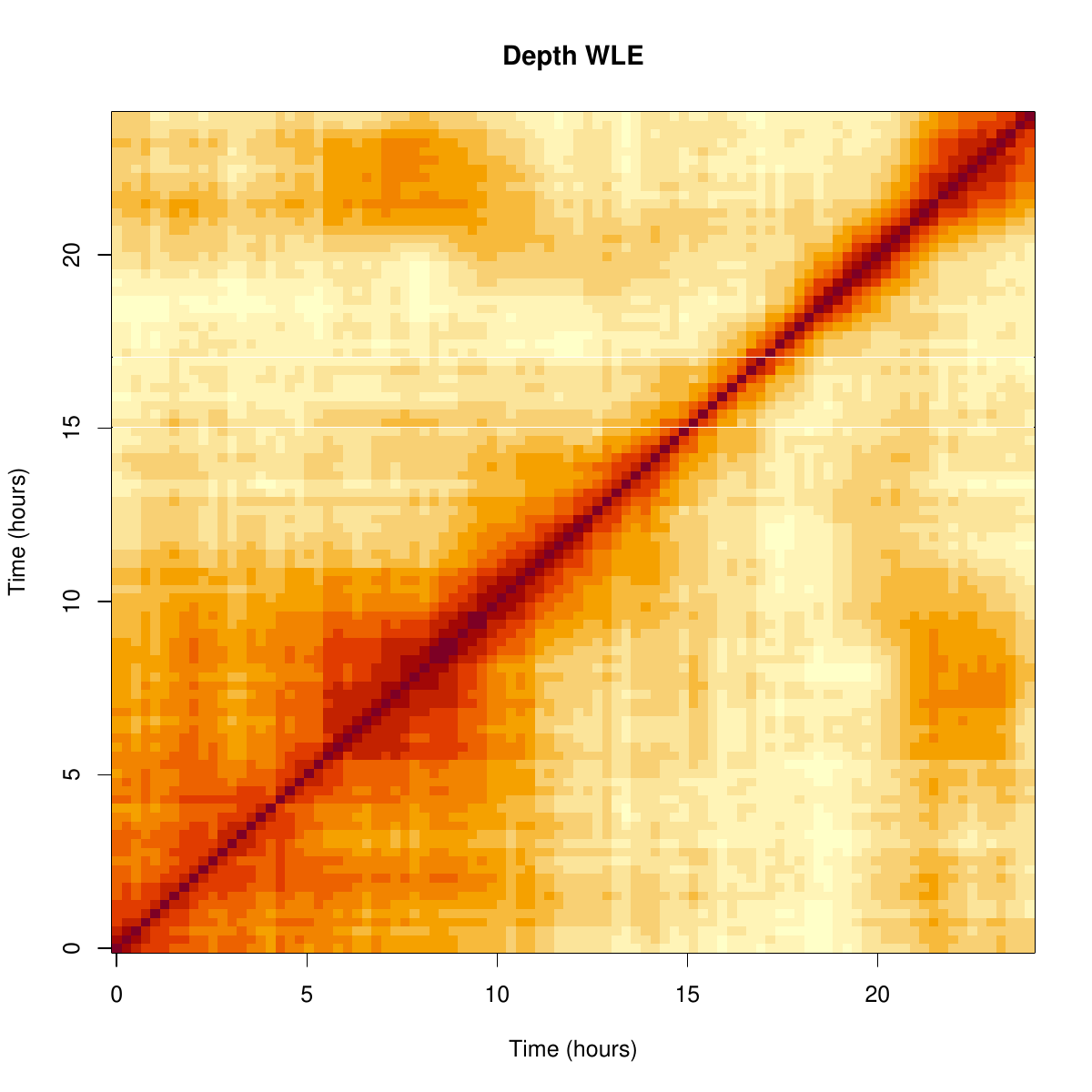} \\
\includegraphics[width=0.4\textwidth]{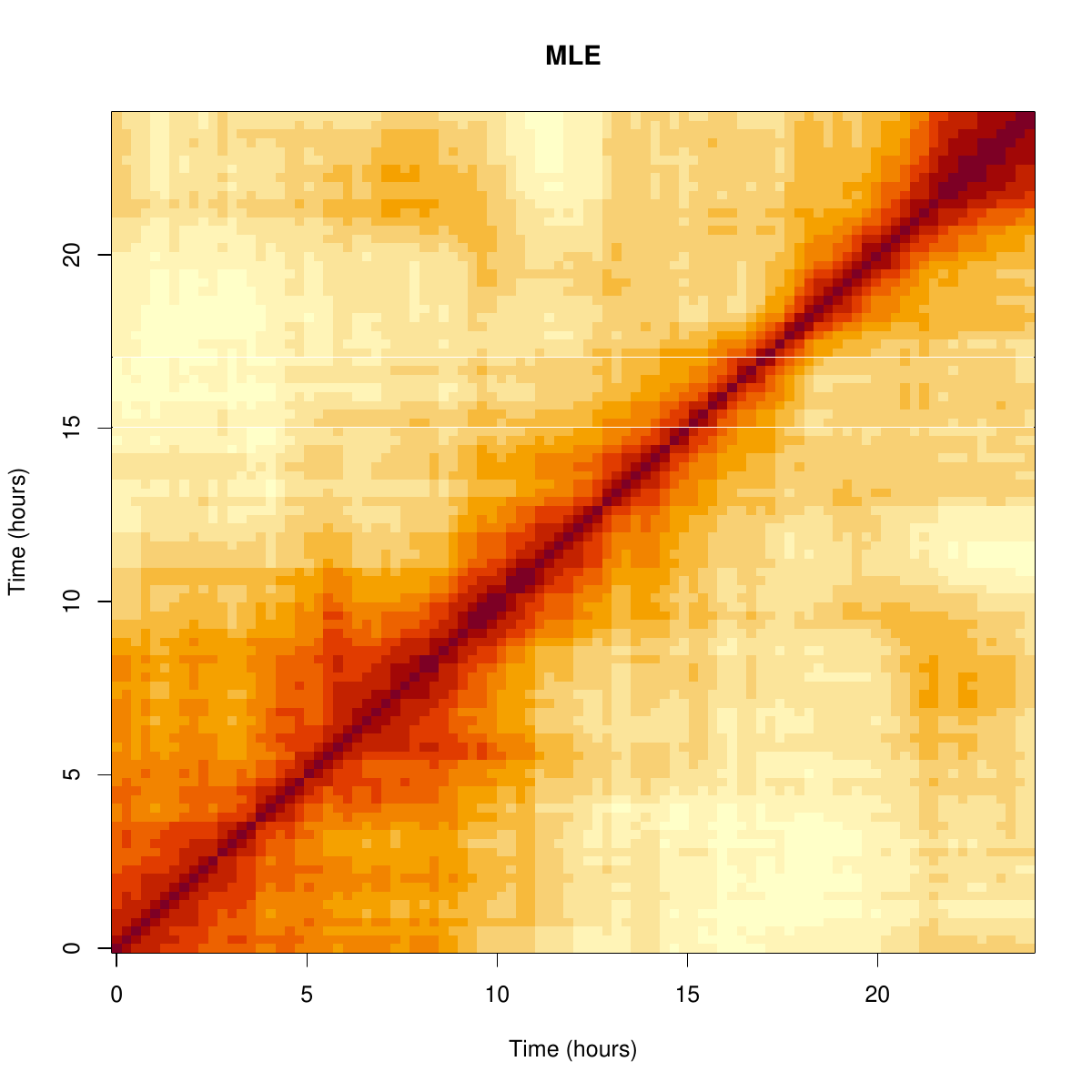}
\includegraphics[width=0.4\textwidth]{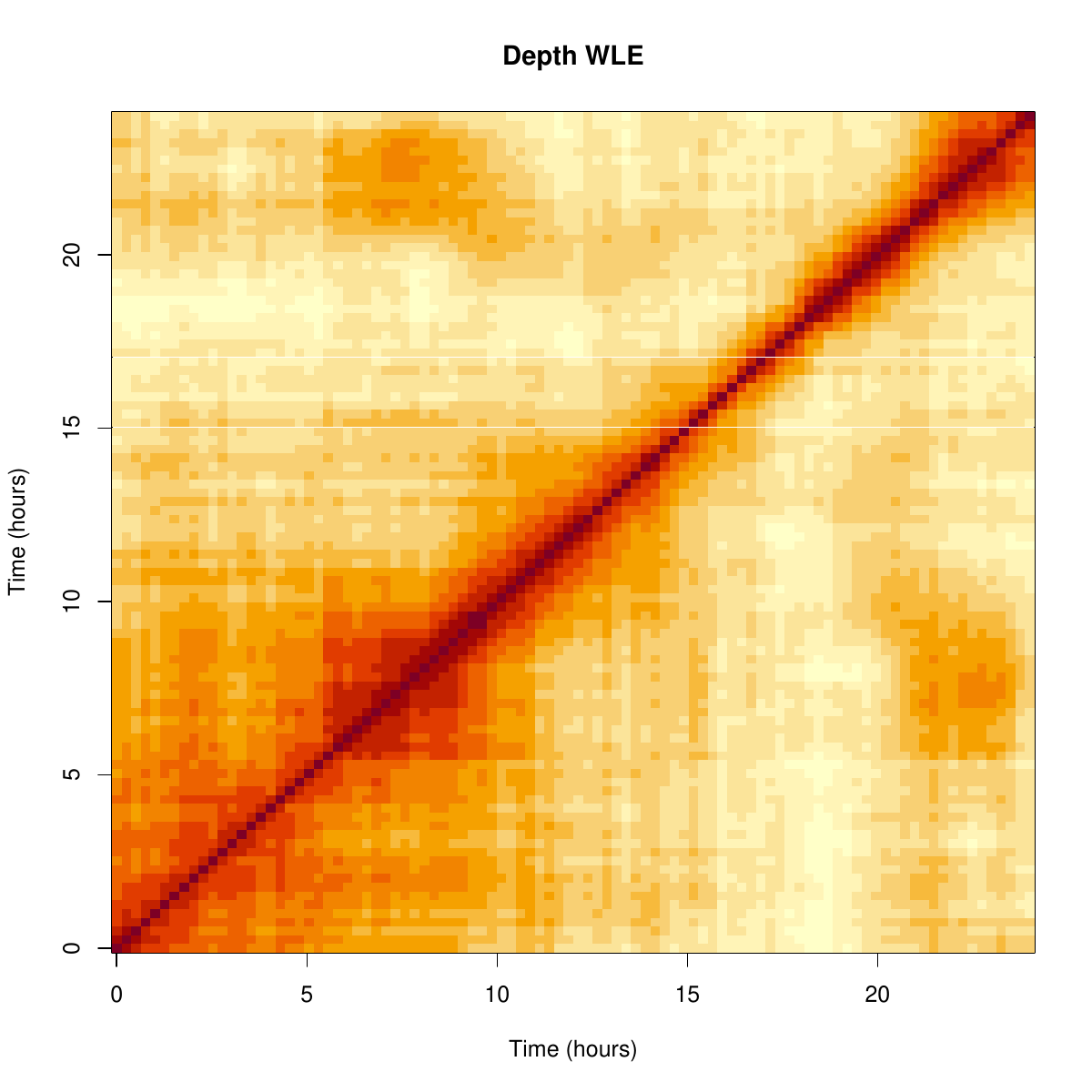} \\
\caption{Wind Data. Estimates correlation matrix provided by MLE (First column) and weighted likelihood based on depth procedure (Second column). First row: no contamination, second row: $6.7\%$ and third row: $12.6\%$.}
\label{sup:fig:wind-cor-1}
\end{figure}  

\begin{figure}
\centering
\includegraphics[width=0.4\textwidth]{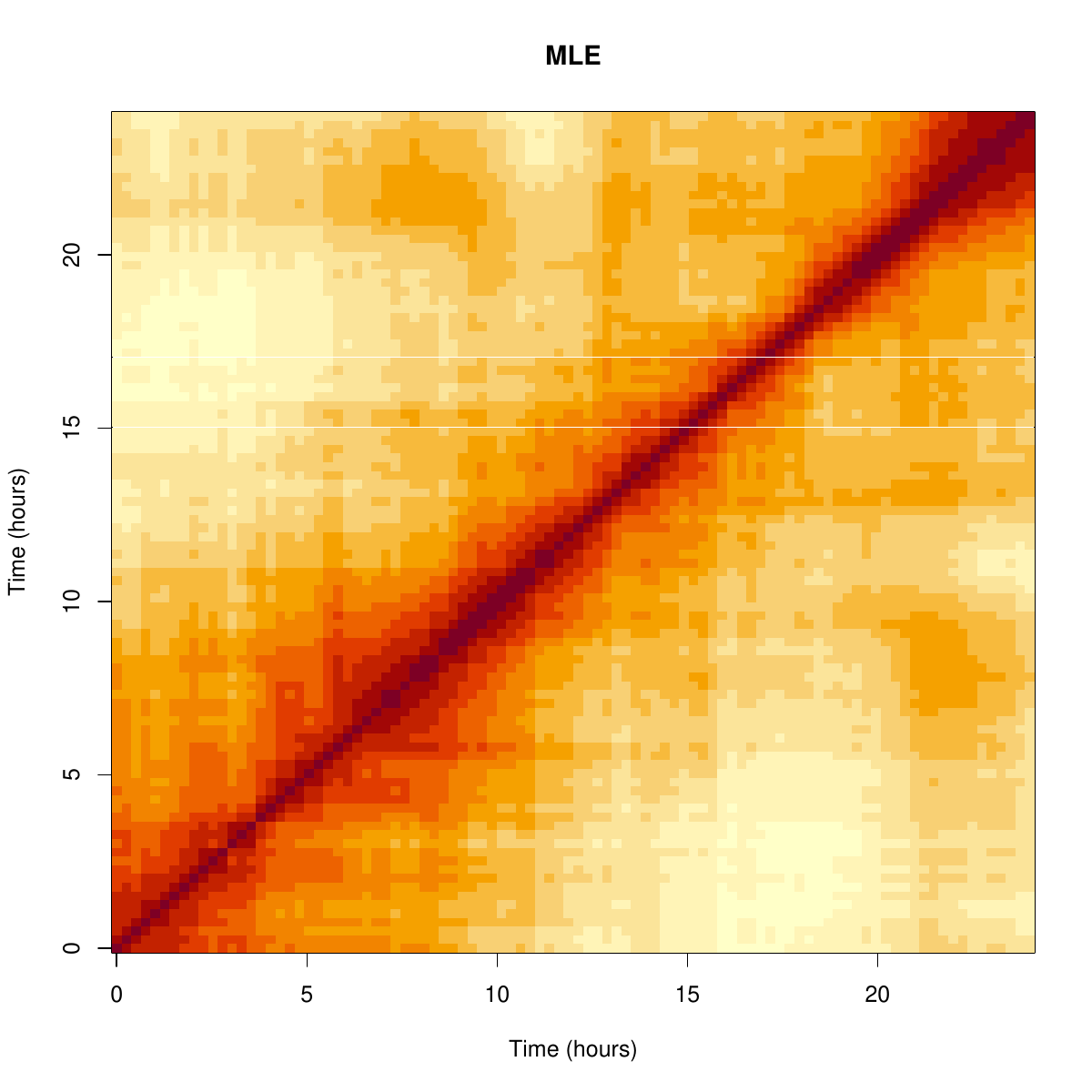}
\includegraphics[width=0.4\textwidth]{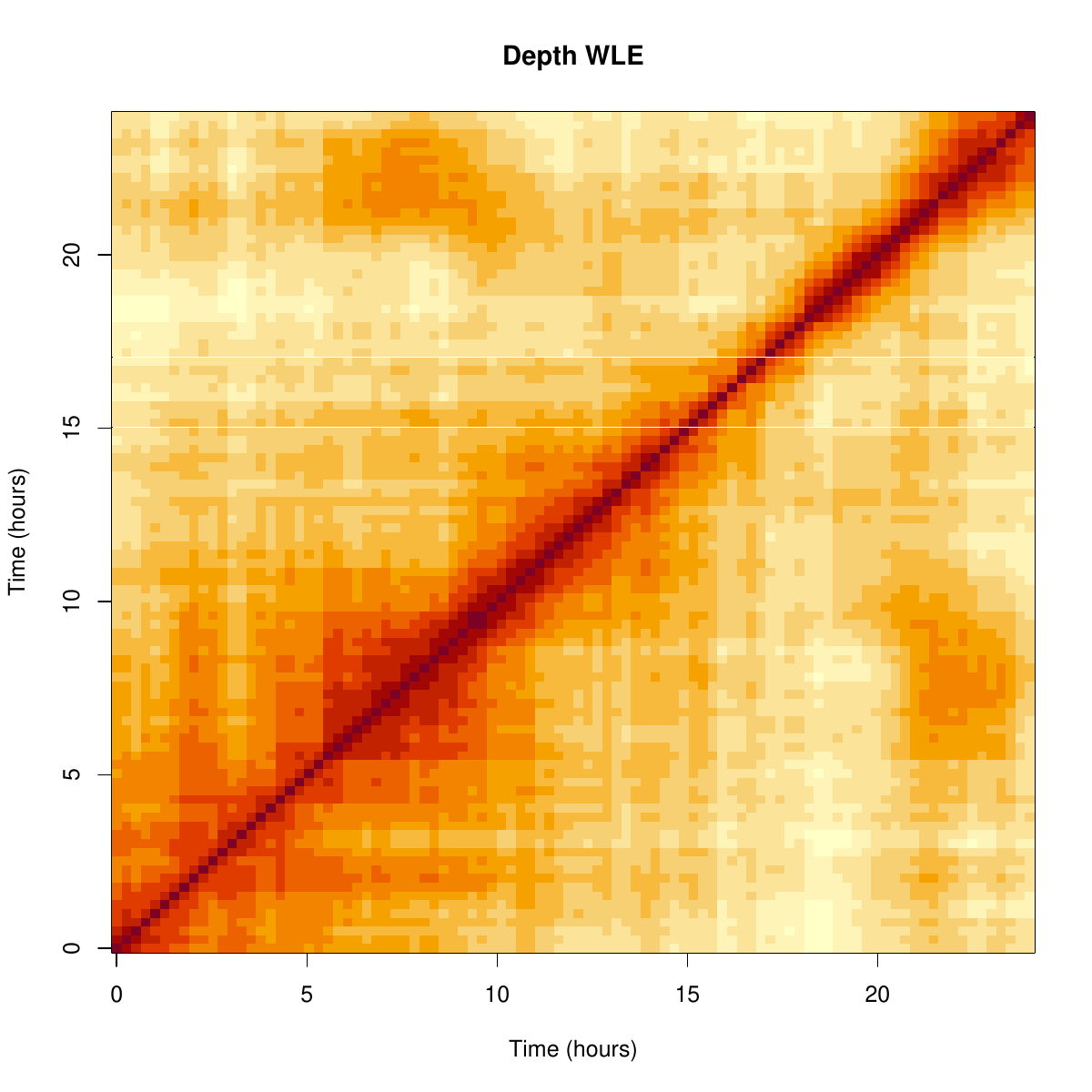} \\
\includegraphics[width=0.4\textwidth]{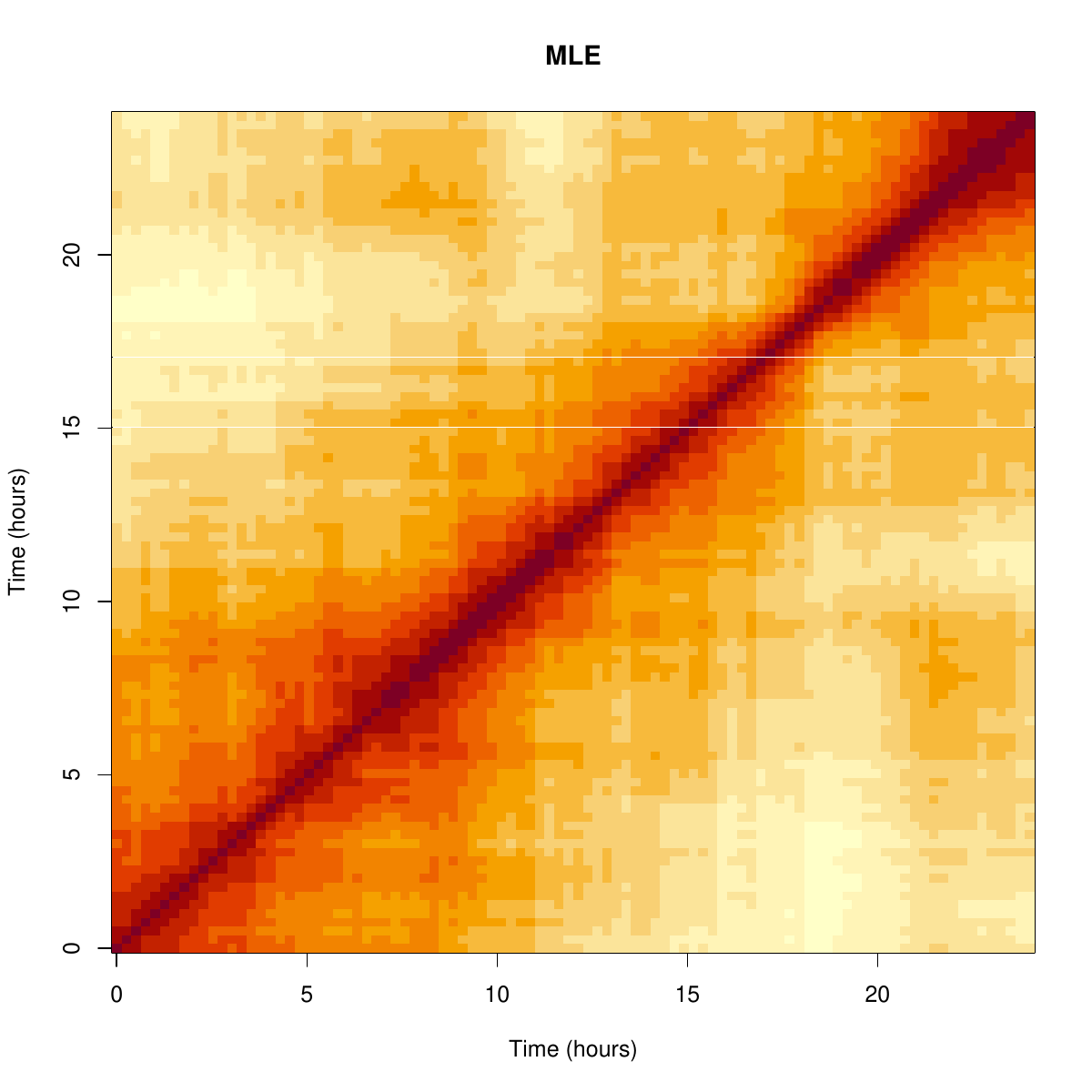}
\includegraphics[width=0.4\textwidth]{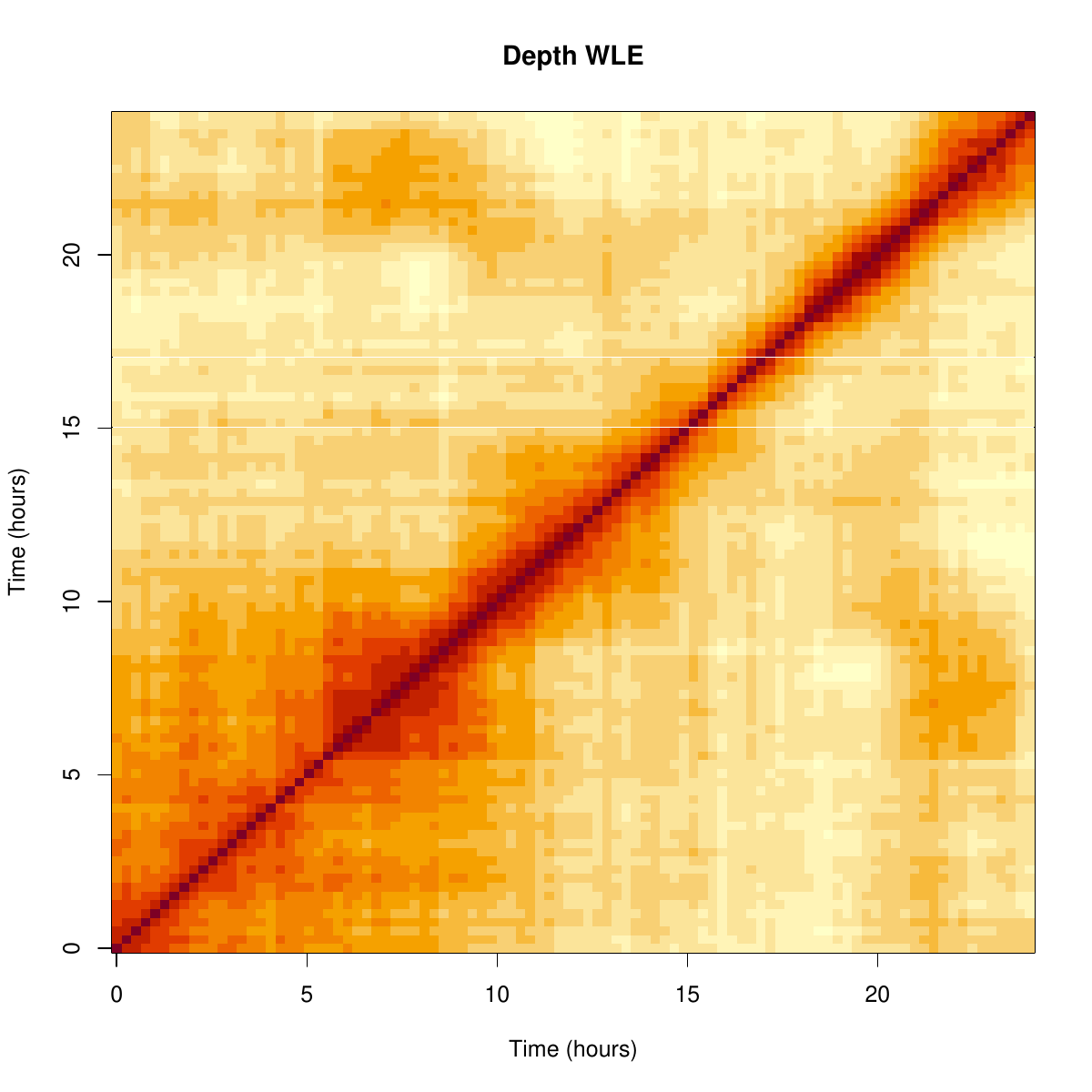} \\
\includegraphics[width=0.4\textwidth]{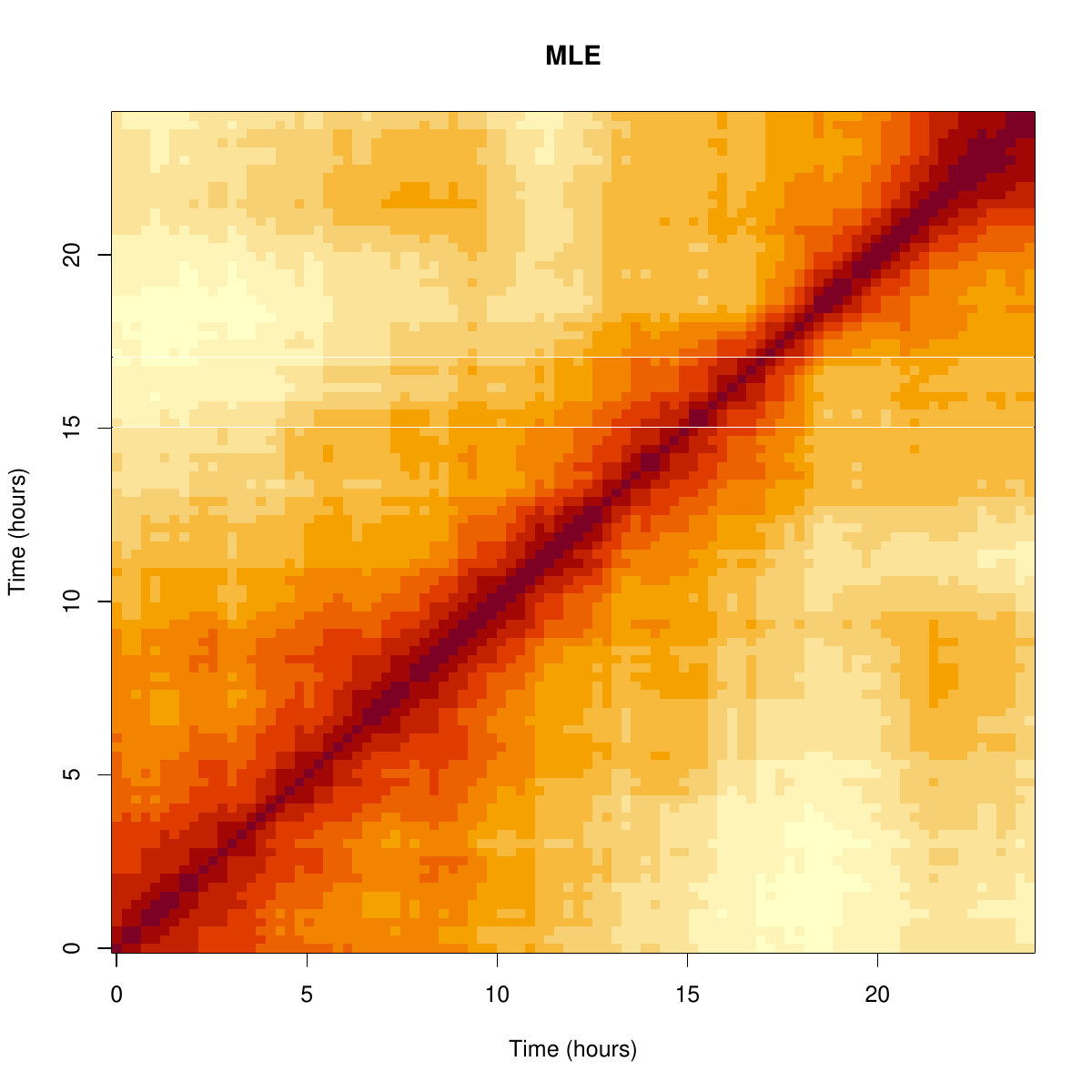}
\includegraphics[width=0.4\textwidth]{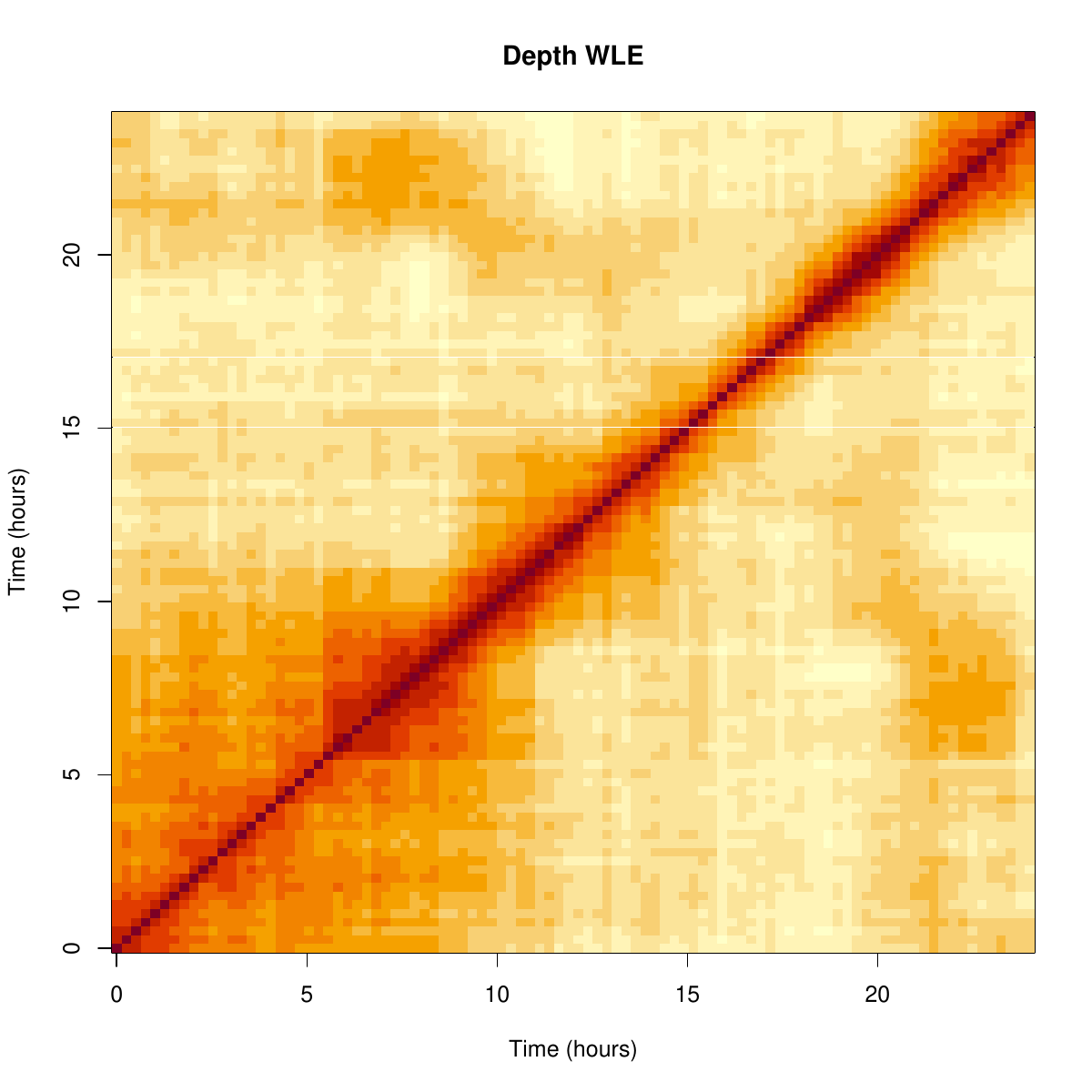} \\
\caption{Wind Data. Estimates correlation matrix provided by MLE (First column) and weighted likelihood based on depth procedure (Second column). Contamination level; first row: $17.9\%$, second row: $22.5\%$ and third row: $24.6\%$.}
\label{sup:fig:wind-cor-2}
\end{figure}  

\section{Monte Carlo simulations}
\label{sup:sec:monte}




In these simulations, data are sampled from a multivariate standard normal distribution, and outliers come from a multivariate $\mathcal{N}(\boldsymbol{\mu}, \sigma^2 I)$, where $\boldsymbol{\mu} = (\mu, \dots, \mu)^\top$. We refer to $\mu$ as contamination average (or location) and to $\sigma$ as contamination standard deviation (abbr. sd, or scale). See the main article for details.


\paragraph{More details on computation}

Consider a multivariate normal model $F_\theta$ with $\theta = (\mu, \Sigma)$, where $\mu$ is the mean vector and $\Sigma$ is the covariance matrix. We illustrate how to evaluate the proposed Depth Pearson Residuals \eqref{equ:dpr} using the half-space depth, but, firstly, we recall the definition of half-space depth \citep{tukey1975,donoho1992} and review some properties of interest.
Let $\operatorname{HS}_{u}(a)$ be the closed half-space $\left\{z \in \mathbb{R}^p: u^\top z \geq a \right\}$. Here $u$ is a $p$-vector satisfying $u^\top u = 1$. Note that $\operatorname{HS}_{u}(a)$ is the positive side of the hyperplane $\operatorname{HP}_{u}(a) = \left\{ z \in \mathbb{R}^p : u^\top z = a \right\}$. The negative side $\operatorname{HS}_{-u}(a)$ of $\operatorname{HP}_{u}(a)$ is similarly defined.
Intuitively, the half-space depth of a point $x \in \mathbb{R}^p$ w.r.t. a probability distribution is the minimum probability of all half-spaces including $x$ on their boundary. Formally, the half-space depth $d_{HS}(x;X)$ maps $x \in \mathbb{R}^p$ to the minimum probability, according to the random vector $X$, of all closed half-spaces including $x$, that is
\begin{align*}
d_{HS}(x; X) & = \inf_{u: u^\top u=1} \Pr(\operatorname{HS}_{u}(u^\top x); X) \\
& = \inf_{u: u^\top u=1} \Pr( z \in \mathbb{R}^p : u^\top z \geq u^\top x; X) \ .
\end{align*}

\citet[][Theorem 3.3, Corollary 4.3]{zuo2000b} show that the half-space depth for a multivariate normal model can be easily obtained since $d_{HS}(x; F_\theta) = (1 - F_{\chi^2_p}(\Delta(x;\theta)))/2$ where $\Delta(x;\theta) = (x - \mu)^\top \Sigma^{-1}(x - \mu)$ is the squared Mahalanobis distance and $F_{\chi^2_p}$ is the distribution function of a $\chi^2_p$ chi-squared with $p$ degrees of freedom random variable.

Once the population and empirical depth are evaluated, the DPRs are easily computed by Equation~\eqref{equ:dpr}.


We consider two different experiments. In the first experiment we study the performance of the proposed method as a function of the weight function parameters for different values of $\alpha$. The implemented algorithm is started at the true values, and we use $N=100$ Monte Carlo replications for each combination of the factors. 
In the second experiment we explore two different strategies to start the implemented algorithm: (i) a classic subsampling technique with subsample of size $p + p \times (p + 1)/2 + 1$ and $B = 500$ initial samples, or (ii) a deterministic approach where we use the observation with the highest depth as initial value for the location and the sample covariance of the $50\%$ of the observations with the highest depth for the shape. 
Again we use $N=100$ Monte Carlo replications for each combination of the factors. Table \ref{sup:tab:monte:n} reports the number of parameters to be estimated (first column) as function of the dimension $p$ and the sample size $n=s (p (p+1)/2+p)$ as a function of both dimension $p$ and sample size factor $s$.

\begin{table}[ht]
\centering
\begin{tabular}{rr|rrr}
  \hline
 & N. of parameters & \multicolumn{3}{c}{s} \\
 &  & 2 & 3 & 4 \\ 
  \hline
$p=1$ & 2 & 4 & 10 & 20 \\ 
  2 & 5 & 10 & 25 & 50 \\ 
  5 & 20 & 40 & 100 & 200 \\ 
 10 & 65 & 130 & 325 & 650 \\ 
 20 & 230 & 460 & 1150 & 2300 \\ 
   \hline
\end{tabular}
\caption{Simulation setting. Number of parameters to be estimated (first column) and sample size $n=s(p(p+1)/2+p)$ as function of the number of variables $p$ and sample size factor $s$.}
\label{sup:tab:monte:n}
\end{table}

Figures \ref{sup:fig:monte:MSE} and \ref{sup:fig:monte:DIV}, regarding the results of the first experiment, report the quantiles of the Mean Square Error and Kullback--Leibler Divergence respectively, as a function of the weight parameters $\gamma$, $\delta$, $\delta_1$, $\delta_2$ and for different values of $\alpha=0.25, 0.5, 0.75, 1$. As it easy to see, the influence of the weight parameters on the performance of the proposed method is extremely limited. 
Hence, for all further simulations we consider a fixed set of parameters depending on $\alpha$ which are summarized in Table \ref{tab:monte:optimal} of the main document.

We now comment the performance of the proposed method when it is started from the true values. This allows us to evaluate the performance in, somehow, the best situation without having to deal with multiple solutions of the estimating equations. Table \ref{sup:tab:eff} compares the efficiency of the proposed method with different values of $\alpha$ in the uncontaminated case with respect to maximum likelihood estimators. In the last rows the efficiency of the constrained M-Estimates of multivariate location and scatter based on the translated biweight function using a high breakdown point initial estimate \citep[\texttt{CovMest},][]{woodruff1994,rocke1996a,rocke1996b,todorov2009} is also reported for comparison reason.
\begin{table}[ht]
\centering
\begin{tabular}{rrrr}
  \hline  
$s$ & 2 & 5 & 10 \\
  \hline
  & \multicolumn{3}{c}{$\alpha=0.25$} \\
  \hline
$p=1$& 1.000 & 1.000 & 1.000 \\ 
  2 & 0.931 & 1.000 & 0.892 \\ 
  5 & 0.985 & 1.125 & 1.076 \\ 
 10 & 1.000 & 0.939 & 0.904 \\ 
 20 & 0.943 & 0.936 & 1.000 \\
  \hline
  & \multicolumn{3}{c}{$\alpha=0.5$} \\
  \hline
$p=1$& 0.933 & 0.916 & 1.441 \\ 
  2 & 0.880 & 1.000 & 0.925 \\ 
  5 & 0.985 & 1.120 & 1.029 \\ 
 10 & 1.000 & 0.926 & 0.904 \\ 
 20 & 0.942 & 0.936 & 1.000 \\
  \hline
  & \multicolumn{3}{c}{$\alpha=0.75$} \\
  \hline
$p=1$& 1.000 & 0.889 & 1.441 \\ 
  2 & 0.921 & 1.000 & 0.925 \\ 
  5 & 0.977 & 1.111 & 1.029 \\ 
 10 & 1.000 & 0.939 & 0.904 \\ 
 20 & 0.942 & 0.936 & 0.998 \\
  \hline
  & \multicolumn{3}{c}{$\alpha=1$} \\
  \hline
$p=1$& 1.406 & 0.848 & 2.502 \\ 
  2 & 0.802 & 1.000 & 0.892 \\ 
  5 & 1.000 & 1.129 & 1.112 \\ 
 10 & 1.000 & 0.914 & 0.912 \\ 
 20 & 0.958 & 0.939 & 1.000 \\ 
  \hline
  & \multicolumn{3}{c}{\texttt{CovMest}} \\
  \hline
$p=1$ & 1.256 & 0.850 & 1.561 \\ 
  2 & 1.056 & 1.565 & 1.101 \\ 
  5 & 1.395 & 1.310 & 1.332 \\ 
 10 & 1.110 & 0.979 & 0.975 \\ 
 20 & 0.975 & 1.013 & 1.133 \\ 
\hline
\end{tabular}
\caption{Monte Carlo simulation. Estimated efficiency of Maximum Likehood method with respect to the proposed method, starting from the true values, for different values of $\alpha$ and \texttt{CovMest}, number of variables $p$ and sample size factor $s$.}
\label{sup:tab:eff}
\end{table}
We can conclude that the efficiency is very high and comparable to that of maximum likelihood estimators for all the considered estimators. We now turn our attention to the finite sample robust performance of the methods. Tables \ref{sup:tab:MSE:max:1}--\ref{sup:tab:MSE:max:2} and \ref{sup:tab:DIV:max:1}--\ref{sup:tab:DIV:max:2} report the maximum estimated Mean Square Errors and maximum estimated Kullback--Leibler divergence for the proposed method as a function of $\alpha$ parameter, dimension and sample size factor. Results are also reported in Figures \ref{sup:fig:monte:MSE:0.25:1}--\ref{sup:fig:monte:MSE:1:2} for the Mean Square Error and in Figures \ref{sup:fig:monte:DIV:0.25:1}--\ref{sup:fig:monte:DIV:1:2} for the Kullback--Leibler divergence. 
Values of $\alpha$ small (e.g. $\alpha=0.25$) seem to perform well for the location parameters in all conditions, while for the scatter parameters the performance is good only for small/moderate level of contamination ($\varepsilon \le 0.2$) or small dimension ($p \le 2$). 
For larger values of $\alpha \ge 0.5$ the performance is extremely good in all conditions with a small deterioration for large level of contamination $\varepsilon \ge 0.3$ as the dimension increases. 
A common pattern in all figures highlights the nice robust behavior of the proposed estimator: after some value of the contamination average $\mu$, outliers have no more impact on the estimated values (sharp decreasing in the curves). Overall the best performance is achieved for $\alpha=0.5$.


We now turn to the second experiment. 
First, we consider the performance of the proposed algorithm when started using a subsampling technique. 
Table \ref{sup:tab:monte:sub} reports the probability of sampling a subsample of size $p+(p+1)\times p/2 + 1$ free of outliers as a function of the dimension $p$, sample factor $s$ and level of contamination $\varepsilon$, when the sampling is performed without replacement.

\begin{table}[ht]
  \setlength{\tabcolsep}{2pt}  
  \centering
  \resizebox{\textwidth}{!}{
  \begin{tabular}{r|rrr|rrr|rrr}
    \hline
    $\varepsilon$ & \multicolumn{3}{c}{0.05} & \multicolumn{3}{c}{0.1} & \multicolumn{3}{c}{0.2} \\
    \hline  
    $s$ & 2 & 5 & 10 & 2 & 5 & 10 & 2 & 5 & 10 \\
    \hline
  $p$=1 &   1 &   1 & 0.85 &   1 & 0.7 & 0.72 & 0.25 & 0.47 & 0.49 \\ 
    2 &   1 & 0.76 & 0.77 & 0.4 & 0.57 & 0.51 & 0.13 & 0.22 & 0.24 \\ 
    5 & 0.22 & 0.3 & 0.32 & 0.042 & 0.083 & 0.096 & \cellcolor{red!25} 0.00098 & 0.0049 & 0.0069 \\ 
   10 & 0.013 & 0.024 & 0.03 & \cellcolor{red!25} 5E-05 & \cellcolor{red!25} 0.00046 & \cellcolor{red!25} 0.00064 & \cellcolor{red!25} 3.9E-10 & \cellcolor{red!25} 5.7E-08 & \cellcolor{red!25} 1.6E-07 \\ 
   20 & \cellcolor{red!25} 6E-08 & \cellcolor{red!25} 1.5E-06 & \cellcolor{red!25} 3.7E-06 & \cellcolor{red!25} 9.3E-16 & \cellcolor{red!25} 1.3E-12 & \cellcolor{red!25} 6.8E-12 & \cellcolor{red!25} 1.3E-33 & \cellcolor{red!25}4.5E-26 & \cellcolor{red!25} 1.8E-24 \\
    \hline
    $\varepsilon$ & \multicolumn{3}{c}{0.3} & \multicolumn{3}{c}{0.4} & \multicolumn{3}{c}{0.45} \\
    \hline   
  $p=1$ & 0.25 & 0.29 & 0.32 & \cellcolor{red!25} 0 & 0.17 & 0.19 & \cellcolor{red!25} 0 & 0.17 & 0.14 \\ 
    2 & 0.033 & 0.07 & 0.1 & 0.0048 & 0.028 & 0.037 & 0.0048 & 0.017 & 0.024 \\ 
    5 & \cellcolor{red!25} 9E-06 & \cellcolor{red!25} 0.00019 & \cellcolor{red!25} 0.00034 & \cellcolor{red!25} 1.5E-08 & \cellcolor{red!25} 3.9E-06 & \cellcolor{red!25} 1E-05 & \cellcolor{red!25} 1.7E-10 & \cellcolor{red!25} 4.1E-07 & \cellcolor{red!25} 1.4E-06 \\ 
   10 & \cellcolor{red!25} 1.7E-16 & \cellcolor{red!25} 1.7E-12 & \cellcolor{red!25} 1.3E-11 & \cellcolor{red!25} 4.6E-25 & \cellcolor{red!25} 1E-17 & \cellcolor{red!25} 2E-16 & \cellcolor{red!25} 1.7E-30 & \cellcolor{red!25} 9.9E-21 & \cellcolor{red!25} 4.1E-19 \\ 
   20 & \cellcolor{red!25} 8.2E-56 & \cellcolor{red!25} 1.1E-41 & \cellcolor{red!25} 7.5E-39 & \cellcolor{red!25} 1.2E-85 & \cellcolor{red!25} 3.6E-60 & \cellcolor{red!25} 1.2E-55 & \cellcolor{red!25} 2.3E-106 & \cellcolor{red!25} 5.3E-71 & \cellcolor{red!25} 3.1E-65 \\ 
     \hline
  \end{tabular}
  }
  \caption{Simulation setting. Probability of sampling a subsample of size $p+(p+1)\times p/2 + 1$ free of outliers as a function of the dimension $p$, sample size factor $s$ and level of contamination $\varepsilon$, when the sampling is performed without replacement. When the probability is smaller than $1/B=1/500$ the cell is colored in red.}
  \label{sup:tab:monte:sub}
\end{table}

Given the number of initial values sampled $B=500$, for certain combinations of $p$, $s$ and $\varepsilon$ we do not expect to be able to sample a subsample that is free of outliers. Those cases are marked in red.
In light of these probabilities, we do not expect satisfying results for the subsampling method for specific combinations of small sample size factor $s$, large number of variables $p$, and high contamination level $\varepsilon$. 

On the $N=100$ Monte Carlo replications we observed the presence of multiple roots, with higher frequency as $p$ and $\varepsilon$ increases. 
The roots are of two types: a robust root ``close'' to the true value and a maximum likelihood root. We are interested in the first type, so we count how many times we are able to retrieve this robust root. 
Figures \ref{sup:fig:subs-true-1}--\ref{sup:fig:subs-true-8b} show the results of the simulation.
When the contamination is small $\varepsilon \leq 0.1$ our procedure with subsampling finds the robust root almost always for $p<20$, with better performance for growing sample size factor $s$ and $\alpha \geq 0.5$. 

When the contamination is near ($\mu-$contamination near zero, light color in the plots) it is more likely that our procedure with subsampling finds the robust root, especially for small to moderate contamination scale (small bubbles). For far location contamination, instead, a small variance makes it harder to identify the robust root. This pattern is clearly visible, for instance, in Figure~\ref{sup:fig:subs-true-3}.

Finally, we consider starting the algorithm using a deterministic approach where initial values are selected using the depth function. Figures~\ref{sup:fig:depth-true-1}-\ref{sup:fig:depth-true-8b} summarize the counts of times the robust root has been identified in the $N=100$ Monte Carlo simulations. Performance are slightly worse for small $p, s$ and $\alpha$, but we observe an improvement w.r.t. subsampling for moderate to large contamination $\varepsilon \geq 0.2$, when the $p \geq 2$ and in particular for $s>2$ and $\alpha \in \{0.25, 0.5\}$, see Figures~\ref{sup:fig:subs-true-2} and \ref{sup:fig:depth-true-2} for a comparison, or for $2<p<20$ independently of $\alpha$. 
This suggests that starting from the deepest point (or sample covariance of deepest points) is a good option when the sample size $n$ is not very large and the contamination level is moderate to high, since in this case a deterministic initialization achieved results which are comparable to the more computation demanding bootstrap with many subsamples.

For the sake of comparison, the same results are available for the \texttt{CovMest} method: Tables \ref{sup:tab:MSE:max:CovMest} and \ref{sup:tab:DIV:max:CovMest} report, respectively, the estimated Mean Square Error and Kullback--Leibler divergence for varying $\varepsilon$ and different combinations of $p$ and $s$. Results are also shown in Figures \ref{sup:fig:monte:MSE:M:1}, \ref{sup:fig:monte:MSE:M:2}, \ref{sup:fig:monte:DIV:M:1} and \ref{sup:fig:monte:DIV:M:2}.


\begin{table}[ht]
\centering
\begin{tabular}{lrrrrrrrrr}
  \hline
  $\varepsilon$ & \multicolumn{3}{c}{0.05} & \multicolumn{3}{c}{0.1} & \multicolumn{3}{c}{0.2} \\
  \hline  
  $s$ & 2 & 5 & 10 & 2 & 5 & 10 & 2 & 5 & 10 \\
  \hline
  & \multicolumn{9}{c}{$\alpha=0.25$} \\
  \hline
$p=1$ & 0.11 & 0.07 & 0.04 & 0.11 & 0.11 & 0.08 & 0.26 & 0.27 & 0.59 \\ 
  2 & 0.08 & 0.04 & 0.02 & 0.10 & 0.06 & 0.10 & 0.18 & 0.36 & 0.37 \\ 
  5 & 0.03 & 0.02 & 0.02 & 0.06 & 0.05 & 0.09 & 0.18 & 0.34 & 0.36 \\ 
 10 & 0.01 & 0.01 & 0.01 & 0.05 & 0.04 & 0.04 & 0.17 & 0.16 & 0.16 \\ 
 20 & 0.01 & 0.01 & 0.01 & 0.04 & 0.04 & 0.04 & 0.16 & 0.16 & 0.16 \\ 
  \hline
  & \multicolumn{9}{c}{$\alpha=0.5$} \\
  \hline
$p=1$ & 0.10 & 0.06 & 0.05 & 0.10 & 0.10 & 0.06 & 0.15 & 0.15 & 0.13 \\ 
  2 & 0.09 & 0.04 & 0.02 & 0.09 & 0.05 & 0.05 & 0.11 & 0.17 & 0.17 \\ 
  5 & 0.03 & 0.01 & 0.01 & 0.03 & 0.03 & 0.04 & 0.06 & 0.16 & 0.17 \\ 
 10 & 0.01 & 0.01 & 0.00 & 0.02 & 0.01 & 0.04 & 0.05 & 0.16 & 0.16 \\ 
 20 & 0.01 & 0.00 & 0.00 & 0.01 & 0.01 & 0.01 & 0.04 & 0.04 & 0.04 \\
  \hline
  & \multicolumn{9}{c}{$\alpha=0.75$} \\
  \hline  
$p=1$ & 0.14 & 0.06 & 0.04 & 0.14 & 0.10 & 0.05 & 0.26 & 0.14 & 0.15 \\ 
  2 & 0.09 & 0.04 & 0.02 & 0.09 & 0.05 & 0.05 & 0.18 & 0.18 & 0.17 \\ 
  5 & 0.03 & 0.02 & 0.01 & 0.03 & 0.05 & 0.04 & 0.11 & 0.17 & 0.17 \\ 
 10 & 0.01 & 0.01 & 0.01 & 0.02 & 0.04 & 0.04 & 0.15 & 0.16 & 0.16 \\ 
 20 & 0.01 & 0.00 & 0.00 & 0.01 & 0.01 & 0.01 & 0.04 & 0.04 & 0.04 \\
  \hline
  & \multicolumn{9}{c}{$\alpha=1$} \\
  \hline
$p=1$ & 0.19 & 0.05 & 0.05 & 0.19 & 0.13 & 0.07 & 0.27 & 0.22 & 0.15 \\ 
  2 & 0.09 & 0.04 & 0.02 & 0.10 & 0.05 & 0.05 & 0.11 & 0.17 & 0.16 \\ 
  5 & 0.03 & 0.01 & 0.01 & 0.03 & 0.02 & 0.04 & 0.06 & 0.15 & 0.16 \\ 
 10 & 0.01 & 0.01 & 0.00 & 0.02 & 0.01 & 0.01 & 0.05 & 0.04 & 0.14 \\ 
 20 & 0.01 & 0.00 & 0.00 & 0.01 & 0.01 & 0.01 & 0.04 & 0.04 & 0.04 \\ 
   \hline
\end{tabular}
\caption{Monte Carlo simulation. Estimated maximum Mean Square Error  for the proposed method, starting from the true values, under contamination ($\varepsilon = 0.05, 0.1, 0.2$) for different values of $\alpha$, number of variables $p$ and sample size factor $s$.}
\label{sup:tab:MSE:max:1}
\end{table}

\begin{table}[ht]
\centering
\begin{tabular}{lrrrrrrrrr}
  \hline
  $\varepsilon$ & \multicolumn{3}{c}{0.3} & \multicolumn{3}{c}{0.4} & \multicolumn{3}{c}{0.45} \\
  \hline    
  $s$ & 2 & 5 & 10 & 2 & 5 & 10 & 2 & 5 & 10 \\
  \hline
  & \multicolumn{9}{c}{$\alpha=0.25$} \\
  \hline  
$p=1$ & 0.26 & 0.80 & 1.47 & 0.95 & 1.49 & 2.59 & 0.95 & 1.49 & 3.31 \\ 
  2 & 0.39 & 0.90 & 0.82 & 0.70 & 1.44 & 1.47 & 0.70 & 1.75 & 1.77 \\ 
  5 & 0.39 & 0.83 & 0.82 & 0.67 & 1.46 & 1.45 & 0.86 & 1.83 & 1.83 \\ 
 10 & 0.37 & 0.37 & 0.36 & 0.66 & 0.65 & 1.42 & 0.81 & 1.73 & 1.80 \\ 
 20 & 0.36 & 0.36 & 0.36 & 0.64 & 0.64 & 0.64 & 0.81 & 0.81 & 0.81 \\ 
  \hline
  & \multicolumn{9}{c}{$\alpha=0.5$} \\
  \hline  
$p=1$ & 0.15 & 0.19 & 0.35 & 0.66 & 0.48 & 0.66 & 0.66 & 0.48 & 1.52 \\ 
  2 & 0.14 & 0.40 & 0.37 & 0.25 & 0.68 & 0.67 & 0.25 & 0.87 & 0.81 \\ 
  5 & 0.13 & 0.37 & 0.37 & 0.53 & 0.68 & 0.66 & 0.78 & 0.90 & 0.85 \\ 
 10 & 0.28 & 0.37 & 0.36 & 0.62 & 0.65 & 0.64 & 0.78 & 0.81 & 0.81 \\ 
 20 & 0.09 & 0.09 & 0.09 & 0.17 & 0.16 & 0.16 & 0.25 & 0.22 & 0.21 \\ 
  \hline
  & \multicolumn{9}{c}{$\alpha=0.75$} \\
  \hline  
$p=1$ & 0.26 & 0.38 & 0.40 & 1.01 & 0.69 & 0.67 & 1.01 & 0.69 & 0.85 \\ 
  2 & 0.38 & 0.43 & 0.37 & 0.71 & 0.68 & 0.67 & 0.71 & 0.85 & 0.80 \\ 
  5 & 0.34 & 0.38 & 0.37 & 0.63 & 0.66 & 0.65 & 0.84 & 0.84 & 0.83 \\ 
 10 & 0.35 & 0.37 & 0.36 & 0.64 & 0.65 & 0.64 & 0.78 & 0.81 & 0.81 \\ 
 20 & 0.09 & 0.09 & 0.09 & 0.18 & 0.17 & 0.16 & 0.67 & 0.54 & 0.22 \\ 
  \hline
  & \multicolumn{9}{c}{$\alpha=1$} \\
  \hline  
$p=1$ & 0.27 & 0.28 & 0.35 & 1.00 & 0.94 & 0.68 & 1.00 & 0.94 & 0.87 \\ 
  2 & 0.15 & 0.43 & 0.38 & 0.28 & 0.74 & 0.70 & 0.28 & 0.94 & 0.85 \\ 
  5 & 0.11 & 0.35 & 0.36 & 0.46 & 0.66 & 0.65 & 0.93 & 0.81 & 0.82 \\ 
 10 & 0.10 & 0.10 & 0.36 & 0.48 & 0.61 & 0.63 & 0.60 & 0.80 & 0.80 \\ 
 20 & 0.10 & 0.09 & 0.09 & 0.67 & 0.63 & 0.61 & 0.54 & 0.54 & 0.82 \\ 
   \hline
\end{tabular}
\caption{Monte Carlo simulation. Estimated maximum Mean Square Error for the proposed method, starting from the true values, under contamination ($\varepsilon = 0.3, 0.4, 0.45$) for different values of $\alpha$, number of variables $p$ and sample size factor $s$.}
\label{sup:tab:MSE:max:2}
\end{table}


\begin{table}[ht]
\centering
\begin{tabular}{rrrrrrrrrr}
  \hline
  $\varepsilon$ & \multicolumn{3}{c}{0.05} & \multicolumn{3}{c}{0.1} & \multicolumn{3}{c}{0.2} \\
  \hline  
  $s$ & 2 & 5 & 10 & 2 & 5 & 10 & 2 & 5 & 10 \\
  \hline
$p=1$ & 0.21 & 0.06 & 0.04 & 0.21 & 0.11 & 0.04 & 0.21 & 0.14 & 0.13 \\ 
  2 & 0.12 & 0.06 & 0.03 & 0.16 & 0.06 & 0.04 & 0.29 & 0.34 & 0.37 \\ 
  5 & 0.04 & 0.01 & 0.01 & 0.05 & 0.02 & 0.02 & 0.78 & 0.46 & 0.47 \\ 
 10 & 0.01 & 0.01 & 0.00 & 0.02 & 0.01 & 0.01 & 0.50 & 0.44 & 0.43 \\ 
 20 & 0.01 & 0.00 & 0.00 & 0.01 & 0.01 & 0.04 & 0.45 & 0.42 & 0.19 \\
  \hline
  $\varepsilon$ & \multicolumn{3}{c}{0.3} & \multicolumn{3}{c}{0.4} & \multicolumn{3}{c}{0.45} \\
  \hline  
$p=1$ & 0.21 & 0.42 & 0.50 & 25.60 & 2.10 & 2.22 & 25.60 & 2.10 & 5.52 \\ 
  2 & 2.84 & 9.19 & 5.22 & 43.37 & 45.04 & 49.19 & 43.37 & 60.23 & 70.93 \\ 
  5 & 25.91 & 4.69 & 3.22 & 48.30 & 52.74 & 31.51 & 61.40 & 72.14 & 70.57 \\ 
 10 & 16.52 & 4.43 & 2.98 & 53.76 & 27.66 & 25.64 & 68.67 & 68.86 & 40.89 \\ 
 20 & 12.69 & 11.78 & 11.59 & 30.29 & 24.98 & 24.36 & 73.77 & 39.39 & 37.07 \\
   \hline
\end{tabular}
\caption{Monte Carlo simulation. Estimated maximum Mean Square Error for \texttt{CovMest} method under contamination ($\varepsilon = 0.05, 0.1, 0.2, 0.3, 0.4, 0.45$) for different number of variables $p$ and sample size factor $s$.}
\label{sup:tab:MSE:max:CovMest}
\end{table}


\begin{table}[ht]
\centering
\begin{tabular}{rrrrrrrrrr}
  \hline
  $\varepsilon$ & \multicolumn{3}{c}{0.05} & \multicolumn{3}{c}{0.1} & \multicolumn{3}{c}{0.2} \\
  \hline  
  $s$ & 2 & 5 & 10 & 2 & 5 & 10 & 2 & 5 & 10 \\
  \hline
  & \multicolumn{9}{c}{$\alpha=0.25$} \\
  \hline
$p=1$ & 0.23 & 0.07 & 0.15 & 0.23 & 0.14 & 0.31 & 0.66 & 0.37 & 0.99 \\ 
  2 & 0.29 & 0.20 & 0.17 & 0.35 & 0.44 & 0.62 & 0.55 & 1.32 & 1.49 \\ 
  5 & 0.59 & 0.41 & 0.97 & 1.06 & 0.93 & 2.45 & 2.08 & 4.75 & 5.62 \\ 
 10 & 1.14 & 1.00 & 0.90 & 2.50 & 2.25 & 2.14 & 4.94 & 5.50 & 5.92 \\ 
 20 & 2.65 & 2.43 & 2.32 & 5.53 & 5.30 & 5.22 & 10.64 & 10.67 & 12.03 \\ 
  \hline
  & \multicolumn{9}{c}{$\alpha=0.5$} \\
  \hline
$p=1$ & 0.70 & 0.11 & 0.04 & 0.70 & 0.19 & 0.10 & 1.17 & 0.43 & 0.29 \\ 
  2 & 0.41 & 0.15 & 0.06 & 0.48 & 0.15 & 0.17 & 0.83 & 0.37 & 0.42 \\ 
  5 & 0.43 & 0.19 & 0.33 & 0.48 & 0.39 & 0.79 & 0.75 & 1.82 & 1.78 \\ 
 10 & 0.51 & 0.25 & 0.17 & 0.72 & 0.43 & 2.10 & 1.28 & 4.50 & 4.46 \\ 
 20 & 0.74 & 0.47 & 0.38 & 1.29 & 1.00 & 0.91 & 2.60 & 2.28 & 2.17 \\ 
  \hline
  & \multicolumn{9}{c}{$\alpha=0.75$} \\
  \hline
$p=1$ & 0.24 & 0.10 & 0.04 & 0.24 & 0.17 & 0.07 & 0.69 & 0.24 & 0.13 \\ 
  2 & 0.30 & 0.13 & 0.06 & 0.34 & 0.15 & 0.17 & 0.59 & 0.40 & 0.43 \\ 
  5 & 0.45 & 0.38 & 0.33 & 0.53 & 0.85 & 0.81 & 1.43 & 1.94 & 1.78 \\ 
 10 & 0.50 & 0.26 & 0.81 & 0.71 & 2.11 & 2.12 & 4.34 & 4.54 & 4.47 \\ 
 20 & 0.74 & 0.47 & 0.38 & 1.29 & 1.00 & 0.91 & 2.61 & 2.28 & 2.18 \\ 
  \hline
  & \multicolumn{9}{c}{$\alpha=1$} \\
  \hline
$p=1$ & 1.08 & 0.20 & 0.18 & 1.08 & 0.56 & 0.29 & 1.89 & 0.72 & 0.51 \\ 
  2 & 0.44 & 0.16 & 0.07 & 0.62 & 0.17 & 0.17 & 1.01 & 0.40 & 0.43 \\ 
  5 & 0.47 & 0.18 & 0.12 & 0.54 & 0.27 & 0.65 & 0.87 & 1.62 & 1.75 \\ 
 10 & 0.51 & 0.25 & 0.16 & 0.74 & 0.43 & 0.35 & 1.32 & 0.96 & 4.11 \\
 20 & 0.75 & 0.47 & 0.38 & 1.31 & 1.00 & 0.91 & 2.68 & 2.32 & 2.20 \\
\hline
\end{tabular} 
\caption{Monte Carlo simulation. Estimated maximum Kullback--Leibler divergence for the proposed method, starting from the true values, under contamination ($\varepsilon = 0.05, 0.1, 0.2$) for different values of $\alpha$, number of variables $p$ and sample size factor $s$.}
\label{sup:tab:DIV:max:1}
\end{table}

\begin{table}[ht]
\centering
\begin{tabular}{rrrrrrrrrr}
  \hline
  $\varepsilon$ & \multicolumn{3}{c}{0.3} & \multicolumn{3}{c}{0.4} & \multicolumn{3}{c}{0.45} \\
  \hline    
  $s$ & 2 & 5 & 10 & 2 & 5 & 10 & 2 & 5 & 10 \\
  \hline
  & \multicolumn{9}{c}{$\alpha=0.25$} \\
  \hline  
$p=1$ & 0.66 & 0.56 & 1.50 & 0.78 & 0.61 & 1.82 & 0.78 & 0.61 & 1.89 \\ 
  2 & 0.75 & 2.02 & 2.30 & 1.07 & 2.43 & 4.91 & 1.07 & 2.59 & 6.61 \\ 
  5 & 2.97 & 8.23 & 8.82 & 4.12 & 11.03 & 19.13 & 5.24 & 15.03 & 25.20 \\ 
 10 & 7.83 & 9.24 & 10.61 & 11.60 & 17.39 & 23.08 & 13.12 & 23.66 & 28.45 \\ 
 20 & 15.31 & 18.44 & 19.38 & 18.62 & 24.98 & 25.22 & 26.13 & 27.45 & 27.55 \\ 
  \hline
  & \multicolumn{9}{c}{$\alpha=0.5$} \\
  \hline  
$p=1$ & 1.17 & 0.54 & 0.35 & 1.90 & 0.82 & 0.52 & 1.90 & 0.82 & 0.57 \\ 
  2 & 1.23 & 0.64 & 0.63 & 1.83 & 0.90 & 0.78 & 1.83 & 1.04 & 0.85 \\ 
  5 & 1.36 & 2.69 & 2.63 & 2.95 & 3.55 & 4.31 & 3.30 & 4.04 & 5.52 \\ 
 10 & 5.58 & 6.38 & 6.23 & 7.35 & 7.44 & 7.30 & 7.74 & 7.73 & 7.58 \\ 
 20 & 4.09 & 3.69 & 3.54 & 6.07 & 5.34 & 5.50 & 12.51 & 6.65 & 8.36 \\ 
  \hline
  & \multicolumn{9}{c}{$\alpha=0.75$} \\
  \hline  
$p=1$ & 0.69 & 0.39 & 0.27 & 0.99 & 0.62 & 0.38 & 0.99 & 0.62 & 0.48 \\ 
  2 & 0.87 & 0.64 & 0.67 & 1.35 & 0.87 & 0.84 & 1.35 & 0.97 & 0.89 \\ 
  5 & 2.74 & 2.70 & 3.09 & 3.32 & 3.46 & 5.09 & 3.54 & 4.88 & 5.97 \\ 
 10 & 6.20 & 6.39 & 6.24 & 7.50 & 7.44 & 7.30 & 7.80 & 7.73 & 7.58 \\ 
 20 & 4.14 & 3.71 & 3.55 & 6.61 & 5.54 & 6.87 & 35.19 & 30.36 & 9.05 \\ 
  \hline
  & \multicolumn{9}{c}{$\alpha=1$} \\
  \hline  
$p=1$ & 1.89 & 0.99 & 0.76 & 5.26 & 3.30 & 4.77 & 5.26 & 3.30 & 7.03 \\ 
  2 & 1.50 & 0.73 & 0.66 & 2.79 & 1.40 & 0.88 & 2.79 & 9.12 & 1.85 \\ 
  5 & 1.55 & 2.62 & 2.56 & 25.04 & 3.17 & 3.09 & 44.36 & 4.70 & 3.20 \\ 
 10 & 2.25 & 1.67 & 6.18 & 11.23 & 7.08 & 7.24 & 58.43 & 11.15 & 10.80 \\ 
 20 & 4.56 & 3.91 & 3.67 & 30.57 & 25.36 & 24.30 & 13.65 & 13.57 & 43.28 \\ 
   \hline
\end{tabular}
\caption{Monte Carlo simulation. Estimated maximum Kullback--Leibler divergence for the proposed method, starting from the true values, under contamination ($\varepsilon = 0.3, 0.4, 0.45$) for different values of $\alpha$, number of variables $p$ and sample size factor $s$.}
\label{sup:tab:DIV:max:2}
\end{table}


\begin{table}[ht]
\centering
\begin{tabular}{rrrrrrrrrr}
  \hline
  $\varepsilon$ & \multicolumn{3}{c}{0.05} & \multicolumn{3}{c}{0.1} & \multicolumn{3}{c}{0.2} \\
  \hline  
  $s$ & 2 & 5 & 10 & 2 & 5 & 10 & 2 & 5 & 10 \\
  \hline
$p=1$ & 0.06 & 0.02 & 0.01 & 0.06 & 0.03 & 0.02 & 0.15 & 0.09 & 0.09 \\ 
  2 & 1.81 & 0.45 & 0.18 & 1.86 & 0.45 & 0.25 & 2.35 & 1.86 & 1.91 \\ 
  5 & 1.27 & 0.27 & 0.15 & 1.80 & 0.39 & 0.25 & 14.57 & 6.63 & 6.49 \\ 
 10 & 0.74 & 0.34 & 0.22 & 1.09 & 0.60 & 0.48 & 17.67 & 15.53 & 15.00 \\ 
 20 & 0.94 & 0.58 & 0.47 & 1.66 & 1.25 & 6.17 & 35.28 & 33.40 & 13.87 \\
  \hline
  $\varepsilon$ & \multicolumn{3}{c}{0.3} & \multicolumn{3}{c}{0.4} & \multicolumn{3}{c}{0.45} \\
  \hline 
$p=1$ & 0.15 & 0.39 & 0.37 & 2.61 & 1.44 & 1.56 & 2.61 & 1.44 & 2.66 \\ 
  2 & 15.90 & 44.70 & 22.27 & 133.02 & 189.47 & 181.54 & 133.02 & 238.58 & 289.30 \\ 
  5 & 342.52 & 50.82 & 33.48 & 454.35 & 506.65 & 228.86 & 390.53 & 645.47 & 566.42 \\ 
 10 & 401.55 & 106.91 & 69.60 & 1011.96 & 442.18 & 407.64 & 993.00 & 978.48 & 505.97 \\ 
 20 & 655.42 & 615.34 & 605.26 & 1026.02 & 863.24 & 837.07 & 2246.27 & 1064.64 & 1000.27 \\ 
   \hline
\end{tabular}
\caption{Monte Carlo simulation. Estimated maximum Kullback--Leibler divergence for \texttt{CovMest} method, under contamination ($\varepsilon = 0.05, 0.1, 0.2, 0.3, 0.4, 0.45$) for different number of variables $p$ and sample size factor $s$.}
\label{sup:tab:DIV:max:CovMest}
\end{table}

\clearpage

\begin{figure}
\centering
\includegraphics[width=0.45\textwidth]{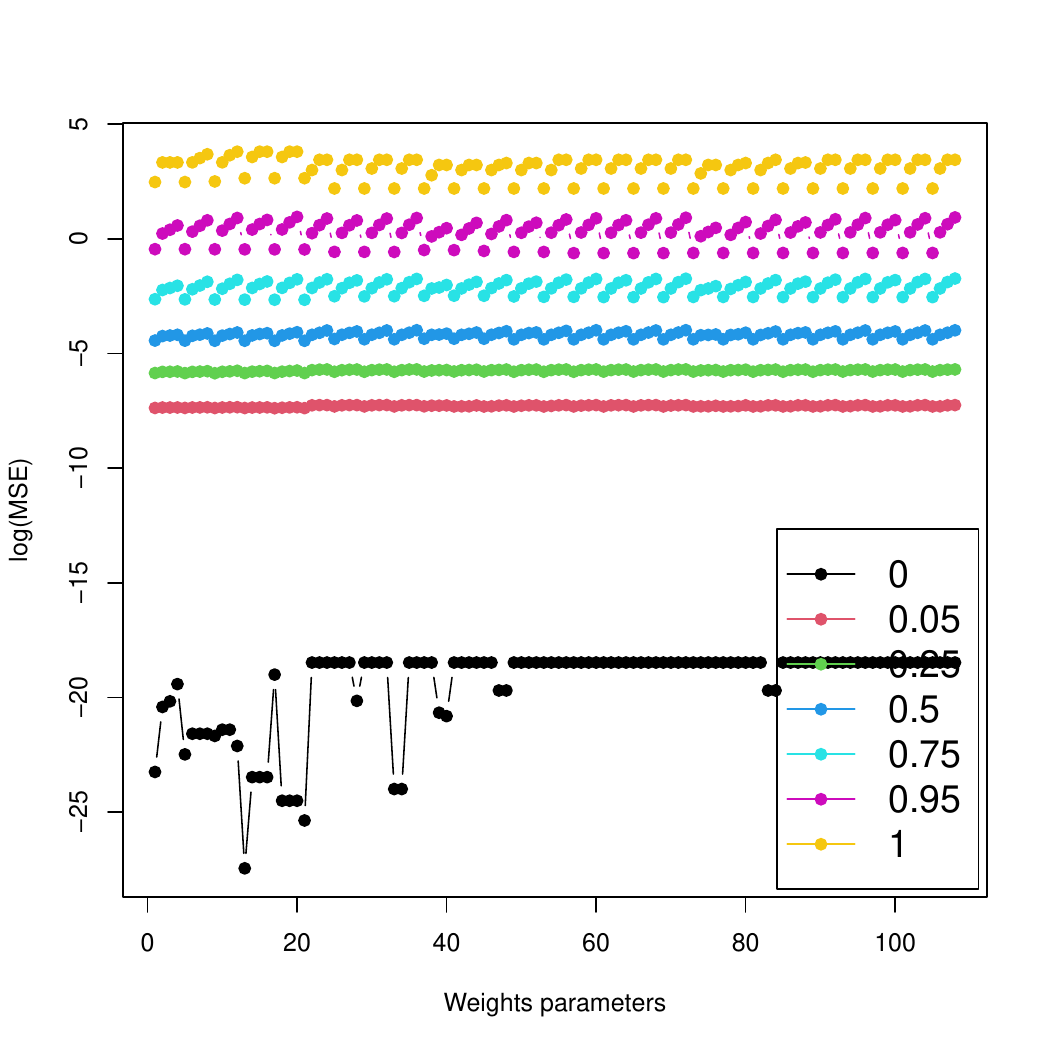}
\includegraphics[width=0.45\textwidth]{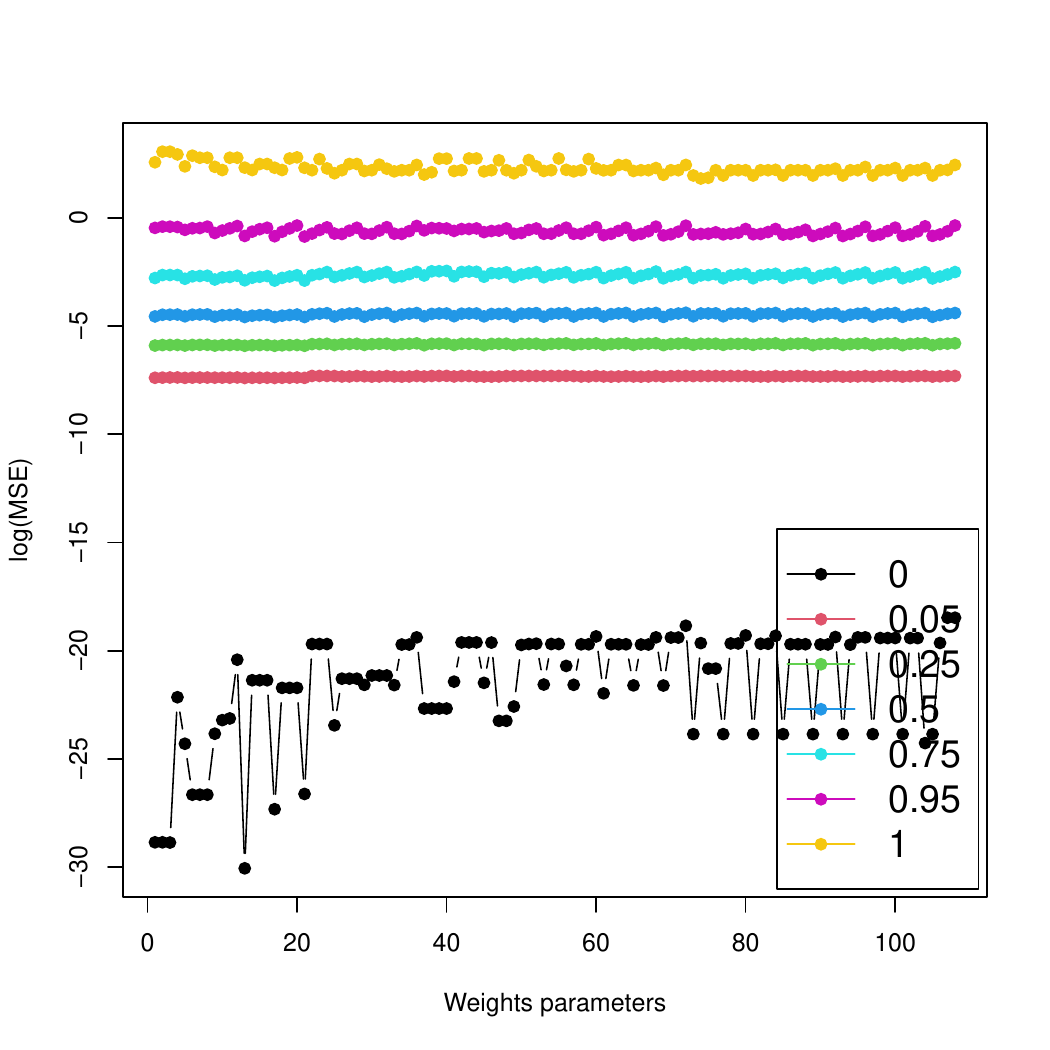} \\
\includegraphics[width=0.45\textwidth]{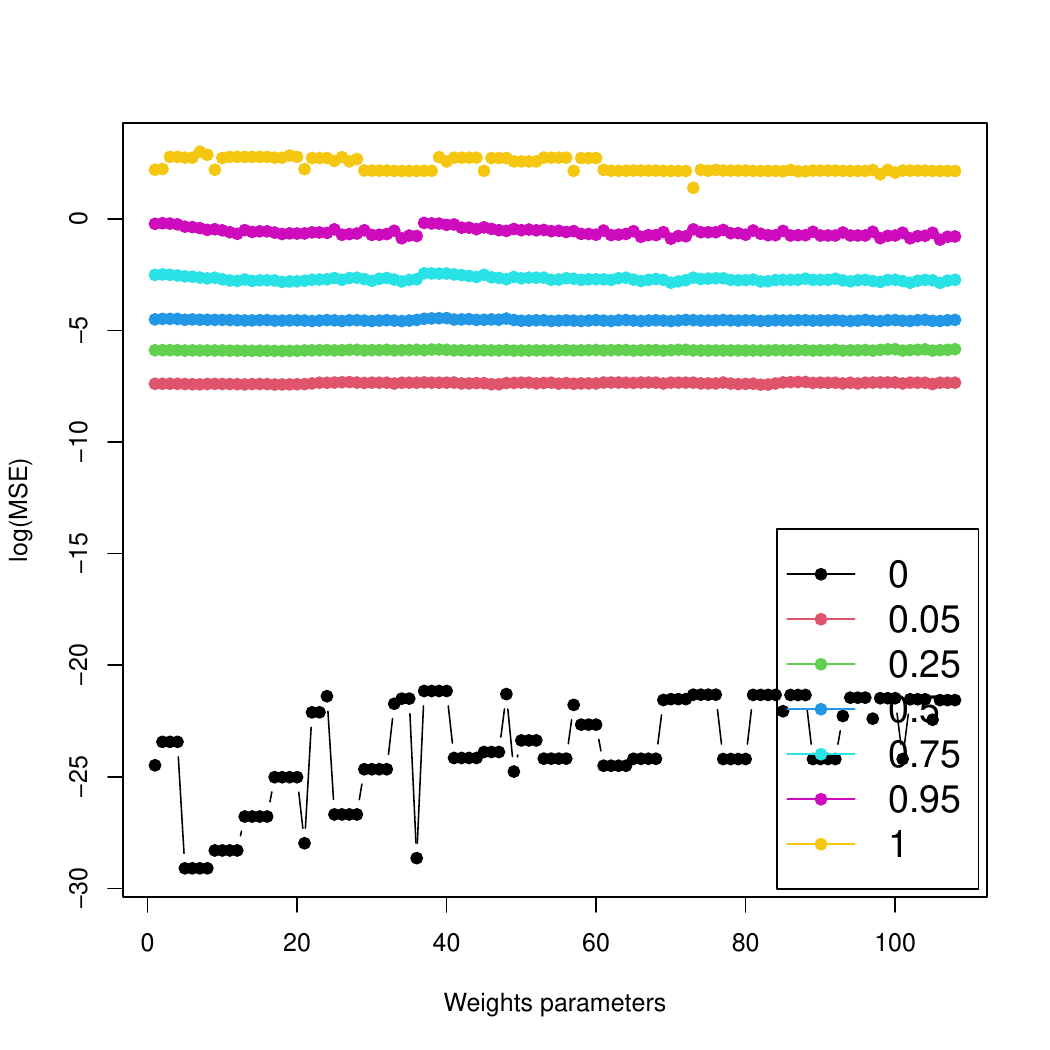}
\includegraphics[width=0.45\textwidth]{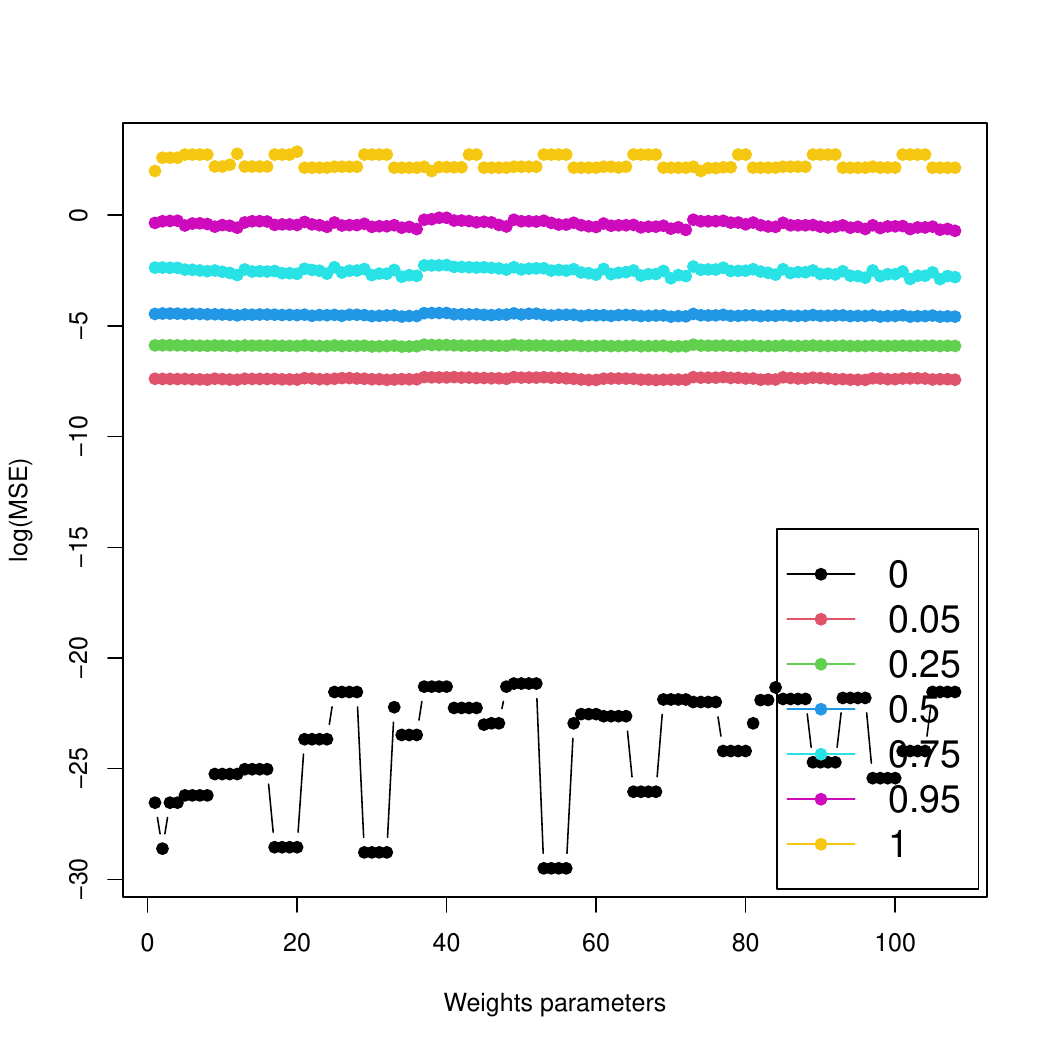}
\caption{Monte Carlo Simulation. Quantiles ($0$, $0.05$, $0.25$, $0.5$, $0.75$, $0.95$, $1$, see legenda) of the log Mean Square Error as a function of the weight function's parameters. Top left: $\alpha=0.25$, top right: $\alpha=0.5$, bottom left: $\alpha=0.75$, bottom right: $\alpha=1$.}
\label{sup:fig:monte:MSE}
\end{figure}  

\begin{figure}
\centering
\includegraphics[width=0.45\textwidth]{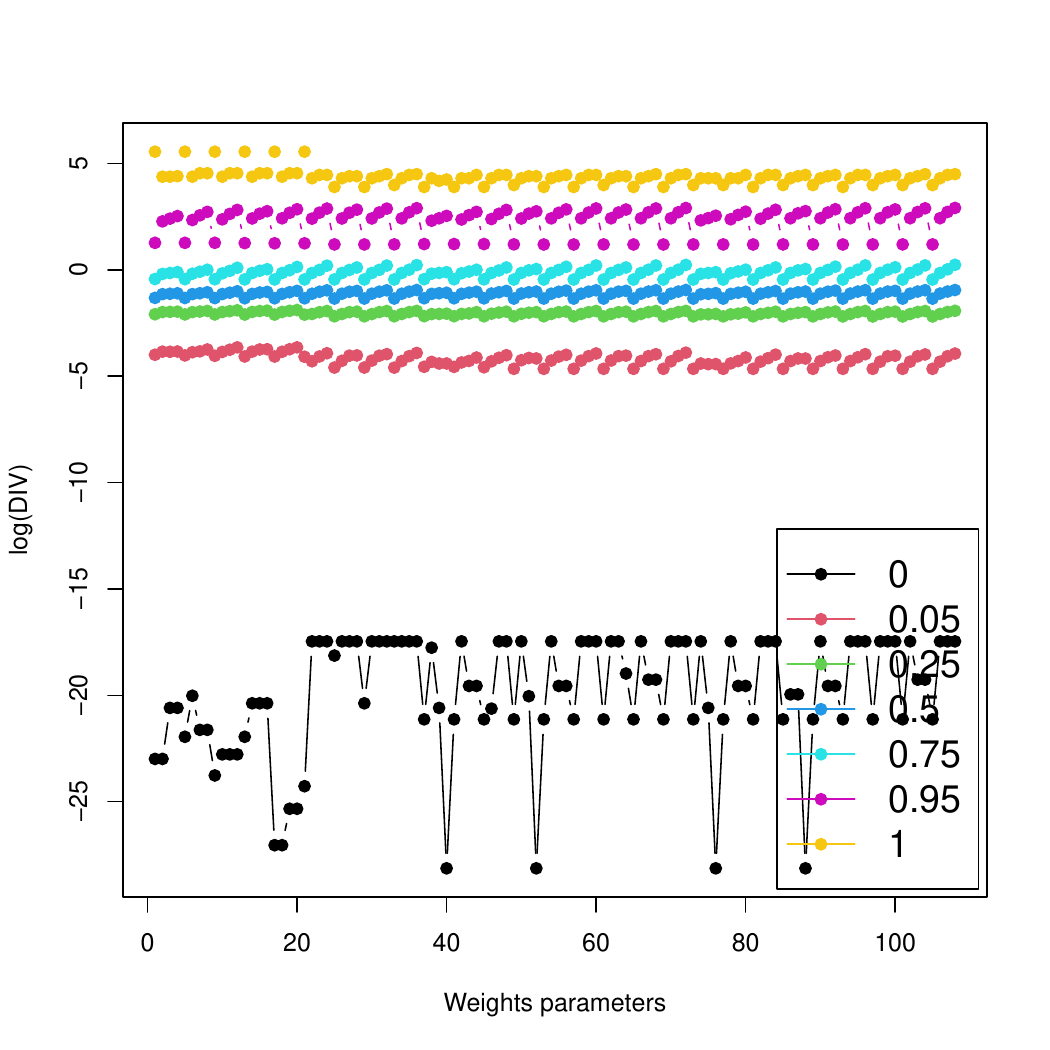}
\includegraphics[width=0.45\textwidth]{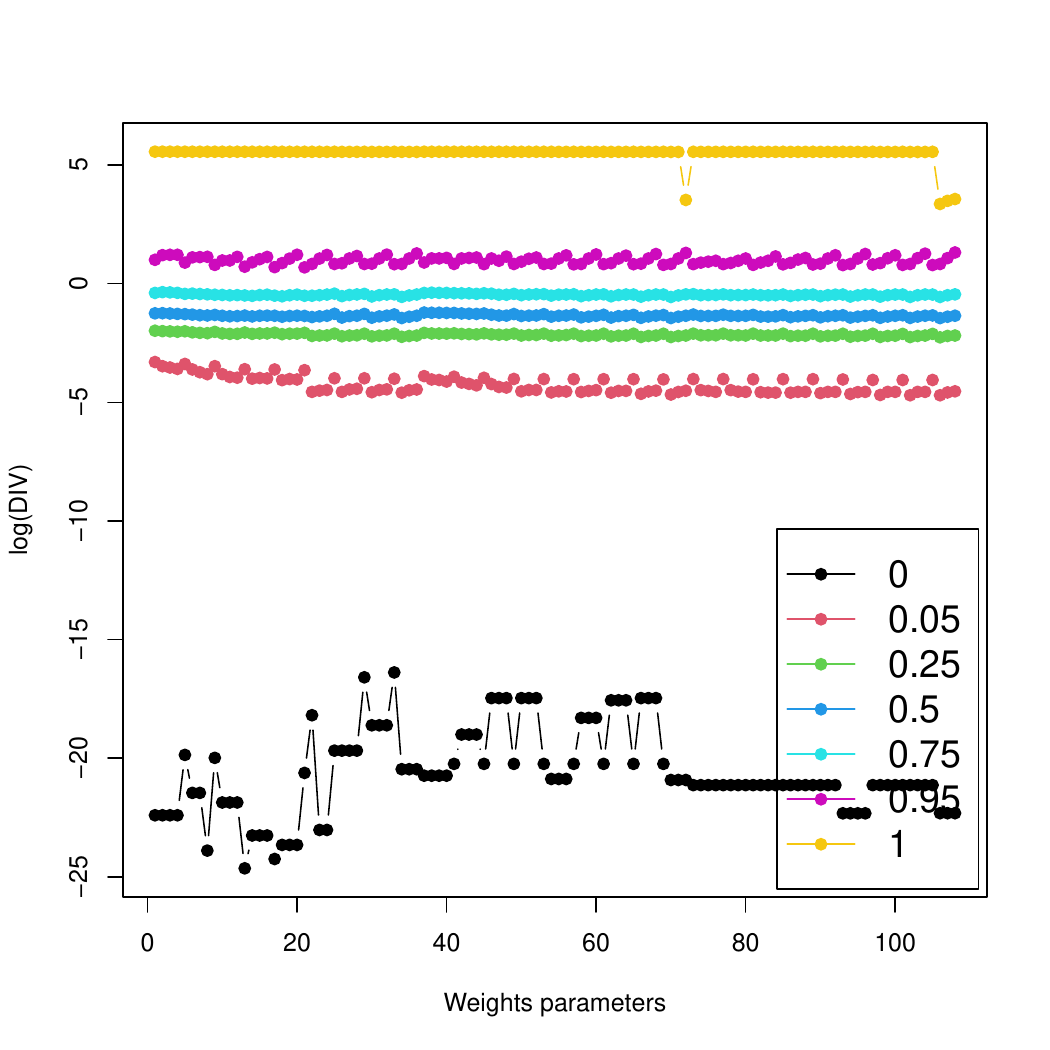} \\
\includegraphics[width=0.45\textwidth]{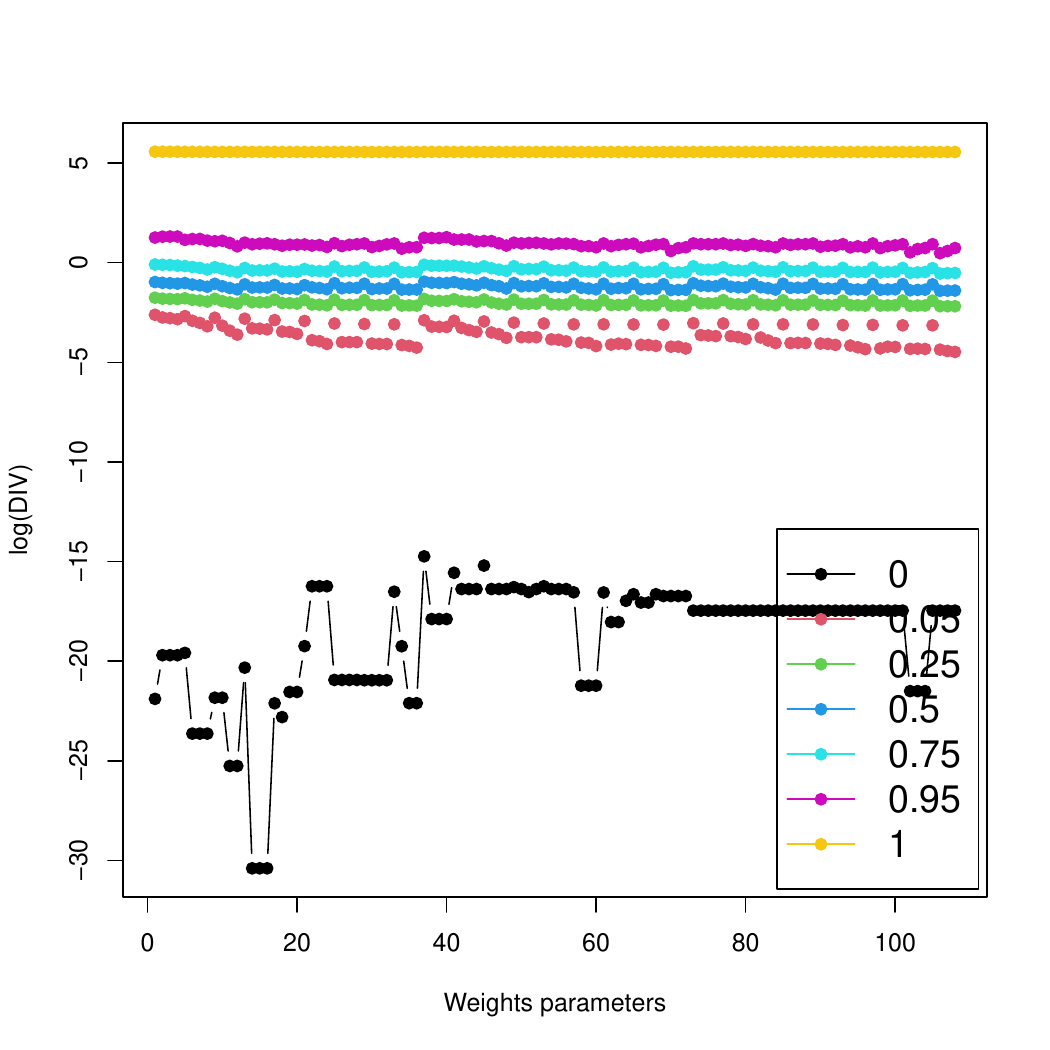}
\includegraphics[width=0.45\textwidth]{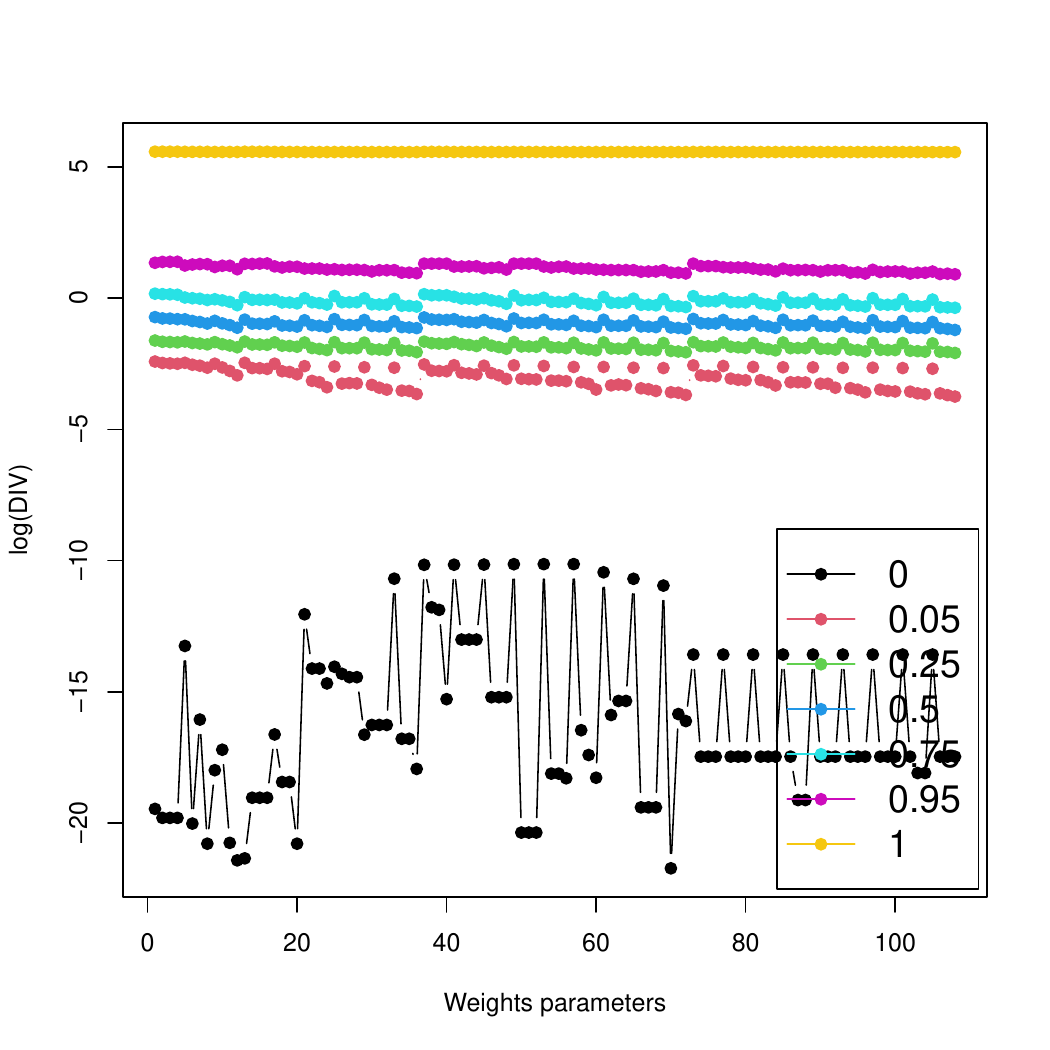}
\caption{Monte Carlo Simulation.  Quantiles ($0$, $0.05$, $0.25$, $0.5$, $0.75$, $0.95$, $1$, see legenda) of the log Kullback--Leibler Divergence as a function of the weight function's parameters. Top left: $\alpha=0.25$, top right: $\alpha=0.5$, bottom left: $\alpha=0.75$, bottom right: $\alpha=1$.}
\label{sup:fig:monte:DIV}
\end{figure}  


\begin{figure}
\centering
\includegraphics[width=0.32\textwidth]{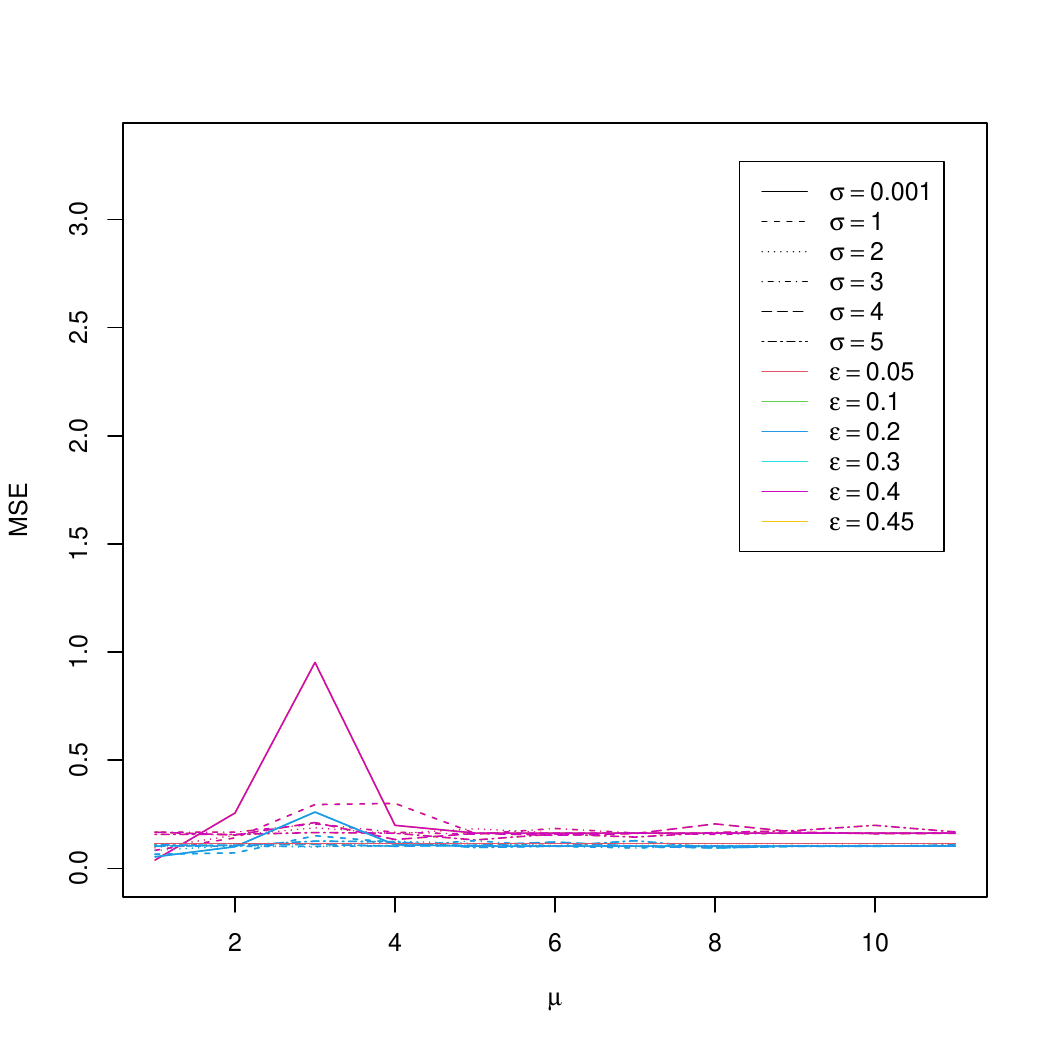}
\includegraphics[width=0.32\textwidth]{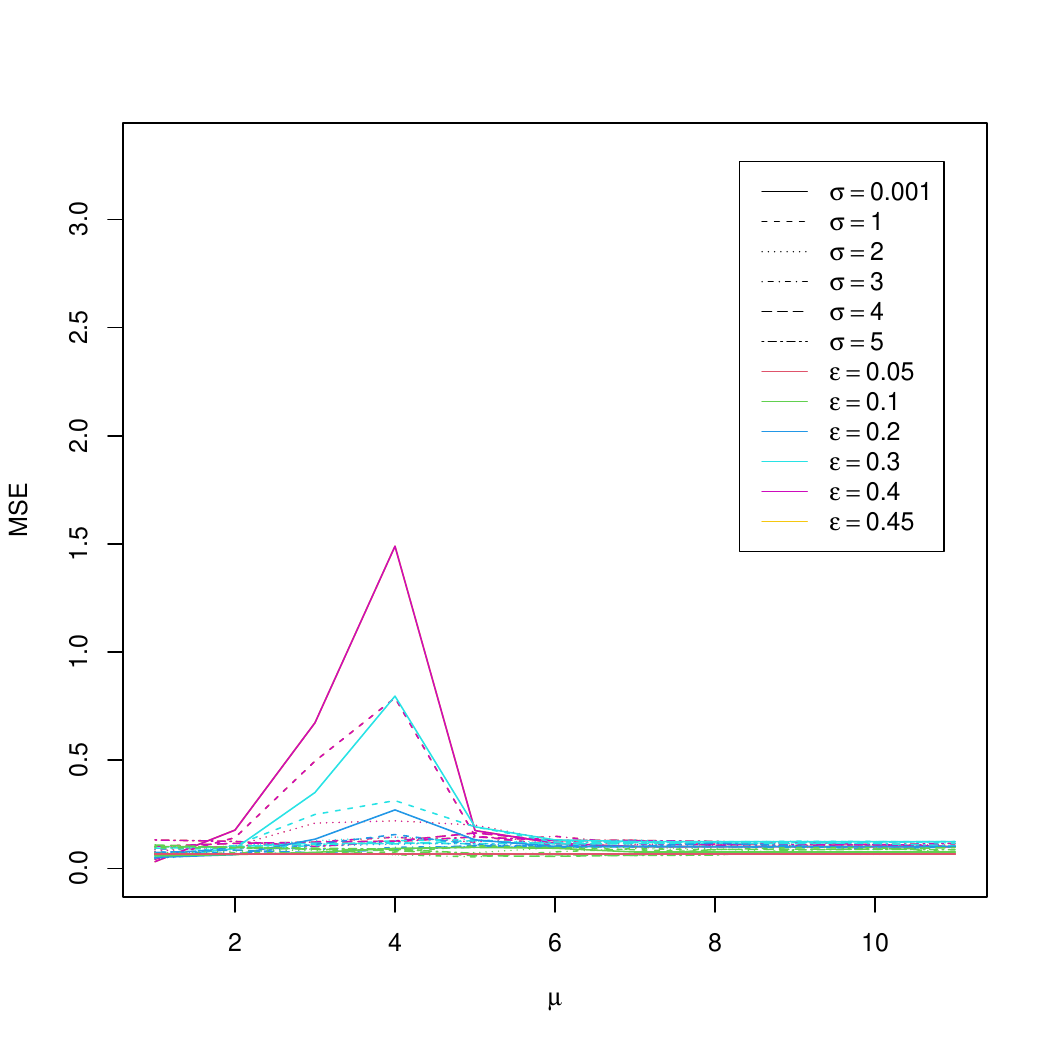} 
\includegraphics[width=0.32\textwidth]{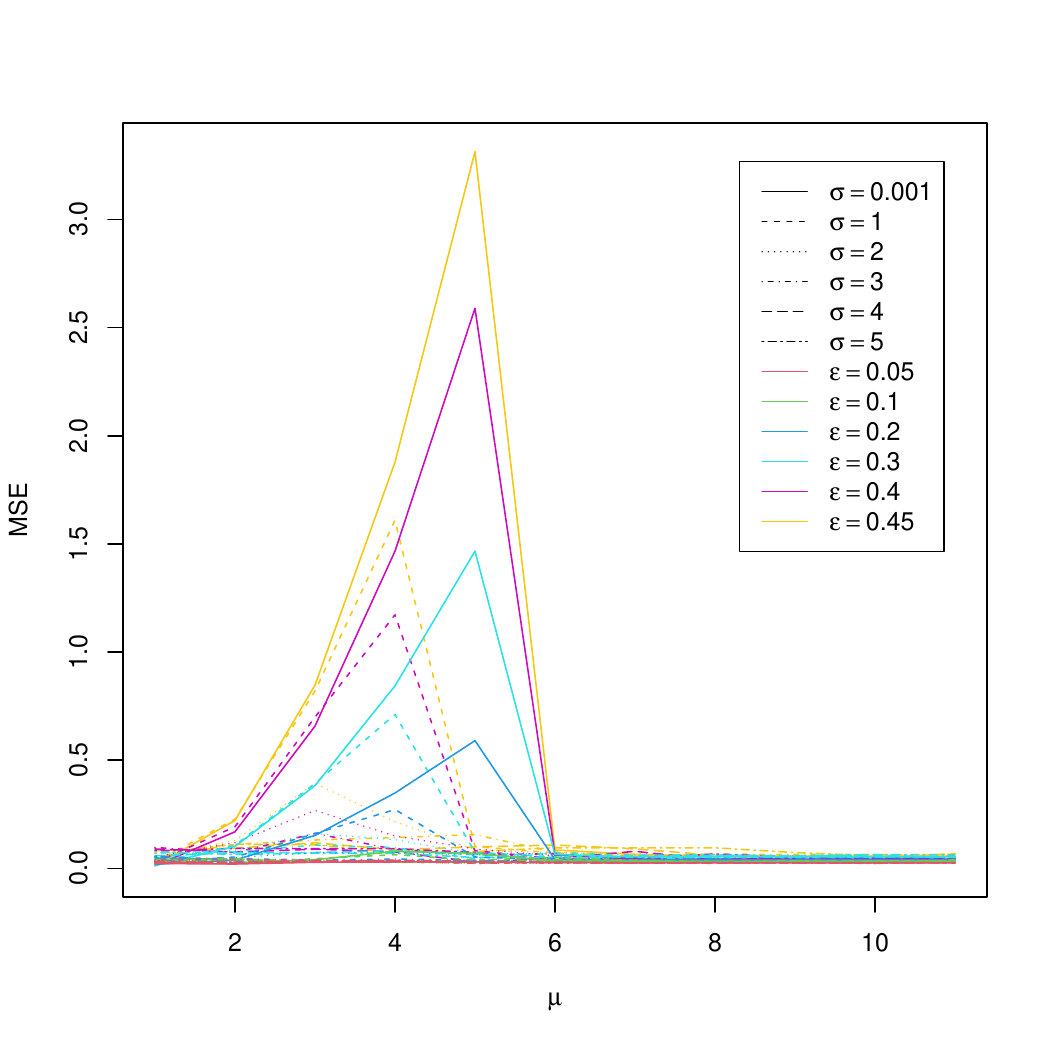} \\
\includegraphics[width=0.32\textwidth]{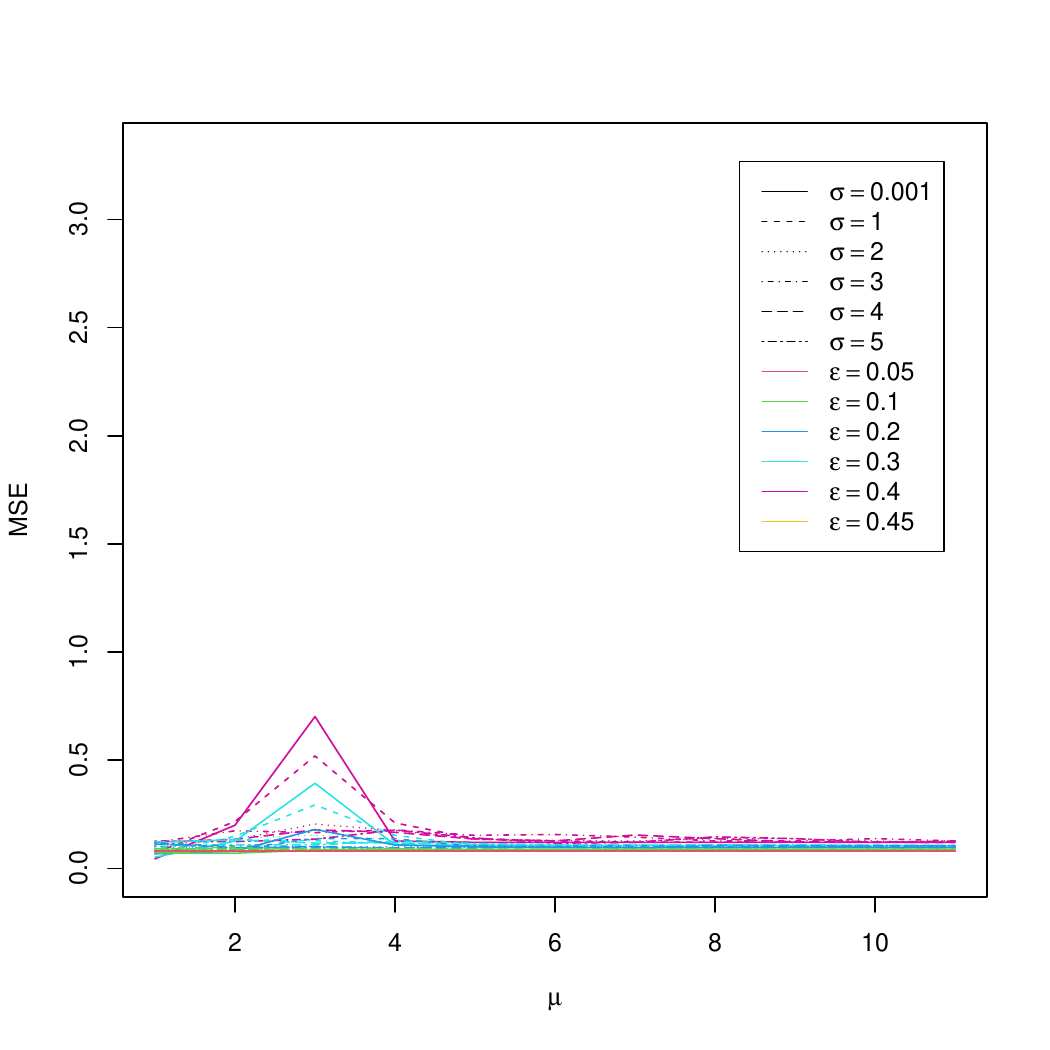}
\includegraphics[width=0.32\textwidth]{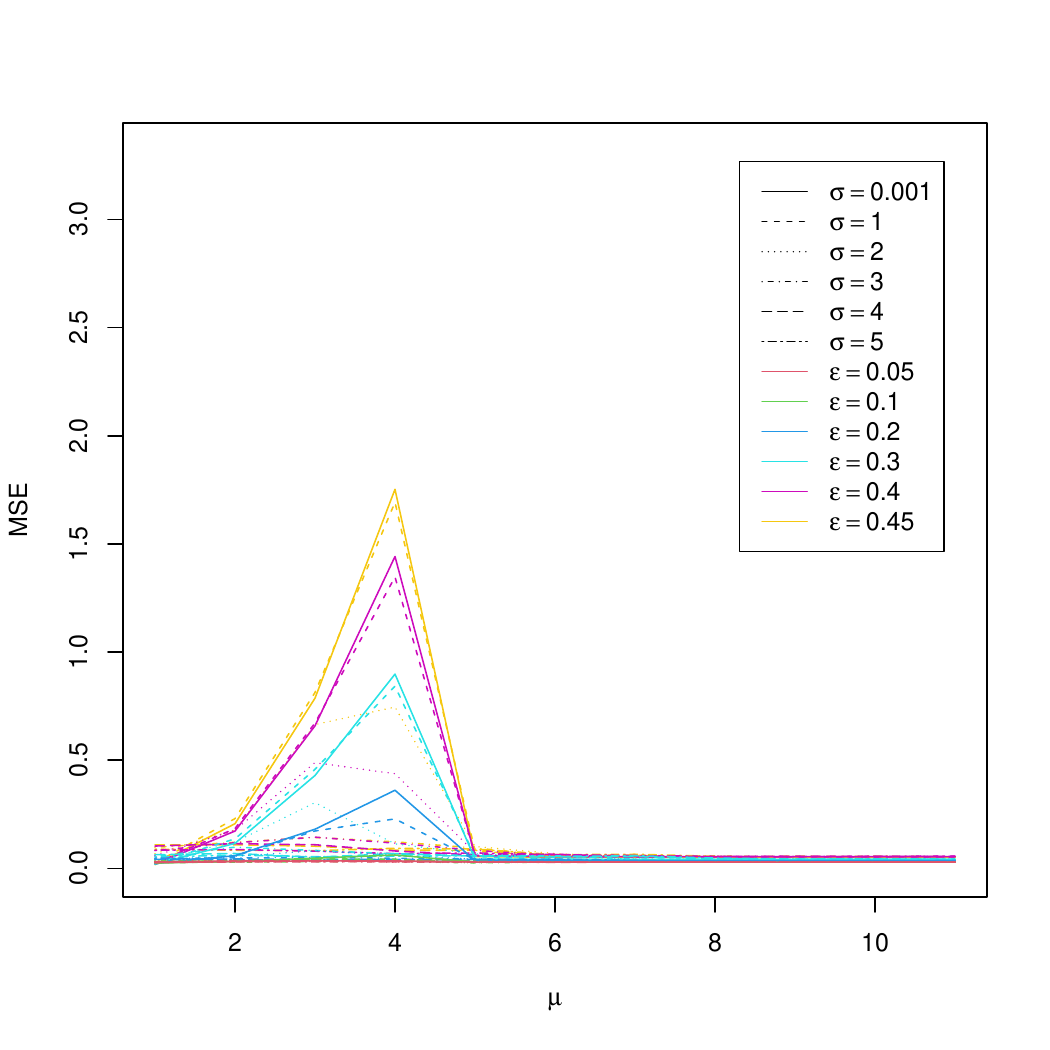} 
\includegraphics[width=0.32\textwidth]{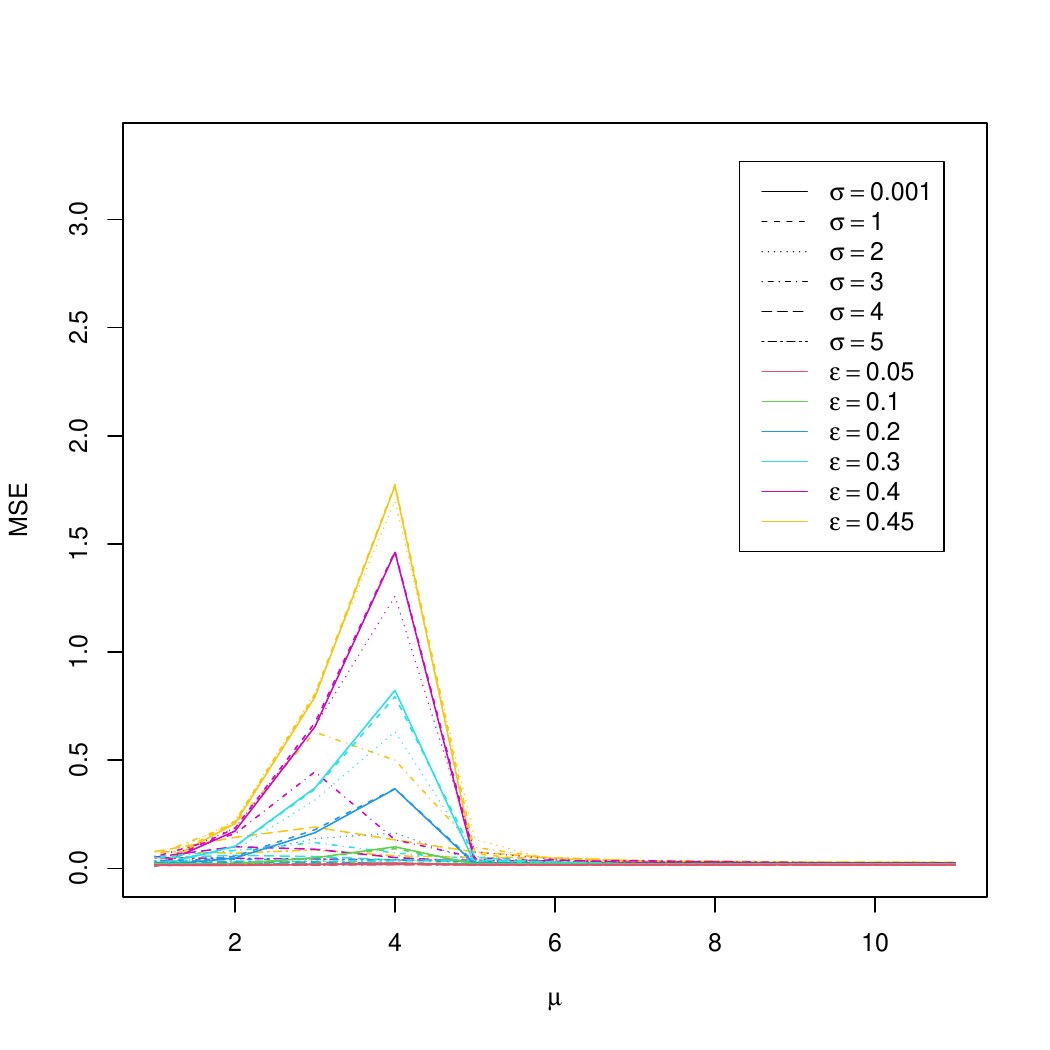} \\
\includegraphics[width=0.32\textwidth]{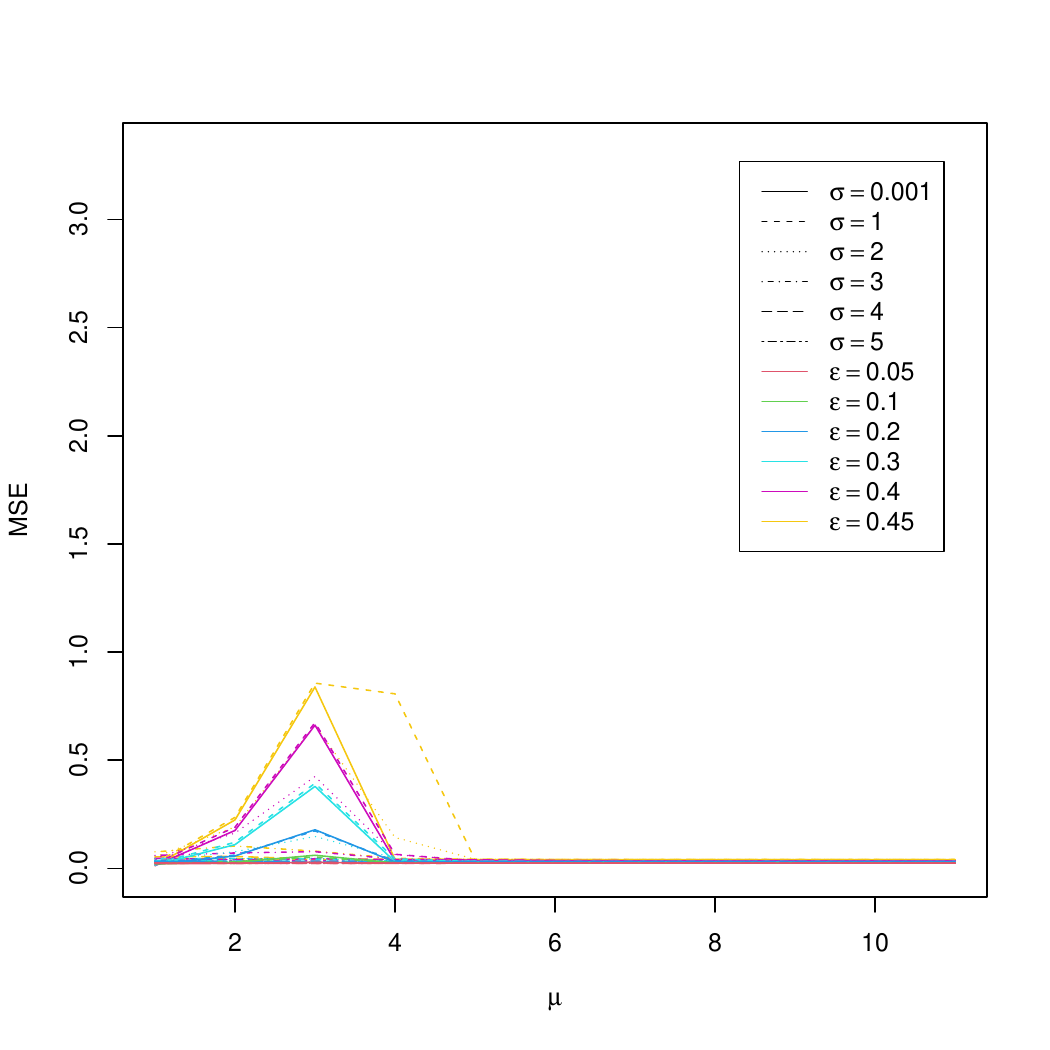}
\includegraphics[width=0.32\textwidth]{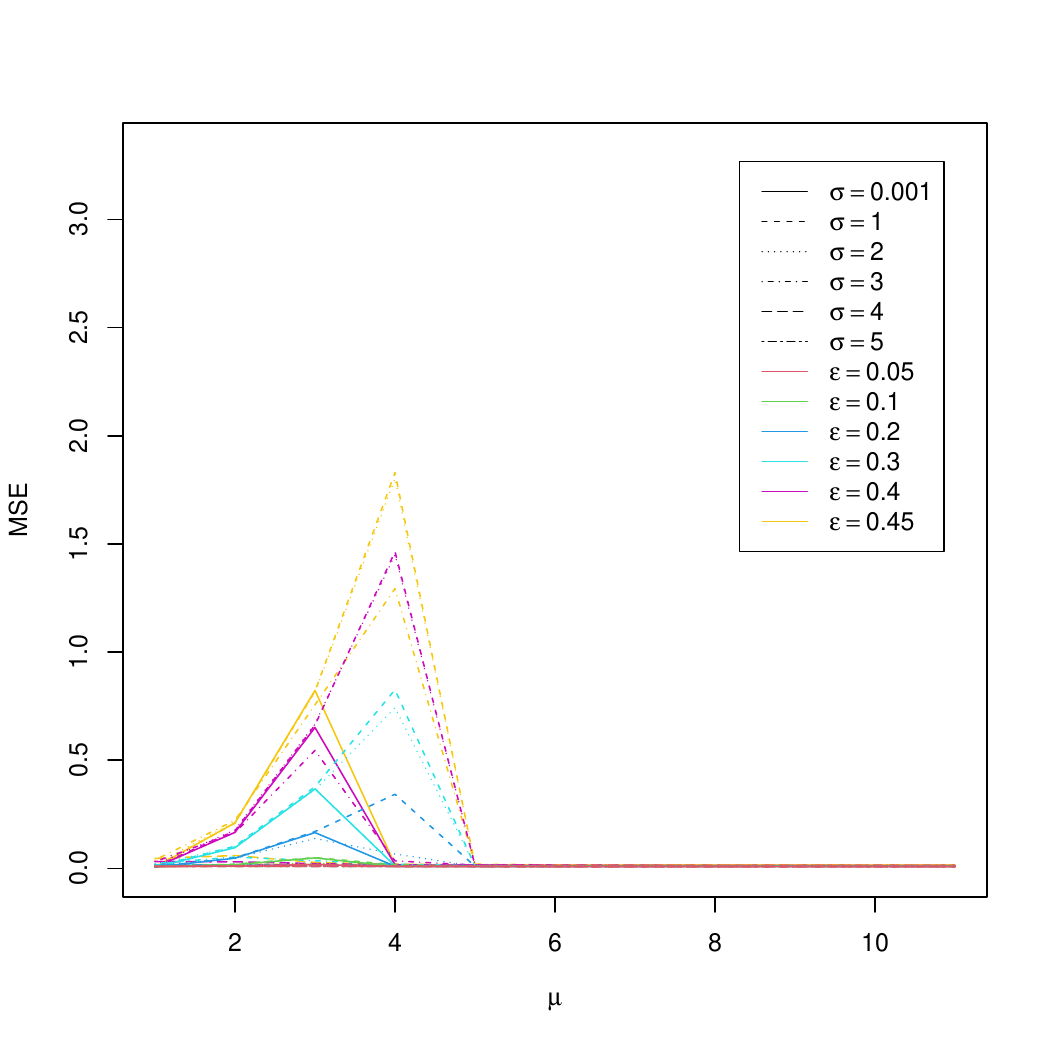} 
\includegraphics[width=0.32\textwidth]{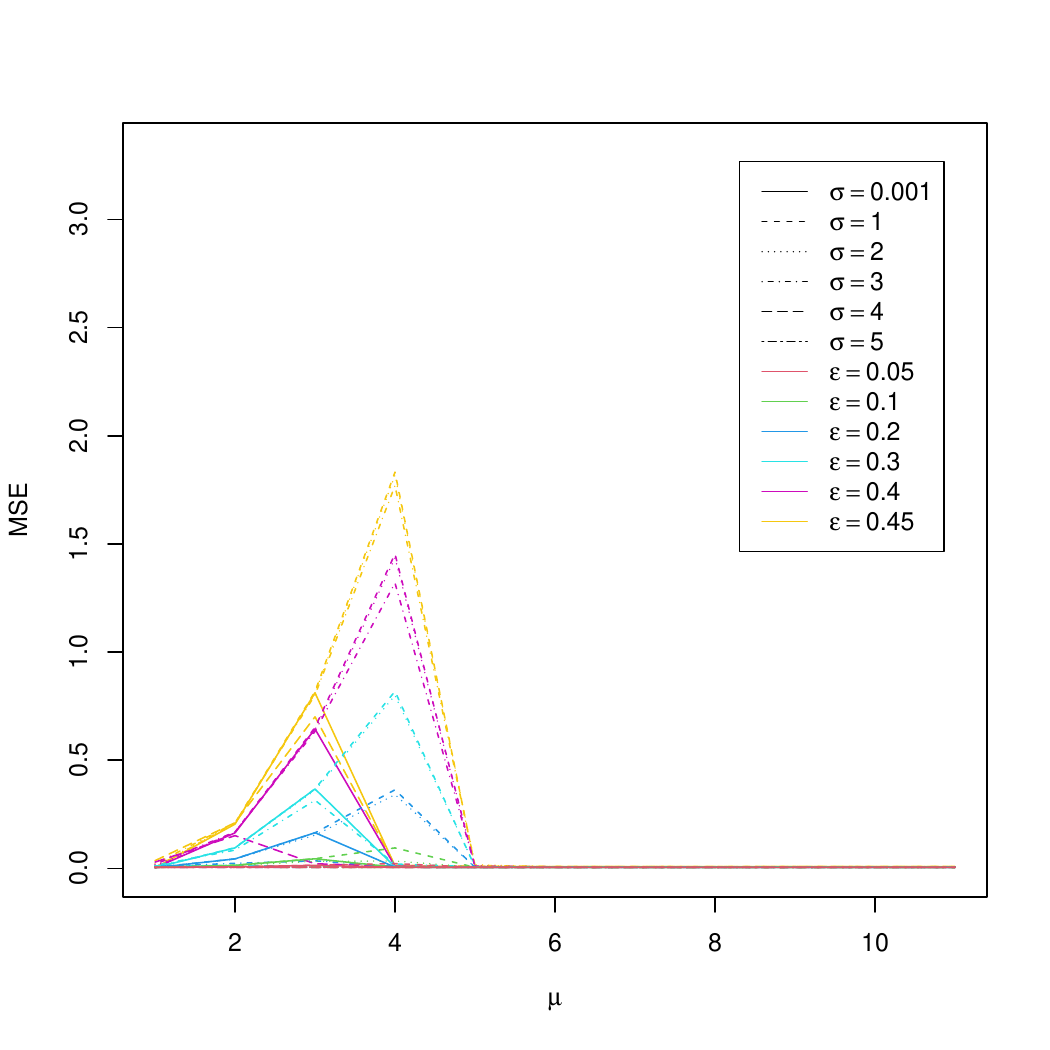} \\
\caption{Monte Carlo Simulation. Mean Square Error for the proposed method starting at the true values with $\alpha=0.25$ as a function of the contamination average $\mu$ ($x$-axis), contamination scale $\sigma$ (different line styles) and contamination level $\varepsilon$ (colors). Rows: number of variables $p=1, 2, 5$ and columns: sample size factor $s=2, 5, 10$.}
\label{sup:fig:monte:MSE:0.25:1}
\end{figure}  

\begin{figure}
\centering
\includegraphics[width=0.32\textwidth]{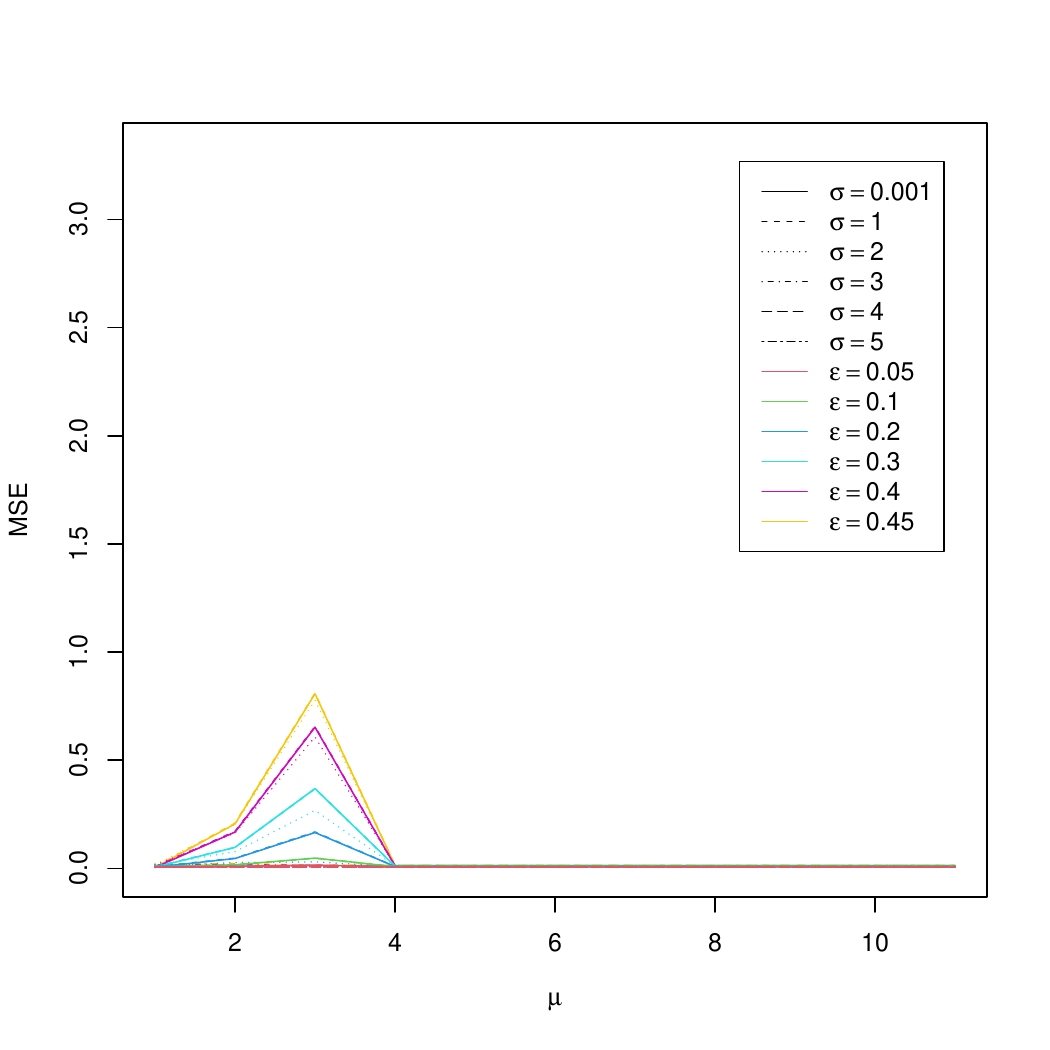}
\includegraphics[width=0.32\textwidth]{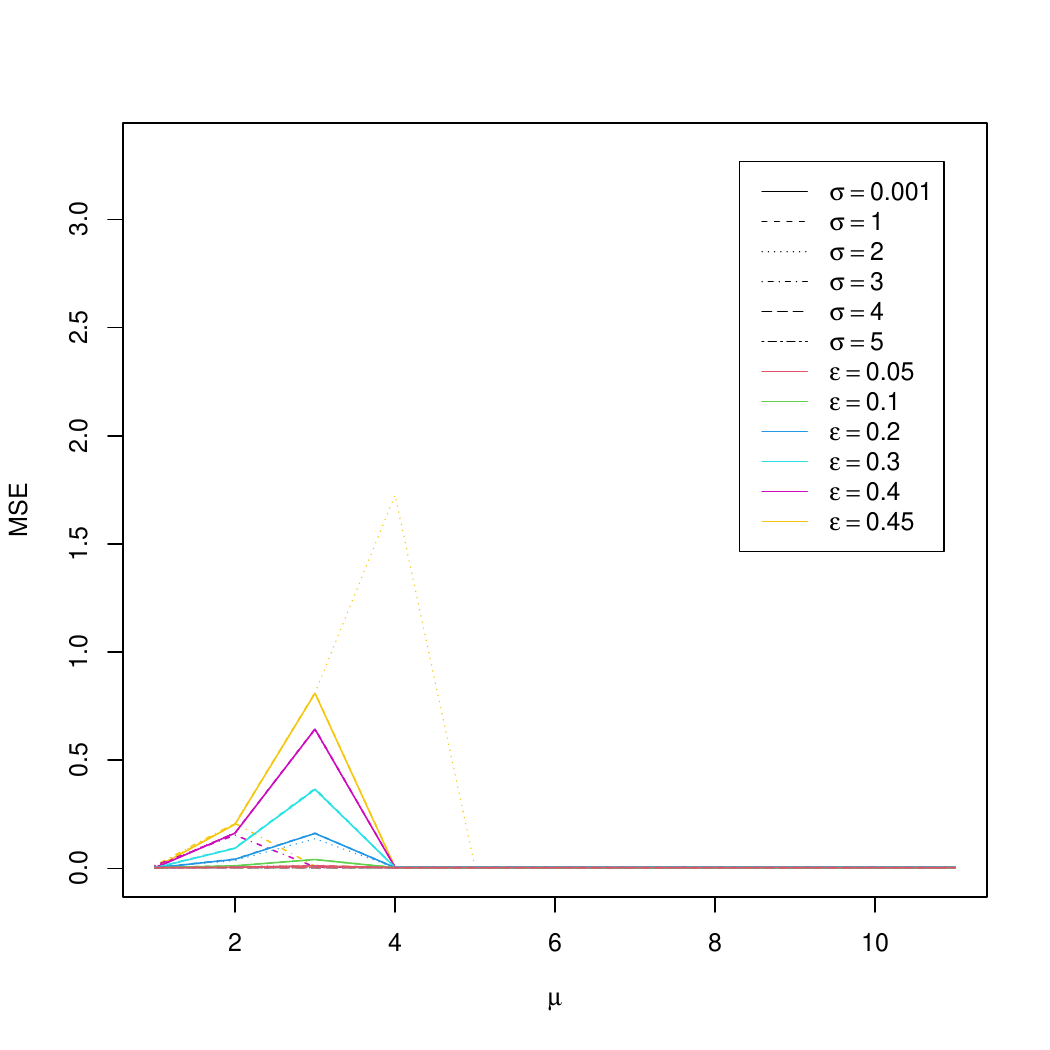} 
\includegraphics[width=0.32\textwidth]{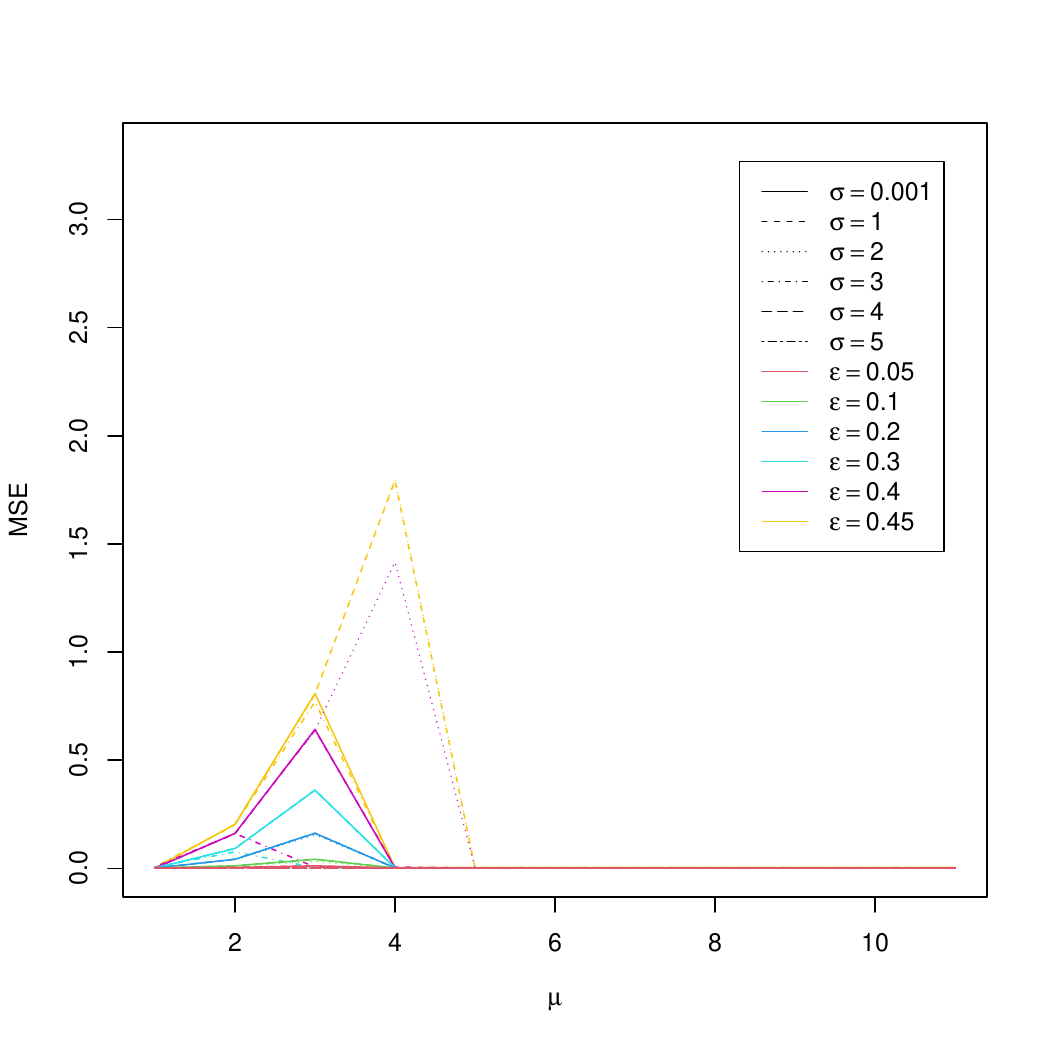} \\
\includegraphics[width=0.32\textwidth]{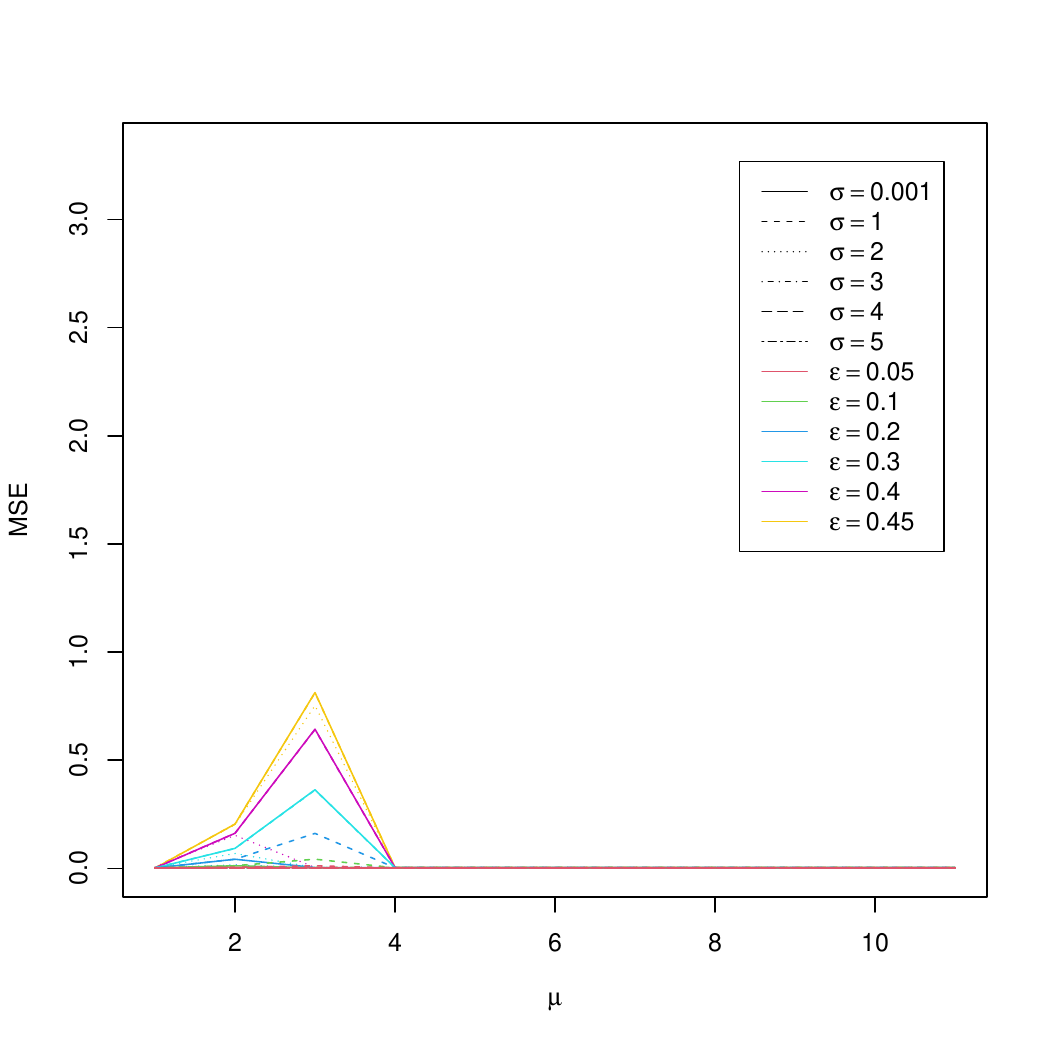}
\includegraphics[width=0.32\textwidth]{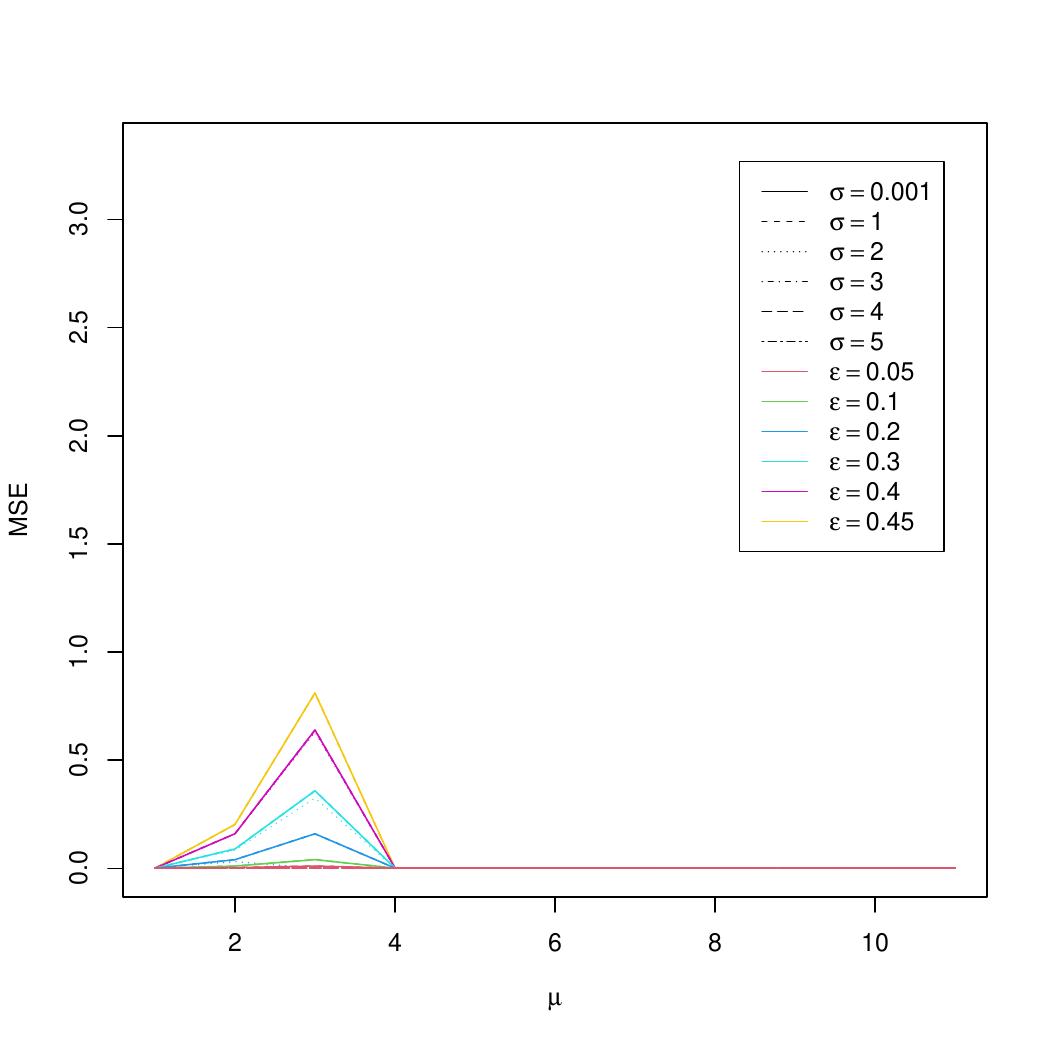} 
\includegraphics[width=0.32\textwidth]{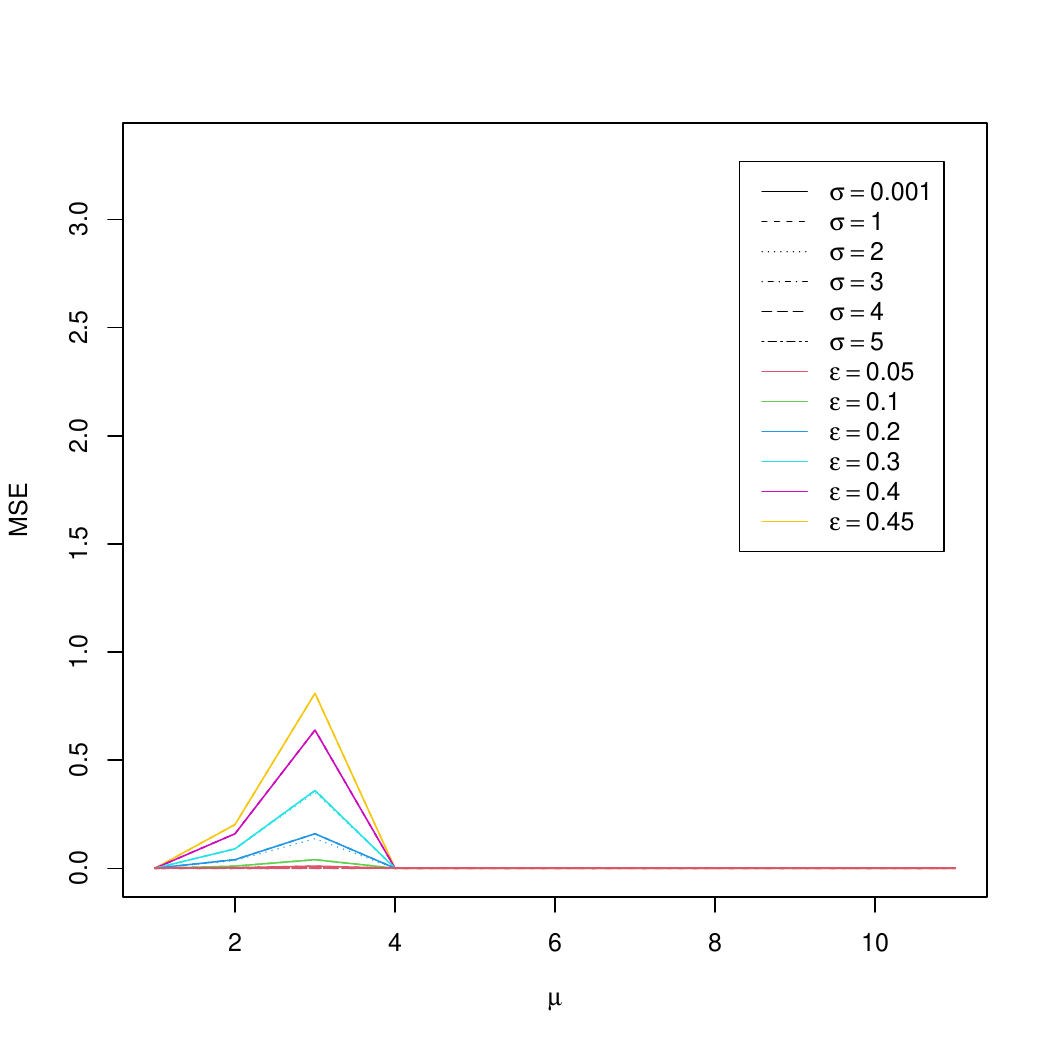}
\caption{Monte Carlo Simulation. Mean Square Error for the proposed method starting at the true values with $\alpha=0.25$ as a function of the contamination average $\mu$ ($x$-axis), contamination scale $\sigma$ (different line styles) and contamination level $\varepsilon$ (colors). Rows: number of variables $p=10, 20$ and columns: sample size factor $s=2, 5, 10$.}
\label{sup:fig:monte:MSE:0.25:2}
\end{figure}  

\begin{figure}
\centering
\includegraphics[width=0.32\textwidth]{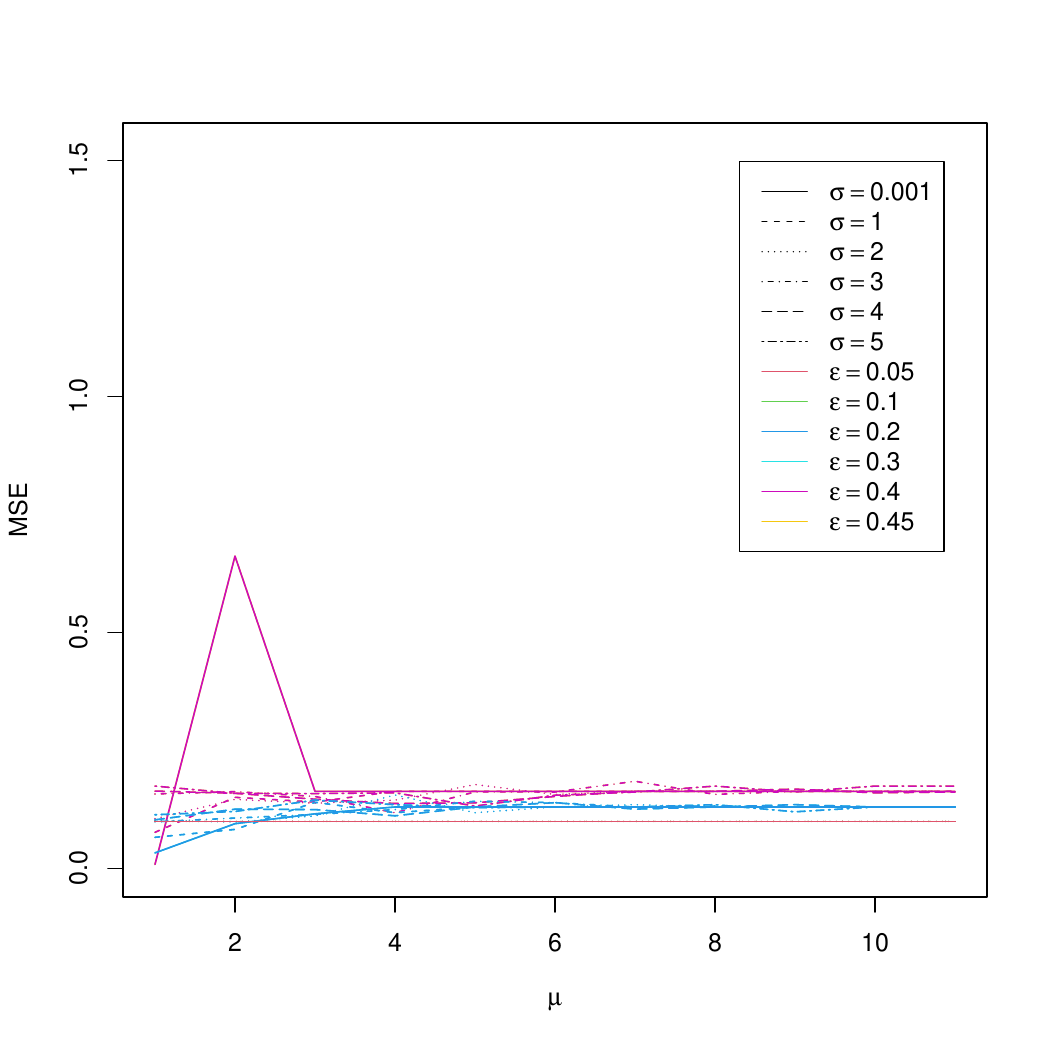}
\includegraphics[width=0.32\textwidth]{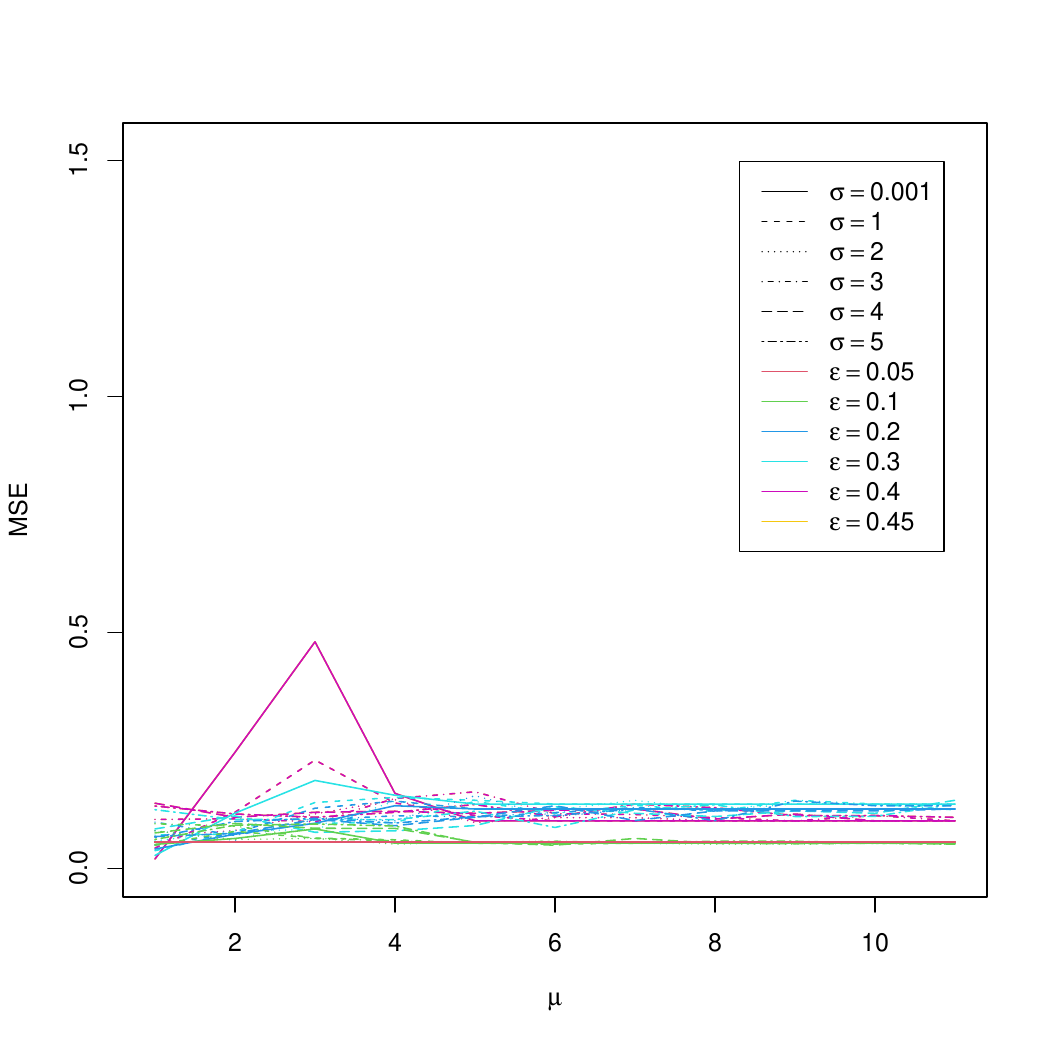} 
\includegraphics[width=0.32\textwidth]{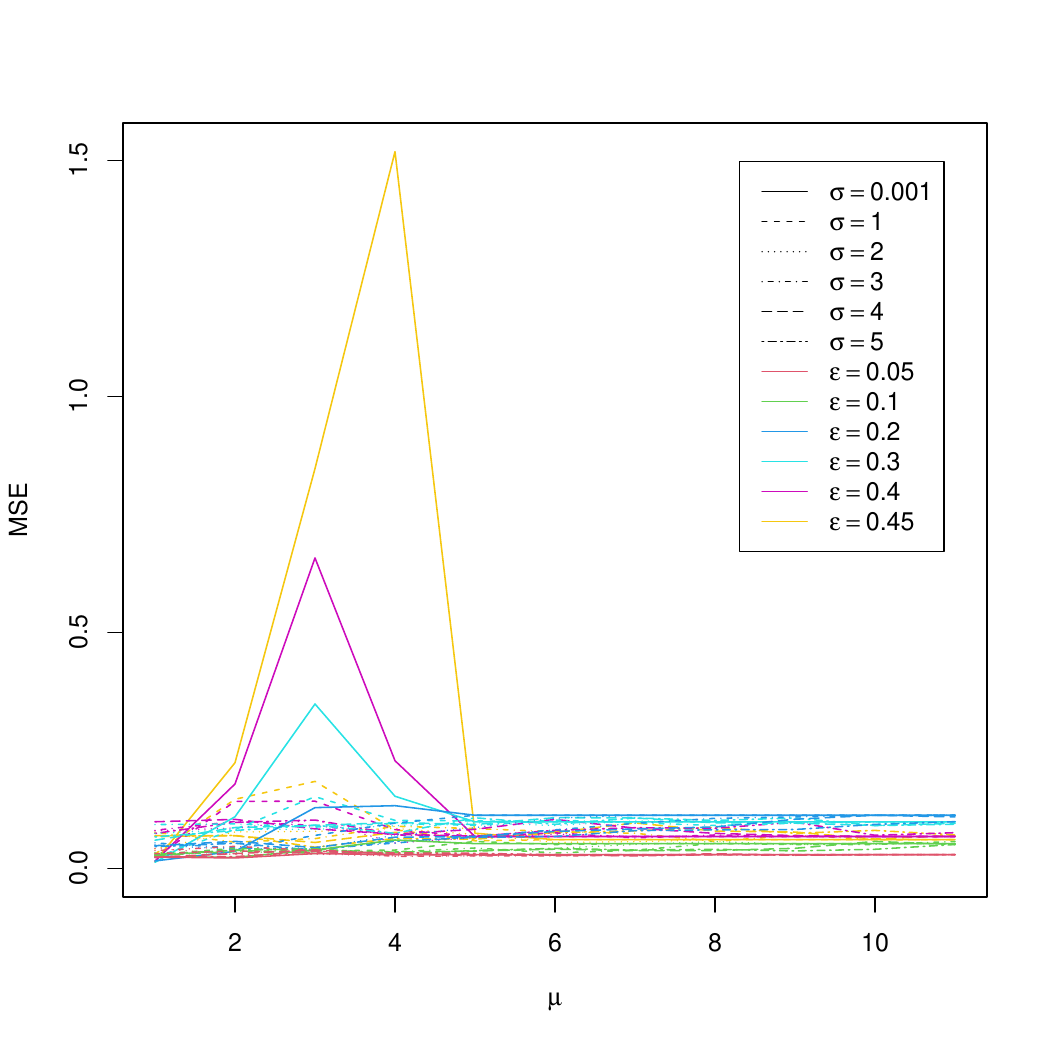} \\
\includegraphics[width=0.32\textwidth]{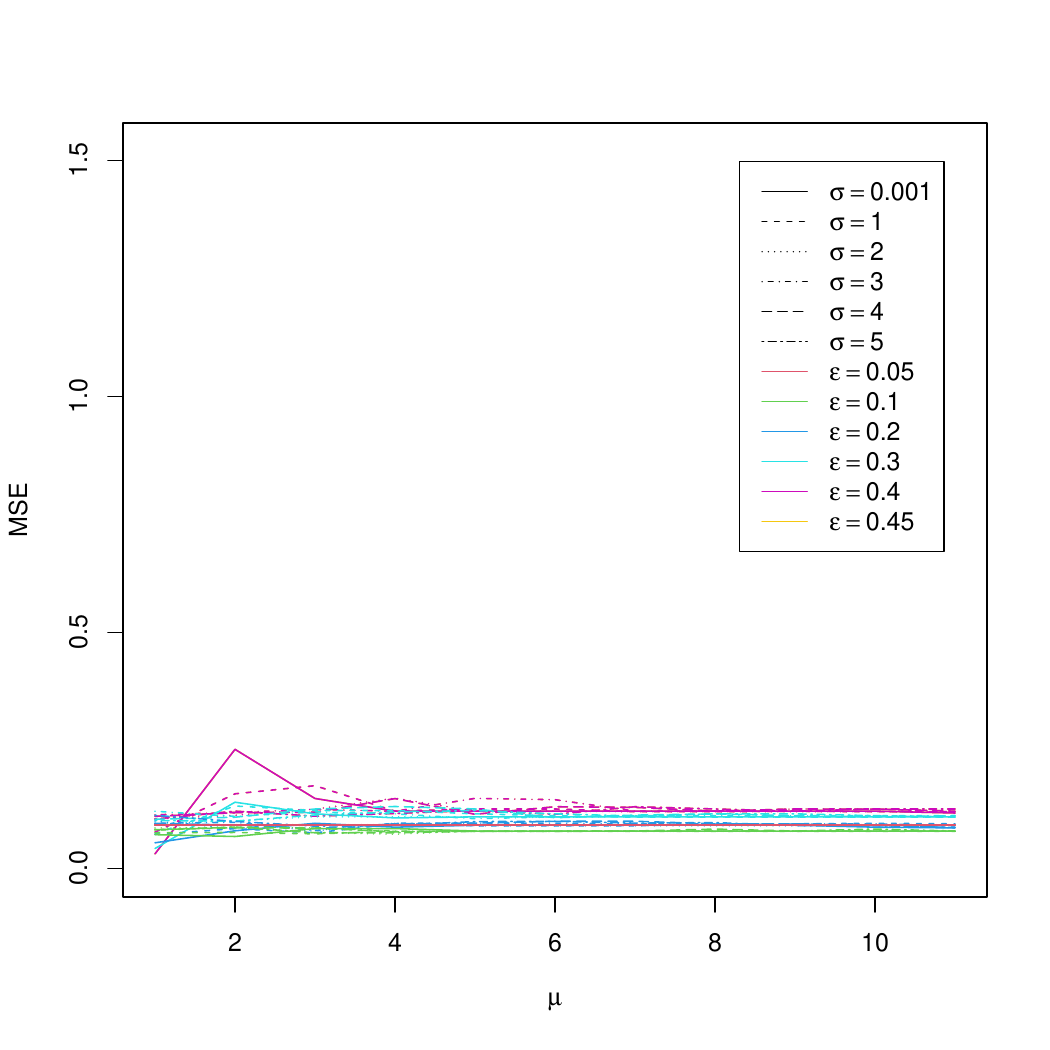}
\includegraphics[width=0.32\textwidth]{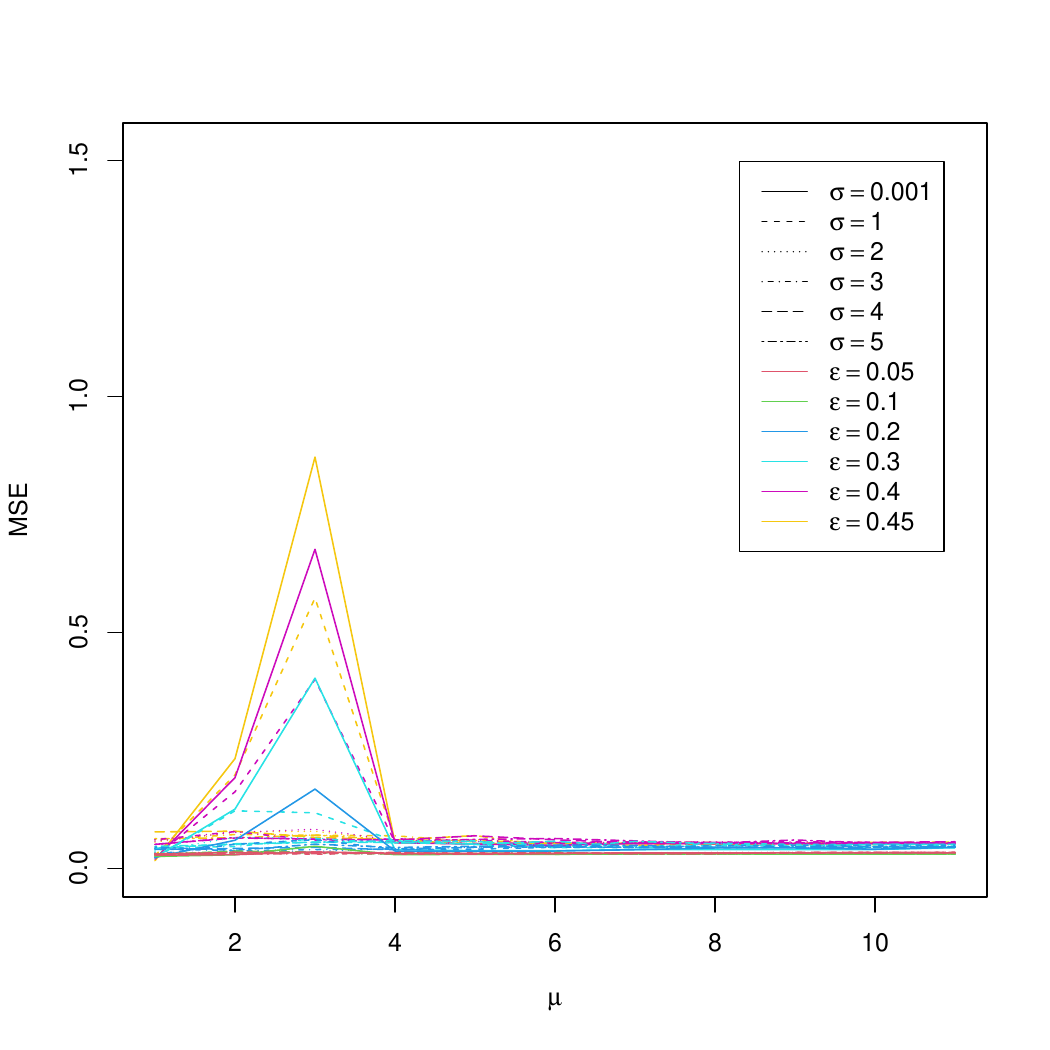} 
\includegraphics[width=0.32\textwidth]{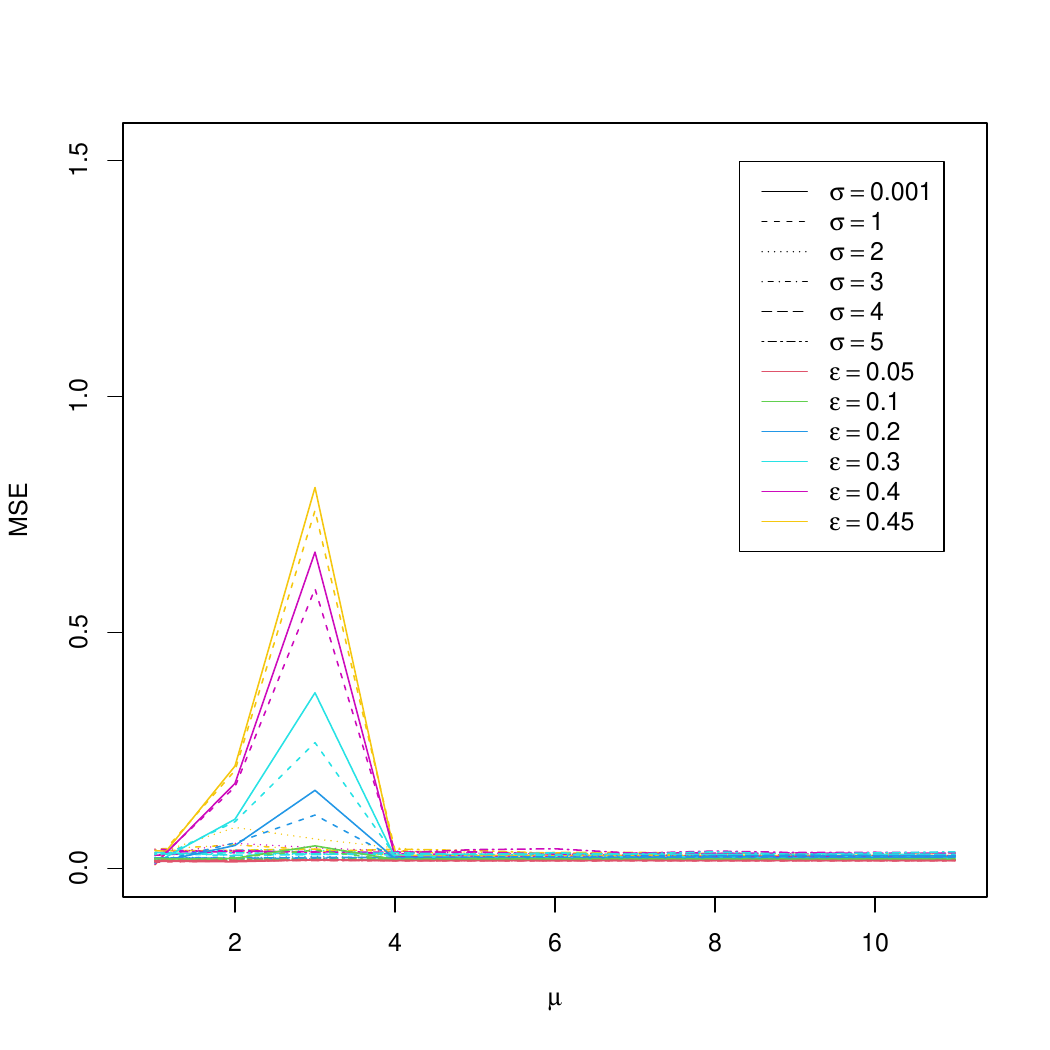} \\
\includegraphics[width=0.32\textwidth]{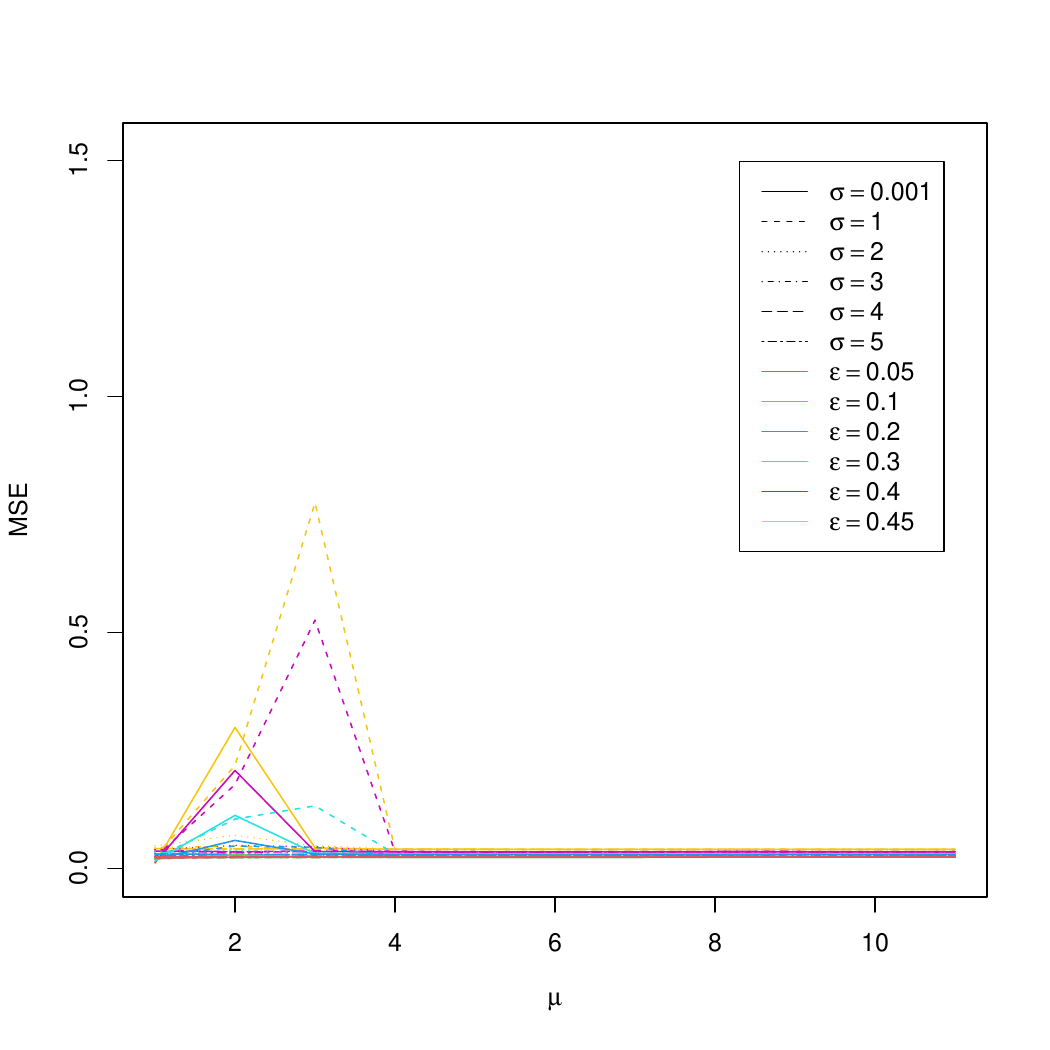}
\includegraphics[width=0.32\textwidth]{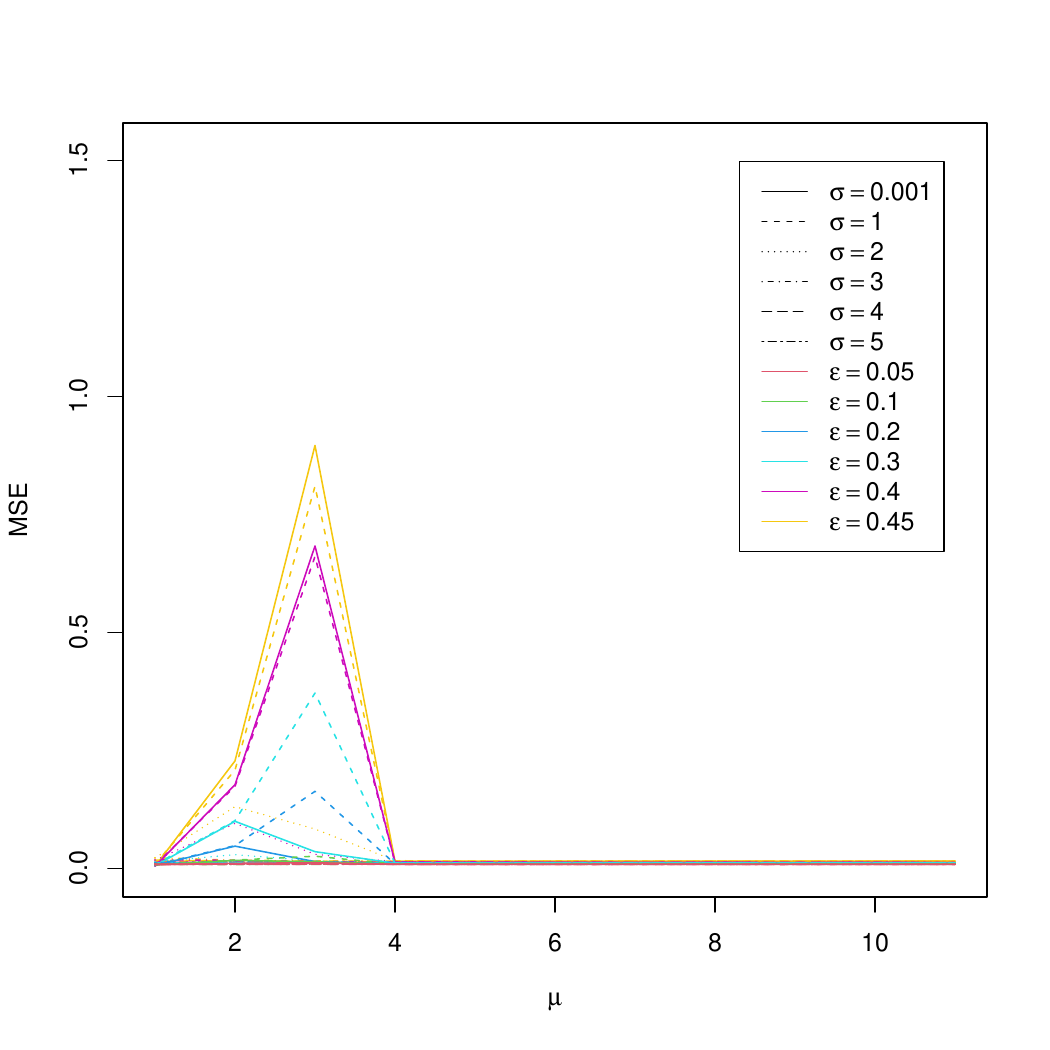} 
\includegraphics[width=0.32\textwidth]{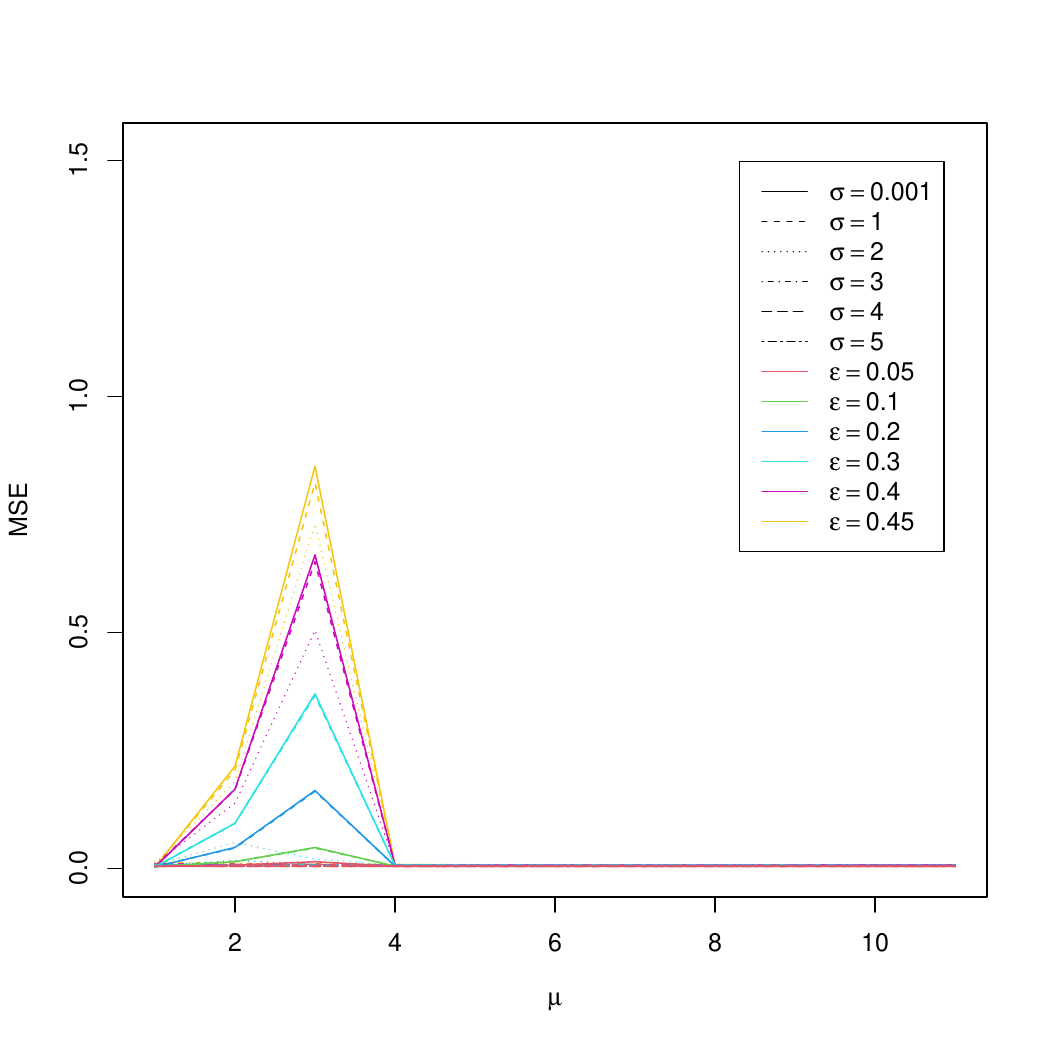} \\
\caption{Monte Carlo Simulation. Mean Square Error for the proposed method starting at the true values with $\alpha=0.5$ as a function of the contamination average $\mu$ ($x$-axis), contamination scale $\sigma$ (different line styles) and contamination level $\varepsilon$ (colors). Rows: number of variables $p=1, 2, 5$ and columns: sample size factor $s=2, 5, 10$.}
\label{sup:fig:monte:MSE:0.5:1}
\end{figure}  

\begin{figure}
\centering
\includegraphics[width=0.32\textwidth]{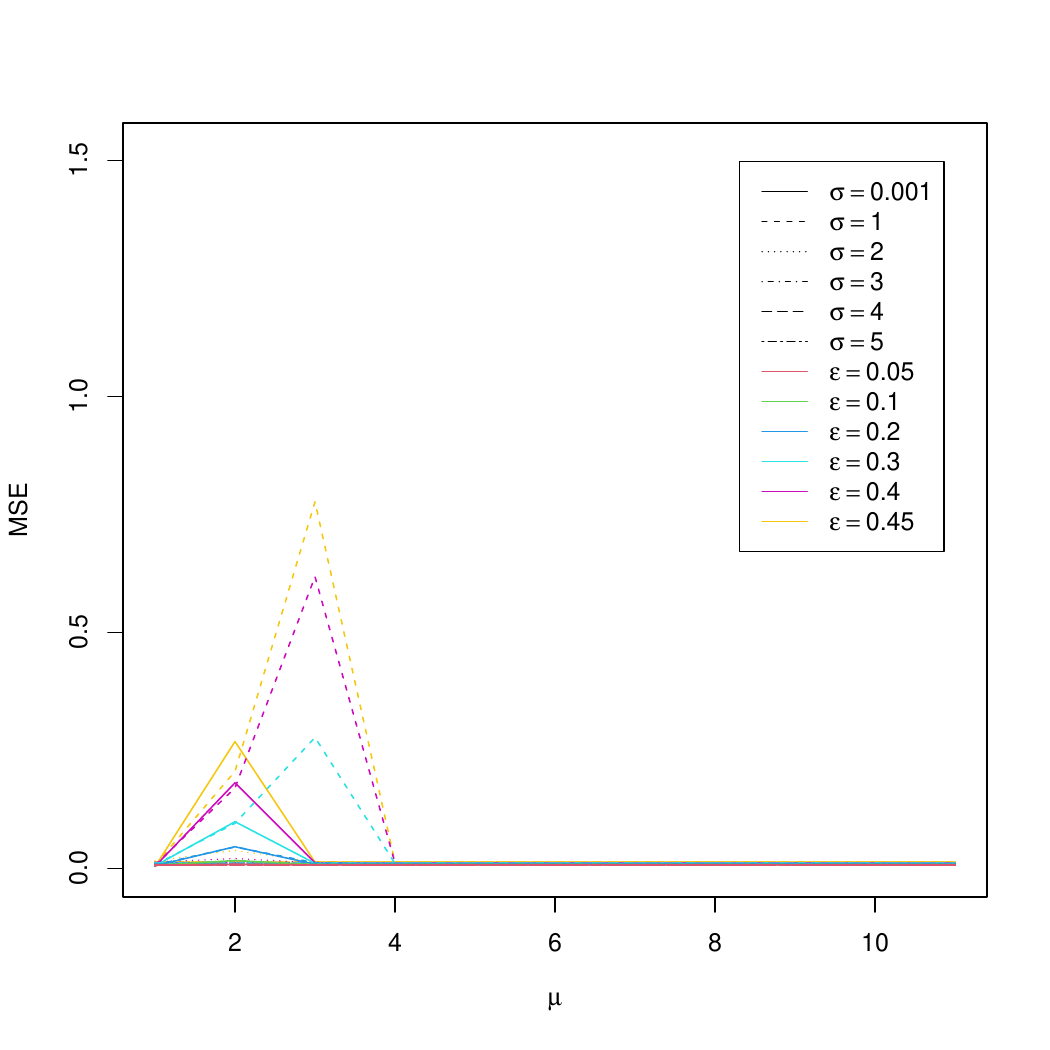}
\includegraphics[width=0.32\textwidth]{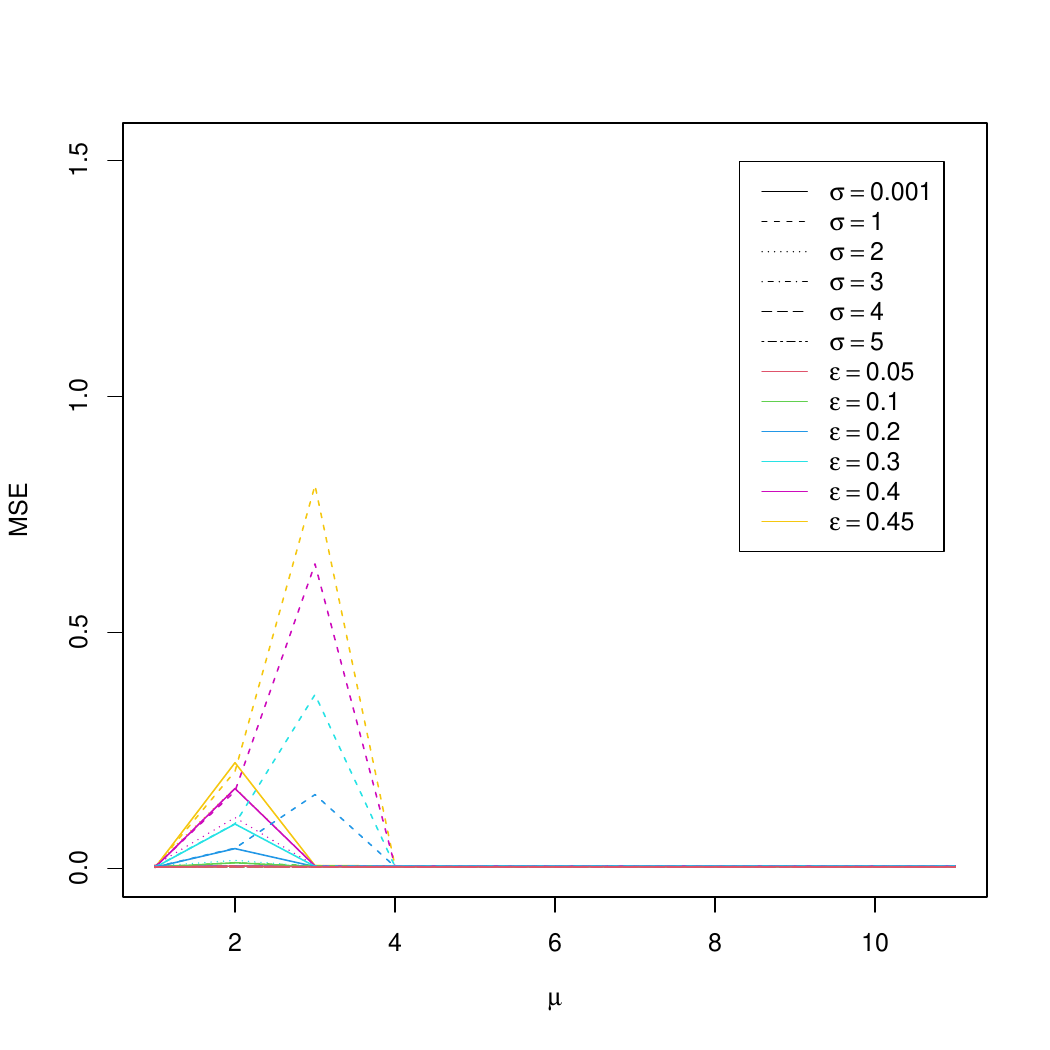} 
\includegraphics[width=0.32\textwidth]{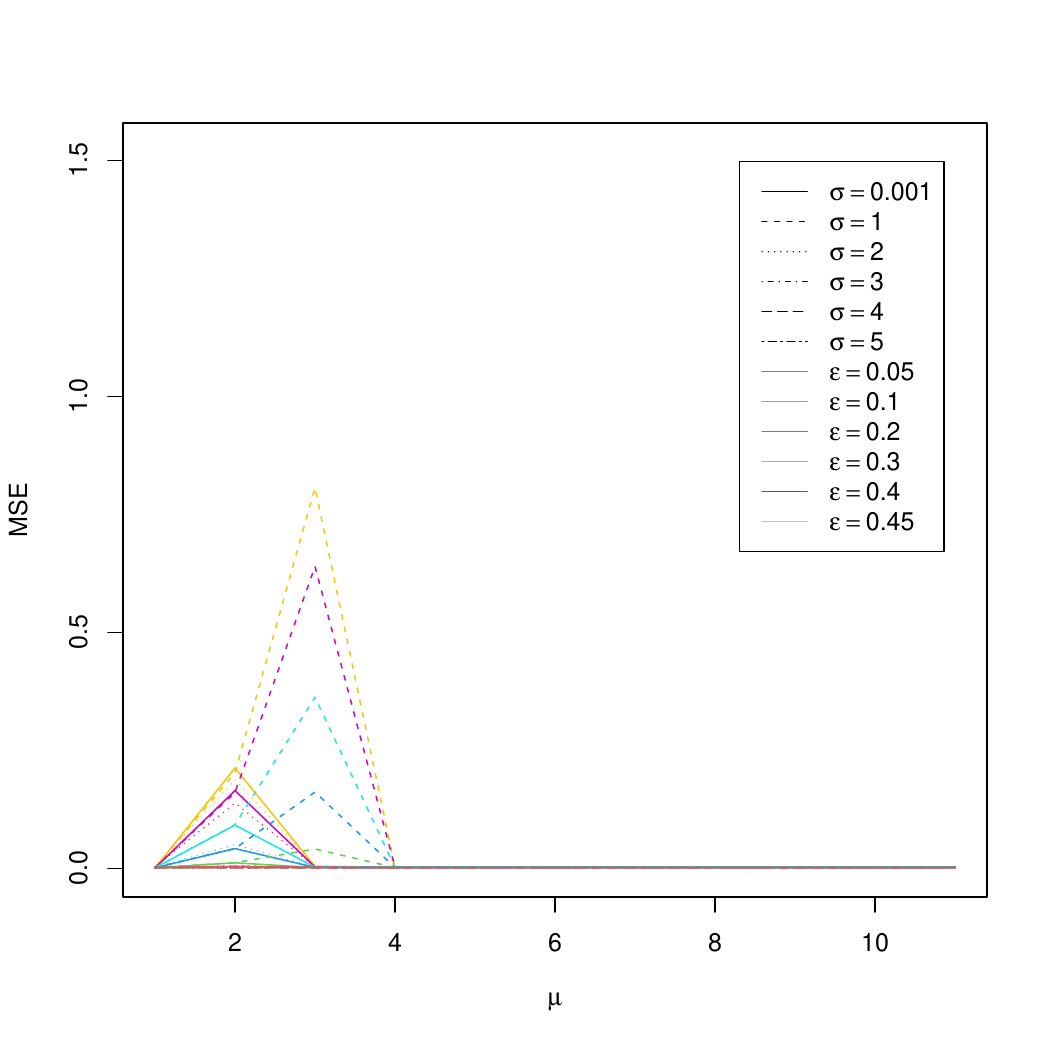} \\
\includegraphics[width=0.32\textwidth]{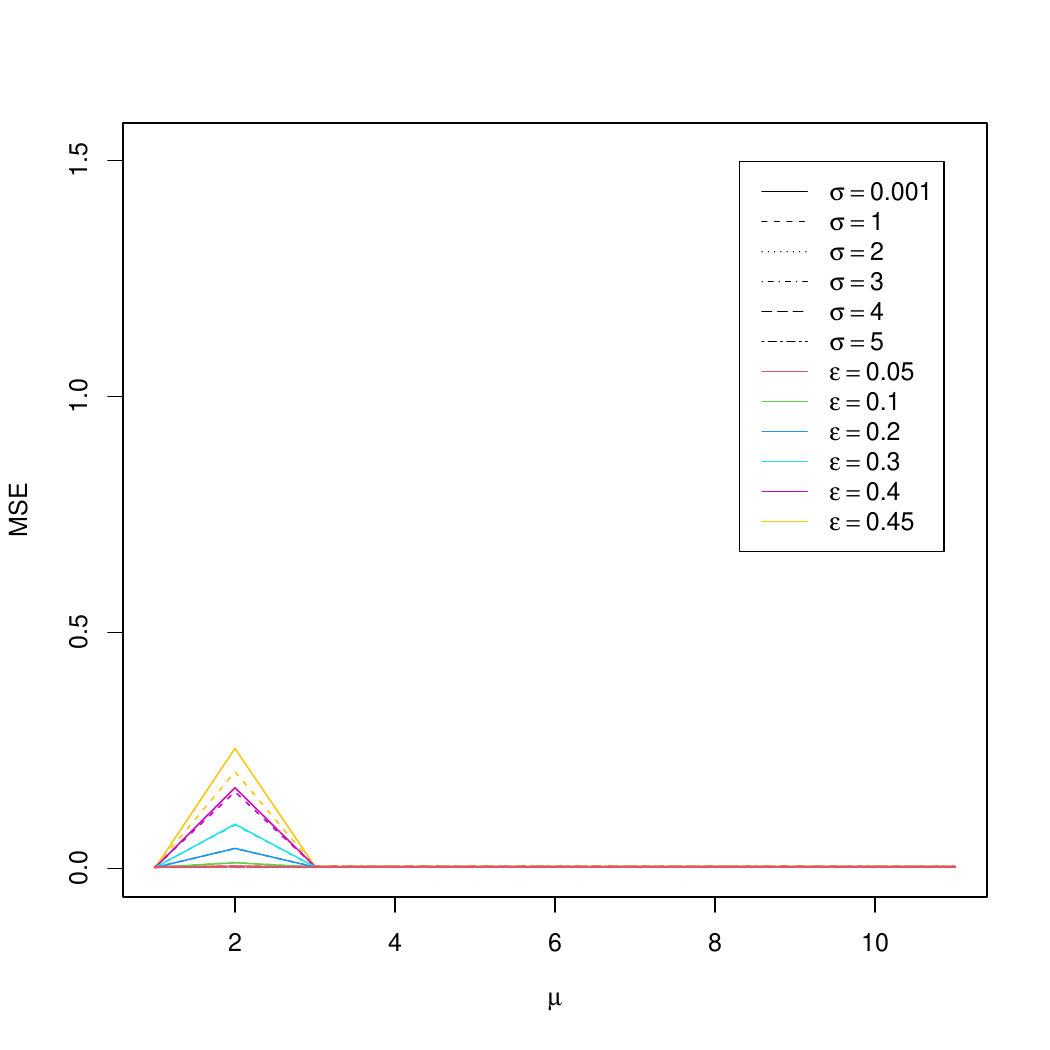}
\includegraphics[width=0.32\textwidth]{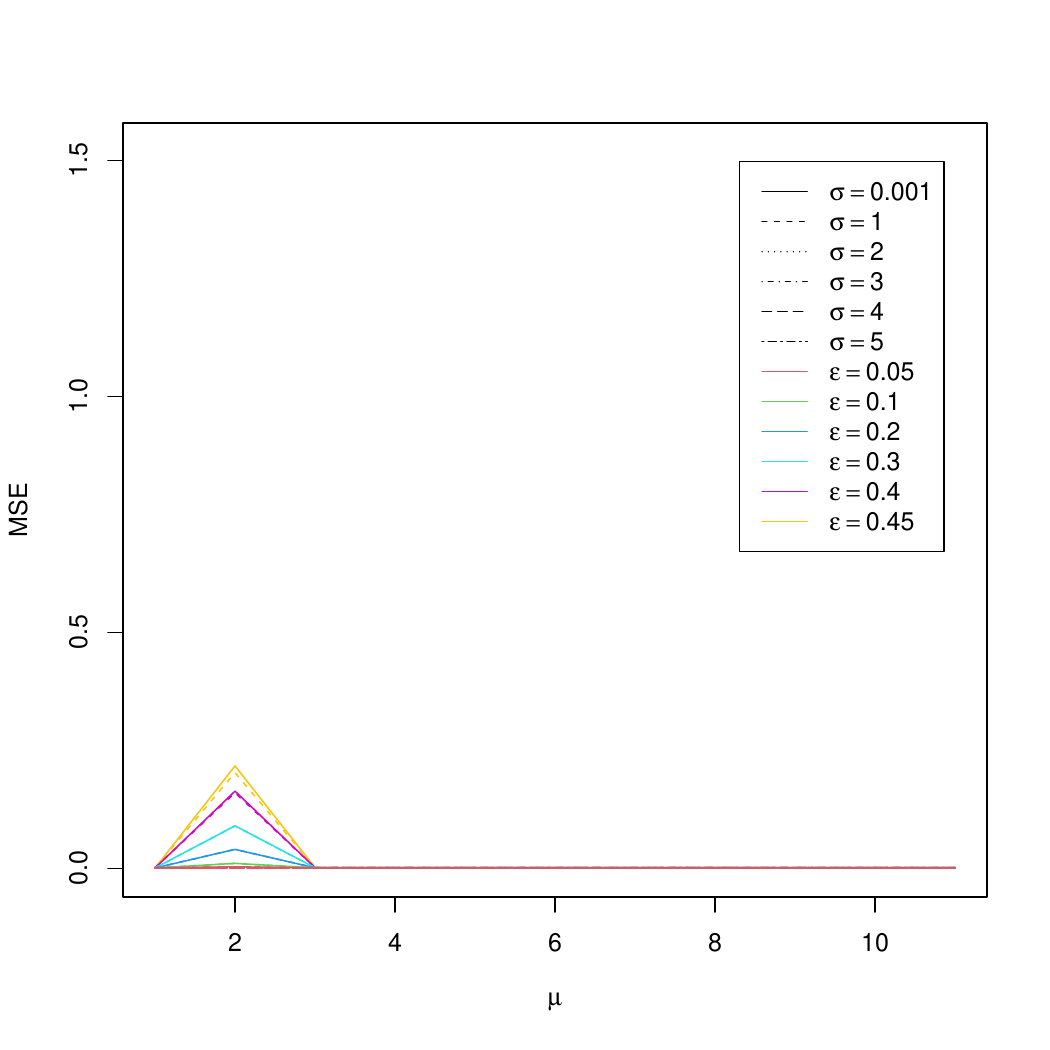} 
\includegraphics[width=0.32\textwidth]{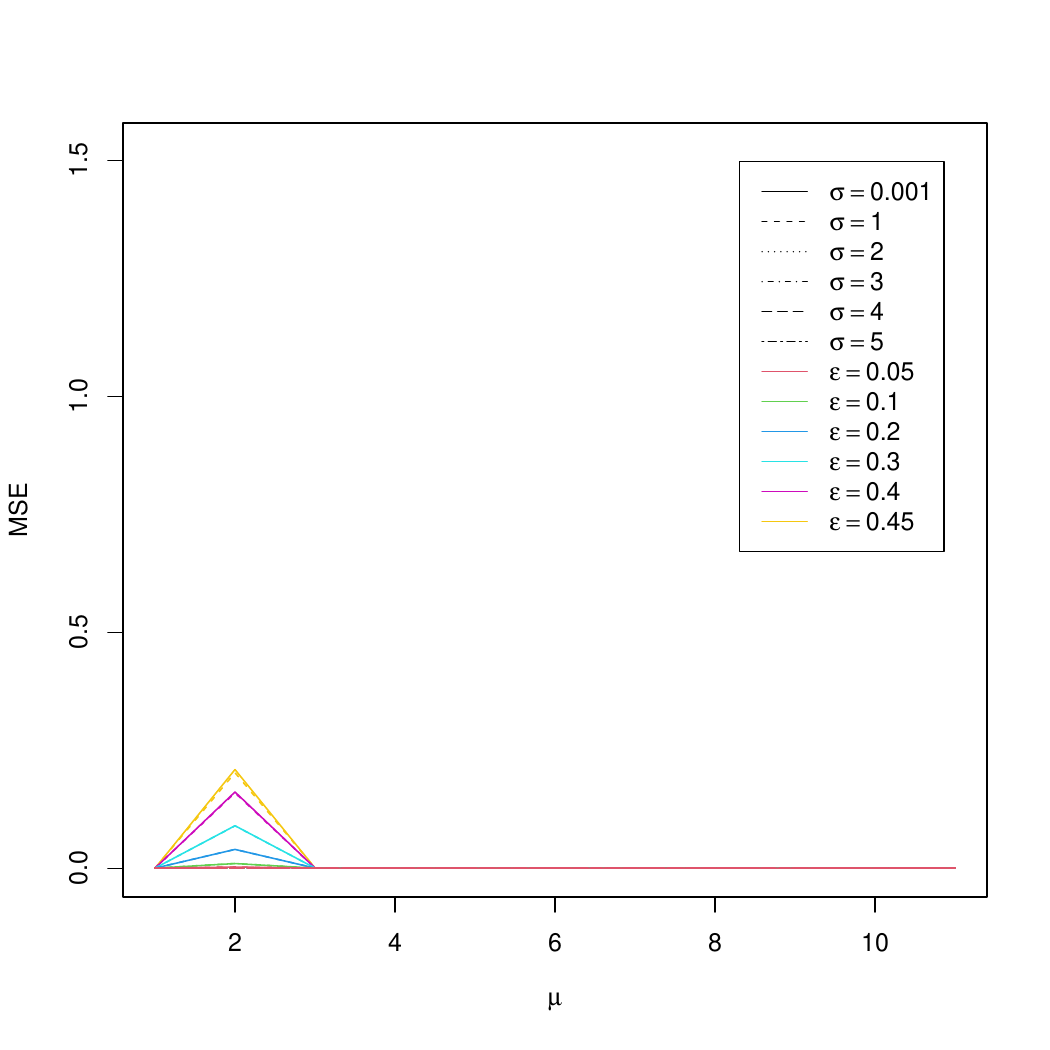}
\caption{Monte Carlo Simulation. Mean Square Error for the proposed method starting at the true values with $\alpha=0.5$ as a function of the contamination average $\mu$ ($x$-axis), contamination scale $\sigma$ (different line styles) and contamination level $\varepsilon$ (colors). Rows: number of variables $p=10, 20$ and columns: sample size factor $s=2, 5, 10$.}
\label{sup:fig:monte:MSE:0.5:2}
\end{figure}  

\begin{figure}
\centering
\includegraphics[width=0.32\textwidth]{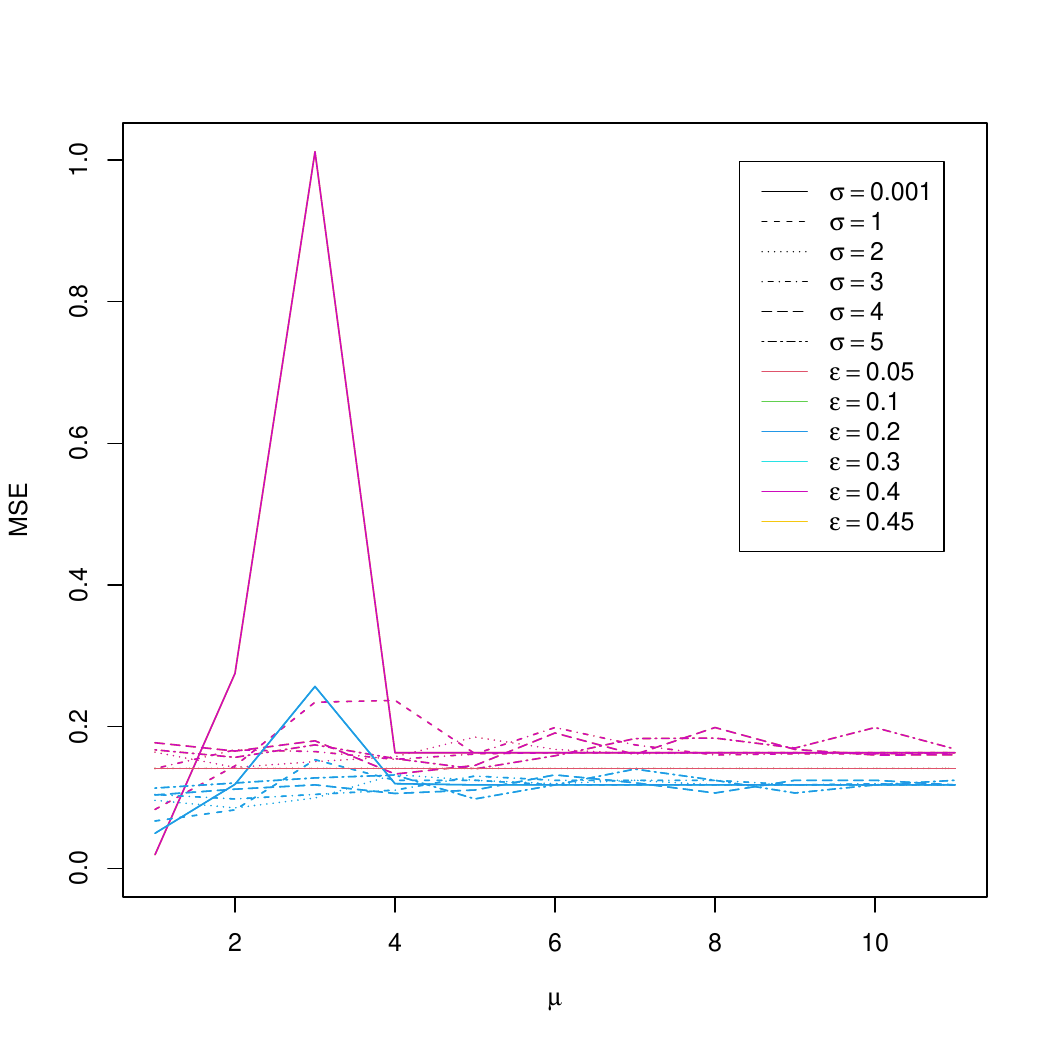}
\includegraphics[width=0.32\textwidth]{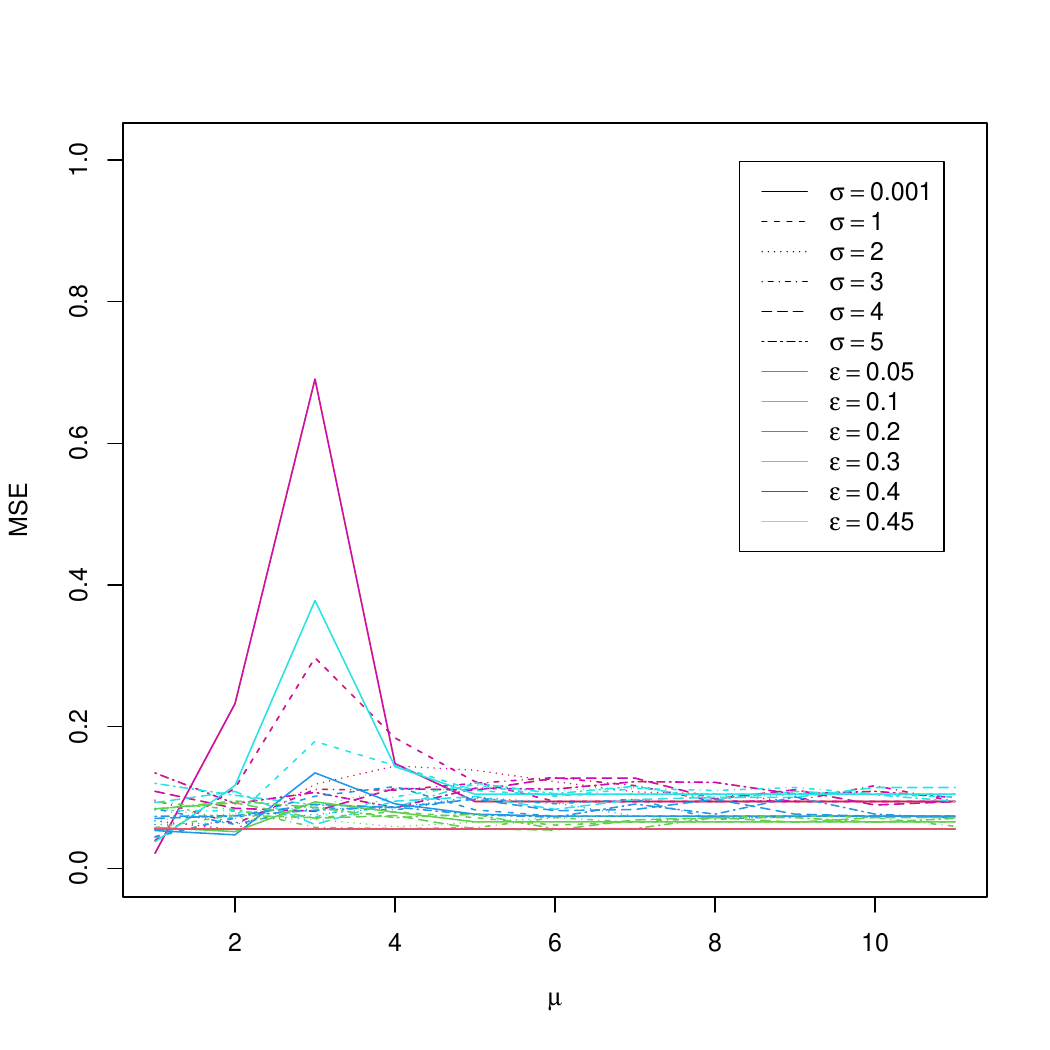} 
\includegraphics[width=0.32\textwidth]{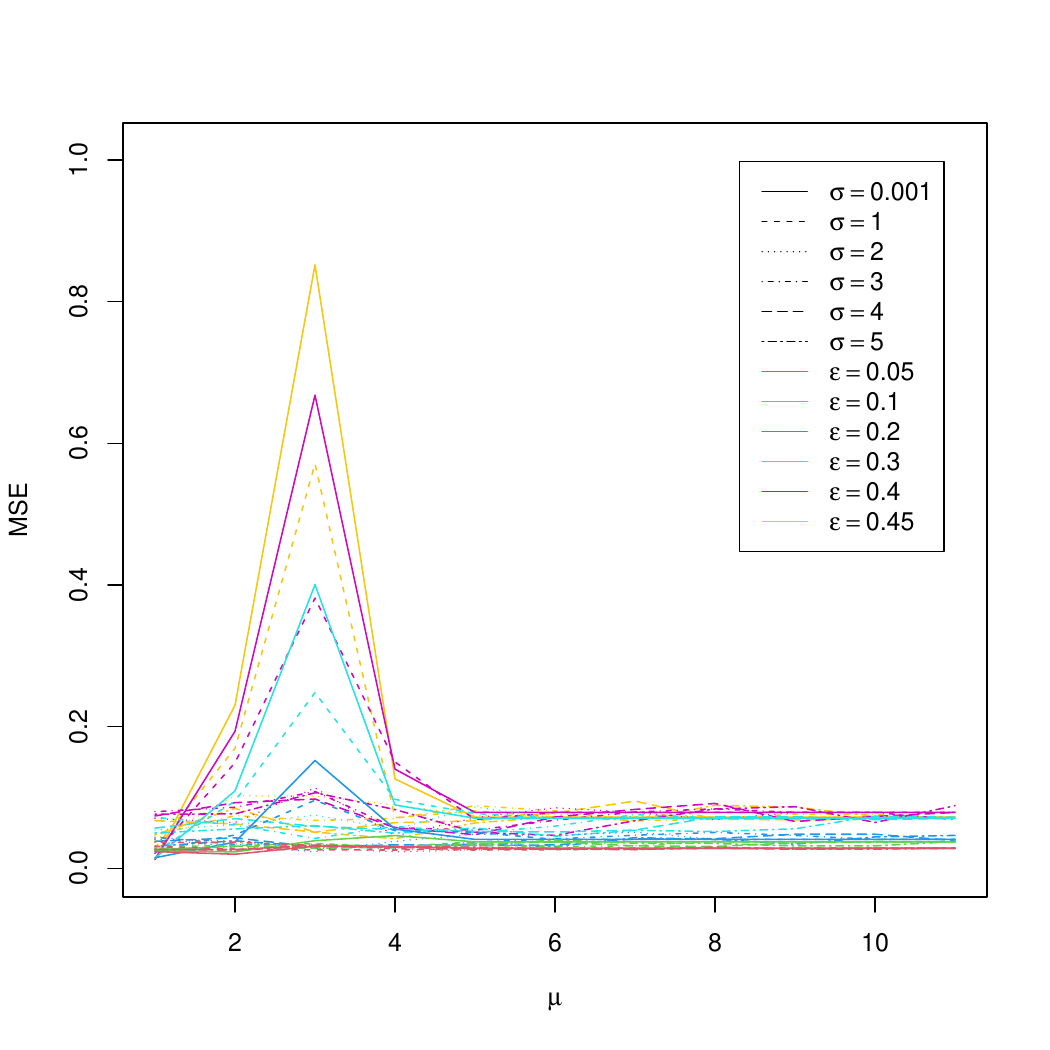} \\
\includegraphics[width=0.32\textwidth]{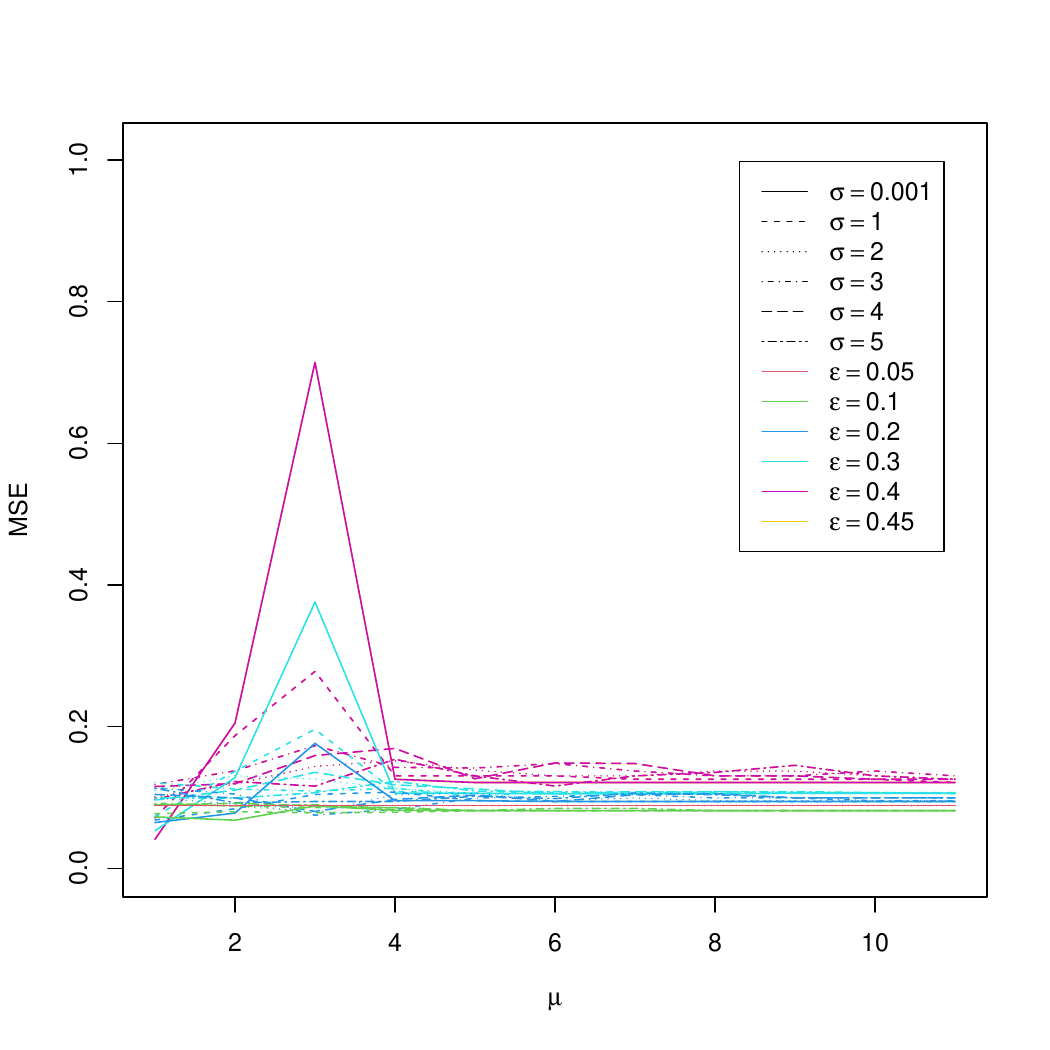}
\includegraphics[width=0.32\textwidth]{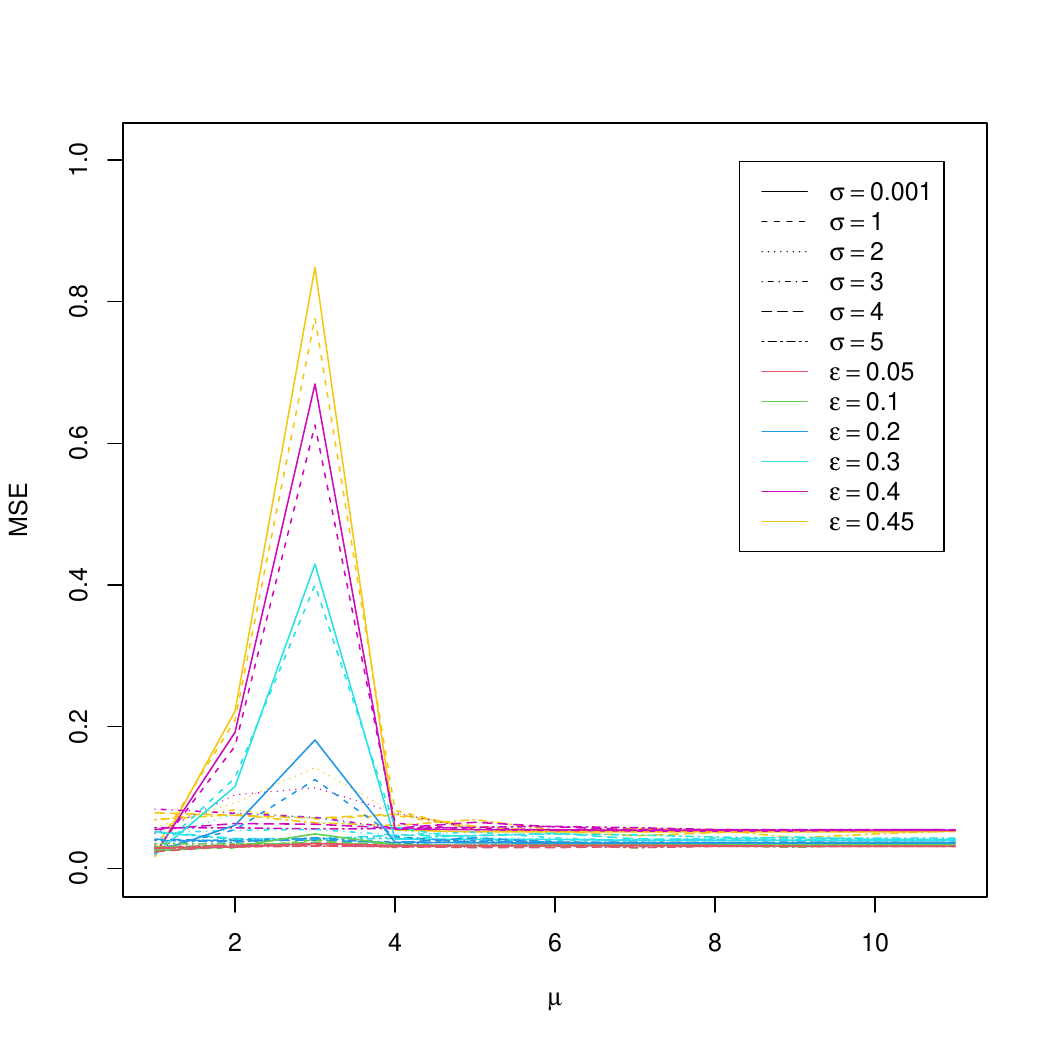} 
\includegraphics[width=0.32\textwidth]{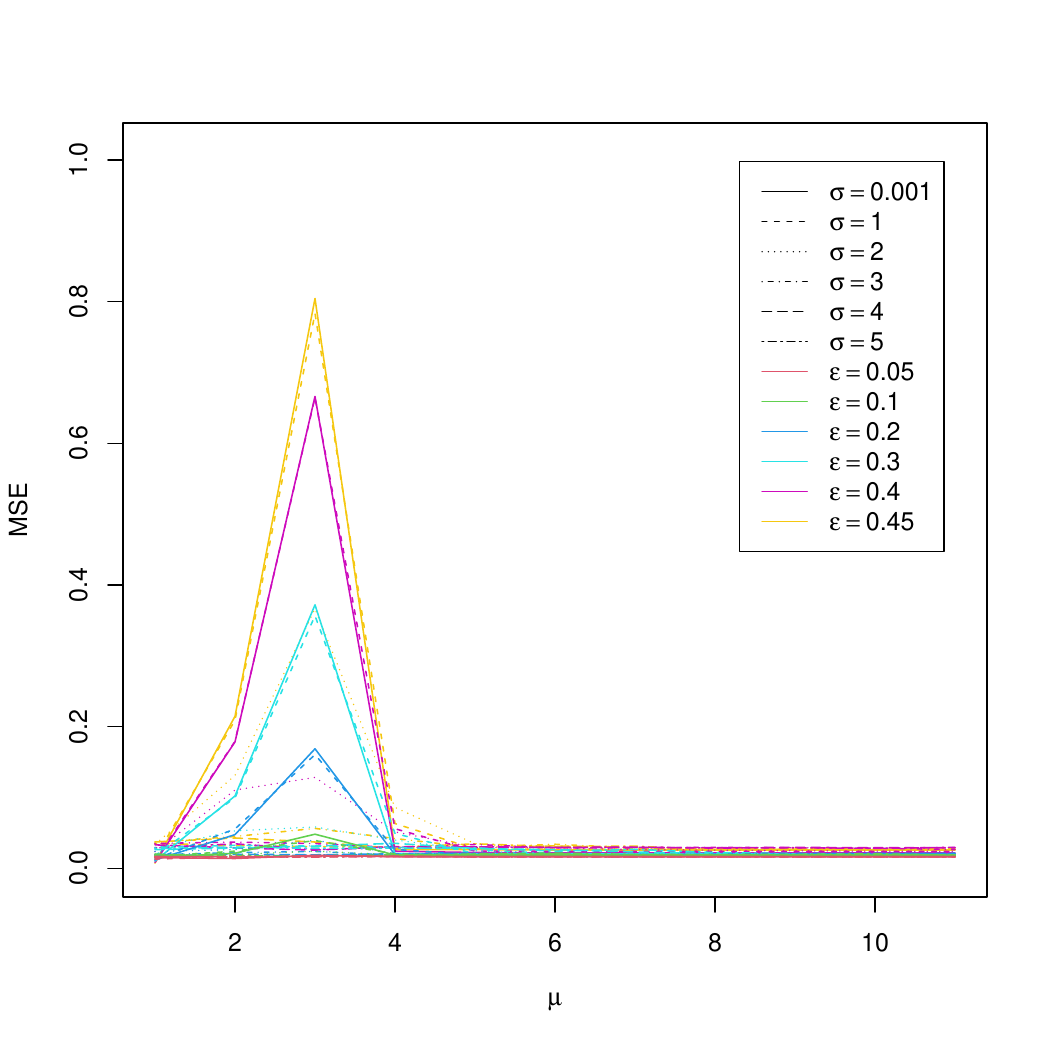} \\
\includegraphics[width=0.32\textwidth]{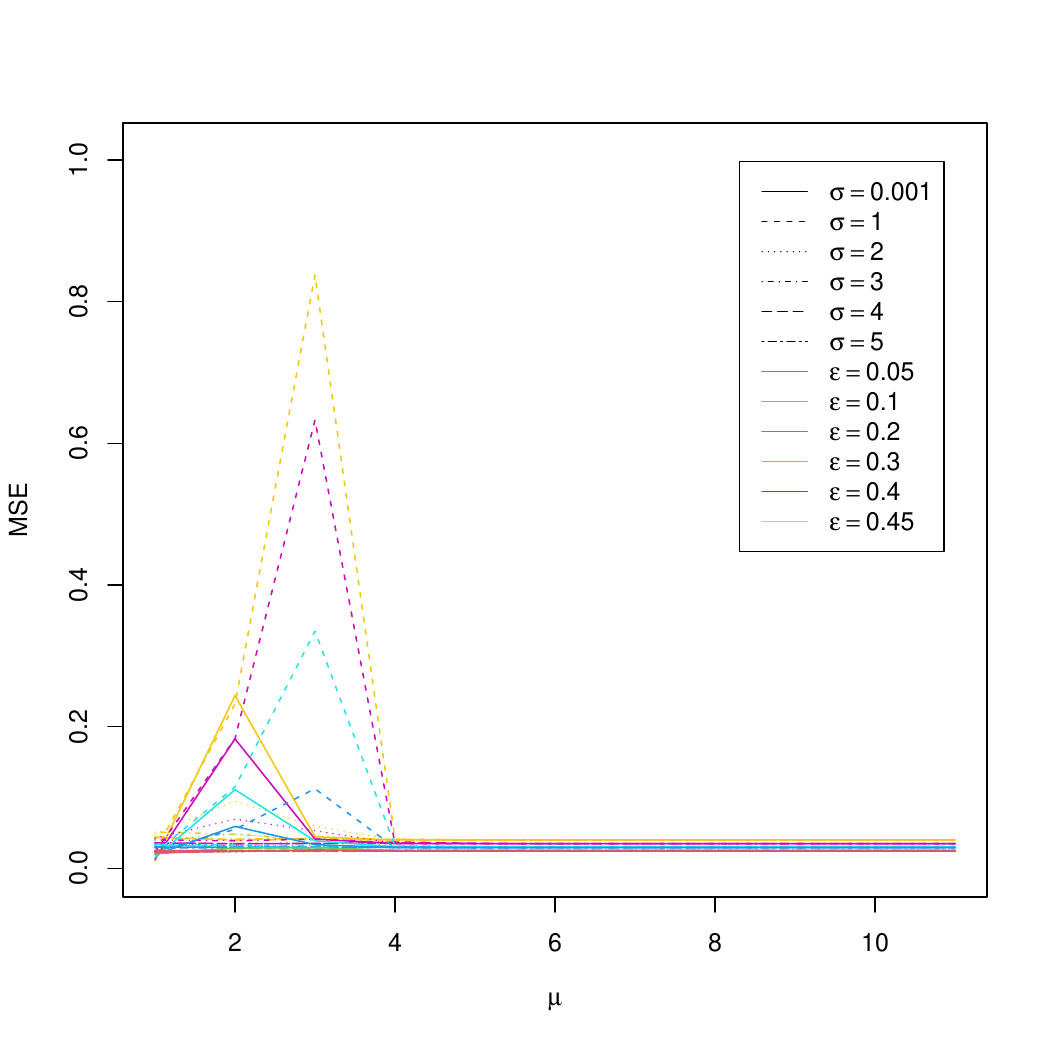}
\includegraphics[width=0.32\textwidth]{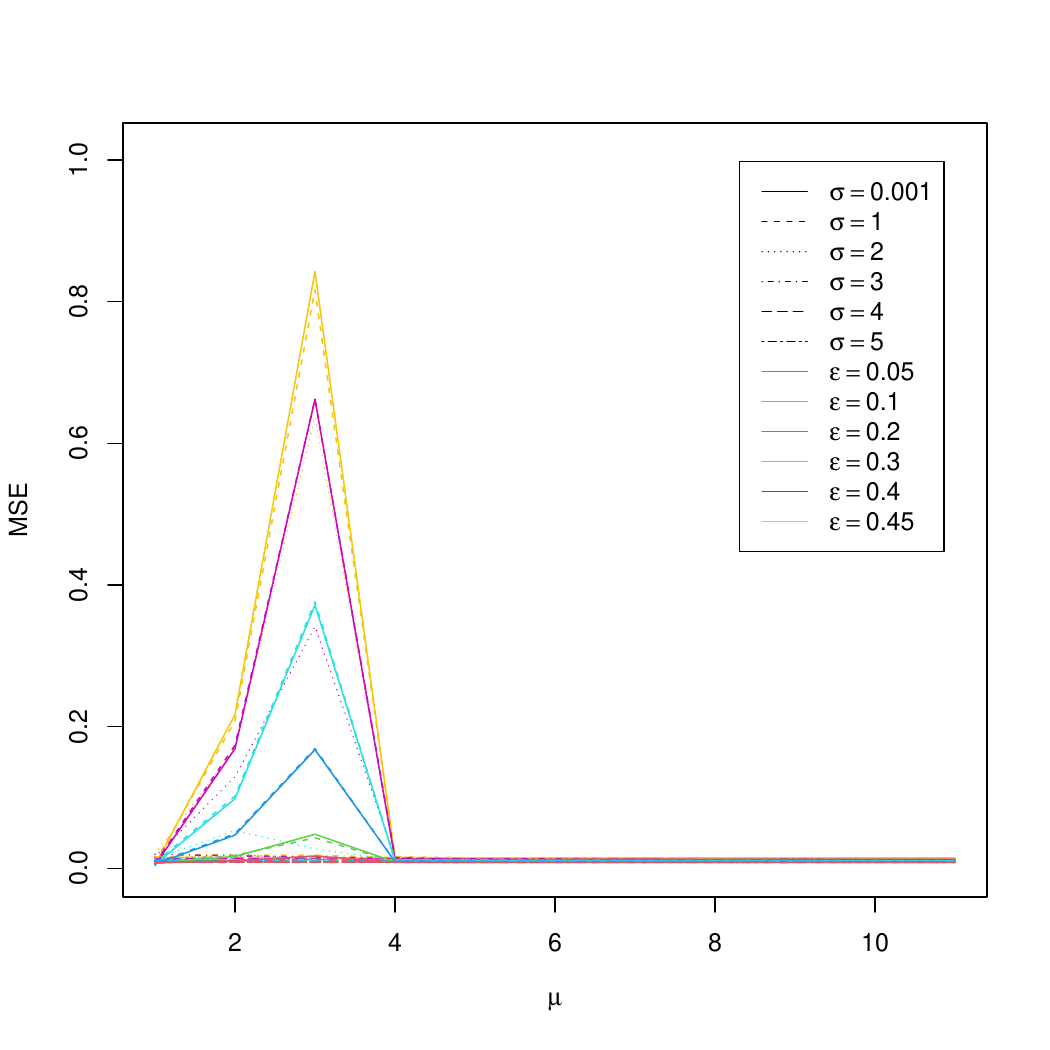} 
\includegraphics[width=0.32\textwidth]{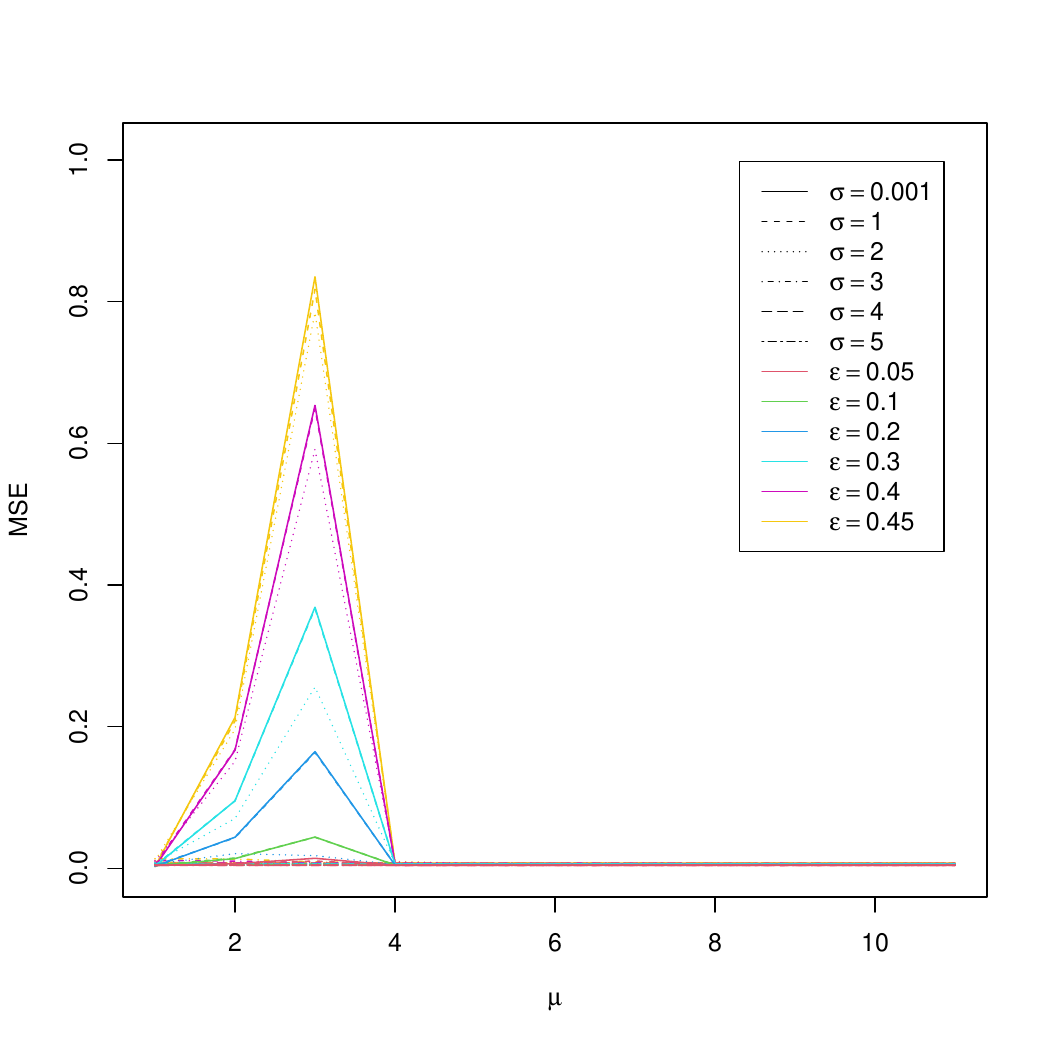} \\
\caption{Monte Carlo Simulation. Mean Square Error for the proposed method starting at the true values with $\alpha=0.75$ as a function of the contamination average $\mu$ ($x$-axis), contamination scale $\sigma$ (different line styles) and contamination level $\varepsilon$ (colors). Rows: number of variables $p=1, 2, 5$ and columns: sample size factor $s=2, 5, 10$.}
\label{sup:fig:monte:MSE:0.75:1}
\end{figure}  

\begin{figure}
\centering
\includegraphics[width=0.32\textwidth]{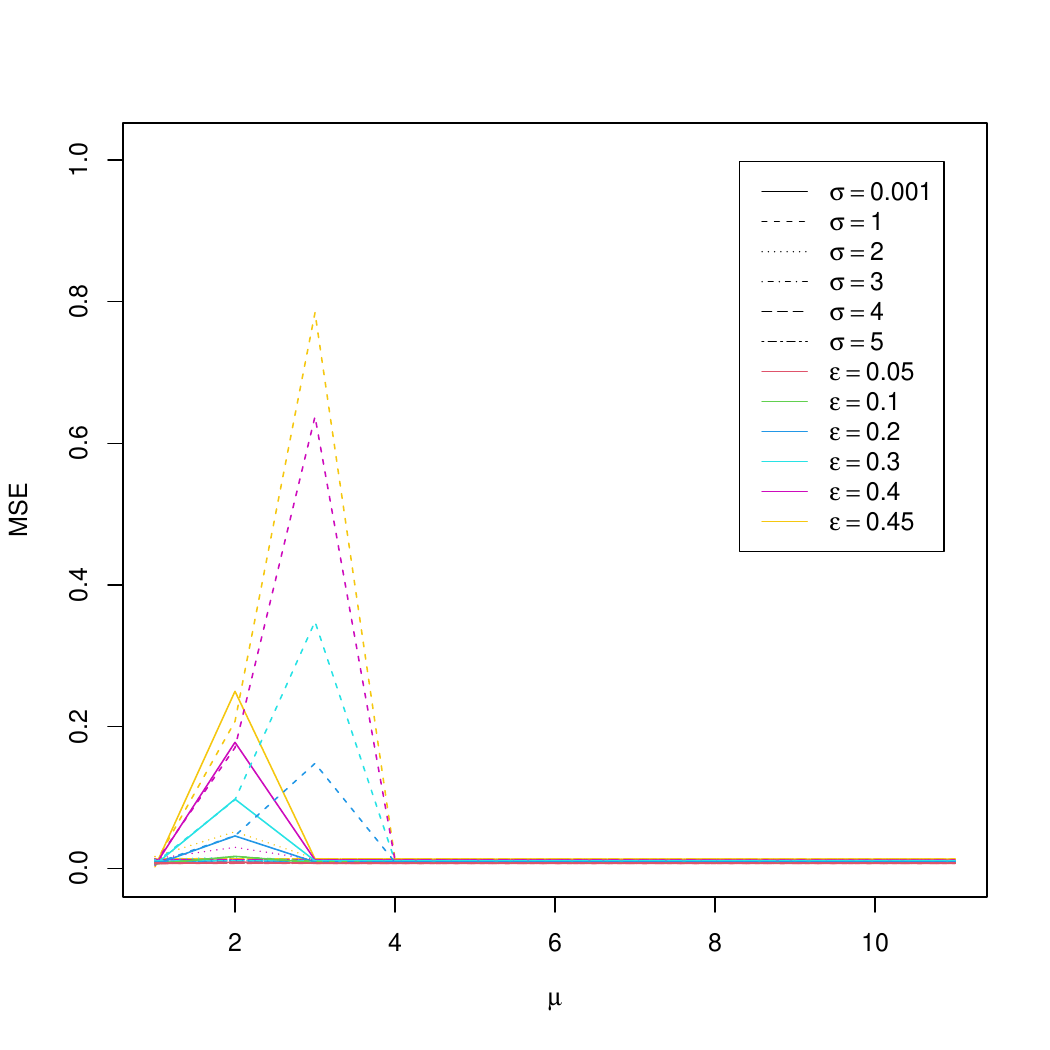}
\includegraphics[width=0.32\textwidth]{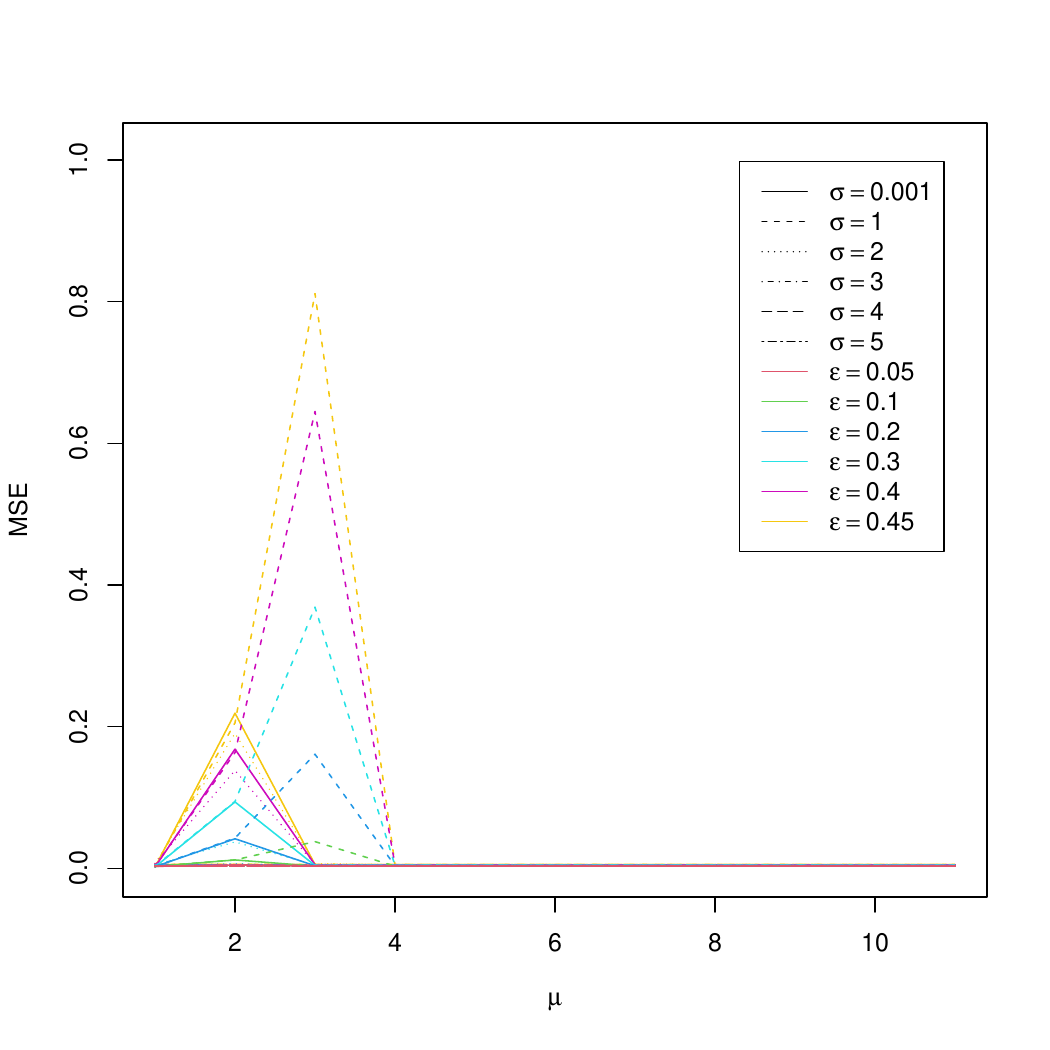} 
\includegraphics[width=0.32\textwidth]{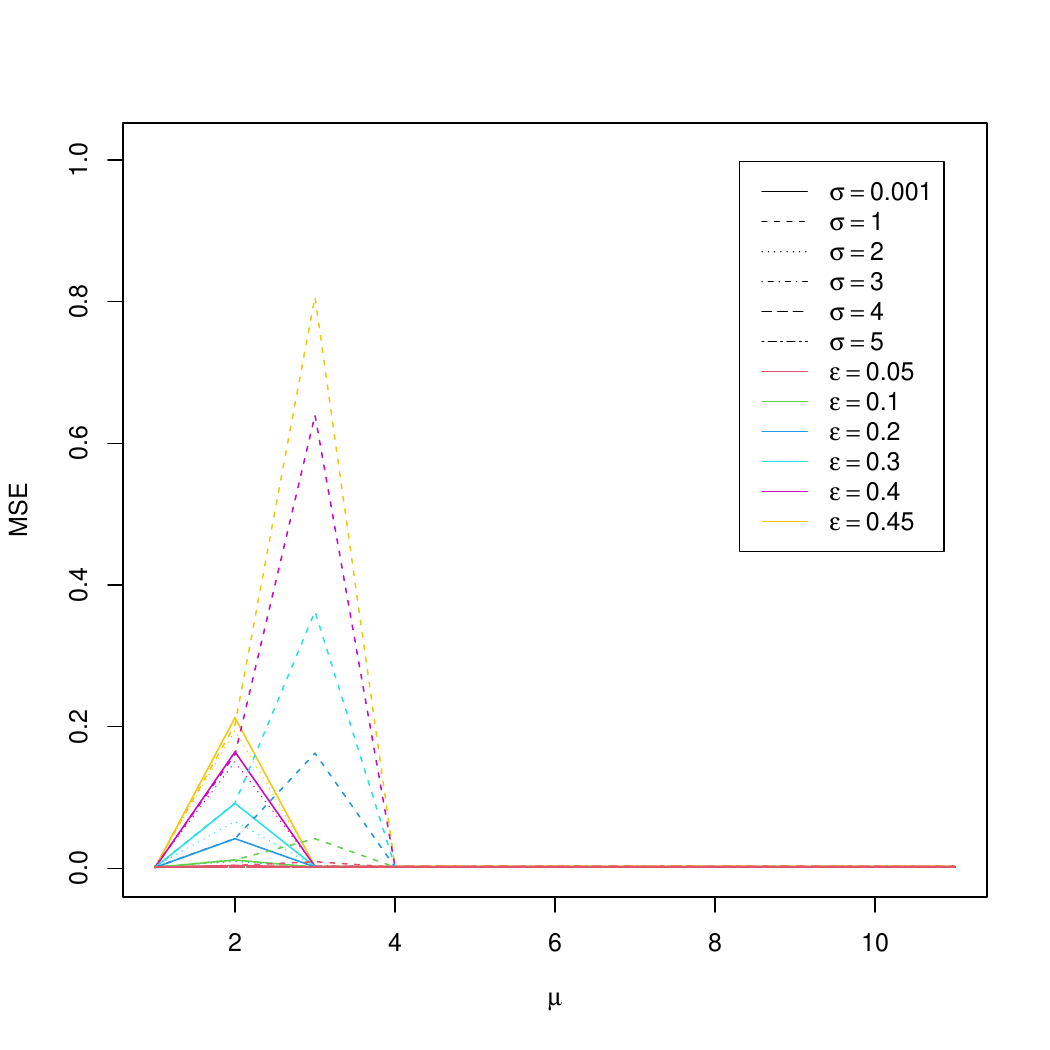} \\
\includegraphics[width=0.32\textwidth]{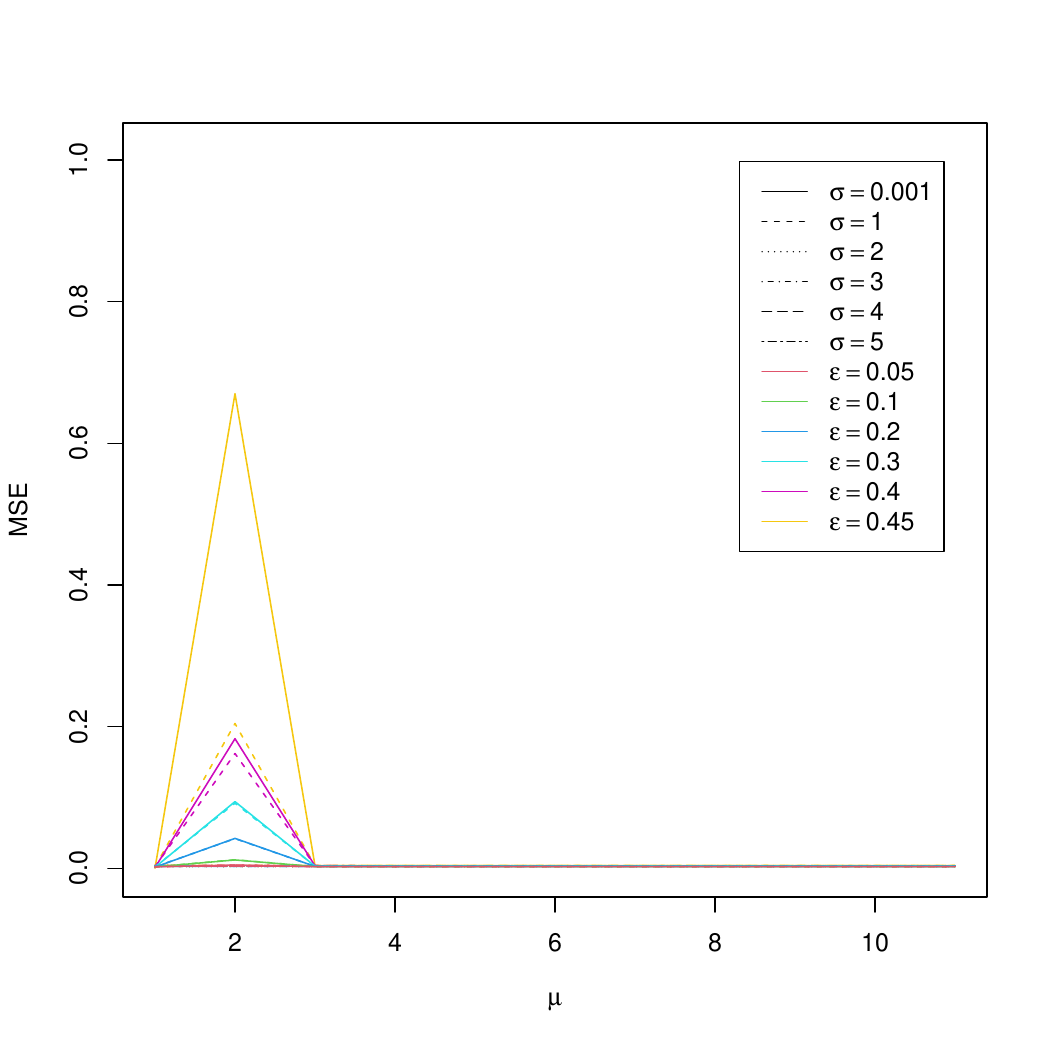}
\includegraphics[width=0.32\textwidth]{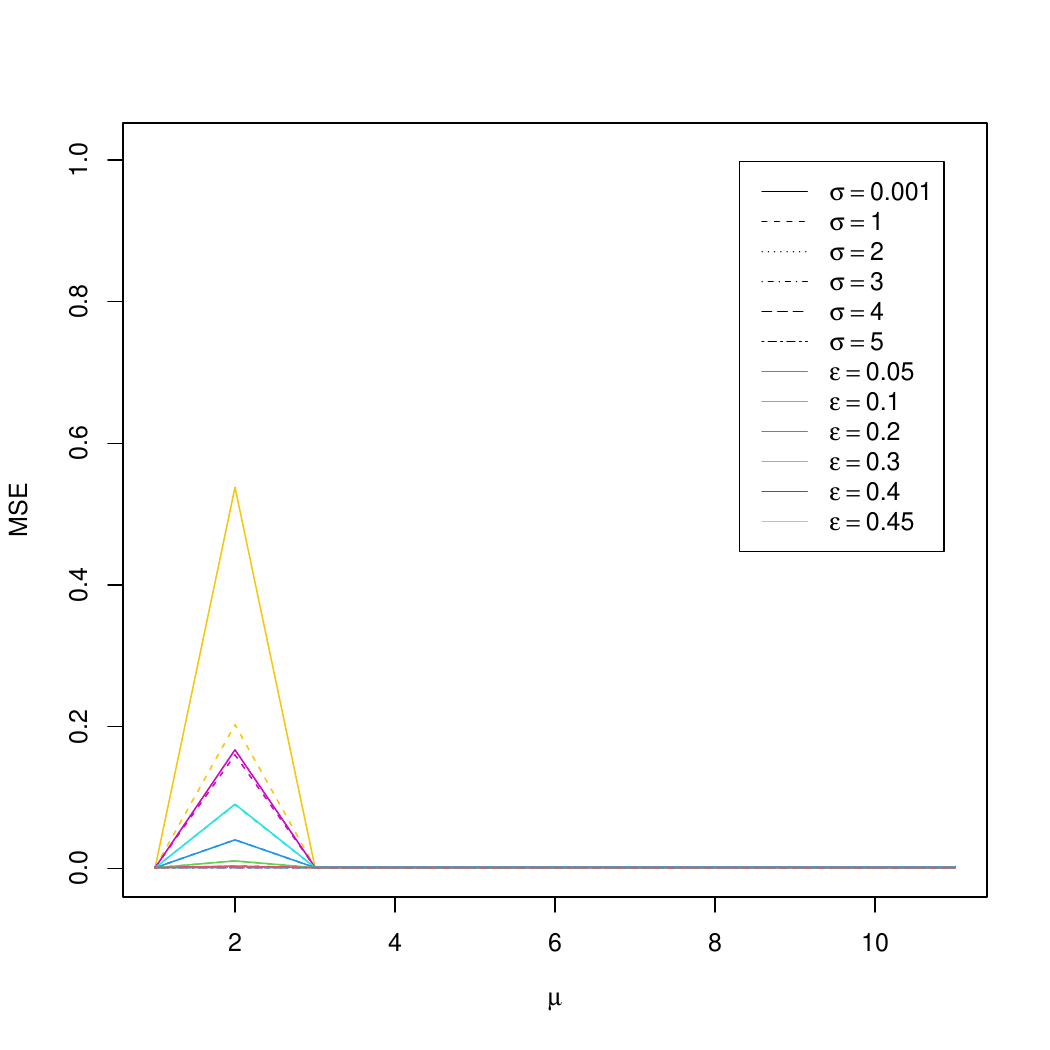} 
\includegraphics[width=0.32\textwidth]{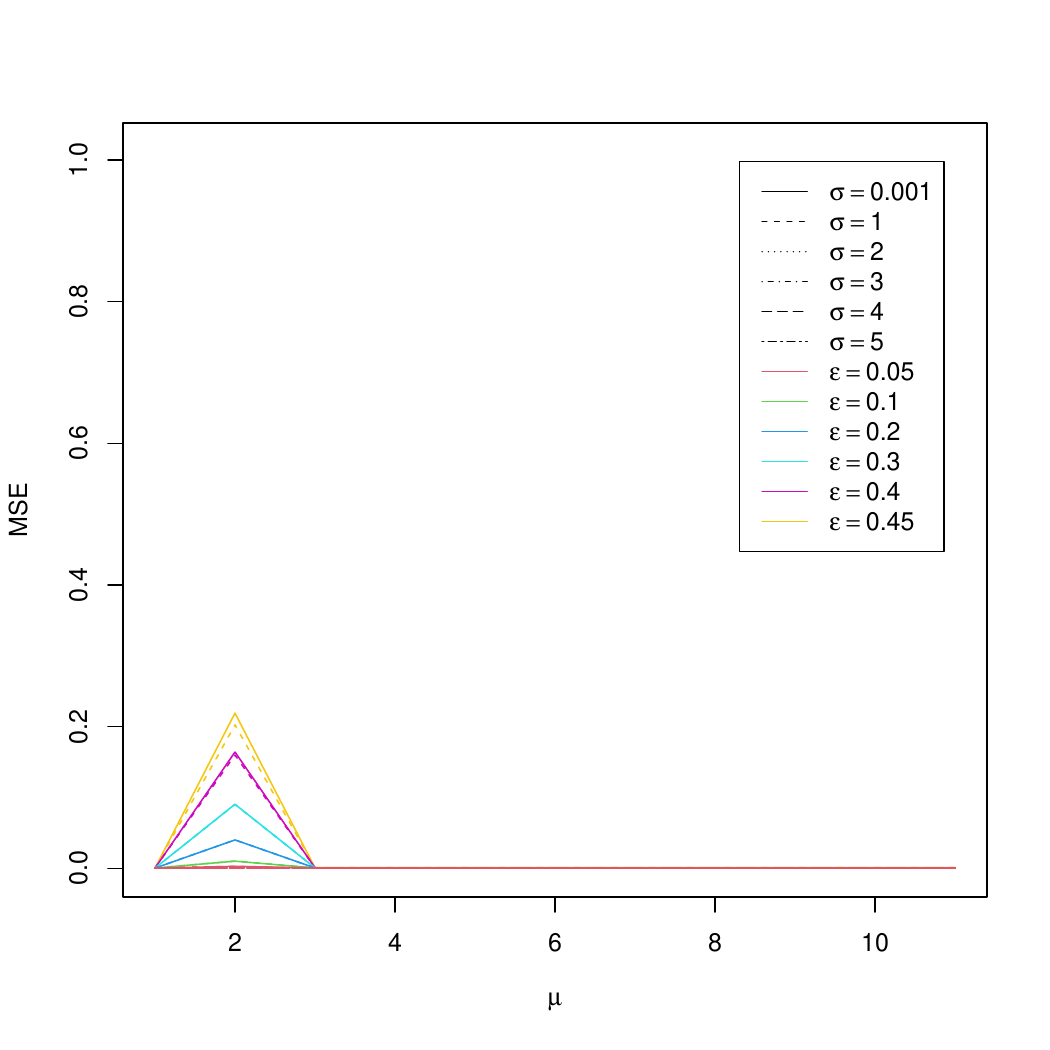}
\caption{Monte Carlo Simulation. Mean Square Error for the proposed method starting at the true values with $\alpha=0.75$ as a function of the contamination average $\mu$ ($x$-axis), contamination scale $\sigma$ (different line styles) and contamination level $\varepsilon$ (colors). Rows: number of variables $p=10, 20$ and columns: sample size factor $s=2, 5, 10$.}
\label{sup:fig:monte:MSE:0.75:2}
\end{figure}  

\begin{figure}
\centering
\includegraphics[width=0.32\textwidth]{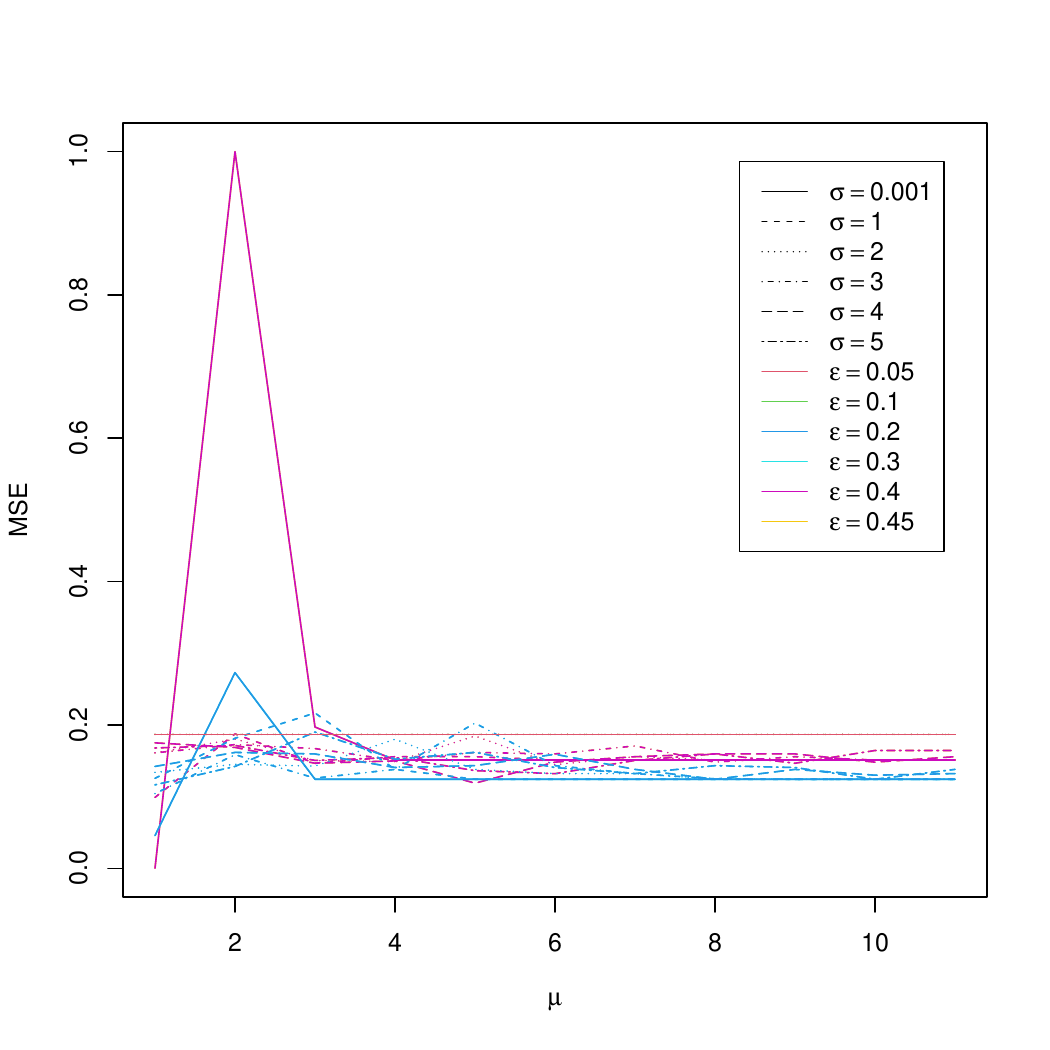}
\includegraphics[width=0.32\textwidth]{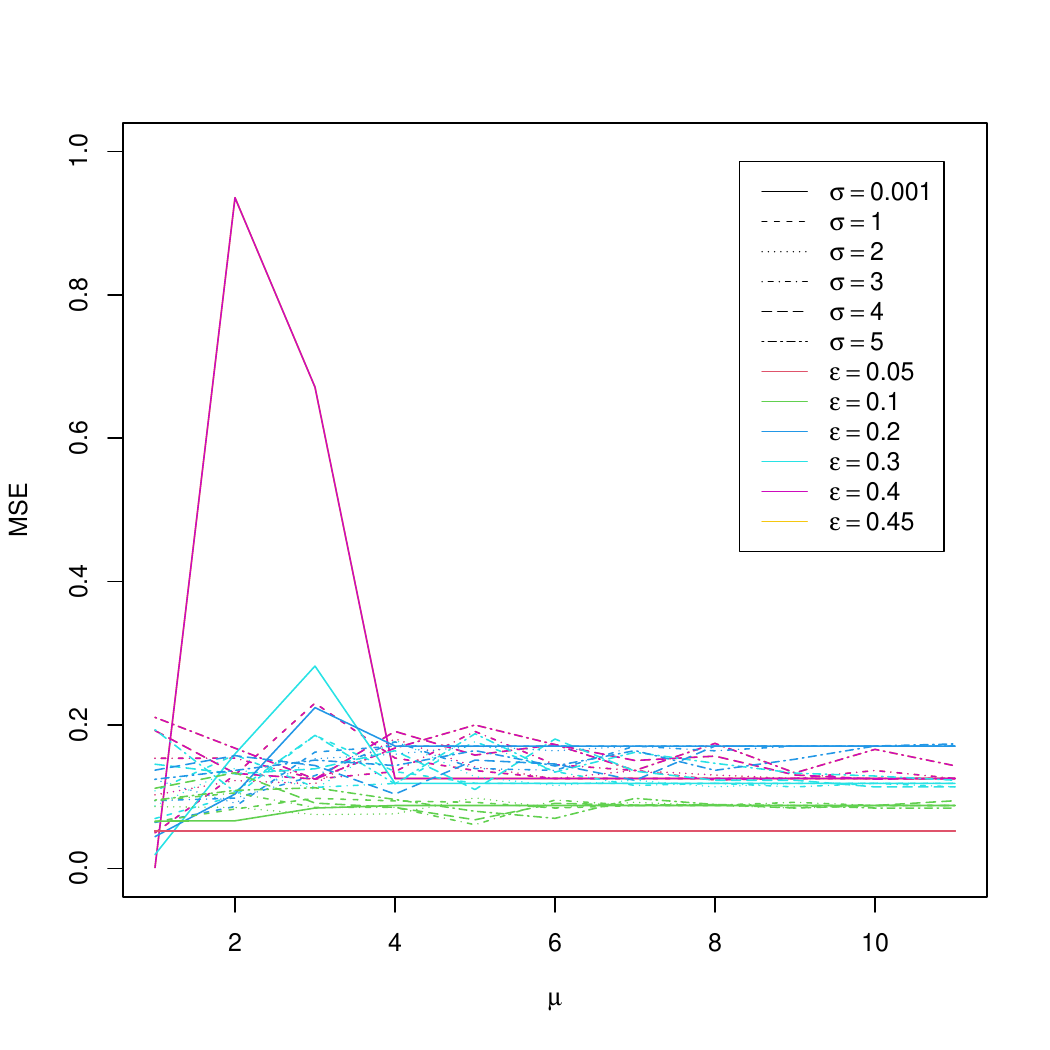} 
\includegraphics[width=0.32\textwidth]{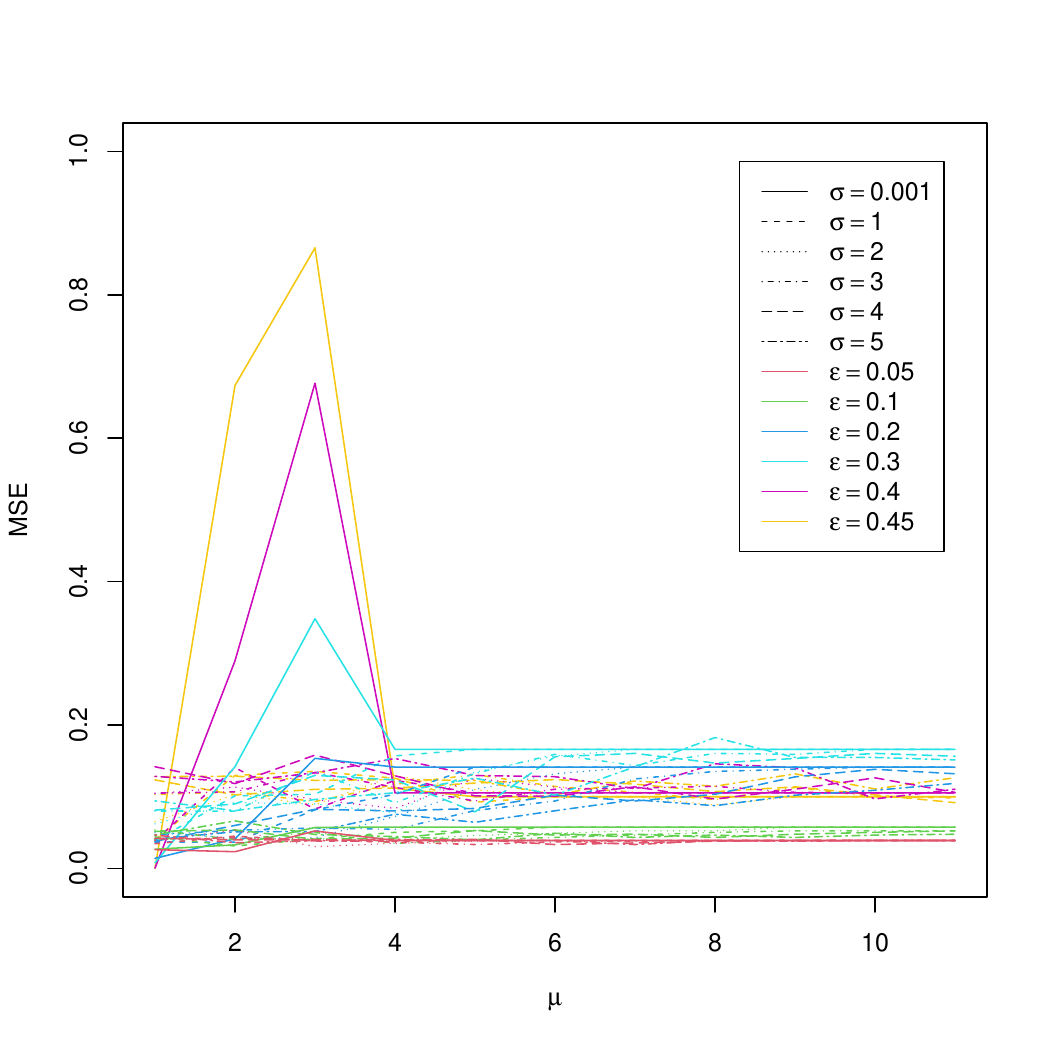} \\
\includegraphics[width=0.32\textwidth]{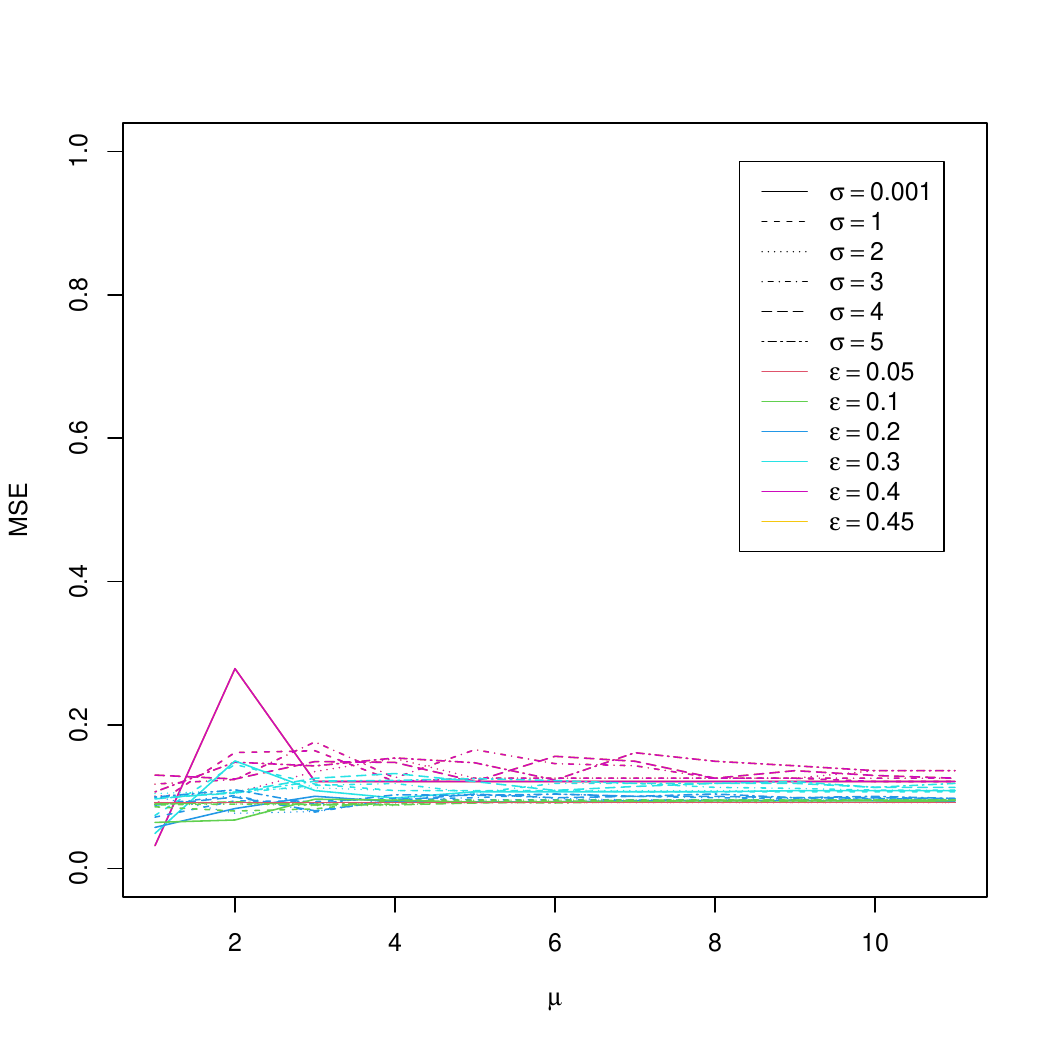}
\includegraphics[width=0.32\textwidth]{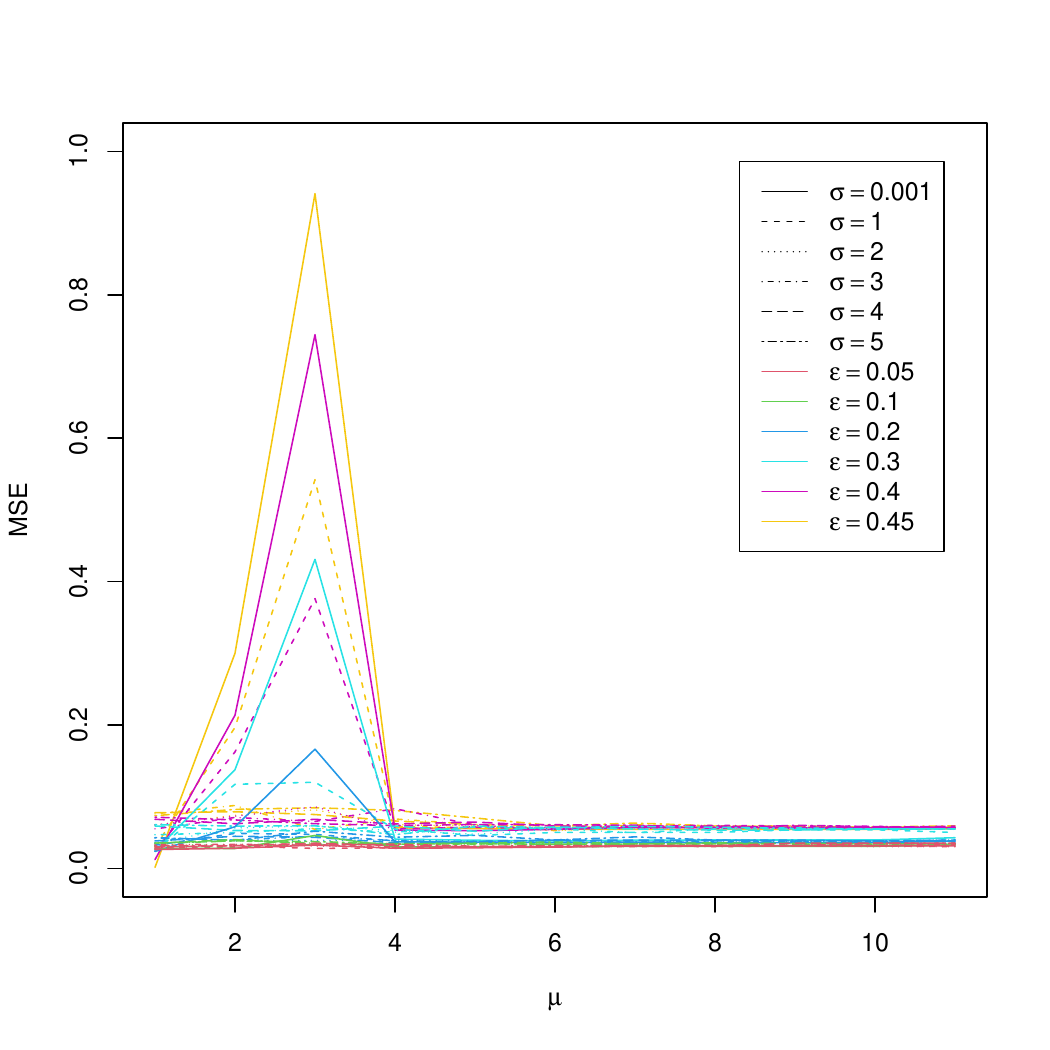} 
\includegraphics[width=0.32\textwidth]{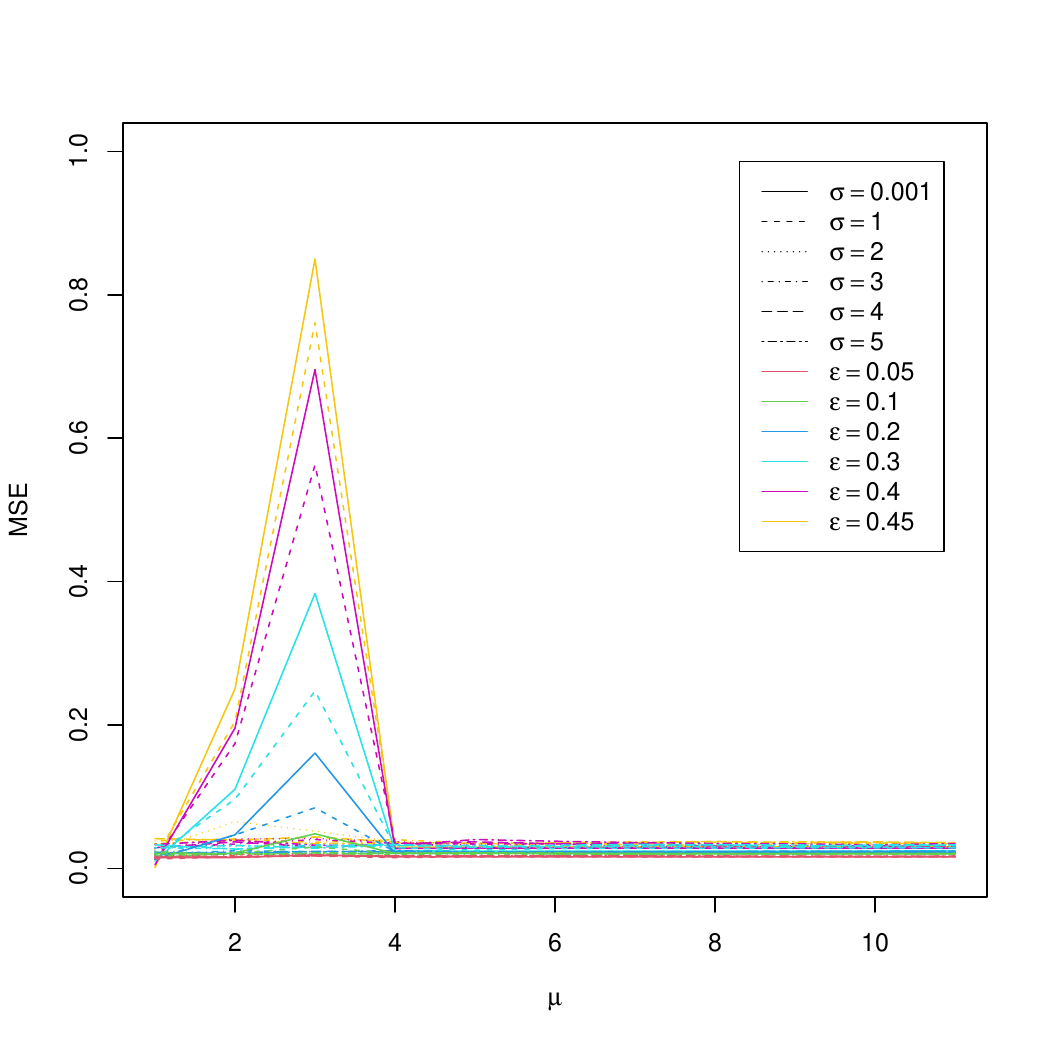} \\
\includegraphics[width=0.32\textwidth]{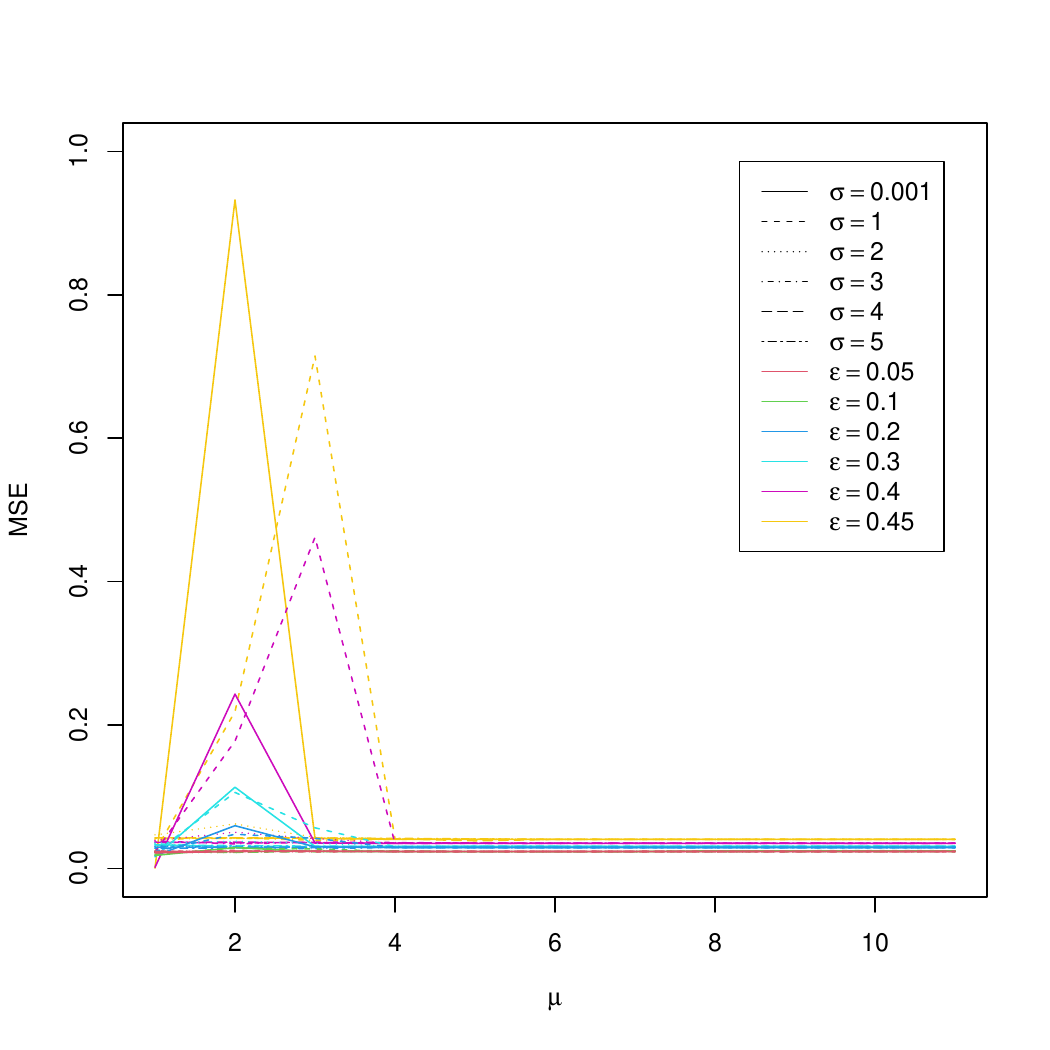}
\includegraphics[width=0.32\textwidth]{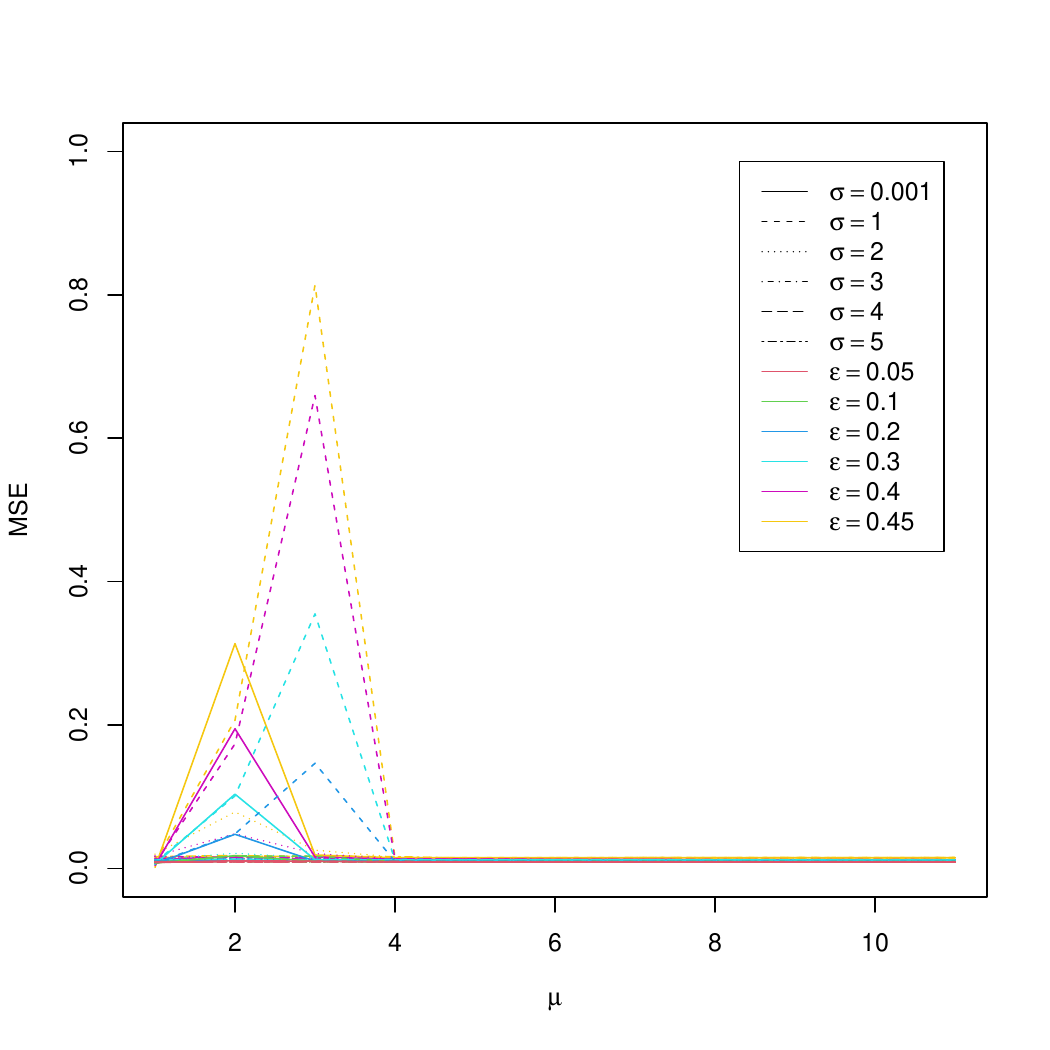} 
\includegraphics[width=0.32\textwidth]{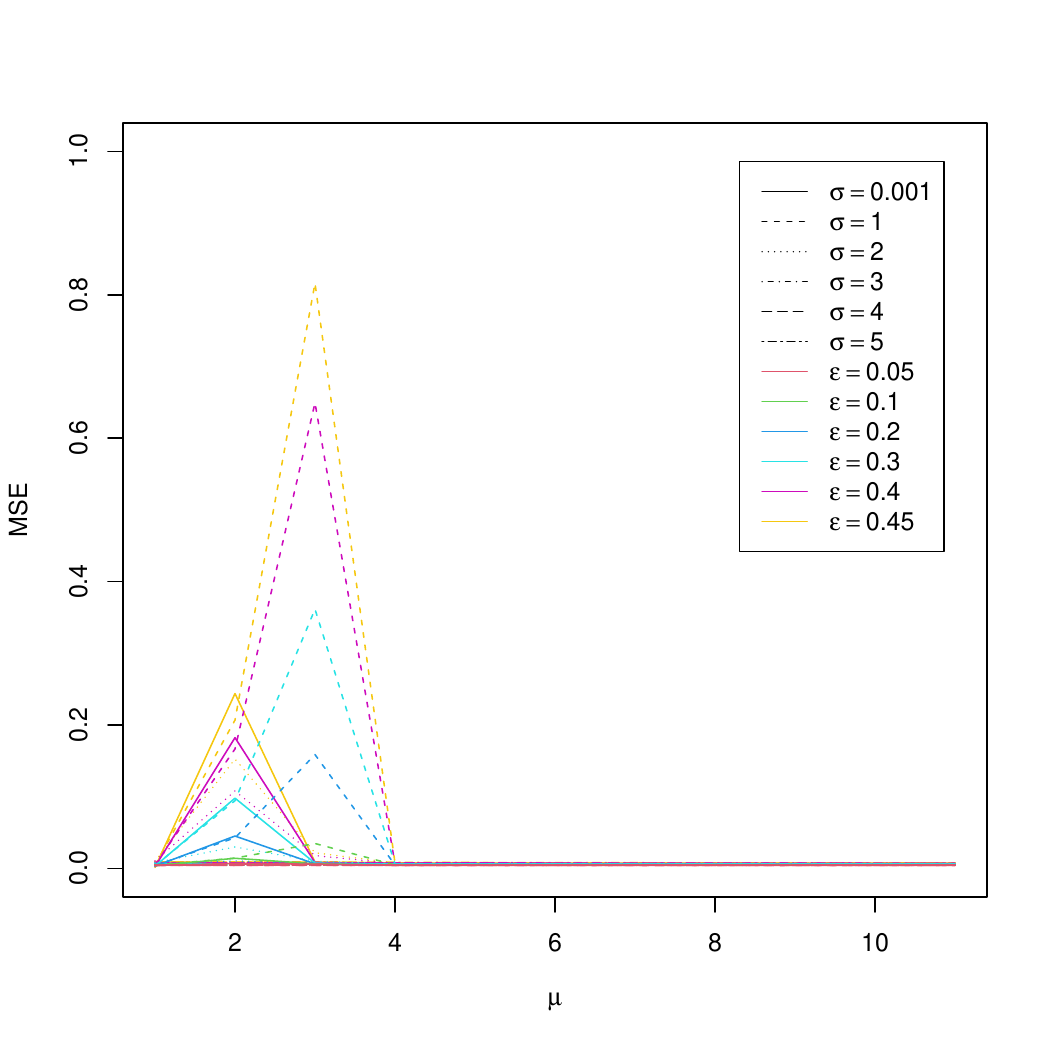} \\
\caption{Monte Carlo Simulation. Mean Square Error for the proposed method starting at the true values with $\alpha=1$ as a function of the contamination average $\mu$ ($x$-axis), contamination scale $\sigma$ (different line styles) and contamination level $\varepsilon$ (colors). Rows: number of variables $p=1, 2, 5$ and columns: sample size factor $s=2, 5, 10$.}
\label{sup:fig:monte:MSE:1:1}
\end{figure}  

\begin{figure}
\centering
\includegraphics[width=0.32\textwidth]{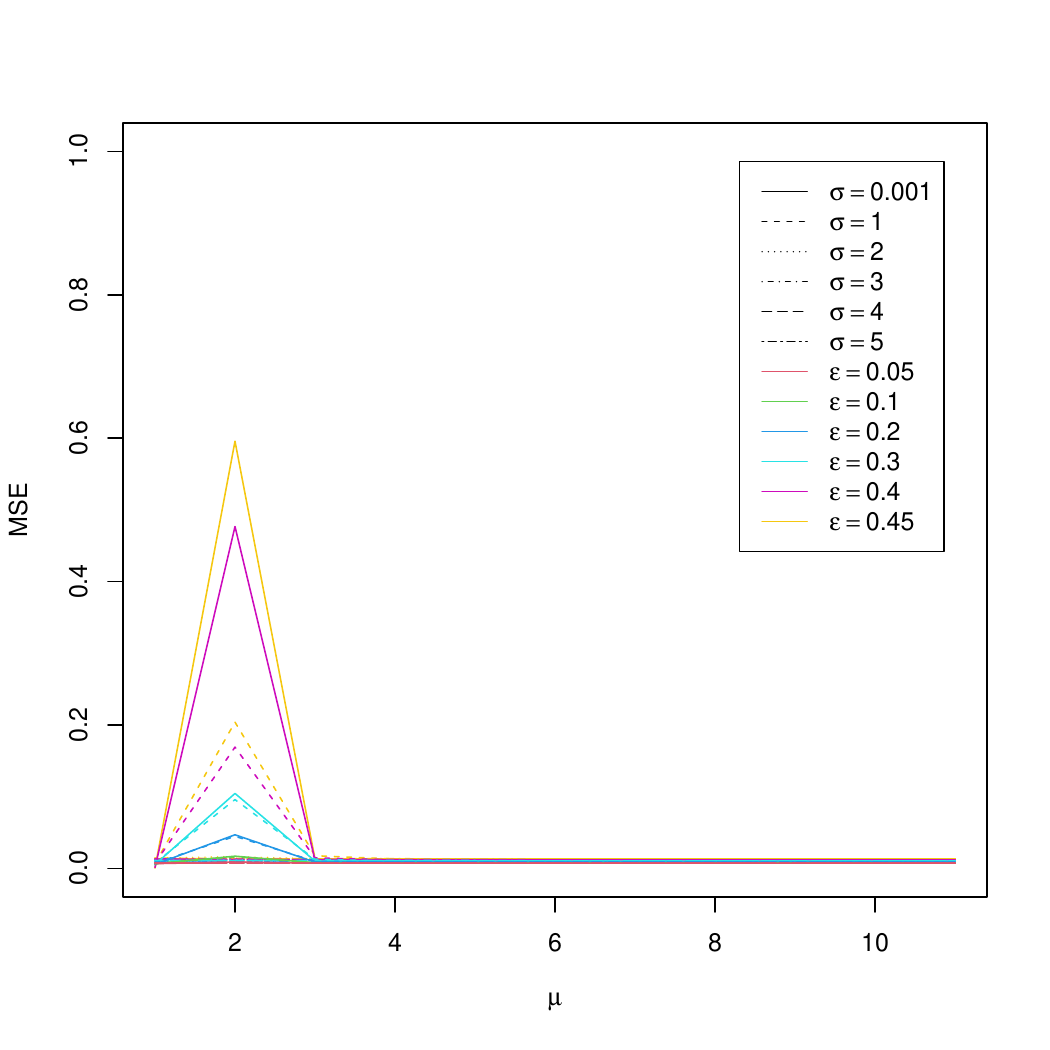}
\includegraphics[width=0.32\textwidth]{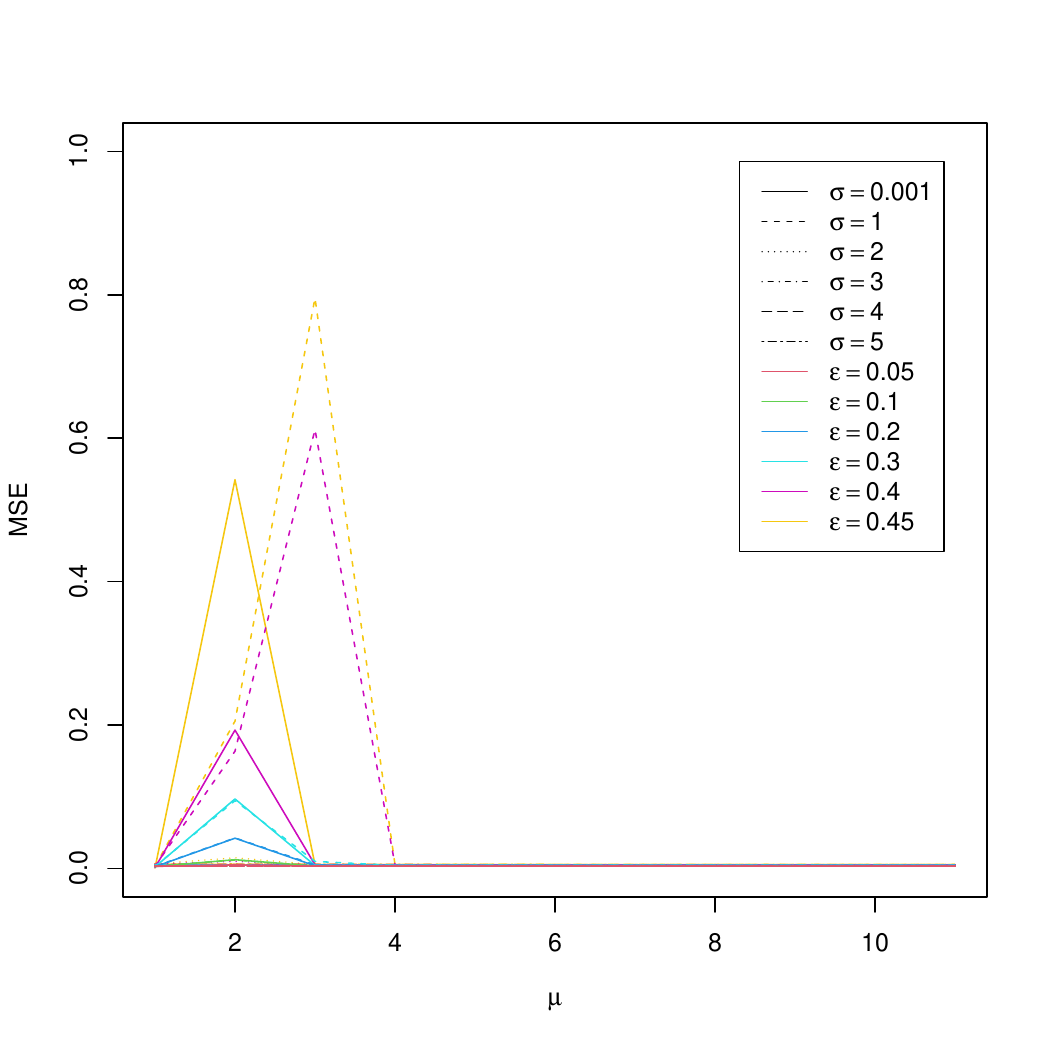} 
\includegraphics[width=0.32\textwidth]{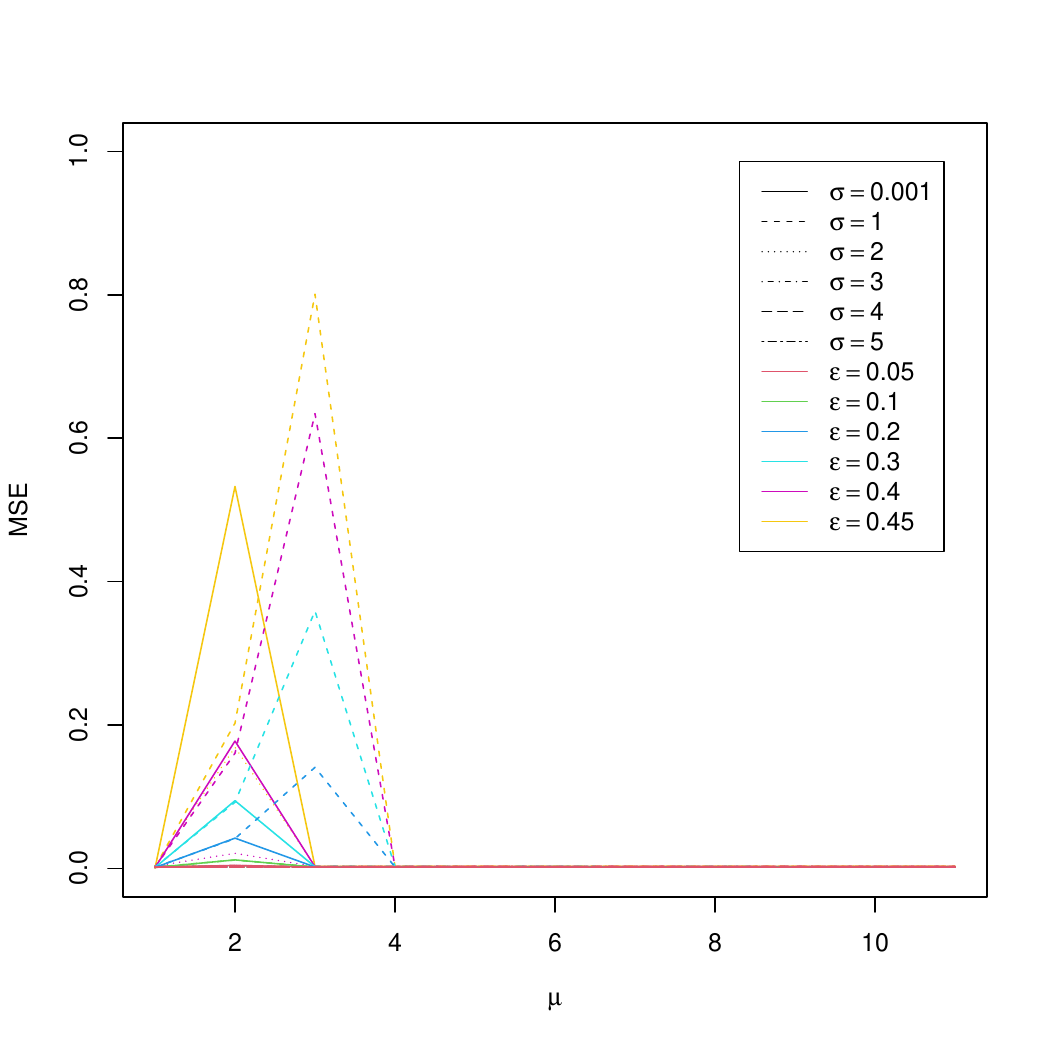} \\
\includegraphics[width=0.32\textwidth]{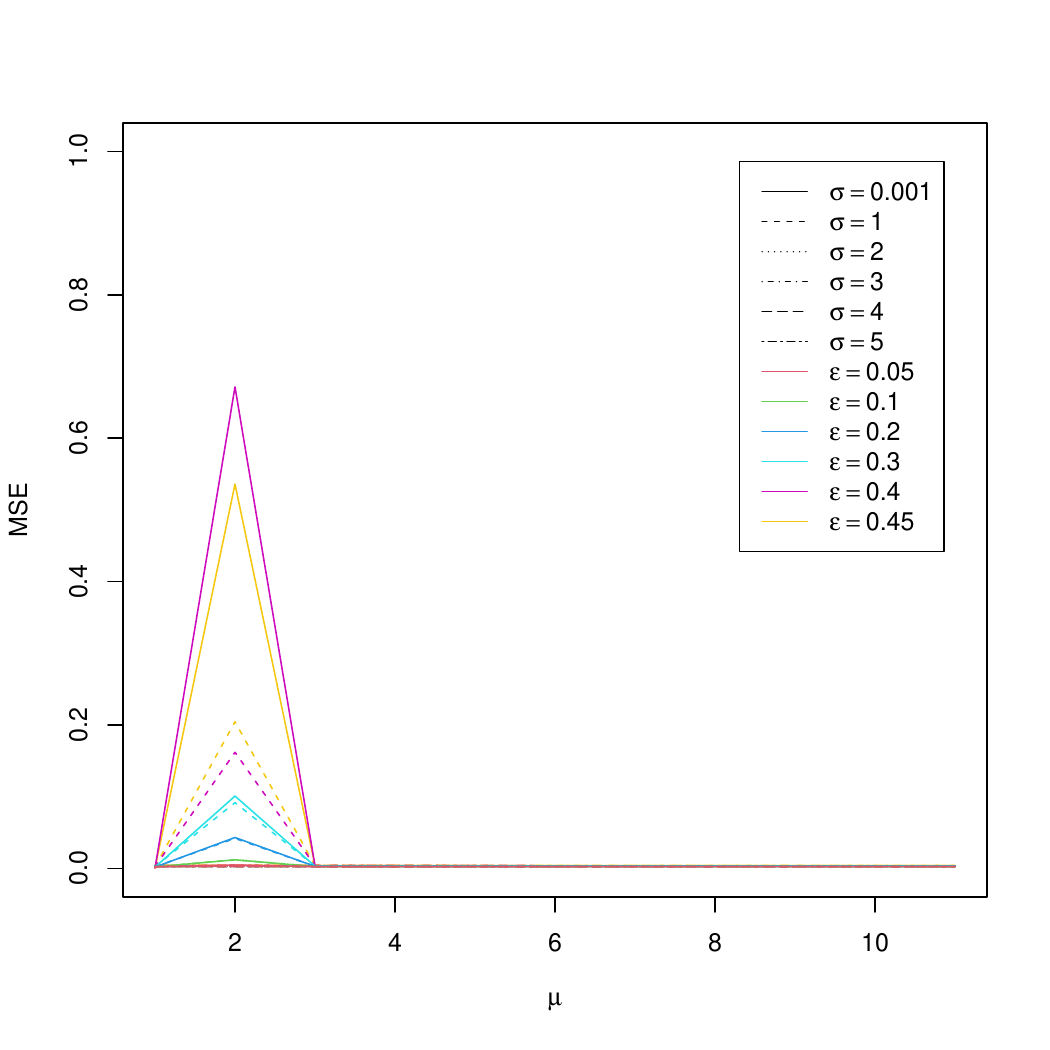}
\includegraphics[width=0.32\textwidth]{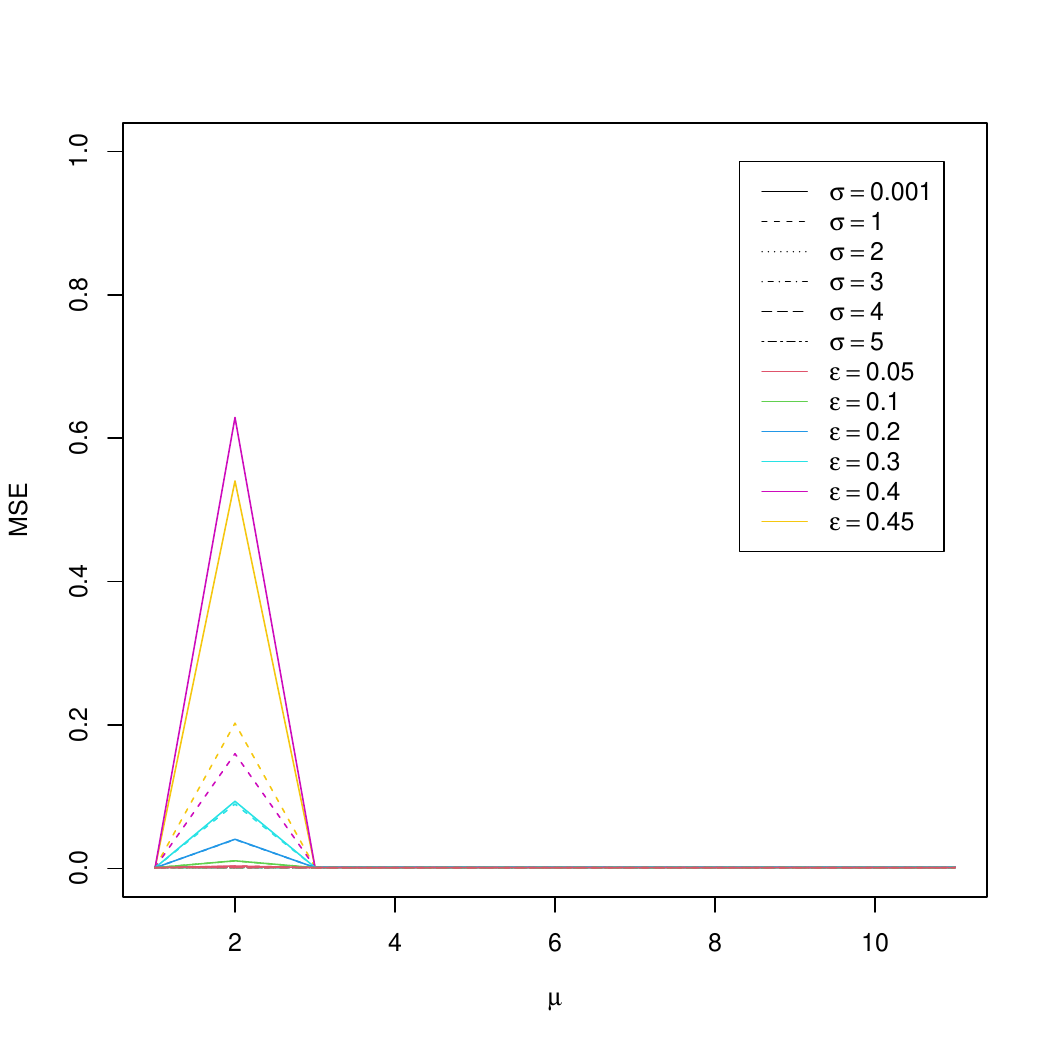} 
\includegraphics[width=0.32\textwidth]{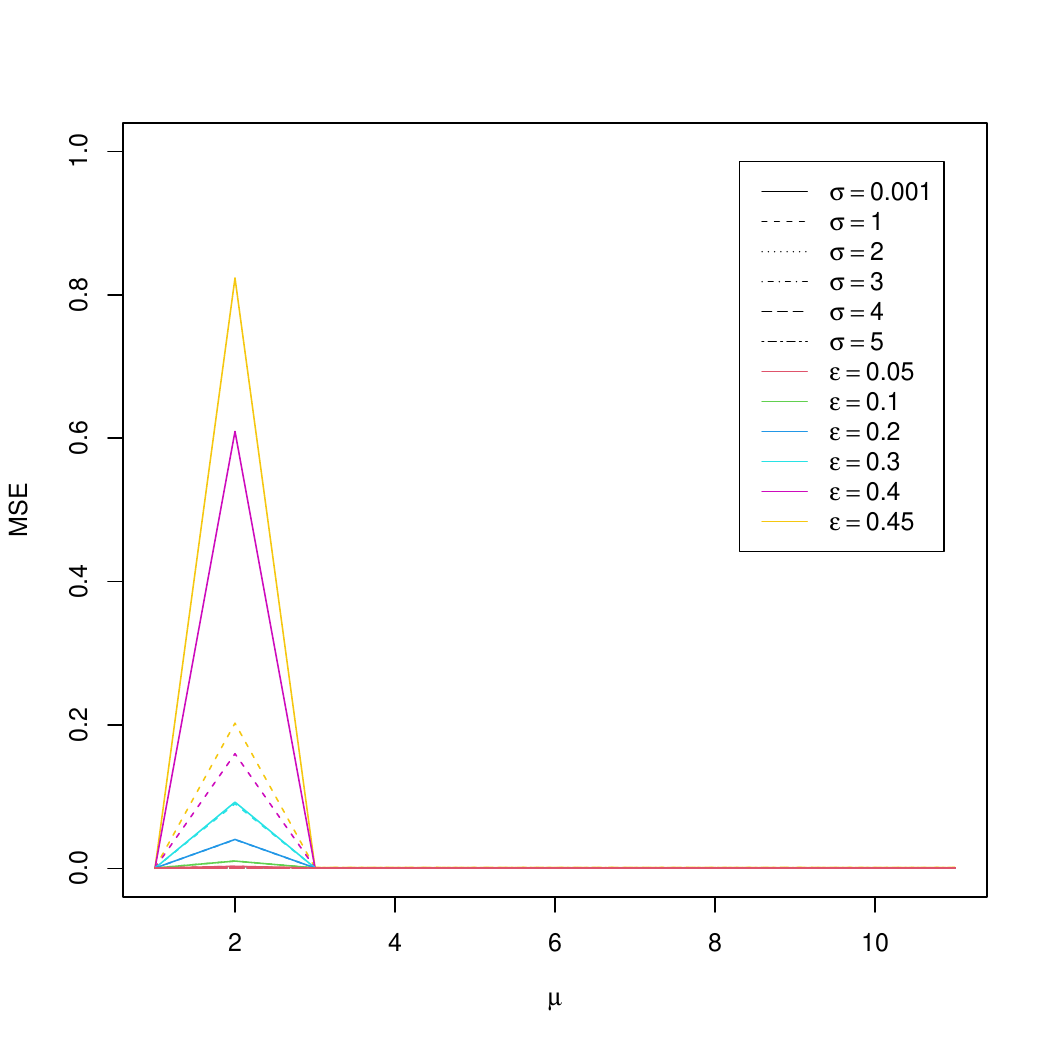}
\caption{Monte Carlo Simulation. Mean Square Error for the proposed method starting at the true values with $\alpha=1$ as a function of the contamination average $\mu$ ($x$-axis), contamination scale $\sigma$ (different line styles) and contamination level $\varepsilon$ (colors). Rows: number of variables $p=10, 20$ and columns: sample size factor $s=2, 5, 10$.}
\label{sup:fig:monte:MSE:1:2}
\end{figure}  

\clearpage


\begin{figure}
\centering
\includegraphics[width=0.32\textwidth]{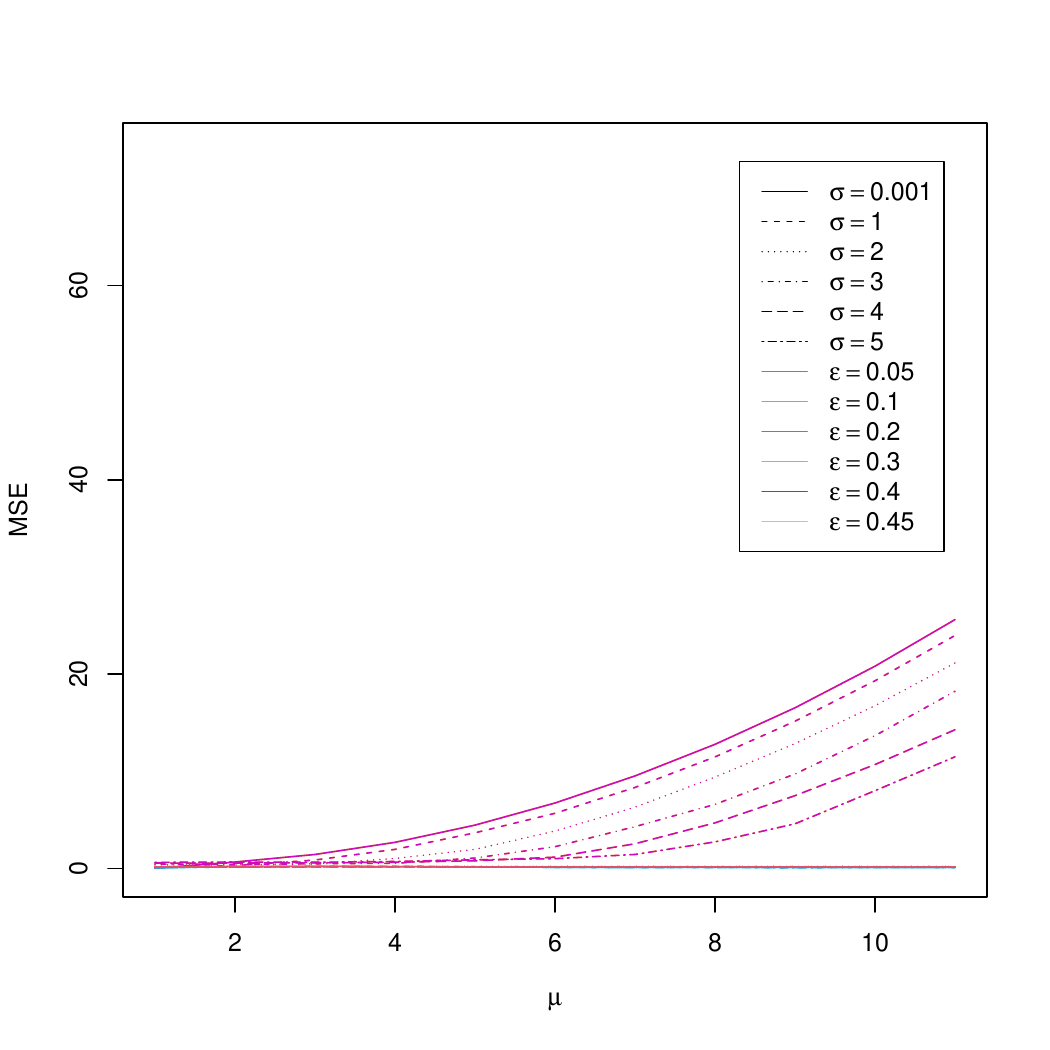}
\includegraphics[width=0.32\textwidth]{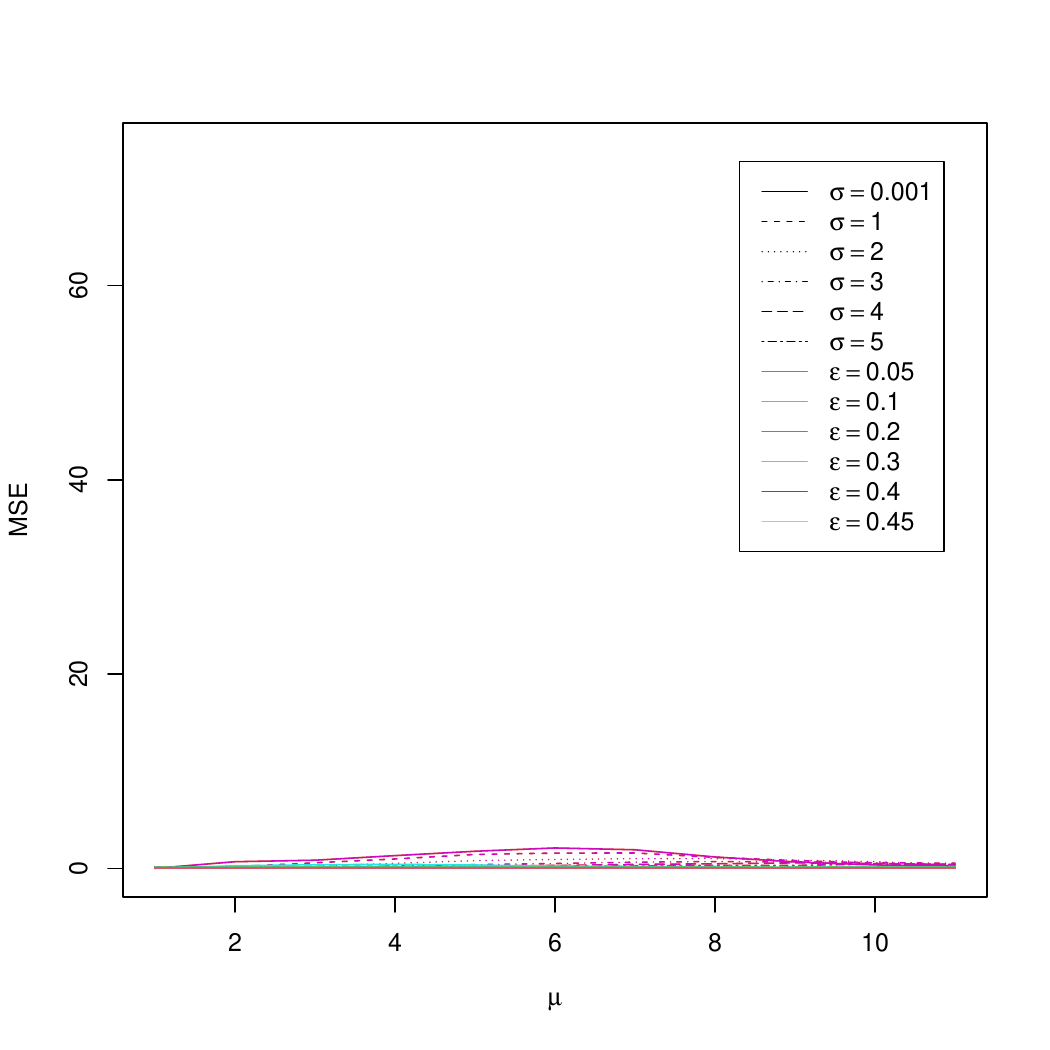} 
\includegraphics[width=0.32\textwidth]{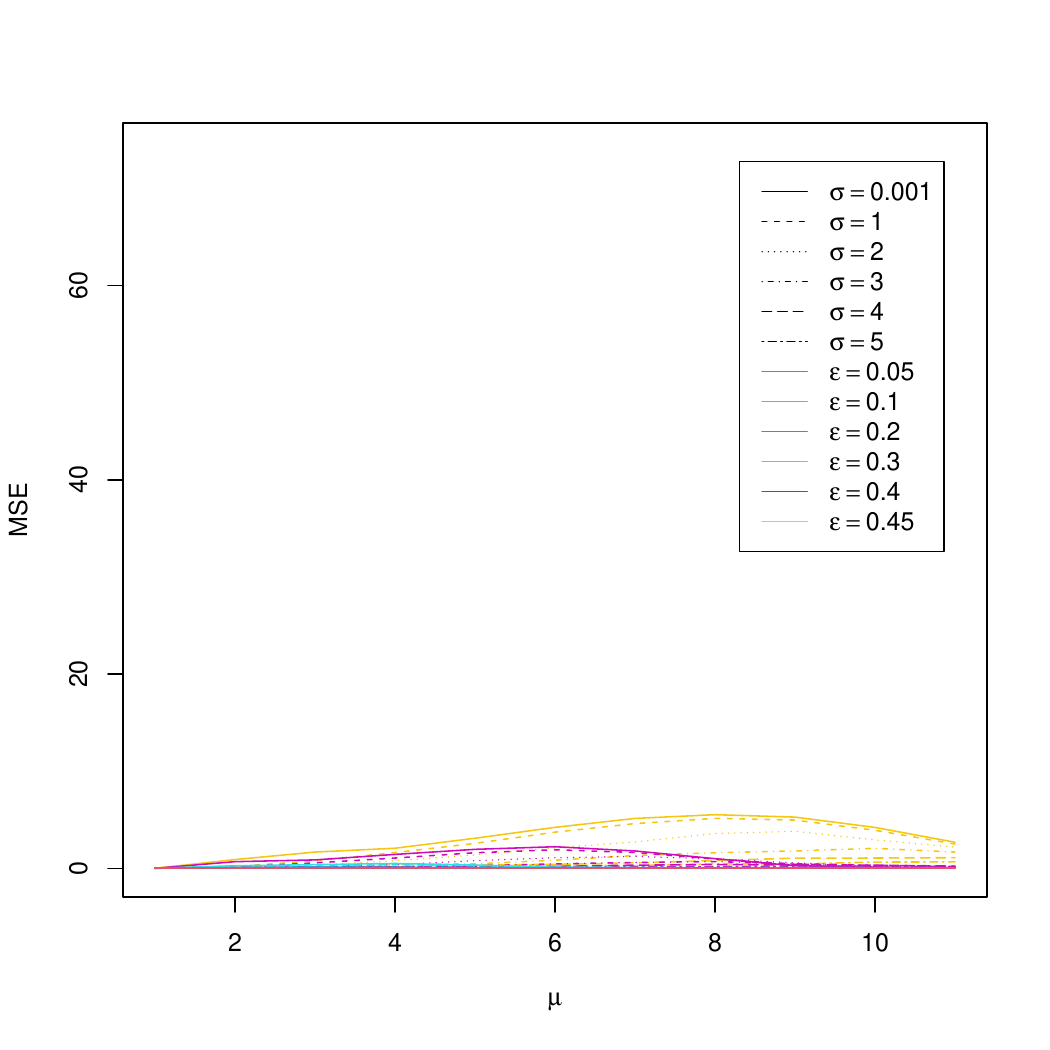} \\
\includegraphics[width=0.32\textwidth]{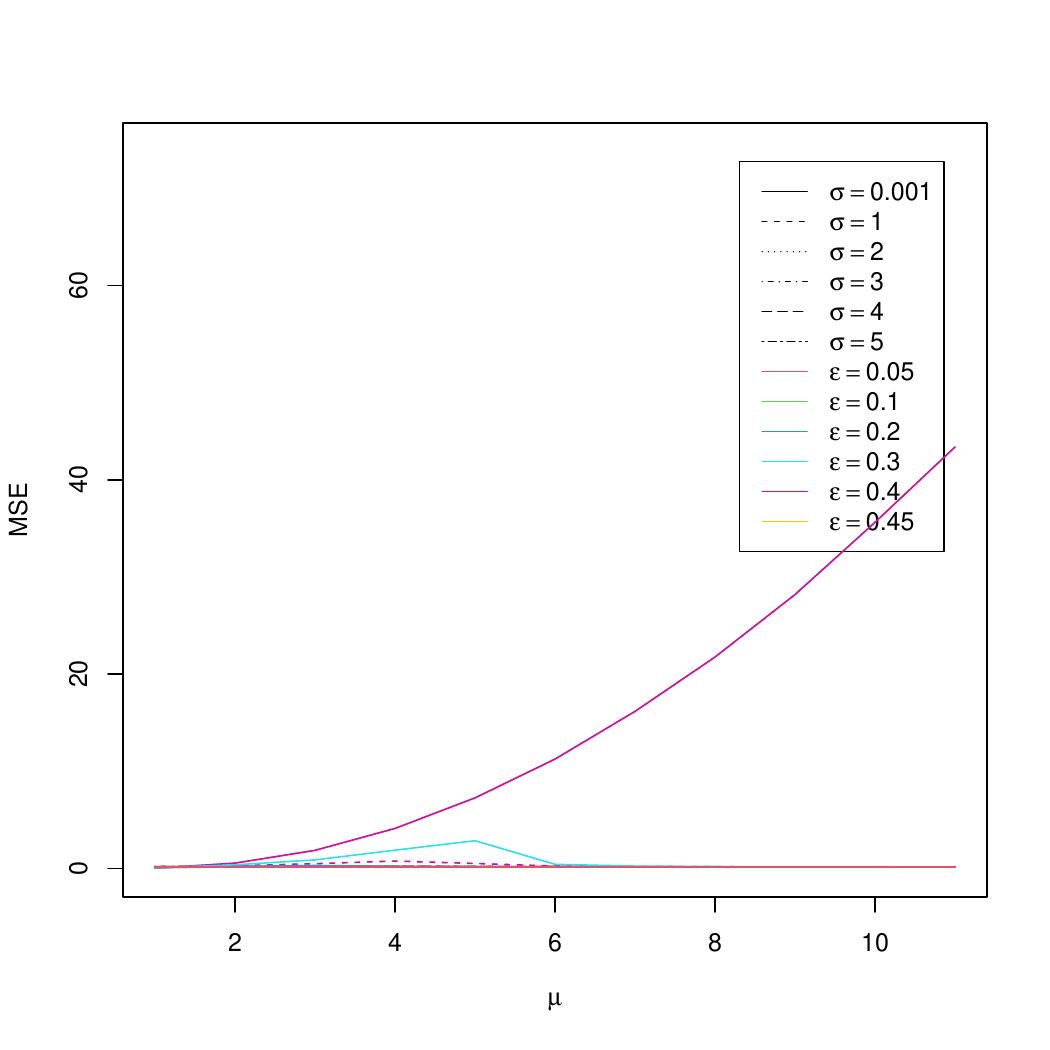}
\includegraphics[width=0.32\textwidth]{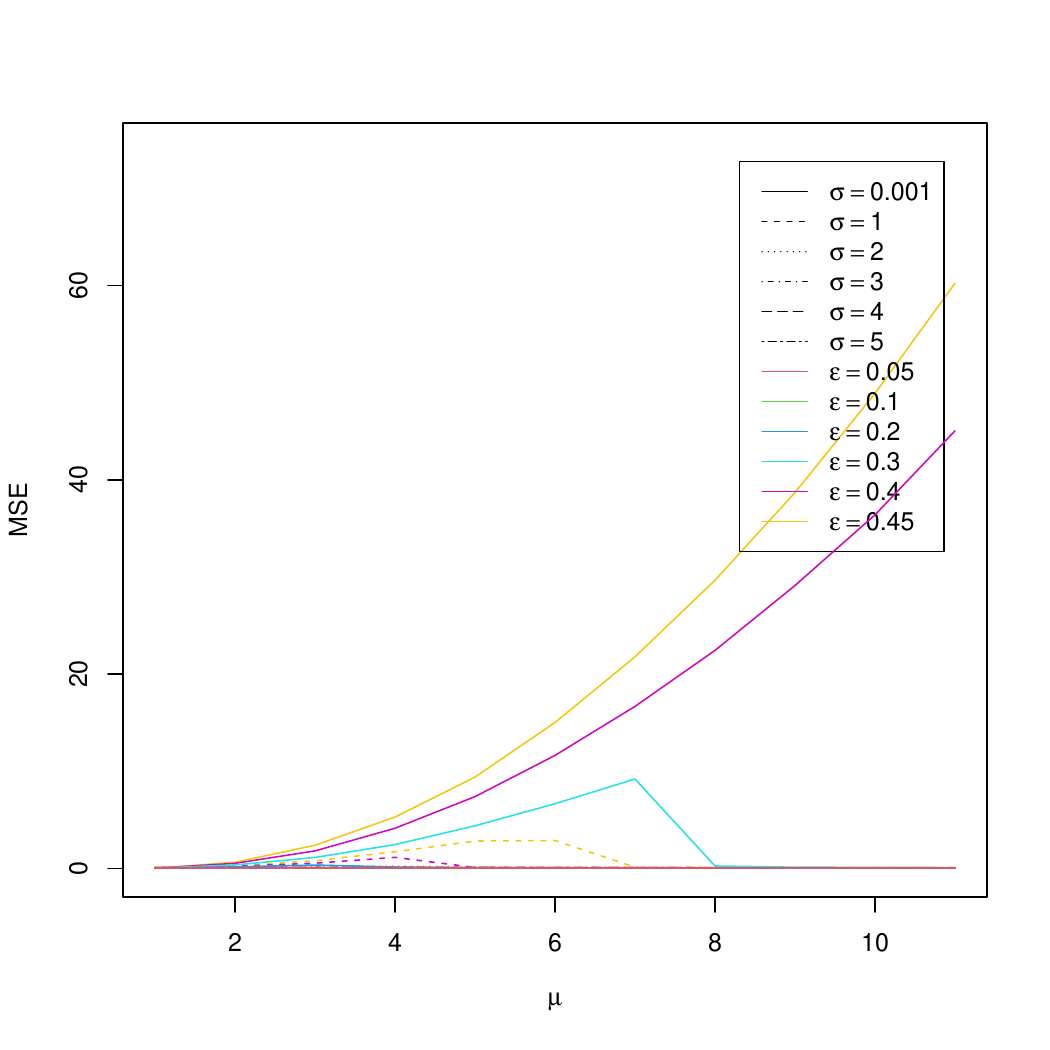} 
\includegraphics[width=0.32\textwidth]{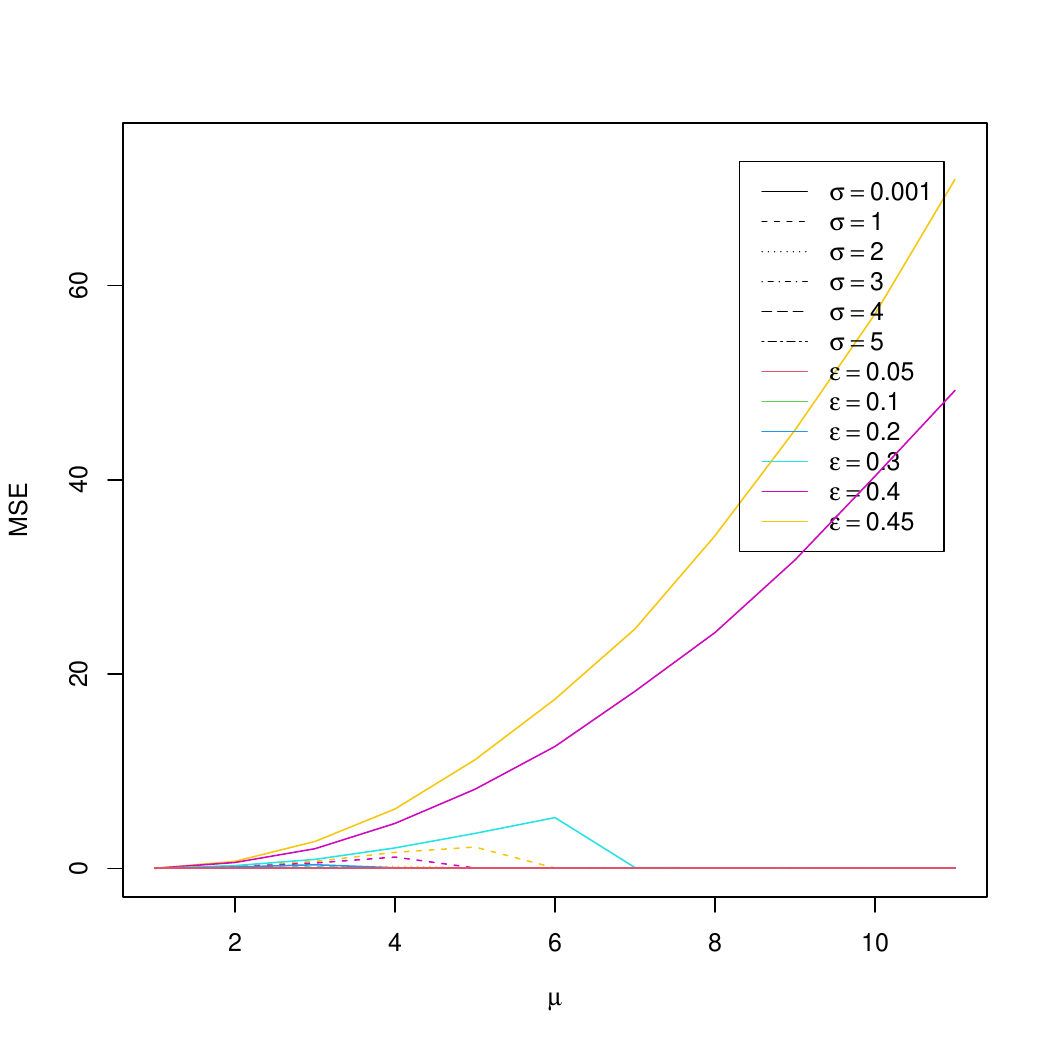} \\
\includegraphics[width=0.32\textwidth]{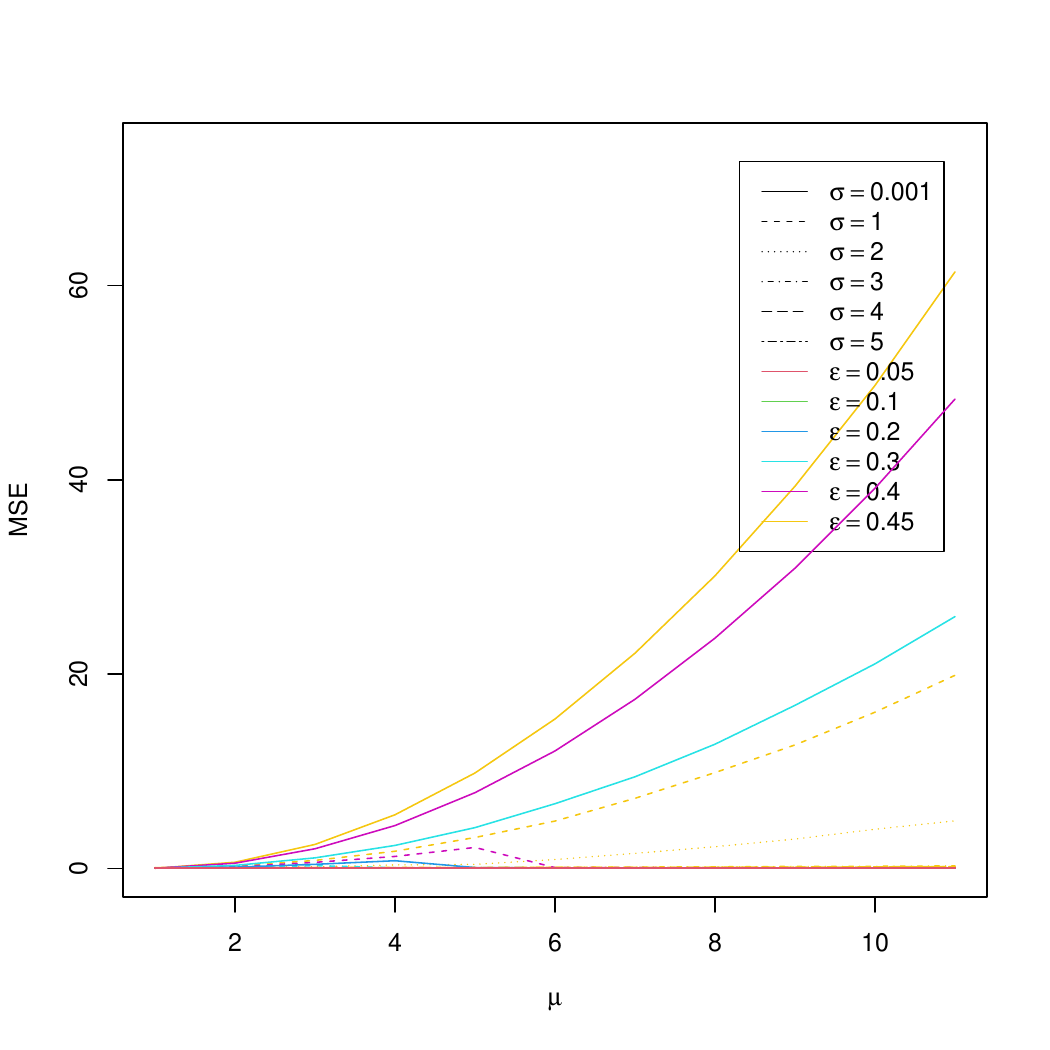}
\includegraphics[width=0.32\textwidth]{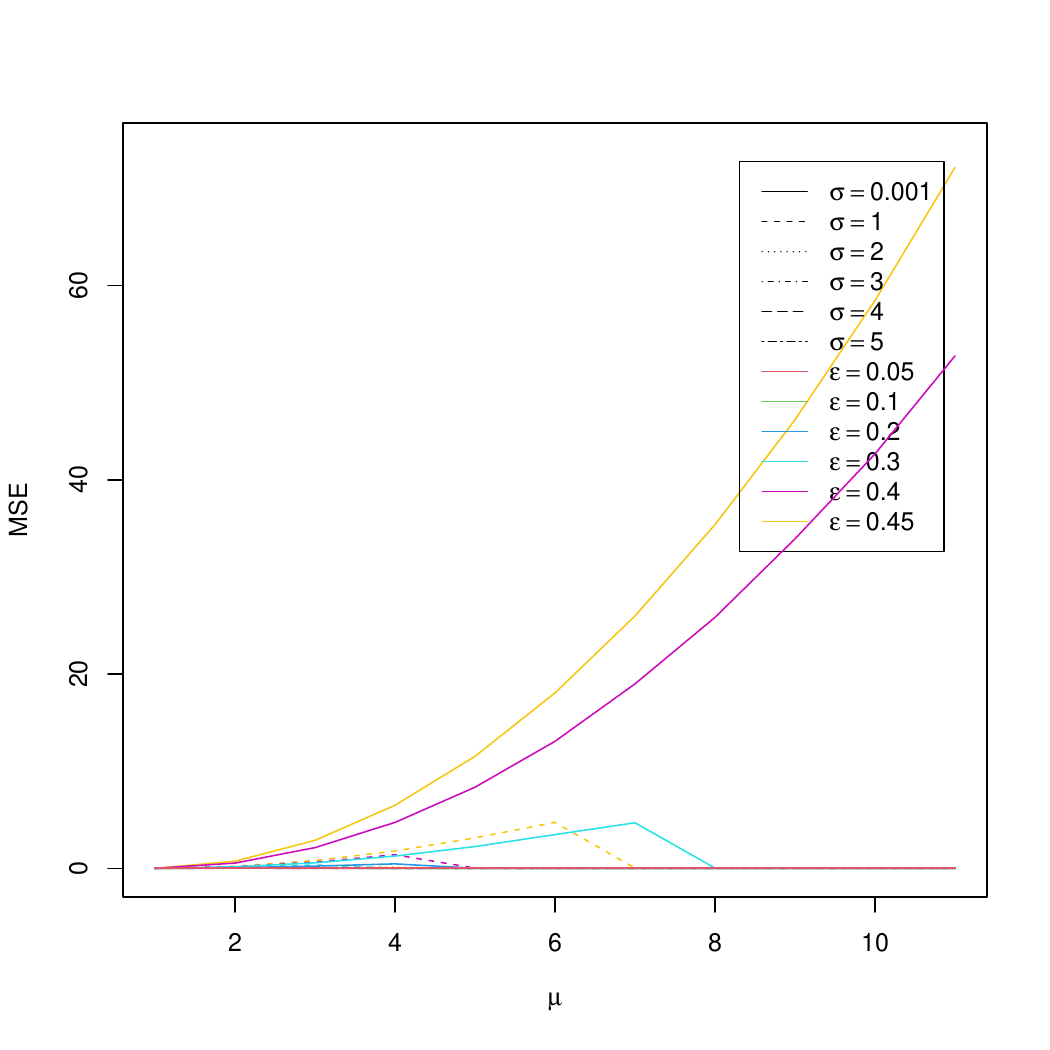} 
\includegraphics[width=0.32\textwidth]{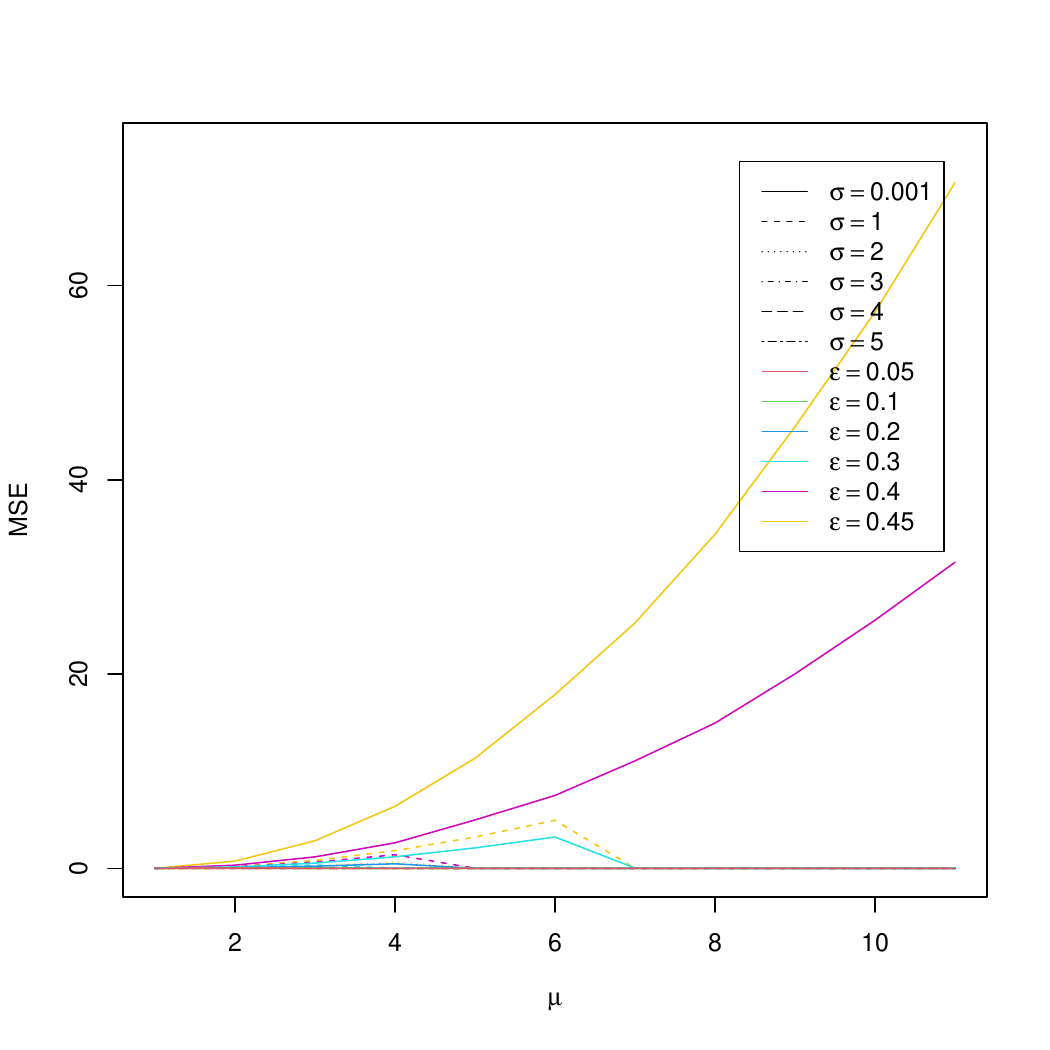} \\
\caption{Monte Carlo Simulation. Mean Square Error for the \texttt{CovMest} method as a function of the contamination average $\mu$ ($x$-axis), contamination scale $\sigma$ (different line styles) and contamination level $\varepsilon$ (colors). Rows: number of variables $p=1, 2, 5$ and columns: sample size factor $s=2, 5, 10$.}
\label{sup:fig:monte:MSE:M:1}
\end{figure}  

\begin{figure}
\centering
\includegraphics[width=0.32\textwidth]{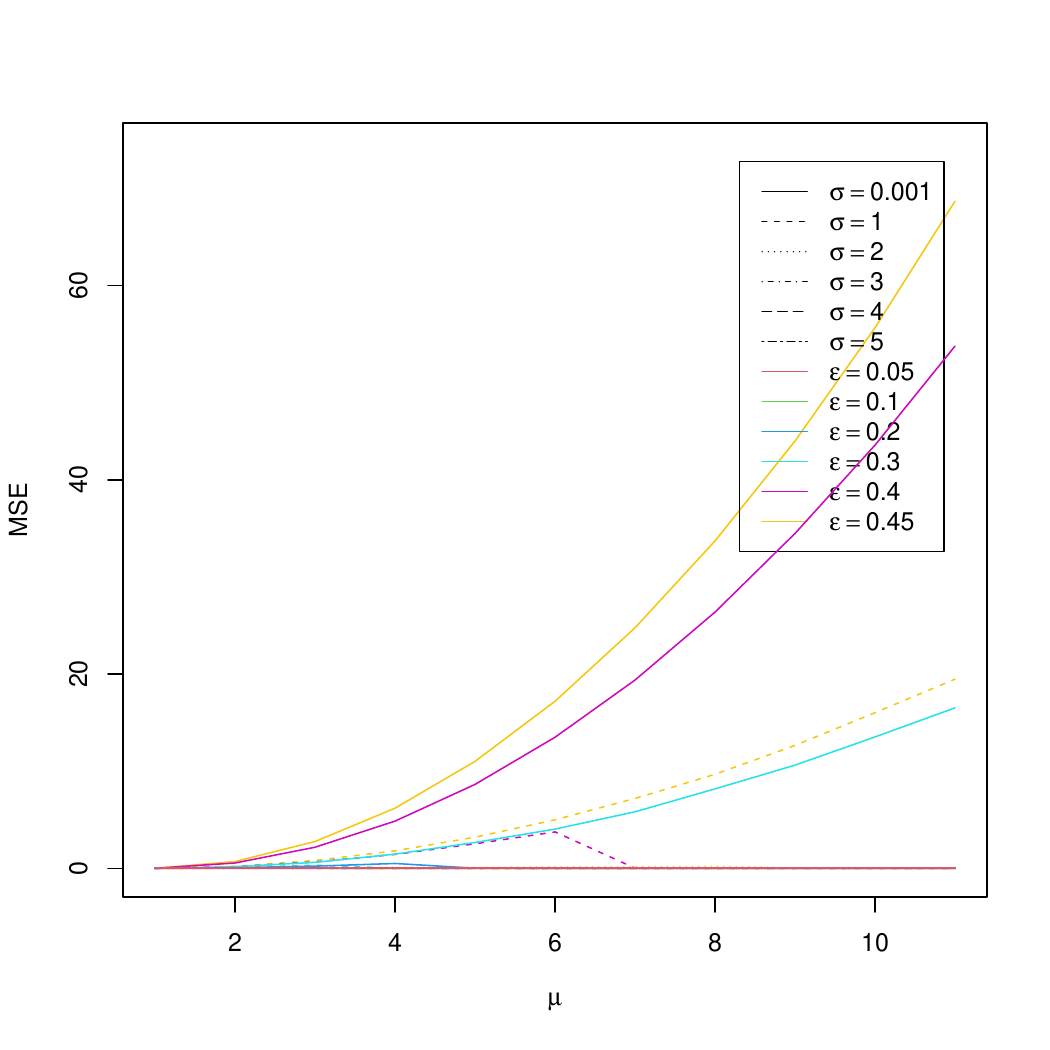}
\includegraphics[width=0.32\textwidth]{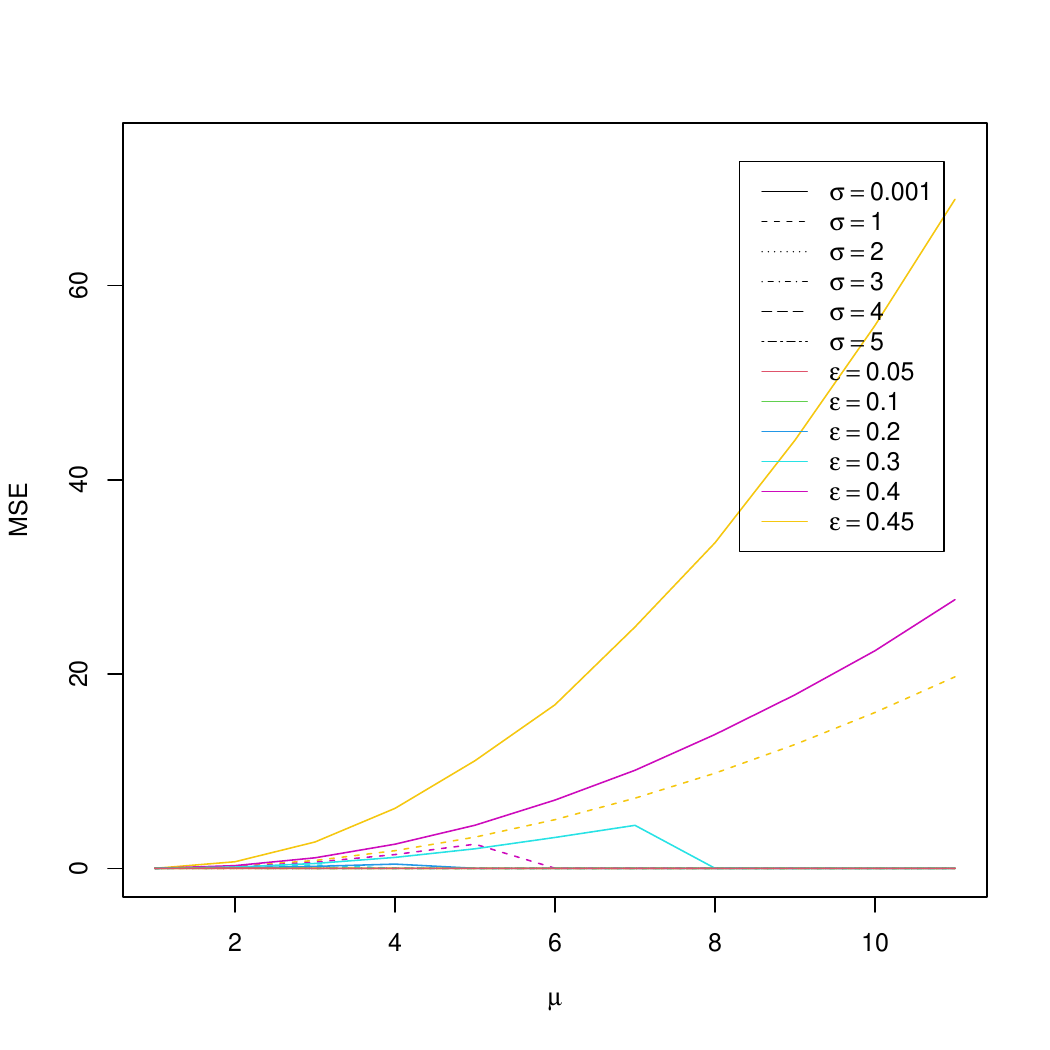} 
\includegraphics[width=0.32\textwidth]{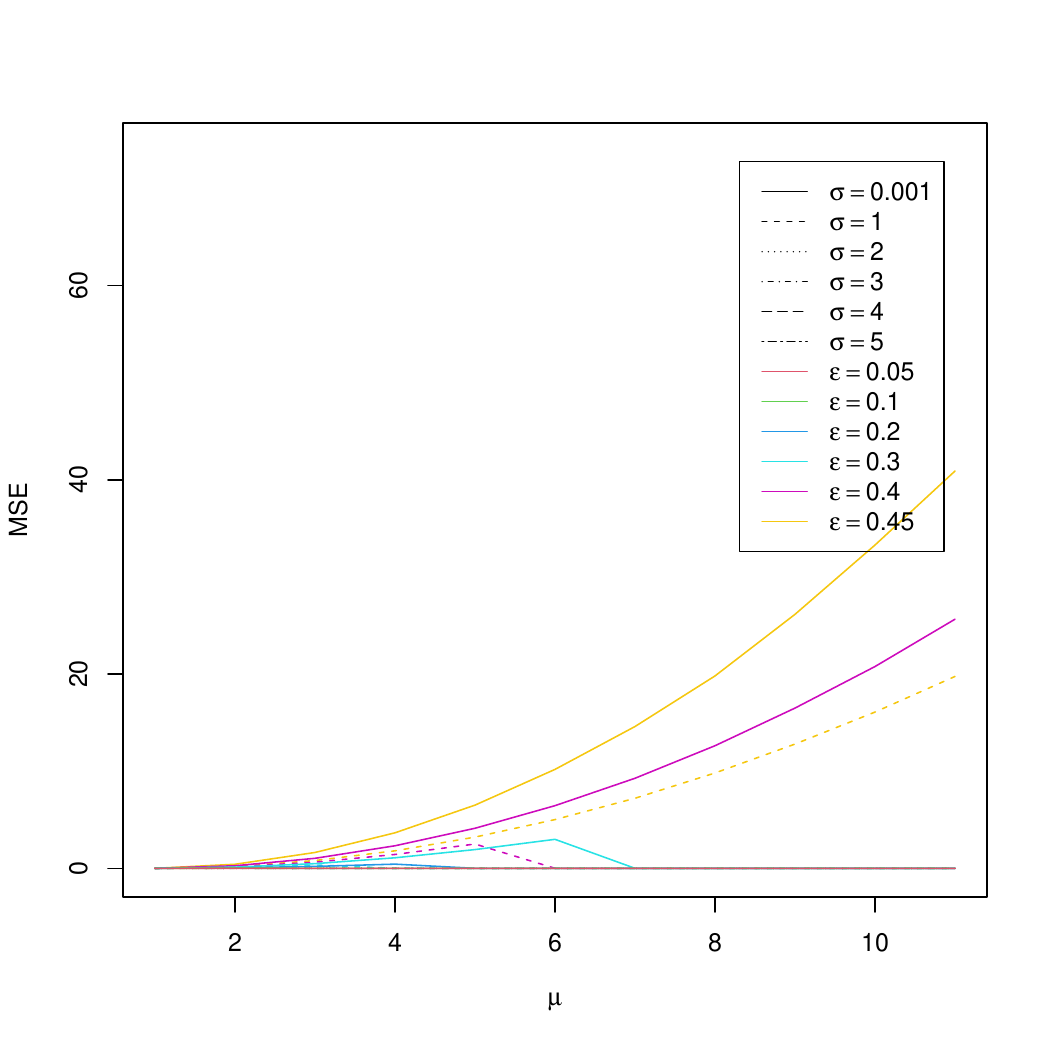} \\
\includegraphics[width=0.32\textwidth]{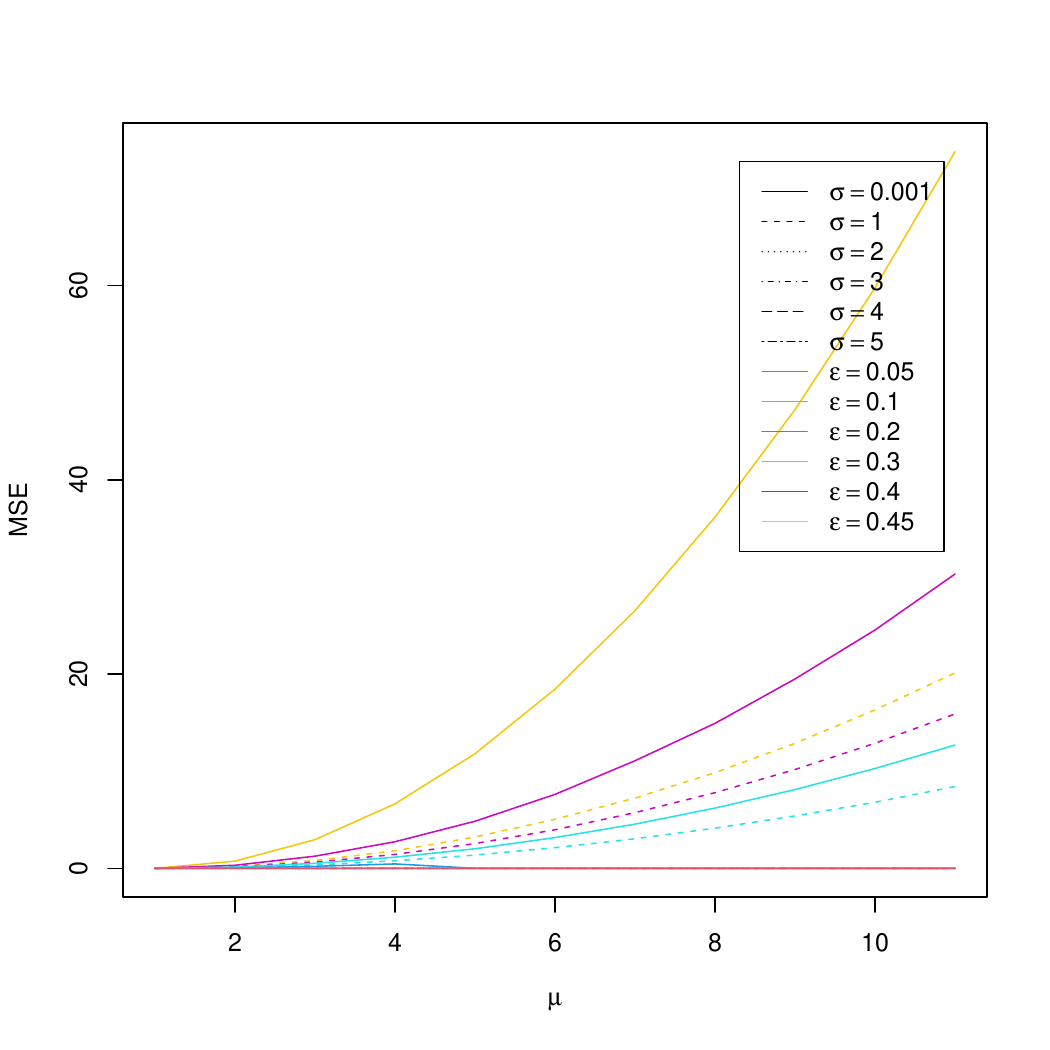}
\includegraphics[width=0.32\textwidth]{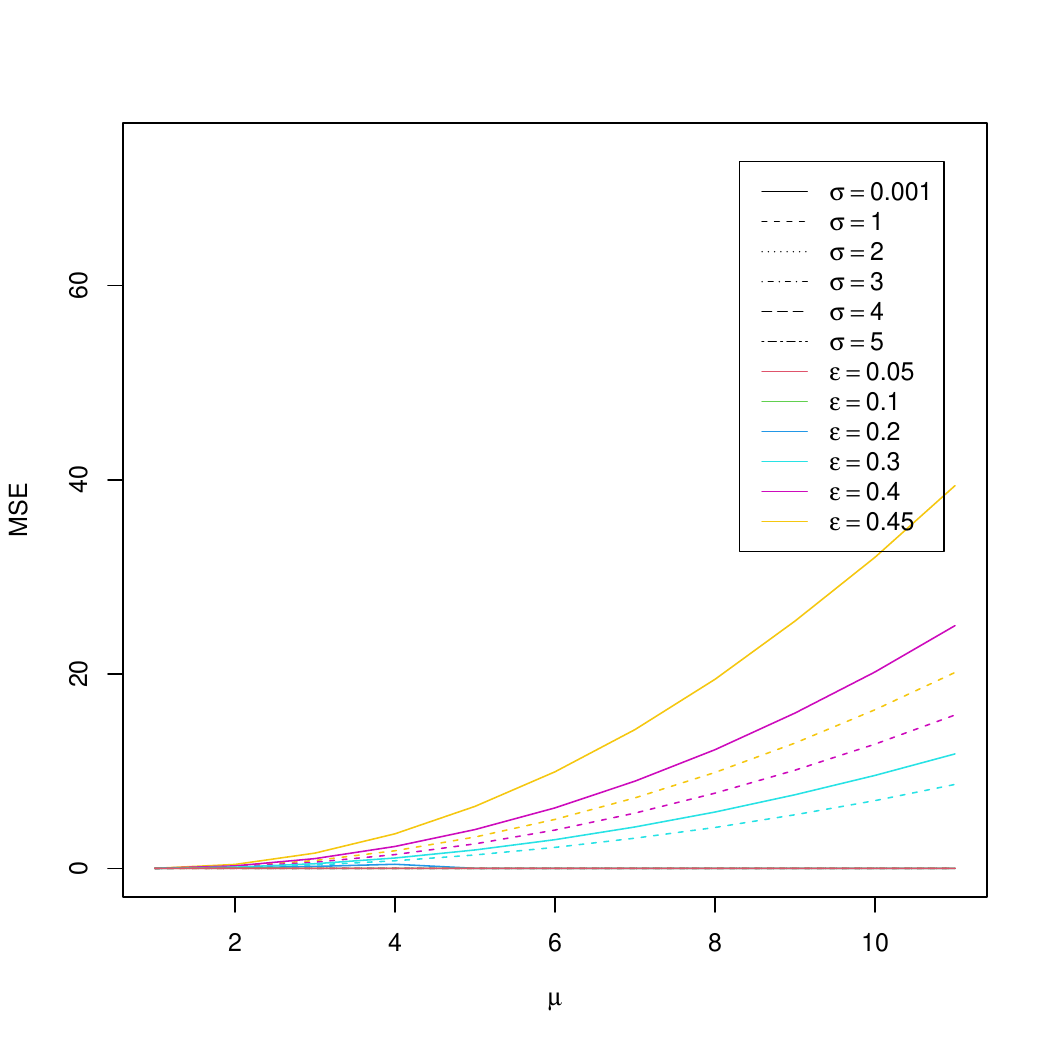} 
\includegraphics[width=0.32\textwidth]{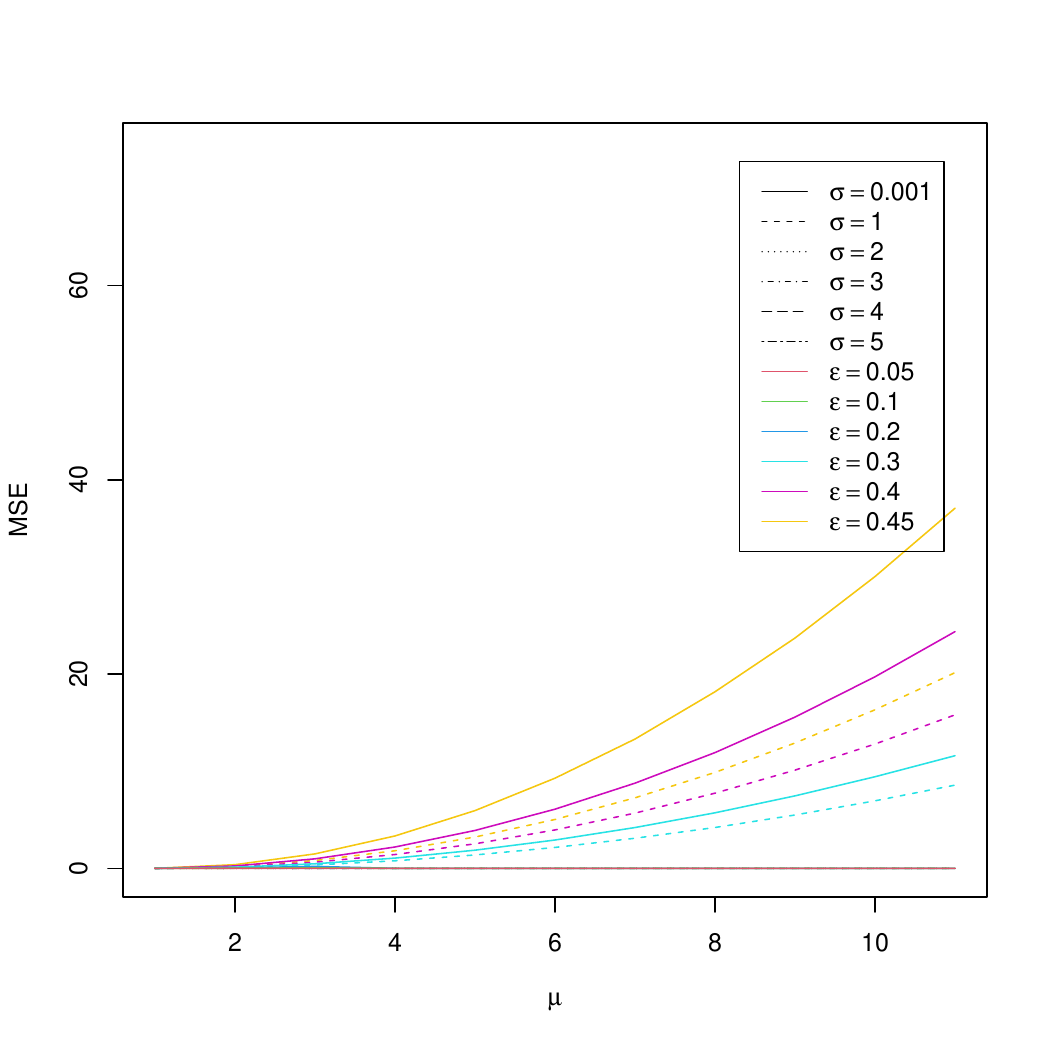}
\caption{Monte Carlo Simulation. Mean Square Error for the \texttt{CovMest} method as a function of the contamination average $\mu$ ($x$-axis), contamination scale $\sigma$ (different line styles) and contamination level $\varepsilon$ (colors). Rows: number of variables $p=10, 20$ and columns: sample size factor $s=2, 5, 10$.}
\label{sup:fig:monte:MSE:M:2}
\end{figure}  

\clearpage


\begin{figure}
\centering
\includegraphics[width=0.32\textwidth]{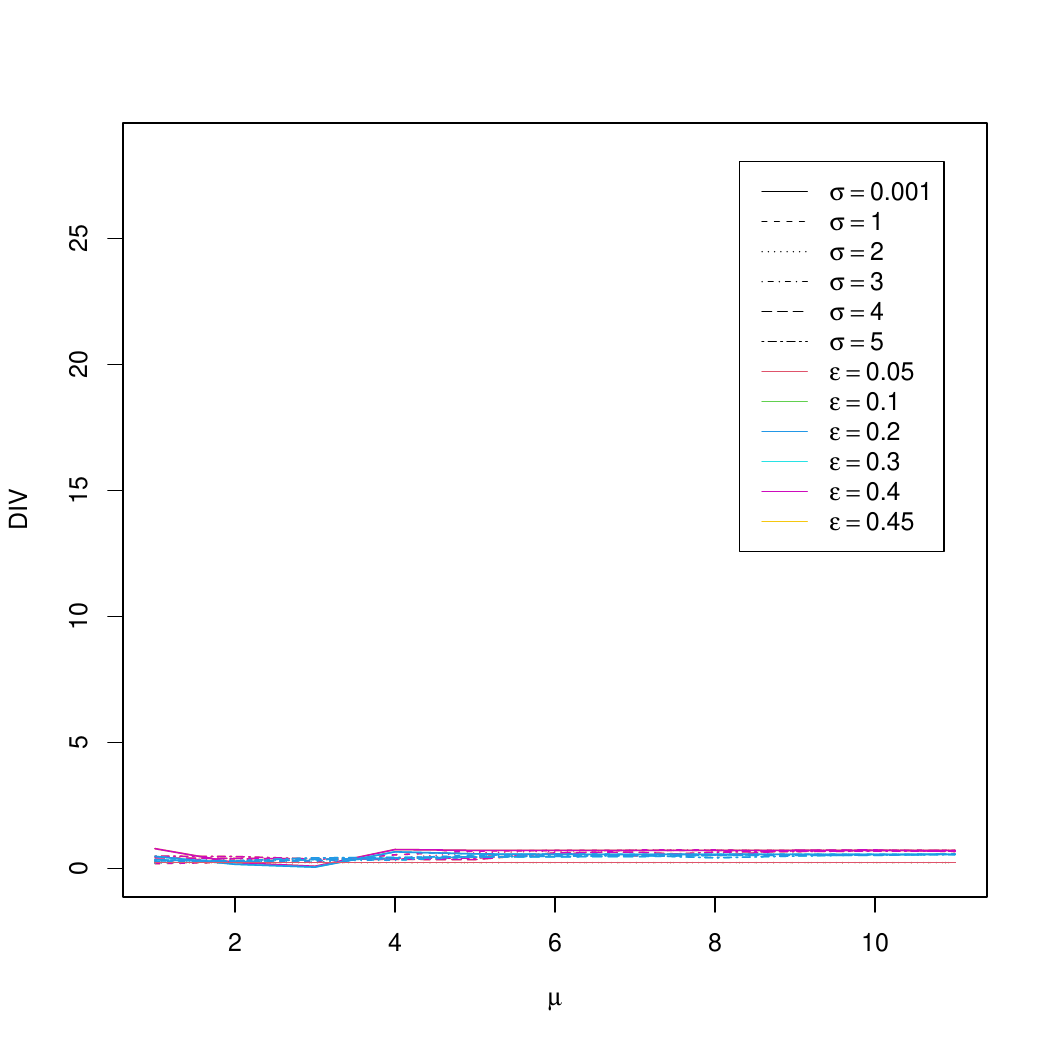}
\includegraphics[width=0.32\textwidth]{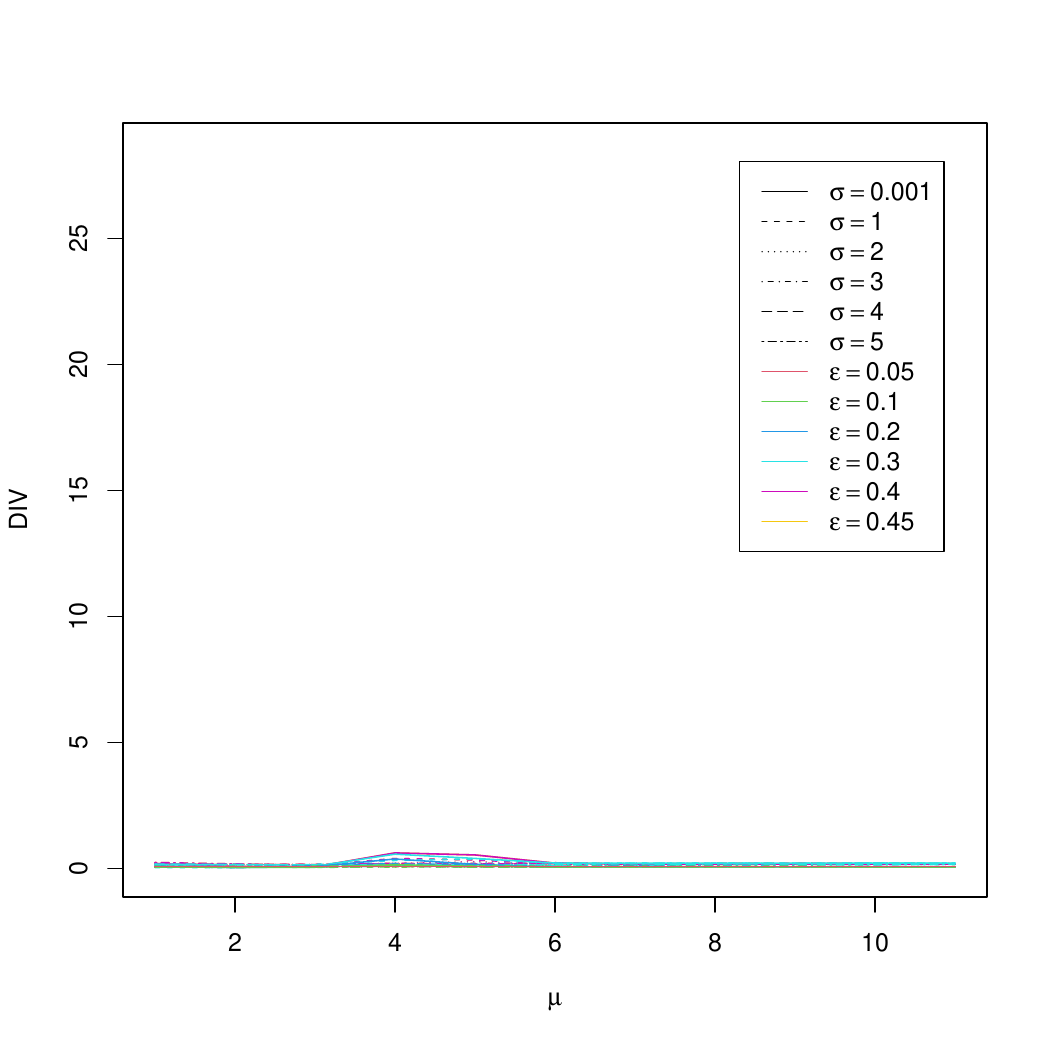} 
\includegraphics[width=0.32\textwidth]{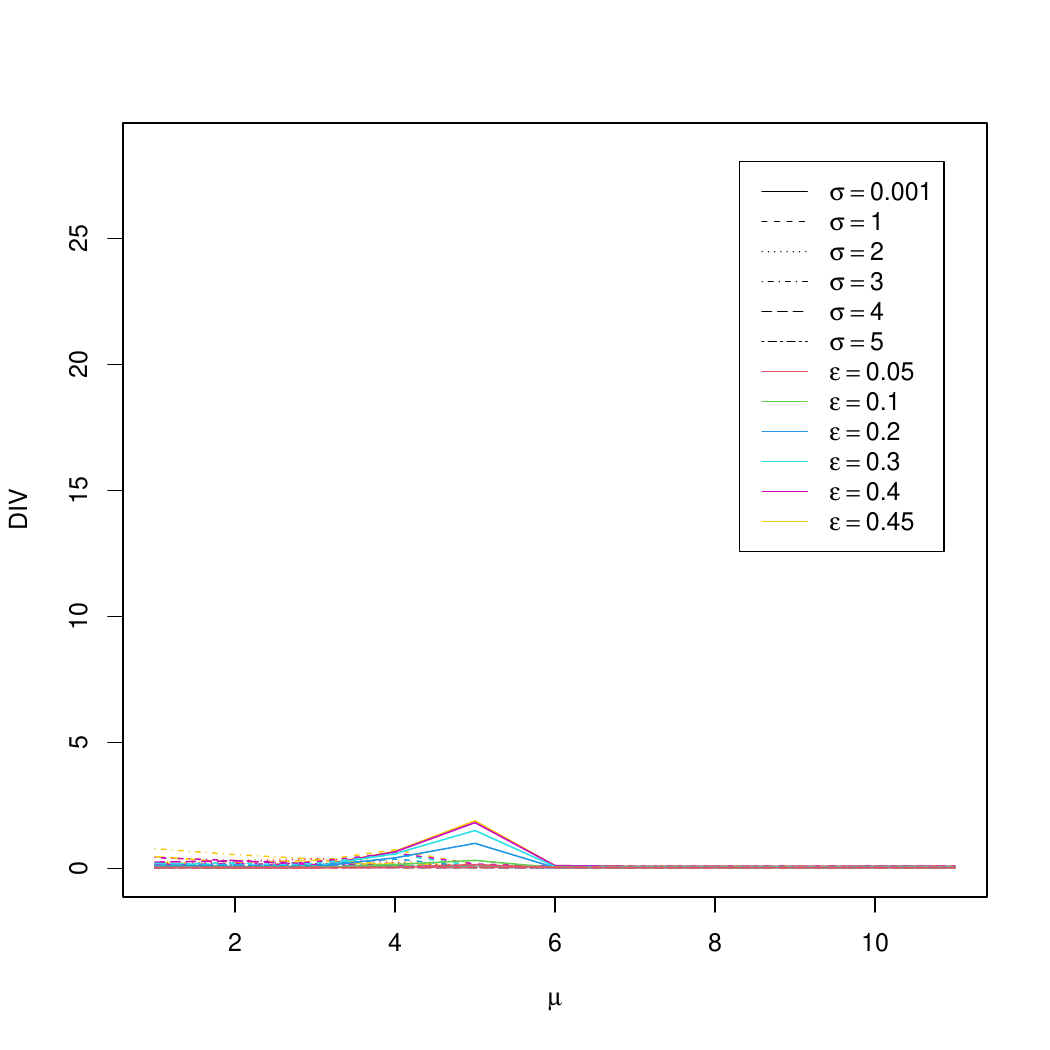} \\
\includegraphics[width=0.32\textwidth]{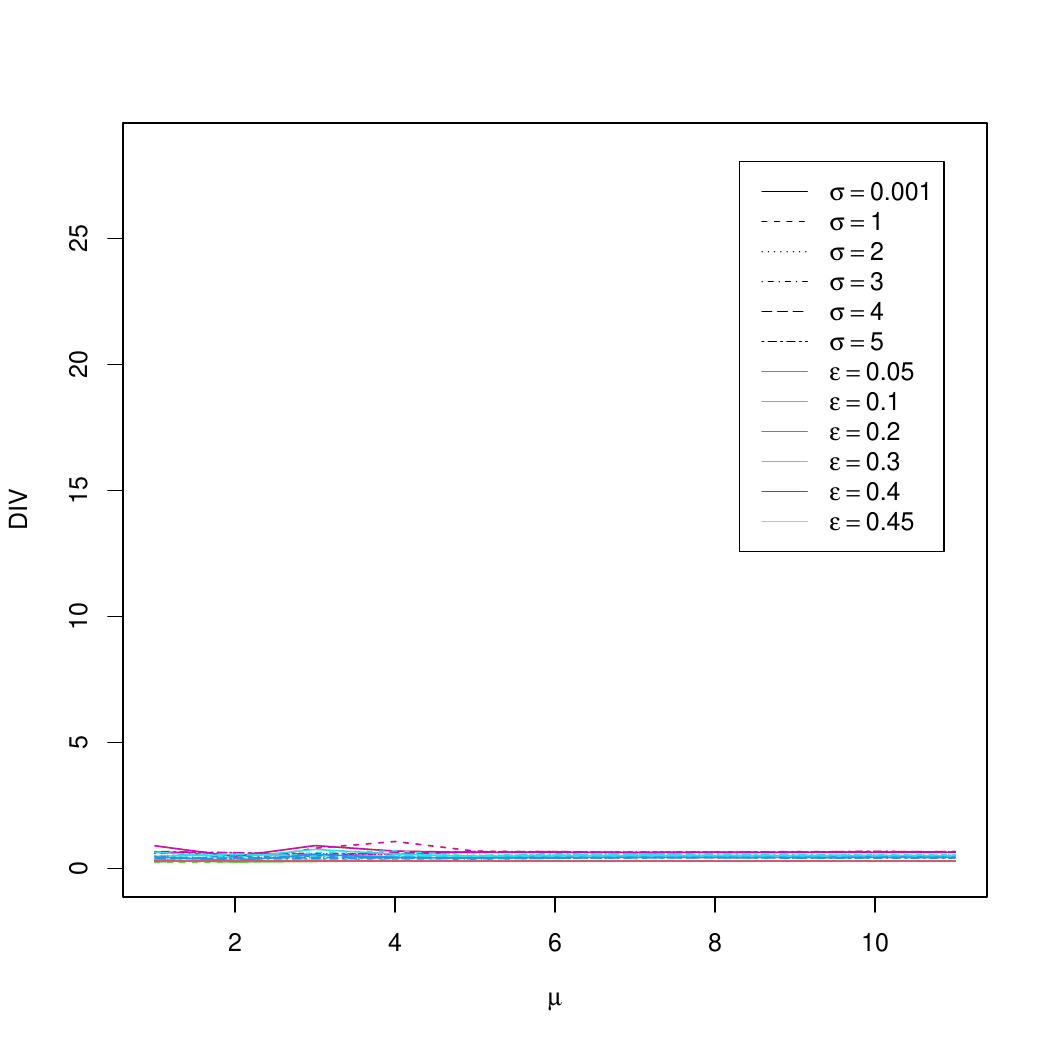}
\includegraphics[width=0.32\textwidth]{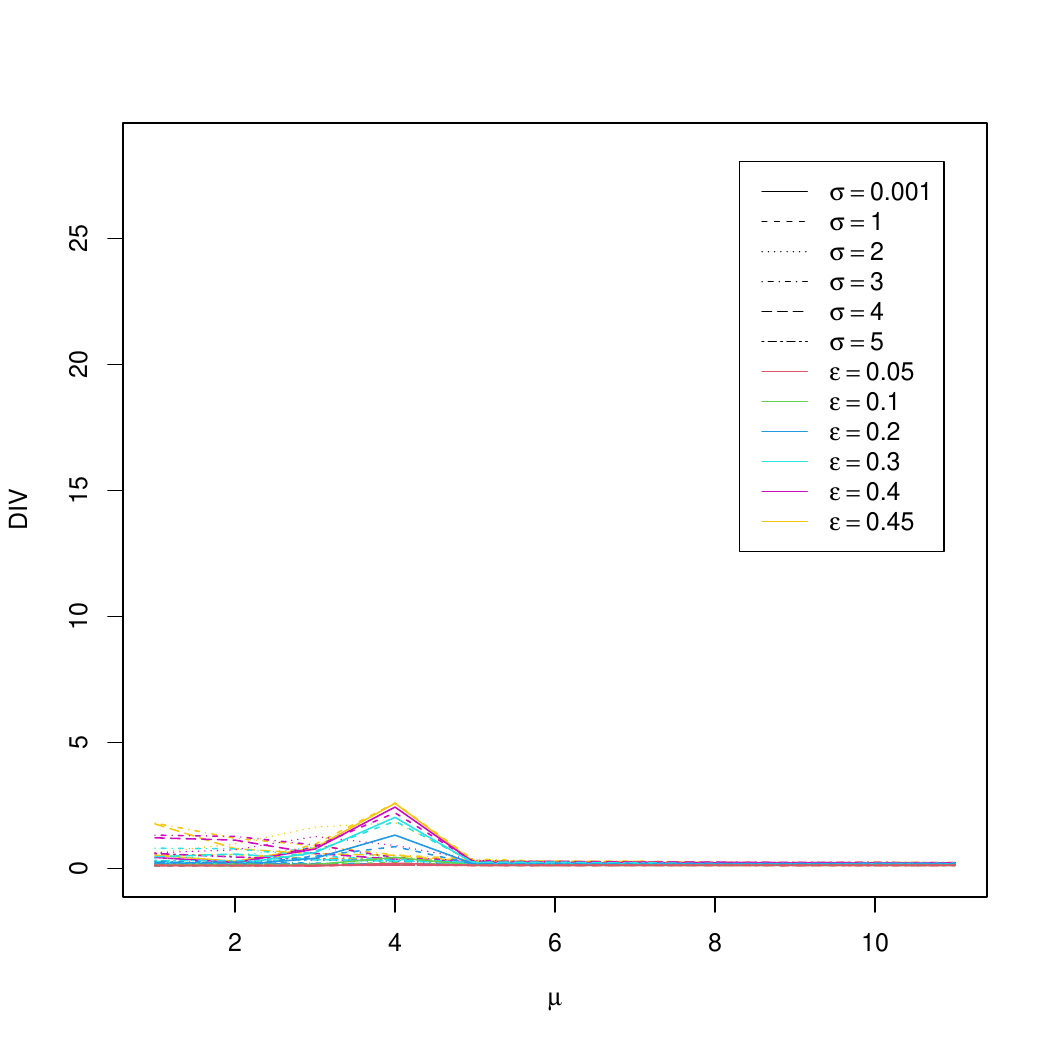} 
\includegraphics[width=0.32\textwidth]{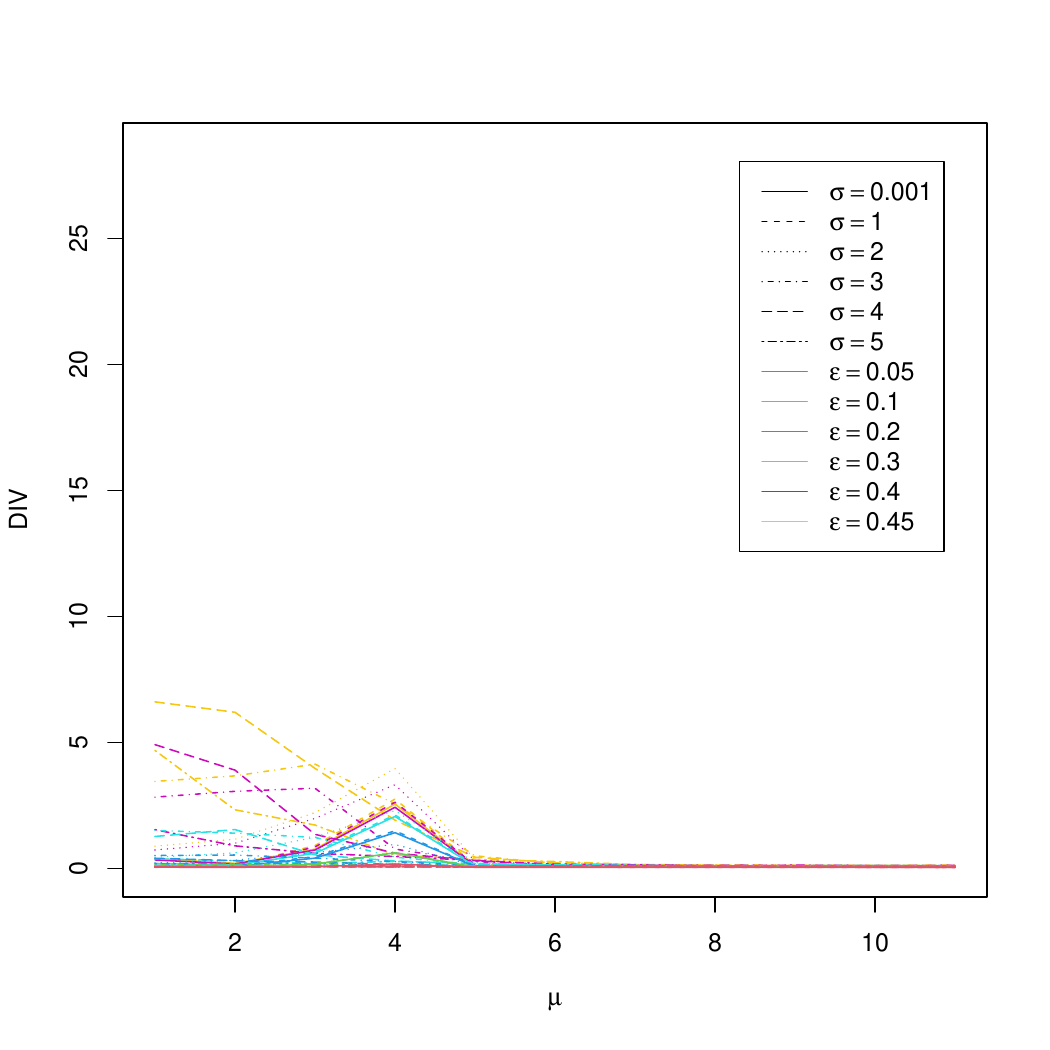} \\
\includegraphics[width=0.32\textwidth]{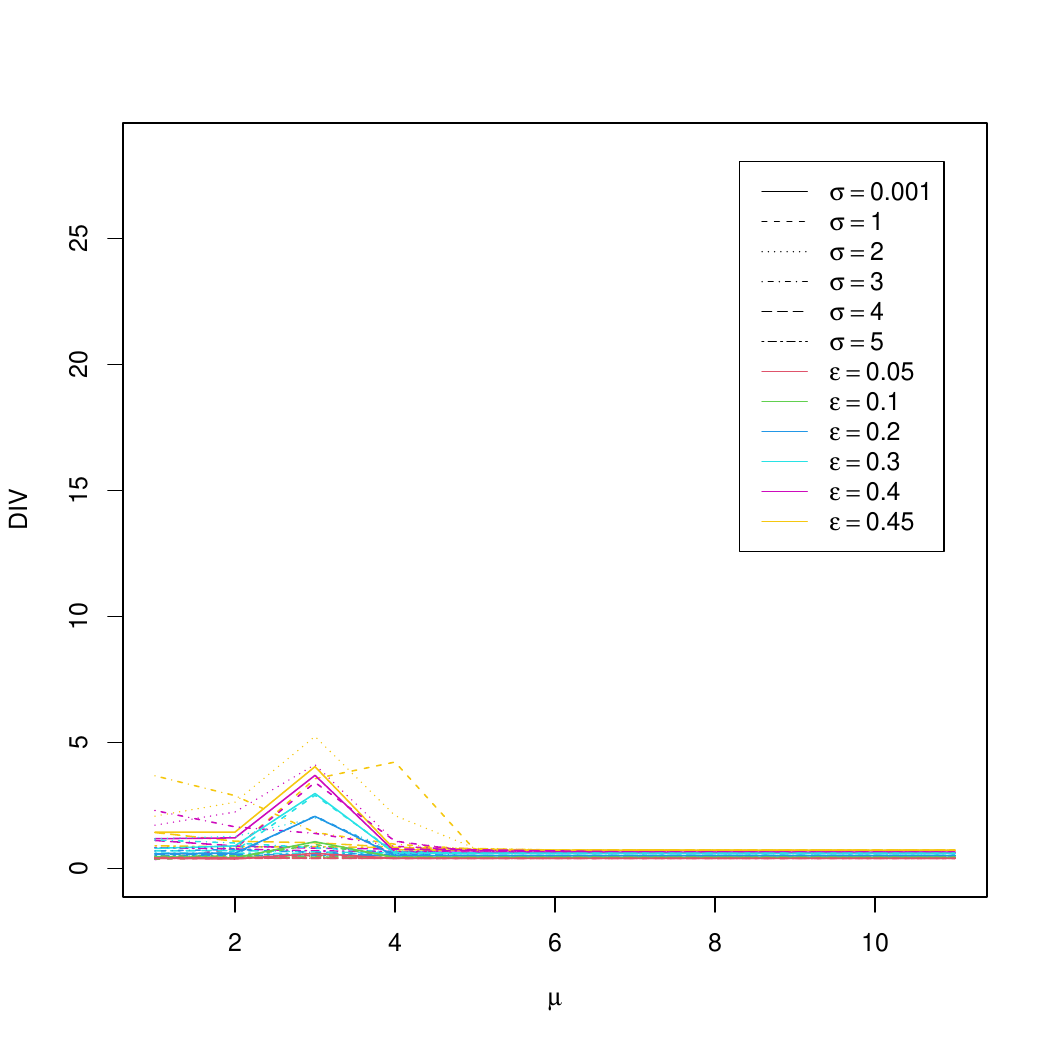}
\includegraphics[width=0.32\textwidth]{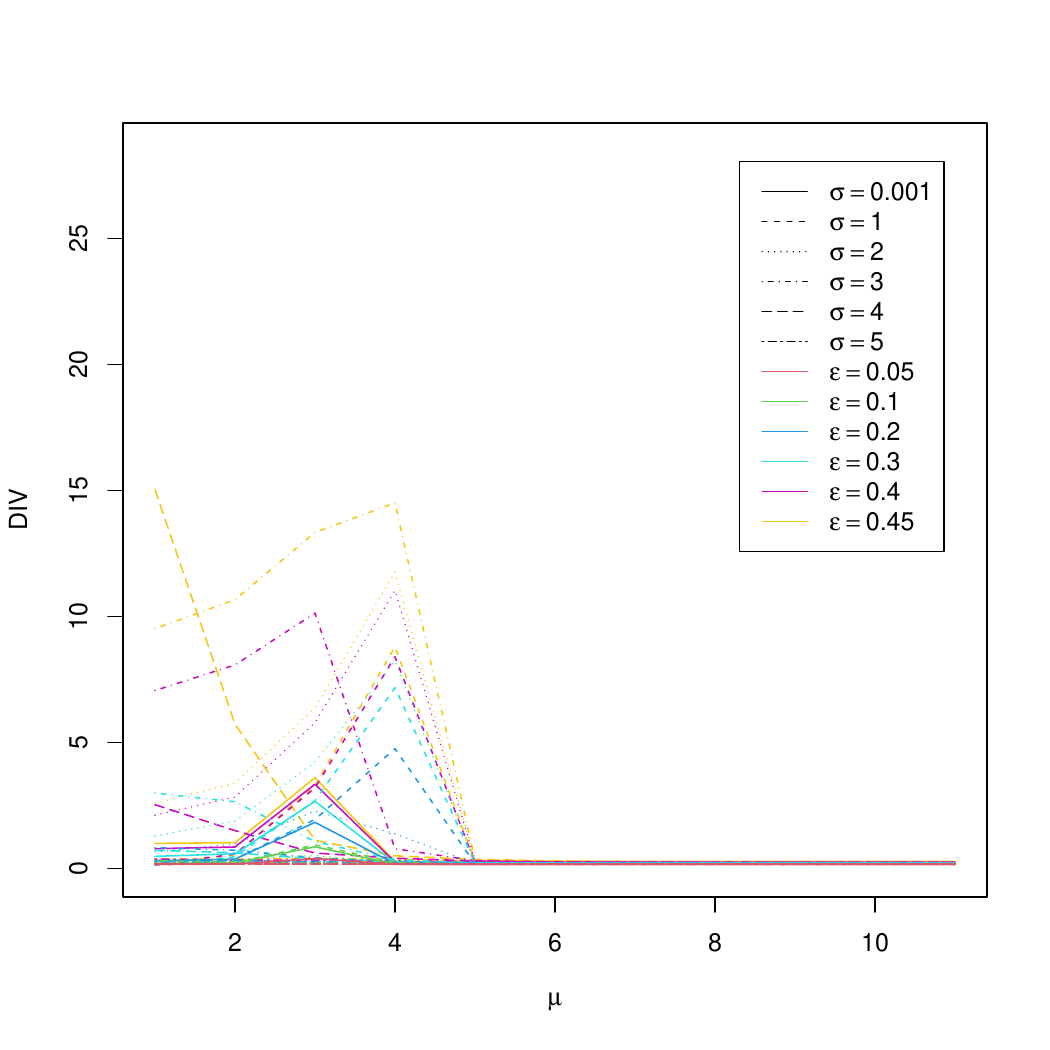} 
\includegraphics[width=0.32\textwidth]{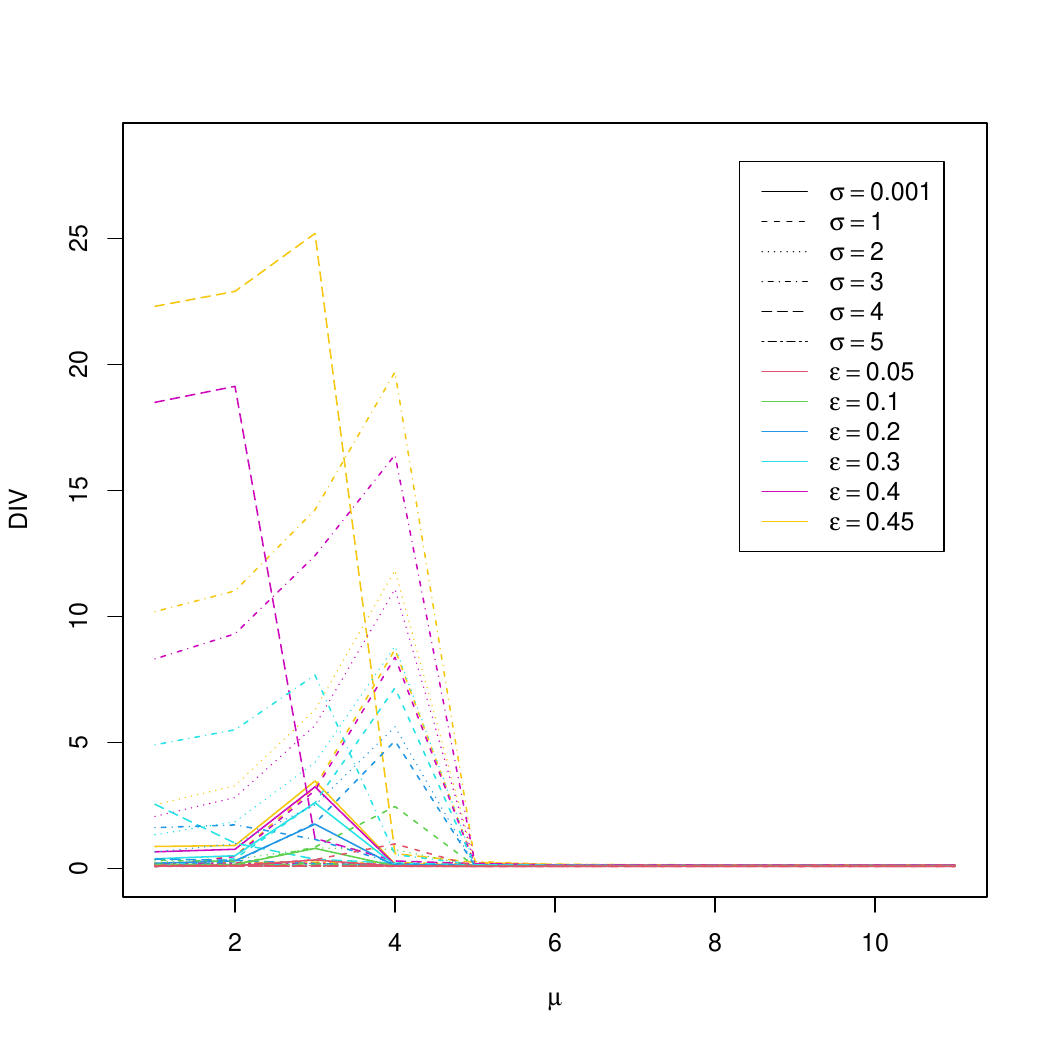} \\
\caption{Monte Carlo Simulation. Kullback--Leibler Divergence for the proposed method starting at the true values with $\alpha=0.25$ as a function of the contamination average $\mu$ ($x$-axis), contamination scale $\sigma$ (different line styles) and contamination level $\varepsilon$ (colors). Rows: number of variables $p=1, 2, 5$ and columns: sample size factor $s=2, 5, 10$.}
\label{sup:fig:monte:DIV:0.25:1}
\end{figure}  

\begin{figure}
\centering
\includegraphics[width=0.32\textwidth]{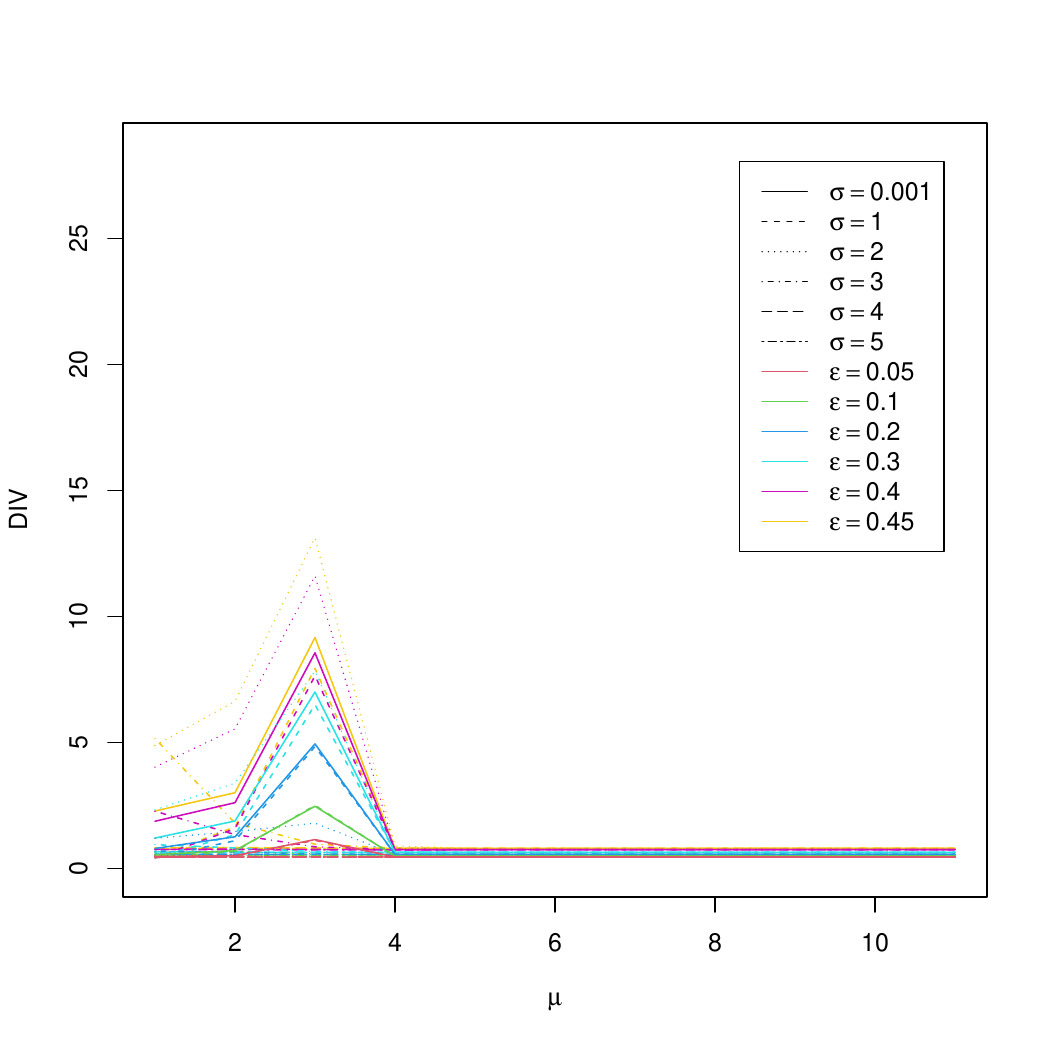}
\includegraphics[width=0.32\textwidth]{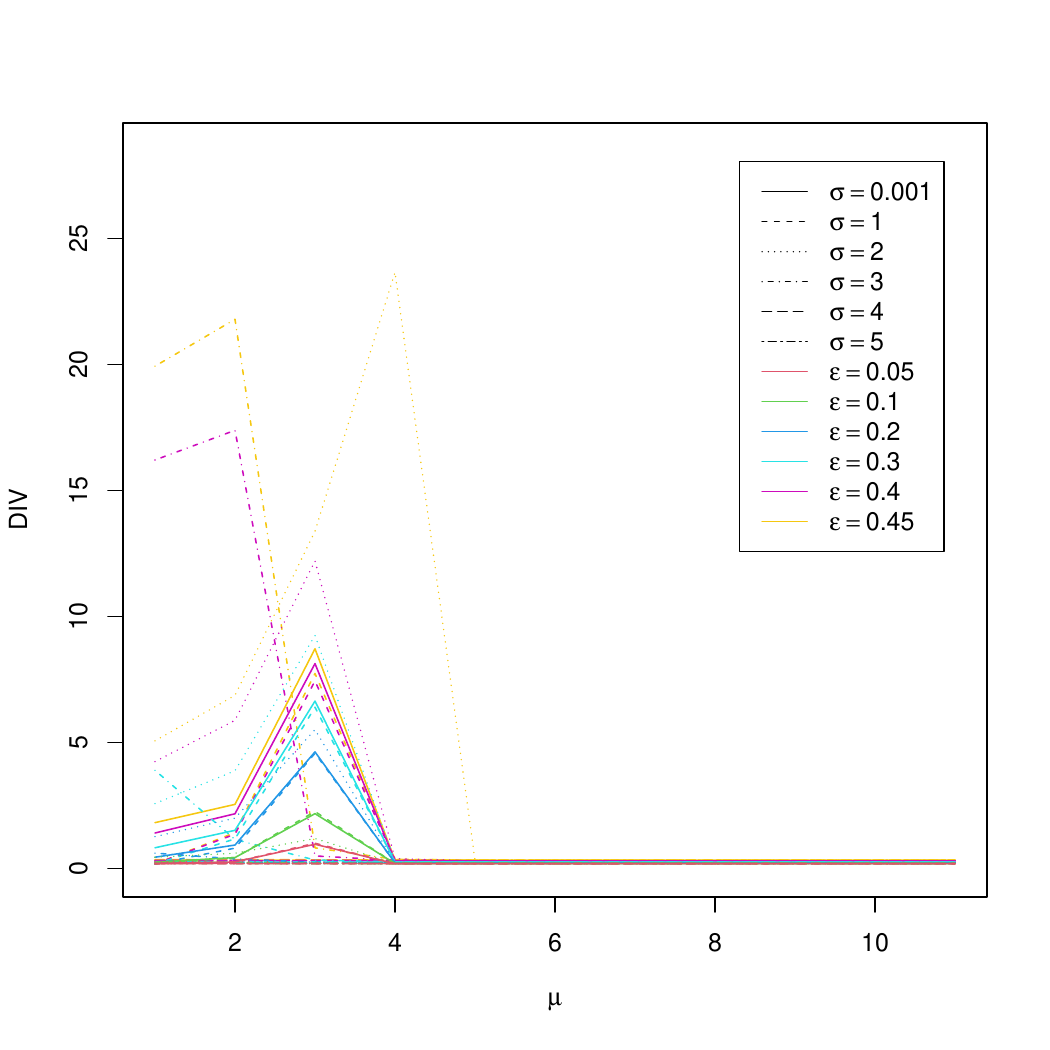} 
\includegraphics[width=0.32\textwidth]{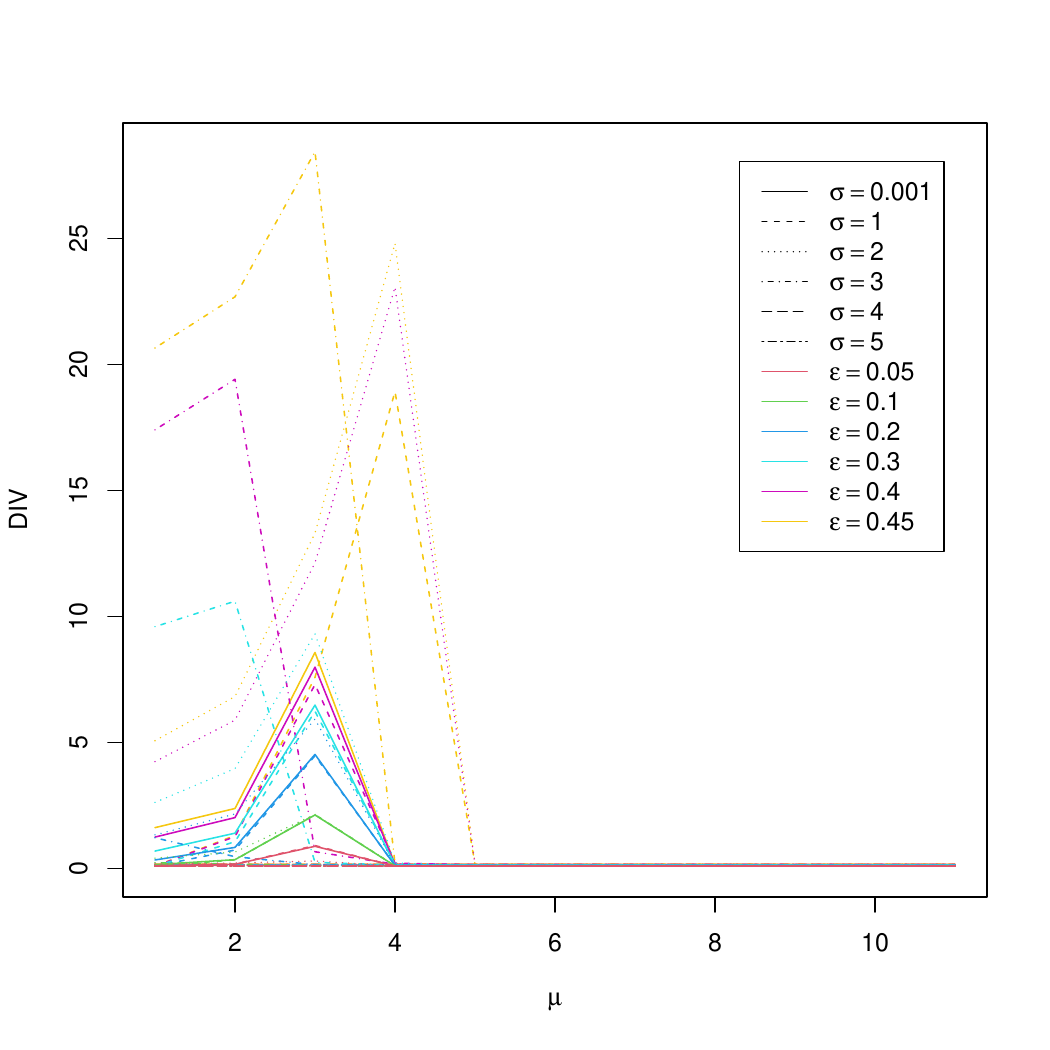} \\
\includegraphics[width=0.32\textwidth]{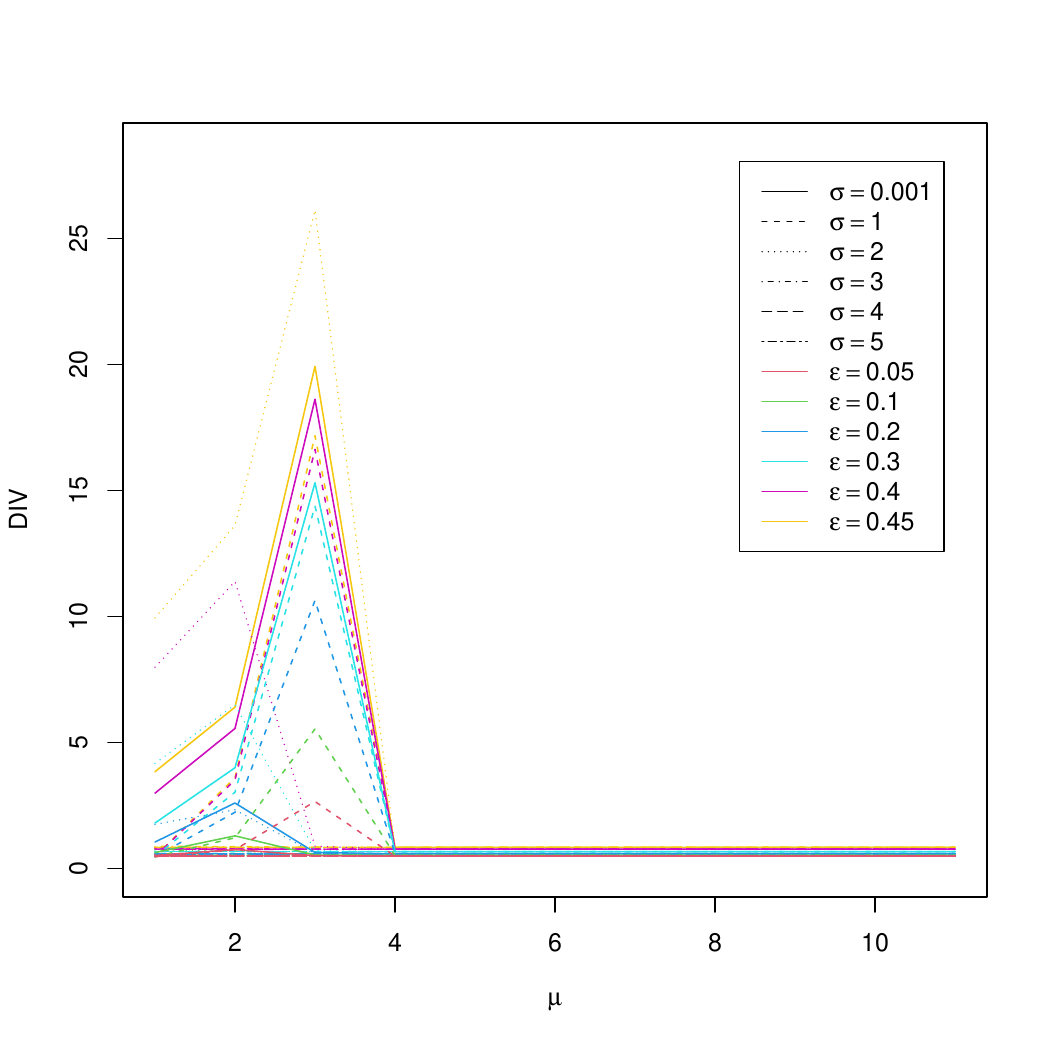}
\includegraphics[width=0.32\textwidth]{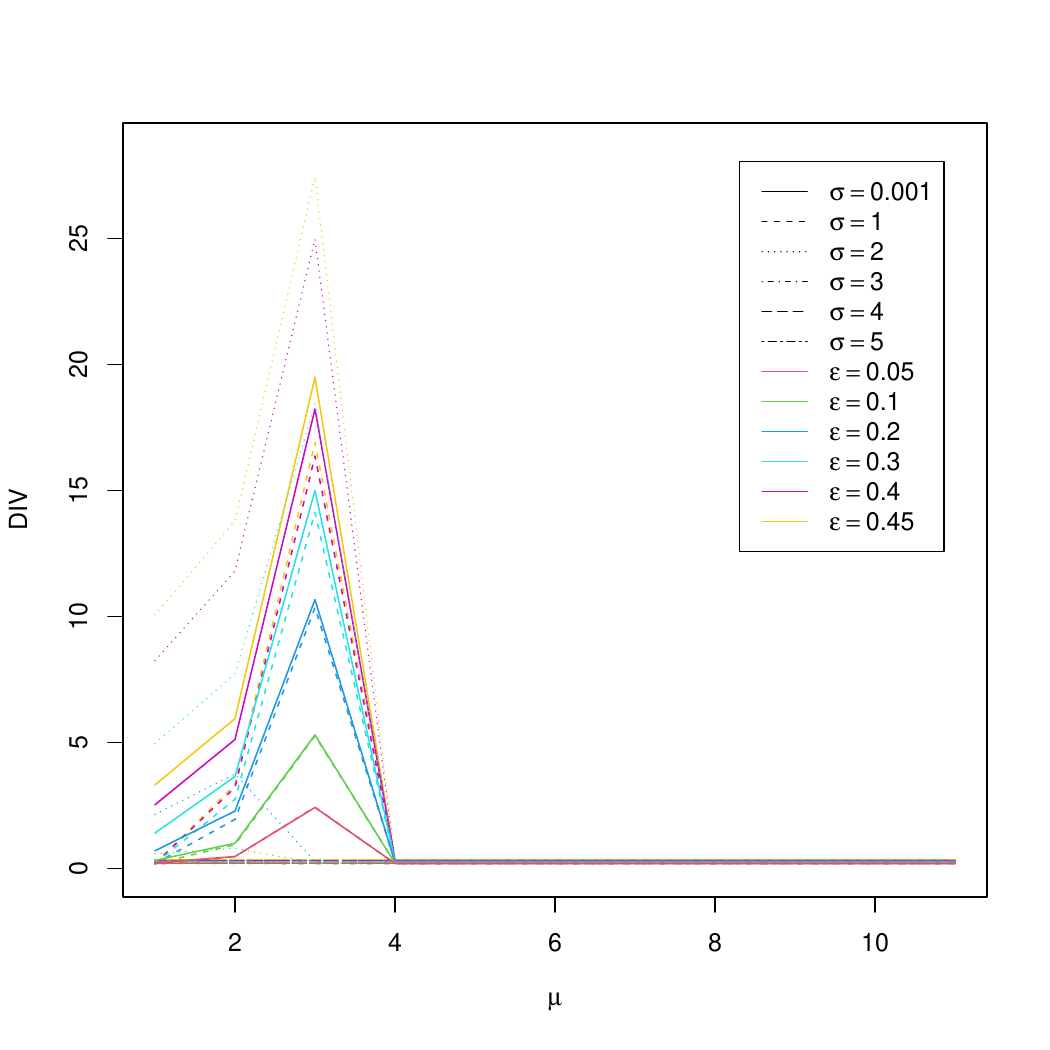} 
\includegraphics[width=0.32\textwidth]{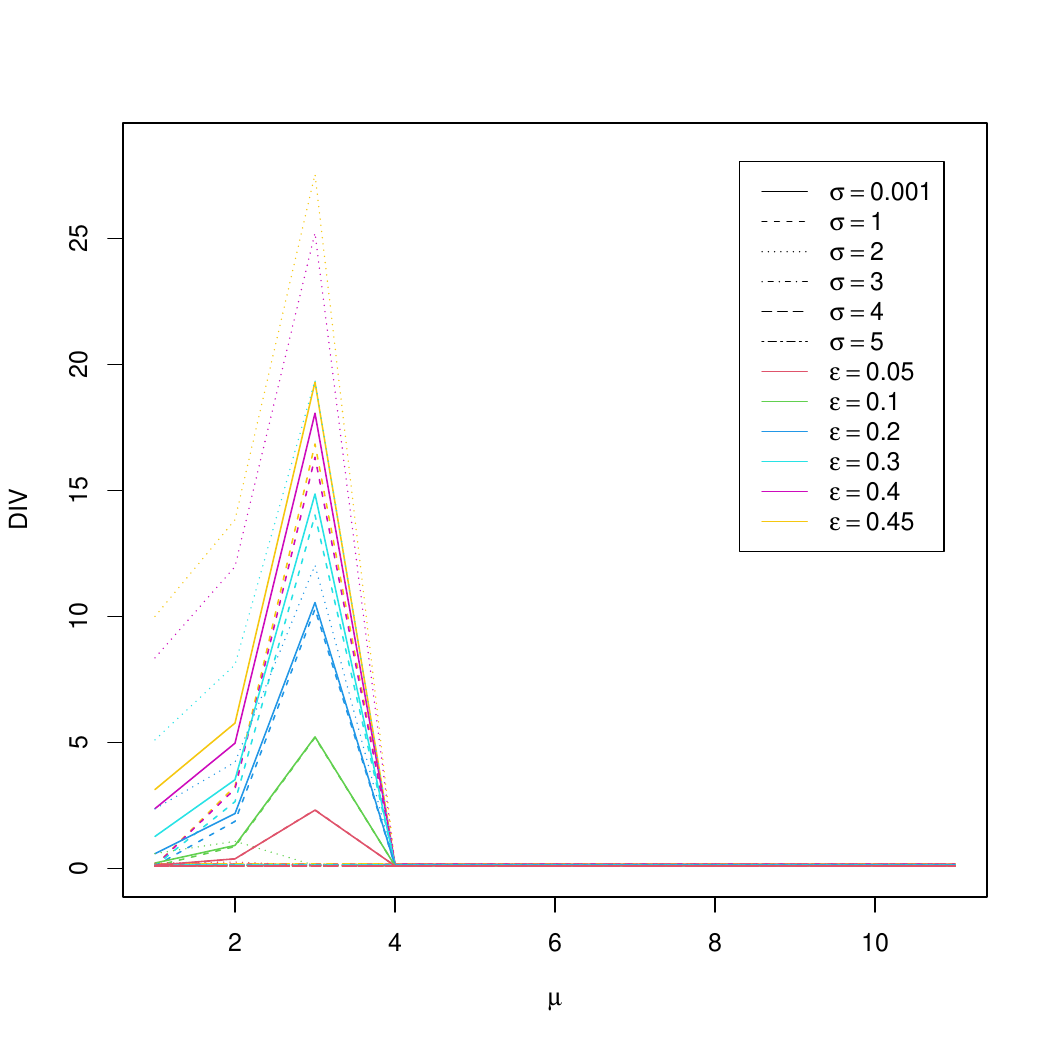}
\caption{Monte Carlo Simulation. Kullback--Leibler Divergence for the proposed method starting at the true values with $\alpha=0.25$ as a function of the contamination average $\mu$ ($x$-axis), contamination scale $\sigma$ (different line styles) and contamination level $\varepsilon$ (colors). Rows: number of variables $p=10, 20$ and columns: sample size factor $s=2, 5, 10$.}
\label{sup:fig:monte:DIV:0.25:2}
\end{figure}  

\begin{figure}
\centering
\includegraphics[width=0.32\textwidth]{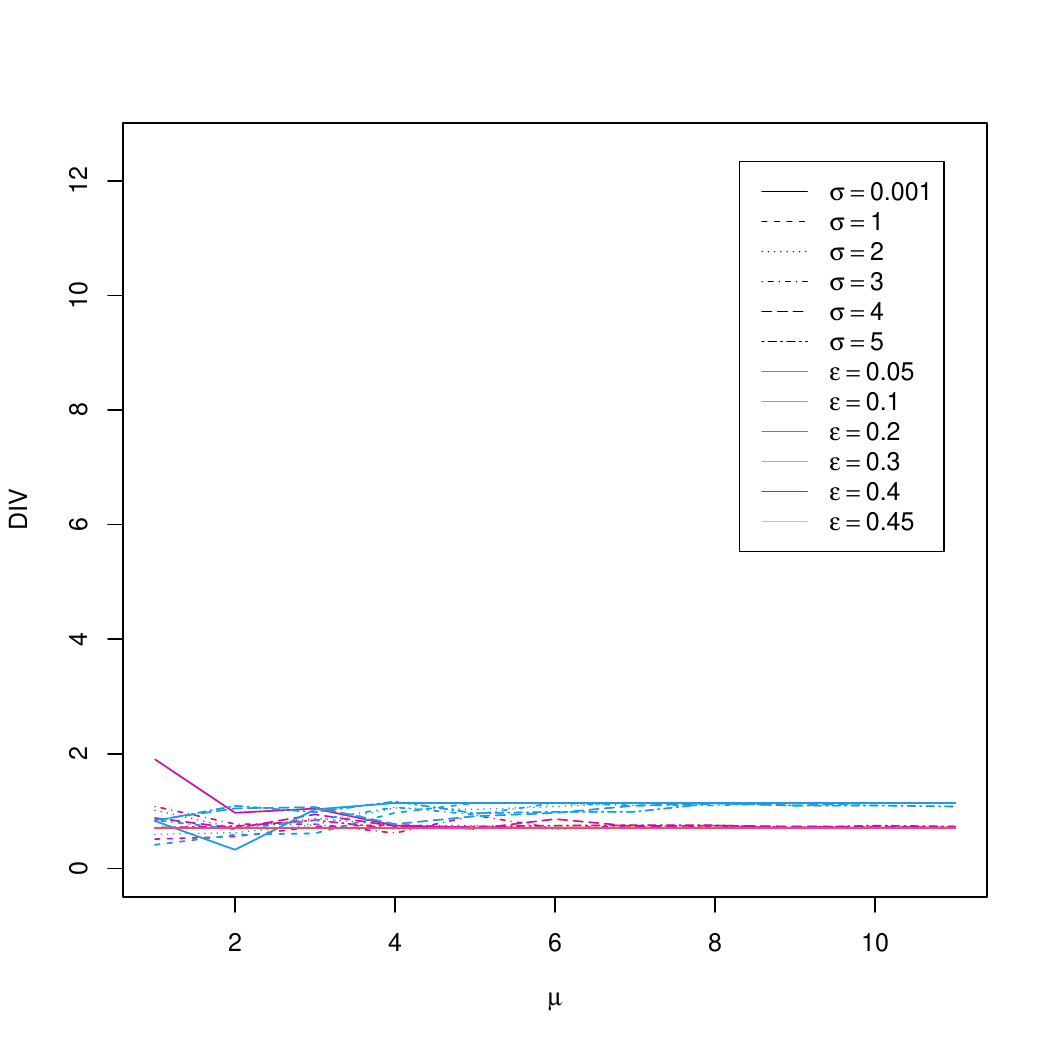}
\includegraphics[width=0.32\textwidth]{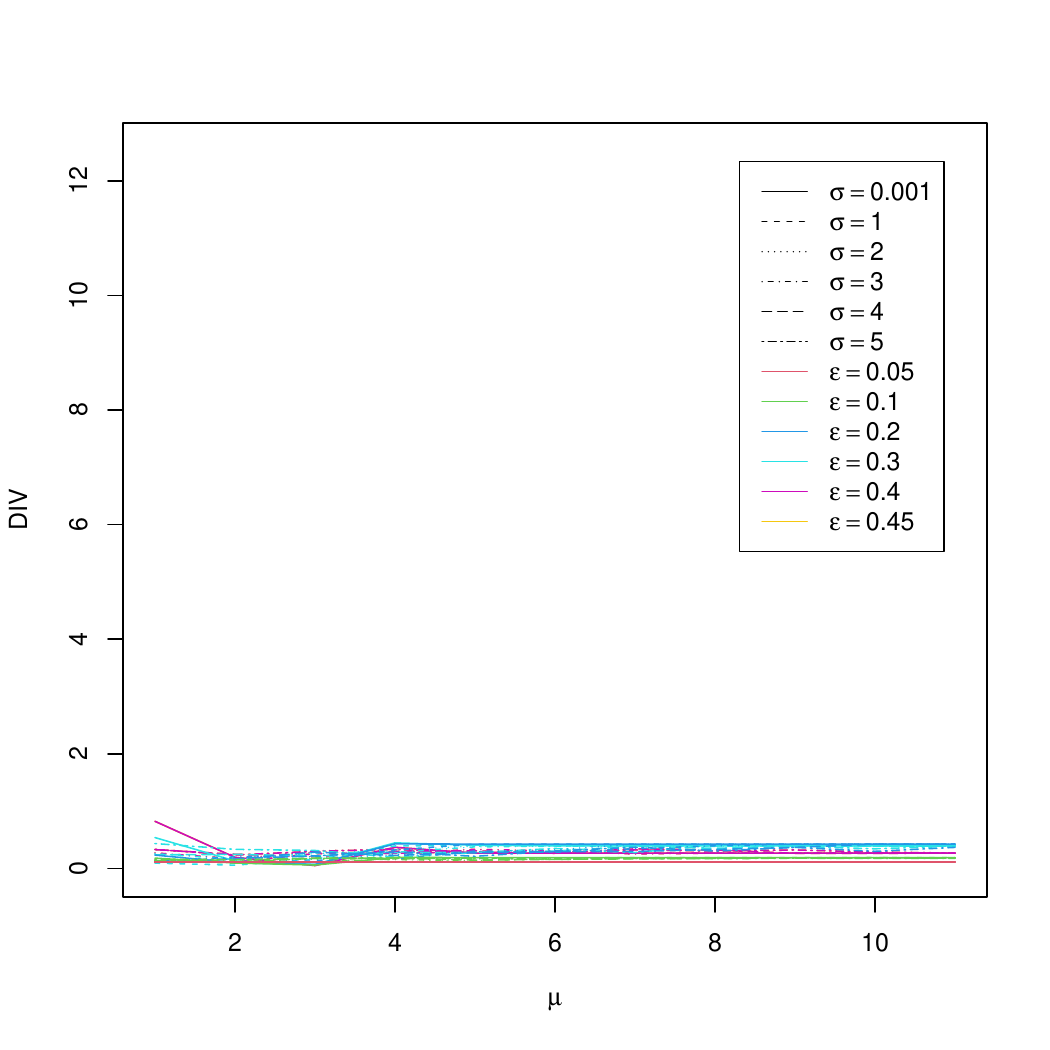} 
\includegraphics[width=0.32\textwidth]{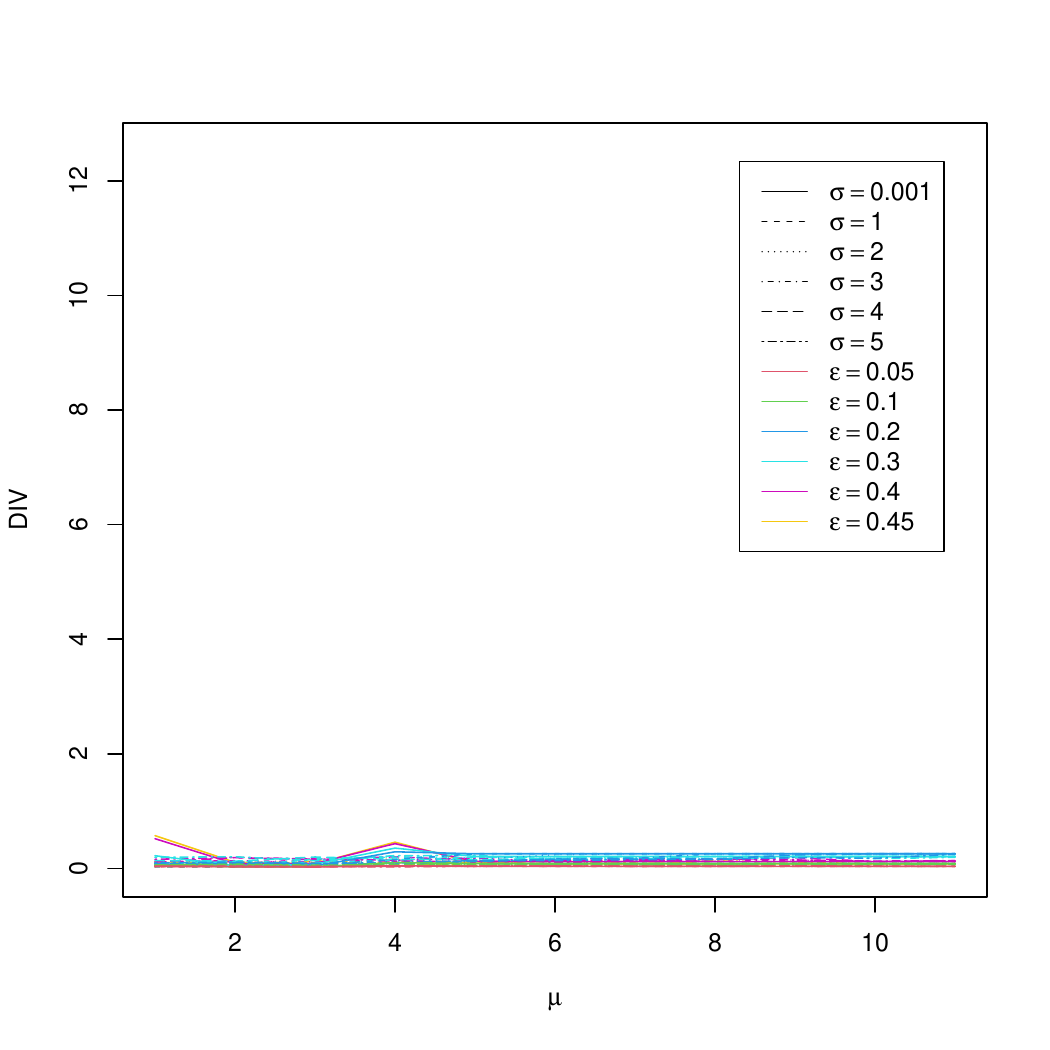} \\
\includegraphics[width=0.32\textwidth]{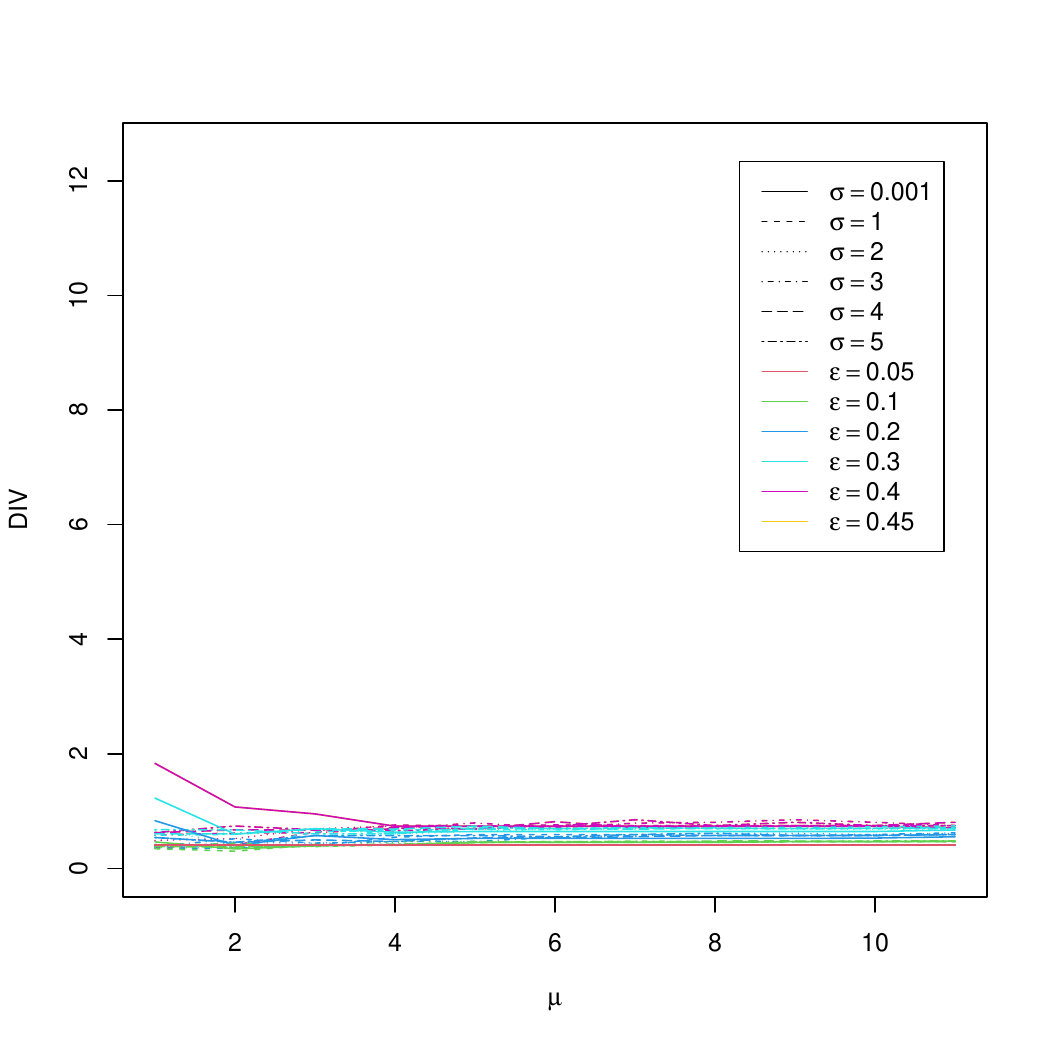}
\includegraphics[width=0.32\textwidth]{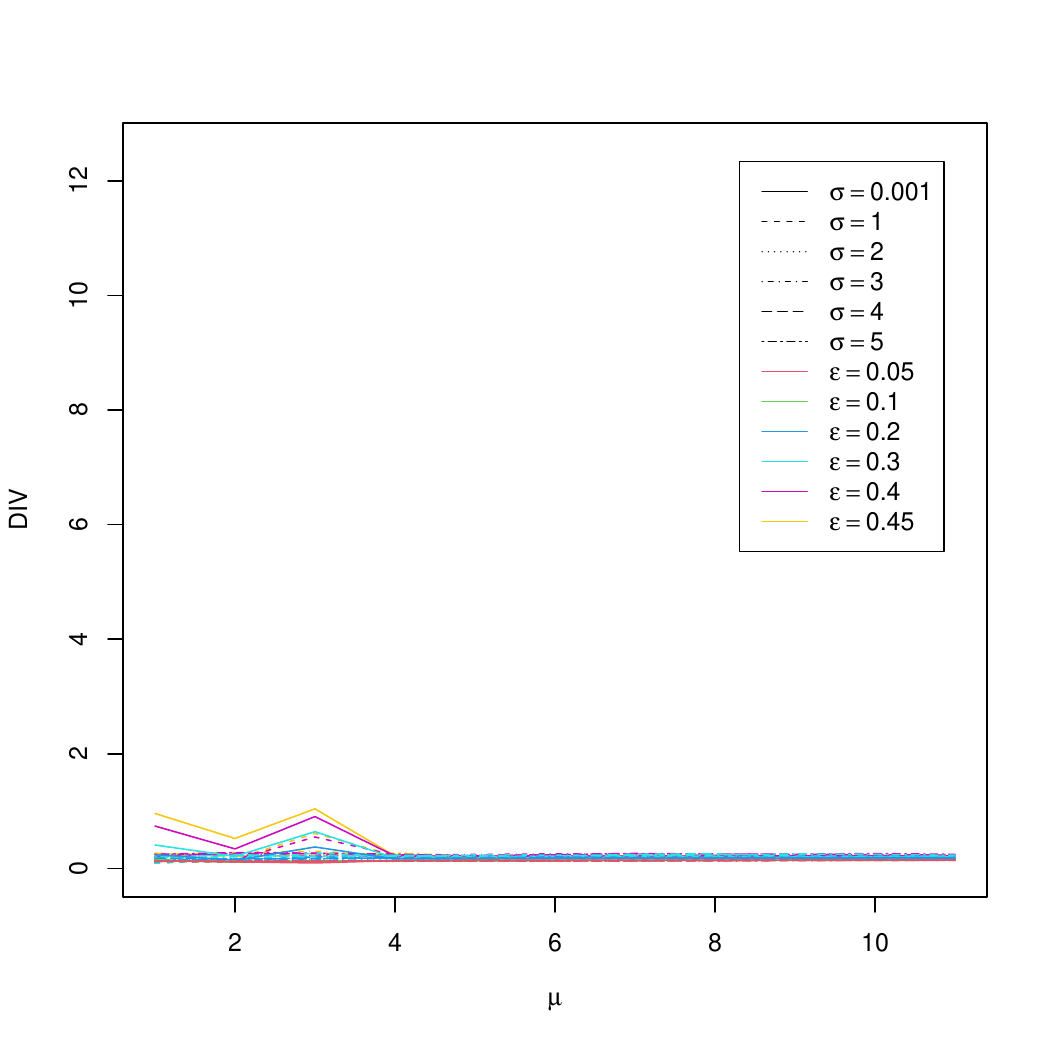} 
\includegraphics[width=0.32\textwidth]{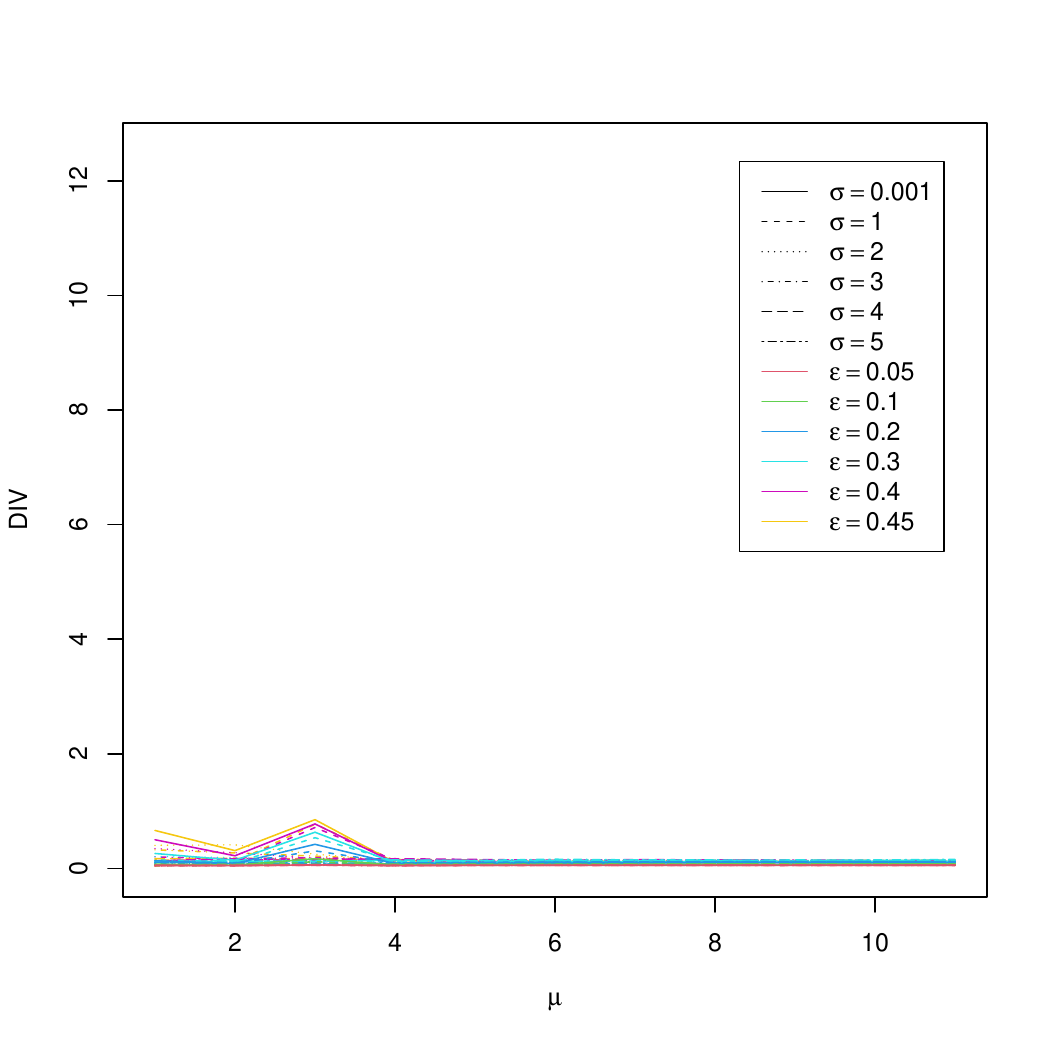} \\
\includegraphics[width=0.32\textwidth]{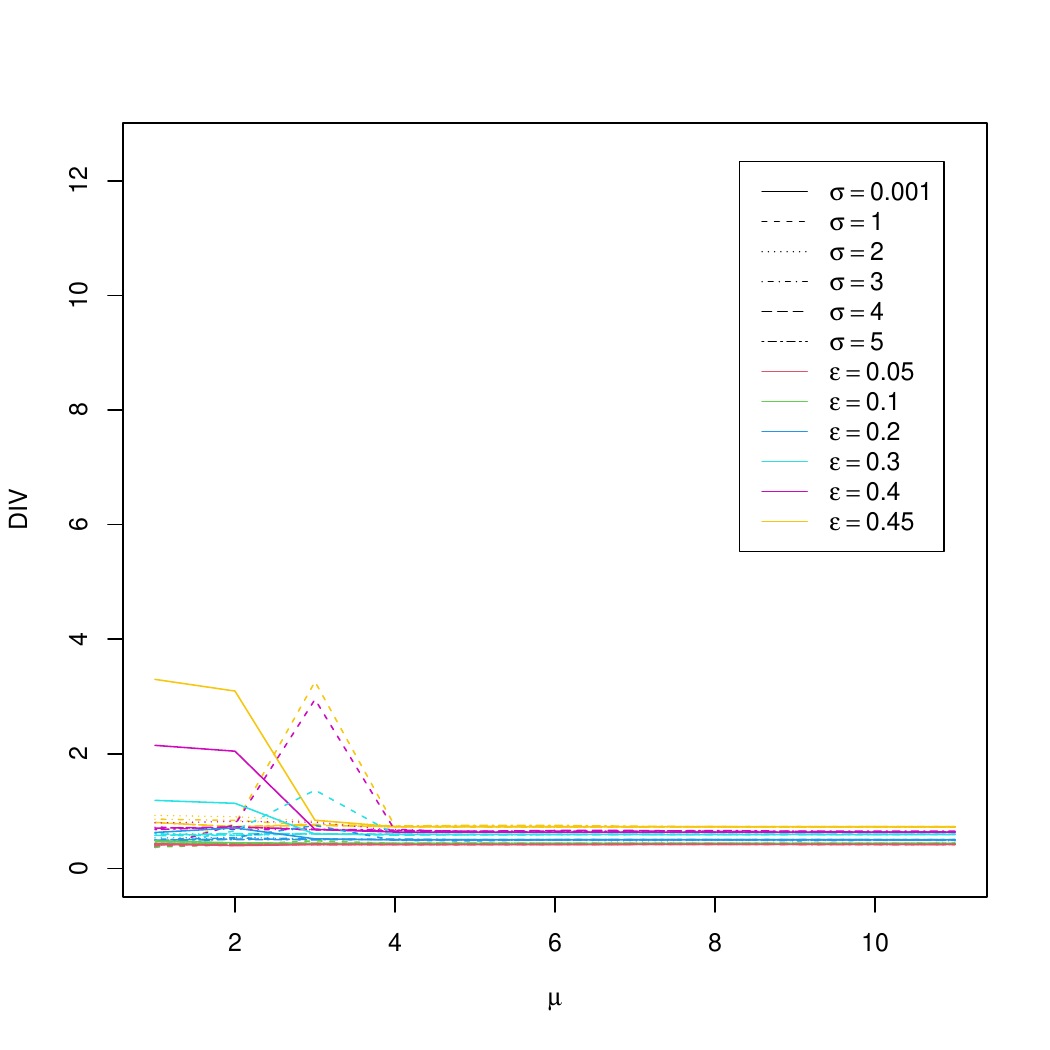}
\includegraphics[width=0.32\textwidth]{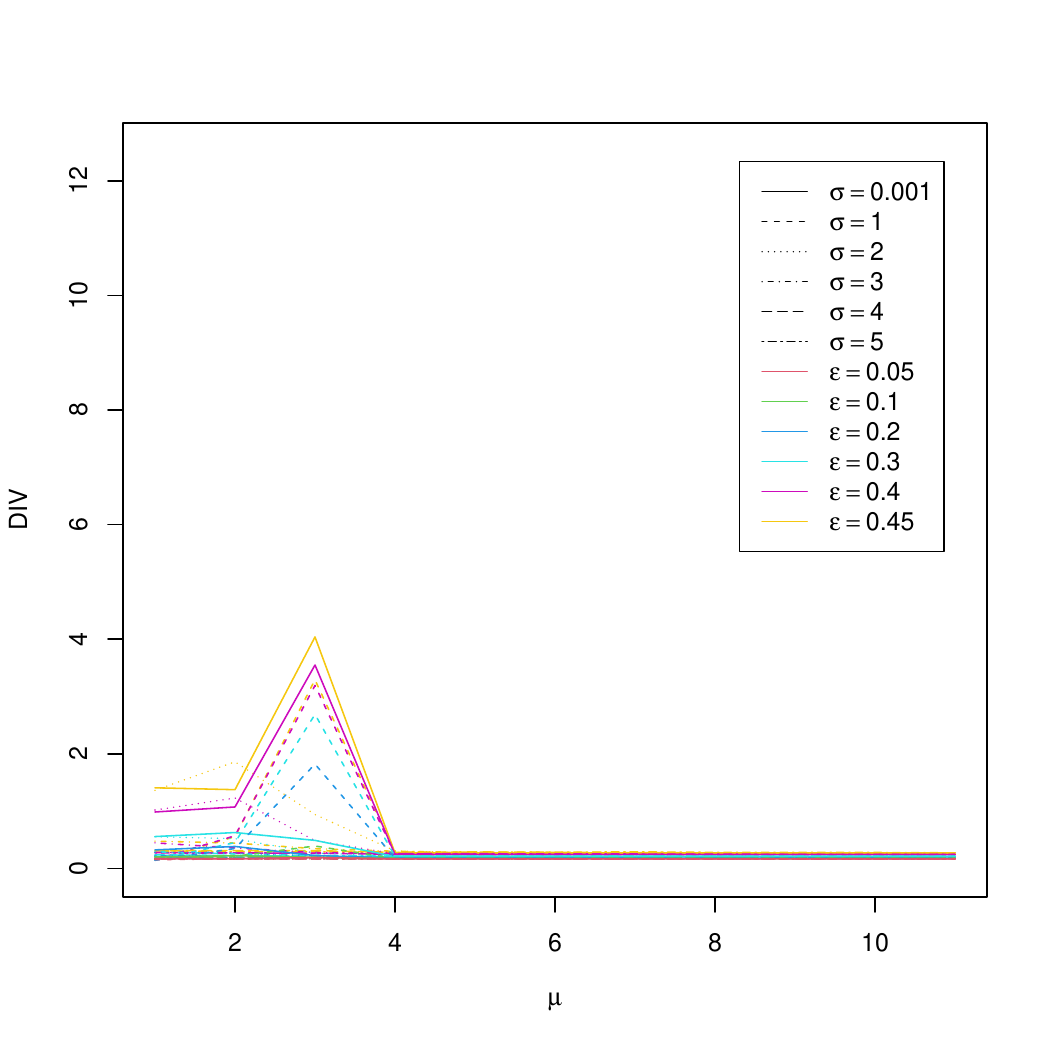} 
\includegraphics[width=0.32\textwidth]{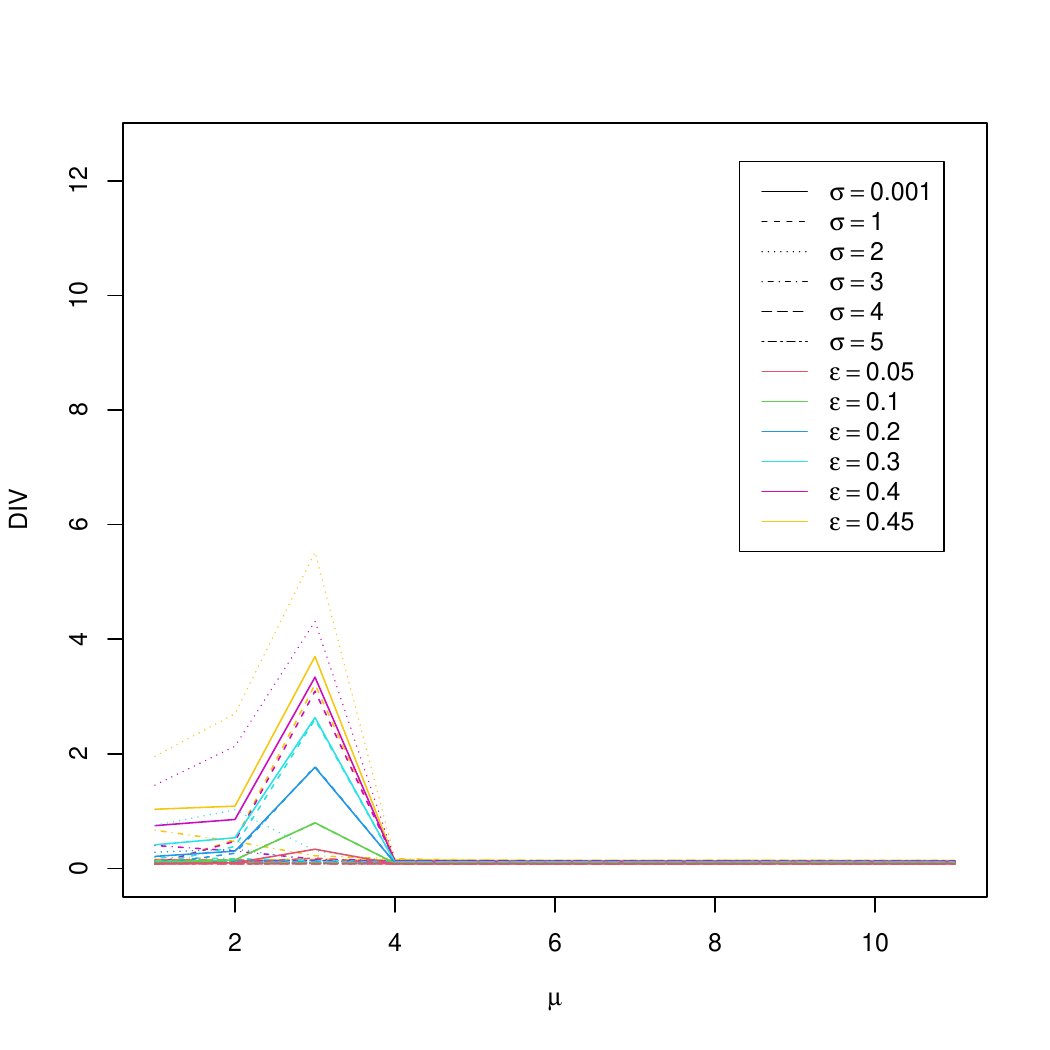} \\
\caption{Monte Carlo Simulation. Kullback--Leibler Divergence for the proposed method starting at the true values with $\alpha=0.5$ as a function of the contamination average $\mu$ ($x$-axis), contamination scale $\sigma$ (different line styles) and contamination level $\varepsilon$ (colors). Rows: number of variables $p=1, 2, 5$ and columns: sample size factor $s=2, 5, 10$.}
\label{sup:fig:monte:DIV:0.5:1}
\end{figure}  

\begin{figure}
\centering
\includegraphics[width=0.32\textwidth]{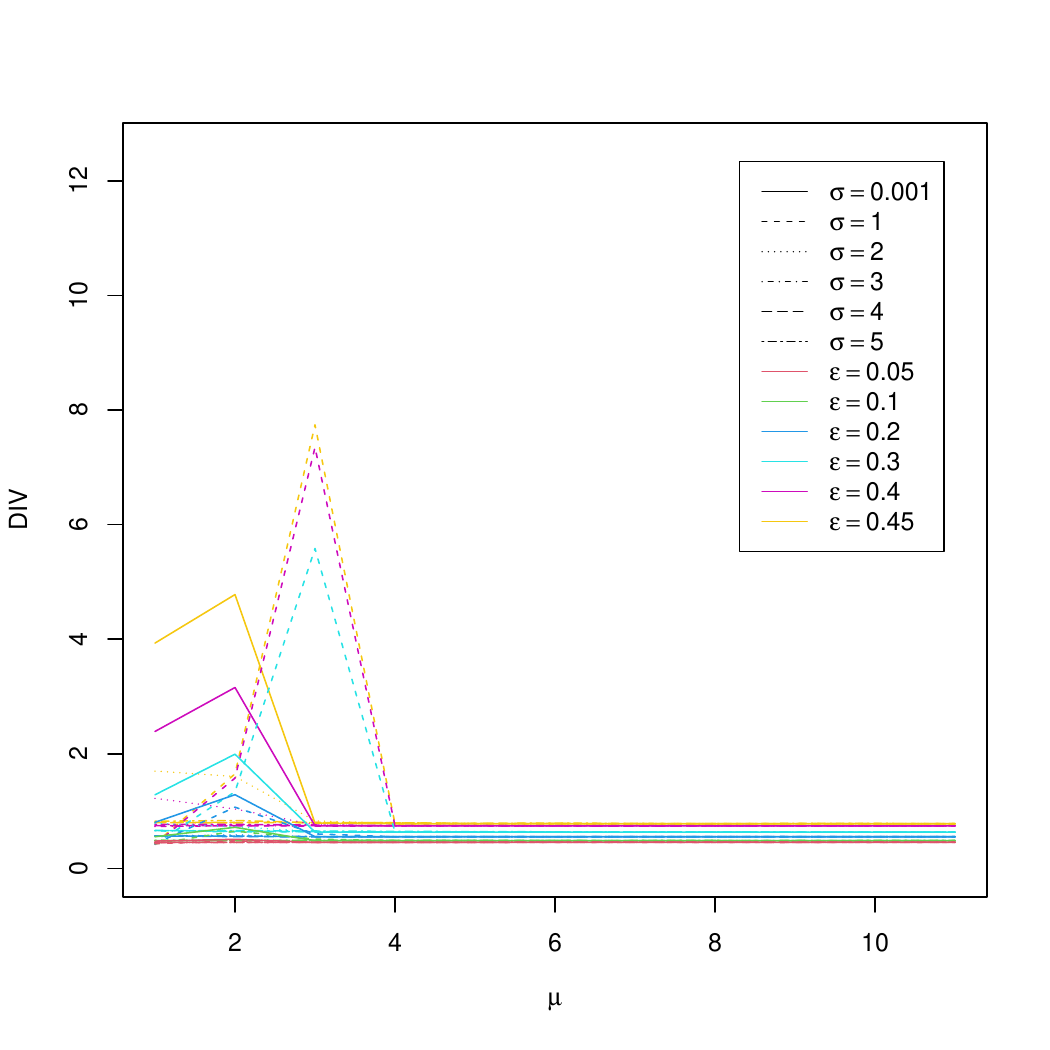}
\includegraphics[width=0.32\textwidth]{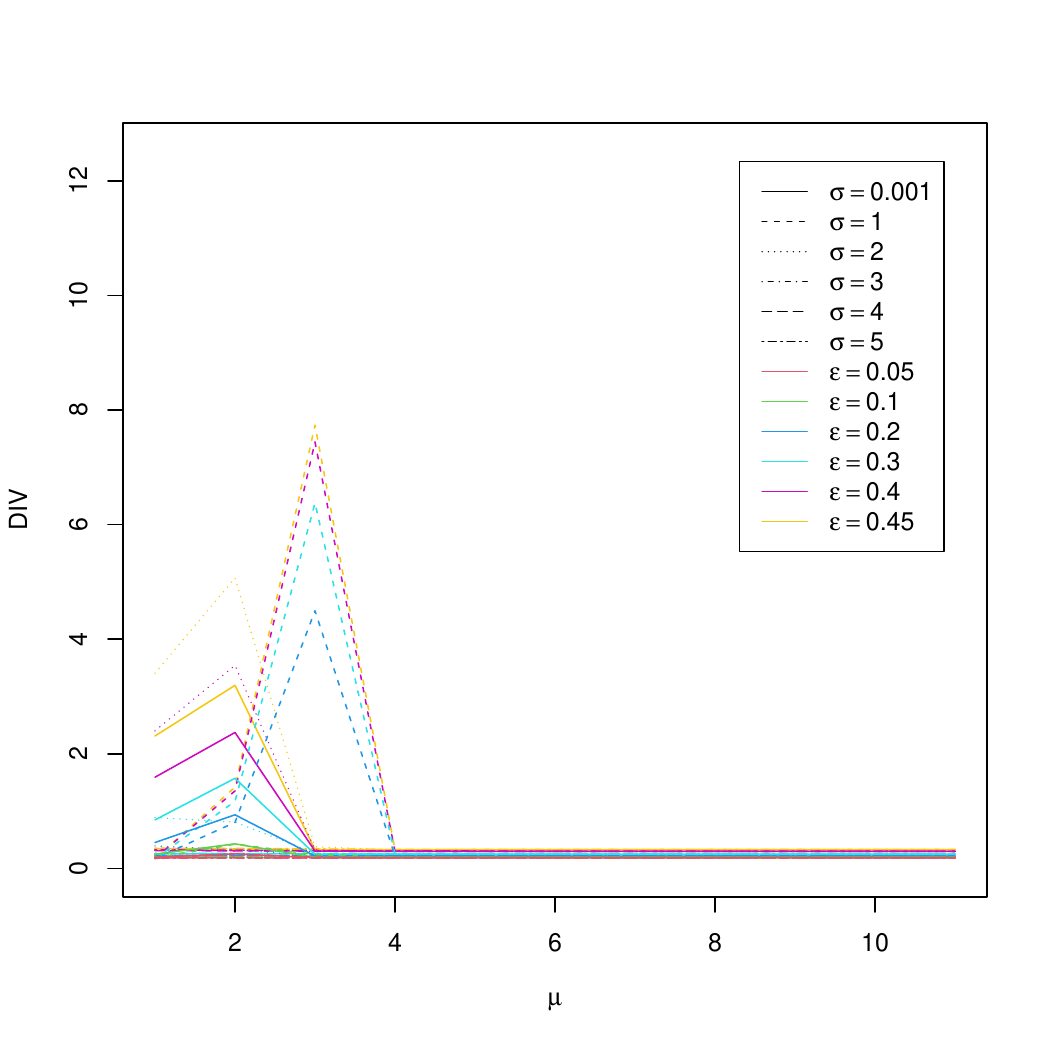} 
\includegraphics[width=0.32\textwidth]{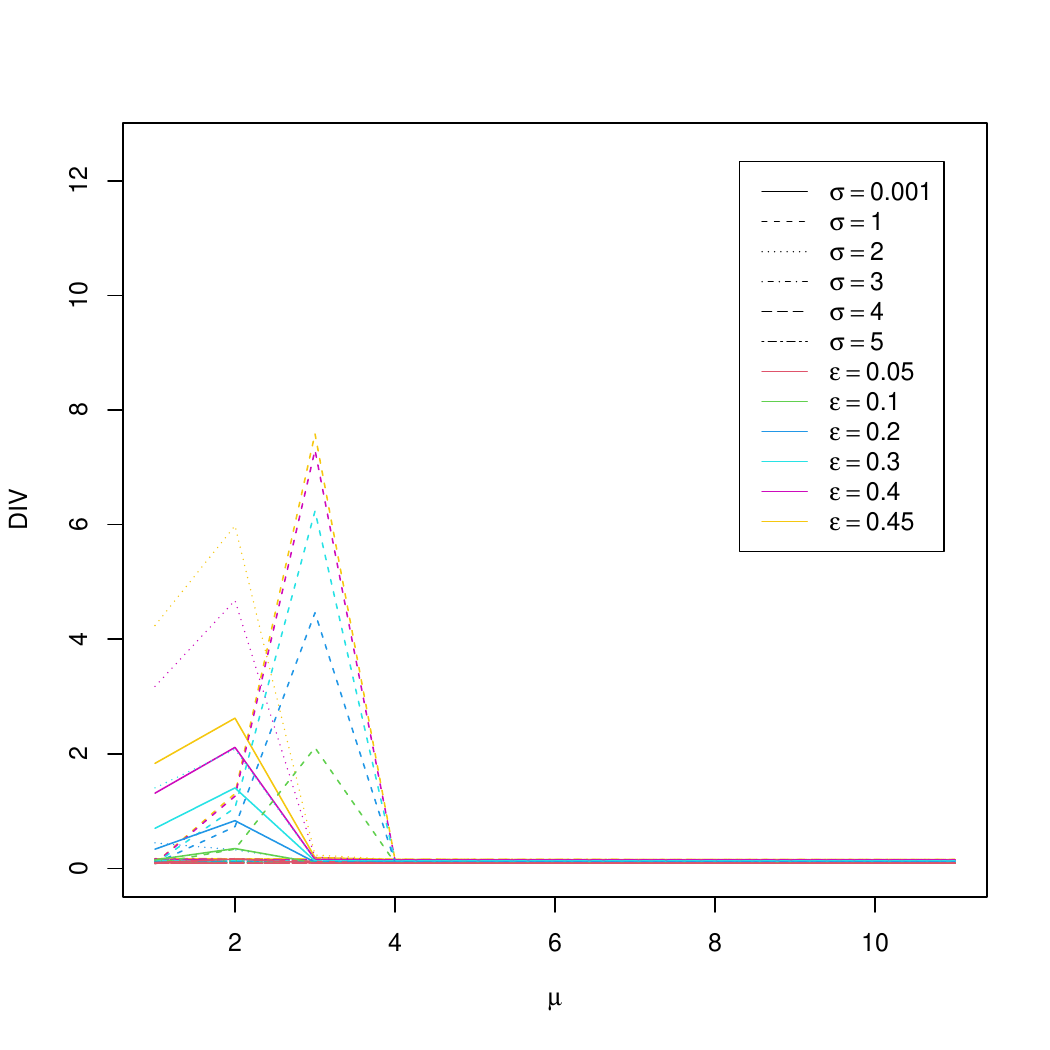} \\
\includegraphics[width=0.32\textwidth]{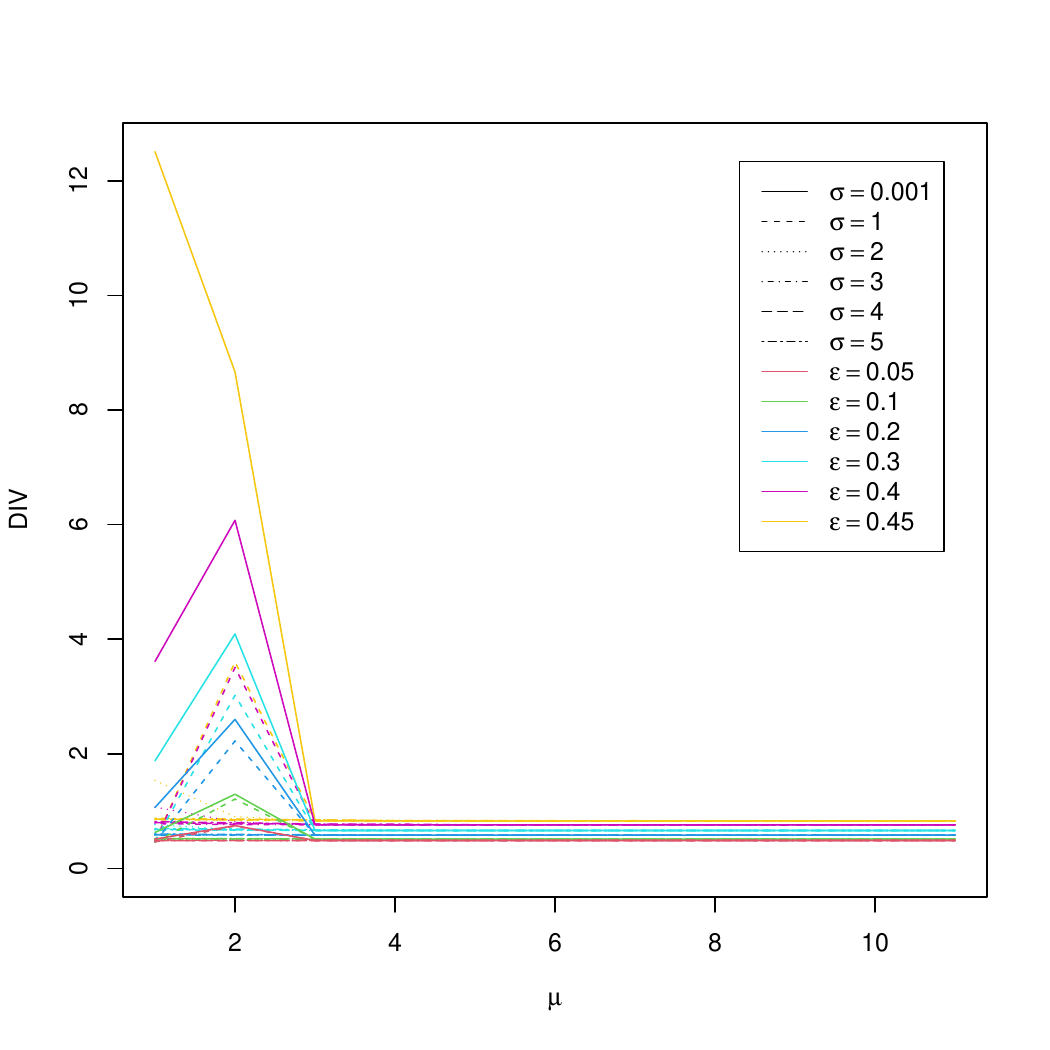}
\includegraphics[width=0.32\textwidth]{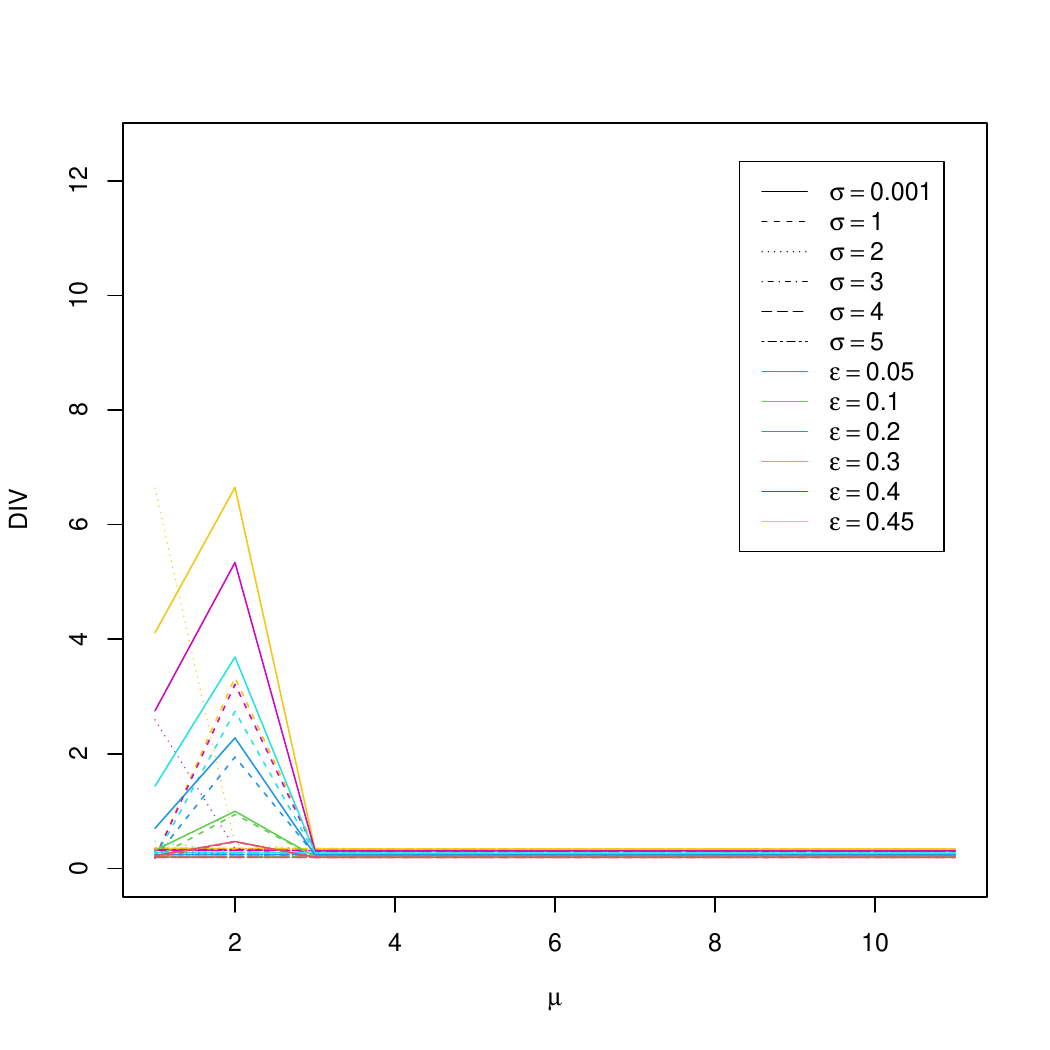} 
\includegraphics[width=0.32\textwidth]{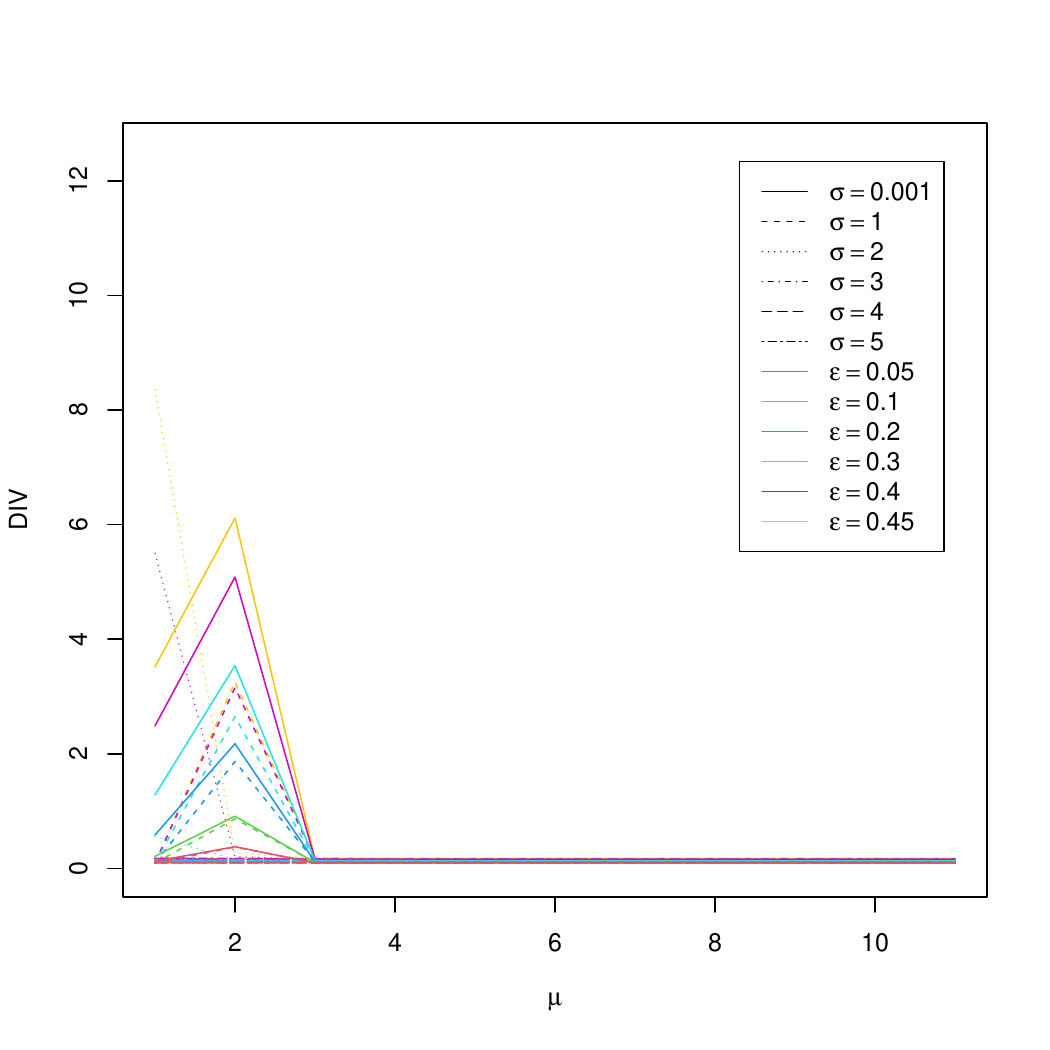}
\caption{Monte Carlo Simulation. Kullback--Leibler Divergence for the proposed method starting at the true values with $\alpha=0.5$ as a function of the contamination average $\mu$ ($x$-axis), contamination scale $\sigma$ (different line styles) and contamination level $\varepsilon$ (colors). Rows: number of variables $p=10, 20$ and columns: sample size factor $s=2, 5, 10$.}
\label{sup:fig:monte:DIV:0.5:2}
\end{figure}  

\begin{figure}
\centering
\includegraphics[width=0.32\textwidth]{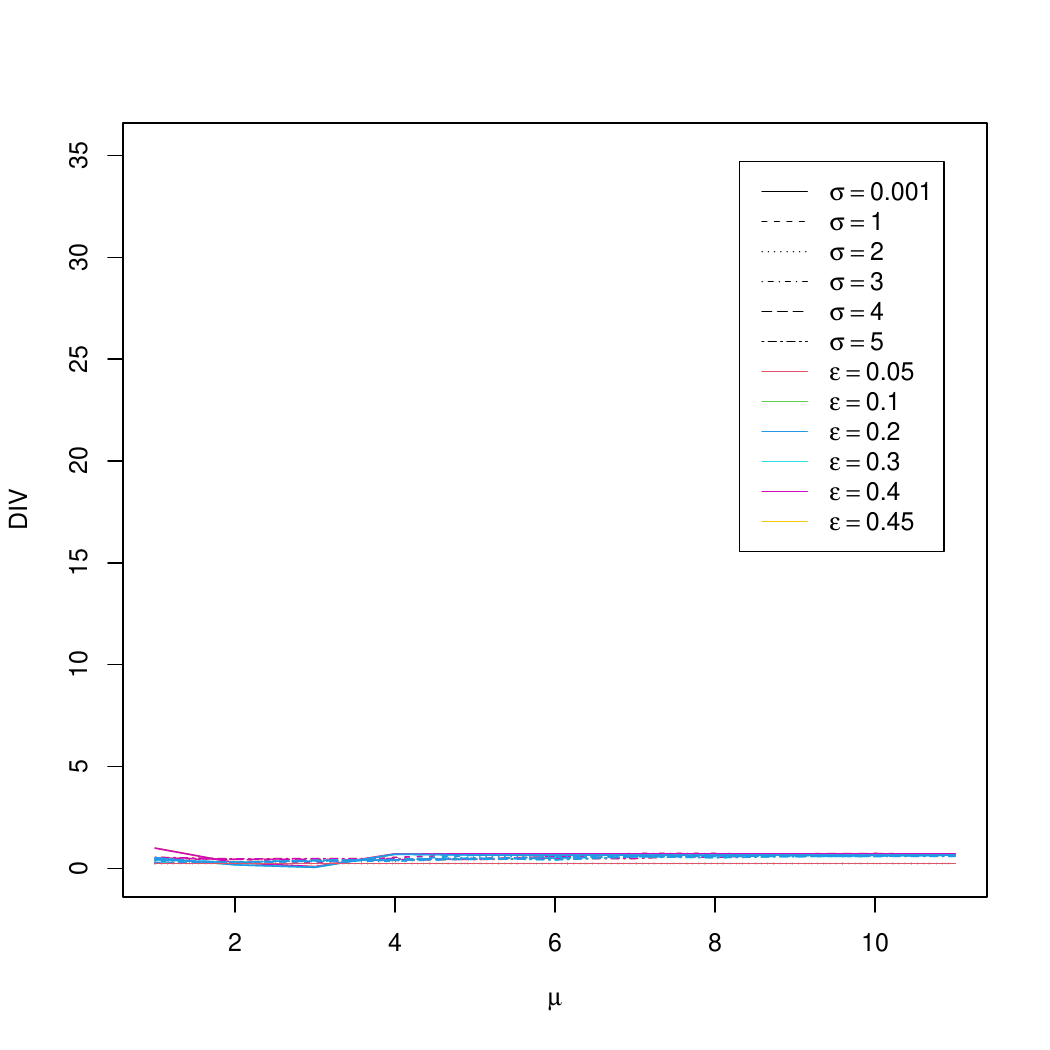}
\includegraphics[width=0.32\textwidth]{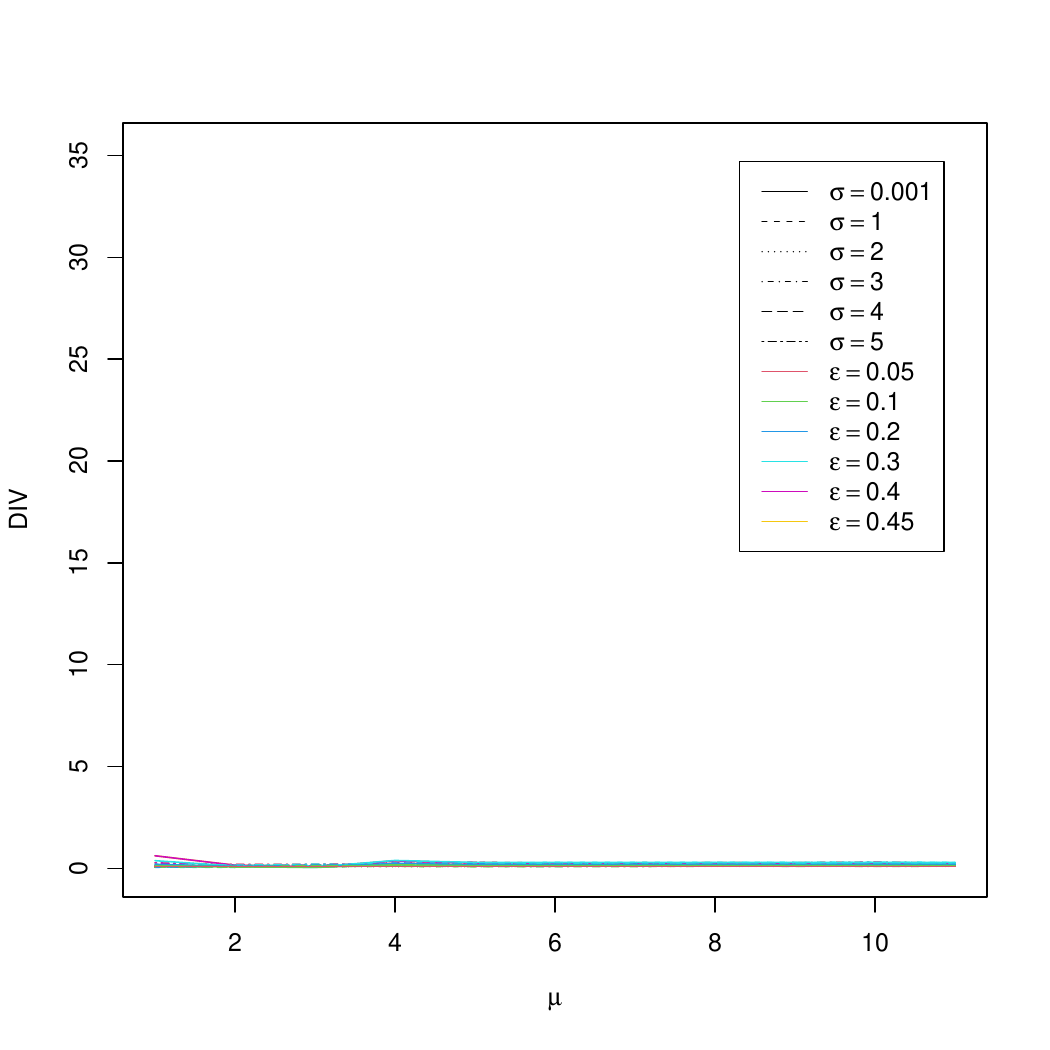} 
\includegraphics[width=0.32\textwidth]{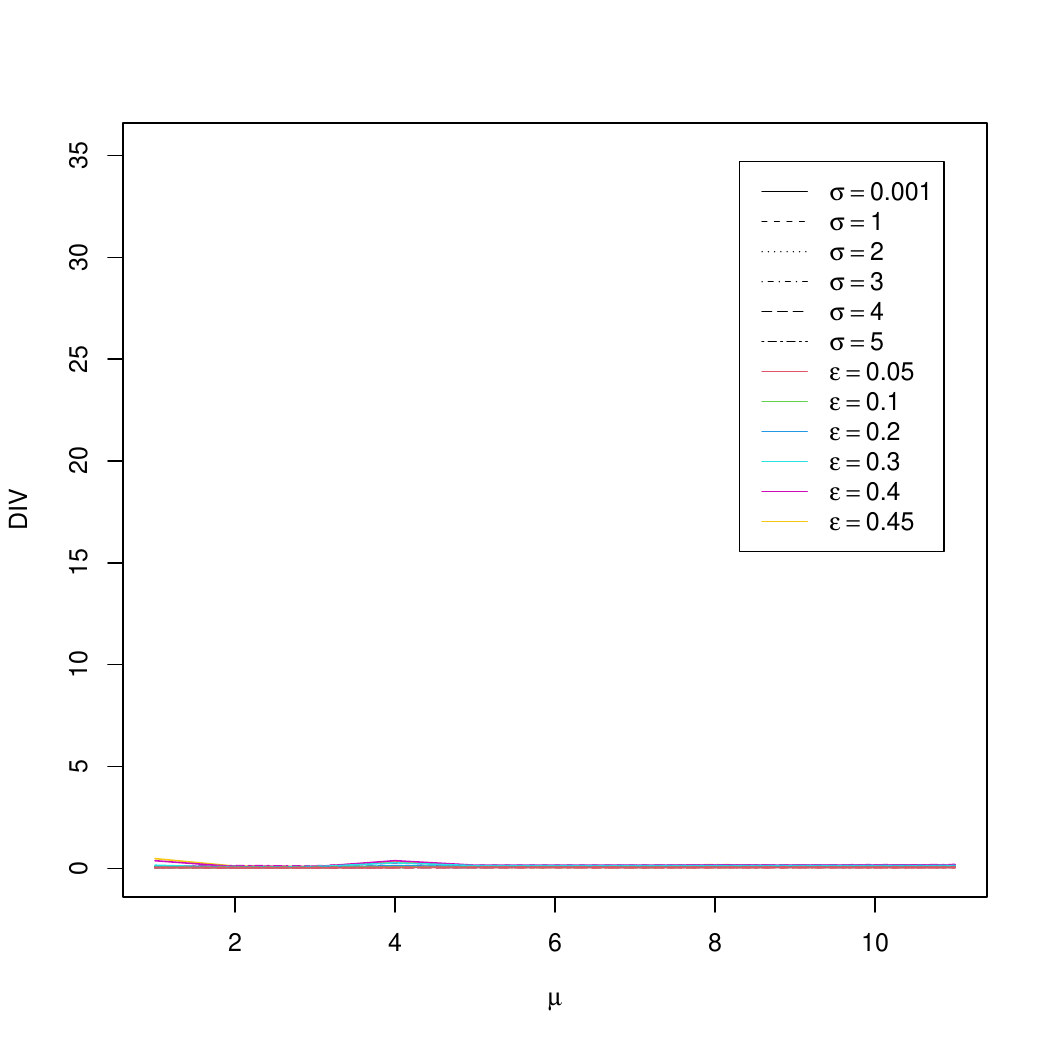} \\
\includegraphics[width=0.32\textwidth]{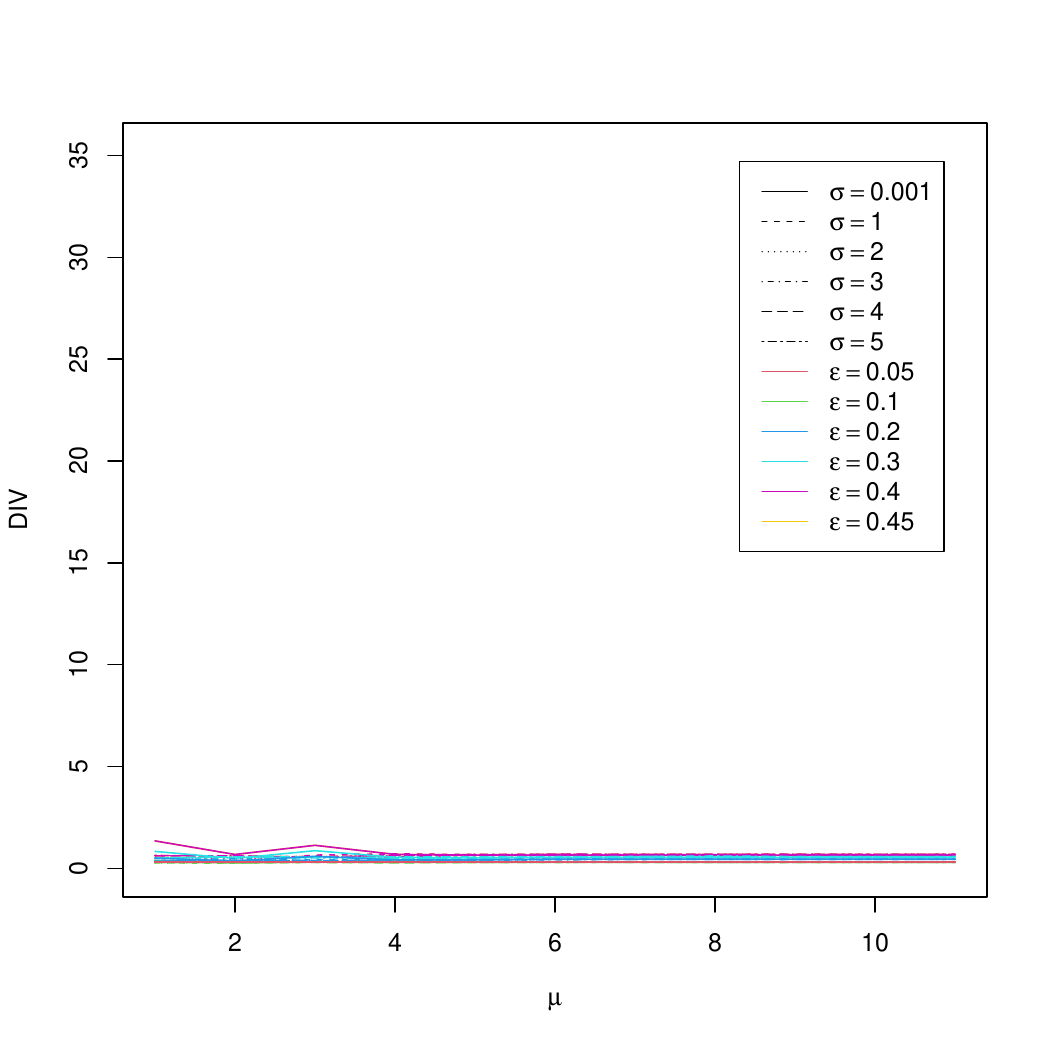}
\includegraphics[width=0.32\textwidth]{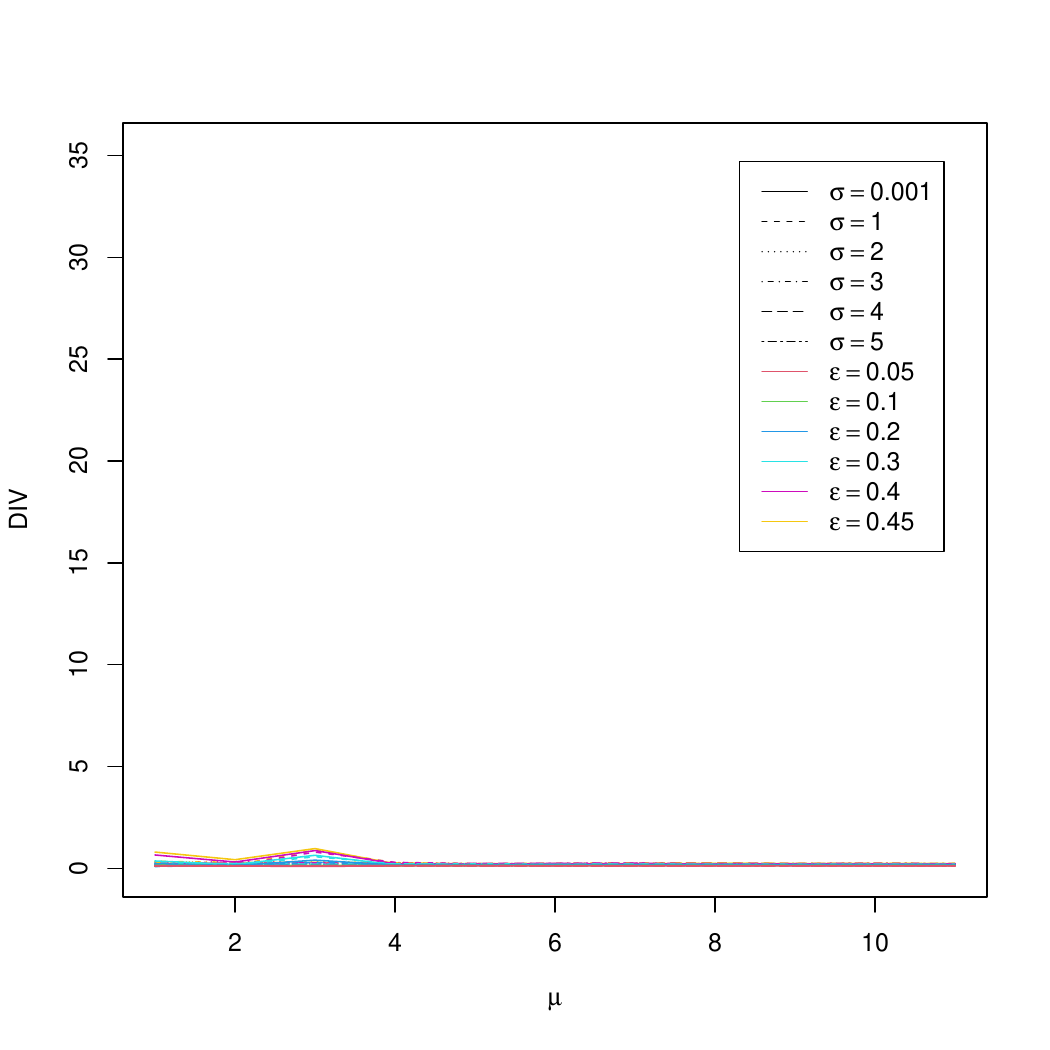} 
\includegraphics[width=0.32\textwidth]{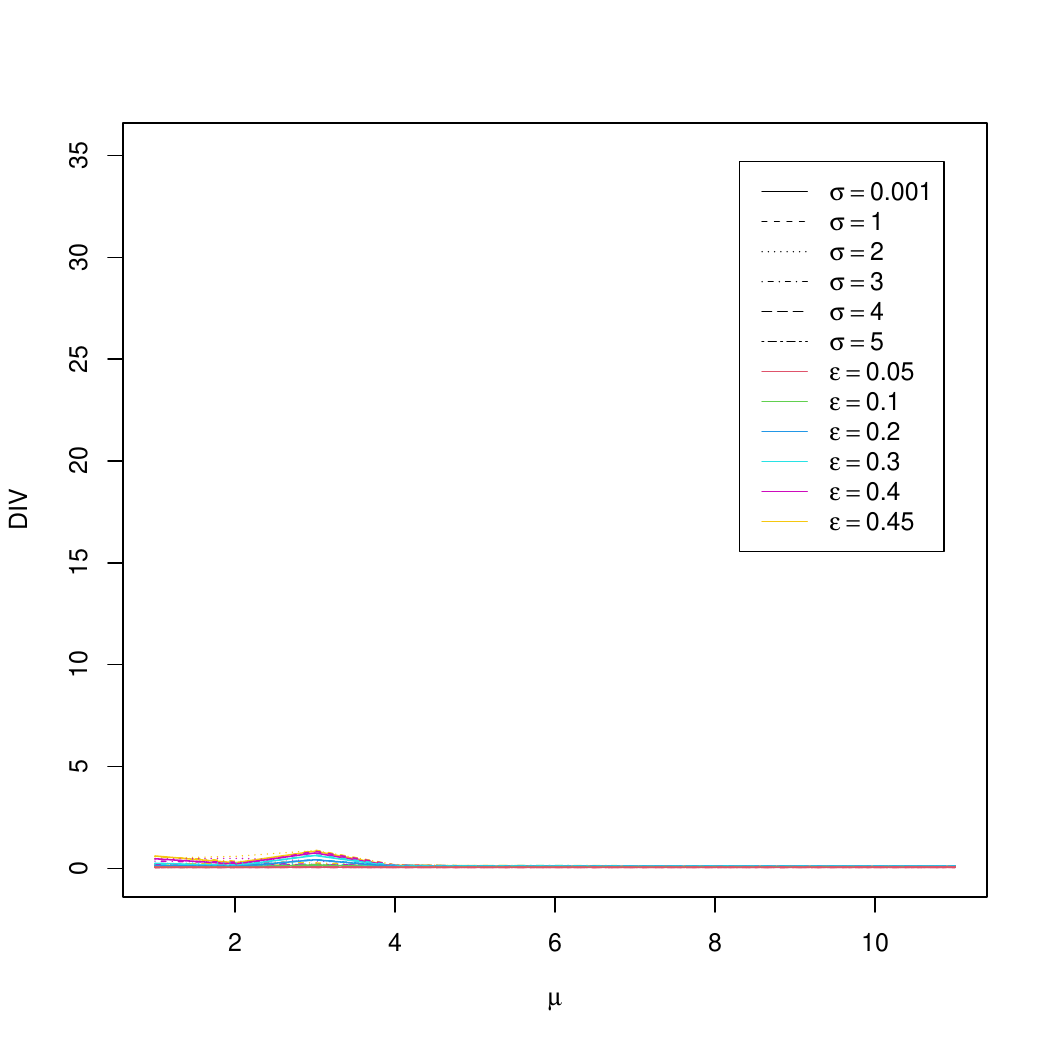} \\
\includegraphics[width=0.32\textwidth]{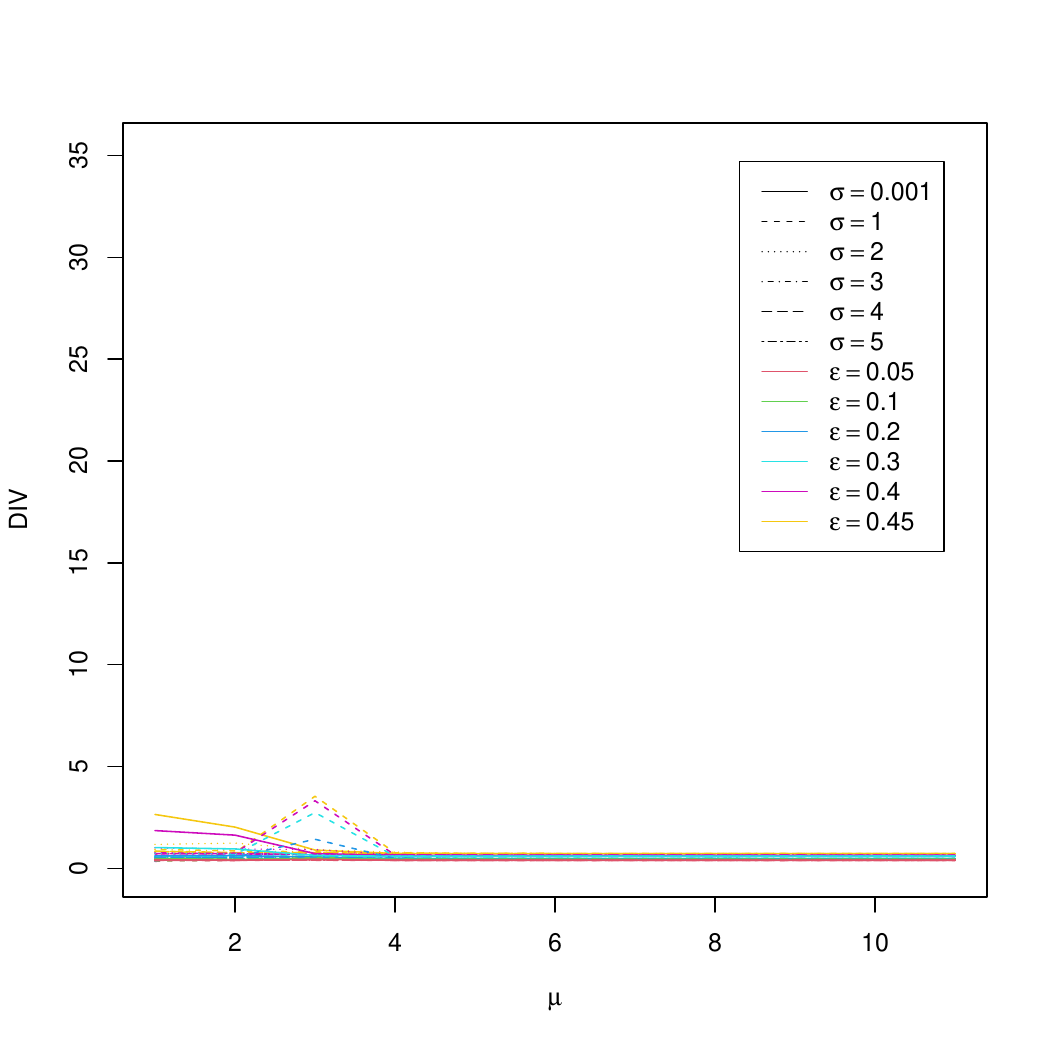}
\includegraphics[width=0.32\textwidth]{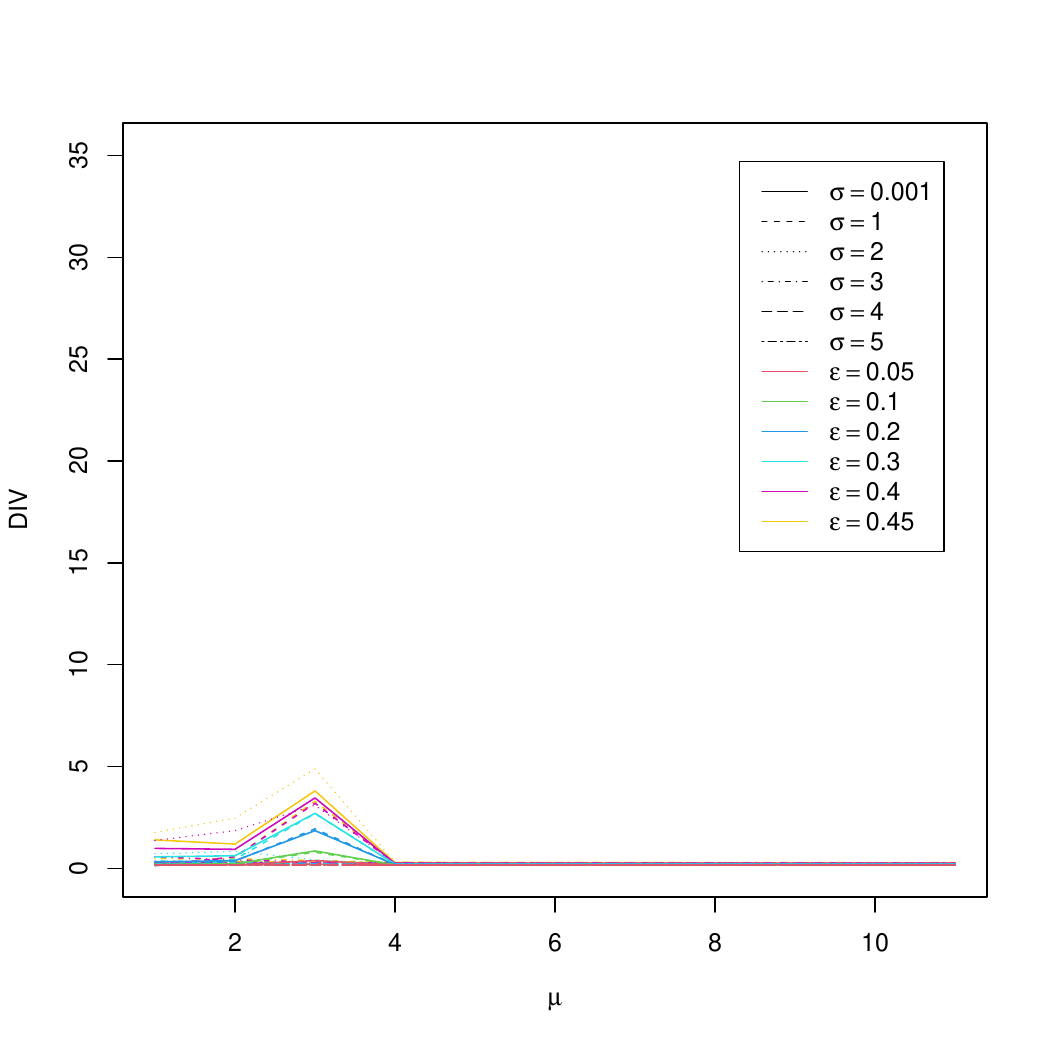} 
\includegraphics[width=0.32\textwidth]{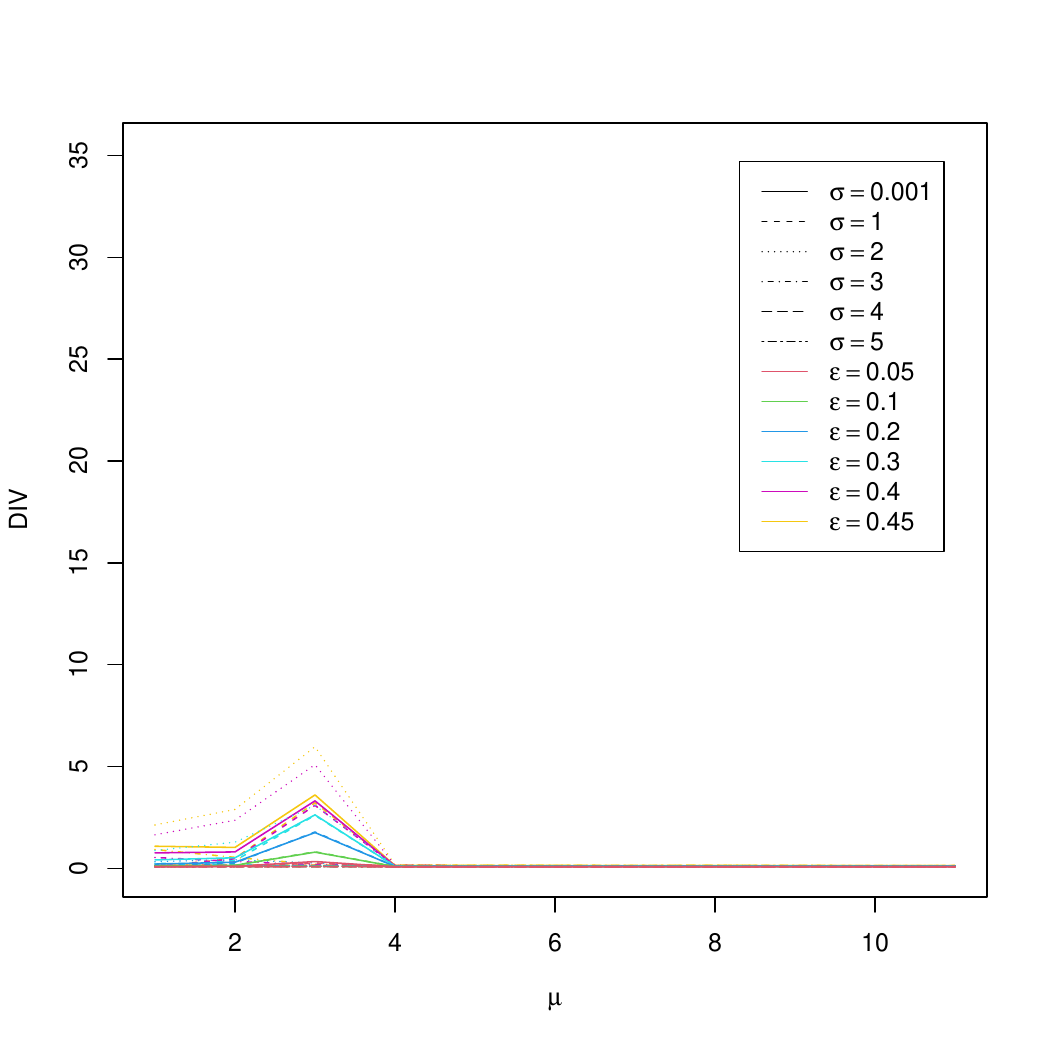} \\
\caption{Monte Carlo Simulation. Kullback--Leibler Divergence for the proposed method starting at the true values with $\alpha=0.75$ as a function of the contamination average $\mu$ ($x$-axis), contamination scale $\sigma$ (different line styles) and contamination level $\varepsilon$ (colors). Rows: number of variables $p=1, 2, 5$ and columns: sample size factor $s=2, 5, 10$.}
\label{sup:fig:monte:DIV:0.75:1}
\end{figure}  

\begin{figure}
\centering
\includegraphics[width=0.32\textwidth]{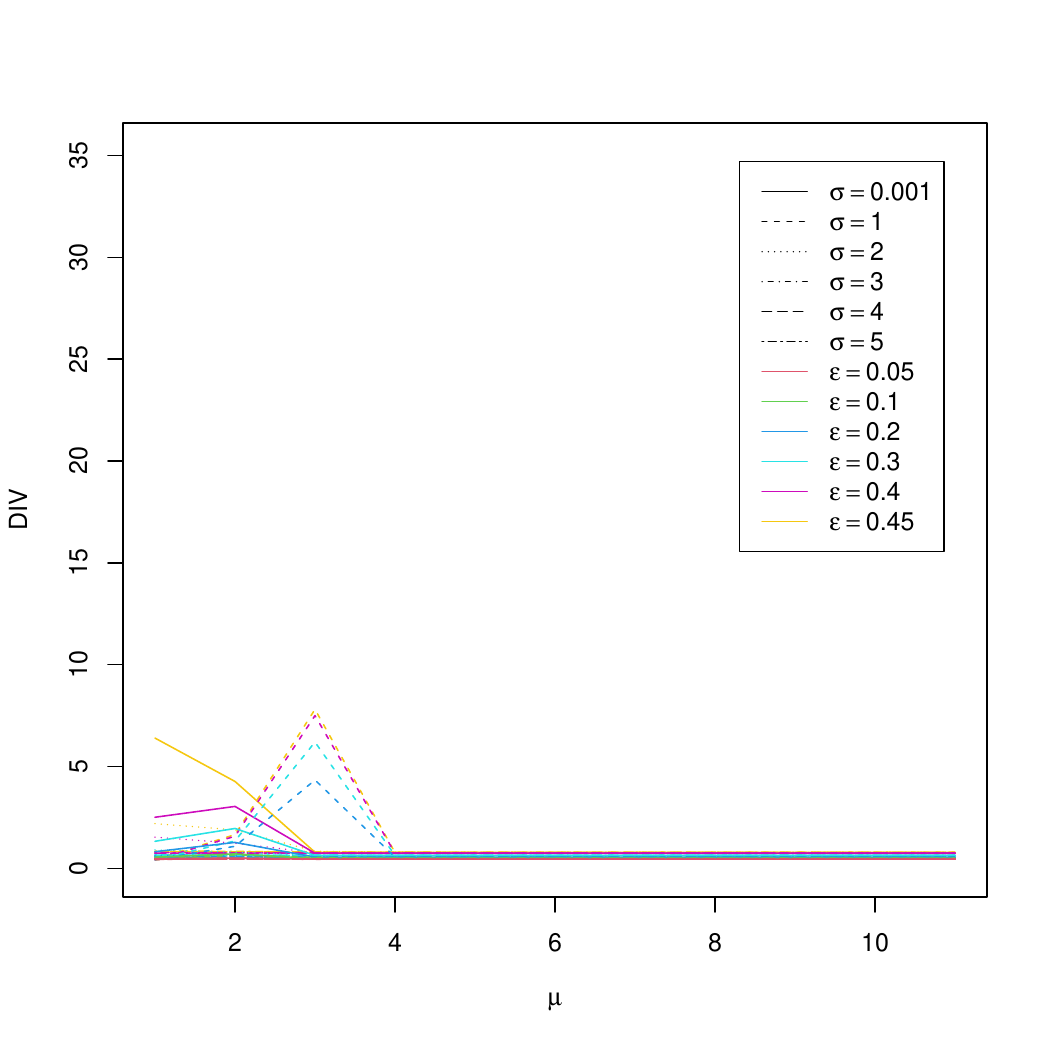}
\includegraphics[width=0.32\textwidth]{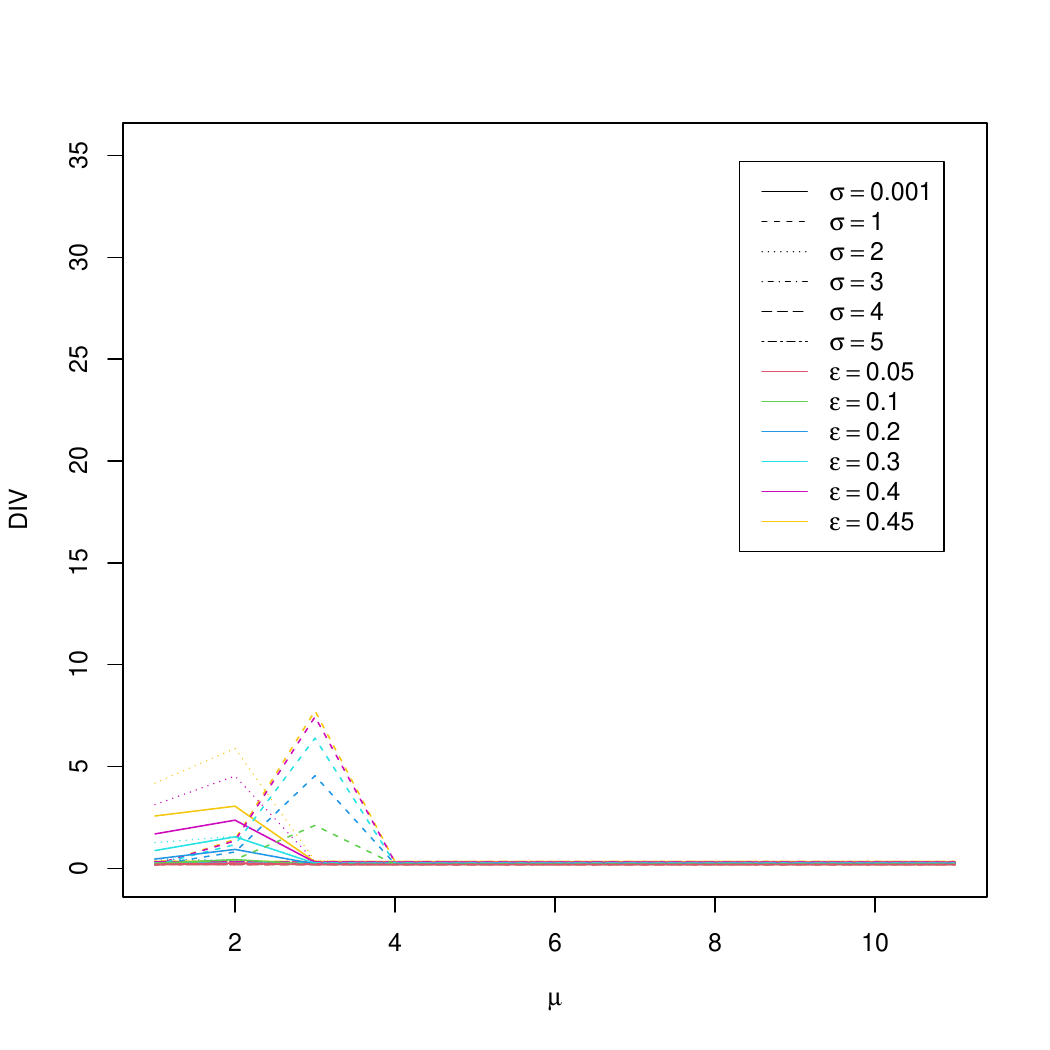} 
\includegraphics[width=0.32\textwidth]{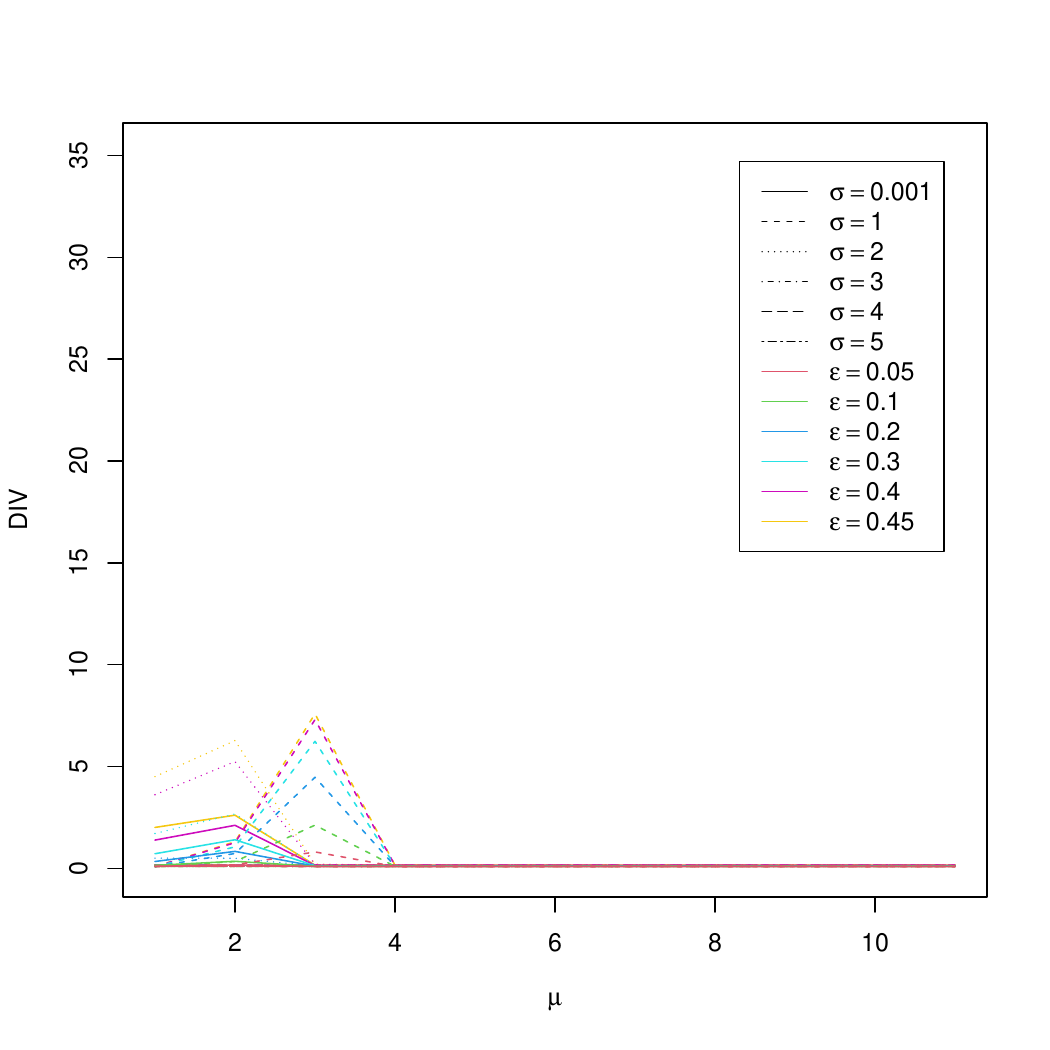} \\
\includegraphics[width=0.32\textwidth]{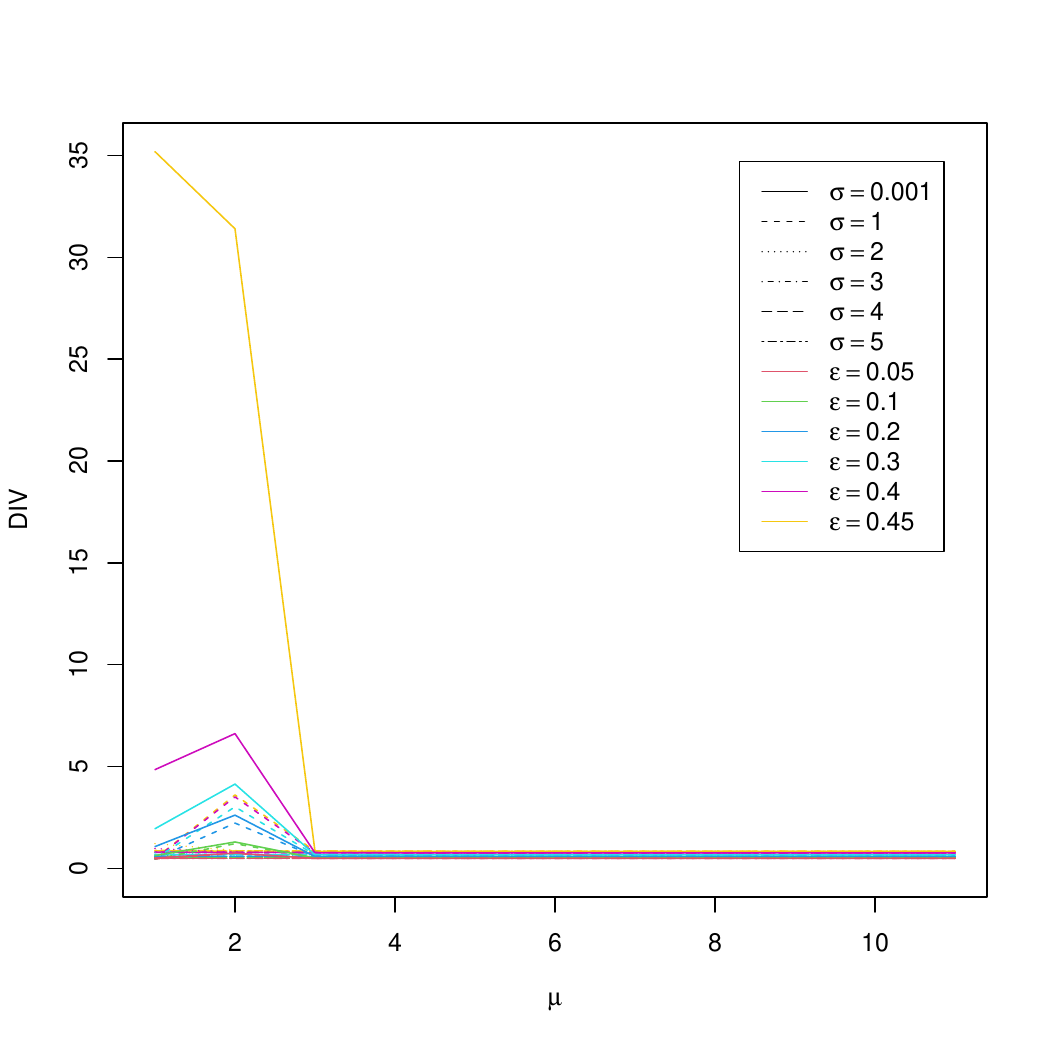}
\includegraphics[width=0.32\textwidth]{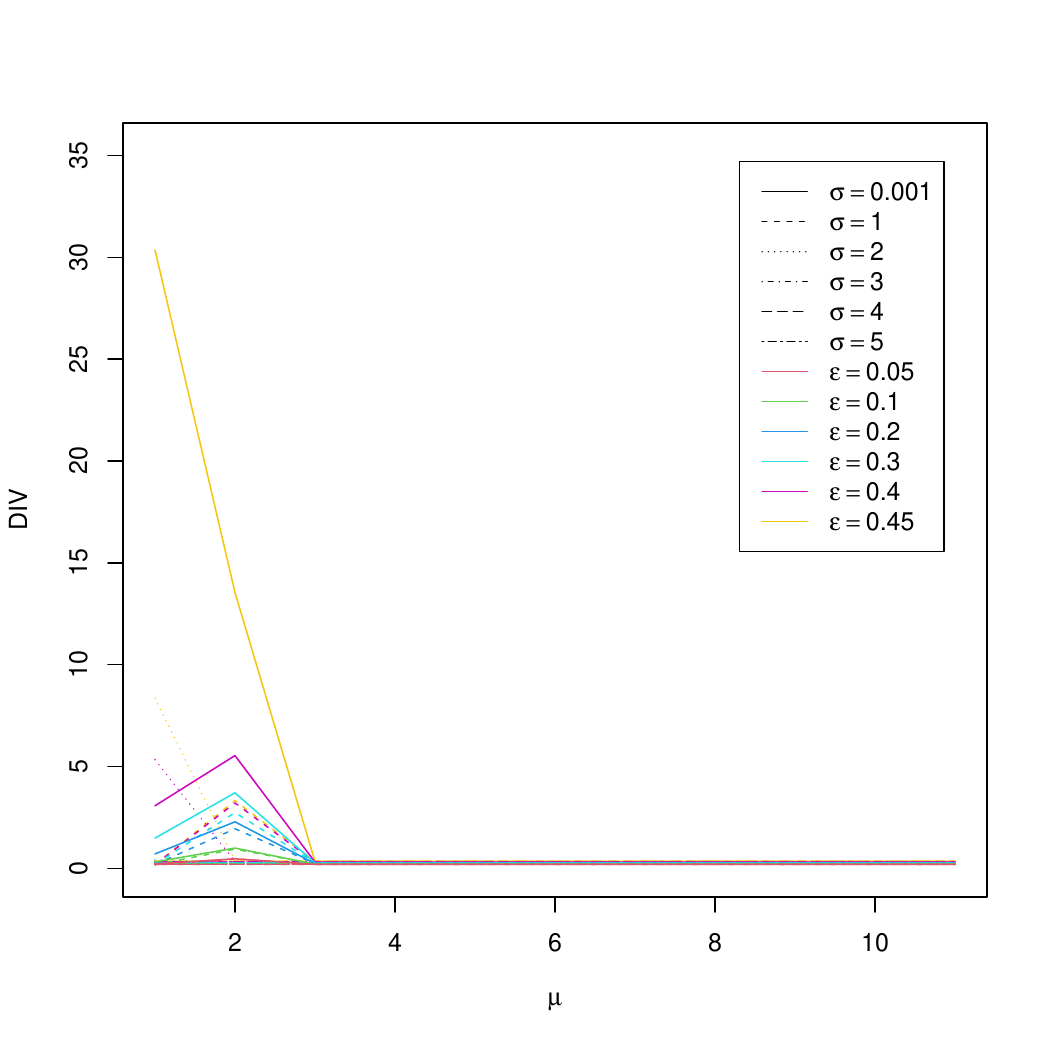} 
\includegraphics[width=0.32\textwidth]{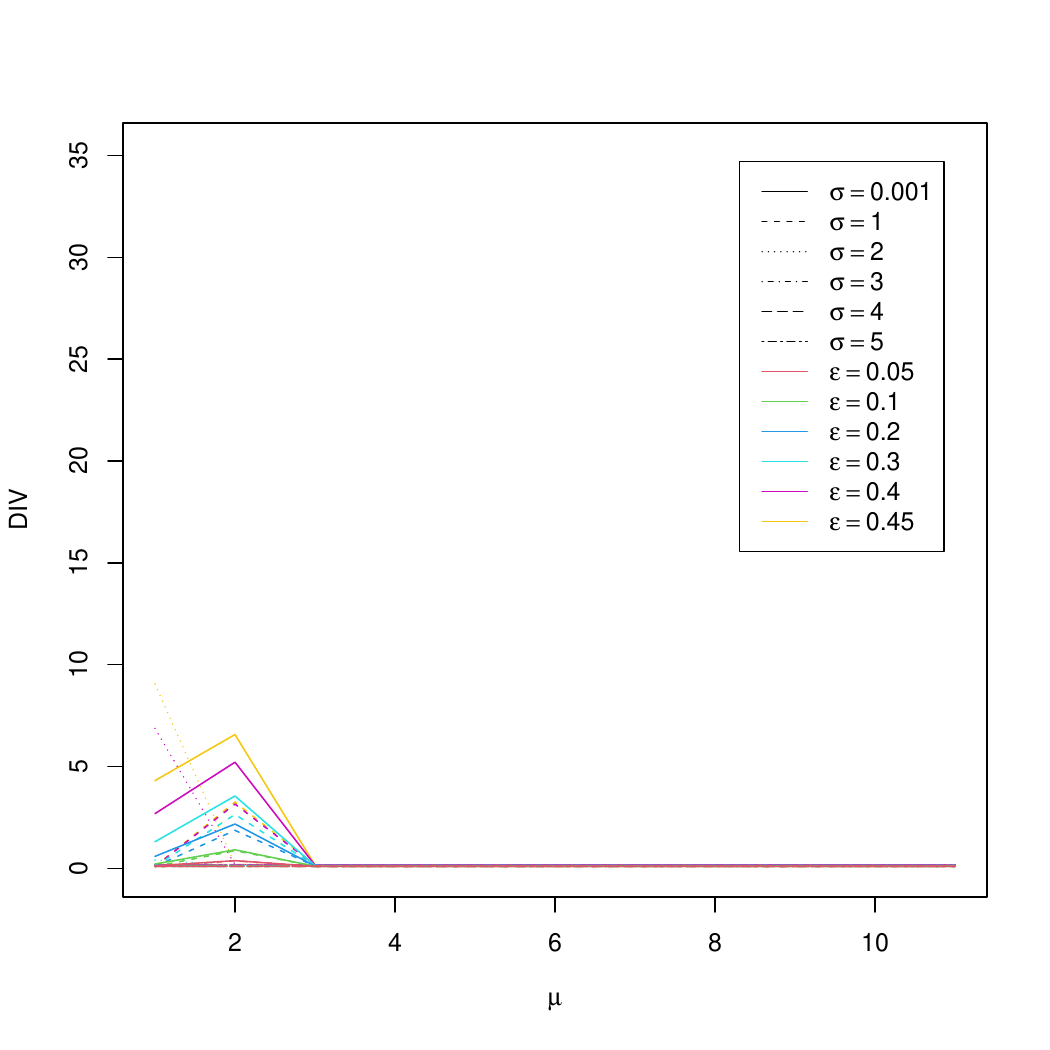}
\caption{Monte Carlo Simulation. Kullback--Leibler Divergence for the proposed method starting at the true values with $\alpha=0.75$ as a function of the contamination average $\mu$ ($x$-axis), contamination scale $\sigma$ (different line styles) and contamination level $\varepsilon$ (colors). Rows: number of variables $p=10, 20$ and columns: sample size factor $s=2, 5, 10$.}
\label{sup:fig:monte:DIV:0.75:2}
\end{figure}  

\begin{figure}
\centering
\includegraphics[width=0.32\textwidth]{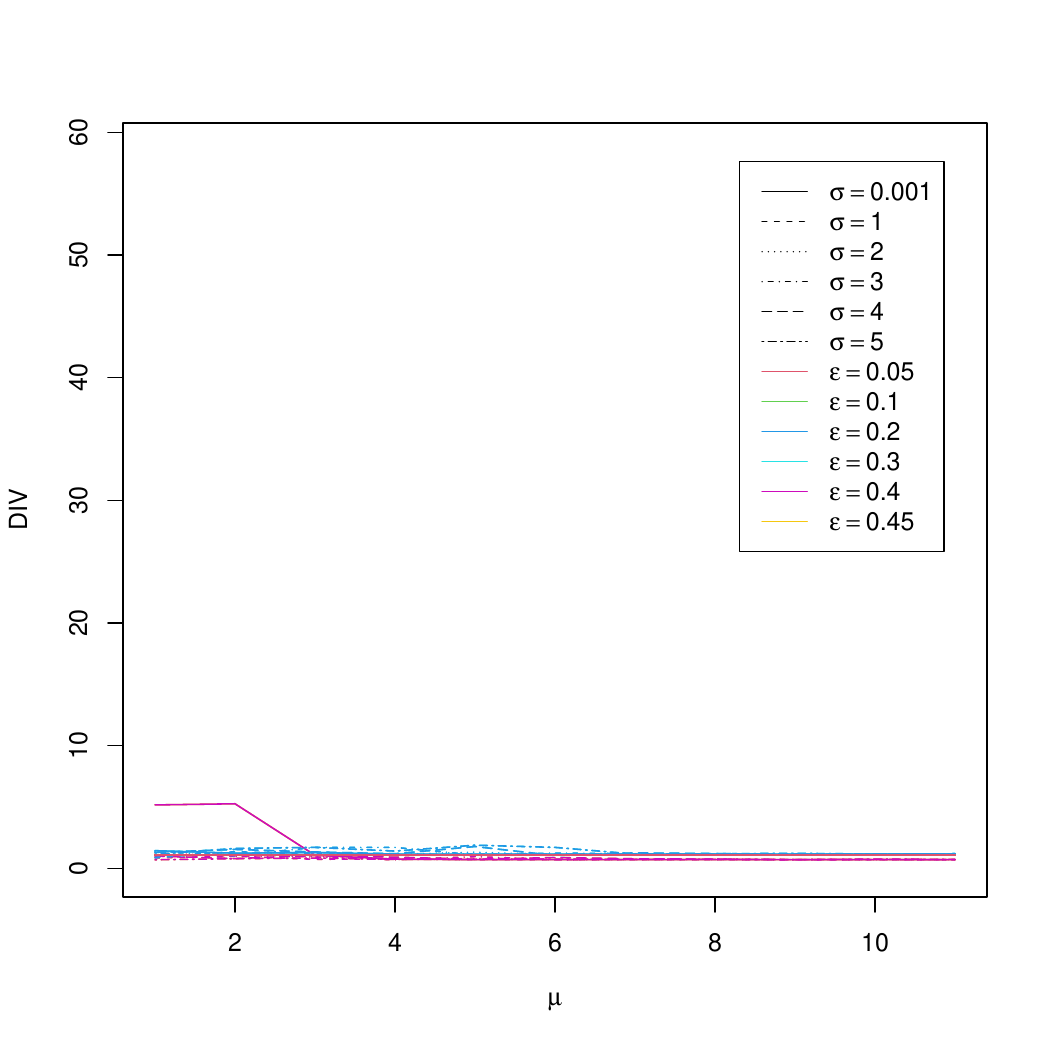}
\includegraphics[width=0.32\textwidth]{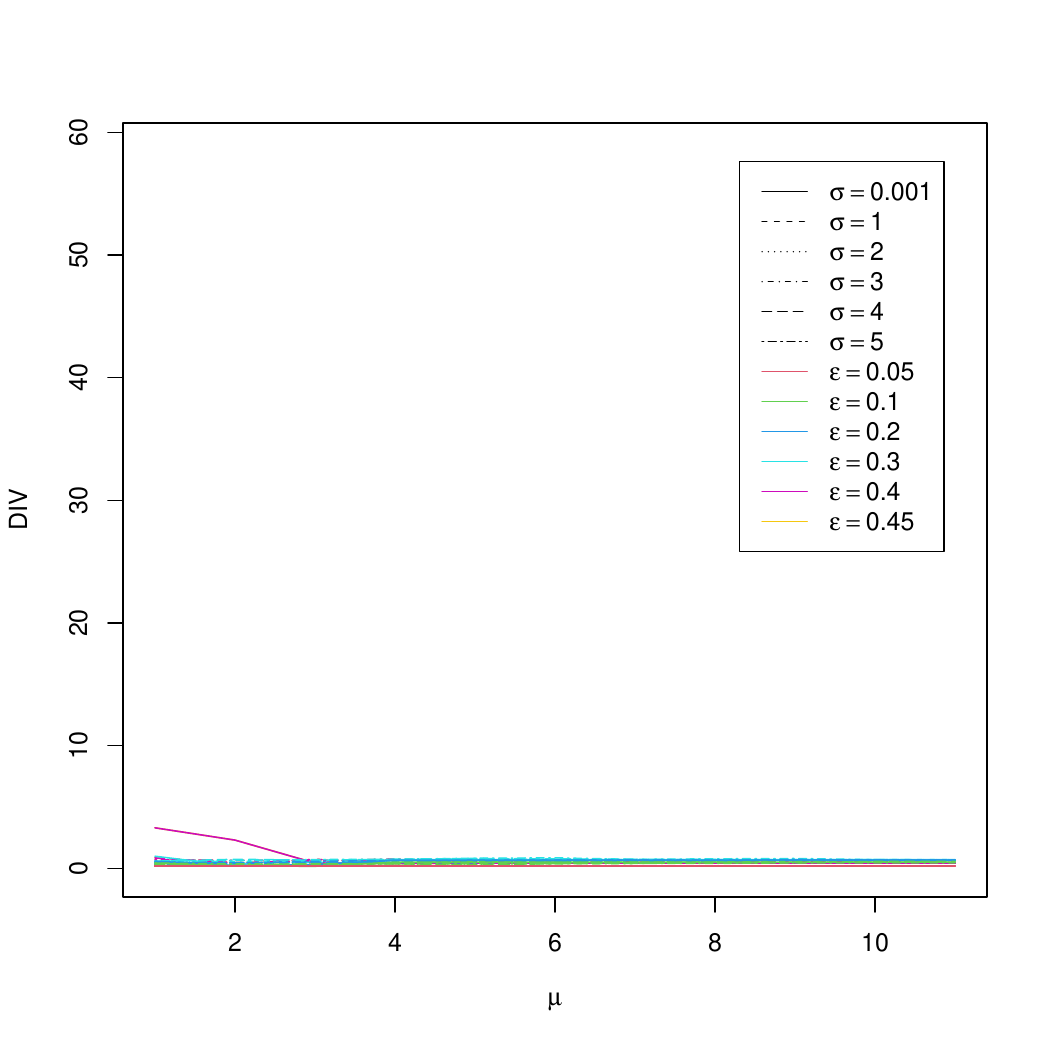} 
\includegraphics[width=0.32\textwidth]{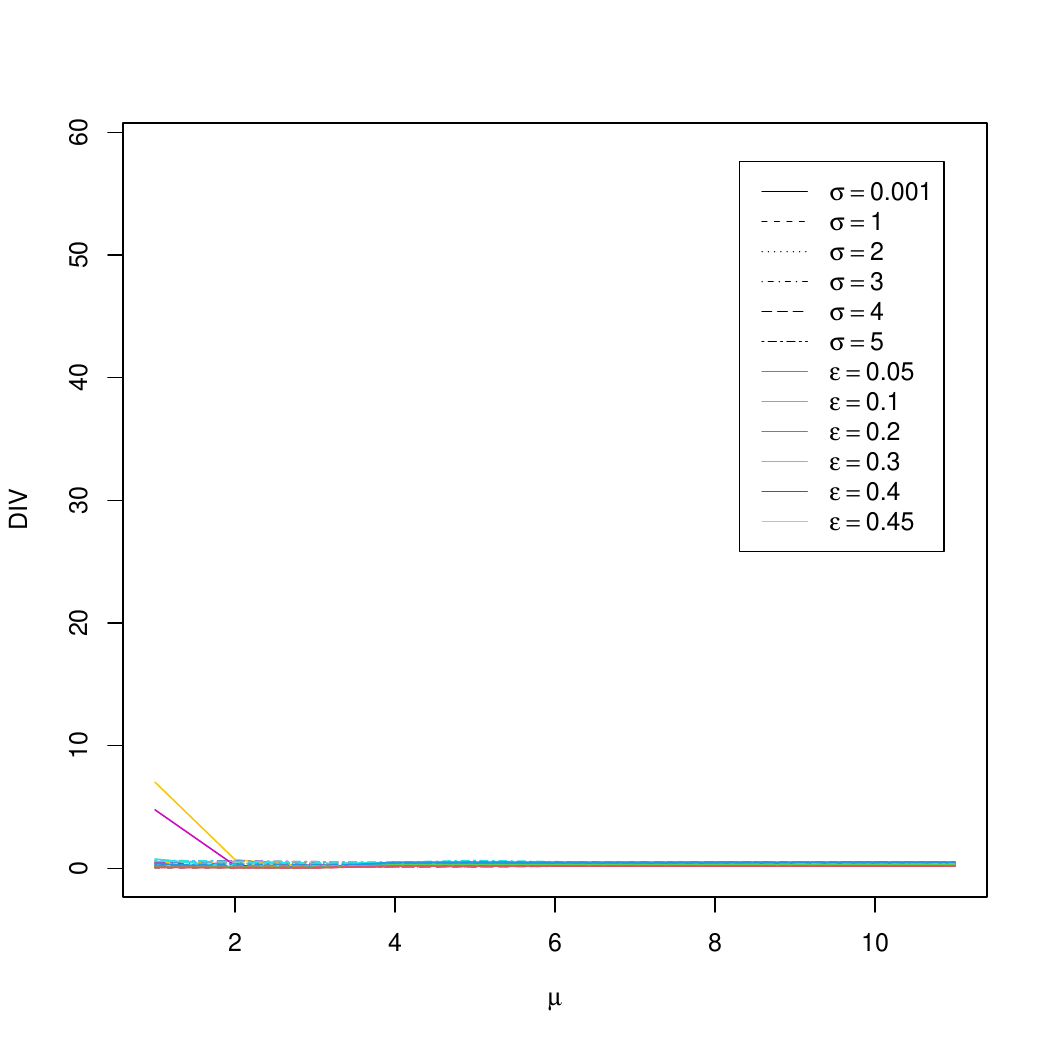} \\
\includegraphics[width=0.32\textwidth]{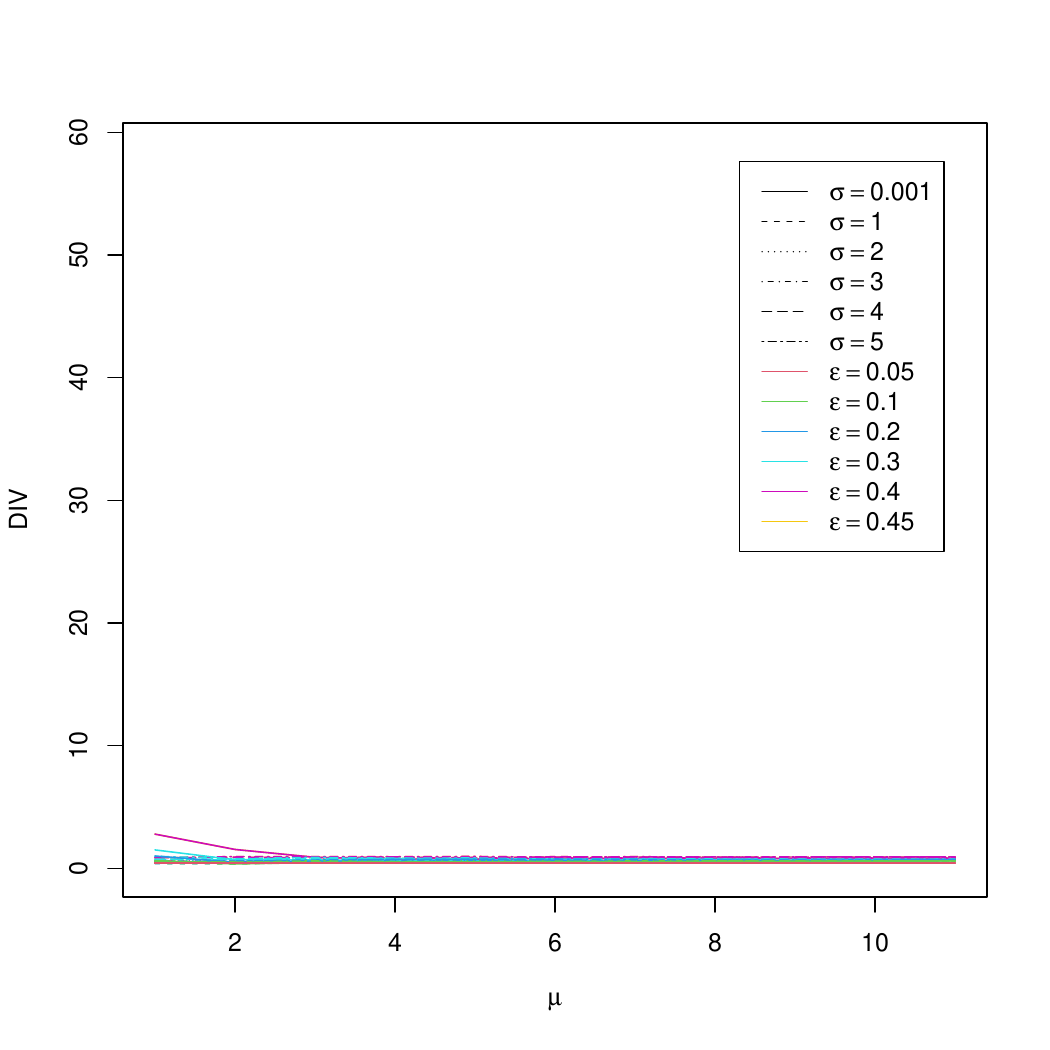}
\includegraphics[width=0.32\textwidth]{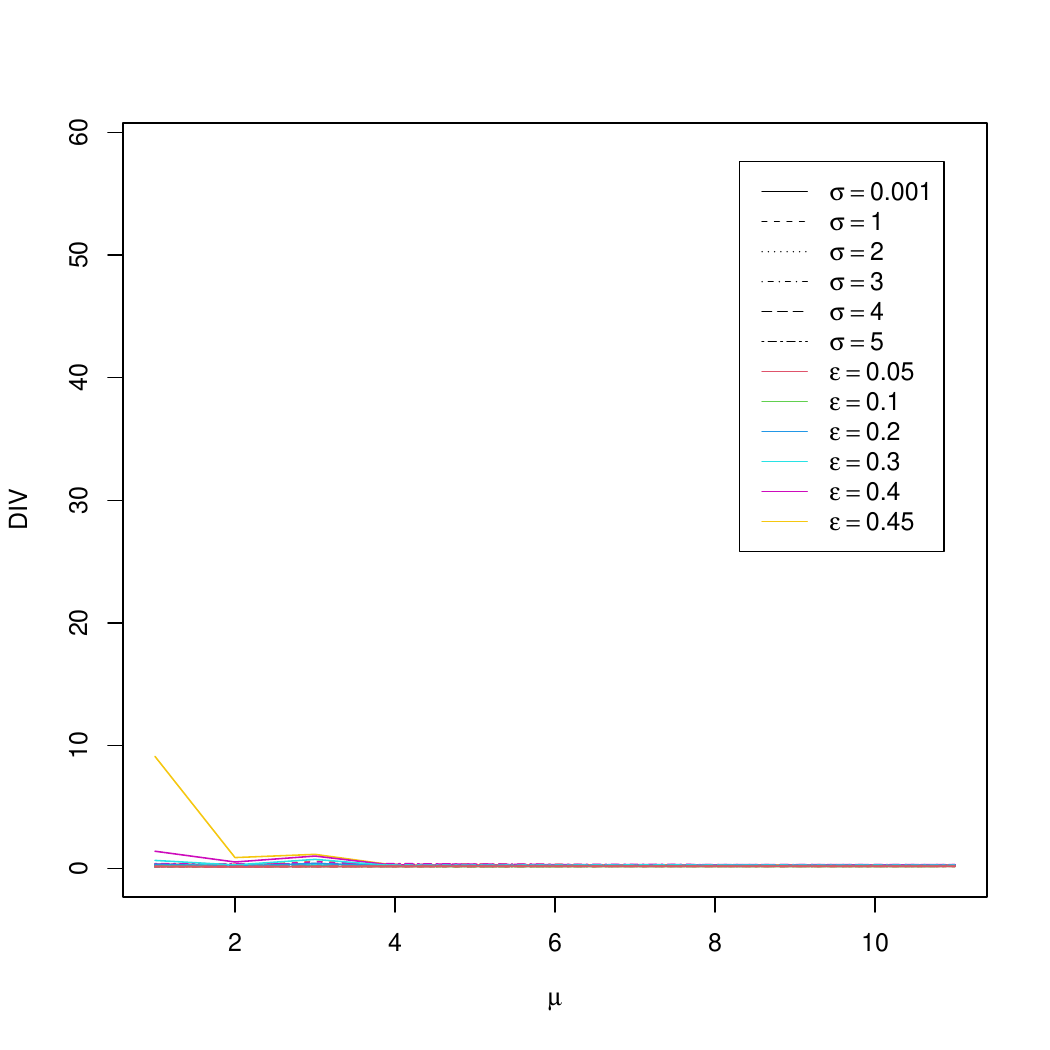} 
\includegraphics[width=0.32\textwidth]{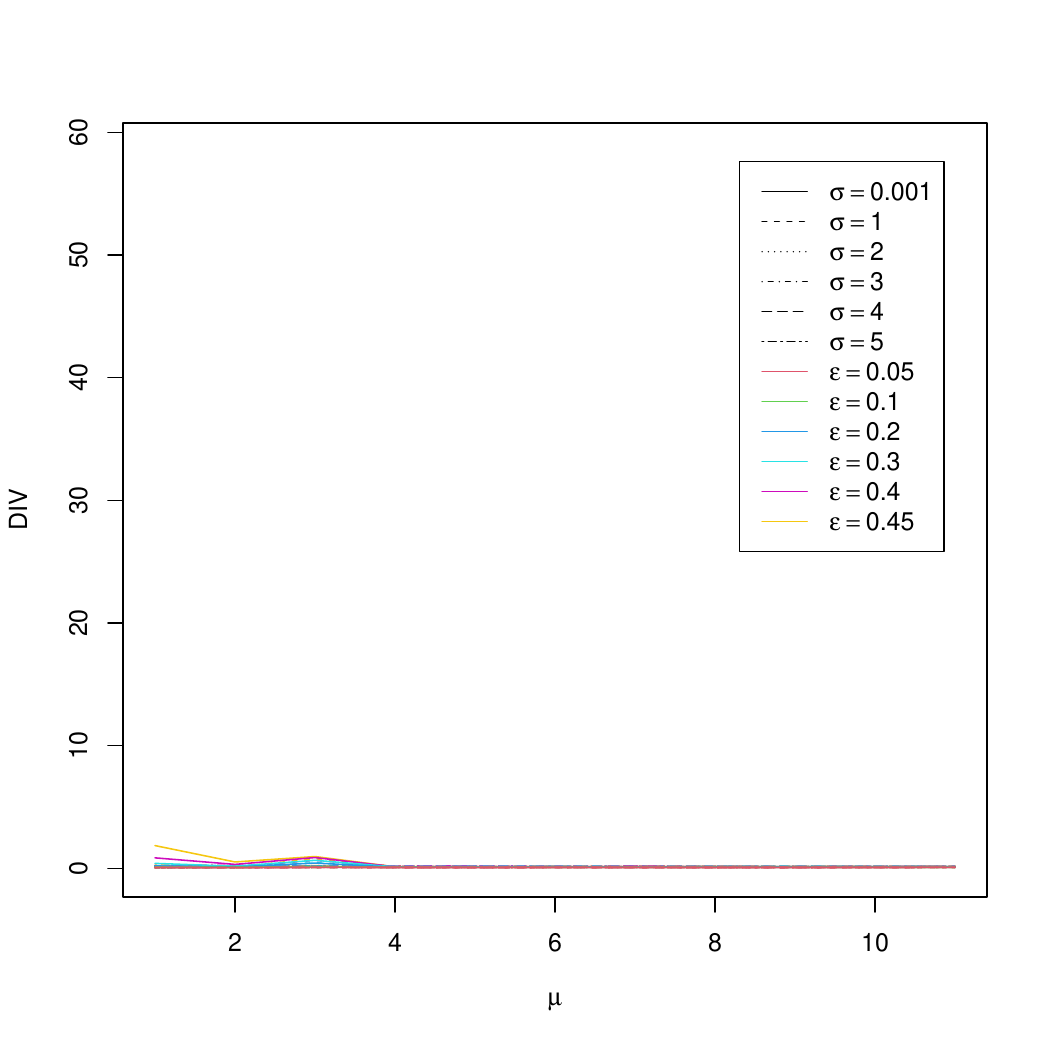} \\
\includegraphics[width=0.32\textwidth]{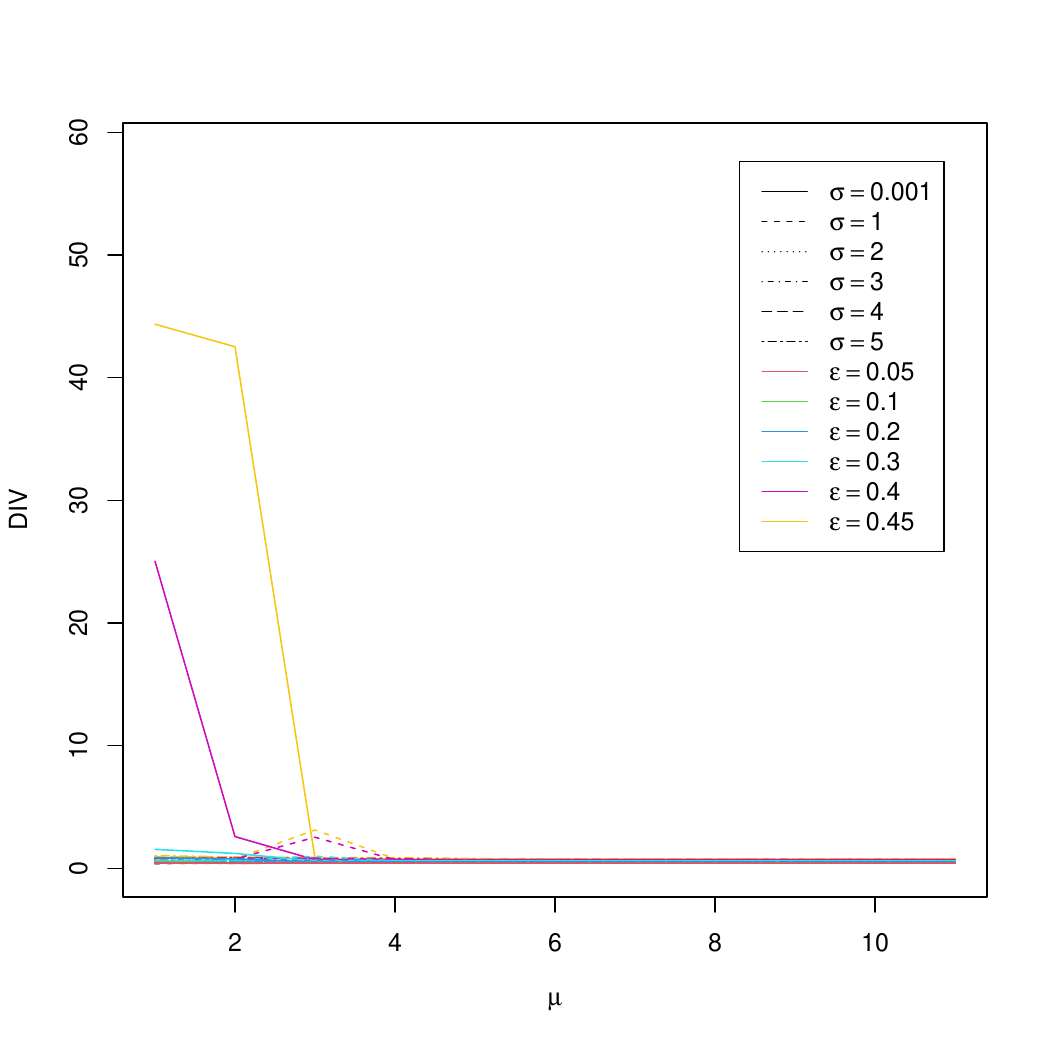}
\includegraphics[width=0.32\textwidth]{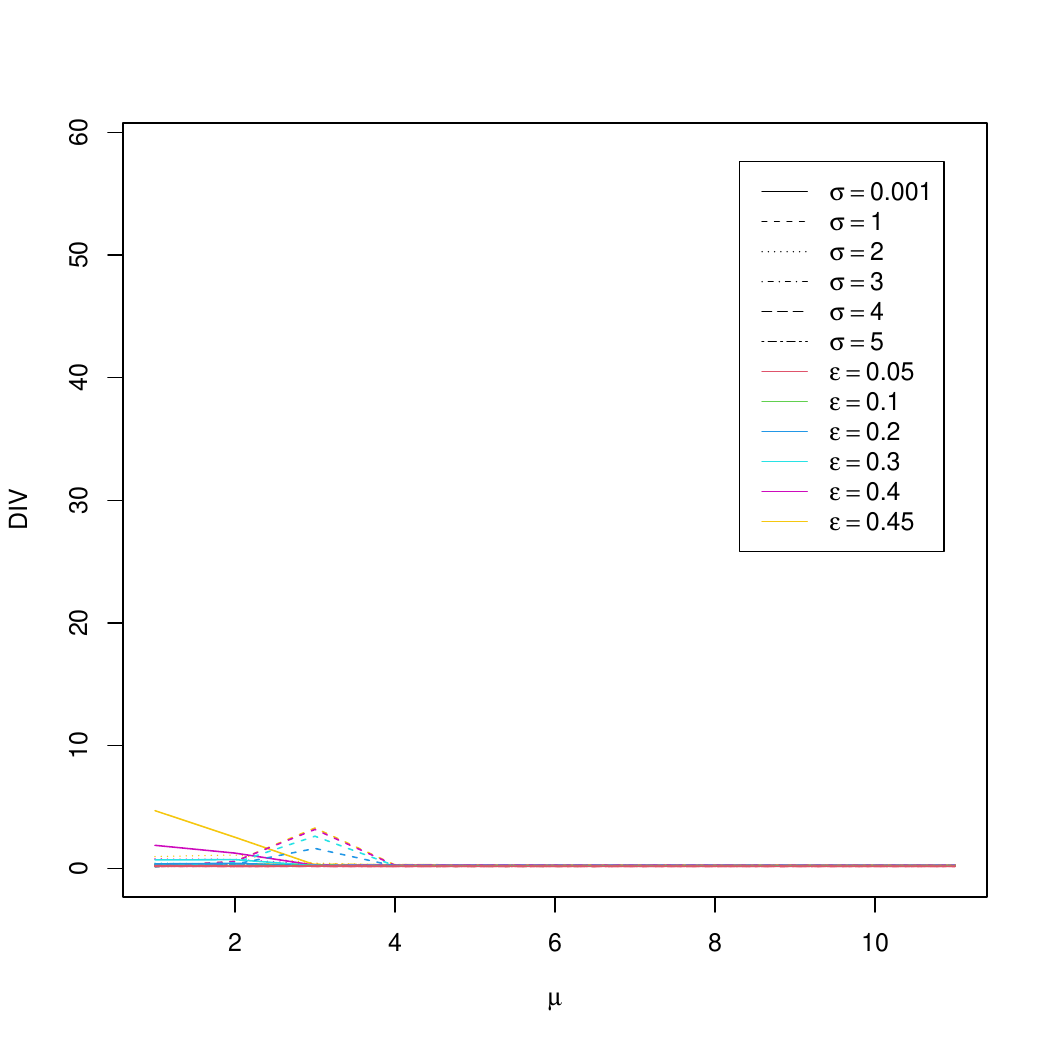} 
\includegraphics[width=0.32\textwidth]{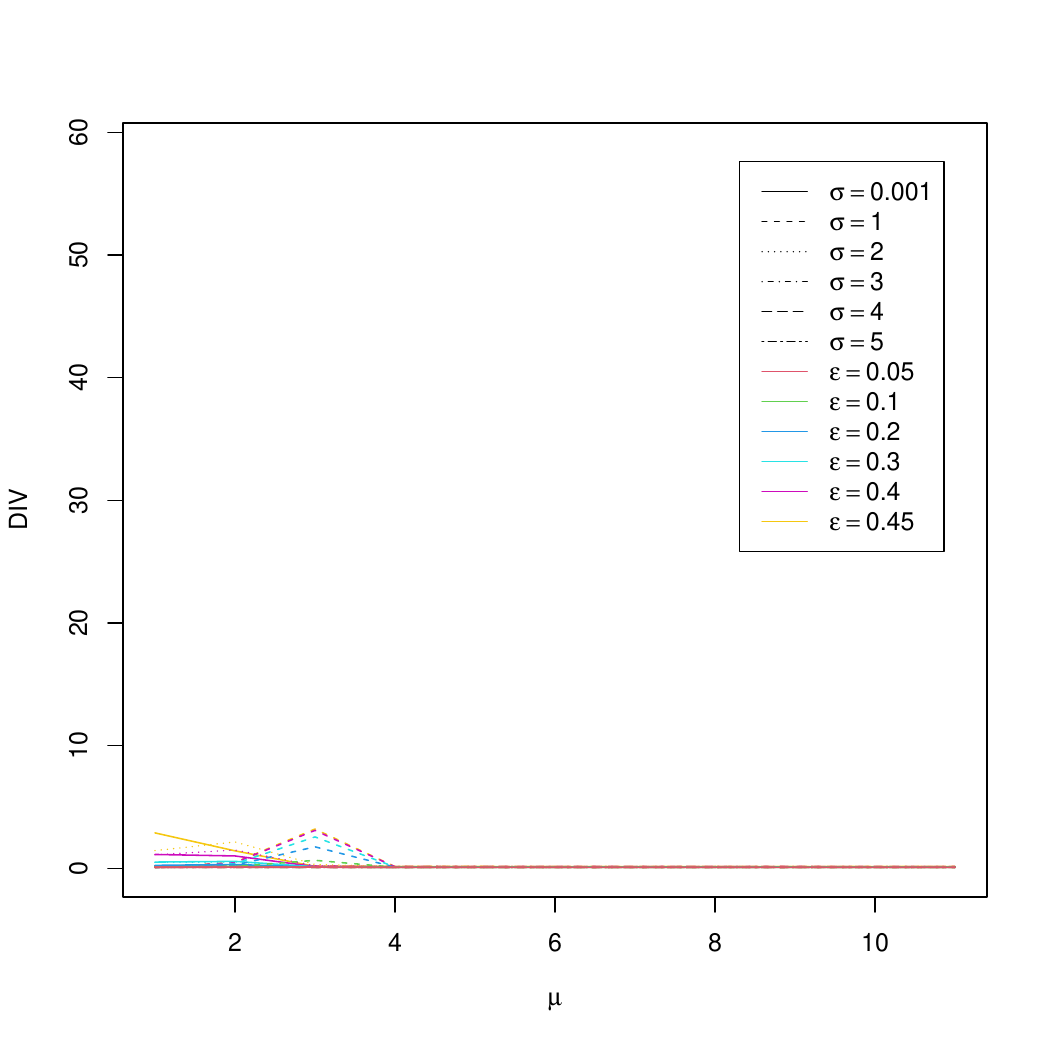} \\
\caption{Monte Carlo Simulation. Kullback--Leibler Divergence for the proposed method starting at the true values with $\alpha=1$ as a function of the contamination average $\mu$ ($x$-axis), contamination scale $\sigma$ (different line styles) and contamination level $\varepsilon$ (colors). Rows: number of variables $p=1, 2, 5$ and columns: sample size factor $s=2, 5, 10$.}
\label{sup:fig:monte:DIV:1:1}
\end{figure}  

\begin{figure}
\centering
\includegraphics[width=0.32\textwidth]{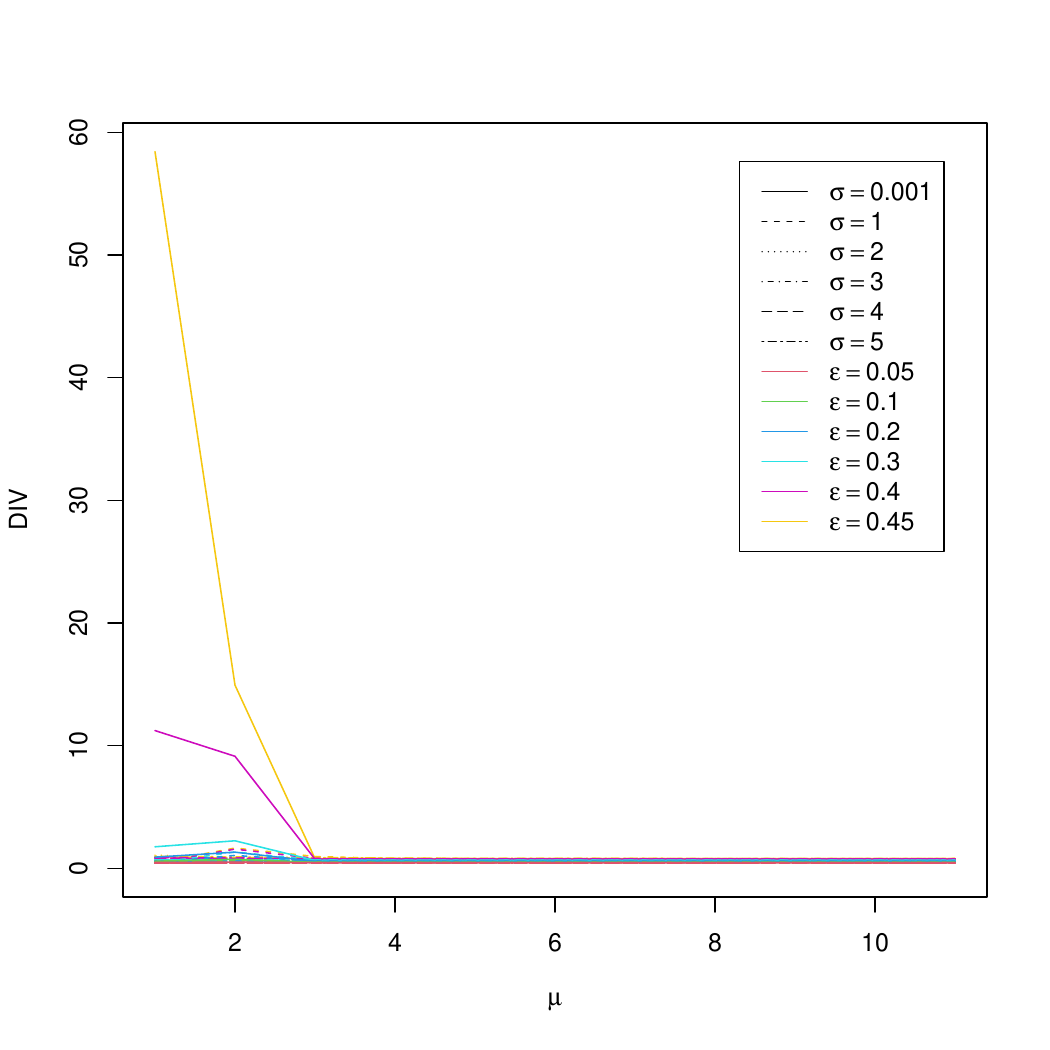}
\includegraphics[width=0.32\textwidth]{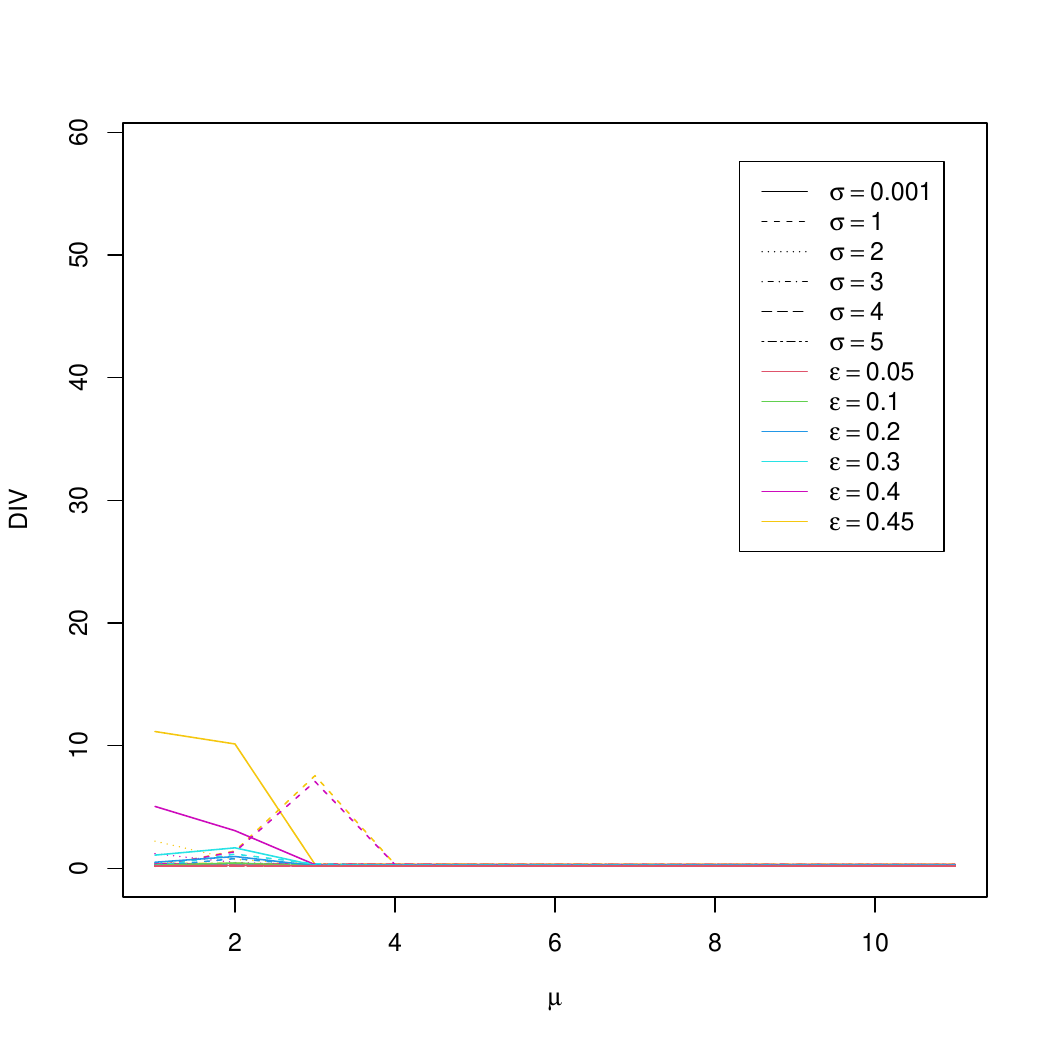} 
\includegraphics[width=0.32\textwidth]{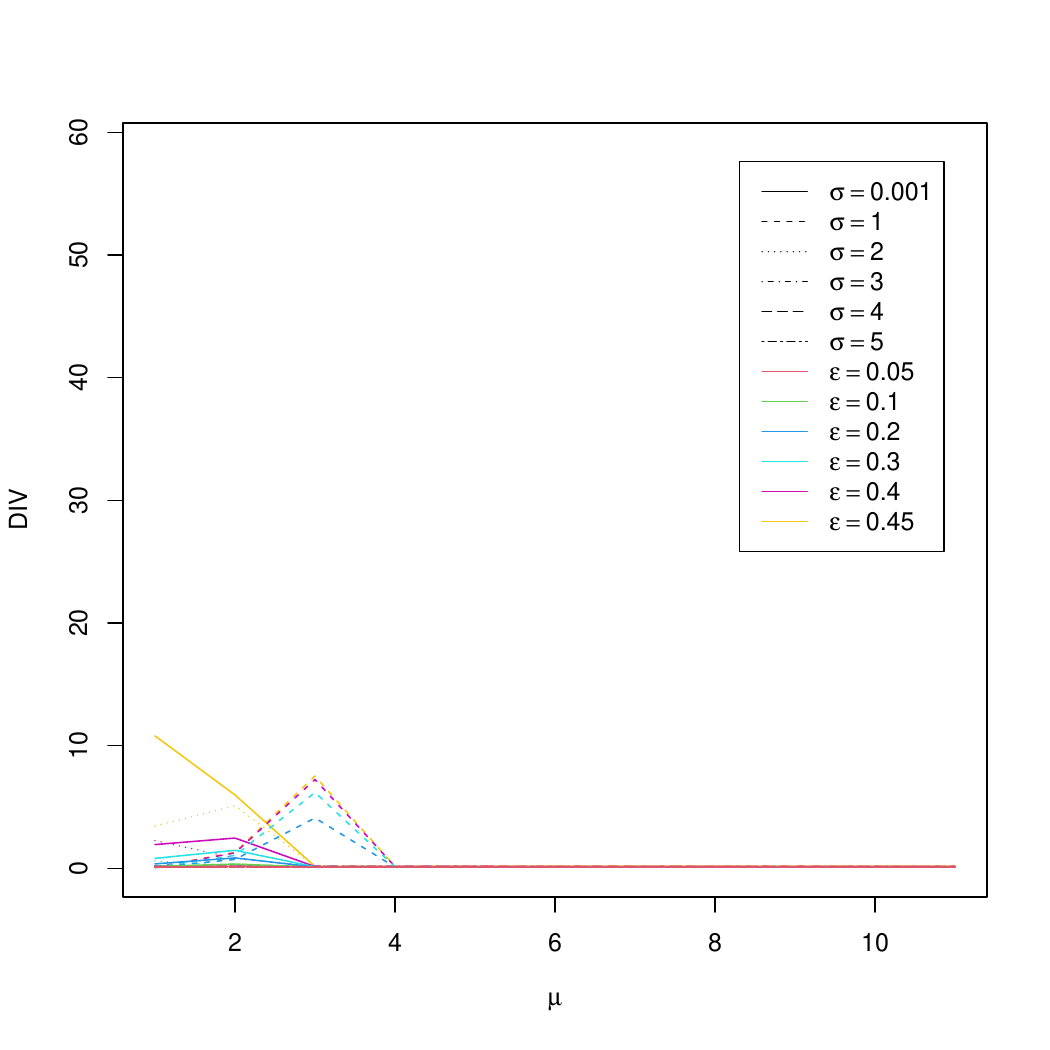} \\
\includegraphics[width=0.32\textwidth]{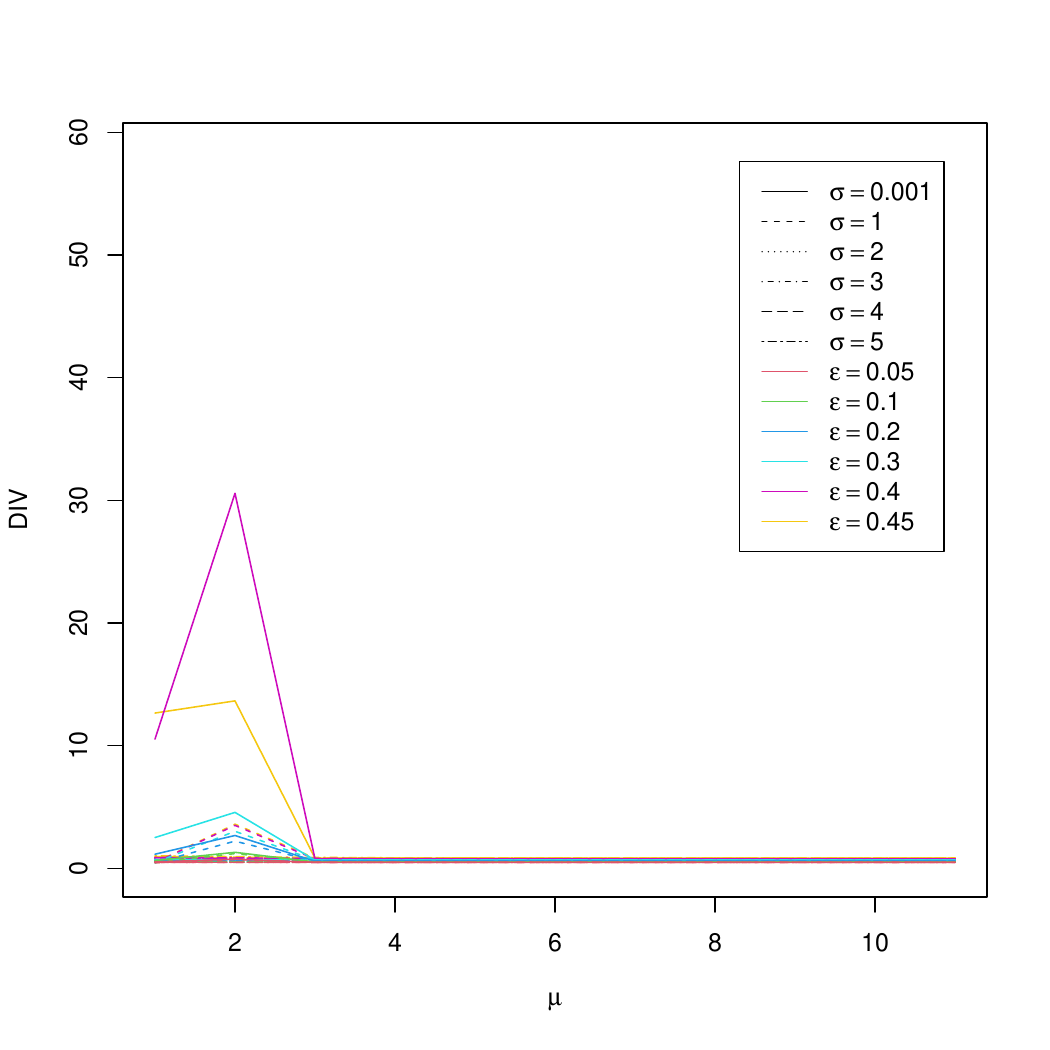}
\includegraphics[width=0.32\textwidth]{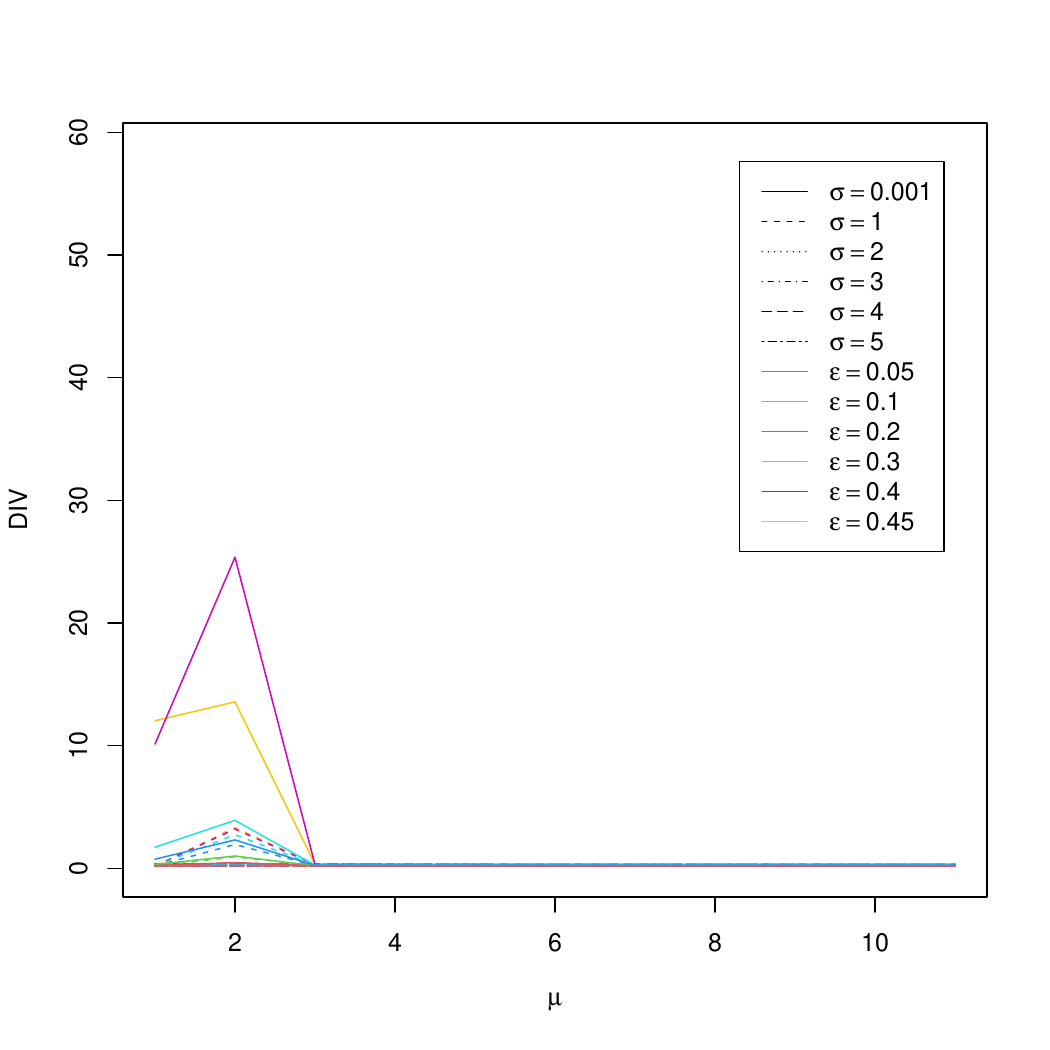} 
\includegraphics[width=0.32\textwidth]{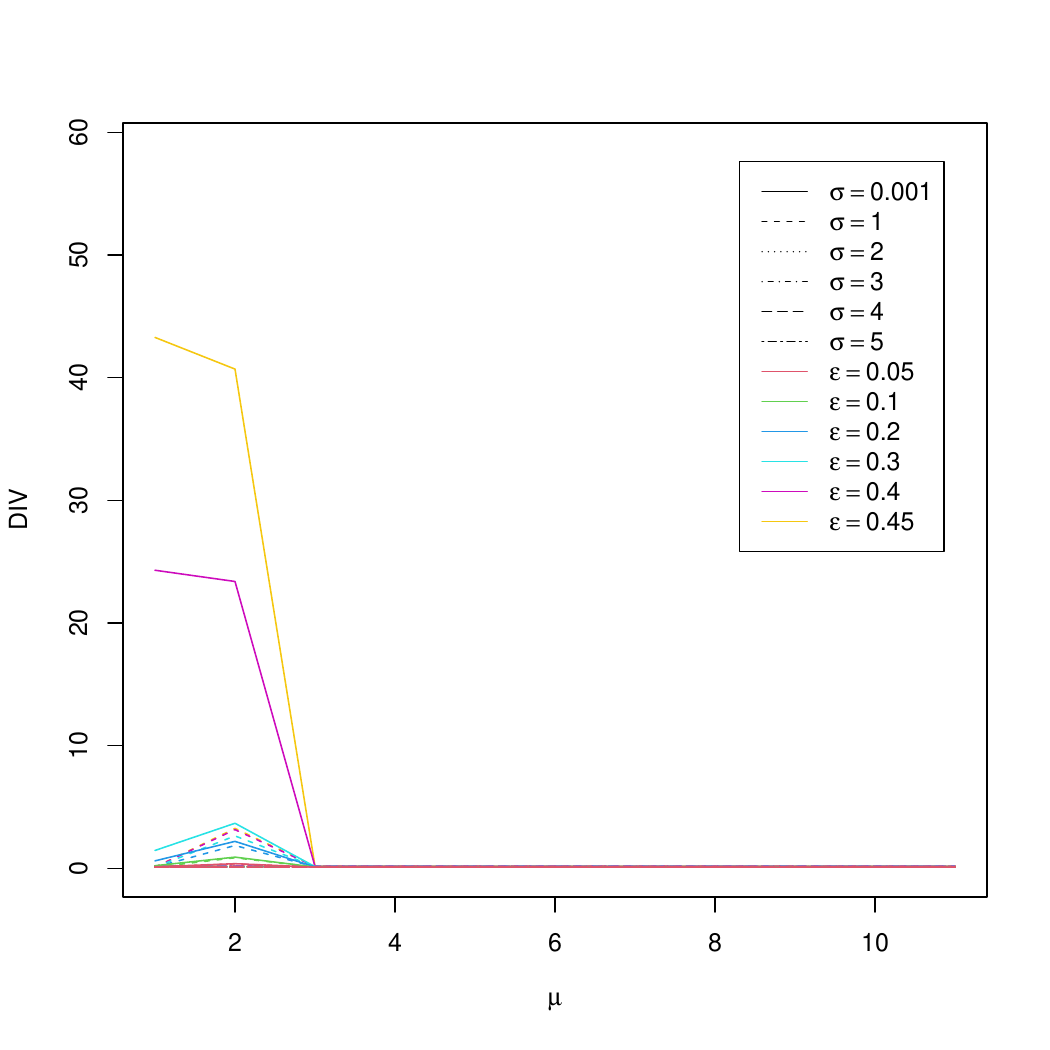}
\caption{Monte Carlo Simulation. Kullback--Leibler Divergence for the proposed method starting at the true values with $\alpha=1$ as a function of the contamination average $\mu$ ($x$-axis), contamination scale $\sigma$ (different line styles) and contamination level $\varepsilon$ (colors). Rows: number of variables $p=10, 20$ and columns: sample size factor $s=2, 5, 10$.}
\label{sup:fig:monte:DIV:1:2}
\end{figure}  

\clearpage

\begin{figure}
  \centering
  \includegraphics[width=\textwidth]{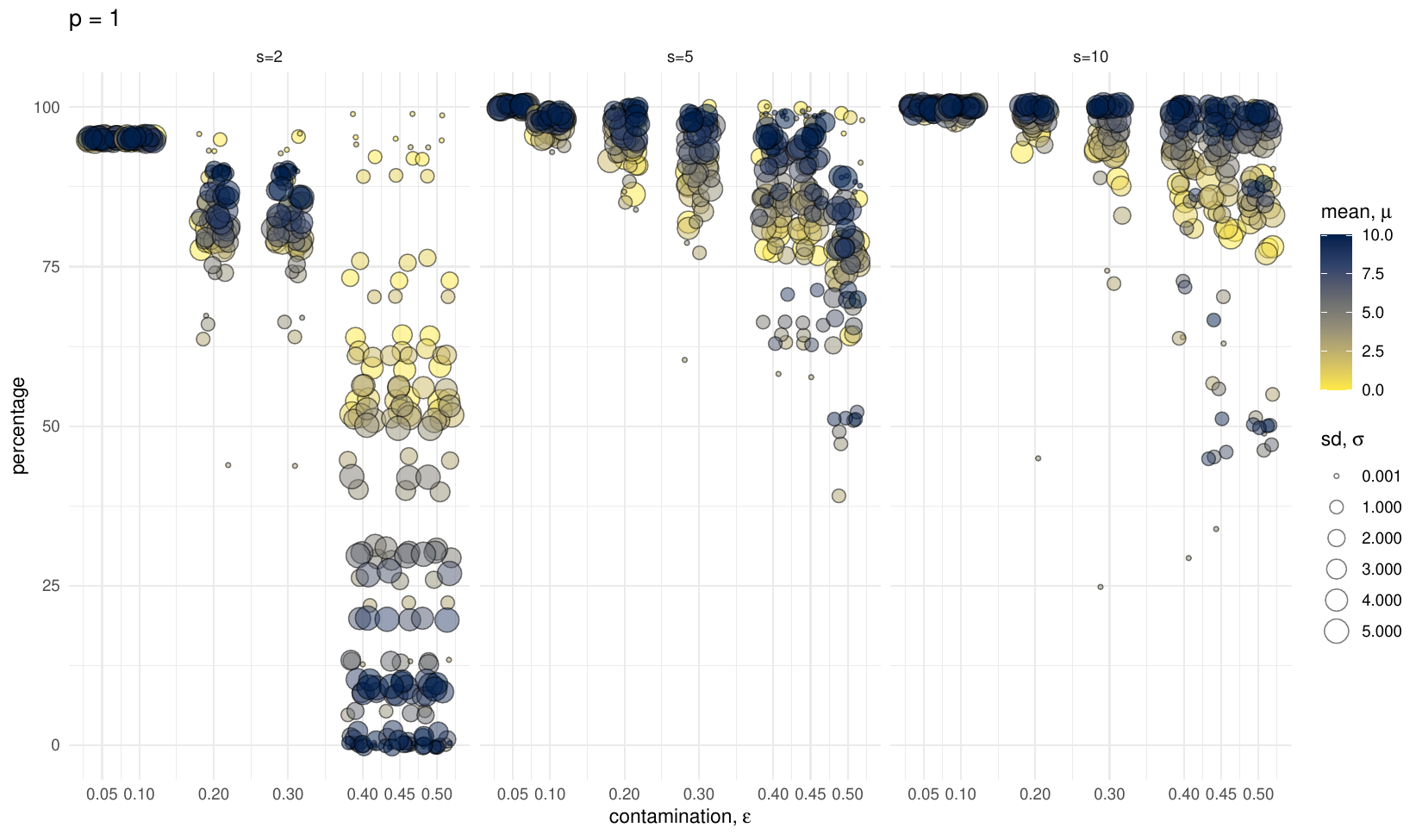}\\
  \includegraphics[width=\textwidth]{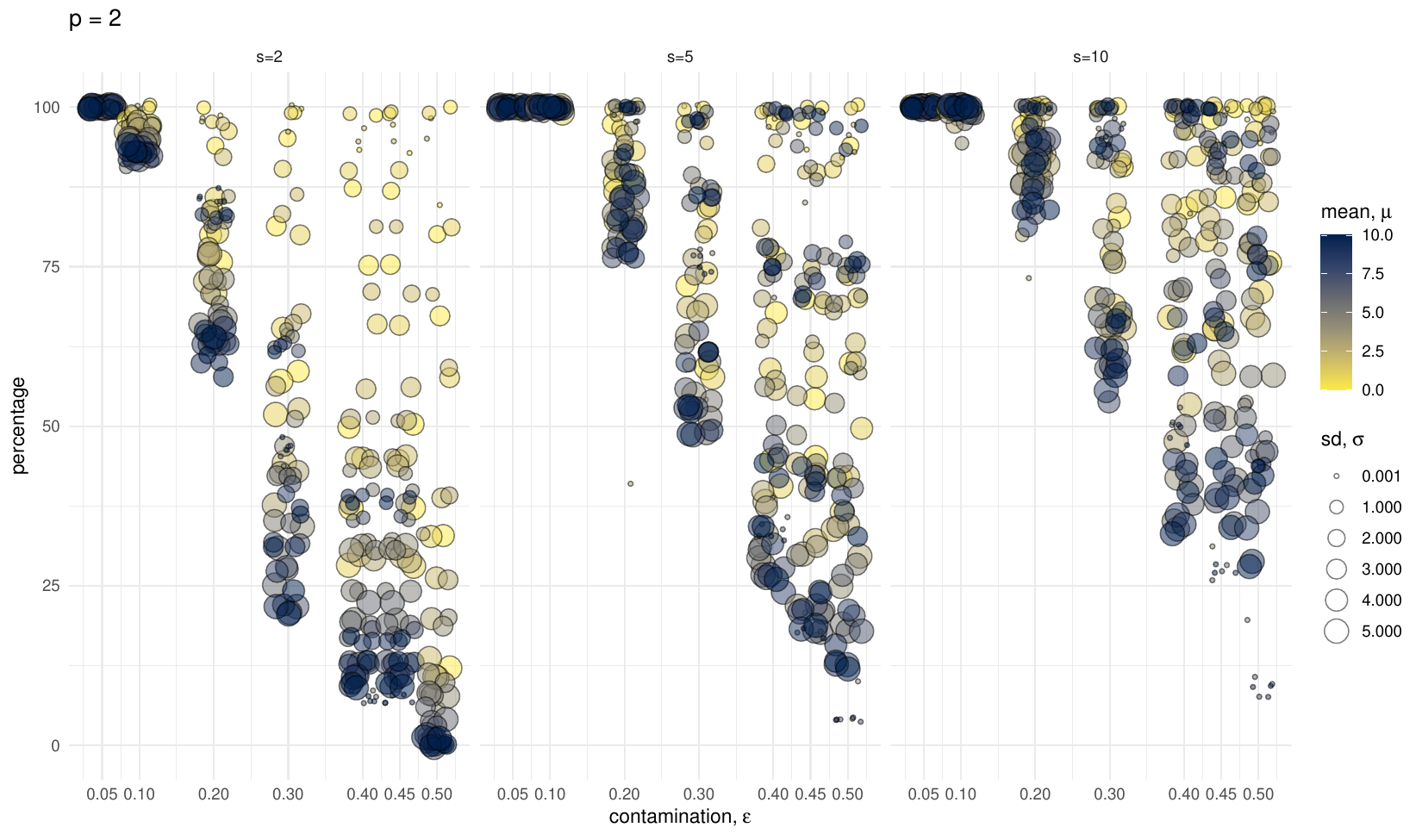} 
  \caption{Monte Carlo Simulation. Percentage of robust root retrieval out of $N=100$ simulations for the proposed method with subsampling for the initial values, as a function of contamination $\varepsilon$ ($x-$axis). Contamination average $\mu$ and scale $\sigma$ are represented by color and size of the bubbles respectively. From left to right, the subplots show the different sample size factors $s=2, 5, 10$. Number of variables $p=1$ (top) and $p=2$ (bottom). $\alpha=0.25$.}
  \label{sup:fig:subs-true-1}
\end{figure}

\begin{figure}
  \centering
  \includegraphics[width=\textwidth]{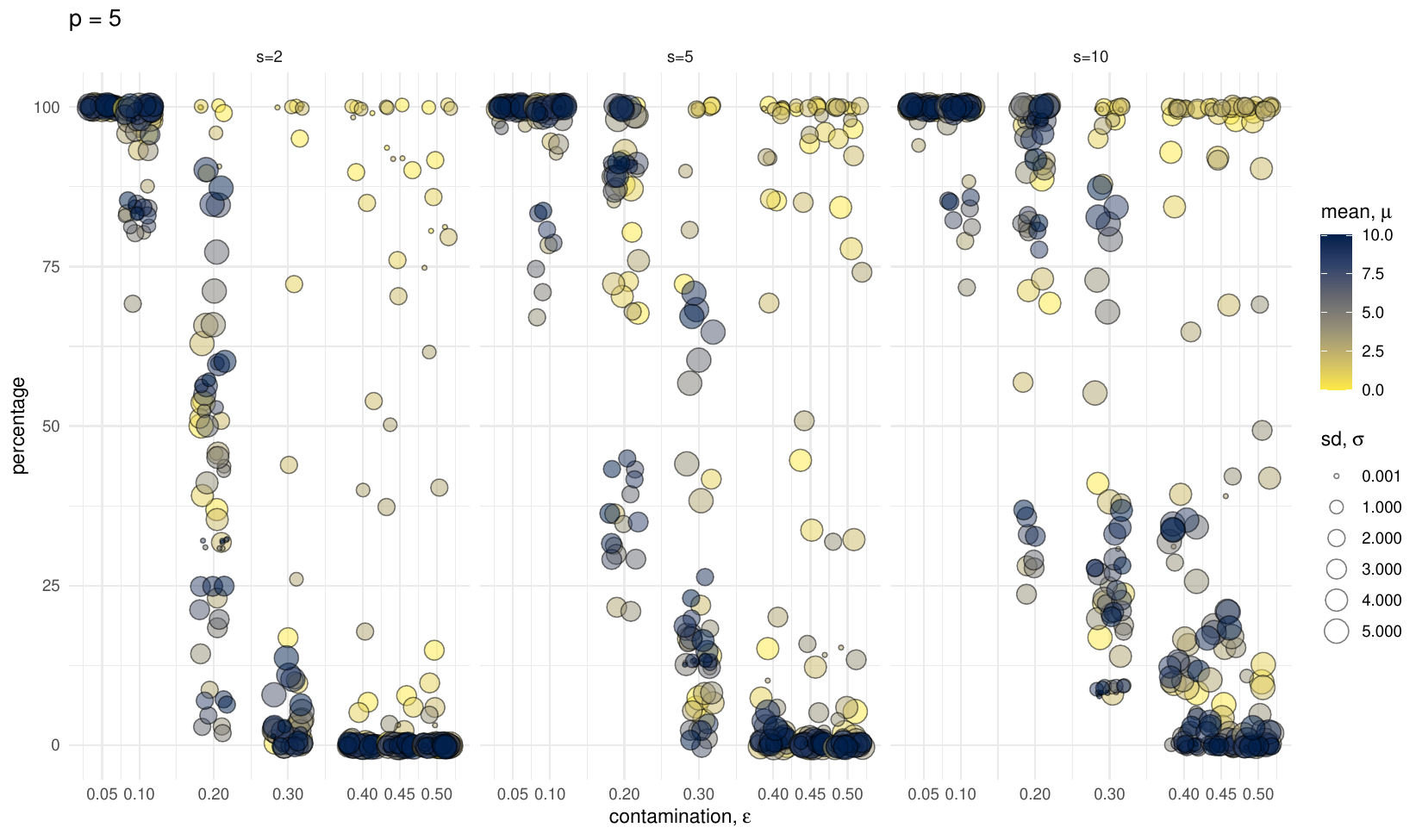}\\
  \includegraphics[width=\textwidth]{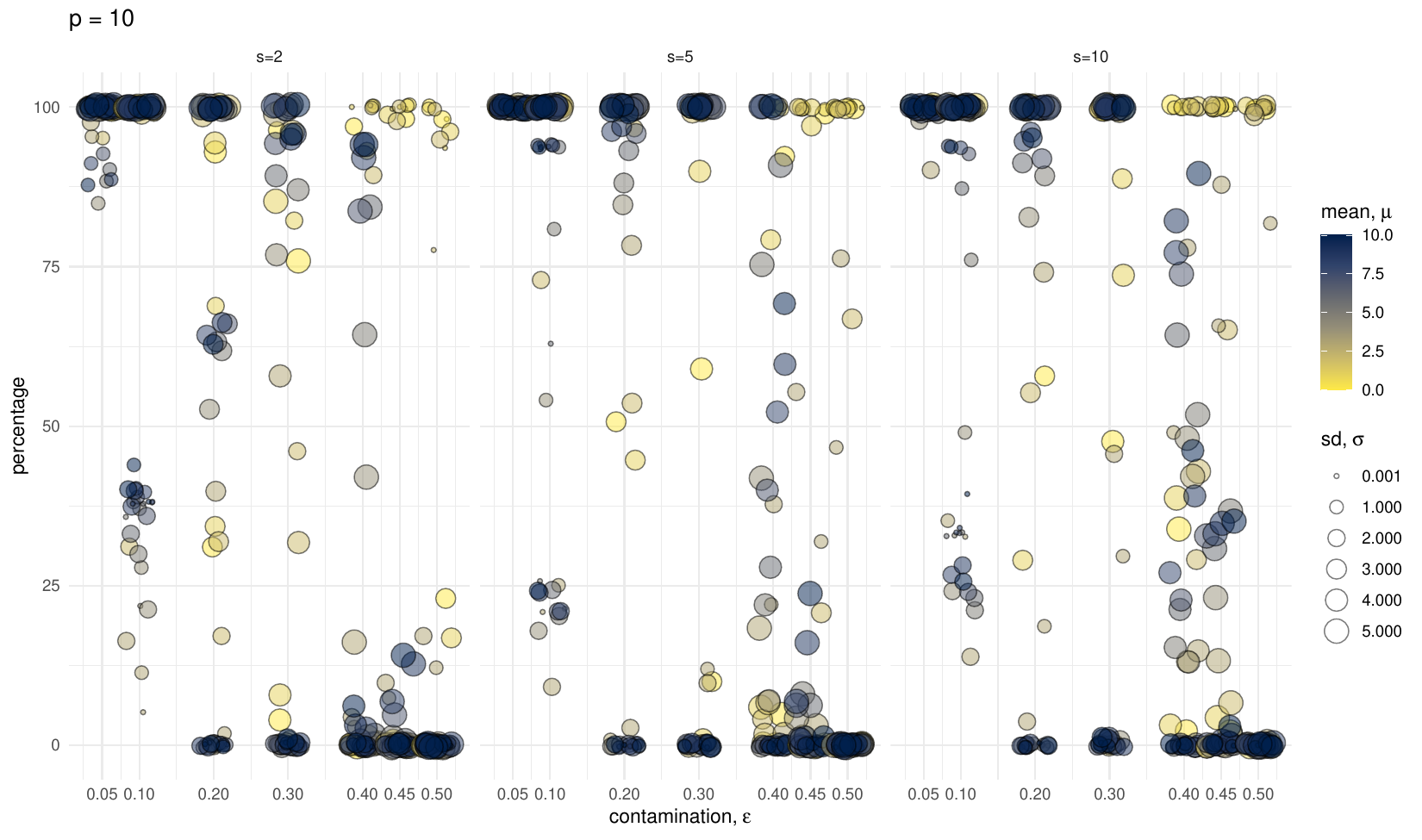}
  \caption{Monte Carlo Simulation. Percentage of robust root retrieval out of $N=100$ simulations for the proposed method with subsampling for the initial values, as a function of contamination $\varepsilon$ ($x-$axis). Contamination average $\mu$ and scale $\sigma$ are represented by color and size of the bubbles respectively. From left to right, the subplots show the different sample size factors $s=2, 5, 10$. Number of variables $p=5$ (top) and $p=10$ (bottom). $\alpha=0.25$.}
  \label{sup:fig:subs-true-2}
\end{figure}

\begin{figure}
  \includegraphics[width=\textwidth]{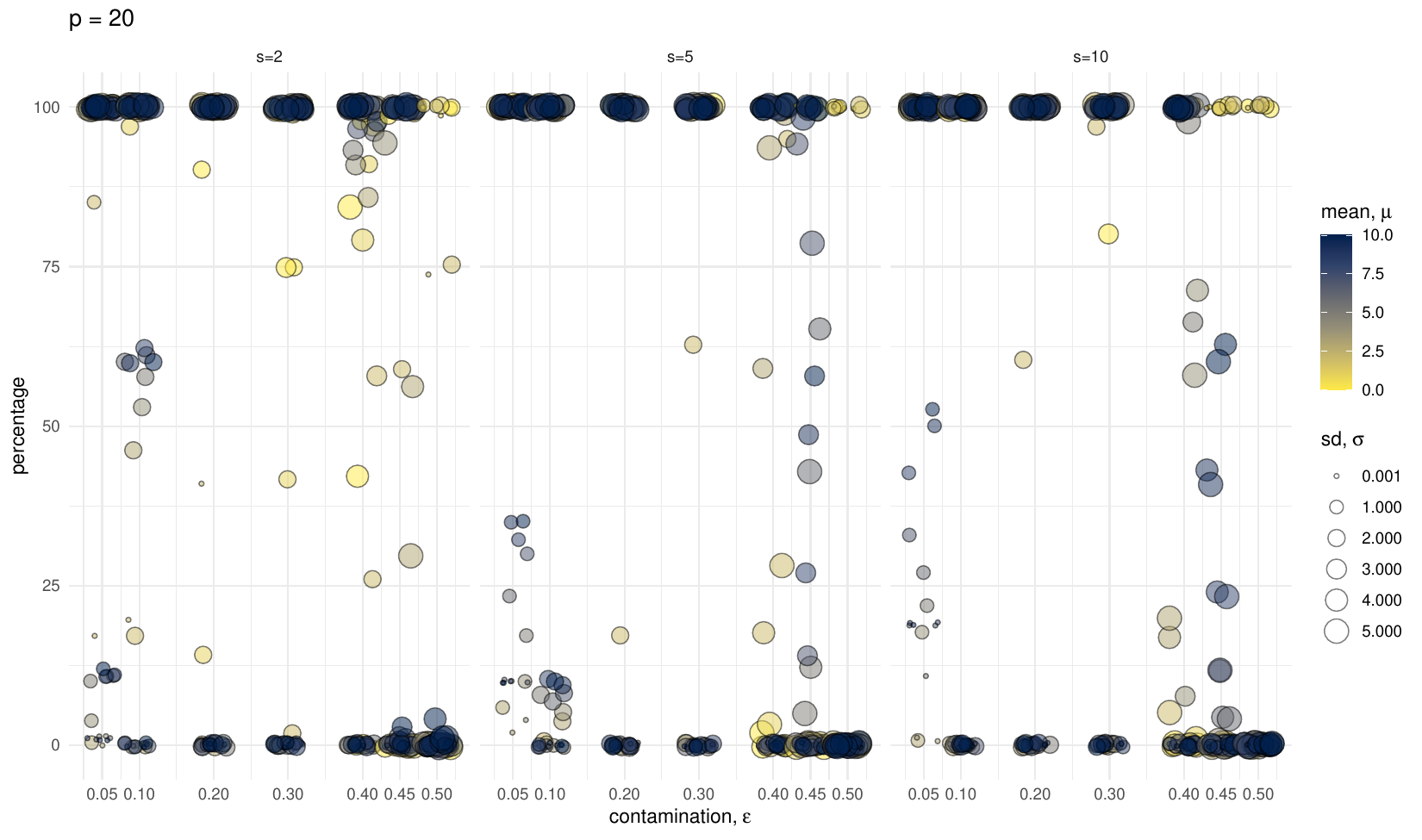}
  \caption{Monte Carlo Simulation. Percentage of robust root retrieval out of $N=100$ simulations for the proposed method with subsampling for the initial values, as a function of contamination $\varepsilon$ ($x-$axis). Contamination average $\mu$ and scale $\sigma$ are represented by color and size of the bubbles respectively. From left to right, the subplots show the different sample size factors $s=2, 5, 10$. Number of variables $p=20$. $\alpha=0.25$.}
  \label{sup:fig:subs-true-2b}
\end{figure}

\begin{figure}
  \centering
  \includegraphics[width=\textwidth]{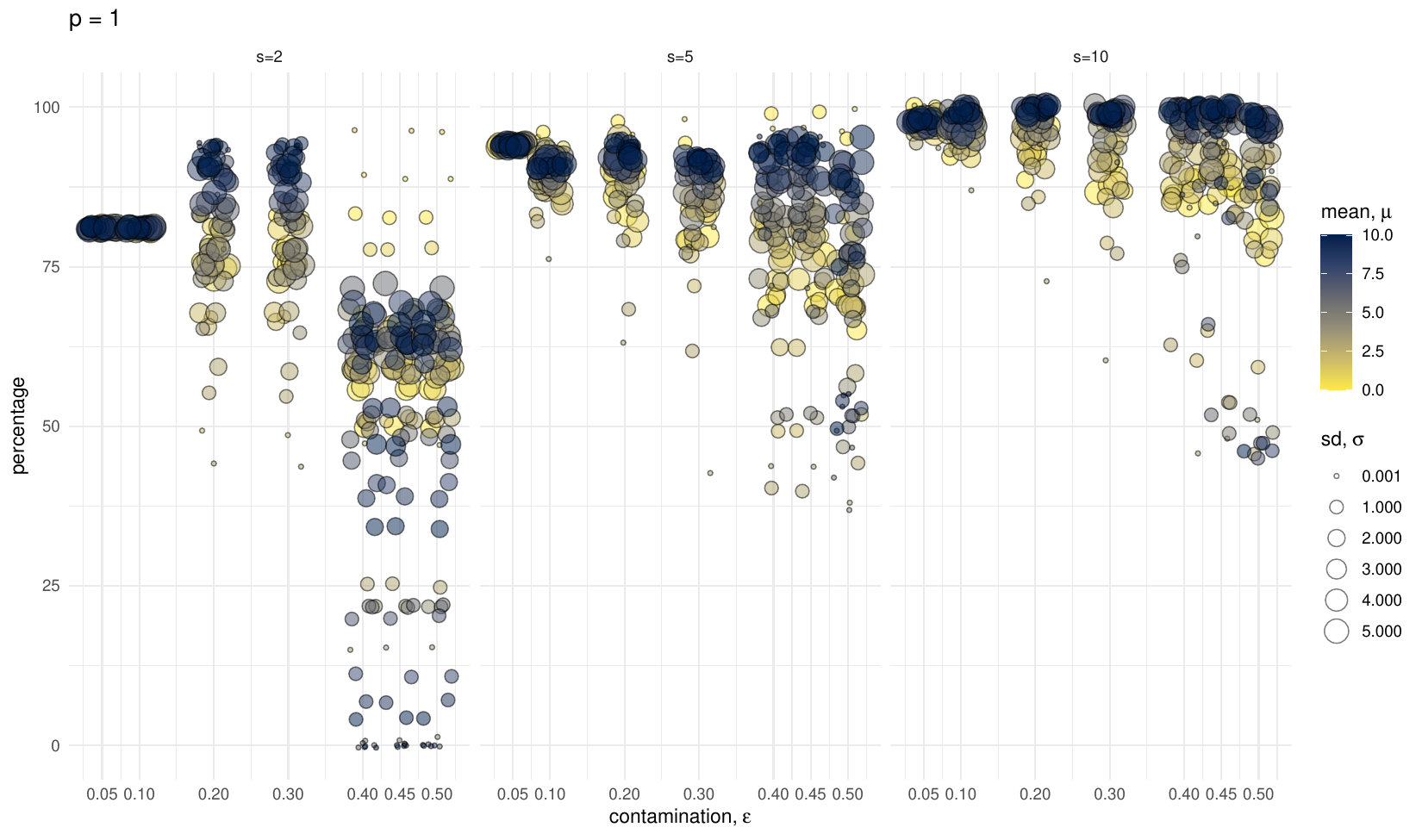}\\
  \includegraphics[width=\textwidth]{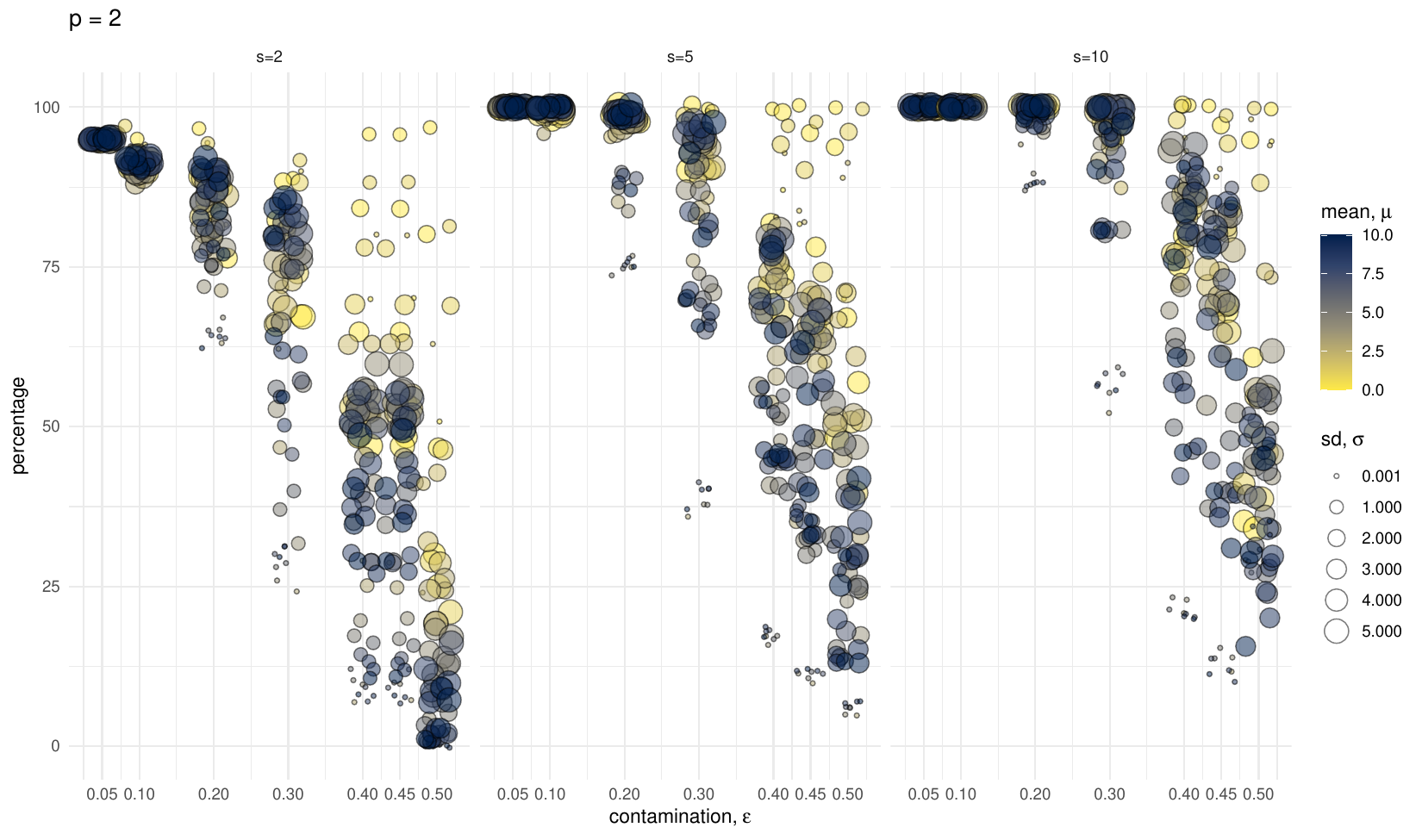} 
  \caption{Monte Carlo Simulation. Percentage of robust root retrieval out of $N=100$ simulations for the proposed method with subsampling for the initial values, as a function of contamination $\varepsilon$ ($x-$axis). Contamination average $\mu$ and scale $\sigma$ are represented by color and size of the bubbles respectively. From left to right, the subplots show the different sample size factors $s=2, 5, 10$. Number of variables $p=1$ (top) and $p=2$ (bottom). $\alpha=0.5$.}
  \label{sup:fig:subs-true-3}
\end{figure}

\begin{figure}
  \centering
  \includegraphics[width=\textwidth]{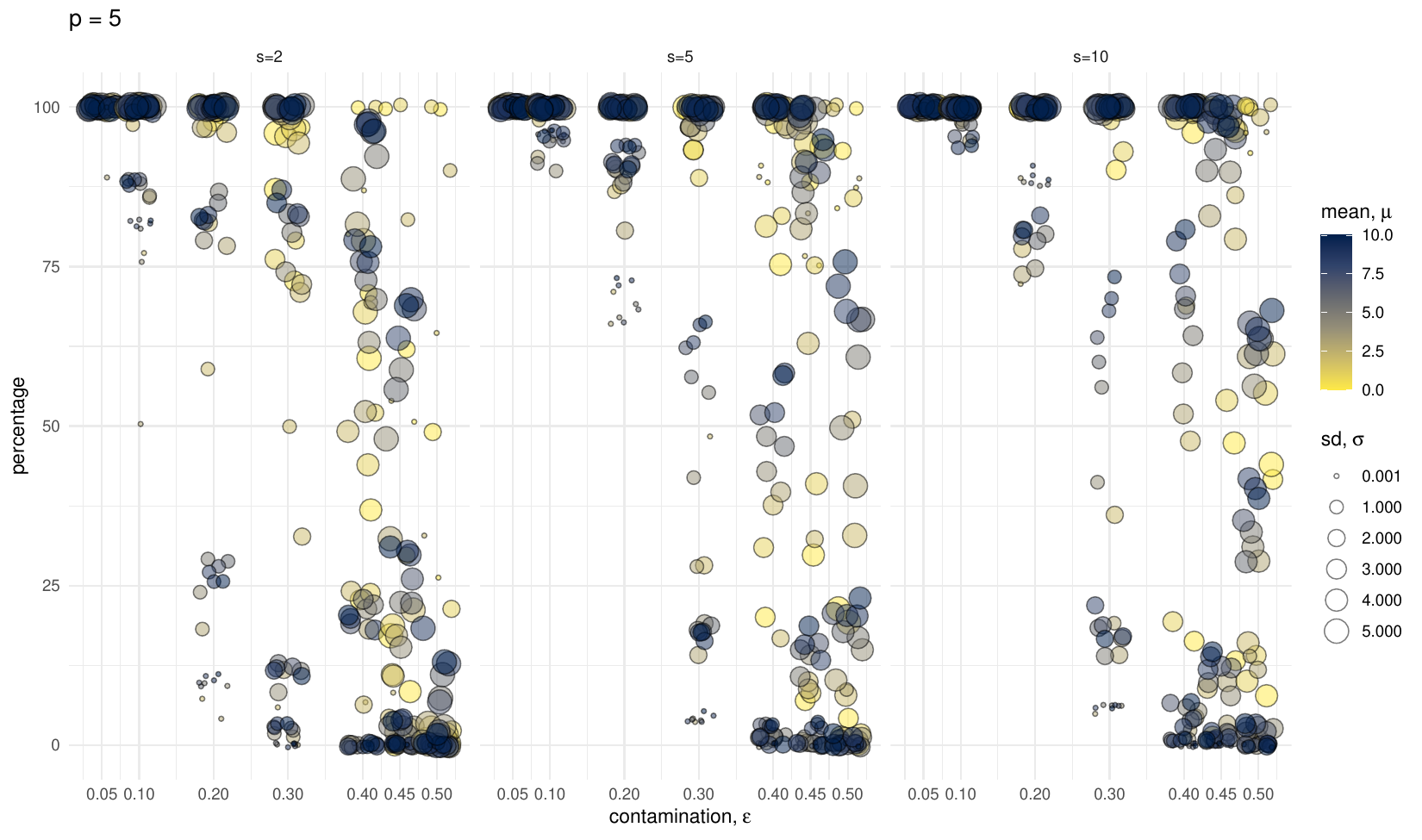}\\
  \includegraphics[width=\textwidth]{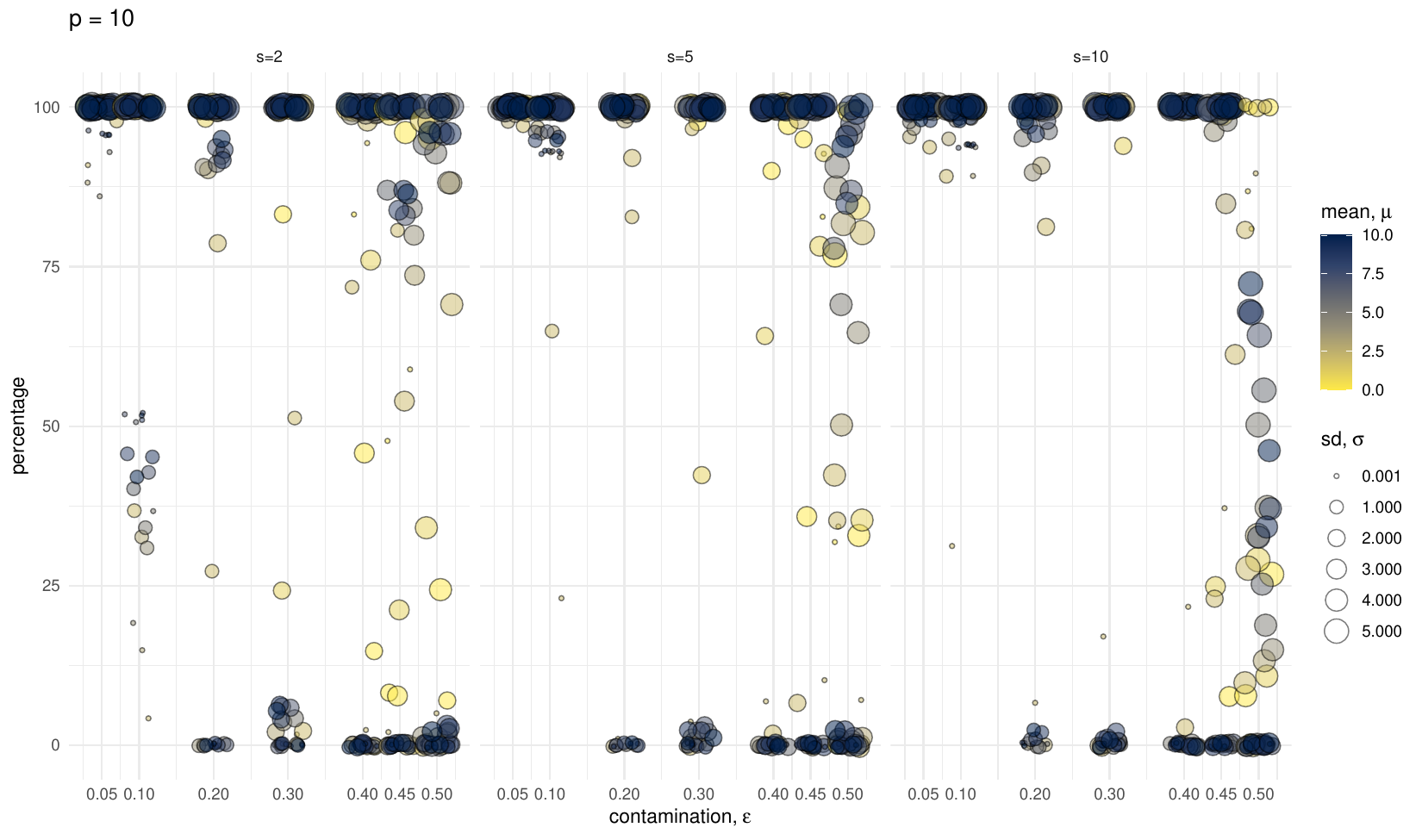}
  \caption{Monte Carlo Simulation. Percentage of robust root retrieval out of $N=100$ simulations for the proposed method with subsampling for the initial values, as a function of contamination $\varepsilon$ ($x-$axis). Contamination average $\mu$ and scale $\sigma$ are represented by color and size of the bubbles respectively. From left to right, the subplots show the different sample size factors $s=2, 5, 10$. Number of variables $p=5$ (top) and $p=10$ (bottom). $\alpha=0.5$.}
  \label{sup:fig:subs-true-4}
\end{figure}

\begin{figure}
  \includegraphics[width=\textwidth]{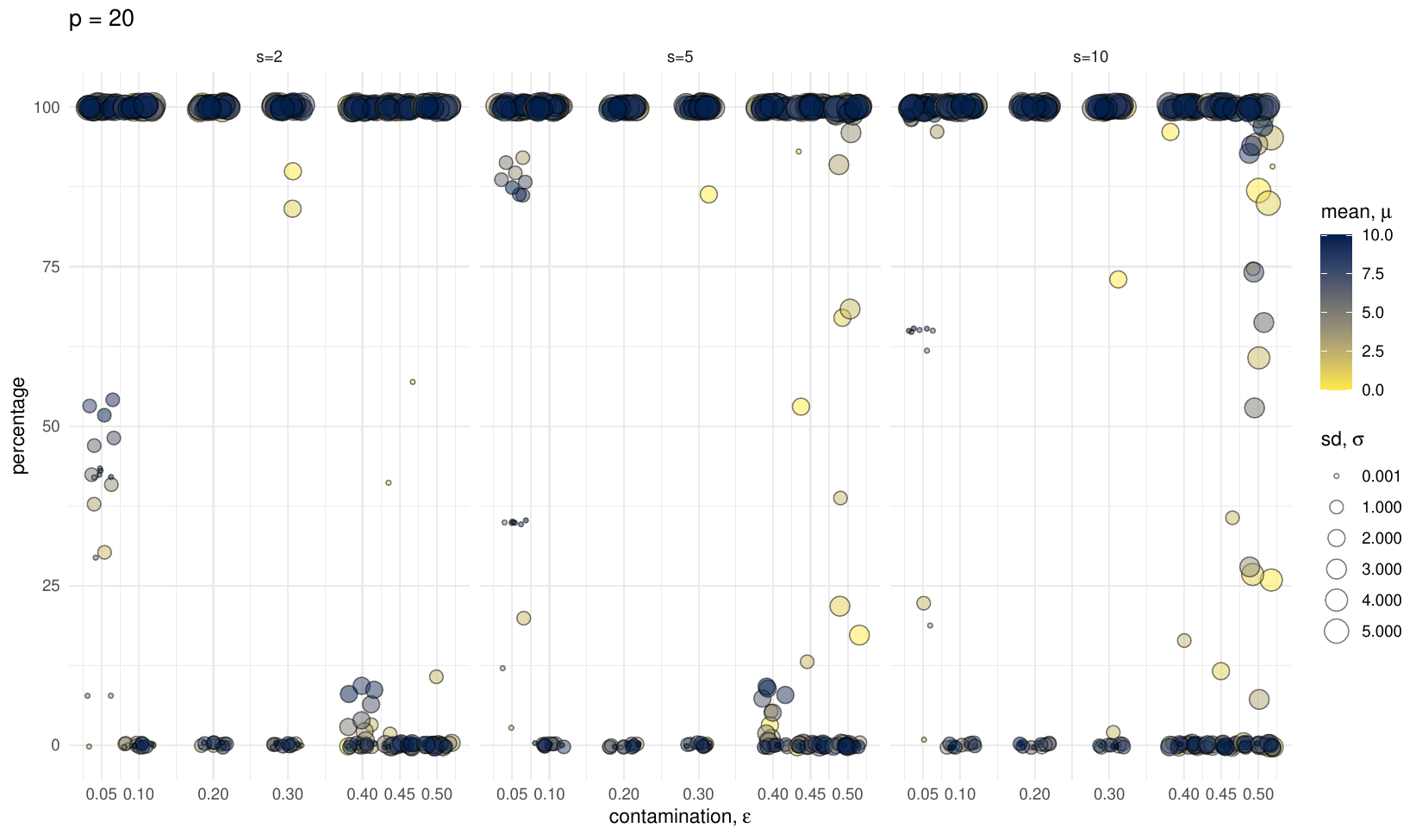}
  \caption{Monte Carlo Simulation. Percentage of robust root retrieval out of $N=100$ simulations for the proposed method with subsampling for the initial values, as a function of contamination $\varepsilon$ ($x-$axis). Contamination average $\mu$ and scale $\sigma$ are represented by color and size of the bubbles respectively. From left to right, the subplots show the different sample size factors $s=2, 5, 10$. Number of variables $p=20$. $\alpha=0.5$.}
  \label{sup:fig:subs-true-4b}
\end{figure}

\begin{figure}
  \centering
  \includegraphics[width=\textwidth]{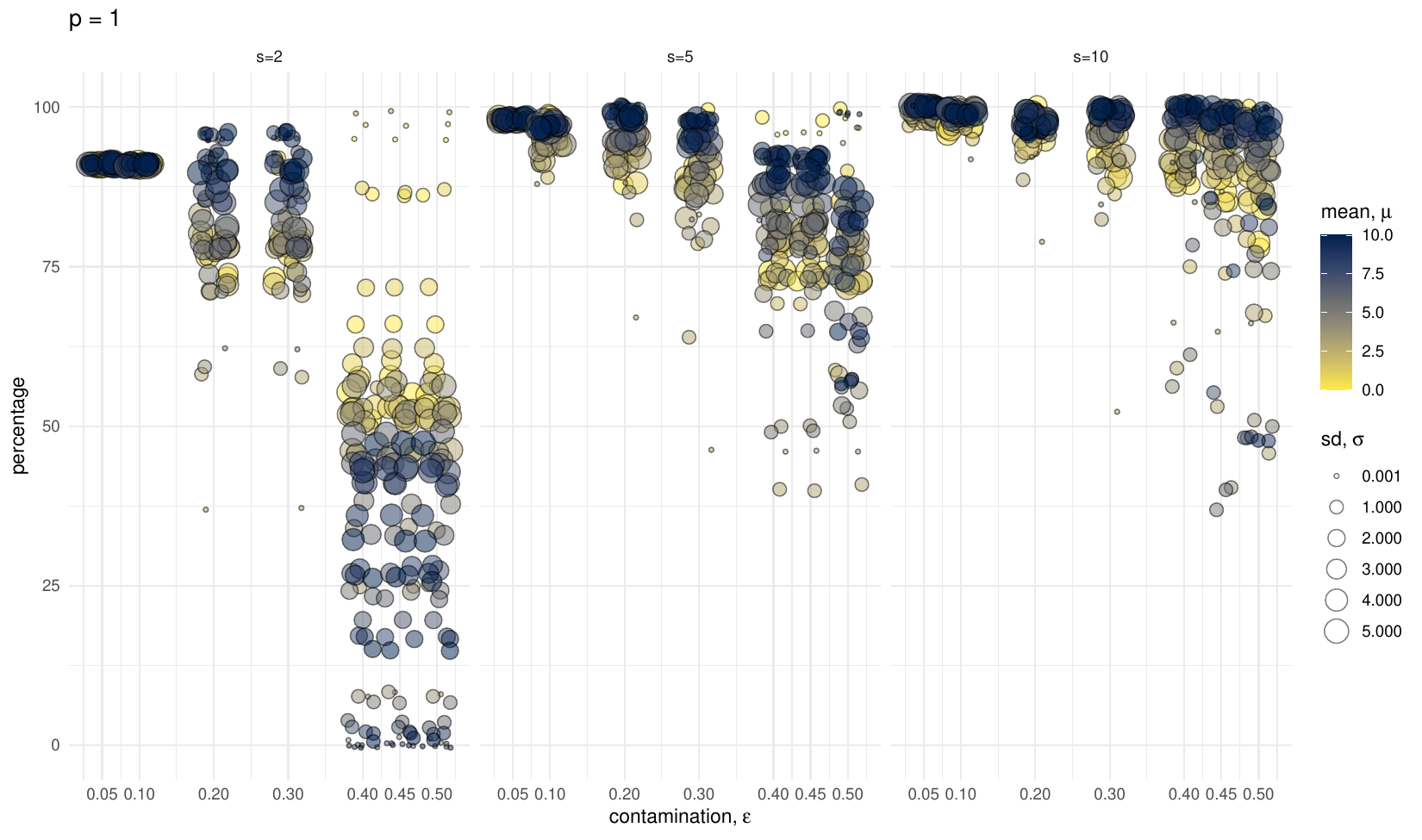}\\
  \includegraphics[width=\textwidth]{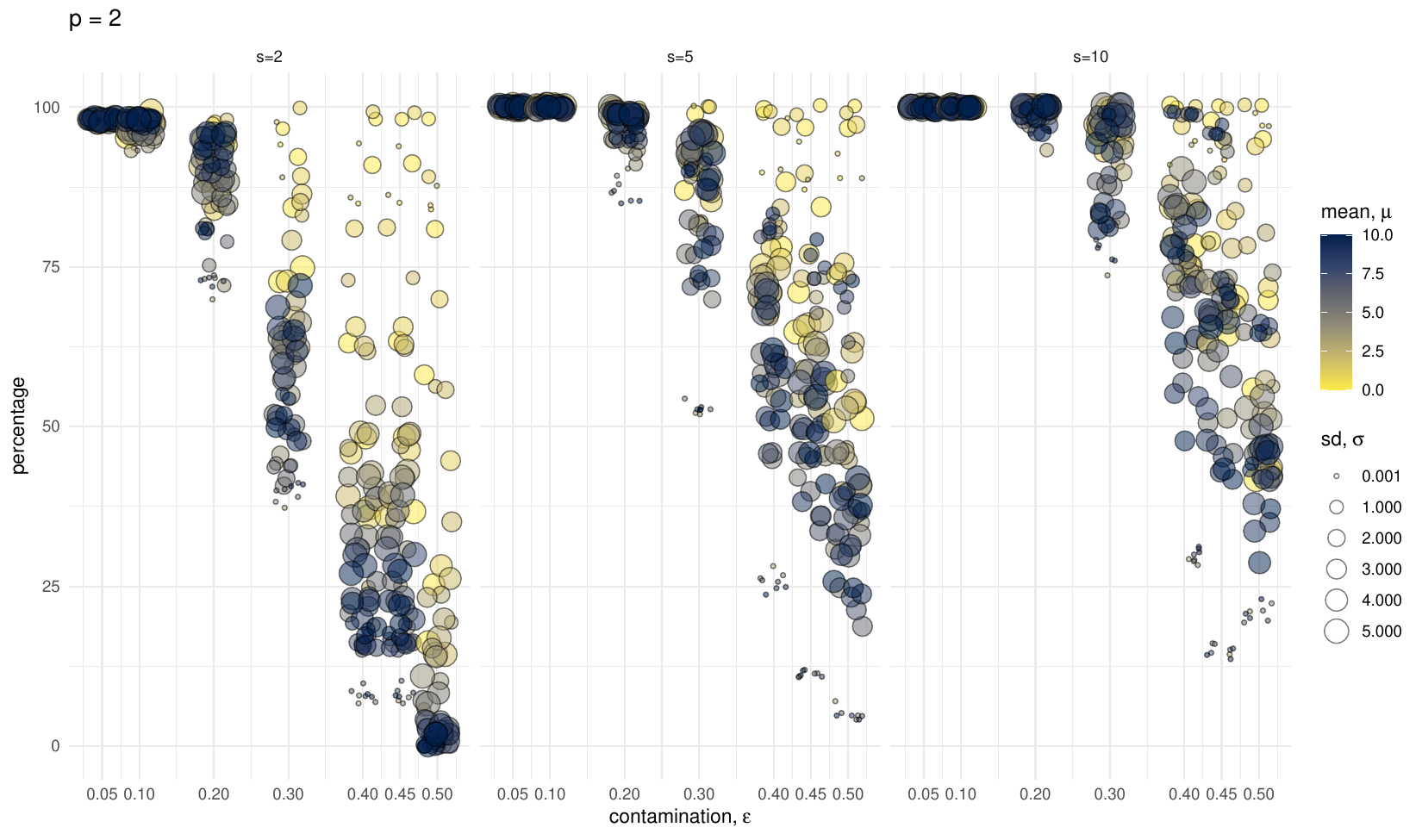} 
  \caption{Monte Carlo Simulation. Percentage of robust root retrieval out of $N=100$ simulations for the proposed method with subsampling for the initial values, as a function of contamination $\varepsilon$ ($x-$axis). Contamination average $\mu$ and scale $\sigma$ are represented by color and size of the bubbles respectively. From left to right, the subplots show the different sample size factors $s=2, 5, 10$. Number of variables $p=1$ (top) and $p=2$ (bottom). $\alpha=0.75$.}
  \label{sup:fig:subs-true-5}
\end{figure}

\begin{figure}
  \centering
  \includegraphics[width=\textwidth]{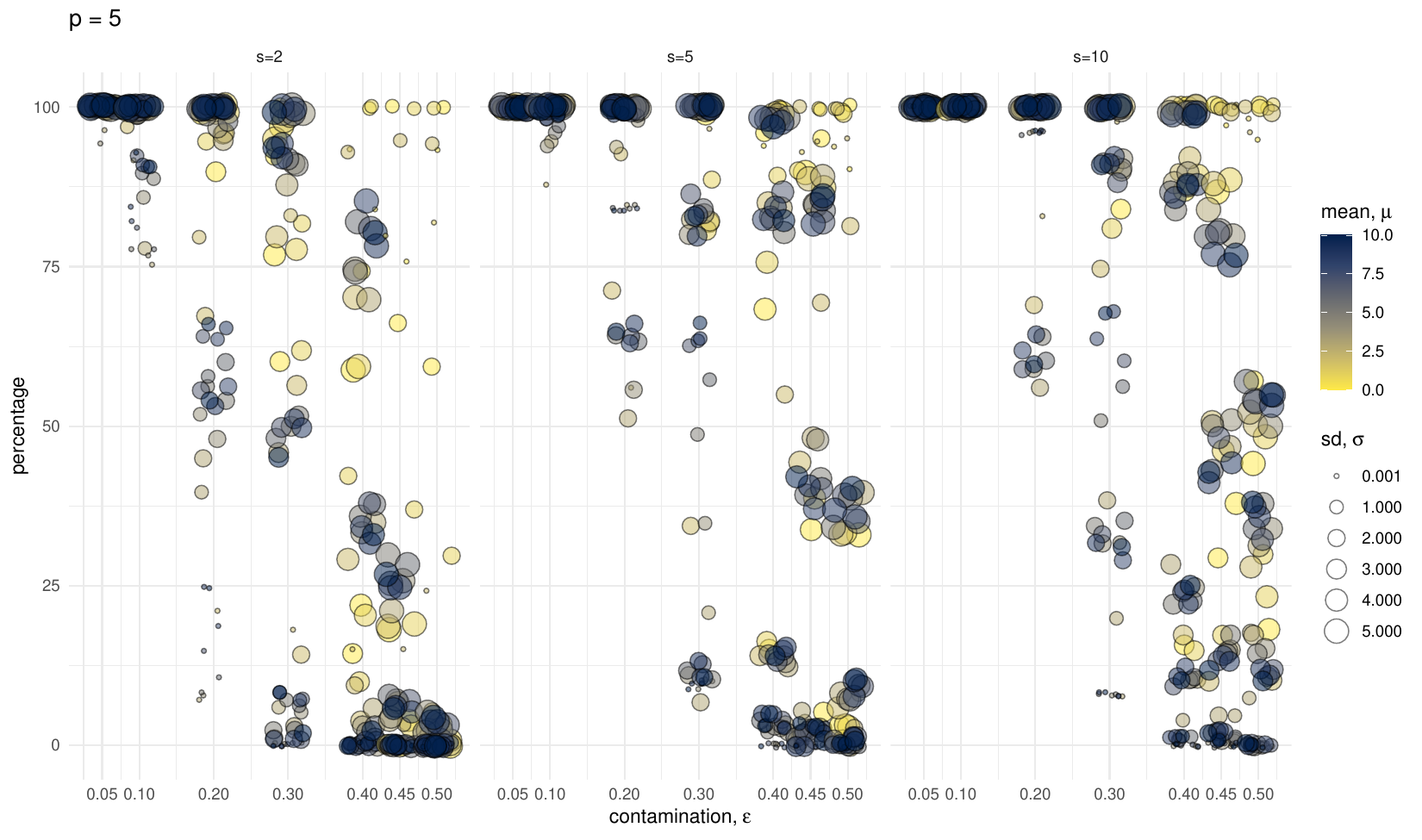}\\
  \includegraphics[width=\textwidth]{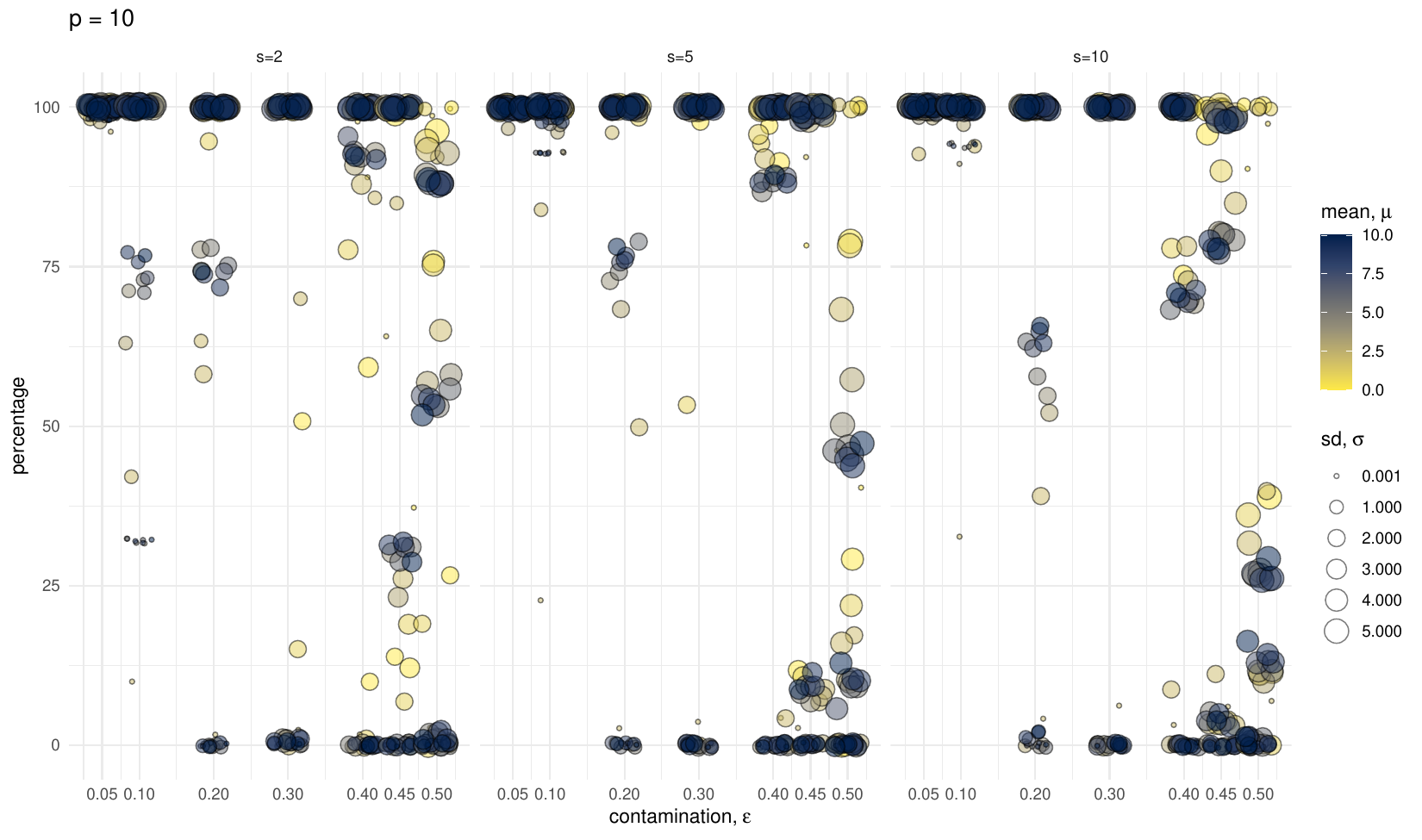}
  \caption{Monte Carlo Simulation. Percentage of robust root retrieval out of $N=100$ simulations for the proposed method with subsampling for the initial values, as a function of contamination $\varepsilon$ ($x-$axis). Contamination average $\mu$ and scale $\sigma$ are represented by color and size of the bubbles respectively. From left to right, the subplots show the different sample size factors $s=2, 5, 10$. Number of variables $p=5$ (top) and $p=10$ (bottom). $\alpha=0.75$.}
  \label{sup:fig:subs-true-6}
\end{figure}

\begin{figure}
  \includegraphics[width=\textwidth]{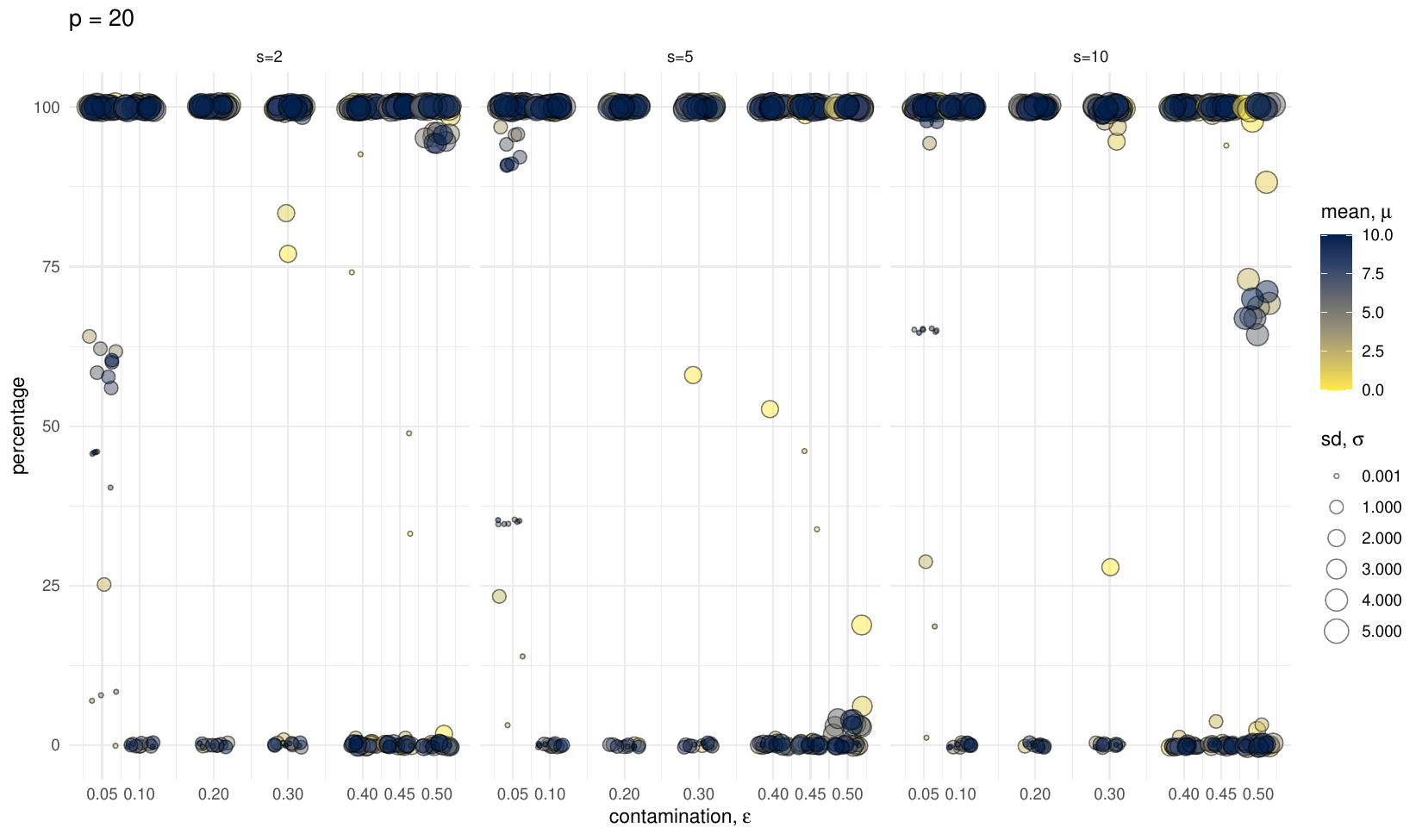}
  \caption{Monte Carlo Simulation. Percentage of robust root retrieval out of $N=100$ simulations for the proposed method with subsampling for the initial values, as a function of contamination $\varepsilon$ ($x-$axis). Contamination average $\mu$ and scale $\sigma$ are represented by color and size of the bubbles respectively. From left to right, the subplots show the different sample size factors $s=2, 5, 10$. Number of variables $p=20$. $\alpha=0.75$.}
  \label{sup:fig:subs-true-6b}
\end{figure}

\begin{figure}
  \centering
  \includegraphics[width=\textwidth]{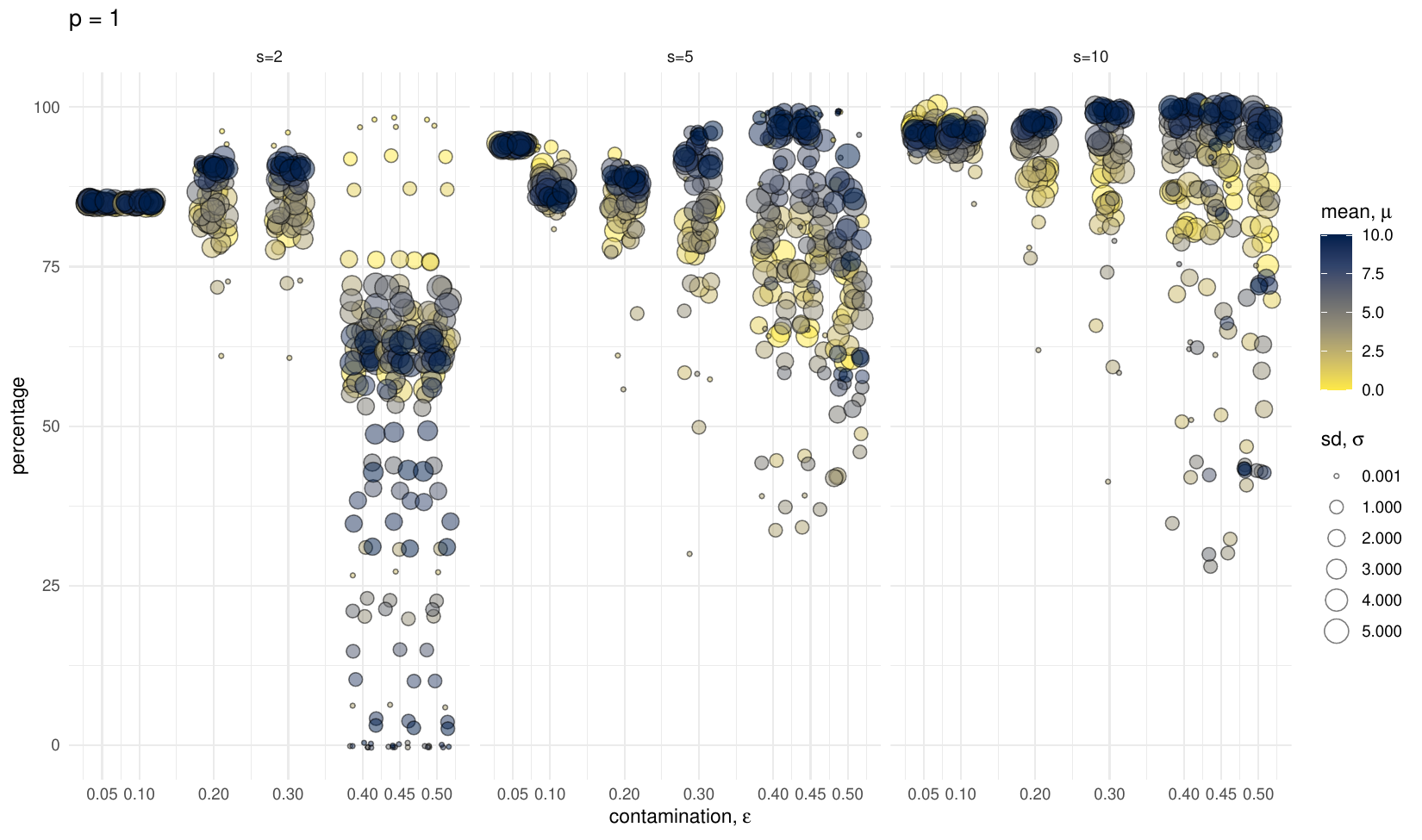}\\
  \includegraphics[width=\textwidth]{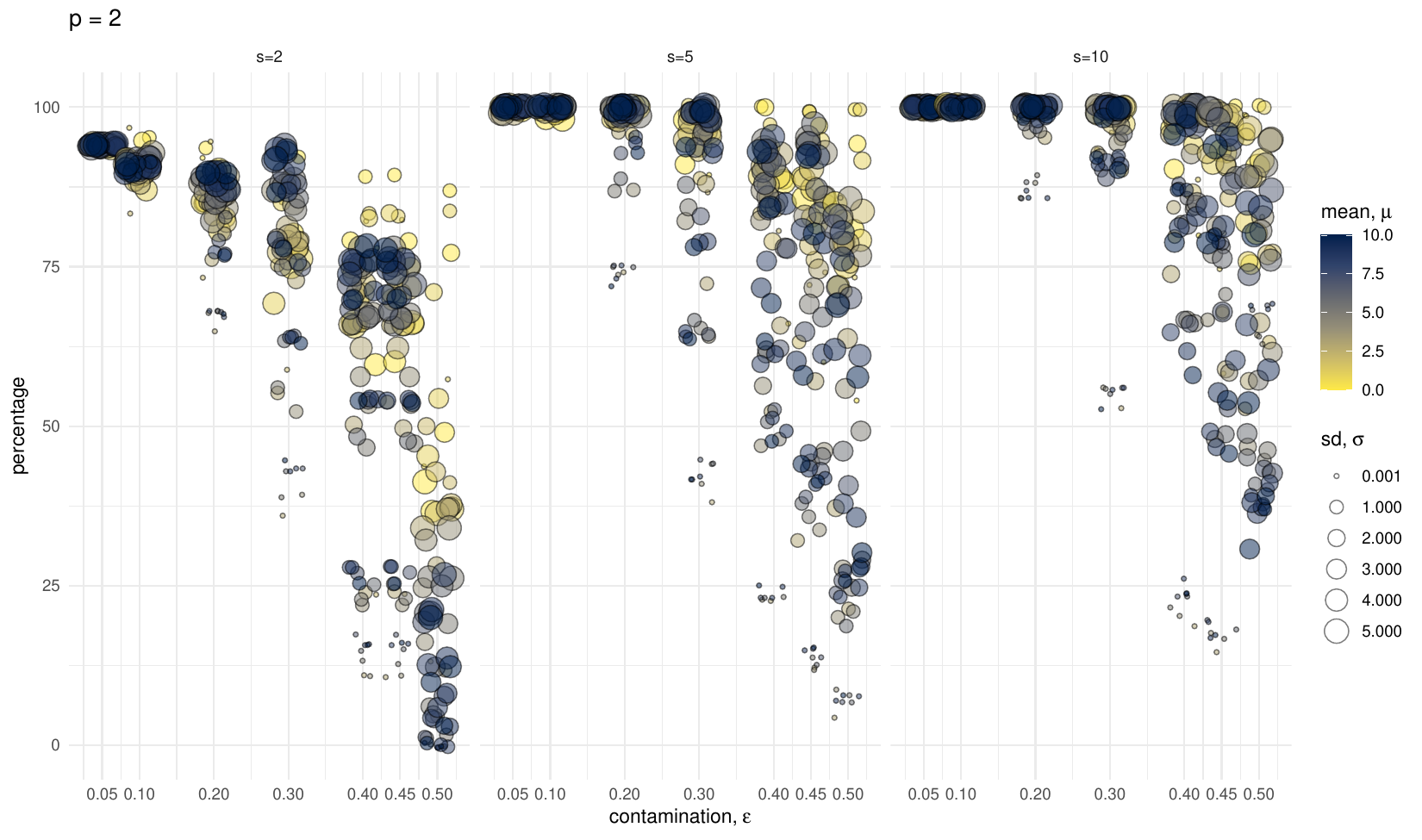} 
  \caption{Monte Carlo Simulation. Percentage of robust root retrieval out of $N=100$ simulations for the proposed method with subsampling for the initial values, as a function of contamination $\varepsilon$ ($x-$axis). Contamination average $\mu$ and scale $\sigma$ are represented by color and size of the bubbles respectively. From left to right, the subplots show the different sample size factors $s=2, 5, 10$. Number of variables $p=1$ (top) and $p=2$ (bottom). $\alpha=1$.}
  \label{sup:fig:subs-true-7}
\end{figure}

\begin{figure}
  \centering
  \includegraphics[width=\textwidth]{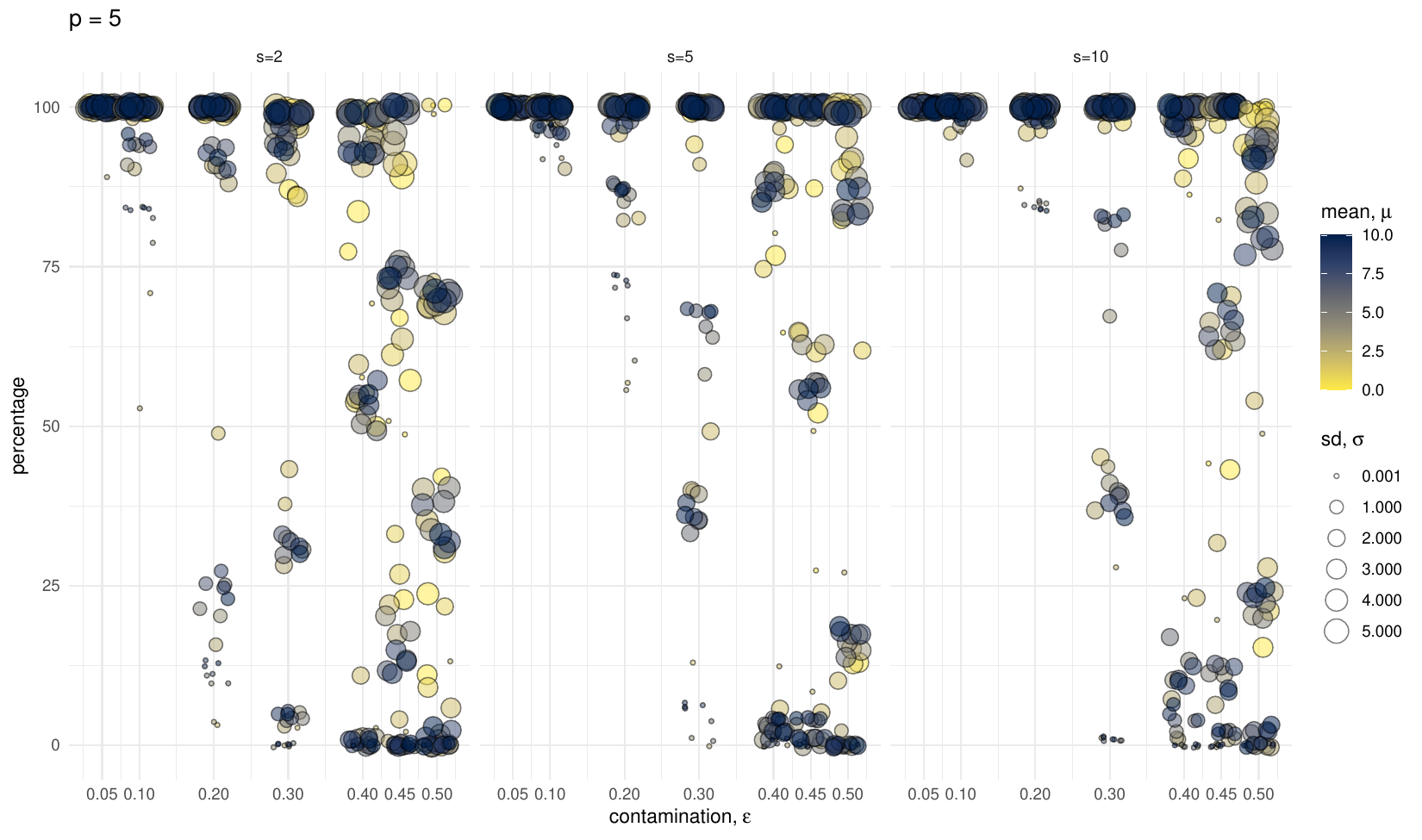}\\
  \includegraphics[width=\textwidth]{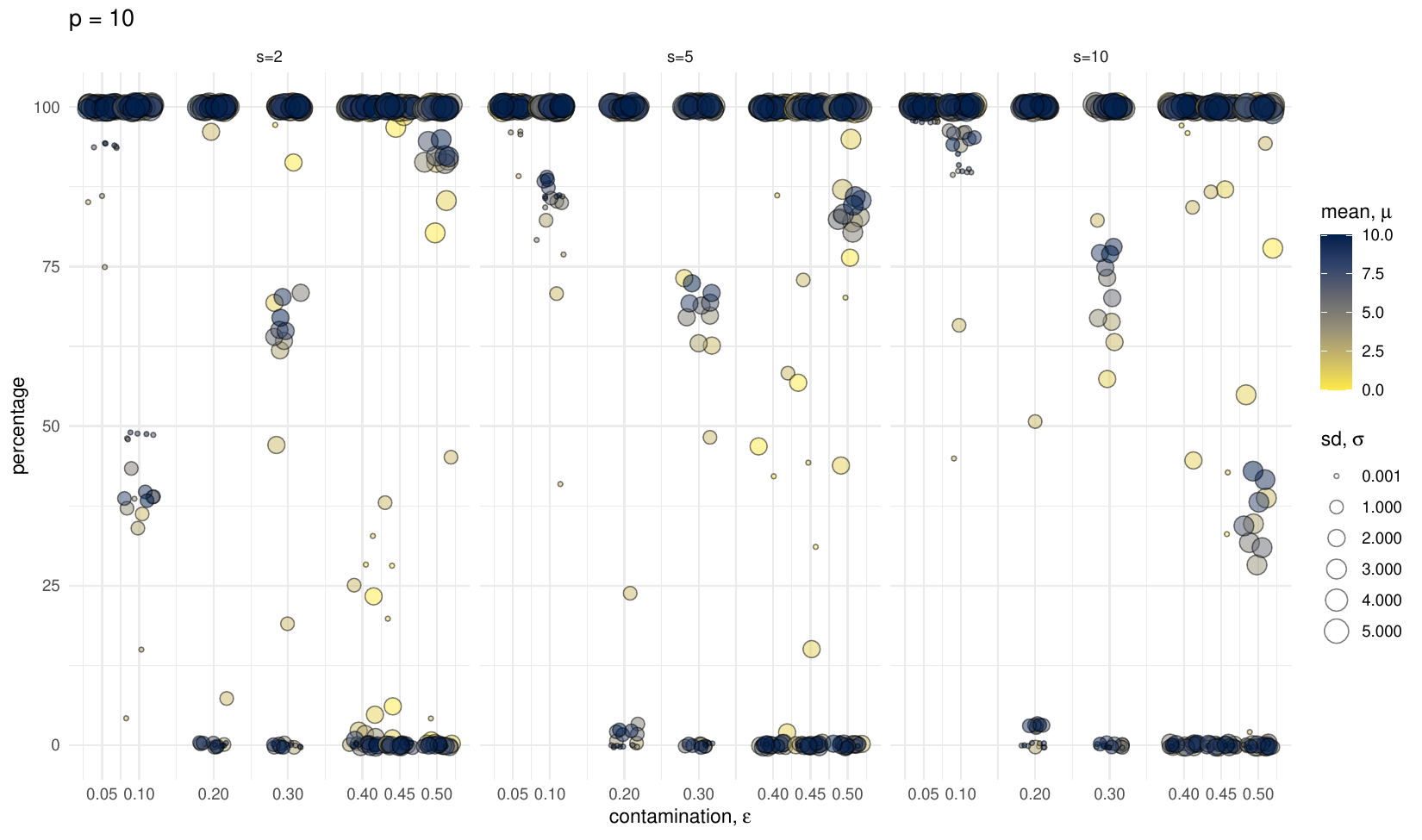}
  \caption{Monte Carlo Simulation. Percentage of robust root retrieval out of $N=100$ simulations for the proposed method with subsampling for the initial values, as a function of contamination $\varepsilon$ ($x-$axis). Contamination average $\mu$ and scale $\sigma$ are represented by color and size of the bubbles respectively. From left to right, the subplots show the different sample size factors $s=2, 5, 10$. Number of variables $p=5$ (top) and $p=10$ (bottom). $\alpha=1$.}
  \label{sup:fig:subs-true-8}
\end{figure}

\begin{figure}
  \includegraphics[width=\textwidth]{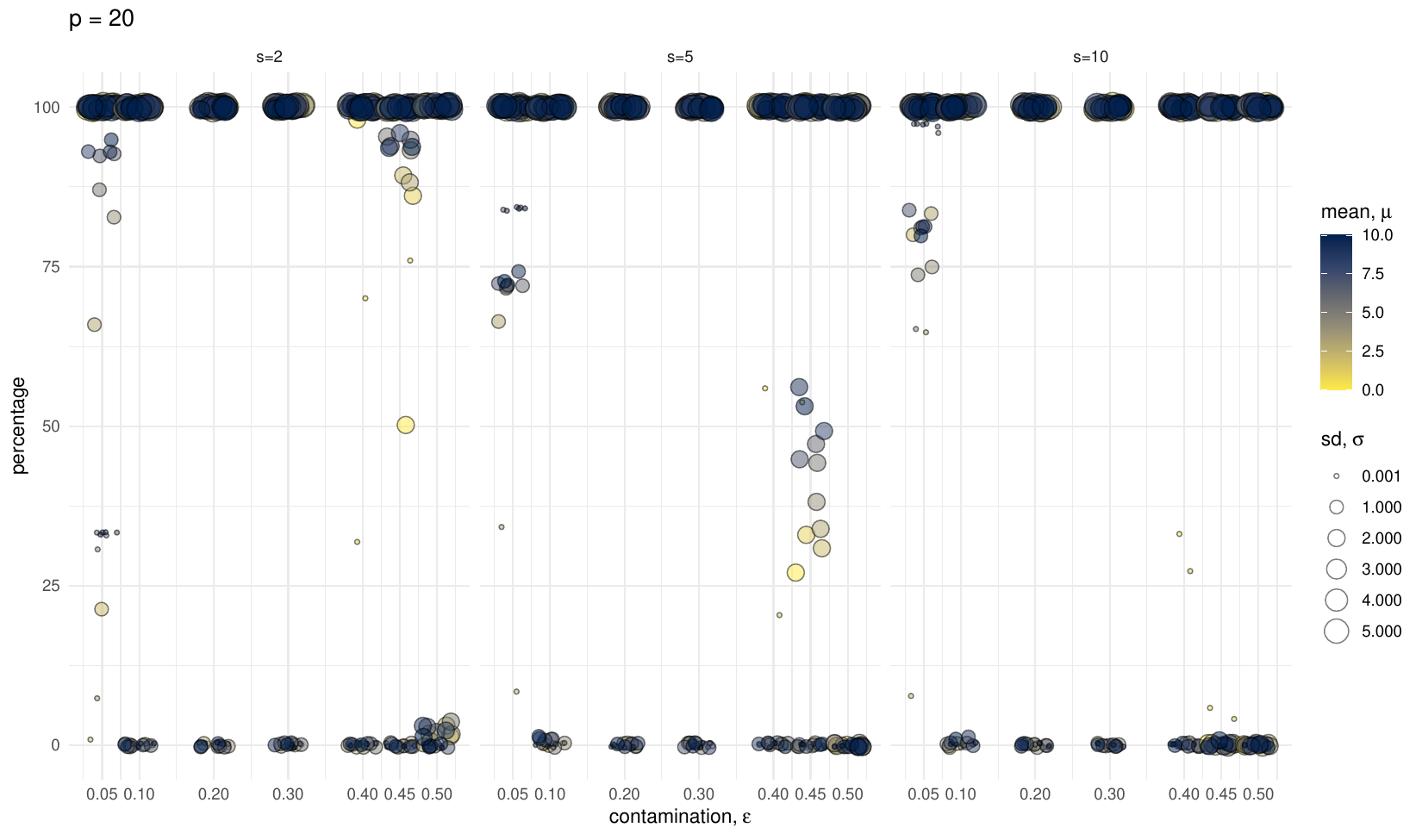}
  \caption{Monte Carlo Simulation. Percentage of robust root retrieval out of $N=100$ simulations for the proposed method with subsampling for the initial values, as a function of contamination $\varepsilon$ ($x-$axis). Contamination average $\mu$ and scale $\sigma$ are represented by color and size of the bubbles respectively. From left to right, the subplots show the different sample size factors $s=2, 5, 10$. Number of variables $p=20$. $\alpha=1$.}
  \label{sup:fig:subs-true-8b}
\end{figure}

\clearpage


\begin{figure}
\centering
\includegraphics[width=0.32\textwidth]{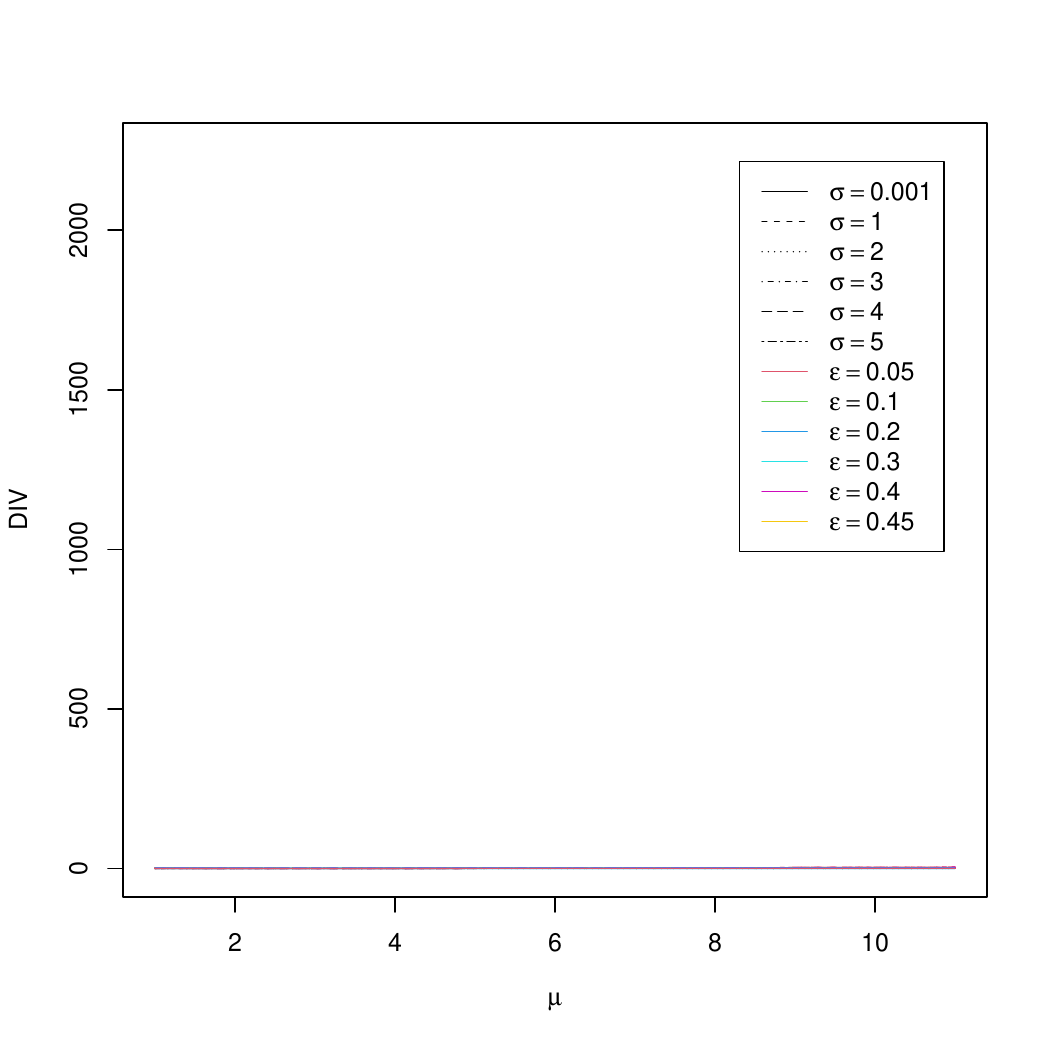}
\includegraphics[width=0.32\textwidth]{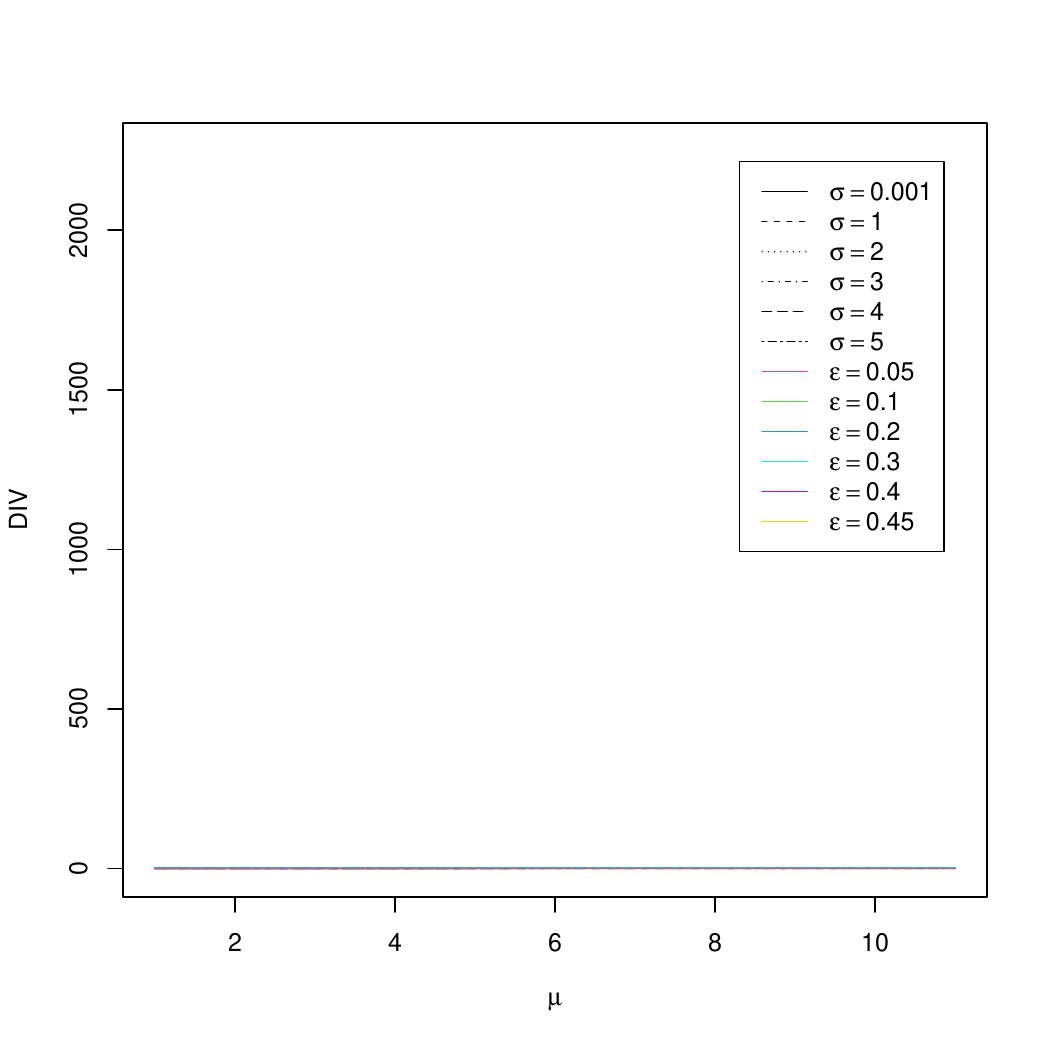} 
\includegraphics[width=0.32\textwidth]{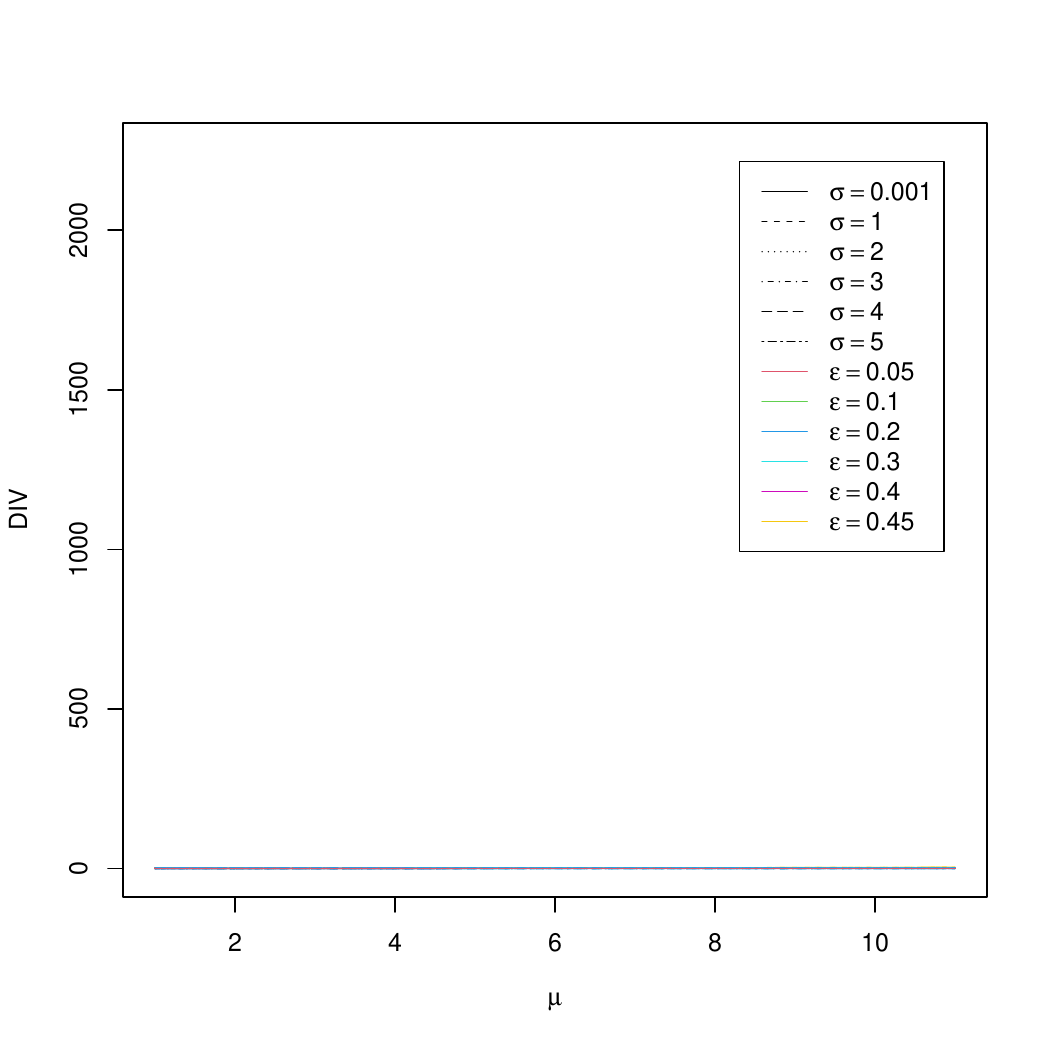} \\
\includegraphics[width=0.32\textwidth]{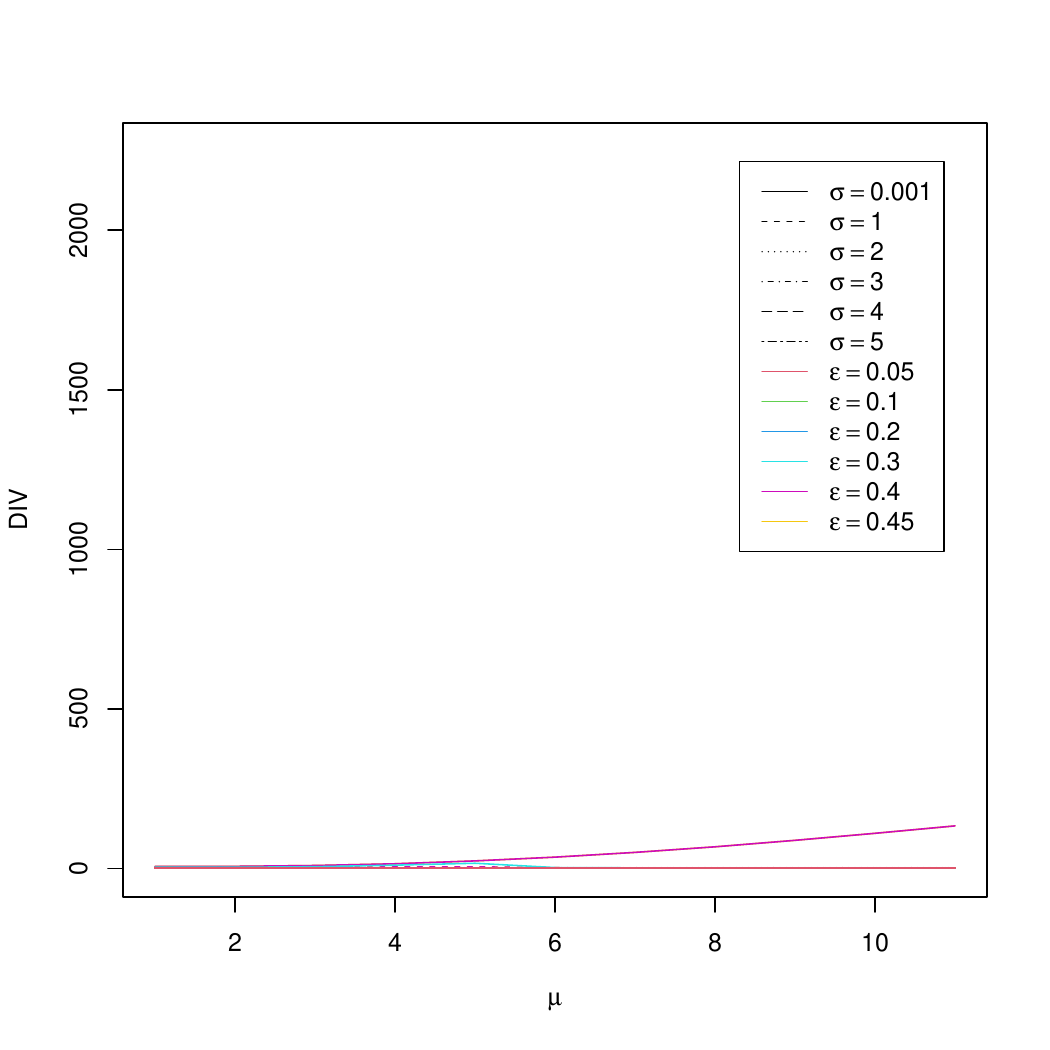}
\includegraphics[width=0.32\textwidth]{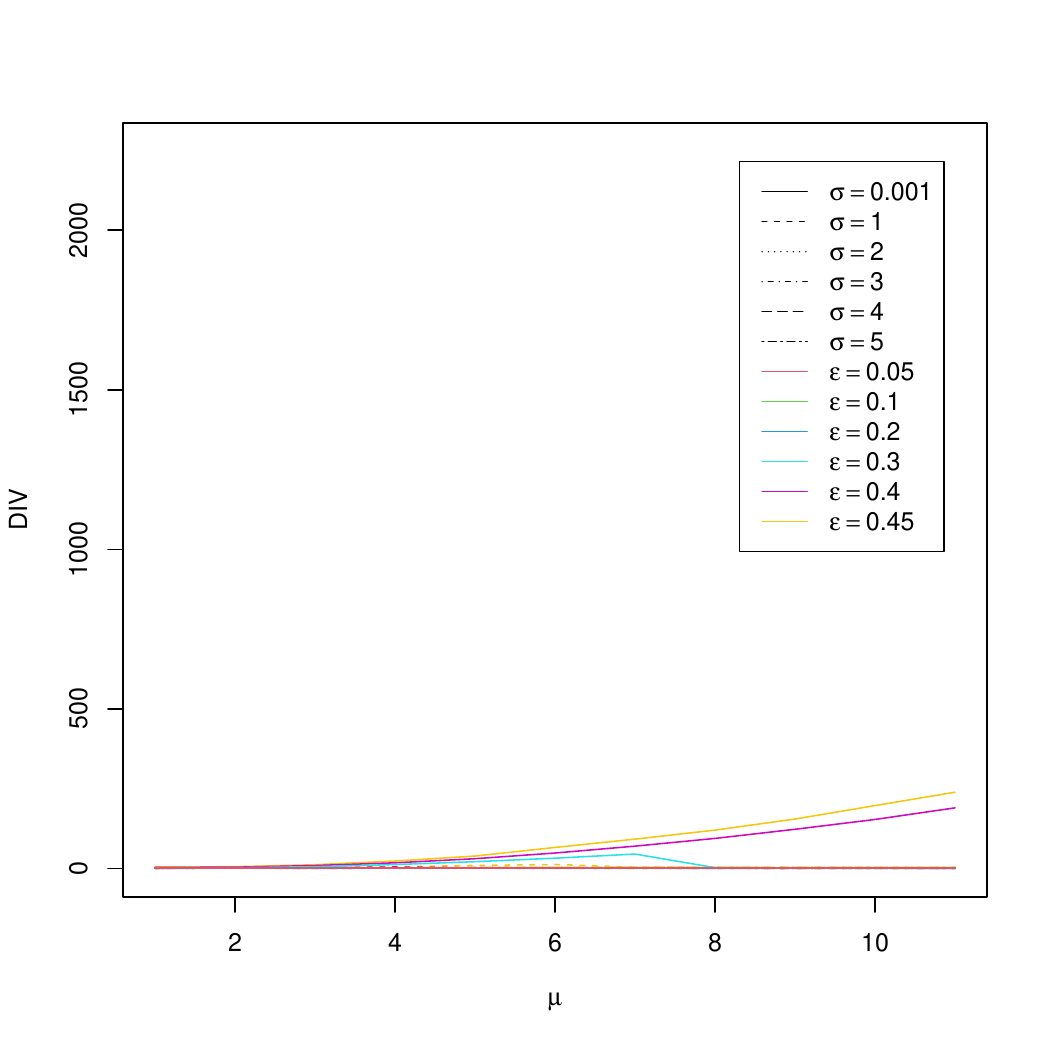} 
\includegraphics[width=0.32\textwidth]{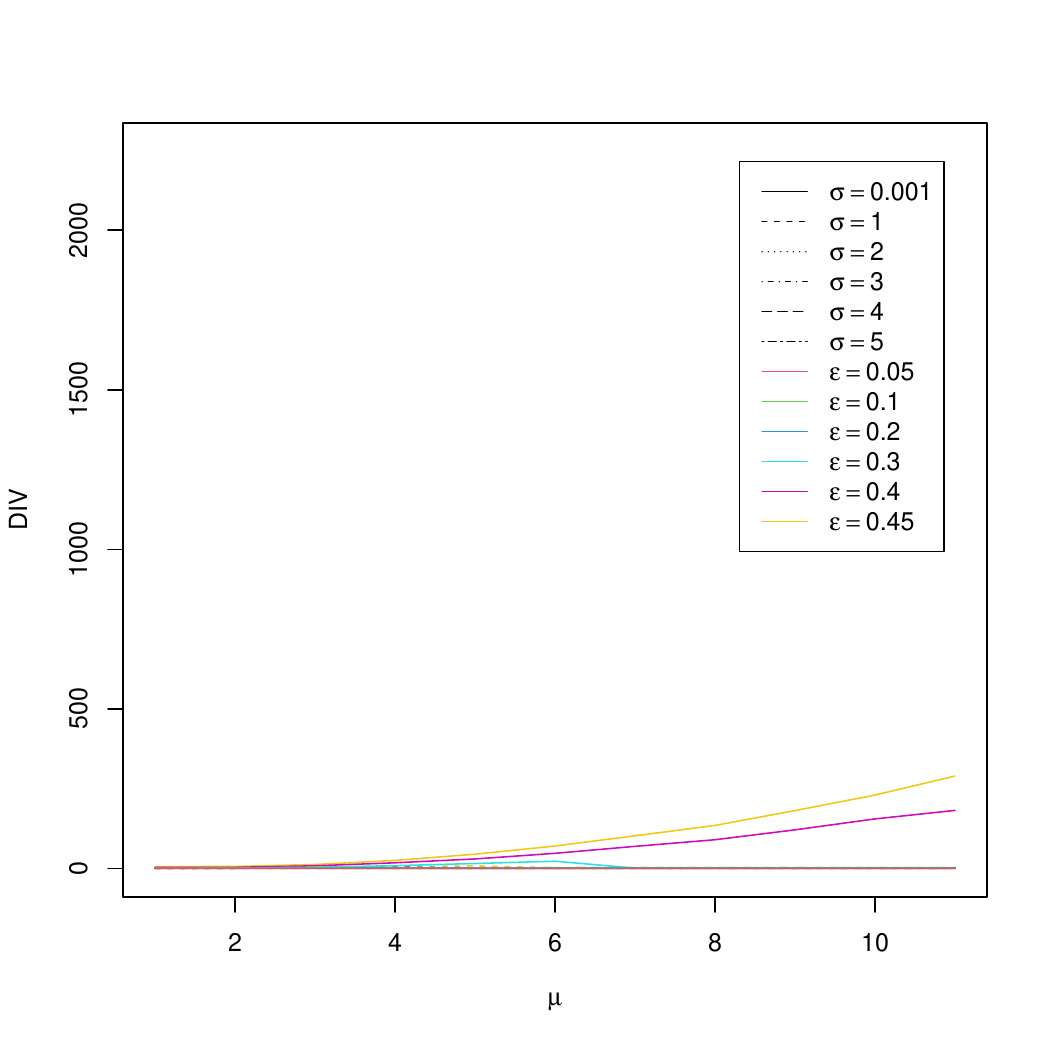} \\
\includegraphics[width=0.32\textwidth]{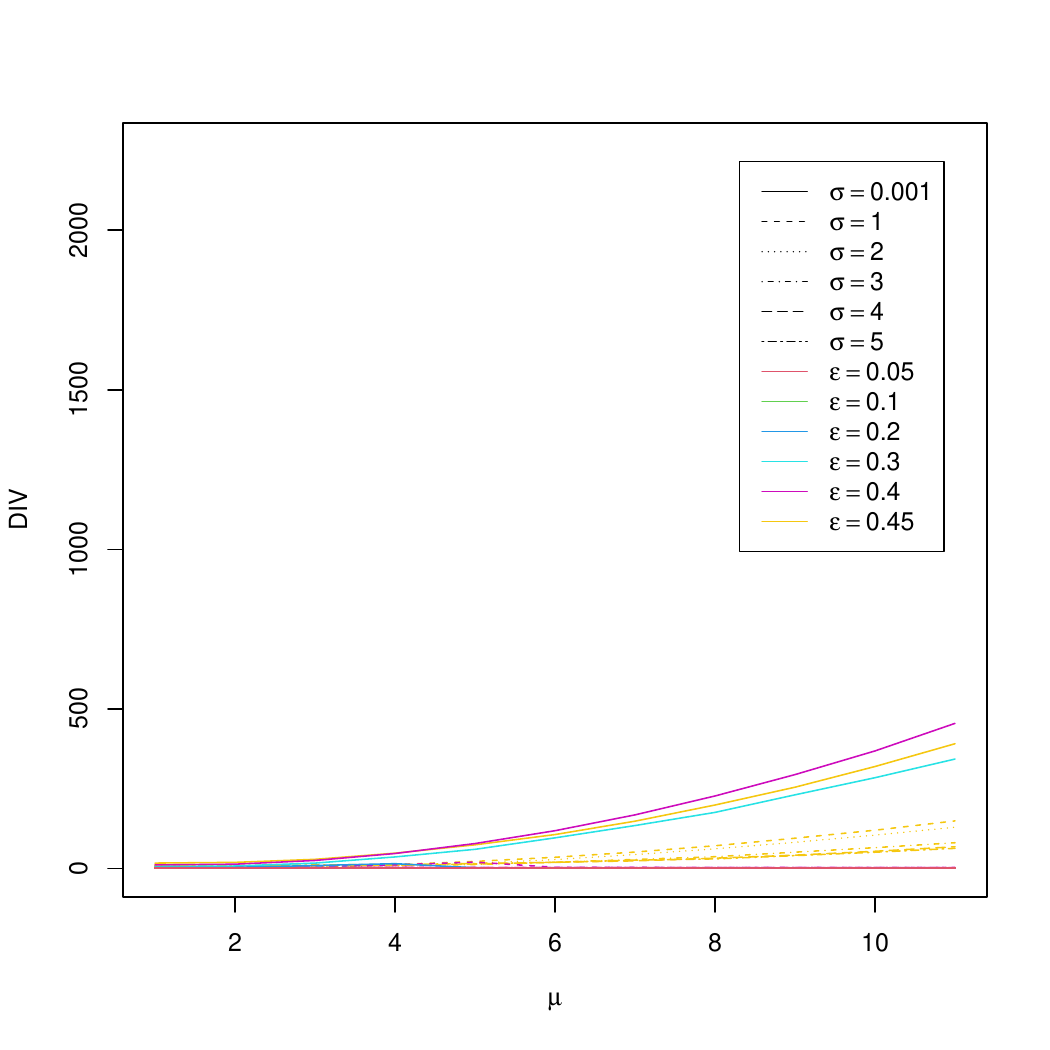}
\includegraphics[width=0.32\textwidth]{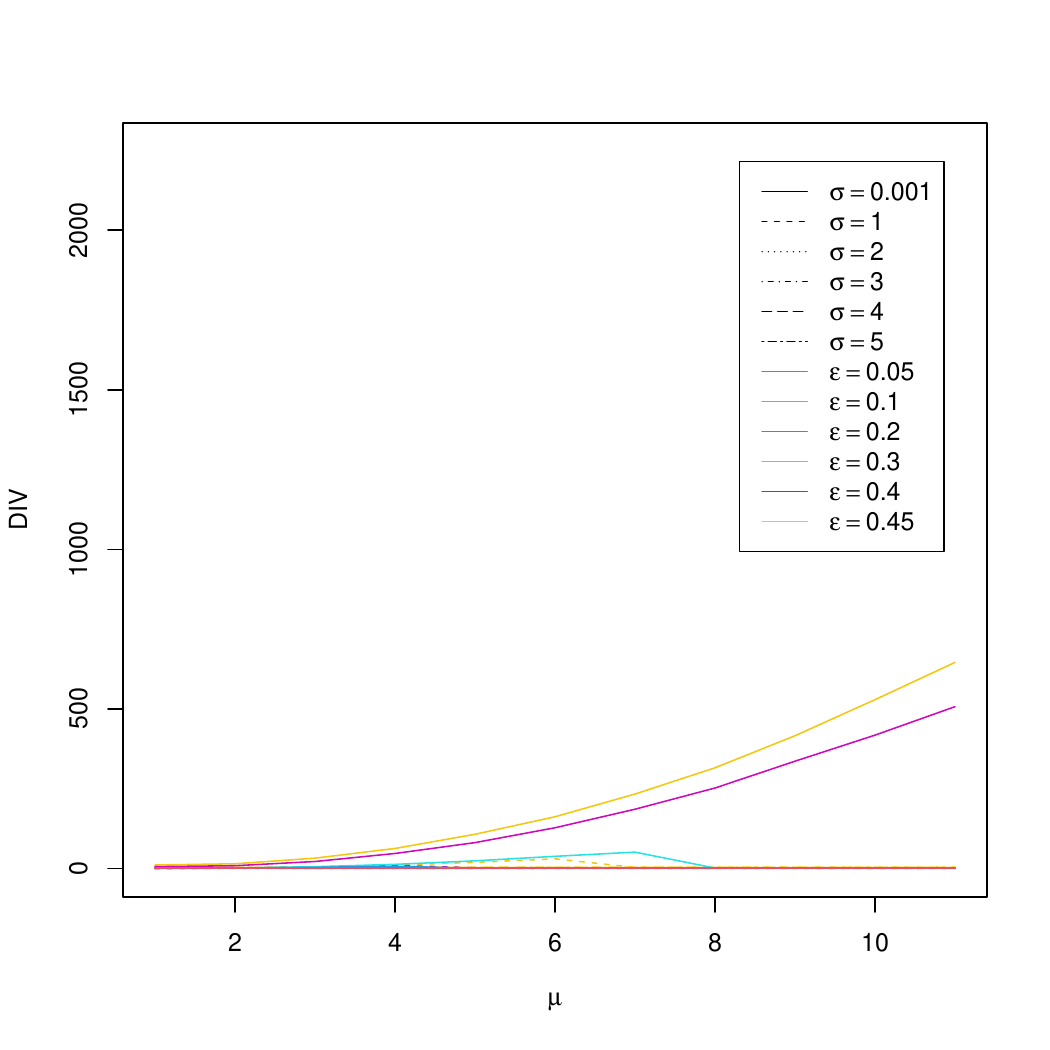} 
\includegraphics[width=0.32\textwidth]{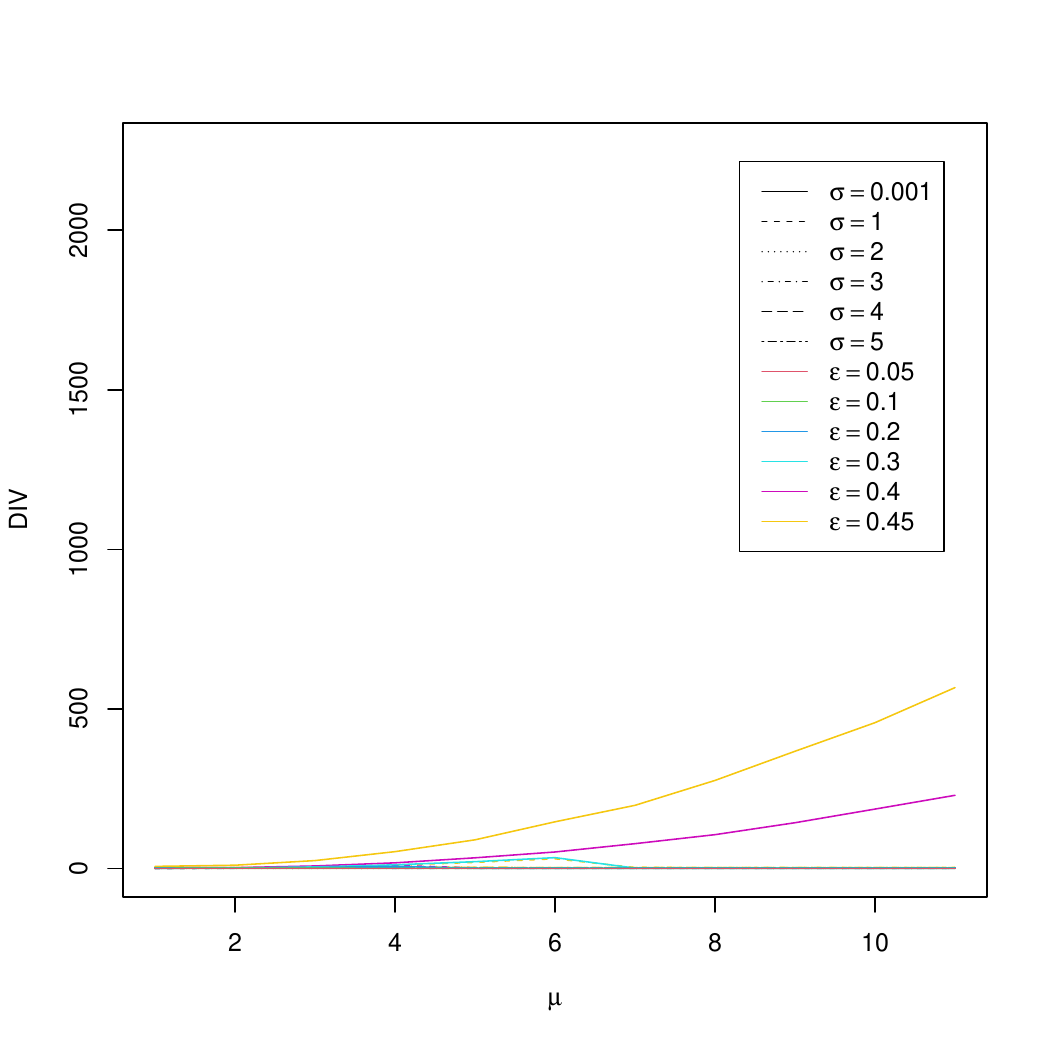} \\
\caption{Monte Carlo Simulation. Kullback--Leibler Divergence for the \texttt{CovMest} method as a function of the contamination average $\mu$ ($x$-axis), contamination scale $\sigma$ (different line styles) and contamination level $\varepsilon$ (colors). Rows: number of variables $p=1, 2, 5$ and columns: sample size factor $s=2, 5, 10$.}
\label{sup:fig:monte:DIV:M:1}
\end{figure}  

\begin{figure}
\centering
\includegraphics[width=0.32\textwidth]{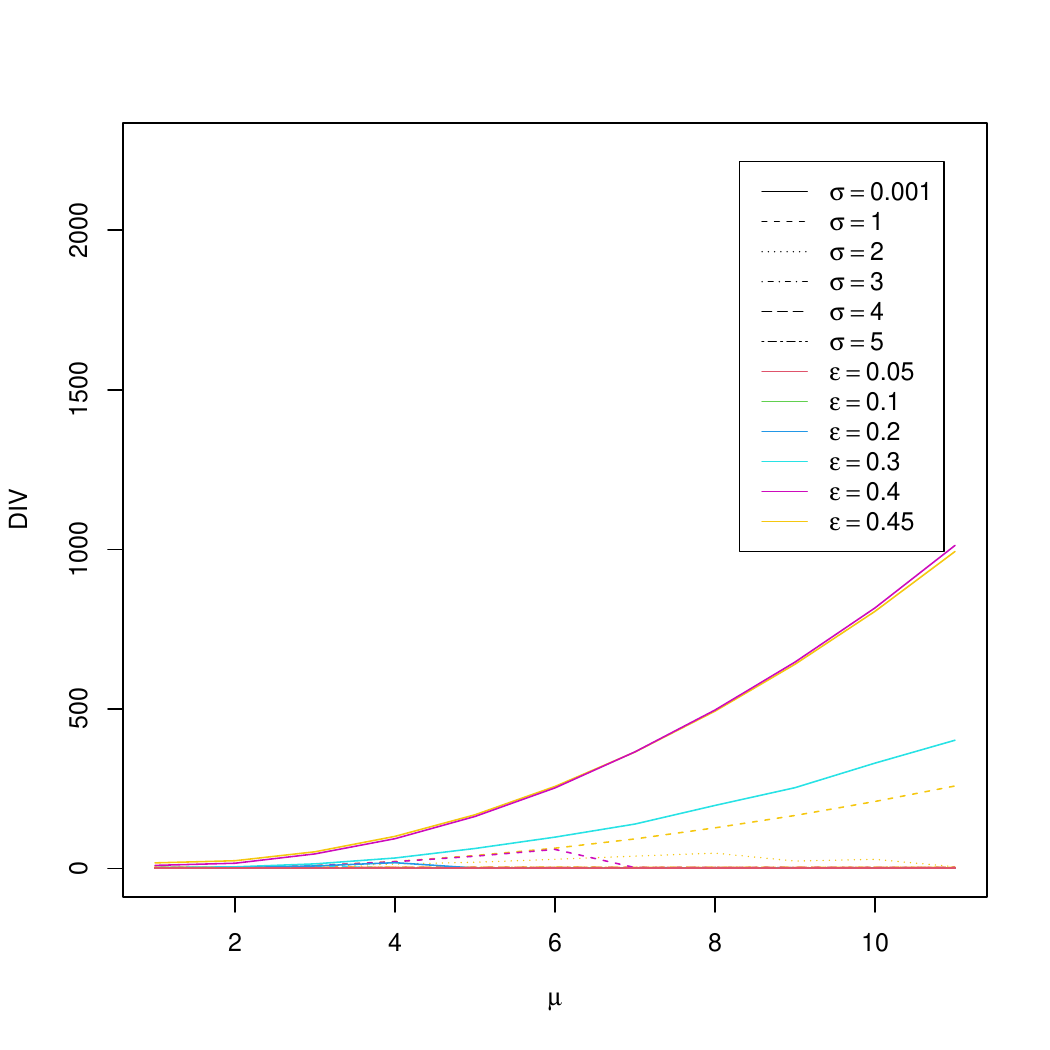}
\includegraphics[width=0.32\textwidth]{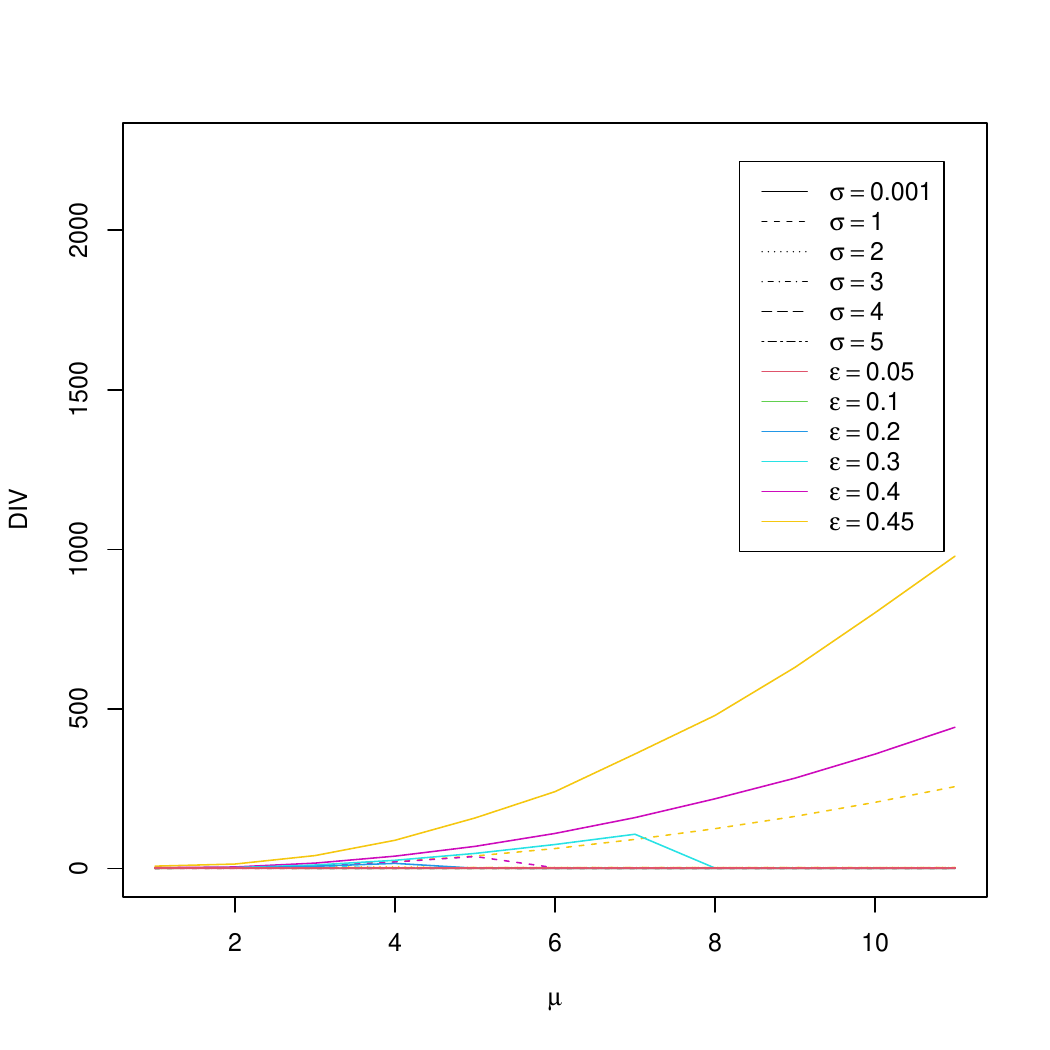} 
\includegraphics[width=0.32\textwidth]{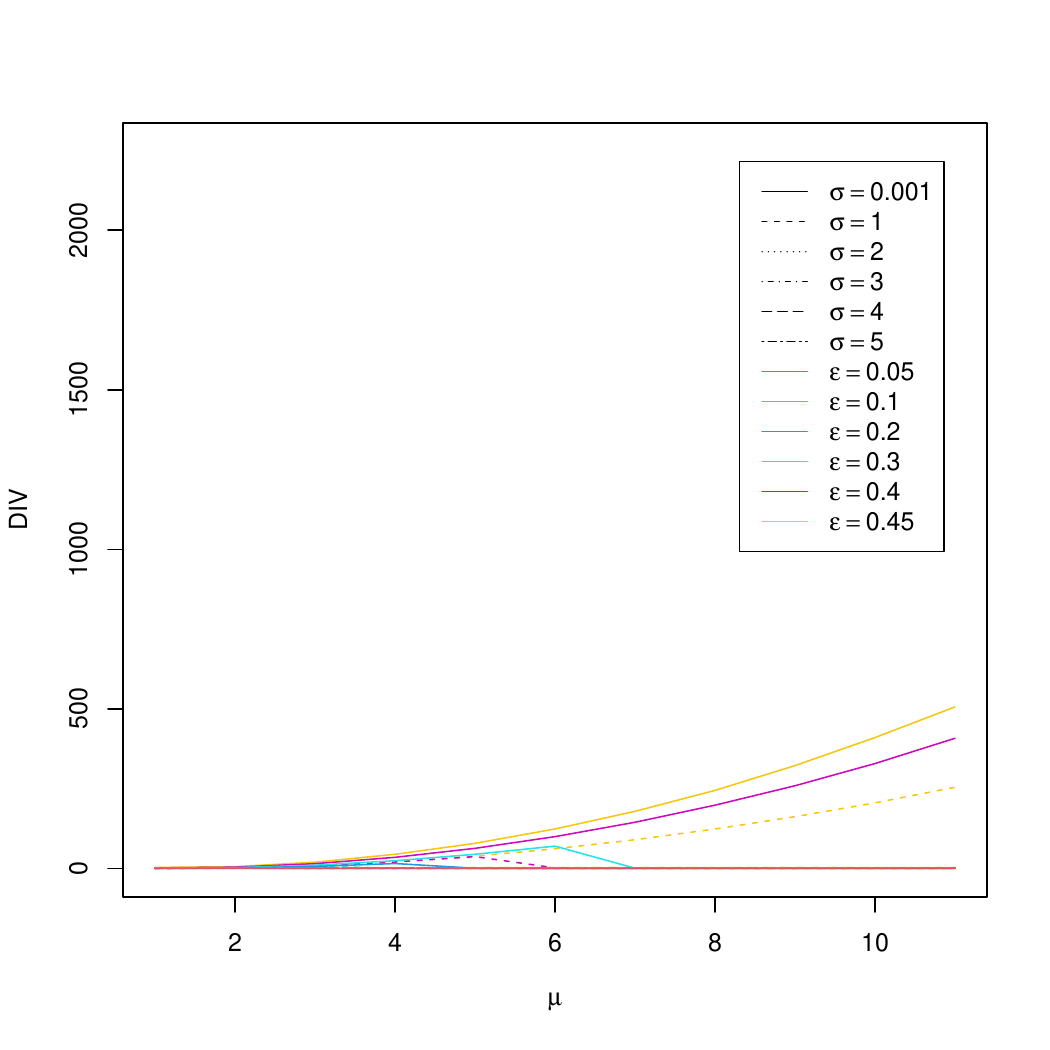} \\
\includegraphics[width=0.32\textwidth]{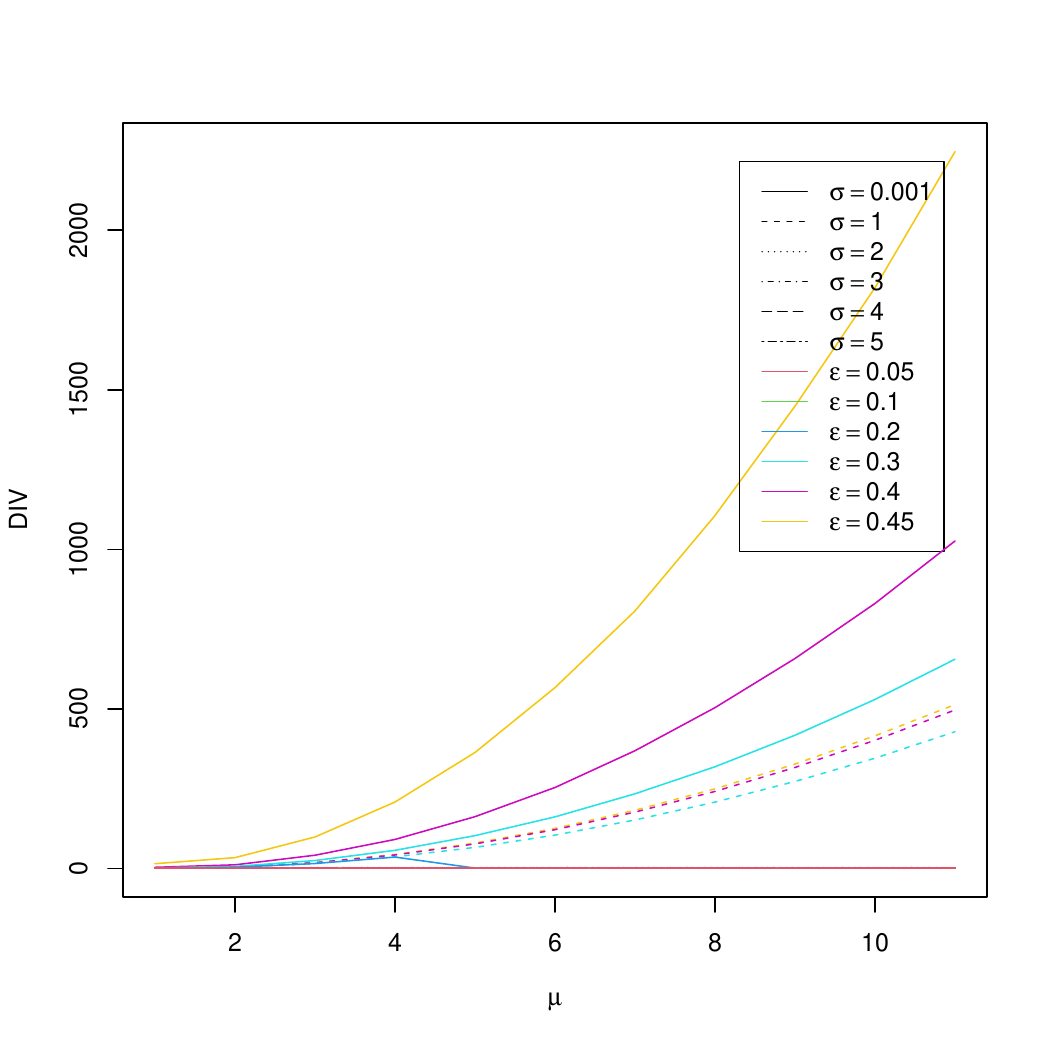}
\includegraphics[width=0.32\textwidth]{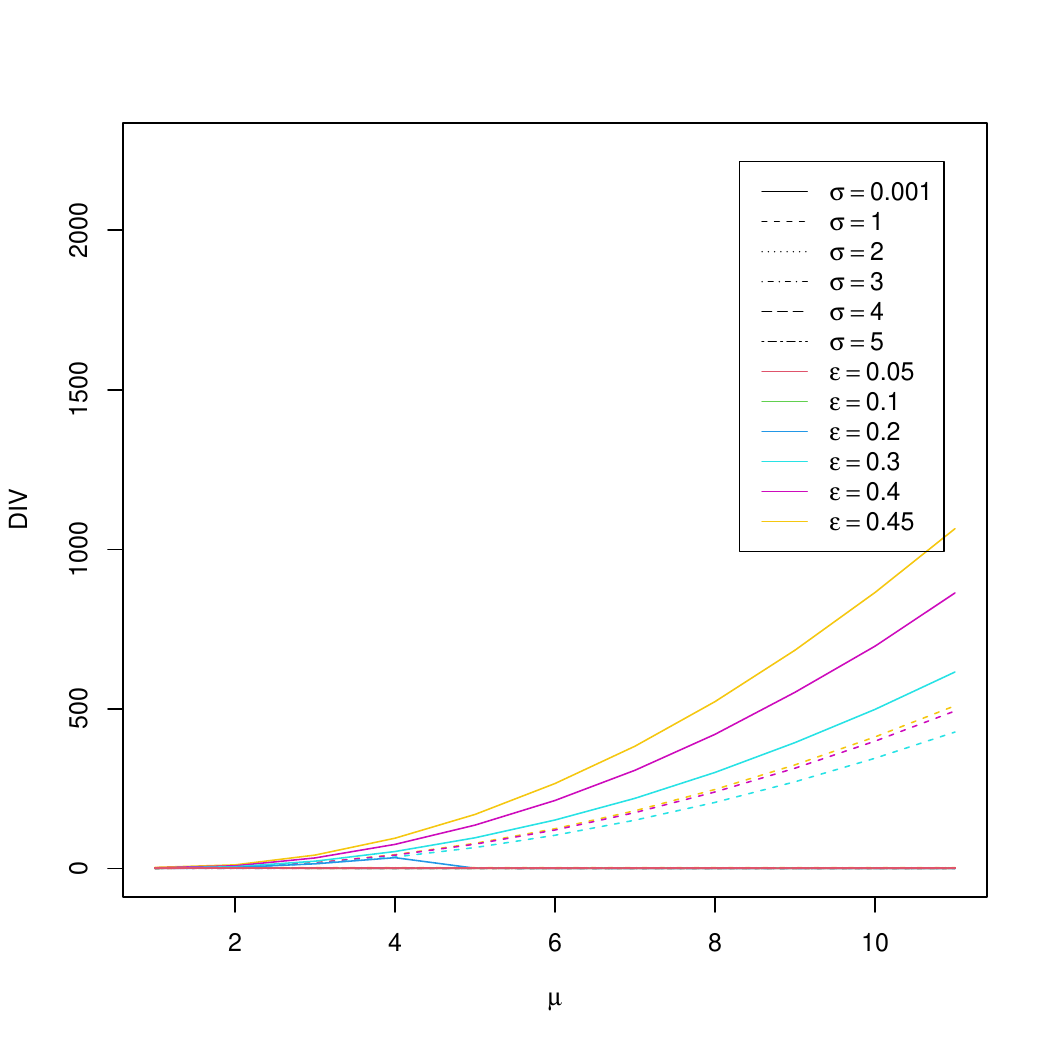} 
\includegraphics[width=0.32\textwidth]{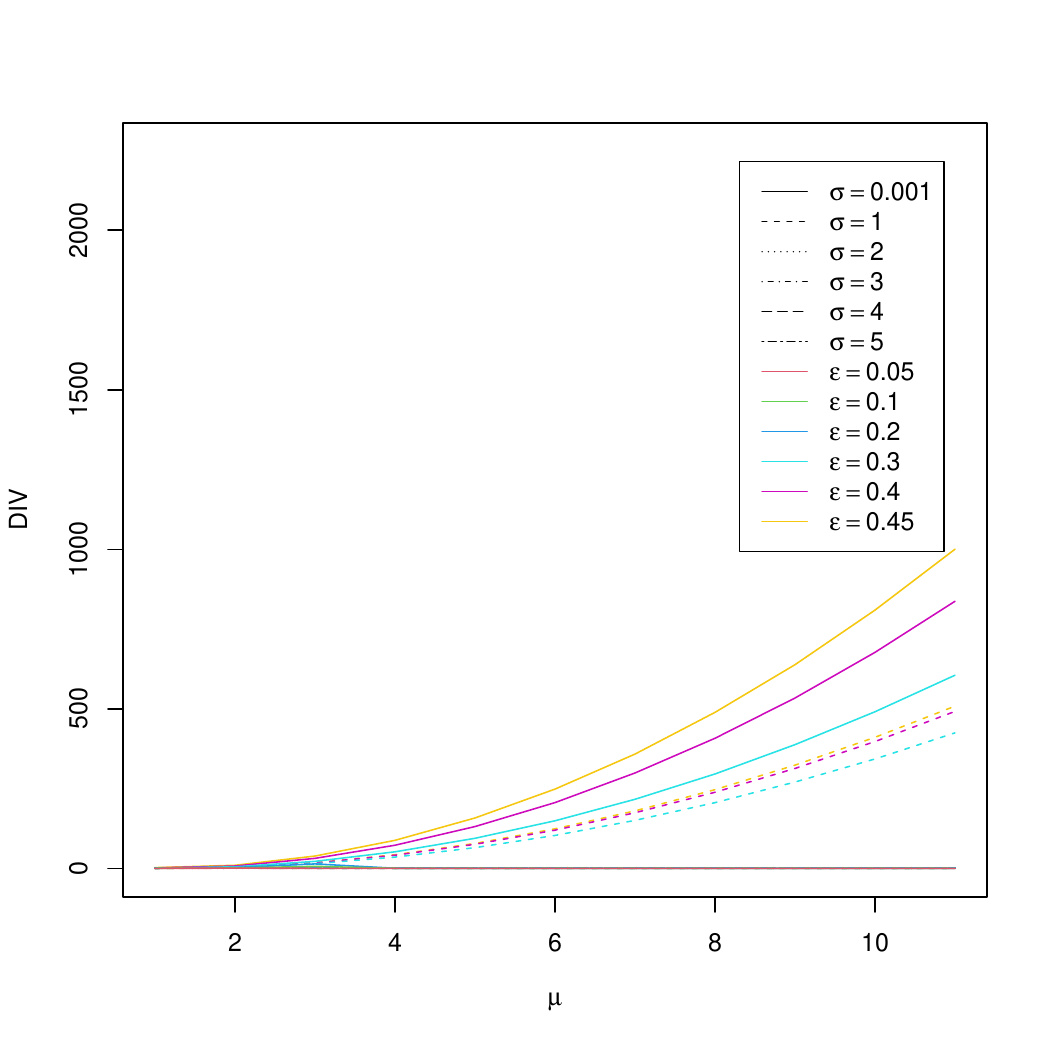}
\caption{Monte Carlo Simulation. Kullback--Leibler Divergence for the \texttt{CovMest} method as a function of the contamination average $\mu$ ($x$-axis), contamination scale $\sigma$ (different line styles) and contamination level $\varepsilon$ (colors). Rows: number of variables $p=10, 20$ and columns: sample size factor $s=2, 5, 10$.}
\label{sup:fig:monte:DIV:M:2}
\end{figure}  

\begin{figure}
  \centering
  \includegraphics[width=\textwidth]{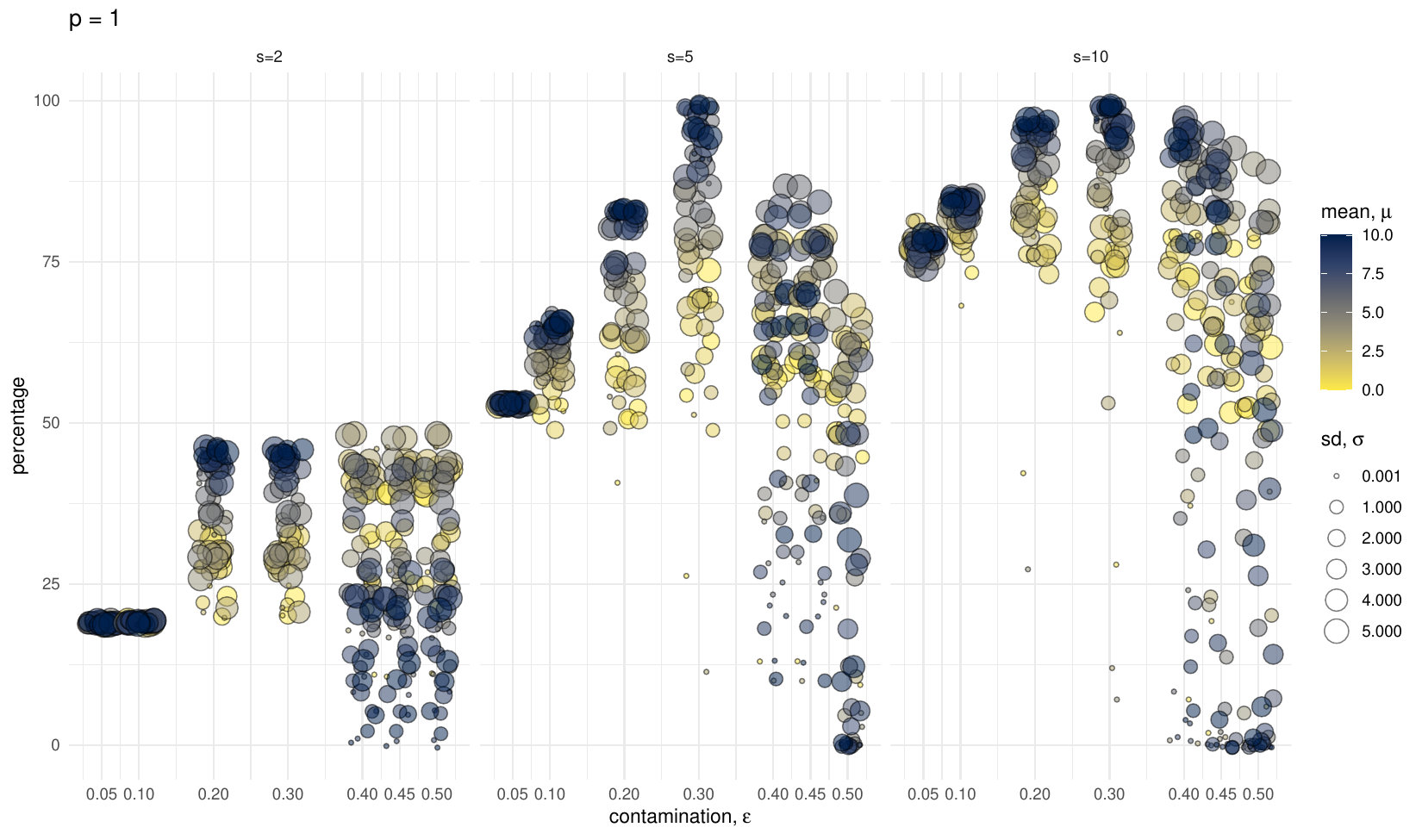}\\
  \includegraphics[width=\textwidth]{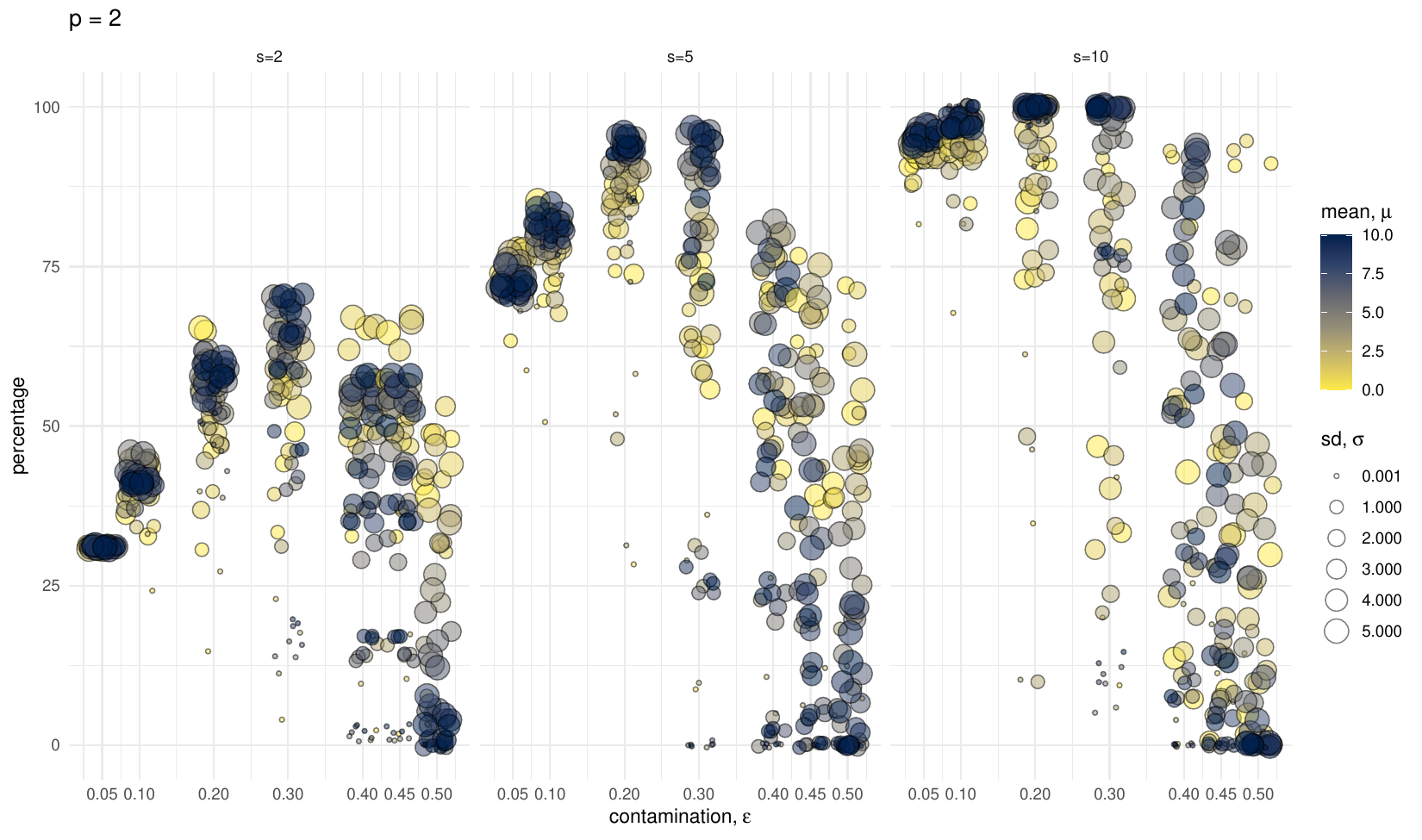} 
  \caption{Monte Carlo Simulation. Percentage of robust root retrieval out of $N=100$ simulations for the proposed method starting deterministically from the deepest points as a function of contamination $\varepsilon$ ($x-$axis). Contamination average $\mu$ and scale $\sigma$ are represented by color and size of the bubbles respectively. From left to right, the subplots show the different sample size factors $s=2, 5, 10$. Number of variables $p=1$ (top) and $p=2$ (bottom). $\alpha=0.25$.}
  \label{sup:fig:depth-true-1}
\end{figure}

\begin{figure}
  \centering
  \includegraphics[width=\textwidth]{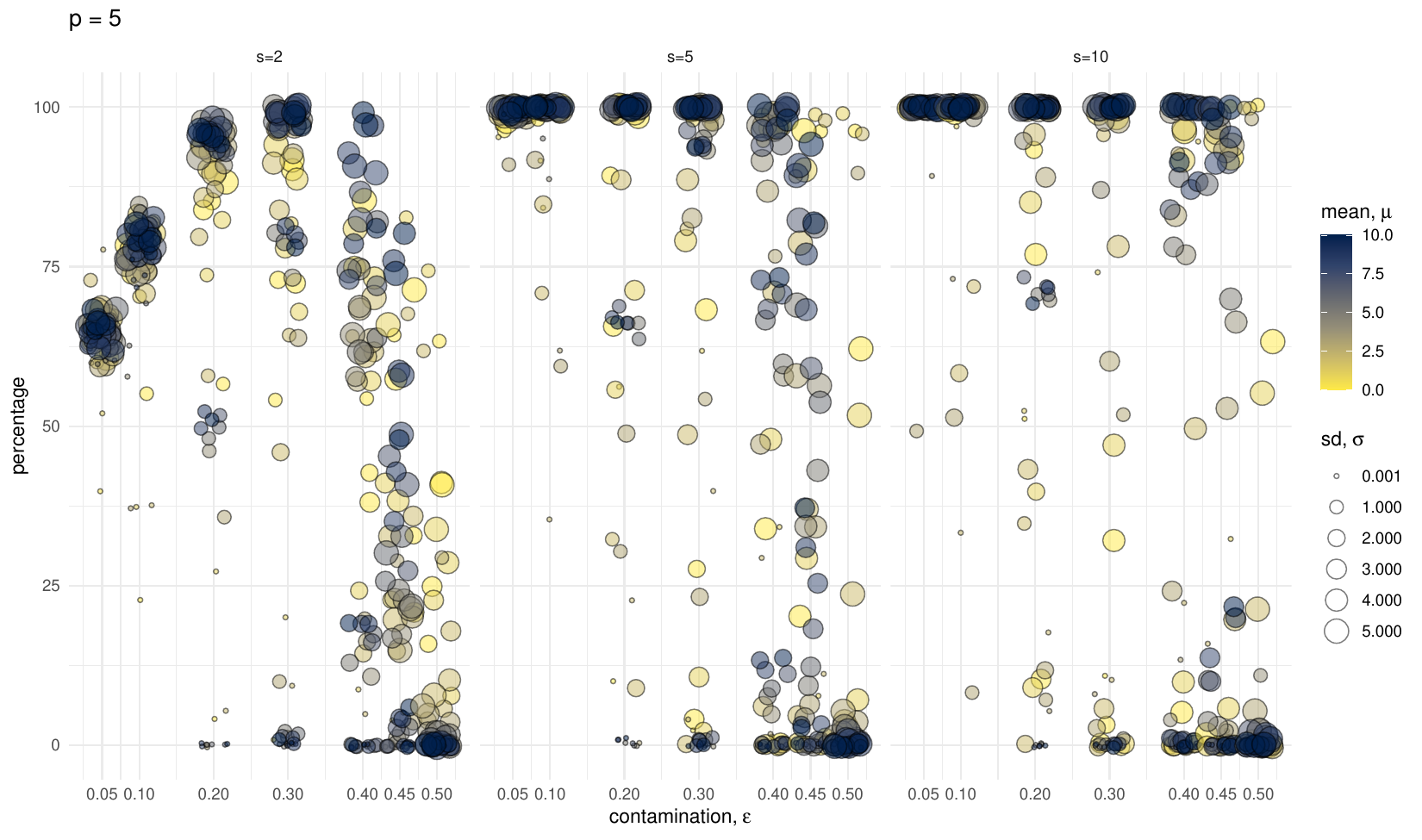}\\
  \includegraphics[width=\textwidth]{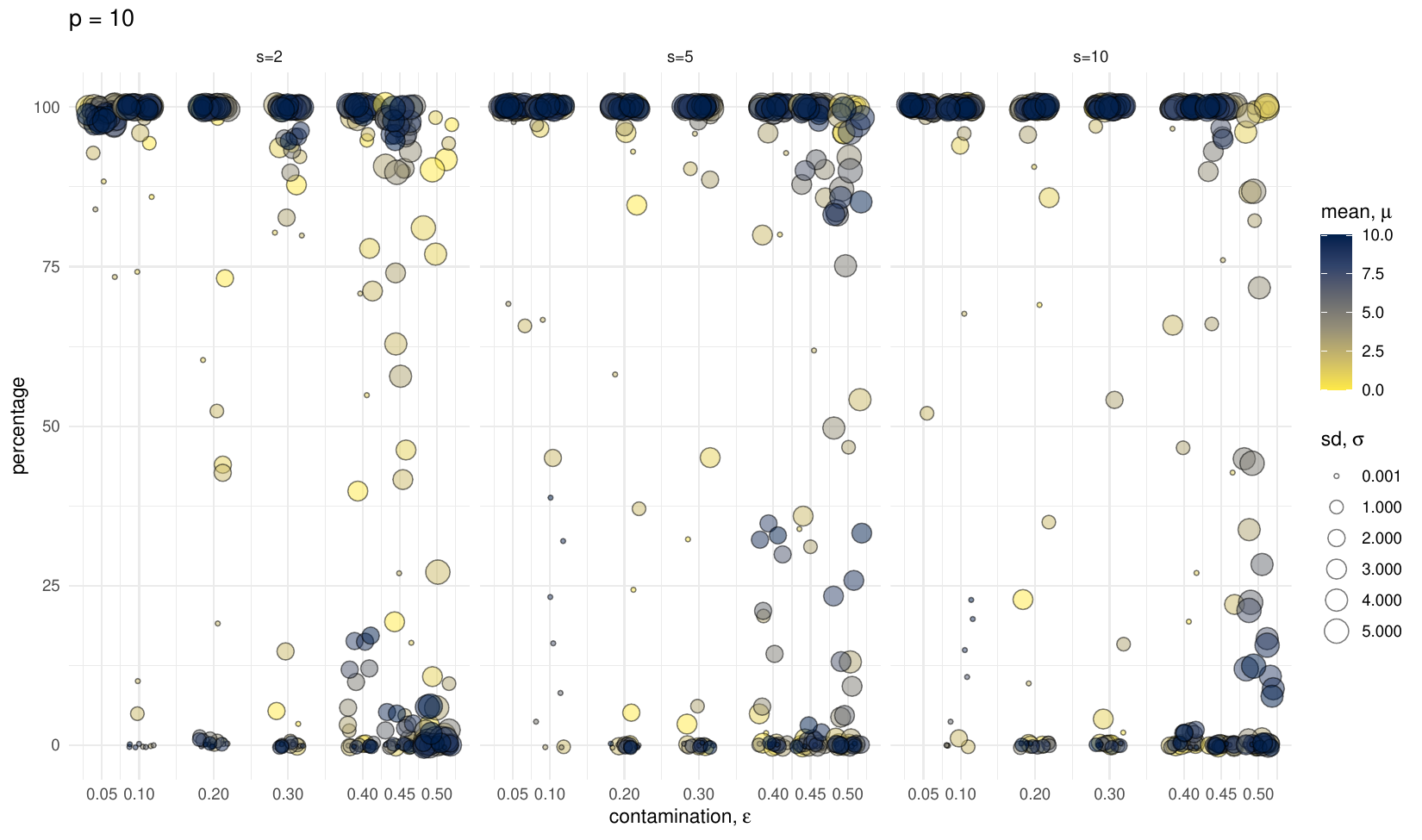}
  \caption{Monte Carlo Simulation. Percentage of robust root retrieval out of $N=100$ simulations for the proposed method starting deterministically from the deepest points as a function of contamination $\varepsilon$ ($x-$axis). Contamination average $\mu$ and scale $\sigma$ are represented by color and size of the bubbles respectively. From left to right, the subplots show the different sample size factors $s=2, 5, 10$. Number of variables $p=5$ (top) and $p=10$ (bottom). $\alpha=0.25$.}
  \label{sup:fig:depth-true-2}
\end{figure}

\begin{figure}
  \centering
  \includegraphics[width=\textwidth]{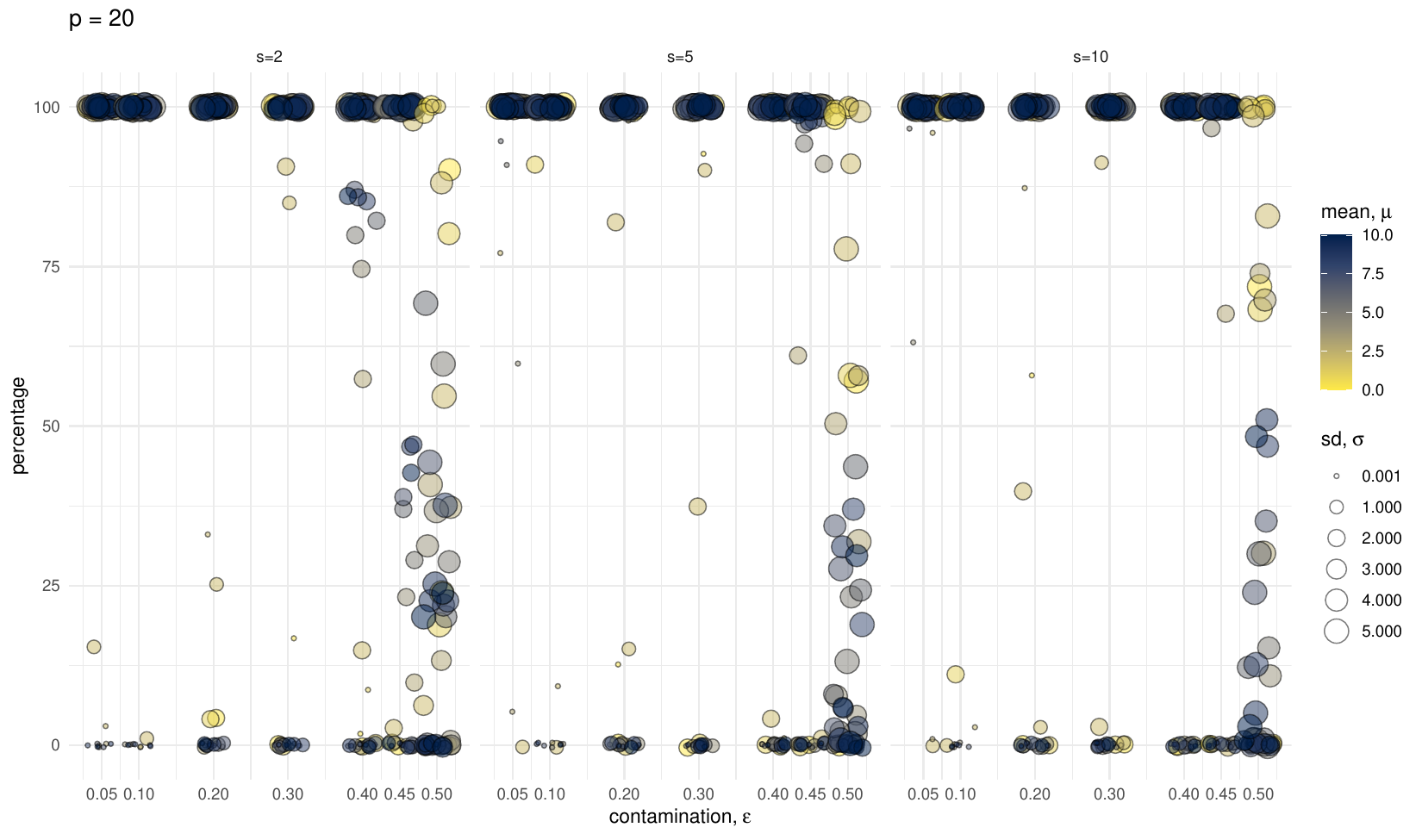}
  \caption{Monte Carlo Simulation. Percentage of robust root retrieval out of $N=100$ simulations for the proposed method starting deterministically from the deepest points as a function of contamination $\varepsilon$ ($x-$axis). Contamination average $\mu$ and scale $\sigma$ are represented by color and size of the bubbles respectively. From left to right, the subplots show the different sample size factors $s=2, 5, 10$. Number of variables $p=20$. $\alpha=0.25$.}
  \label{sup:fig:depth-true-2b}
\end{figure}

\begin{figure}
  \centering
  \includegraphics[width=\textwidth]{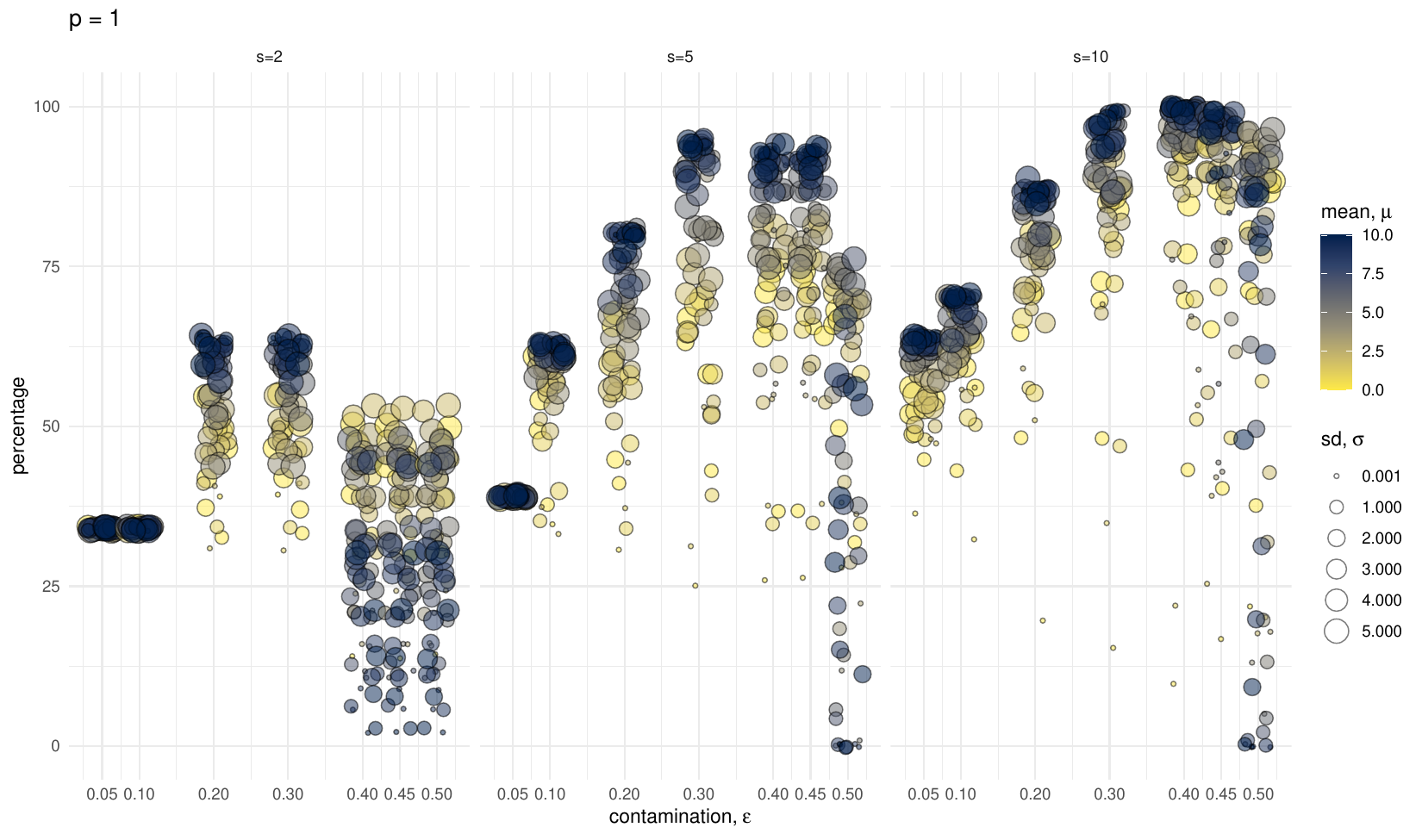}\\
  \includegraphics[width=\textwidth]{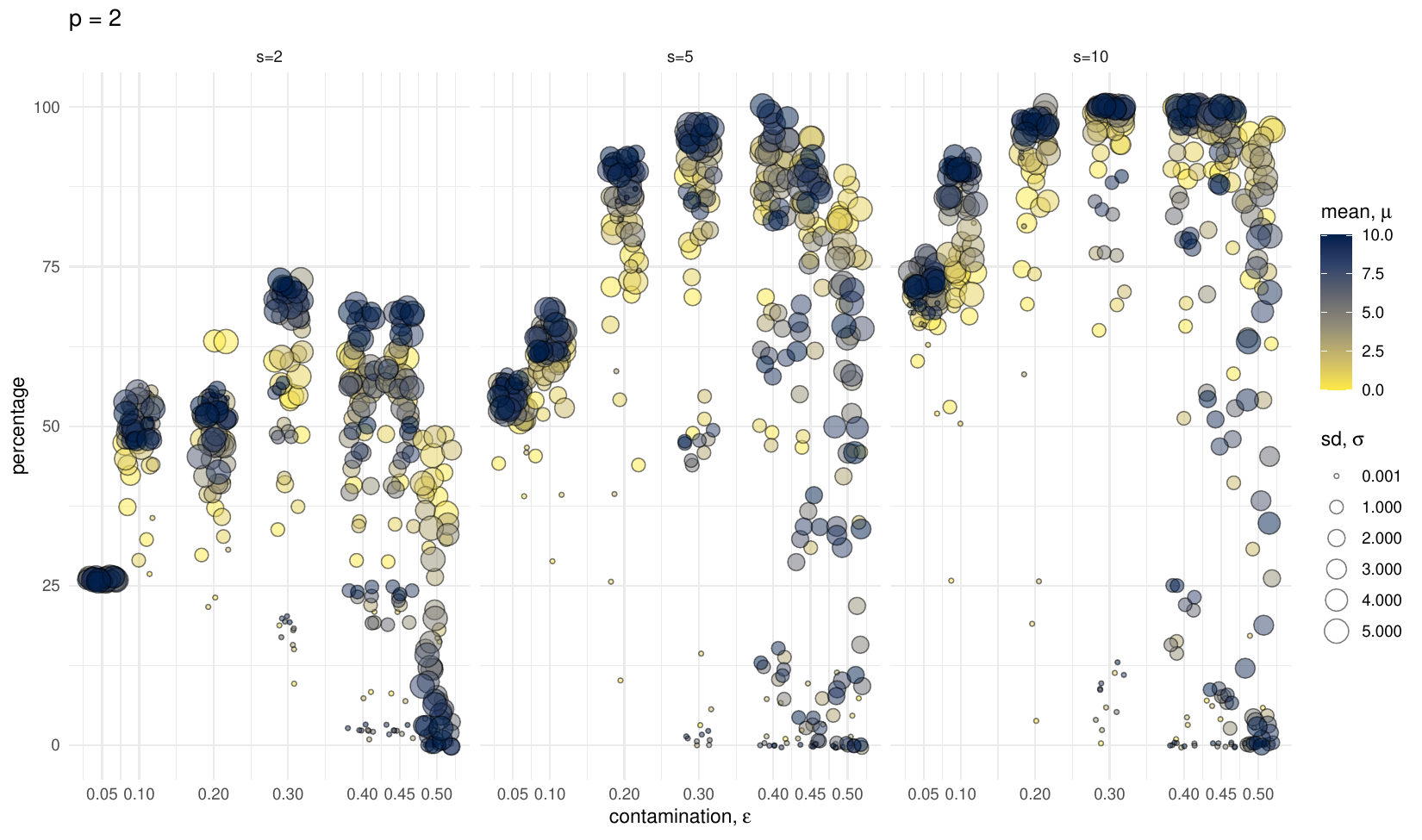} 
  \caption{Monte Carlo Simulation. Percentage of robust root retrieval out of $N=100$ simulations for the proposed method starting deterministically from the deepest points as a function of contamination $\varepsilon$ ($x-$axis). Contamination average $\mu$ and scale $\sigma$ are represented by color and size of the bubbles respectively. From left to right, the subplots show the different sample size factors $s=2, 5, 10$. Number of variables $p=1$ (top) and $p=2$ (bottom). $\alpha=0.5$.}
  \label{sup:fig:depth-true-3}
\end{figure}

\begin{figure}
  \centering
  \includegraphics[width=\textwidth]{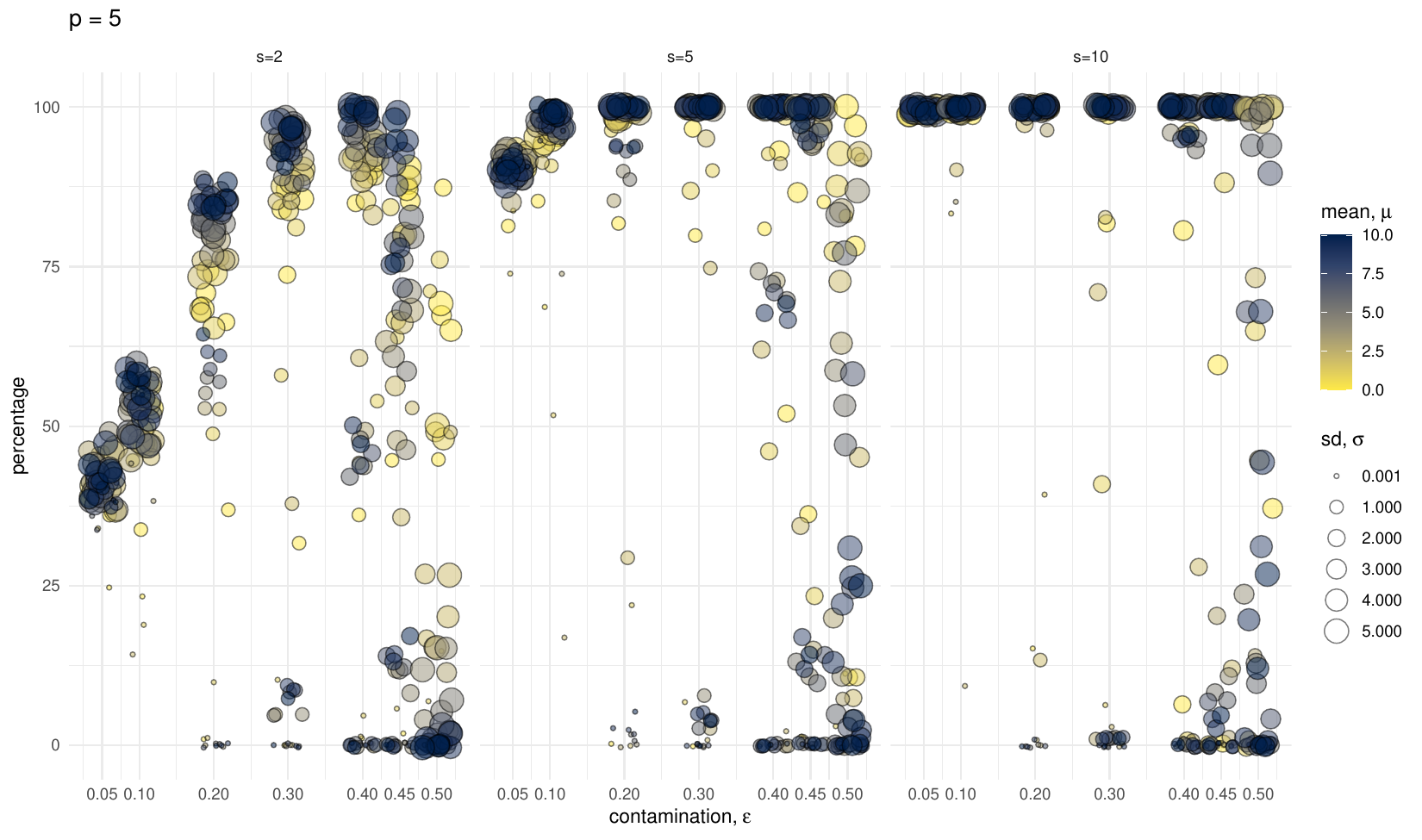}\\
  \includegraphics[width=\textwidth]{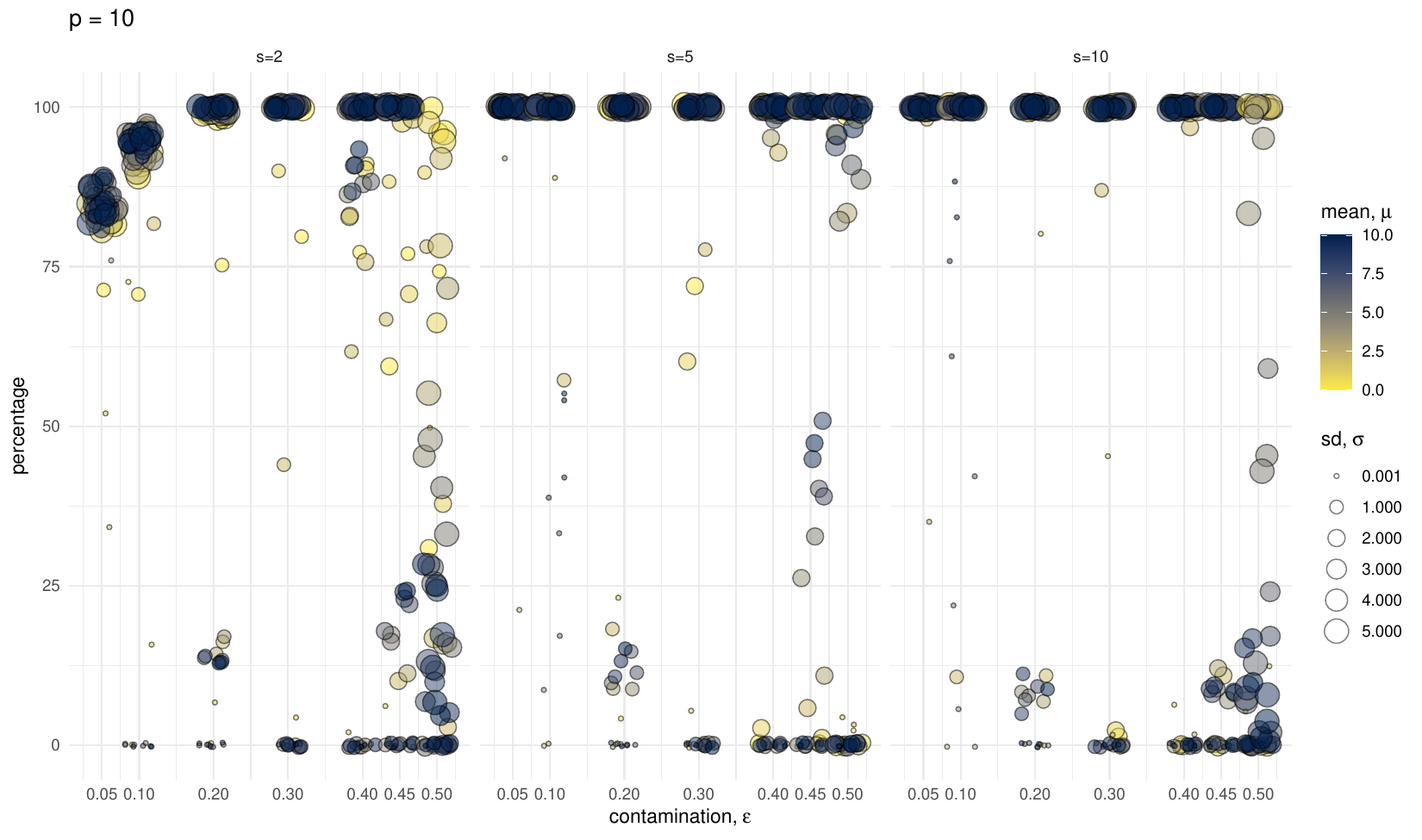}
  \caption{Monte Carlo Simulation. Percentage of robust root retrieval out of $N=100$ simulations for the proposed method starting deterministically from the deepest points as a function of contamination $\varepsilon$ ($x-$axis). Contamination average $\mu$ and scale $\sigma$ are represented by color and size of the bubbles respectively. From left to right, the subplots show the different sample size factors $s=2, 5, 10$. Number of variables $p=5$ (top) and $p=10$ (bottom). $\alpha=0.5$.}
  \label{sup:fig:depth-true-4}
\end{figure}

\begin{figure}
  \includegraphics[width=\textwidth]{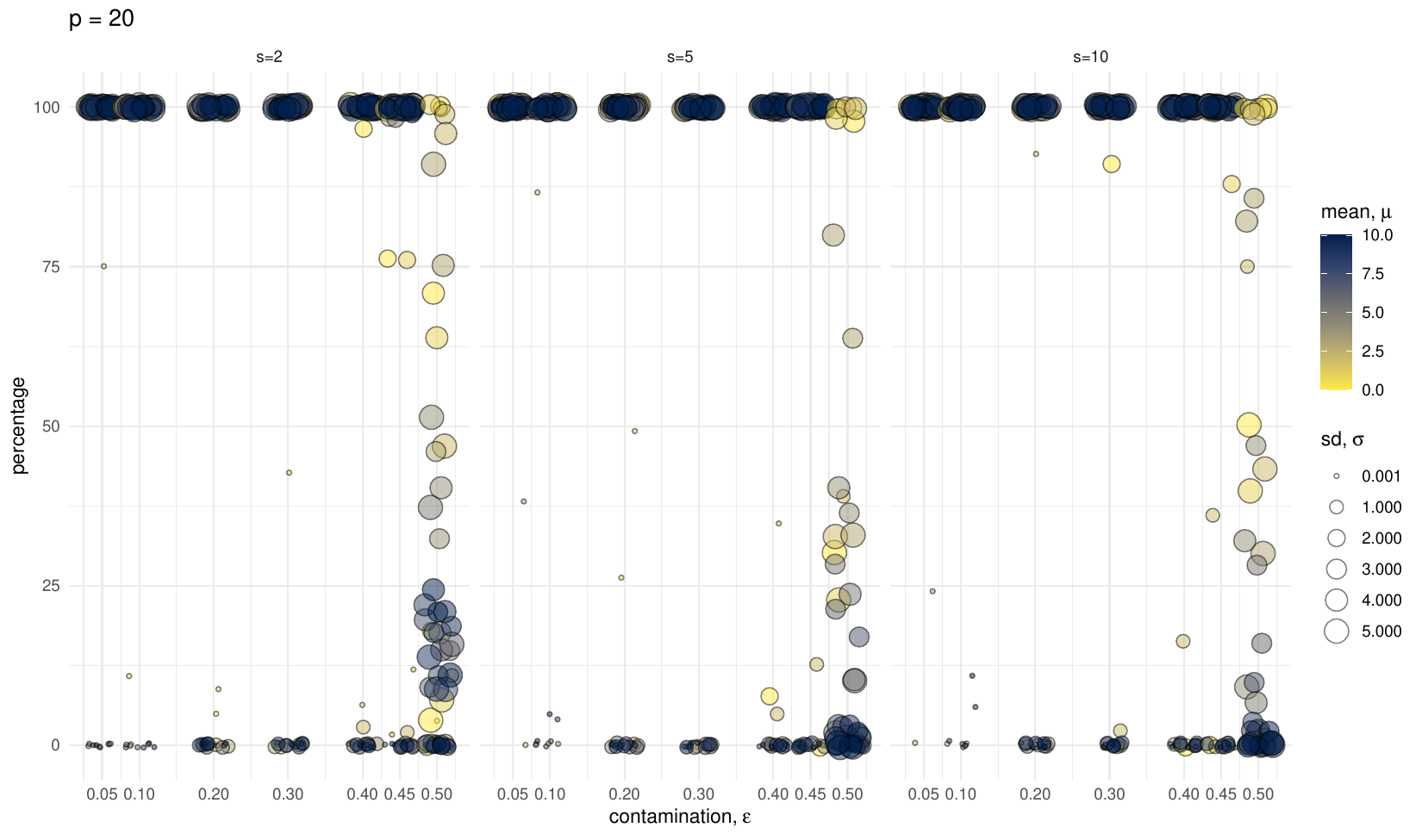}
  \caption{Monte Carlo Simulation. Percentage of robust root retrieval out of $N=100$ simulations for the proposed method starting deterministically from the deepest points as a function of contamination $\varepsilon$ ($x-$axis). Contamination average $\mu$ and scale $\sigma$ are represented by color and size of the bubbles respectively. From left to right, the subplots show the different sample size factors $s=2, 5, 10$. Number of variables $p=20$. $\alpha=0.5$.}
  \label{sup:fig:depth-true-4b}
\end{figure}

\begin{figure}
  \centering
  \includegraphics[width=\textwidth]{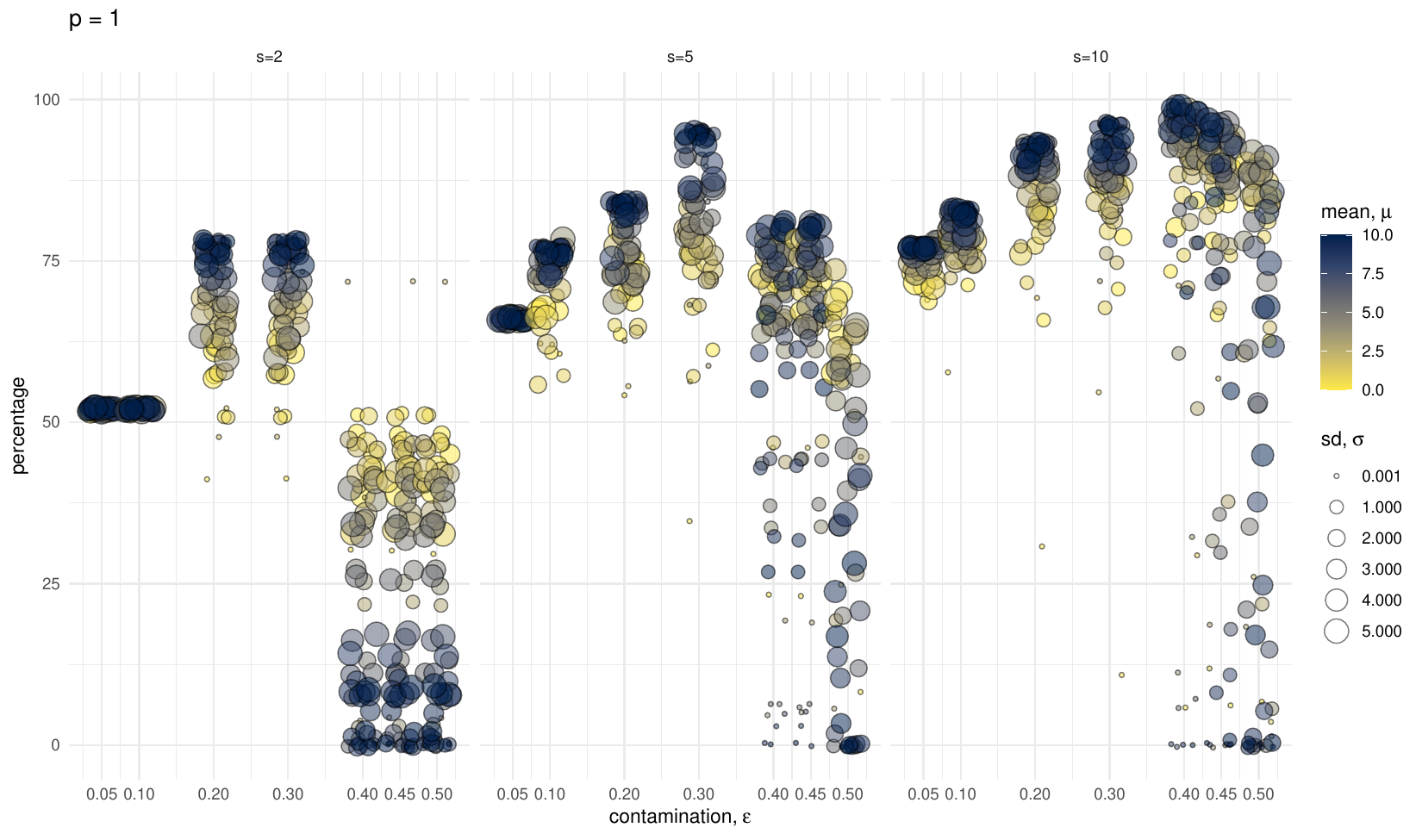}\\
  \includegraphics[width=\textwidth]{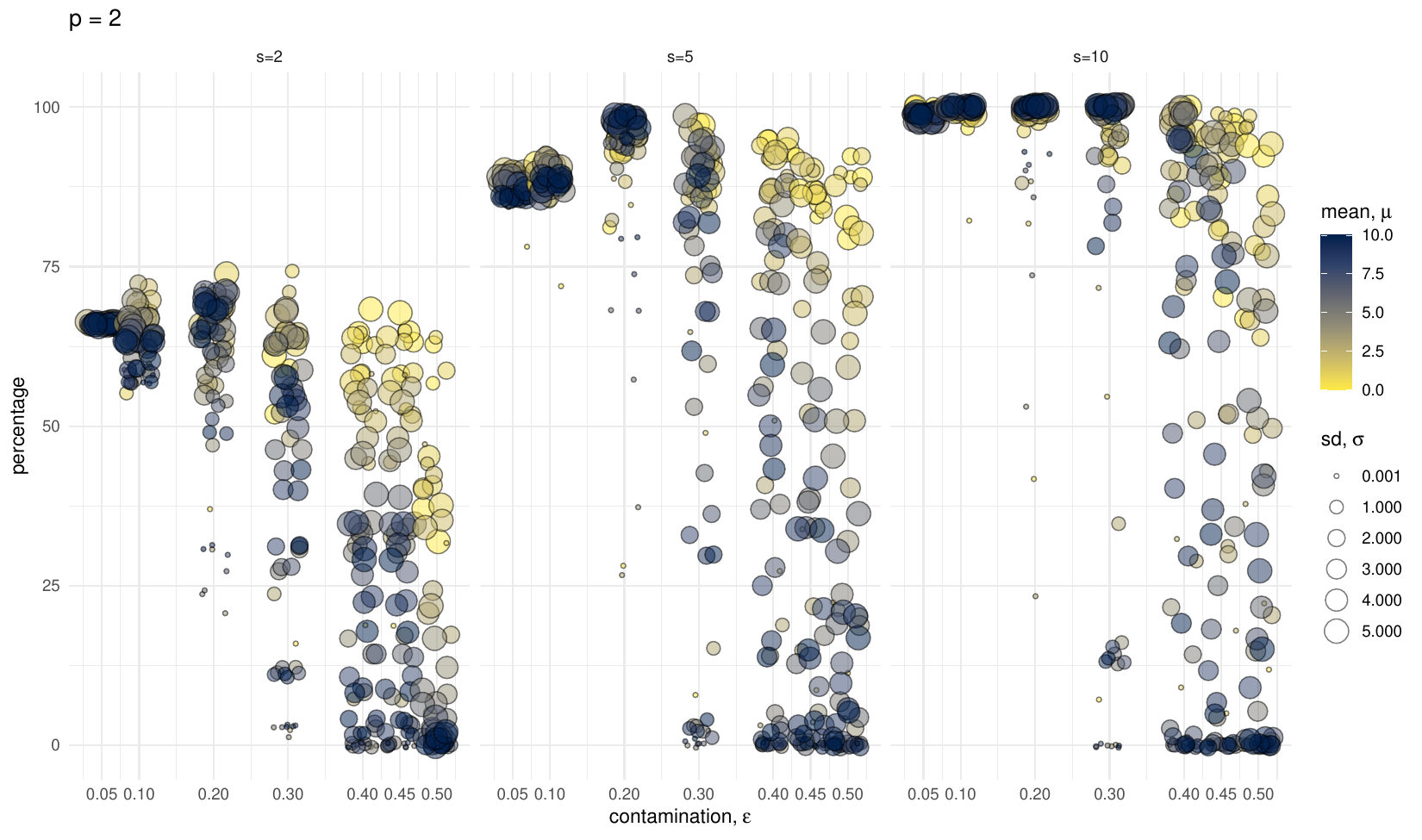} 
  \caption{Monte Carlo Simulation. Percentage of robust root retrieval out of $N=100$ simulations for the proposed method starting deterministically from the deepest points as a function of contamination $\varepsilon$ ($x-$axis). Contamination average $\mu$ and scale $\sigma$ are represented by color and size of the bubbles respectively. From left to right, the subplots show the different sample size factors $s=2, 5, 10$. Number of variables $p=1$ (top) and $p=2$ (bottom). $\alpha=0.75$.}
  \label{sup:fig:depth-true-5}
\end{figure}

\begin{figure}
  \centering
  \includegraphics[width=\textwidth]{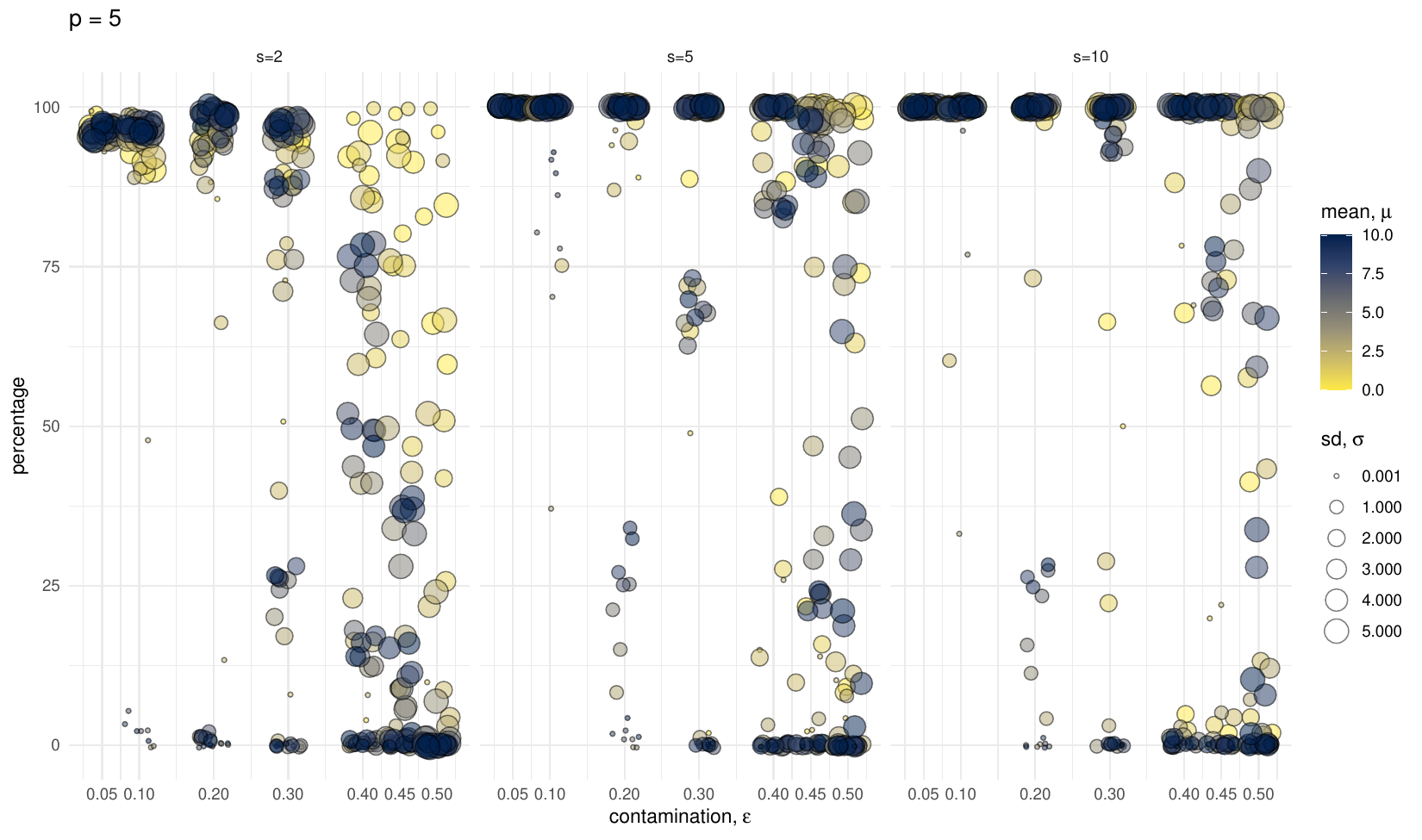}\\
  \includegraphics[width=\textwidth]{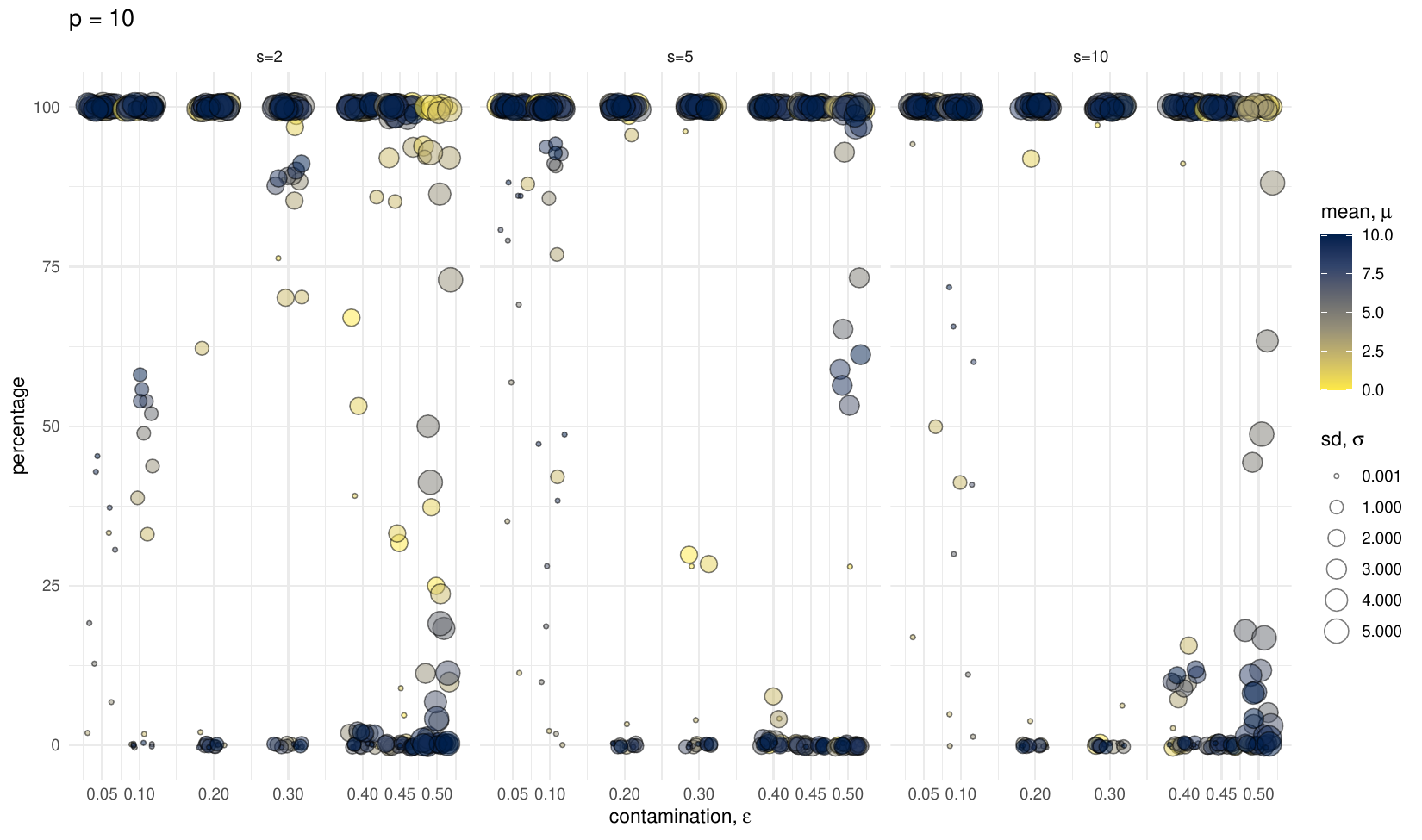}
  \caption{Monte Carlo Simulation. Percentage of robust root retrieval out of $N=100$ simulations for the proposed method starting deterministically from the deepest points as a function of contamination $\varepsilon$ ($x-$axis). Contamination average $\mu$ and scale $\sigma$ are represented by color and size of the bubbles respectively. From left to right, the subplots show the different sample size factors $s=2, 5, 10$. Number of variables $p=5$ (top) and $p=10$ (bottom). $\alpha=0.75$.}
  \label{sup:fig:depth-true-6}
\end{figure}

\begin{figure}
  \includegraphics[width=\textwidth]{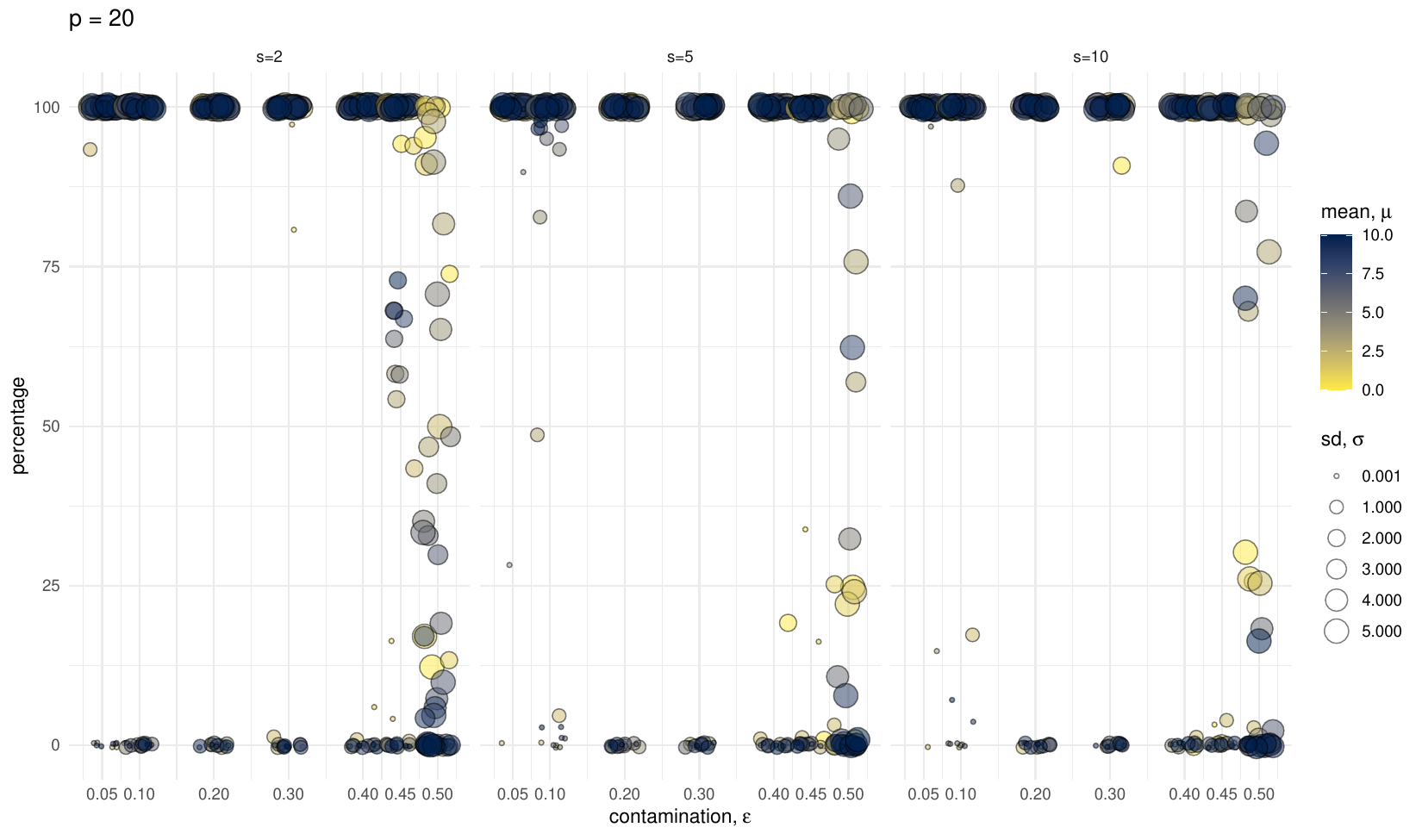}
  \caption{Monte Carlo Simulation. Percentage of robust root retrieval out of $N=100$ simulations for the proposed method starting deterministically from the deepest points as a function of contamination $\varepsilon$ ($x-$axis). Contamination average $\mu$ and scale $\sigma$ are represented by color and size of the bubbles respectively. From left to right, the subplots show the different sample size factors $s=2, 5, 10$. Number of variables $p=20$. $\alpha=0.75$.}
  \label{sup:fig:depth-true-6b}
\end{figure}

\begin{figure}
  \centering
  \includegraphics[width=\textwidth]{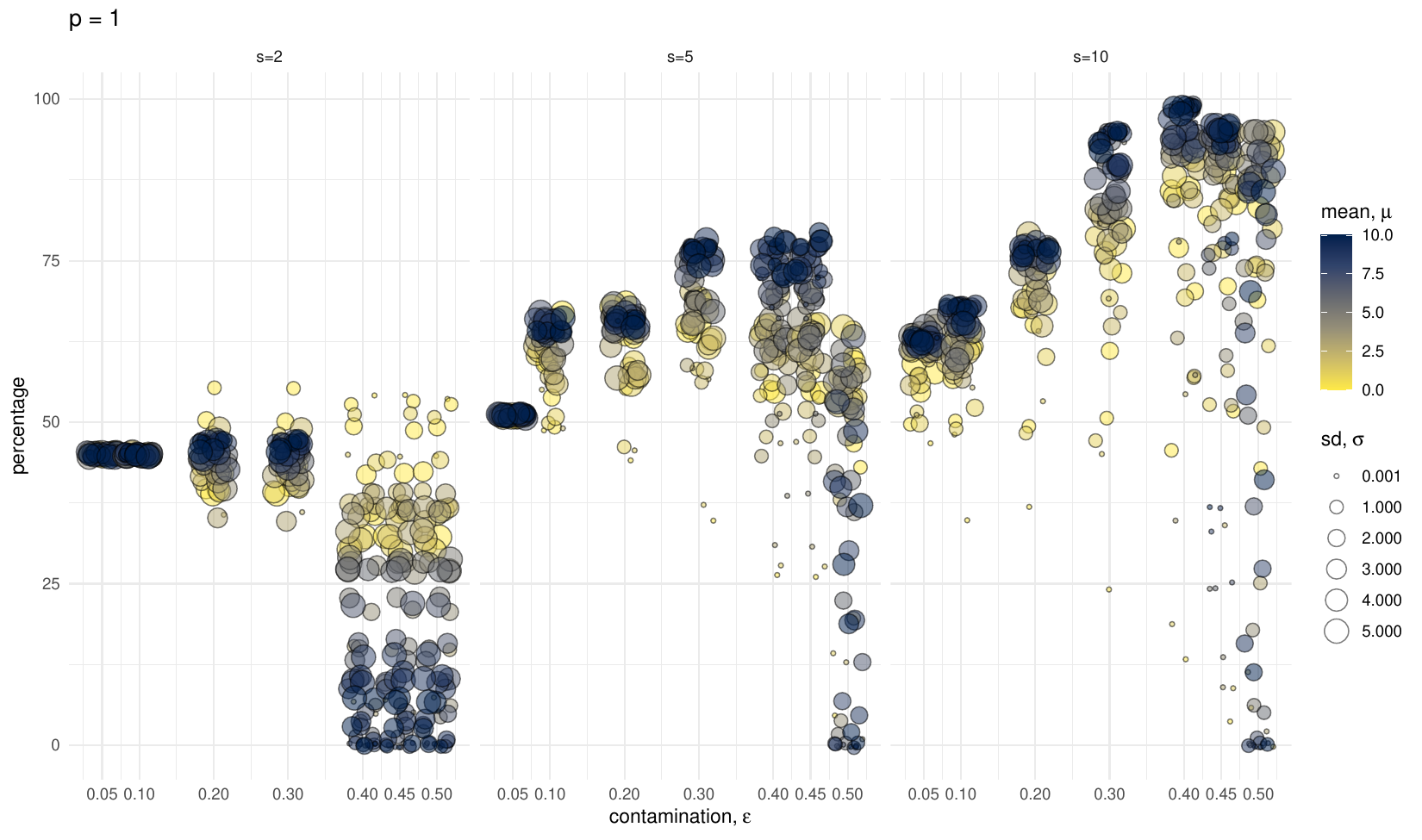}\\
  \includegraphics[width=\textwidth]{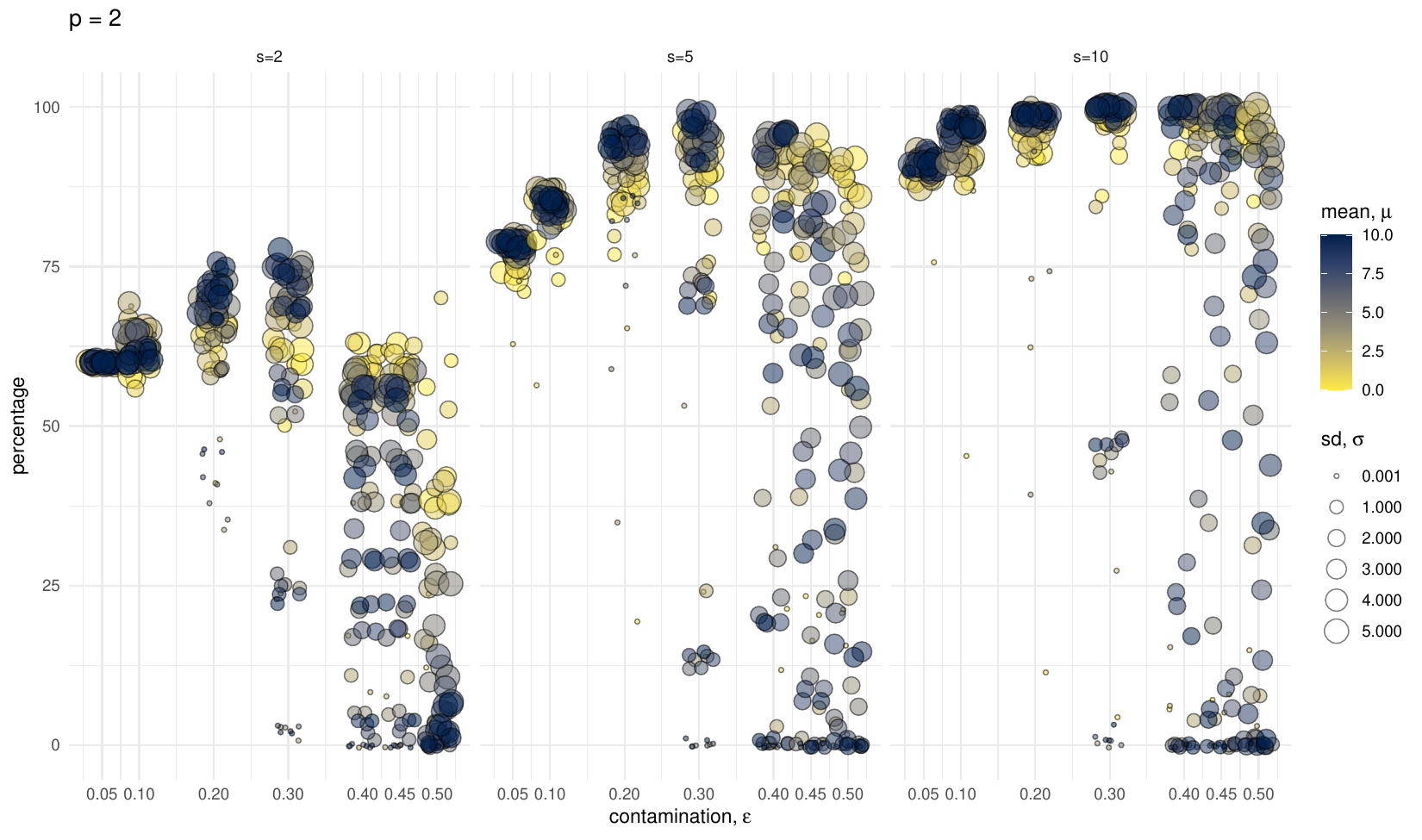} 
  \caption{Monte Carlo Simulation. Percentage of robust root retrieval out of $N=100$ simulations for the proposed method starting deterministically from the deepest points as a function of contamination $\varepsilon$ ($x-$axis). Contamination average $\mu$ and scale $\sigma$ are represented by color and size of the bubbles respectively. From left to right, the subplots show the different sample size factors $s=2, 5, 10$. Number of variables $p=1$ (top) and $p=2$ (bottom). $\alpha=1$.}
  \label{sup:fig:depth-true-7}
\end{figure}

\begin{figure}
  \centering
  \includegraphics[width=\textwidth]{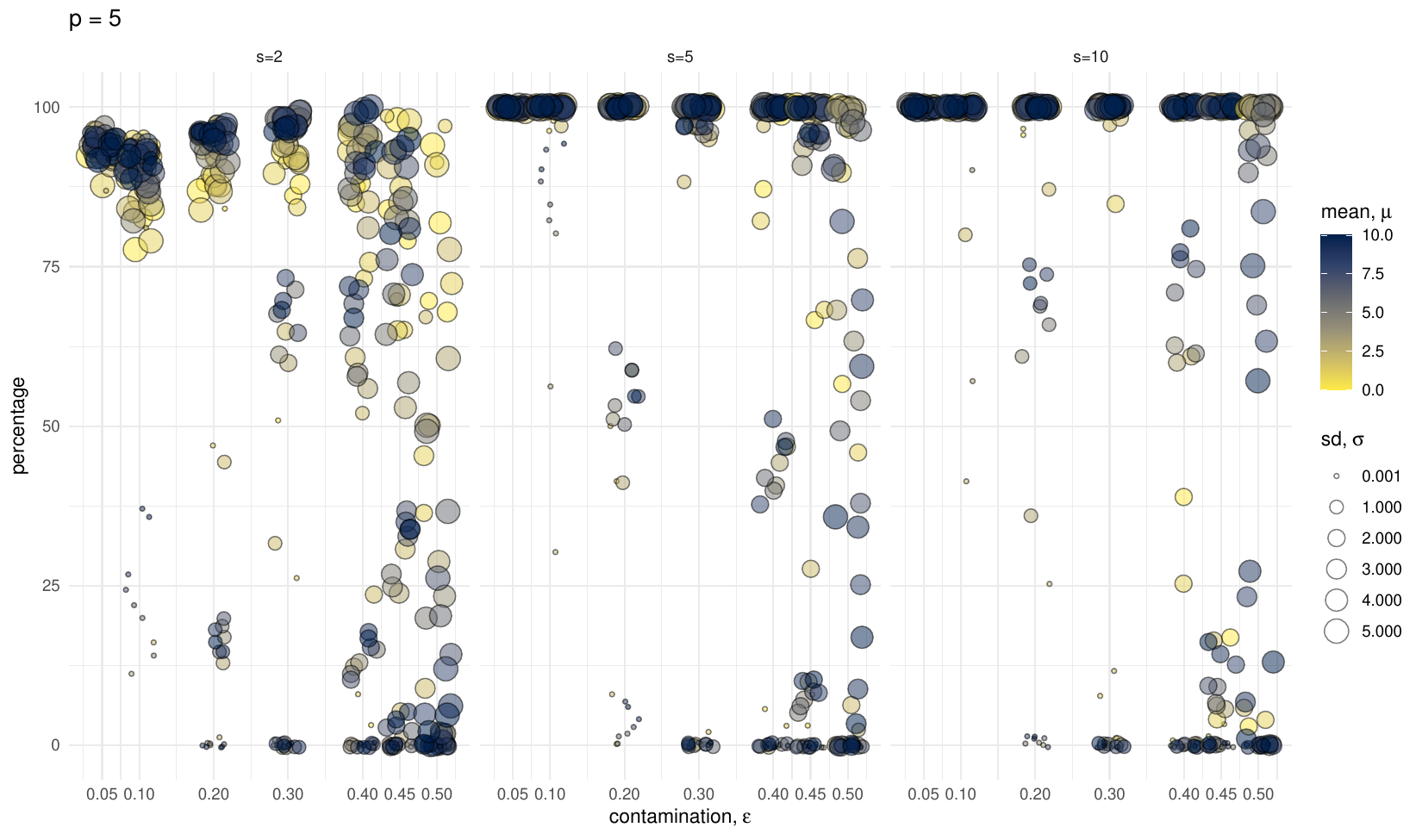}\\
  \includegraphics[width=\textwidth]{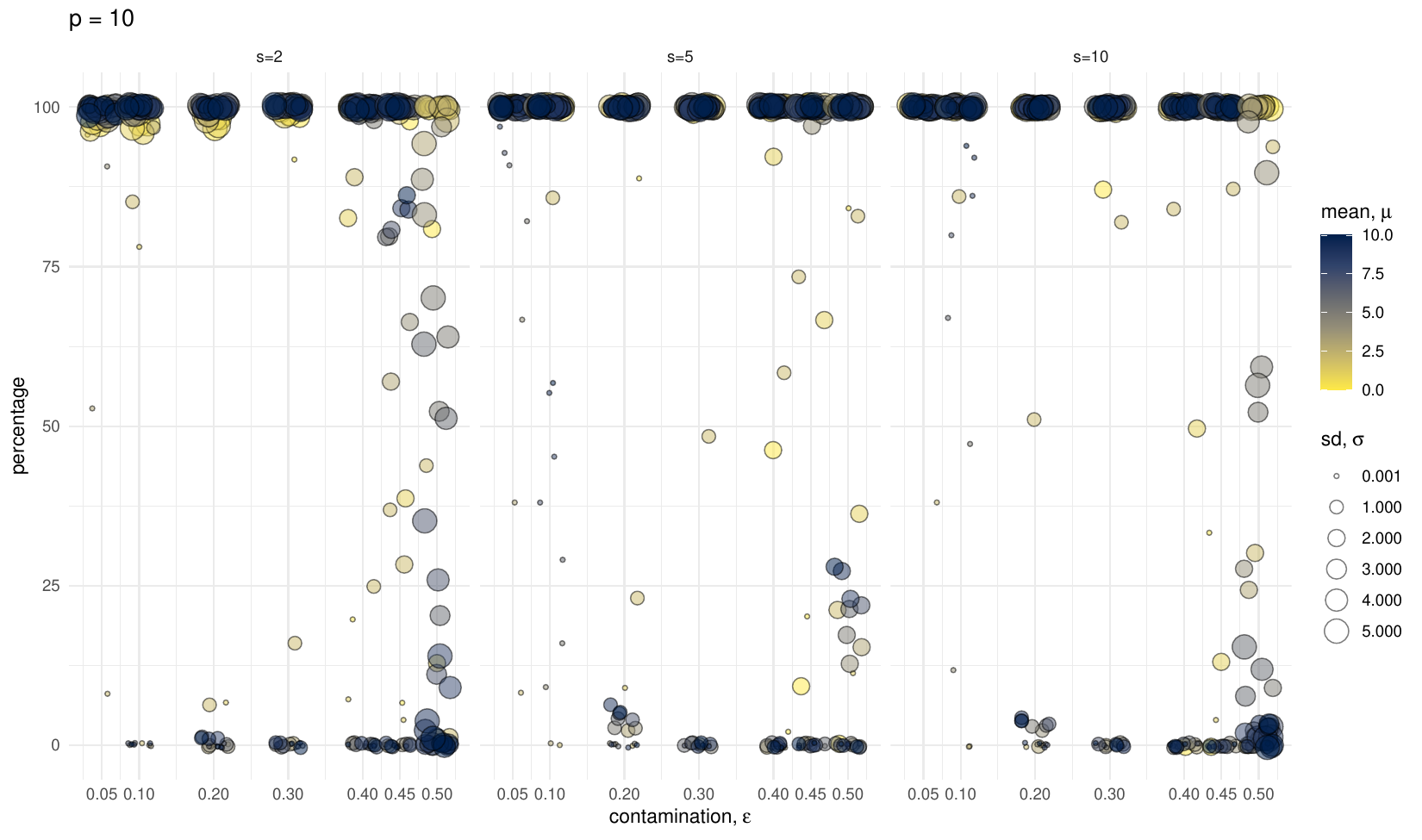}
  \caption{Monte Carlo Simulation. Percentage of robust root retrieval out of $N=100$ simulations for the proposed method starting deterministically from the deepest points as a function of contamination $\varepsilon$ ($x-$axis). Contamination average $\mu$ and scale $\sigma$ are represented by color and size of the bubbles respectively. From left to right, the subplots show the different sample size factors $s=2, 5, 10$. Number of variables $p=5$ (top) and $p=10$ (bottom). $\alpha=1$.}
  \label{sup:fig:depth-true-8}
\end{figure}

\begin{figure}
  \includegraphics[width=\textwidth]{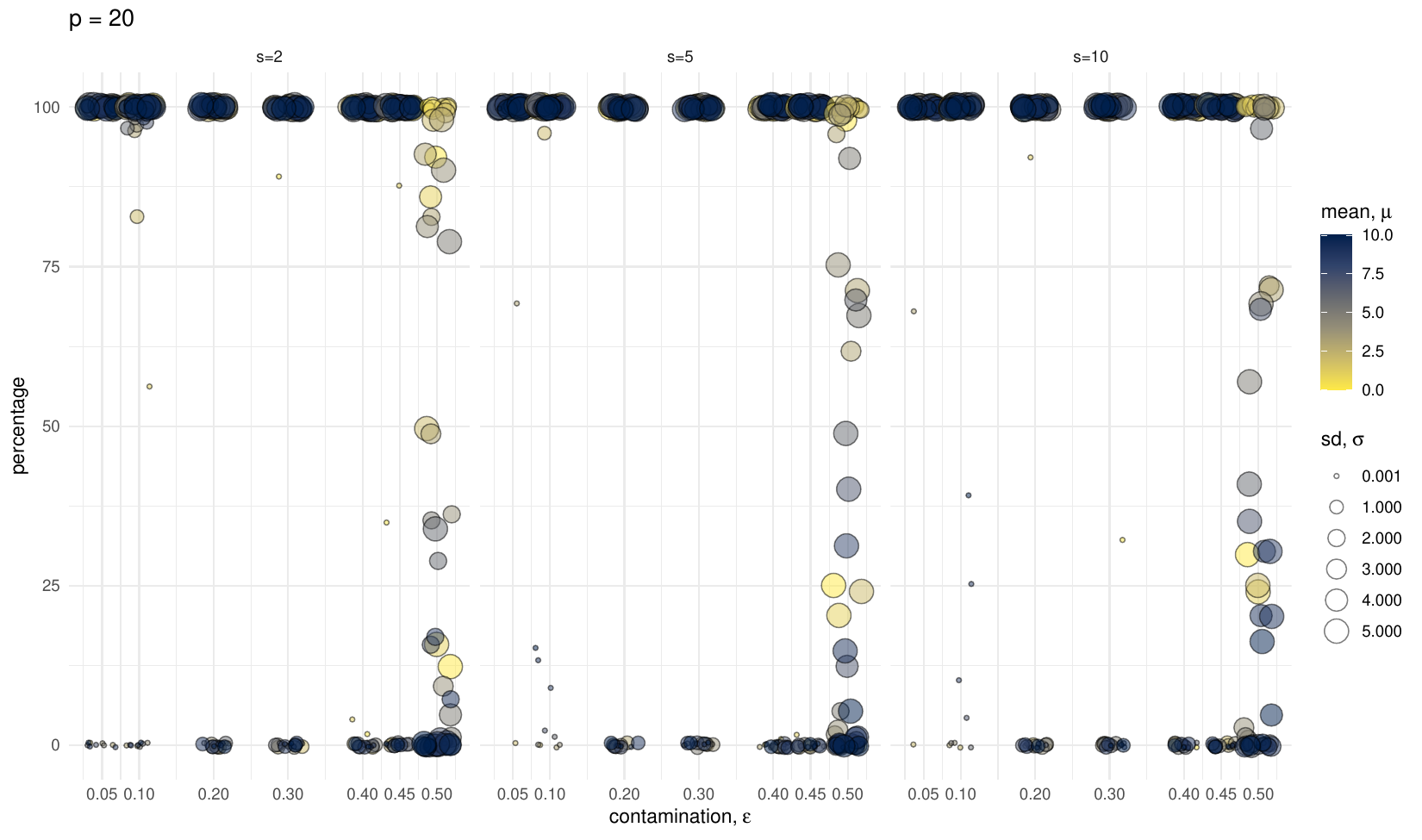}
  \caption{Monte Carlo Simulation. Percentage of robust root retrieval out of $N=100$ simulations for the proposed method starting deterministically from the deepest points as a function of contamination $\varepsilon$ ($x-$axis). Contamination average $\mu$ and scale $\sigma$ are represented by color and size of the bubbles respectively. From left to right, the subplots show the different sample size factors $s=2, 5, 10$. Number of variables $p=20$. $\alpha=1$.}
  \label{sup:fig:depth-true-8b}
\end{figure}

\end{document}